\documentclass{amsart}
\usepackage{amsmath,amssymb,mathrsfs,bbm,yhmath,longtable,mathtools}
\usepackage{amsthm,rotating,xcolor,epsfig,hyperref,url}

\usepackage{graphicx}
\allowdisplaybreaks

\date{}

\def\interior{\qopname\relax o{int}}
\def\tb{\qopname\relax o{tb}}

\def\lk{\qopname\relax o{lk}}

\theoremstyle{theorem}
\newtheorem{theo}{Theorem}[section]
\newtheorem{lemm}{Lemma}[section]
\newtheorem{prop}{Proposition}[section]
\newtheorem{coro}{Corollary}[section]

\theoremstyle{remark}
\newtheorem{rema}{Remark}[section]
\newtheorem{exam}{Example}[section]

\theoremstyle{definition}
\newtheorem{defi}{Definition}[section]
\newtheorem{conv}{Convention}[section]

\numberwithin{equation}{section}
\numberwithin{figure}{section}

\graphicspath{{./pictures/}}	

\author {Ivan Dynnikov and Maxim Prasolov}
\address{Steklov Mathematical Institute of Russian Academy of Sciences, 8 Gubkina Str., Moscow 119991, Russia}
\email{dynnikov@mech.math.msu.su}
\email{0x00002a@gmail.com}
\thanks{This work was performed at the Steklov International Mathematical Center and supported by the Ministry of Science and Higher Education of the Russian Federation (agreement no. 075-15-2019-1614).
The work of M.\,Prasolov has also been partially supported by the Young Russian Mathematics award.}
\title{Rectangular diagrams of surfaces: distinguishing Legendrian knots}

\begin{document}
\maketitle

\begin{abstract}
In an earlier paper we introduced rectangular diagrams of surfaces and showed
that any isotopy class of a surface in the three-sphere can be presented
by a rectangular diagram. Here we study transformations of those diagrams
and introduce basic moves that allow the transition between diagrams representing
isotopic surfaces. We also introduce more general combinatorial
objects called mirror diagrams and various moves for them that can be
used to transform presentations of isotopic surfaces to each other. The moves are
divided into two (non-exclusive) types so that, vaguely speaking, type~I moves commute with type~II ones.
This commutation is the matter of the main technical result of the paper.
We use it as well as a relation of the moves to Giroux's convex
surfaces to propose a new method for distinguishing Legendrian knots.
We apply this method to show that two Legendrian knots having topological type $6_2$ are not equivalent.
More applications of the method will be the subject of subsequent papers.
\end{abstract}

\tableofcontents

\section{Introduction}

For a pre-introduction the reader is referred to the work~\cite{dp17}.

This paper is focused on combinatorial properties of certain objects
that generalize rectangular diagrams of links and provide for a nice way
to represent surfaces and ribbon graphs in the three-sphere. We call these objects rectangular
diagrams of surfaces, and mirror diagrams, respectively. The
combinatorial formalism we develop here is strongly related to two contact structures
of~$\mathbb S^3$, the standard one,~$\xi_+$, and its mirror image,~$\xi_-$.
This is by no means unexpected since rectangular diagrams of links
are already known to provide for a convenient framework to study Legendrian links in~$\mathbb S^3$.

Various diagrams of `rectangular' kind that we consider represent topological objects
that are in a nice position with respect to both contact structures $\xi_+$ and~$\xi_-$, simultaneously.
Namely, links represented in the `rectangular' way are Legendrian with respect to both contact structures, and surfaces represented
in the `rectangular' way
are convex in Giroux's sense, also with respect to both contact structures. So, each single `rectangular' object represents two topological
objects that are of interest from the contact topology point of view.
The key circumstance about these two objects is their mutual independence to
an extent allowed by the topological settings.

In the case of links, this independence, which is discovered in~\cite{DyPr},
means that, for any~$\xi_+$-Legendrian link type~$\mathscr L_+$ and~$\xi_-$-Legendrian link type~$\mathscr L_-$ which
belong
to the same topological link type, there is a rectangular diagram
representing both of them. Moreover, any stabilization and destabilization of one of these
Legendrian link types can be realized, without altering the other, by elementary moves of
rectangular diagrams.

The independence of any such~$\mathscr L_+$ and~~$\mathscr L_-$ has remarkable consequences, one of which
is a proof of the Jones conjecture~\cite{DyPr}. Another consequence is an algorithm for computing
the maximal possible Thurston--Bennequin number of Legendrian links having a given
topological link type.

In the present paper we formulate and prove a similar independence property for Giroux's convex surfaces.
A simplified version of this property says that if~$F_+$ and~$F_-$ are isotopic closed surfaces
in~$\mathbb S^3$ such that~$F_+$ is convex (in Giroux's sense) with respect to~$\xi_+$,
and~$F_-$ is convex with respect to~$\xi_-$, then there is a rectangular diagram of a surface~$\Pi$
such that the associated surface~$\widehat\Pi$ is isotopic to~$F_+$ through $\xi_+$-convex surfaces
and to~$F_-$ through $\xi_-$-convex surfaces. The importance of this fact comes from
the existence of an algorithm that, given certain combinatorial
information about~$F_+$ and~$F_-$, produces finitely many candidates for such~$\Pi$ among which at least one has the required properties.

The information supplied to the algorithm includes the combinatorial structure of the dividing sets~$\delta_+$ and~$\delta_-$ of~$F_+$ and~$F_-$,
respectively, and their mutual position.
By dividing sets one means certain one-dimensional submanifolds in Giroux's convex surfaces
defined with a reference to the respective contact structures.
The isotopy class of a dividing set is a very strong invariant of a convex surface as discovered by E.\,Giroux~\cite{Gi}.
The mutual position of~$\delta_+$ and~$\delta_-$ is understood as follows.
The surfaces~$F_+$ and~$F_-$ are identified using the given isotopy between them,
so~$\delta_+$ and~$\delta_-$ are viewed as submanifolds
in a single surface. After putting them in general position the union~$\delta_+\cup\delta_-$
is a finite one-dimensional CW-subcomplex of the surface.

The algorithm mentioned above takes a combinatorial description of~$\delta_+\cup\delta_-$ as the input
and produces finitely many rectangular diagrams of a surface such that at least one of them
has the desired properties, that is, represents a surface isotopic to~$F_+$ through $\xi_+$-convex
surfaces and to~$F_-$ through $\xi_-$-convex ones. This can be used,
for an arbitrary given one-dimensional submanifold~$\delta$ of a surface~$F\in\mathbb S^3$, to decide whether or not
the isotopy class of~$\delta$ can be realized by a dividing set of a convex surface isotopic
to~$F$.

This approach extends, with some limitations, to surfaces with Legendrian boundary, providing for a powerful tool
for distinguishing Legendrian knots.

The comparing of
Legendrian knots having the same classical invariants (which are topological type,
Thurston--Bennequin number, and rotation number) is a difficult problem in general. The first success in this direction
was made in Yu.\,Chekanov's work~\cite{che2002}, where he distinguished two
Legendrian knots having topological type~$5_2$
by means of new algebraic invariants. The latter were extracted from a differential graded algebra associated
with a Lagrangian projection of the knot. Similar construction appeared about the same time in
Ya.\,Eliashberg's work~\cite{El}.

Further algebraic invariants having proved useful to distinguish Legendrian knots are constructed
by D.\,Fuchs~\cite{fuchs2003}, L.\,Ng~\cite{ng2003,ng2011}, P.\,Pushkar' and Yu.\,Chekanov~\cite{PC2005}, and
by P.\,Ozsv\'ath, Z.\,Szab\'o, and D.\,Thurston~\cite{ost2008}.

However, algebraic invariants do not provide for a systematic way to compare Legendrian knots.
There are still many examples of pairs of Legendrian knot types for which
the existing evidence suggests that they are distinct
but the known computable invariants fail to confirm that. The reader is referred to the Legendrian knot atlas
by W.\,Chongchitmate and L.\,Ng~\cite{chong2013} for an overview of the state
of the art.

A number of classification
results on Legendrian knots, which are due to Y.\,Eliashberg and M.\,Fraser~\cite{EF1,EF2},
J.\,Etnyre and K.\,Honda~\cite{etho2001}, J.\,Etnyre, D.\,LaFountain, and B.\,Tosun\cite{etlafato2012},
J.\,Etnyre, L.\,Ng, and V.\,V\'ertesi~\cite{etngve2013}, J.\,Etnyre and V.\,V\'ertesi~\cite{etver2016},
are obtained by means of a different approach
based on the study of Giroux's convex surfaces and characteristic foliations.

Even if Legendrian types are classified for a given topological type of a knot,
it may be still difficult to recognize them if the knot is not Legendrian simple. This is where
the technique of the present paper may also be useful, since it allows one, under certain
circumstances, to find convex surfaces with desired structure of the dividing set
or to prove that no such surface exists. Combined with the ideas of the paper~\cite{hokama2000}
by K.\,Honda, W.\,Kazez, and G.\,Mati\'c,
where Haken hierarchies built from convex surfaces are studied, this approach
has a potential to yield a complete algorithm for comparing Legendrian knots.

The biggest difficulty with making this work
comes from the fact that the topological orientation-preserving symmetry group of a knot may
be infinite and is unknown in general\footnote{Since submitting the original version of this paper we have found a way
to overcome this difficulty. In a future paper we will show that
equivalence of Legendrian knots is decidable.}.
If this group is known
to be trivial for a given knot type, then comparing Legendrian knots of this
topological type becomes fairly easy as explained
in~\cite{dysha18,dysha20??}. In particular, Conjecture~1 of~\cite{dp17}
is confirmed there, as well as a number of conjectures of~\cite{chong2013}
about concrete knots with trivial orientation-preserving symmetry group. This includes
the topological types~$9_{42}$--$9_{45}$, $10_{128}$, and~$10_{160}$.

In order to illustrate our method here, we picked the simplest unresolved case
from~\cite{chong2013}, which deals with two conjecturably inequivalent Legendrian knots having
topological type~$6_2$. We show that the knots are indeed not
Legendrian equivalent. Up to this writing,
we have also confirmed, in a similar fashion, the conjectures of~\cite{chong2013} about
Legendrian knots having topological type~$7_6$ and maximizing
the Thurston--Bennequin number~\cite{dynn-pras-7-6}. This is the example we tried next after~$6_2$.
It looks quite feasible to resolve all open question in~\cite{chong2013}
in the near future by means of the method described in this paper.

The techniques developed in this paper was originally motivated by
an attempt to extend the monotonic simplification approach of~\cite{dyn06}
to general links. This approach utilized the ideas of the preceding
works by J.\,Birman and W.\,Menasco~\cite{bm4,bm5}, and P.\,Cromwell~\cite{crom}.
W.\,Menasco's work~\cite{men}, though not explicitly used here, has
given us a hint for understanding the deep connection between the monotonic simplification
approach and contact topology (the graphs~$G_{++}\cup G_{--}$
and~$G_{+-}\cup G_{-+}$
in that work are close analogues of the Giroux graphs, which are behind the scenes of our method as outlined
in Section~\ref{invariance-sec}).

\subsection{Prerequisites}
We will use the following definitions and notation from
paper~\cite{dp17}:
\begin{itemize}
\item coordinate system $(\theta,\varphi,\tau)$ on $\mathbb S^3$ coming from the join presentation $\mathbb S^3=\mathbb S^1*\mathbb S^1$,
\item torus projection,
\item rectangular diagram of a link,
\item link $\widehat R$ associated with a rectangular diagram of a link $R$,
\item connected component of a rectangular diagram of a link,
\item oriented rectangular diagram of a link,
\item cusp-free curve,
\item framed link,
\item framed rectangular diagram of a link,
\item rectangular diagram of a graph,
\item admissible framing of a link of the form~$\widehat R$,
\item graph $\widehat G$ associated with a rectangular diagram of a graph $G$,
\item a rectangle~$r\subset\mathbb T^2$,
\item tile~$\widehat r$ associated with a rectangle $r\in\mathbb T^2$;
\item rectangular diagram of a surface,
\item surface $\widehat\Pi$ associated with a rectangular diagram of a surface $\Pi$,
\item boundary of a rectangular diagram of a surface,
\item the contact structures $\xi_\pm$ on $\mathbb S^3$,
\item Thurston--Bennequin numbers $\tb_\pm(L)$,
\item relative Thurston--Bennequin numbers $\tb_\pm(L;F)$,
\item surface with corners,
\item contact line element field,
\item Giroux's convex surface,
\item $0$-arc and $-1$-arc,
\item ((very) nice) characteristic foliation,
\item the Giroux graph of a convex surface with a very nice characteristic foliation,
\item an extended Giroux graph of a convex surface with a very nice characteristic foliation,
\item equivalence of Giroux's convex surfaces with corners,
\item exchange moves of rectangular diagrams of links,
\item type I (type II) (de)stabilization of a rectangular diagram of a link.
\end{itemize}

\subsection{Some general conventions and notation}\label{general-conventions}
We work in the piecewise smooth category. Surfaces that we consider
are assumed to be $C^1$-smooth and are allowed to have corners at the boundary according to~\cite[Definition~5]{dp17}.
Isotopies of various objects in~$\mathbb S^3$ are understood as those that can
be extended to ambient piecewise-$C^1$ isotopies.

Unless otherwise specified, an isotopy of a surface
is assumed to be performed within the class of surfaces with corners
so that the tangent plane to the surface depends continuously on~$(x,t)$, where~$x$ is
a point of the surface, and~$t$ is the isotopy parameter.
This matters when we consider an isotopy of a surface with corners relative to its boundary---the tangent
plane to the surface at a singularity of the boundary must remain fixed during the isotopy.
More general isotopies, which are not required to keep the surface in the class of surfaces with corners,
are referred to as $C^0$-isotopies.
`An isotopy in the class of Giroux's convex surfaces' is understood as
stated in~\cite[Definition~25]{dp17}.

If~$X$ is an arc or a surface we use the notation~$\interior(X)$ for~$X\setminus\partial X$.

If~$\Pi$ is a rectangular diagram of a surface (see~\cite[Definition~1]{dp17}, by
\emph{a vertex of~$\Pi$} we call a vertex of any rectangle~$r\in\Pi$, and by
\emph{an occupied level} of~$\Pi$ we mean a meridian~$\{\theta\}\times\mathbb S^1$
or a longitude~$\mathbb S^1\times\{\varphi\}$ of~$\mathbb T^2=\mathbb S^1\times\mathbb S^1$
that contains a vertex of~$\Pi$ (see Subsection~\ref{notation-subsec} for details).

The intersection points of the associated surface~$\widehat\Pi$ (see~\cite[Definition~9]{dp17})
with the circles~$\mathbb S^1_{\tau=1}$ and~$\mathbb S^1_{\tau=0}$ are called \emph{the vertices} of~$\widehat\Pi$,
and the sides of the tiles in~$\widehat\Pi$ \emph{the edges} of~$\widehat\Pi$.

Throughout the paper~$\mathbb S^1$ stands for an \emph{oriented} circle. If~$p,q\in\mathbb S^1$ are two distinct
points, we denote by~$[p;q]$ the arc~$\alpha\subset\mathbb S^1$ such that, if we view it as a $1$-chain, then~$\partial\alpha=q-p$.
Accordingly, $[p;q)$, $(p;q]$, and~$(p;q)$ denote $[p;q]\setminus\{q\}$, $[p;q]\setminus\{p\}$,
and~$[p;q]\setminus\{p,q\}$, respectively.

We use a similar notation for intervals of meridians and longitudes of~$\mathbb T^2$ (which are also oriented).
For instance, if~$p$ and~$q$ are two distinct points of a meridian~$m$, then~$(p;q)$ and~$(q;p)$ refer
to disjoint intervals of~$m$.

In order to avoid confusion, we denote by~$(x,y)$ a point with coordinates~$x$  and~$y$, and by~$(x;y)$, $[x;y]$, etc., the intervals
between~$x$ and~$y$.

\subsection{Organization of the paper}
In Section~\ref{equivalence-of-legendrian-links-sec} we describe the main application of
the machinery developed in this paper. As an illustration of the method,
we demonstrate how two Legendrian knots whose
inequivalence has been previously unknown can be distinguished. Some technical
details of the proof are placed in the Appendix~A at the end of the paper.

In Section~\ref{basic-moves-sec} we introduce certain transformations, called basic moves,
of rectangular diagrams of surfaces, and discuss
their properties. The discussion of one of the basic moves, called a flype, is
postponed till Appendix~B since it is not involved in the proof of the main result.

Section~\ref{spatial-ribbon-graph-sec} consists mostly of definitions of various objects related to spatial
ribbon graphs. One of the ideas behind our approach is that dealing with ribbon graphs, which are essentially
one-dimensional objects, is much easier than dealing with surfaces, whereas ribbon graphs may carry enough information
about the surfaces that we want to study.

Spatial ribbon graphs are conveniently represented by yet another kind of rectangular diagrams,
which we call mirror diagrams. These are introduced in Section~\ref{mir-diagr-sec}.
The moves that generate important equivalence relations between mirror diagrams
are defined in Section~\ref{elementary-moves-sec}. In Section~\ref{neat-decomp-sec}
further moves are introduced, and various relations between them are established.

In section~\ref{invariance-sec} we discuss the connection between
contact topology and mirror diagrams in a more detail. Though this connection is the actual origin
of the ideas behind the key technical result of this paper, which is the relative
commutation theorem, its formalization
appears to us more difficult than presenting the proof in purely combinatorial terms. The latter is done
in Section~\ref{commutation-sec}, where the formulation of the commutation
theorems is also given.

\section{A test for equivalence of Legendrian links}\label{equivalence-of-legendrian-links-sec}

\subsection{Dividing configurations}

\begin{defi}\label{abstract-dividing-def}
Let $F$ be a compact surface (with or without boundary, not necessarily orientable and connected). By \emph{an abstract dividing set on $F$}
we mean an oriented one-dimensional submanifold $\delta\subset F$ such that $\delta\cap\partial F=\partial\delta$ and the following
holds: if $d\subset F$ is a closed embedded disc such that $d\cap\delta\subset\partial d$, then~$\partial d$
admits an orientation that agrees with the orientation of~$\delta$
on every arc in~$\partial d\cap\delta$. Additionally, each
non-orientable connected component of~$F$ is required to have a non-empty intersection with~$\delta$.
\end{defi}

Another characterization of an abstract dividing set~$\delta\subset F$ is as follows. Let~$F'$ be the compact surface obtained from~$F$
by cutting along~$\delta$ (so that~$\delta$ is doubled), and~$\pi:F'\rightarrow F$ be the natural projection.
Then~$\delta$ is an abstract dividing set if and only if $F'$
admits an orientation such that the induced orientation of~$\partial F'$ coincides on~$\pi^{-1}(\delta)$ with the one inherited from~$\delta$.

\begin{defi}
An ordered pair $D=(\delta^+,\delta^-)$ of abstract dividing sets on $F$ is called \emph{a dividing configuration on $F$}. A
dividing configuration~$(\delta^+,\delta^-)$ is called
\emph{admissible} if the following holds:
\begin{enumerate}
\item
$\delta^+$ and $\delta^-$ are transverse to one another;
\item
each connected component of~$\delta^+$ non-trivially intersects~$\delta^-$, and vice versa;
\item
the intersection $\delta^+\cap\delta^-$ is disjoint from $\partial F$;
\item
each connected component of $F\setminus(\delta^+\cup\delta^-)$ is
contractible and has either empty or contractible intersection with $\partial F$.
\end{enumerate}

Dividing configurations (or abstract dividing sets) related by a self-homeomorphism of $F$ homotopic to the identity
are regarded as \emph{equivalent}.

Two dividing configurations $(\delta_1^+,\delta_1^-)$, $(\delta_2^+,\delta_2^-)$ on $F$ are \emph{weakly equivalent}
if $\delta_1^+$ is equivalent to $\delta_2^+$ and~$\delta_1^-$ is equivalent to $\delta_2^-$, that is, if
there are two self-homeomorphisms of $F$ homotopic to the identity such that one of them takes $\delta_1^+$ to $\delta_2^+$,
and the other $\delta_1^-$ to $\delta_2^-$.
\end{defi}

\begin{lemm}\label{admissible-conf-exists-lem}
Let~$D=(\delta^+,\delta^-)$ be a dividing configuration on a compact surface~$F$.
An admissible dividing configuration on~$F$ weakly equivalent to~$D$
exists if and only if the following two conditions hold:
\begin{enumerate}
\item
any connected component of~$F$ has a non-empty intersection with both~$\delta^+$ and~$\delta^-$;
\item
any connected component of~$\partial F$ has a non-empty intersection with~$\delta^+\cup\delta^-$.
\end{enumerate}
\end{lemm}

We omit the easy proof.

The homeomorphism
class of a compact surface~$F$ endowed with an admissible dividing configuration~$D=(\delta^+,\delta^-)$
is a simple combinatorial object that can be encoded as follows. Number all the intersection points of~$\delta^+$ with~$\delta^-$.
For~$\gamma$ a connected component of~$\delta^+$ or~$\delta^-$, denote by~$s(\gamma)$
the sequence obtained by listing the numbers of all the points from~$\delta^+\cap\delta^-$ contained in~$\gamma$
in the order they follow on~$\gamma$. If $\gamma$ is a closed curve, we start
from any of the points, go around~$\gamma$ once, and finish by repeating the starting number.
Let~$\gamma_i^+$, $i=1,\ldots,k$, be the connected components of~$\delta^+$,
and~$\gamma_i^-$, $i=1,\ldots,l$, the connected components of~$\delta^-$.

One can see that the homeomorphism class of~$(F,D)$
can be recovered from the following data:
$$\Bigl(\bigl\{s(\gamma_1^+),\ldots,s(\gamma_k^+)\bigr\},\bigl\{s(\gamma_1^-),\ldots,s(\gamma_l^-)\bigr\}\Bigr),$$
which will be referred to as \emph{a dividing code of~$(F,D)$}.
There is an arbitrariness in the definition of a dividing code, and there are certain restrictions on the data
that can occur as a dividing code of a surface endowed with a dividing configuration. We need not discuss
these issues here.

Note, however, that the set of all possible dividing codes of an admissible dividing configuration is
determined by any single one. Two dividing codes are said to be \emph{isomorphic} if they
are associated with the same homeomorphism class of surfaces endowed with an admissible dividing configuration.

\begin{defi}\label{canonic-def}
Let $\Pi$ be a rectangular diagram of a surface (see~\cite[Section~2]{dp17} for
the definition). A dividing configuration $D=(\delta^+,\delta^-)$
on the associated surface $\widehat\Pi$ will be said to be \emph{canonic}
if the intersection of $\delta^+$ (respectively, $\delta^-$) with every tile $\widehat r$, $r\in\Pi$,
is an arc connecting the midpoints of two opposite sides of $\widehat r$, and the functions
$\theta$, $\varphi$ (respectively, $-\theta$, $\varphi$) are locally increasing on~$\delta^+$ (respectively, $\delta^-$);
see Figure~\ref{canonic}.
\begin{figure}[ht]
\includegraphics[scale=0.7]{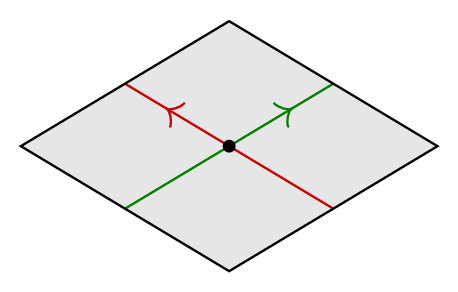}\put(-80,13){$\widehat r$}\put(-55,52){$\delta^+$}\put(-110,52){$\delta^-$}\hskip1cm
\includegraphics[scale=0.7]{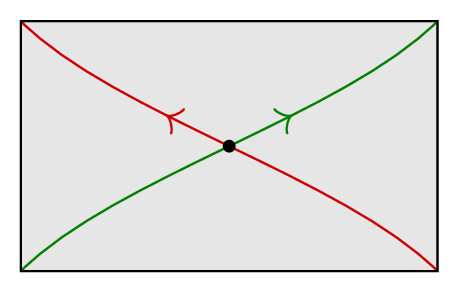}\put(-80,13){$r$}\put(-55,52){$\delta^+$}\put(-110,52){$\delta^-$}
\caption{A canonic dividing configuration in a single tile and its torus projection}\label{canonic}
\end{figure}
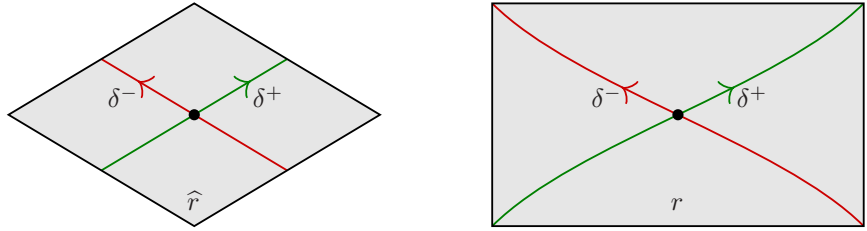

\end{defi}

This definition is motivated by the fact that any surface of the form $\widehat\Pi$ is convex in Giroux's sense
with respect to both contact structures $\xi_+$ and $\xi_-$,
and submanifolds~$\delta^\pm$ forming a canonic dividing configuration are suitable
for the respective dividing sets, see~\cite[Subsection~4.4]{dp17}.

Note that if~$(\delta^+,\delta^-)$ is a canonic dividing configuration of~$\widehat\Pi$, then
every tile of~$\widehat\Pi$ contains exactly one intersection point of~$\delta^+$ with~$\delta^-$.
Thus, there is a natural bijection~$\delta^+\cap\delta^-\leftrightarrow\Pi$.

A canonic dividing configuration of $\widehat\Pi$ is an object that
is dual to the tiling of~$\widehat\Pi$ in a natural sense. Namely, if $\widehat\Pi$ is a closed surface,
then its canonic dividing configuration makes up the~$1$-skeleton of a cell decomposition of~$\widehat\Pi$
dual to the tiling.

\begin{prop}
For any rectangular diagram of a surface~$\Pi$:

\emph{(i)} there exists a canonic dividing configuration of~$\widehat\Pi$;

\emph{(ii)} a canonic dividing configuration of~$\widehat\Pi$ is admissible;

\emph{(iii)} any two canonic dividing configurations of~$\widehat\Pi$
are equivalent.
\end{prop}

The proof is easy and left to the reader. We only remark that part~(i) of this proposition just says
that the oriented submanifolds~$\delta^\pm$ in Definition~\ref{canonic-def} satisfy
the orientation compatibility condition of Definition~\ref{abstract-dividing-def}.

\begin{defi}
Let  $D=(\delta^+,\delta^-)$ be a dividing configuration
on a compact surface $F$. By \emph{a realization of~$D$} we call
a pair~$(\Pi,\phi)$ in which $\Pi$ is a rectangular diagram of a surface, and~$\phi$ is
an embedding~$\phi:F\rightarrow\mathbb S^3$ such that $\phi(F)=\widehat\Pi$
and $\phi$ takes $D$ to a canonic dividing configuration on $\widehat\Pi$.

If $F\subset\mathbb S^3$ is
a compact surface embedded in $\mathbb S^3$, and $D$ is a dividing configuration on $F$,
then by \emph{a proper realization of $D$} we call a realization $(\Pi,\phi)$ such that
$\phi$ can be extended to an orientation-preserving self-homeomorphism of $\mathbb S^3$.
Such an extension will be assumed to be chosen for each proper realization and denoted
by the same letter as the embedding~$F\rightarrow\mathbb S^3$ being extended.
\end{defi}

All the useful information about a realization~$(\Pi,\phi)$ of~$D=(\delta^+,\delta^-)$ is actually contained in the diagram~$\Pi$
equipped with the induced correspondence between the rectangles of~$\Pi$ and the points from~$\delta^+\cap\delta^-$,
so we will often omit mentioning the corresponding embedding~$\phi$
and refer to the diagram~$\Pi$ as a realization of~$D$.

Suppose we are given a dividing code $\bigl(\{s_1^+,\ldots,s_{k}^+\},\{s_1^-,\ldots,s_l^-\}\bigr)$
of an admissible dividing configuration, and want to find its realizations.
If properness of the realization is not required, then this task amounts to finding
rectangular diagrams~$\Pi$ that admit a numeration of rectangles such
that every number assigned to a rectangle~$r\in\Pi$ appears
in the sequence $s_i^+$ for some~$i$ (and also in the sequence $s_j^-$ for some~$j$)
and the following holds:
\begin{enumerate}
\item the top right vertex of the $i$th rectangle of~$\Pi$
coincides with the bottom left vertex of the~$j$th one if and only if
$i$ is followed by~$j$
in one of the sequences~$s_m^+$, $m=1,\ldots,k$;
\item the top left vertex of the $i$th rectangle of~$\Pi$
coincides with the bottom right vertex of the~$j$th one if and only if
$i$ is followed by~$j$
in one of the sequences~$s_m^-$, $m=1,\ldots,l$.
\end{enumerate}

\begin{exam}\label{realizations-exam}
Shown in Figure~\ref{fig8-example} on the left is a Seifert surface
for the Figure Eight knot endowed with a dividing configuration~$D=(\delta^+,\delta^-)$, with~$\delta^+$ shown in green and~$\delta^-$ in red (see Convention~\ref{color-conv} below
for further clarification).
This configuration can be encoded as follows:
$$\{(1,2,3,4,5),(6),(7)\},\ \{(1,7,5,6,1),(2),(3),(4)\}.$$

The central picture in Figure~\ref{fig8-example} provides an example
of a proper realization of~$D$ (more precisely, it is a rectangular diagram
of a surface that gives rise to a proper realization of~$D$).
The intersection points of~$\delta^+$ and~$\delta^-$ and
the rectangles of the realization are numbered accordingly.

The right picture in Figure~\ref{fig8-example} provides an example of
a realization of~$D$ that is not proper. Indeed, one can see that
the boundary of the corresponding surface is unknotted.
\begin{figure}[ht]
\centerline{\includegraphics[width=120pt]{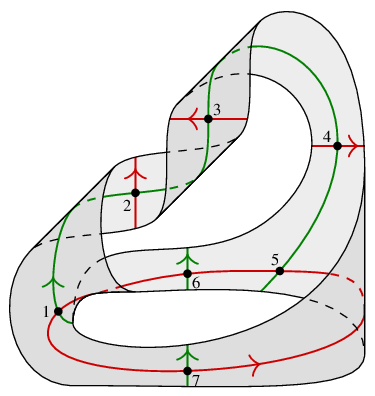}\hskip1cm\includegraphics[width=120pt]{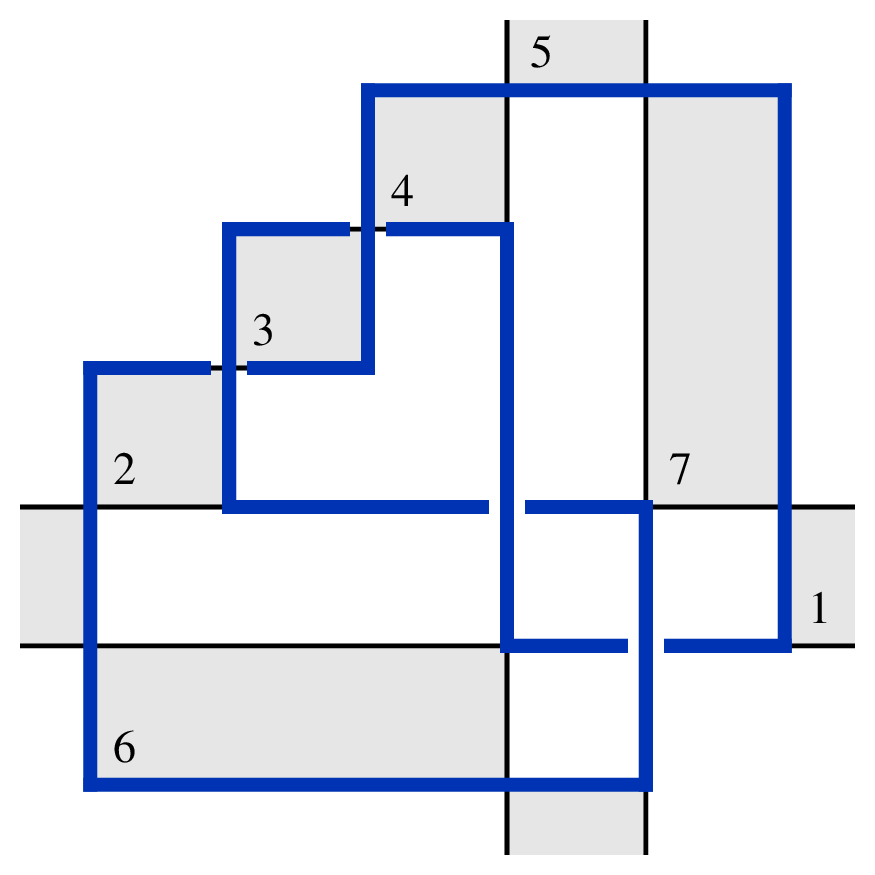}\hskip1cm\includegraphics[width=120pt]{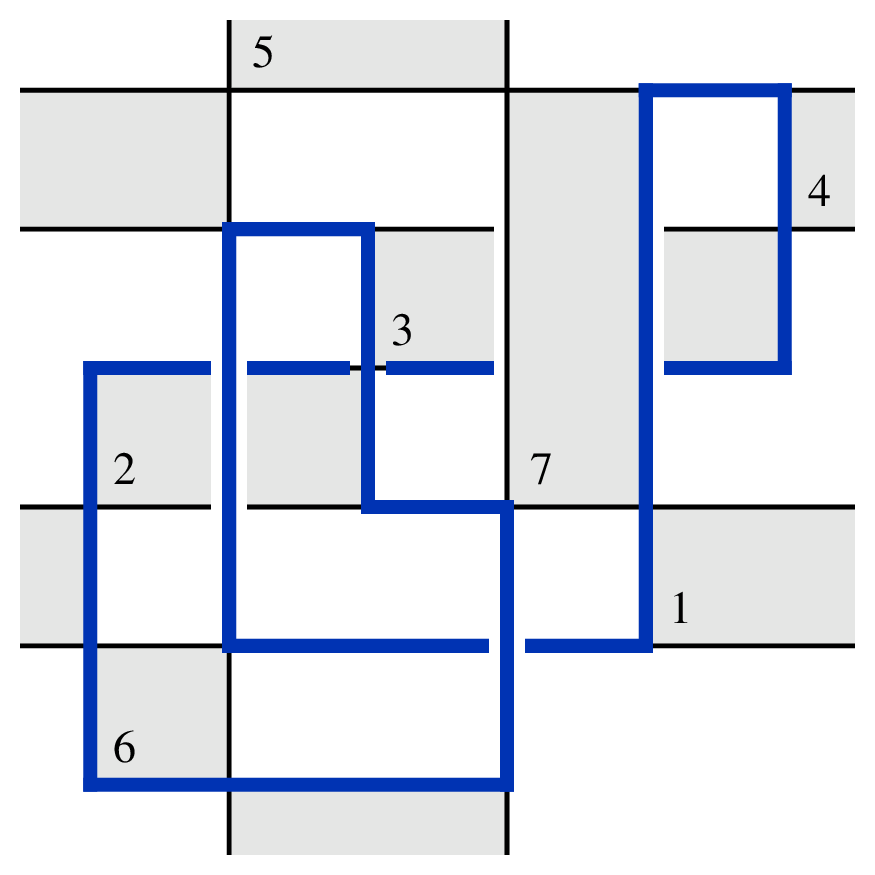}}
\caption{A surface with a dividing configuration (left), a proper realization of this configuration (center),
and a realization that is not proper (right)}\label{fig8-example}
\end{figure}
\end{exam}

\begin{conv}\label{color-conv}
The style in which the pictures in Figure~\ref{fig8-example} are drawn will be accepted throughout the paper. Namely,
if a picture is supposed to show a dividing configuration~$(\delta^+,\delta^-)$, then~$\delta^+$
will be shown in green and~$\delta^-$ in red.

A canonic dividing configuration~$(\delta^+,\delta^-)$ of a surface presented by a rectangular diagram is never
displayed at the diagram because it has a standard form in each rectangle (as shown in Figure~\ref{canonic})
and carries no additional information.
The rectangles are typically numbered, and the points of~$\delta^+\cap\delta^-$ are numbered
accordingly.

In our illustrations, rectangular diagrams of a surface will often be accompanied by a rectangular diagram of their boundary link
(as the two diagrams in Figure~\ref{fig8-example}). The latter will be drawn in the most natural way for human perception, which usually
corresponds neither to the boundary framing (see~\cite[Definitions~12--14 and Figure~12]{dp17}) nor to
the boundary circuits (see Definition~\ref{mirror-circuit-def} below). We draw the horizontal edges passing under all the
rectangles, and the vertical edges passing over all the rectangles
with exceptions made for the rectangles having a side that is contained in an edge. We find
this way of drawing our diagrams the most suggestive.
\end{conv}

\begin{defi}\label{comb-equiv-def}
Two rectangular diagrams of a surface (of a link, of a graph) are said to be
\emph{combinatorially equivalent} (or of the same \emph{combinatorial type}) if one can be taken to the other
by a self-homeomorphism of~$\mathbb T^2$ that has
the form~$(\theta,\varphi)\mapsto\bigl(f(\theta),g(\varphi)\bigr)$,
where~$f$ and~$g$ are orientation-preserving self-homeomorphisms of~$\mathbb S^1$.
\end{defi}

Example~\ref{realizations-exam} demonstrates that canonic dividing configurations of combinatorially non-equivalent diagrams
may have the same dividing code.
The following statement is obvious.

\begin{prop}\label{finite-prop}
If~$D$ is a canonic dividing configuration of a rectangular diagram~$\Pi$ of a surface, then
the isomorphism class of a dividing code of~$D$ is uniquely determined by
the combinatorial type of~$\Pi$.

For any~$n$, there are only finitely many combinatorial types of
rectangular diagrams of surfaces having exactly~$n$ rectangles.
\end{prop}

\subsection{An invariant of Legendrian links}

\begin{conv}
By a \emph{a link} in this paper, we mean a piecewise smooth link in $\mathbb S^3$ whose components are ordered.
Thus, by saying that a link $L=\cup_iL_i$ is taken to $L'=\cup_iL_i'$ by a homeomorphism $\phi$, where~$L_i$ and $L_i'$ are
connected components, $i=1,2,\ldots$, we mean that $\phi$ takes $L_i$ to $L_i'$ for each applicable~$i$. If the links $L$ and $L'$
are oriented (which will typically be the case), then by writing $\phi(L)=L'$ we assume that~$\phi$ preserves the orientation.

The same refers to rectangular diagrams of links: connected components are silently assumed to be ordered,
and the components of the associated link are assumed to be ordered respectively.
\end{conv}

\begin{defi}
Let $L$ be a fixed oriented link, and let $R$ be an oriented
rectangular diagram of a link such that~$L$ and $\widehat R$ are isotopic.
Let also $F\subset\mathbb S^3$ be a compact
surface such that $F\cap L$ is a sublink of~$\partial F$.
We say that~$F$ is \emph{$+$-compatible} (respectively, \emph{$-$-compatible}) \emph{with $R$} if for some (and then any)
orientation-preserving homeomorphism $\phi:\mathbb S^3\rightarrow\mathbb S^3$
taking $L$ to $\widehat R$ we have
$\tb_+(K,\phi(F))\leqslant0$ (respectively, $\tb_-(K,\phi(F))\leqslant0$) for any connected component $K$ of
$\phi(L\cap\partial F)$.
\end{defi}

This definition is motivated by Theorem~1 of~\cite{dp17}, which states, in the present terms, that the $C^0$-isotopy class of a surface~$F$
can be represented by a rectangular diagram so that a link~$L$ related to~$F$ as above
is simultaneously represented by a given rectangular diagram of a link~$R$, if and only if~$F$ is both $+$-compatible and $-$-compatible
with~$R$.

\begin{defi}
Let $L\subset\mathbb S^3$ be a fixed oriented link, and let $F\subset\mathbb S^3$ be a fixed compact surface
such that $F\cap L$ is a sublink of $\partial F$. Let also $R$ be an oriented rectangular diagram of a link.
We say that an abstract dividing set $\delta\subset F$ is \emph{properly $+$-realizable}
(respectively, \emph{properly $-$-realizable}) \emph{at $R$}
if there exists an abstract dividing set $\delta'\subset F$ and a proper realization $(\Pi,\phi)$
of $(\delta,\delta')$ (respectively, of $(\delta',\delta)$) such
that~$\phi(L)=\widehat R$.
In this case we also say that~$(\Pi,\phi)$  is \emph{a proper $+$-realization of~$\delta$}
(respectively, \emph{a proper $-$-realization of~$\delta$}) \emph{at}~$R$.

The set of equivalence classes of all properly $+$-realizable (respectively, properly $-$-realizable) abstract dividing sets
at~$R$ will be denoted by $\mathscr I_{F,L,+}(R)$ (respectively, $\mathscr I_{F,L,-}(R)$).
\end{defi}

\begin{theo}\label{invariant-thm}
Let $L\subset\mathbb S^3$ be a fixed oriented link, and let~$F\subset\mathbb S^3$ be a fixed compact surface
such that $F\cap L$ is a sublink of~$\partial F$. Let also $R$, $R'$ be oriented rectangular diagrams of a link such
that the links~$\widehat R$ and $\widehat R'$ are equivalent to $L$, and
$F$ is $-$-compatible with both of them.
If $\widehat R$ and $\widehat R'$ are equivalent as Legendrian links, then
$\mathscr I_{F,L,+}(R)=\mathscr I_{F,L,+}(R')$.

If one replaces the standard contact structure $\xi_+$ by its mirror image $\xi_-$
in the definition of a Legendrian link, then the above statement holds for $\mathscr I_{F,L,-}$
in place of $\mathscr I_{F,L,+}$ and $+$-compatibility in place of $-$-compatibility.
\end{theo}

This statement is a consequence of the representability result of~\cite{dp17}, properties of
dividing sets of Giroux's convex surfaces \cite{Gi}, and the fact
that, for a rectangular diagram of a surface~$\Pi$, the associated surface~$\widehat\Pi$
is convex in Giroux's sense with respect to both contact structures $\xi_\pm$, and a pair of dividing sets of
$\widehat\Pi$ with respect to $\xi_+$ and $\xi_-$ forms a canonic dividing configuration on $\widehat\Pi$.

A proof can also been given in a more combinatorial manner, which will be done in Section~\ref{invariance-sec}.

\subsection{Distinguishing Legendrian links}
In view of Theorem~\ref{invariant-thm}, to distinguish Legendrian links presented by oriented rectangular diagrams $R_1$ and $R_2$, say,
with~$\widehat R_1$ and~$\widehat R_2$ isotopic to a fixed oriented link~$L$,
it suffices to show that $\delta\in\mathscr I_{F,L,+}(R_1)$ and $\delta\notin\mathscr I_{F,L,+}(R_2)$ hold
for some surface $F\subset\mathbb S^3$ such that
$F\cap L$ is a sublink of $\partial F$ and $F$ is $-$-compatible with~$R_1$, $R_2$,
and an abstract dividing set $\delta$ on~$F$. Clearly, the claim $\delta\in\mathscr I_{F,L,+}(R)$, if true,
can be confirmed by producing a realization explicitly. The main achievement of this paper is a method for proving, under certain circumstances,
claims of the opposite kind,
$\delta\notin\mathscr I_{F,L,+}(R)$. The method is based on the following statement.

\begin{theo}\label{maintheo}
Let $L\subset\mathbb S^3$ be a fixed oriented link, $F\subset\mathbb S^3$ a fixed compact surface
such that $F\cap L$ is a sublink of~$\partial F$, and let $R$ be an oriented rectangular diagram of a link.
Let also $\delta_0$ and $\delta_1$ be two abstract dividing sets on~$F$. Assume that
there exist a proper $+$-realization $(\Pi_0,\phi_0)$ of $\delta_0$
and a proper $-$-realization $(\Pi_1,\phi_1)$ of $\delta_1$ at~$R$.
Assume also that there exists a $C^0$-isotopy $\phi_t$, $t\in[0;1]$, from $\phi_0$ to $\phi_1$
fixed at~$L$.

Then there exists an admissible dividing configuration $(\delta^+,\delta^-)$ weakly equivalent to $(\delta_0,\delta_1)$,
and for any such dividing configuration
there exist a proper realization $(\Pi,\xi)$ of $(\delta^+,\delta^-)$ and a rectangular diagram of a link~$R'$
such that~$\xi(L)=\widehat R'$ and~$R'$ is
obtained from~$R$ by a finite sequence of exchange moves.
\end{theo}

\begin{proof}
We use here many intermediate results that are distributed over the Sections~\ref{basic-moves-sec}--\ref{commutation-sec}.
We also use terminology and notation introduced in those sections. So, logically this proof belongs to
the end of Section~\ref{commutation-sec}. However, we place it here in belief that
the reader will profit from seeing the general idea and the structure of the proof before falling into details (some of which may be
pretty boring).

By the assumption of the theorem $\Pi_0$ and $\Pi_1$ are rectangular diagrams of a surface such that
$\phi_i(F)=\widehat\Pi_i$, $i=0,1$. The embedding $\phi_0$ realizes the dividing configuration $(\delta_0,\delta_0')$
with some $\delta_0'$, and $\phi_1$ realizes $(\delta_1',\delta_1)$ with some $\delta_1'$.
We give the proof in eight steps, of which the most tricky is Step~6.

\medskip\noindent\emph{Step 1.} Prove the existence of an admissible dividing configuration~$(\delta^+,\delta^-)$
weakly equivalent to~$(\delta_0,\delta_1)$:\\
Let~$\gamma$ be a connected component of~$\partial F$, and let~$Q_0$ and~$Q_1$ be the components of~$\partial\Pi_0$ and~$\partial\Pi_1$,
respectively, corresponding to~$\gamma$, that is,
such that~$\widehat Q_0=\phi_0(\gamma)$ and~$\widehat Q_1=\phi_1(\gamma)$. Then the number of points
in~$\gamma\cap\delta_0$ (respectively, in~$\gamma\cap\delta_1$) is equal to~$-2\tb_+(Q_0;\Pi_0)$
(respectively, to~$-2\tb_-(Q_1;\Pi_1)$)\footnote{To simplify notation,
we define $\tb_\pm(Q;\Pi)$ as~$\tb_\pm\bigl(\widehat Q;\widehat\Pi\bigr)$.}. We also have
$$\tb_+(Q_0;\Pi_0)+\tb_-(Q_1;\Pi_1)=\tb_+(Q_0)-\lk\bigl(\gamma,\gamma^F\bigr)+\tb_-(Q_1)+\lk\bigl(\gamma,\gamma^F\bigr)=
\tb_+(Q_0)+\tb_-(Q_1).$$
The latter sum must be negative, since, by~\cite[Theorem~7]{DyPr}, there exists a rectangular diagram of a knot~$Q$ such that~$\bigl(\tb_+(Q),\tb_-(Q)\bigr)=\bigl(\tb_+(Q_0),\tb_-(Q_1)\bigr)$.
The number of vertices in~$Q$ equals~$-2\bigl(\tb_+(Q)+\tb_-(Q)\bigr)$,
hence~$\tb_+(Q)+\tb_-(Q)<0$.

Thus, for any connected component~$\gamma$ of~$\partial F$, at least one of the sets~$\gamma\cap\delta_0$ and~$\gamma\cap\delta_1$ is not empty. The intersection of any connected component of~$F$ with both~$\delta_0$ and~$\delta_1$ is also not empty,
since, by the assumption of the theorem, the abstract dividing sets~$\delta_0$ and~$\delta_1$ have realizations ($+$- or $-$-).
By~Lemma~\ref{admissible-conf-exists-lem} there exists an admissible dividing configuration~$(\delta^+,\delta^-)$
weakly equivalent to~$(\delta_0,\delta_1)$.

\medskip\noindent\emph{Step 2.} Reduce to the case when~$L\subset\partial F$:\\
Let~$\gamma$ be a connected component of~$L\setminus\partial F$, and let~$Q$ be the connected component of~$R$ corresponding to~$\gamma$. Let also
$$(\theta_1,\varphi_1),(\theta_1,\varphi_2),(\theta_2,\varphi_2),\ldots,(\theta_m,\varphi_m),(\theta_m,\varphi_1)$$
be the vertices of~$Q$ listed in the order they follow on~$Q$. For~$\varepsilon>0$ denote by~$\Pi^\varepsilon$
the following collection of rectangles:
$$\Pi^\varepsilon=\bigl\{[\theta_i;\theta_i+\varepsilon]\times[\varphi_i+\varepsilon;\varphi_{i+1}],
[\theta_i+\varepsilon;\theta_{i+1}]\times[\varphi_{i+1};\varphi_{i+1}+\varepsilon]\bigr\}_{i=1,\ldots,m},$$
where indices are regarded modulo~$m$. If~$\varepsilon$ is small enough, then~$\Pi^\varepsilon$
is a rectangular diagram of a surface, and the associated surface~$\widehat\Pi^\varepsilon$ is an annulus
disjoint from~$\widehat\Pi_0$ and~$\widehat\Pi_1$. Fix such an~$\varepsilon$ from now on.

We can find an annulus~$A\subset\mathbb S^3$
disjoint from~$F$ such that~$\gamma\subset\partial A$ and the linking number of the connected components
of~$\partial A$ is the same as that of the connected components of~$\partial\widehat\Pi^\varepsilon$. 
(All such annuli are pairwise $C^0$-isotopic relative to~$F\cup L$, so
the choice does not matter.)

We can extend~$\phi_0$ and~$\phi_1$ to~$A$ so that~$\phi_0(A)=\phi_1(A)=\widehat\Pi^\varepsilon$.
Then replace~$F$, $\Pi_0$, and~$\Pi_1$ by~$F\cup A$, $\Pi_0\cup\Pi^\varepsilon$, and~$\Pi_1\cup\Pi^\varepsilon$,
respectively. This reduces the number of connected components of~$L$ not contained in~$\partial F$,
whereas all assumptions of the theorem still hold after extending~$\delta_0$ and~$\delta_1$ to~$A$ accordingly.

Thus, we may assume from now on that~$L\subset\partial F$.

\medskip\noindent\emph{Step 3.} Reduce to the case when the isotopy~$\phi$ complies with the conventions introduced
in Subsection~\ref{general-conventions}, and for any point~$p\in\widehat R$,
the tangent plane to~$\phi_t(F)$ at~$p$ remains fixed during the isotopy:\\
There is no problem to make $\phi$ $C^1$-smooth outside~$\partial F$ and also at all points~$x\in\partial F$ such that~$\phi_t(x)$
is fixed and is not a singularity of~$\partial(\phi_t(F))$. This smoothness is assumed in the sequel.

Suppose that there is a connected component~$Q$ of~$R$ such that the tangent plane to~$\phi_t(F)$ at~$p\in\widehat Q\setminus\bigl(\mathbb
S^1_{\tau=0}\cup\mathbb S^1_{\tau=1}\bigr)$ varies when~$t$ runs from~$0$ to~$1$.
Suppose also that~$\tb_+(Q;\Pi_0)<0$ (observe that $\tb_+(Q;\Pi_0)=\tb_+(Q;\Pi_1)$).
It follows from Proposition~\ref{twist-prop} that
we can adjust~$\Pi_0$ and~$\phi$ so that:
\begin{enumerate}
\item
the assumptions of the theorem still hold;
\item
$\widehat\Pi_0$ is altered only slightly (in terms of the $C^0$-topology);
\item
for any connected component~$Q'\ne Q$ of~$R$, any point~$p\in\widehat Q'\setminus\bigl(\mathbb S^1_{\tau=0}\cup\mathbb
S^1_{\tau=1}\bigr)$, and~$t\in[0,1]$, the tangent plane to~$\phi_t(F)$ at~$p$ is unaltered;
\item
the tangent plane to~$\phi_t(F)$ at any point~$p\in\widehat Q$ is defined and does not depend on~$t$ after the adjustment.
\end{enumerate}
If~$\tb_-(Q;\Pi_0)<0$ we can similarly adjust~$\Pi_1$ instead of~$\Pi_0$.

At least one of the inequalities~$\tb_+(Q;\Pi_0)<0$
and~$\tb_-(Q;\Pi_0)<0$ must hold, since~$2(-\tb_+(Q;\Pi_0)-\tb_-(Q;\Pi_0))$ is the number of vertices
in~$Q$. Therefore, after doing finitely many such adjustments we come to
the situation in which the tangent plane to~$\phi_t(F)$ at any point~$p\in\widehat R$ does not depend on~$t$,
and~$\phi$ is an isotopy complying with the conventions introduced in Subsection~\ref{general-conventions}.
We assume this from now on.

\medskip\noindent\emph{Step 4.} Make holes (switch to mirror diagrams):\\
Let~$M=M(\Pi_0)$ and~$M'=M(\Pi_1)$ be the enhanced mirror diagrams associated
with~$\Pi_0$ and~$\Pi_1$, respectively (see Section~\ref{mir-diagr-sec}),
and let~$\eta$ be the morphism of the respective spatial ribbon graphs~$\widehat M\rightarrow\widehat M'$
induced by the isotopy~$\phi$ (see Subsection~\ref{morphisms-subsec} for the definition of a morphism),
that is, the one defined by~$\bigl(\widehat\Pi_0,\widehat\Pi_1,\phi_1\circ\phi_0^{-1}\bigr)\in\eta$.

The diagrams~$M$ and~$M'$ have a common simple collection~$C$ of boundary circuits that represents
the framed rectangular diagram of a link~$(R,f)$, where~$f$ is the framing opposite to~$f^{\Pi_0}|_R=f^{\Pi_1}|_R$.

\medskip\noindent\emph{Step 5.} Find a sequence of moves representing the isotopy:\\
By Theorem~\ref{relative-stable-equivalence-th} there exists a sequence of elementary moves preserving
all boundary circuits in~$C$
\begin{equation}\label{sequence}
M=M_0\xmapsto{\eta_1}M_1\xmapsto{\eta_2}M_2\xmapsto{\eta_3}\ldots\xmapsto{\eta_N}M_N=M'
\end{equation}
such that~$\eta=\eta_N\circ\ldots\circ\eta_2\circ\eta_1$.

\medskip\noindent\emph{Step 6.} Rearrange the moves using the relative commutation theorem:\\
By Proposition~\ref{obvious-prop-1} the enhanced mirror diagrams~$M$ and~$M'$
have simple essential boundary. Since~$C\subset\partial_{\mathrm e}M\cap\partial_{\mathrm e}M'$,
it follows from Proposition~\ref{simple-essential-boundary-implies-flexibility-prop} that~$M$
and~$M'$ are flexible relative to~$C$.

Now by Theorem~\ref{commutation-2-thm} the sequence~\eqref{sequence} can be 
modified so that, after the modification, it will consist of elementary moves together with
some number of jump moves, and remain $C$-delicate
(see Definition~\ref{c-delicate-def}), and for some $k\in[0,N]$
the $C$-delicate subsequences~$M_0\xmapsto{\eta_1}M_1\xmapsto{\eta_2}M_2\xmapsto{\eta_3}\ldots\xmapsto{\eta_k}M_k$
and~$M_k\xmapsto{\eta_{k+1}}M_{k+1}\xmapsto{\eta_{k+2}}M_{k+2}\xmapsto{\eta_{k+3}}\ldots\xmapsto{\eta_N}M_N$ will be
of type~I  and type~II, respectively. We fix these sequence and number~$k$ from now on
and denote the morphisms~$\eta_k\circ\ldots\circ\eta_2\circ\eta_1:\widehat M\rightarrow\widehat M_k$ and
$\eta_N\circ\ldots\eta_{k+2}\circ\eta_{k+1}:\widehat M_k\rightarrow M'$ by~$\eta_{\mathrm I}$ and
$\eta_{\mathrm{II}}$, respectively.

\medskip\noindent\emph{Step 7.} Patch the holes:\\
Let~$S$ be a surface carried by~$\widehat M_k$
such that some isotopy inducing the morphism~$\eta_{\mathrm I}$
brings~$\widehat\Pi_0$ to~$S$, and let~$\psi$ be a homeomorphism~$\widehat\Pi_0\rightarrow S$
such that~$(\widehat\Pi_0,S,\psi)\in\eta_{\mathrm I}$. Clearly we also have~$(\widehat\Pi_1,
S,\psi\circ\phi_0\circ\phi_1^{-1})\in\eta_{\mathrm{II}}$.

By Lemma~\ref{rectangular-representative-lem} the surface~$S$ can be chosen in
the form~$\widehat\Pi$ for some rectangular diagram of a surface~$\Pi$.
We assume for the rest of the proof that such a choice has been made.
Since the sequence~\eqref{sequence} is~$C$-delicate
we may also assume that~$\psi$ takes~$\widehat R$ to a union of
essential boundary circuits of~$\widehat M_k$ that has the form~$\widehat R'$,
where~$R'$ is a rectangular diagram of a link.
Moreover, it follows that such~$R'$ (viewed combinatorially) is obtained from~$R$ by a sequence of exchange moves.

 Let~$(\delta_+,\delta_-)$ be a canonic
dividing configuration of~$\widehat\Pi$.
Since the enhanced mirror diagrams~$M$ and~$M'$ are associated with rectangular diagrams of surfaces,
all inessential boundary circuits of~$M$ and~$M'$ are simultaneously
$+$-negligible and $-$-negligible (see Definition~\ref{neglibigle-circuit-def}).
It now follows from Lemma~\ref{realization-invariance-lemm}
that the abstract dividing set~$(\psi\circ\phi_0)(\delta_0)$ is equivalent to~$\delta_+$, and, by symmetry,~$(\psi\circ\phi_0)(\delta_1)$
is equivalent to~$\delta_-$. In other words, $(\Pi,\psi\circ\phi_0)$ is a realization of the dividing configuration~$\bigl((\psi\circ\varphi_0)^{-1}(\delta_+),
(\psi\circ\varphi_0)^{-1}(\delta_-)\bigr)$, which is
weakly equivalent to~$(\delta_0,\delta_1)$.
Since~$S=\widehat\Pi$ has been chosen isotopic to~$\widehat\Pi_0$, with
the isotopy realizing~$\eta_{\mathrm I}$, this realization is proper.
We also have~$\psi(\phi_0(L))=\widehat R'$.

\medskip\noindent\emph{Step 8.} Produce realizations for all admissible dividing configurations weakly equivalent to~$(\delta_0,\delta_1)$
from a single one:\\
By this point, we have shown that for \emph{some} dividing configuration on~$F$ weakly equivalent
to~$(\delta_0,\delta_1)$, there exists a proper realization~$(\Pi,\xi)$ such that~$\xi(L)$ has the form~$\widehat R'$,
where~$R'$ is obtained from~$R$ by a sequence of exchange moves.
It follows from Proposition~\ref{weak-equivalence-via-moves} and Lemma~\ref{exchange-means-exchange-lem}
that the same holds for \emph{any} admissible
dividing configuration on~$F$ weakly equivalent to~$(\delta_0,\delta_1)$, which concludes the proof of the theorem.
\end{proof}

Theorem~\ref{maintheo} immediately implies the following.

\begin{coro}\label{maincoro}
Let $L$ be a fixed oriented link, and let $F\subset\mathbb S^3$ be a fixed compact surface such that $F\cap L$ is a sublink of $\partial F$.
Let also $R$ be an oriented rectangular diagram of a link such that $\widehat R$ is isotopic to~$L$. Finally, let~$A$ be
a family of abstract dividing sets on~$F$ such that, for any orientation-preserving self-homeomorphism~$\phi$ of~$\mathbb S^3$
taking~$L$ to~$\widehat R$, there exist~$\delta'\in A$ and a proper $-$-realization~$(\Pi,\phi')$ of $\delta'$ at~$R$
with~$\phi'$ $C^0$-isotopic to~$\phi$ relative to~$L$.

Then, for an abstract dividing set $\delta\subset F$, we have $\delta\in\mathscr I_{F,L,+}(R)$
if and only if, for some $\delta'\in A$, the following two conditions hold:
\begin{enumerate}
\item
there exists an admissible dividing configuration~$(\delta^+,\delta^-)$ weakly equivalent to $(\delta,\delta')$;
\item
whenever~$(\delta^+,\delta^-)$ is such a configuration, there exist a rectangular diagram of a link $R'$ obtained from $R$
by a finite sequence of exchange moves, and a proper realization $(\Pi',\phi)$ of $(\delta^+,\delta^-)$
such that $\phi(L)=\widehat R'$.
\end{enumerate}
\end{coro}

\begin{defi}
A family~$A$ of abstract dividing sets satisfying
the assumptions of Corollary~\ref{maincoro} will be called \emph{$-$-representative
for~$R$}. One similarly defines a $+$-representative family, by applying
the symmetry~$\xi_+\leftrightarrow\xi_-$.
\end{defi}

Corollary~\ref{maincoro} is most useful when one can find a \emph{finite} $-$-representative family $A$ for~$R$.
In this case, it is a finite procedure to check whether or not $\delta\in\mathscr I_{F,L,+}(R)$.
Indeed, for each $\delta'\in A$ it is easy to find an admissible dividing configuration
$(\delta^+,\delta^-)$ weakly equivalent to $(\delta,\delta')$, and, due to Proposition~\ref{finite-prop}, each
dividing configuration admits, up to combinatorial equivalence,
only finitely many realizations, which can be searched.

For each found realization, one should, of course, check whether it is proper or not.
This amounts to comparing certain Haken three-manifolds with boundary pattern,
which is doable in general~\cite{mat} but sometimes is simply not needed because no realization
exists.

\subsection{An example: the knot~$6_2$}
We illustrate how the method described above works by proving the following statement.

\begin{prop}\label{6-2-prop}
The following two Legendrian knots, which have topological type $6_2$, presented by
front projections are not equivalent.

\nobreak
\vskip0.2cm
\centerline{\includegraphics[scale=0.48]{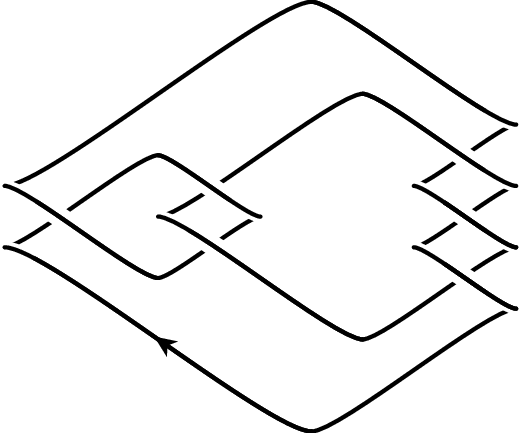}\put(0,0){$K_1$}
\hskip2cm
\includegraphics[scale=0.43]{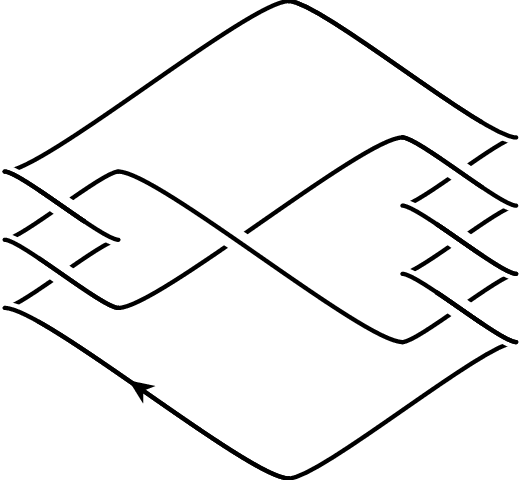}\put(0,0){$K_2$}}
\end{prop}

\begin{proof}
Shown in Figure~\ref{r1r2}
\begin{figure}[ht]
\begin{tabular}{ccc}
\includegraphics[scale=0.4]{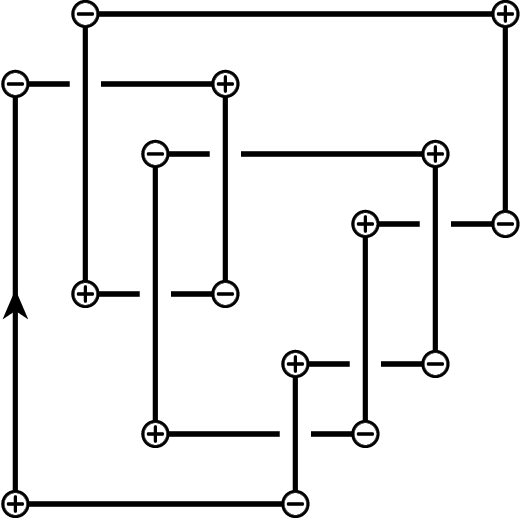}&\hbox to1cm{}&\includegraphics[scale=0.4]{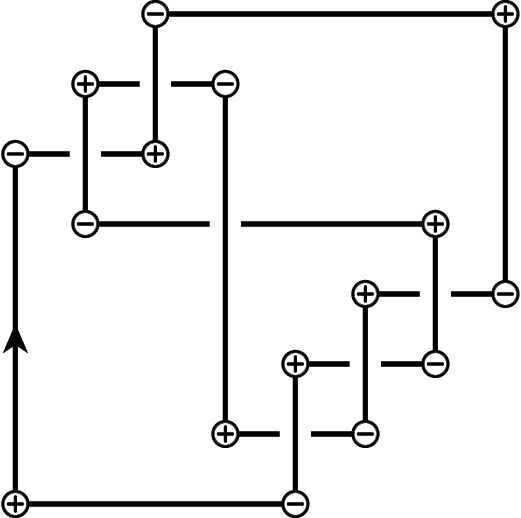}\\
$R_1$&&$R_2$
\end{tabular}
\caption{Rectangular presentations of $K_1$ and $K_2$}\label{r1r2}
\end{figure}
are oriented rectangular diagrams $R_1$ and $R_2$ such that
$\widehat R_i$ is equivalent to $K_i$ as a $\xi_+$-Legendrian knot, $i=1,2$.
The knots $\widehat R_1$ and $\widehat R_2$ are equivalent as $\xi_-$-Legendrian knots.
Indeed, there is a sequence of moves including exchange moves
and type~II stabilizations and destabilizations transforming $R_1$ to $R_2$. It is sketched in Figure~\ref{transform}.
\begin{figure}[ht]
\includegraphics[scale=0.4]{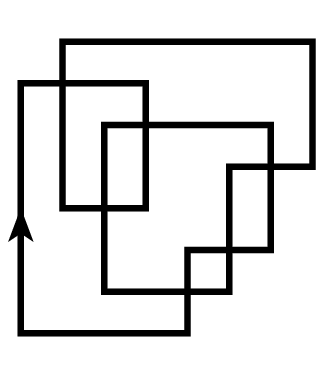}\raisebox{34pt}{$\rightarrow$}%
\includegraphics[scale=0.4]{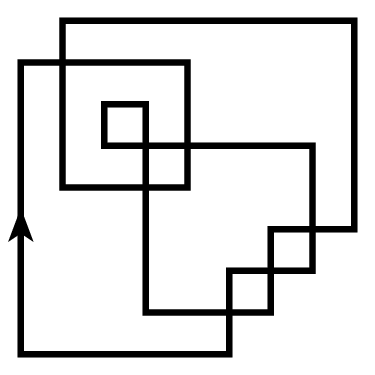}\raisebox{34pt}{$\rightarrow$}%
\includegraphics[scale=0.4]{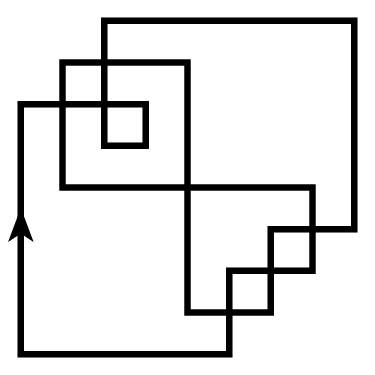}\raisebox{34pt}{$\rightarrow$}%
\includegraphics[scale=0.4]{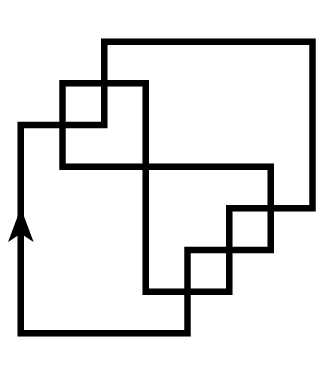}
\caption{Transforming $R_1$ to $R_2$ by exchange moves and type~II (de)stabilizations}\label{transform}
\end{figure}

One easily finds that $\tb_+(K_1)=\tb_+(K_2)=-7$, so, if $F$ is any Seifert surface for
$\widehat R_i$, $i=1,2$, then we have $\tb_+(\widehat R_i;F)=-7<0$.
Therefore, for any oriented knot $K$ having topological type $6_2$, and
any Seifert surface $F$ for $K$, we have by Theorem~\ref{invariant-thm}
$\mathscr I_{F,K,-}(R_1)=\mathscr I_{F,K,-}(R_2)$.

We are going to show
that $\mathscr I_{F,L,+}(R_1)\ne\mathscr I_{F,L,+}(R_2)$ for a specific choice
of $F$ and $L=\partial F$. This specific choice is presented in Figure~\ref{seifert} in the rectangular form.
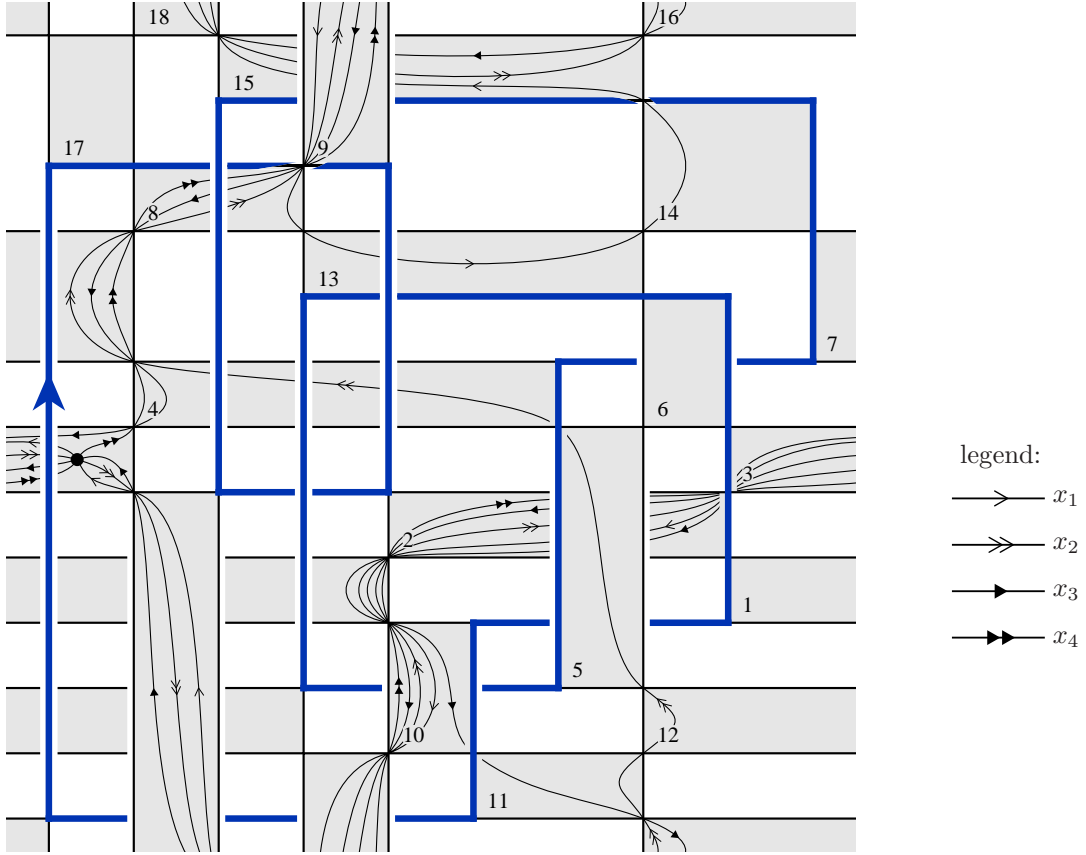
\begin{figure}[ht]
\includegraphics[scale=0.8]{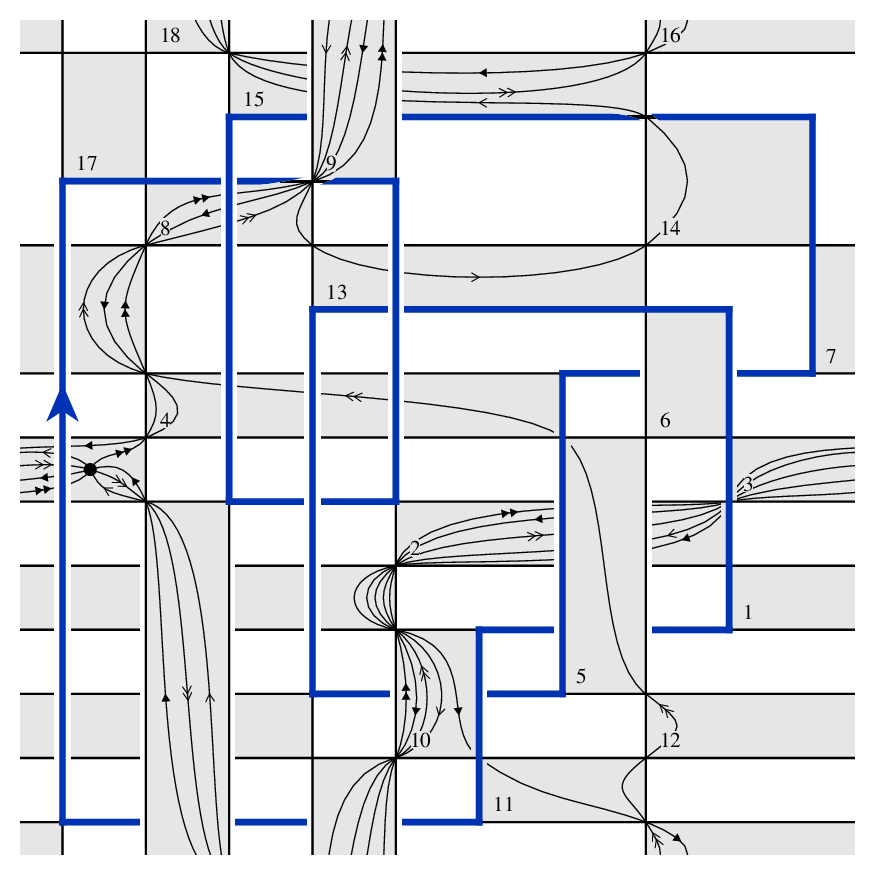}\hskip1cm\raisebox{80pt}{\includegraphics[scale=0.7]{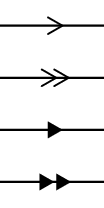}\put(-32,75){legend:}%
\put(3,60){$x_1$}\put(3,42.3){$x_2$}\put(3,24.6){$x_3$}\put(3,6.9){$x_4$}}
\caption{Rectangular diagram~$\Pi^*$ of a Seifert surface for $\widehat R_1$}\label{seifert}
\end{figure}
What is shown in Figure~\ref{seifert} is a rectangular diagram of a surface, which we denote by~$\Pi^*$,
together with the rectangular diagram of the boundary link, and some additional data.

It is a direct check that:
\begin{enumerate}
\item
the boundary of $\Pi^*$ coincides---combinatorially---with $R_1$ (provided that
the orientation of $\partial\Pi^*$ is chosen as in Figure~\ref{seifert});
\item
$\widehat\Pi$ is a genus two orientable surface with a single boundary component
(there are 18 rectangles in $\Pi^*$, which correspond
to the tiles of~$\widehat\Pi^*$, 44 vertices of~$\Pi^*$, which correspond to the sides of the tiles,
and 23 occupied levels, which correspond to the vertices of the tiles,
thus the Euler characteristic is $18-44+23=-3$);
\item
the loops on $\widehat\Pi^*$ whose torus projections are indicated in Figure~\ref{seifert} by curved lines with various arrowheads
generate the fundamental group of $\widehat\Pi^*$ and cut $\widehat\Pi^*$ into an octagon disc with a hole.
\end{enumerate}
Moreover, one can verify that, after the cutting, a canonic dividing configuration $(\delta_+,\delta_-)$
of $\widehat\Pi^*$ looks as shown in Figure~\ref{divconf}
on the left (with Convention~\ref{color-conv} in force).
\begin{figure}[ht]
\begin{tabular}{ccc}
\includegraphics[scale=0.48]{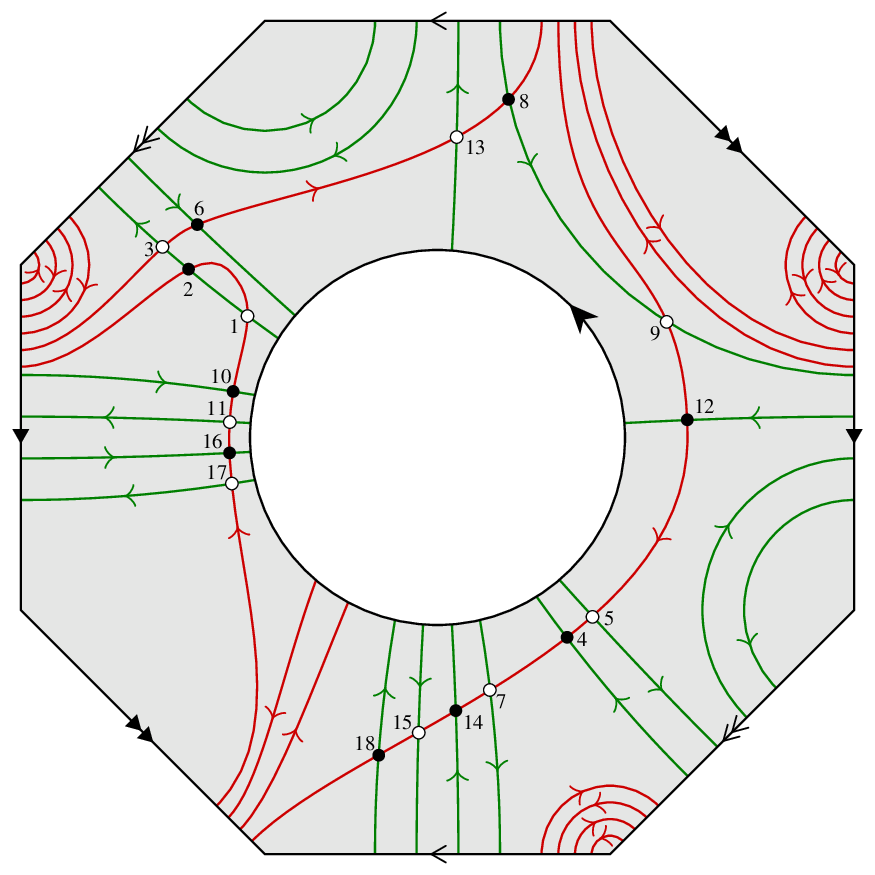}&\hbox to 0.5cm{}&\includegraphics[scale=0.48]{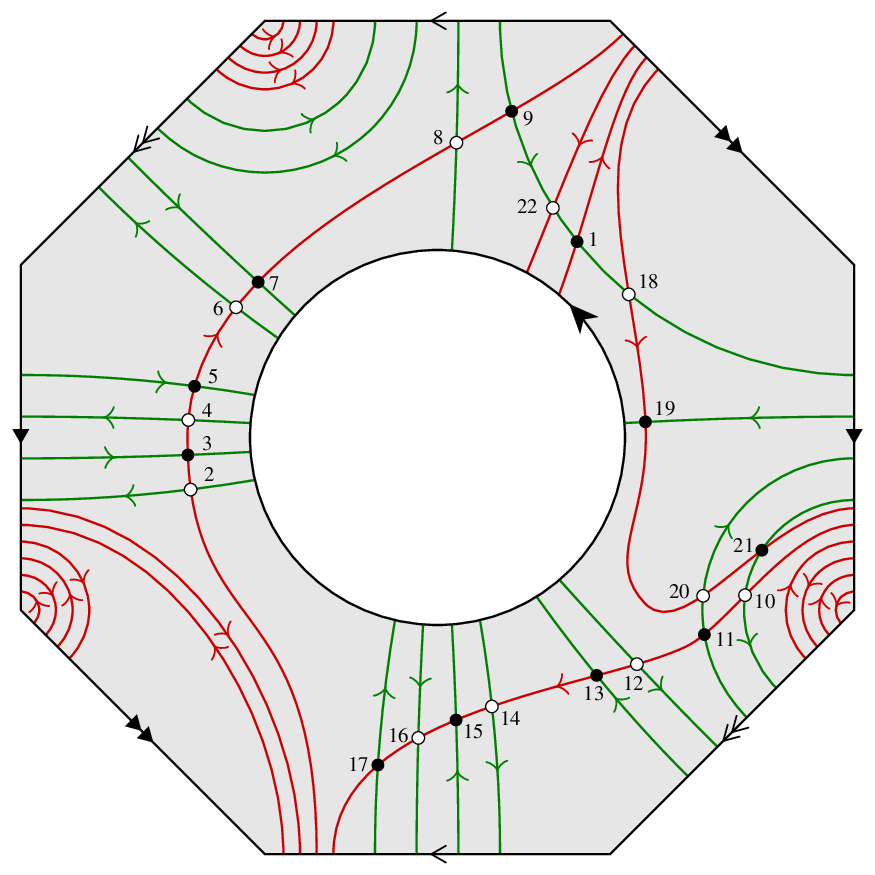}\\
\includegraphics{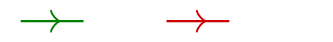}\put(-115,7){$\delta_+$}\put(-45,7){$\delta_1=\delta_-$}&&
\includegraphics{legend1.eps}\put(-115,7){$\delta_+$}\put(-45,7){$\delta_2=\sigma(\delta_-)$}
\end{tabular}
\caption{Dividing configurations $(\delta_+,\delta_-)$ and $(\delta_+,\sigma(\delta_-))$}\label{divconf}
\end{figure}
An intersection of $\delta_+$ and $\delta_-$ is marked by a black dot if
the surface is shown near this point in the same orientation in Figure~\ref{divconf}
as in the torus projection in Figure~\ref{seifert}, and by
a white dot otherwise.

Now we are going to apply Corollary~\ref{maincoro}, in which we put $L=\widehat R_1$, $R=R_2$, and $F=\widehat\Pi^*$.
By Theorem~\ref{invariant-thm}, since $\widehat R_1$ and~$\widehat R_2$ are equivalent as $\xi_-$-Legendrian
knots and~$\tb_+(R_2)=-7<0$, there exists a $-$-realization of~$\delta_-$ at~$R_2$. However, the one-element set $\{\delta_-\}$ is not known to be (and actually
is not) $-$-representative for~$R_2$,
so Corollary~\ref{maincoro} cannot be applied immediately.

To construct a $-$-representative family we must take into account the symmetry group
of the knot~$6_2$, that is, the mapping class group of the pair $(\mathbb S^3,\widehat R_1)$. This group
is known to be isomorphic to the dihedral group~$D_2\cong\mathbb Z_2\oplus\mathbb Z_2$, see~\cite{sak90,ks92}, but
this includes two elements inverting the orientation of~$\widehat R_1$, which do not bother us. The subgroup
of orientation-preserving elements is just $\mathbb Z_2$. The only nontrivial
element of this subgroup can be represented
by an orientation-preserving self-homeomorphism~$\sigma$ of $\mathbb S^3$ that
preserves the surface $\widehat\Pi^*$ and sends each generator $x_i$, $i=1,2,3,4$,
indicated in Figure~\ref{seifert} to its inverse.

In order to make the reader able to verify this, we note that the knot $6_2$
is fibered~\cite{stal61,mur63}
and has genus two~\cite{sei35}. So, the surface $\widehat\Pi^*$ can be
taken for a fiber. The loops $x_i$, $i=1,2,3,4$,
freely generate~$\pi_1(\widehat\Pi^*)$. The complement of $\widehat R_1$ can
be identified with the mapping torus of a self-homeomorphism of $\widehat\Pi^*\setminus\partial\widehat\Pi^*$
inducing the following automorphism of $\pi_1(\widehat\Pi^*)$:
$$x_1\mapsto x_1x_2x_1^2,\quad x_2\mapsto x_1^{-2}x_3x_1^{-1},\quad x_3\mapsto x_1x_4x_1^2,\quad x_4\mapsto x_1^{-1}.$$

Thus, the fundamental group of $\mathbb S^3\setminus\widehat R_1$ has the following
presentation:
$$\langle x_1,x_2,x_3,x_4,t\;;\;tx_1t^{-1}=x_1x_2x_1^2,\ tx_2t^{-1}=x_1^{-2}x_3x_1^{-1},\ tx_3t^{-1}=x_1x_4x_1^2,\ tx_4t^{-1}=x_1^{-1}\rangle.$$
All the generators of this presentation are shown in Figure~\ref{xxxxt}.
\begin{figure}[ht]
\includegraphics[scale=0.7]{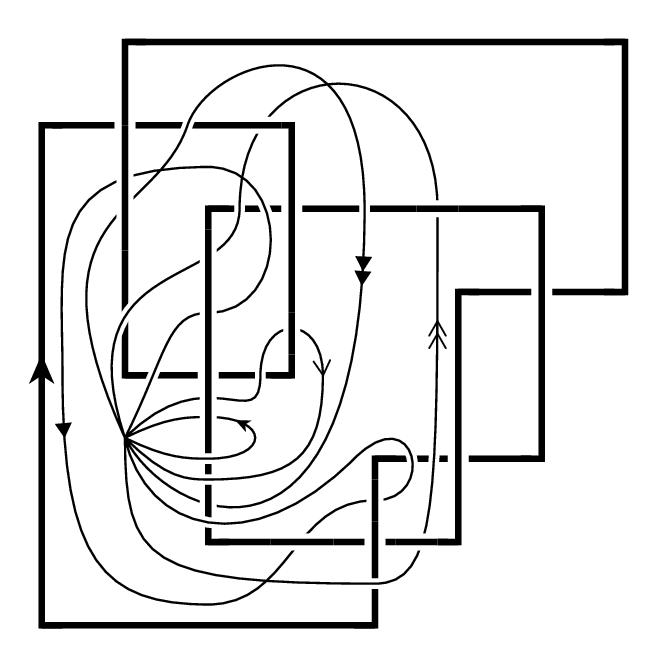}\put(-137,77){$t$}\put(-118,110){$x_1$}\put(-88,120){$x_2$}\put(-102,120){$x_4$}\put(-201,70){$x_3$}
\caption{Generators $x_1,x_2,x_3,x_4,t$ of $\pi_1(\mathbb S^3\setminus\widehat R_1)$}\label{xxxxt}
\end{figure}
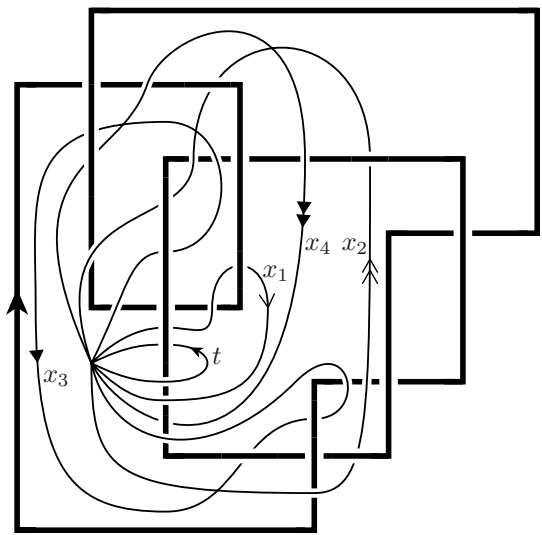
One can check the above relations
directly by switching to the Wirtinger presentation.

Now one can easily see that the formulas
\begin{equation}\label{involution}
x_1\mapsto x_1^{-1},\quad x_2\mapsto x_2^{-1},\quad x_3\mapsto x_3^{-1},\quad x_4\mapsto x_4^{-1},\quad t\mapsto x_1t
\end{equation}
define an automorphism of $\pi_1(\mathbb S^3\setminus\widehat R_1)$, which
corresponds to a self-homeomorphism~$\sigma$ of $\mathbb S^3\setminus\widehat R_1$ (extendable
to the whole of~$\mathbb S^3$). We
claim that such a homeomorphism represents the sought-for element of the symmetry group. Indeed,
it follows from~\eqref{involution} that the orientations of $\widehat\Pi^*$, $\widehat R_1$, and $\mathbb S^3$ are preserved,
the automorphism has order two and represents a non-trivial element in the group
of outer automorphisms of~$\pi_1(\mathbb S^3\setminus\widehat R_1)$.

The homeomorphism~$\sigma$ can be chosen to be an involution. Moreover,
in terms of the punctured octagon obtained
from $\widehat\Pi^*$ by cutting along $x_i$, $i=1,2,3,4$, the map $\sigma|_{\widehat\Pi^*}$ can be turned
into a rotation by $\pi$ around the center.

We now claim that the family~$A=\{\delta_-,\sigma(\delta_-)\}$ is $-$-representative for~$R_2$. Indeed,
let~$(\Pi_1,\phi_1)$ be a proper $-$-realization of~$\delta_-$ at~$R_2$. Then~$(\Pi_1,\phi_1\circ\sigma)$
is a proper $-$-realization of~$\sigma(\delta_-)$. By construction, any
orientation-preserving self-homeomorphism of~$\mathbb S^3$ that takes $L=\widehat R_1$ to $\widehat R_2$
is isotopic relative to~$L$ either to~$\phi_1$ or to~$\phi_1\circ\sigma$, hence~$A$ is $-$-representative.

The right picture in Figure~\ref{divconf} shows the dividing configuration~$(\delta_+,\sigma(\delta_-))$,
which, for our particular choice of $\delta_\pm$ shown in Figure~\ref{divconf}, is admissible.

By Corollary~\ref{maincoro}, we
have $\delta_+\in\mathscr I_{\widehat\Pi^*,\widehat R_1,+}(R_2)$ if and only if there exists a proper realization $(\Pi',\phi)$
of one of the dividing configurations $(\delta_+,\delta_-)$ or $(\delta_+,\sigma(\delta_-))$, such that $\partial\Pi'$ is obtained
from $R_2$ by exchange moves. The diagram~$R_2$ does not admit any exchange move (we call such diagrams \emph{rigid}),
so, we may additionally demand that~$\partial\Pi'=R_2$.
An exhaustive search, which is finite (and very small in this case), of all (combinatorial types of) rectangular diagrams of surfaces $\Pi'$ that give rise to a realization of $(\delta_+,\delta_-)$
or $(\delta_+,\sigma(\delta_-))$ shows that none of them has $\partial\Pi'=R_2$. Therefore, $\delta_+\notin\mathscr I_{\widehat\Pi^*,\widehat R_1,+}$,
and hence the Legendrian knots $\widehat R_1$ and $\widehat R_2$ are not equivalent.

Moreover, the exhaustive search shows that the configuration $(\delta_+,\delta_-)$ has, up to combinatorial
equivalence, only one realization, the one we started with.
The configuration~$(\delta_+,\sigma(\delta_-))$ has only two realizations, which are shown in Figure~\ref{realizations-after-flip}.
\begin{figure}[ht]
\includegraphics[scale=0.5]{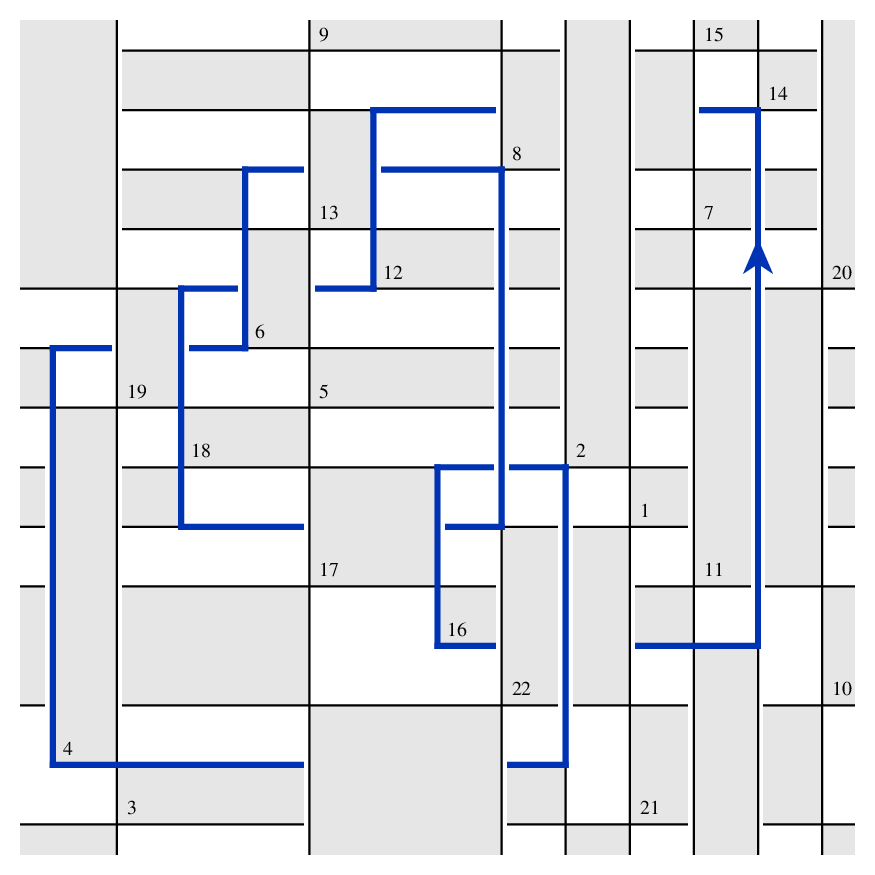}\hskip0.5cm\includegraphics[scale=0.5]{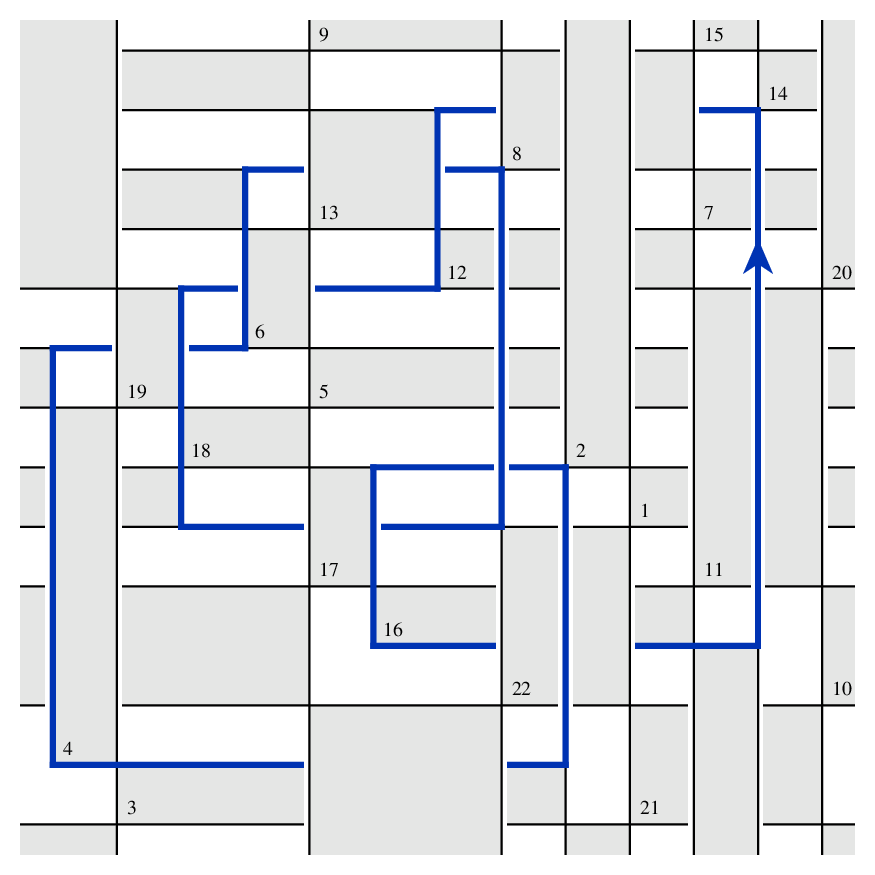}
\caption{The only two realizations of $(\delta_+,\sigma(\delta_-))$}\label{realizations-after-flip}
\end{figure}
(These realizations are proper as the boundary knots of the obtained surfaces have topological
type~$6_2$, and there is only one $C^0$-isotopy class of genus two Seifert surfaces for this knot.)

A proof that there are no more combinatorial types of realizations of~$(\delta_+,\delta_-)$ and~$(\delta_+,\sigma(\delta_-))$
is given in the Appendix~A. We also sketch here an `ideological' argument that allows one,
in this particular case, to see that there is no realization $(\Pi',\phi)$ of $(\delta_+,\delta_-)$
or $(\delta_+,\sigma(\delta_-))$ with $\partial\Pi'=R_2$.

For brevity, we will say that a dividing configuration $D$ is \emph{compatible} with a rectangular diagram of
a knot $R$ if there is a proper realization $(\Pi',\phi)$ of $D$ such that $\partial\Pi'$ can be obtained from $R$
by exchange moves.
We also say that two rectangular diagrams of a knot $R'$ and $R''$, say, are \emph{bi-Legendrian equivalent}
if the knots $\widehat R'$ and $\widehat R''$ are Legendrian equivalent with respect to
both contact structures $\xi_+$ and~$\xi_-$.

For an oriented rectangular diagram of a knot~$R$, we
denote by $\rho(R)$ the diagram obtained from~$R$ by reflecting in the origin, that is, by the map~$(\theta,\varphi)\mapsto(-\theta,-\varphi)$,
and reversing the orientation.
We learn from Figure~\ref{realizations-after-flip} that $(\delta_+,\sigma(\delta_-))$
is compatible with $\rho(R_1)$. The diagrams $R_1$ and $\rho(R_1)$
\begin{figure}[ht]
\begin{tabular}{ccccccccc}
\includegraphics[scale=0.3]{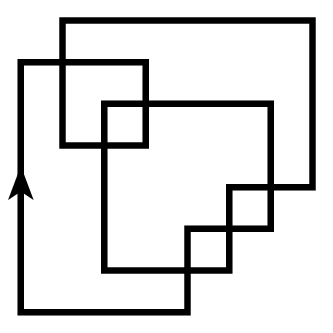}&
\raisebox{22pt}{$=$}&
\includegraphics[scale=0.3]{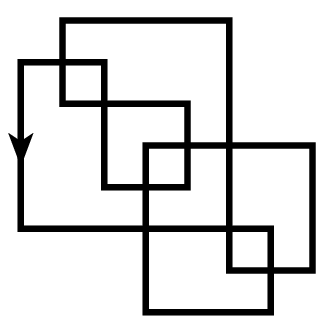}&
\raisebox{22pt}{$\stackrel{\mathrm I}\longleftrightarrow$}&
\includegraphics[scale=0.3]{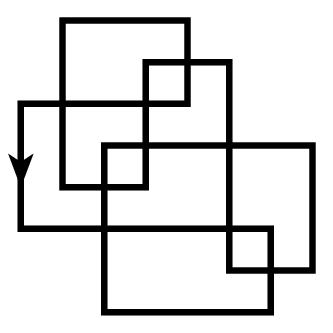}&
\raisebox{22pt}{$\longleftrightarrow$}&
\includegraphics[scale=0.3]{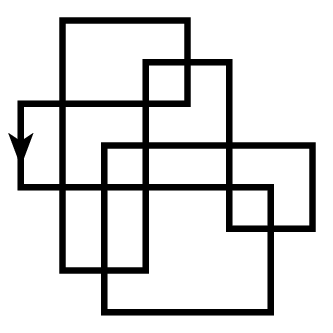}&&
\includegraphics[scale=0.3]{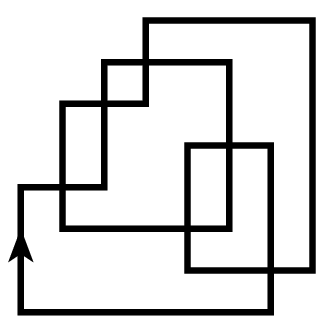}\put(0,0){$\rho(R_1)$}\\
\raisebox{5pt}{\rotatebox{-90}{$\longleftrightarrow$}}&&\raisebox{5pt}{\rotatebox{-90}{$\stackrel{\rotatebox{90}{\scriptsize II}}\longleftrightarrow$}}&&&&\rotatebox{-90}{$=$}&&
\raisebox{5pt}{\rotatebox{-90}{$\longleftrightarrow$}}\\[15pt]
\includegraphics[scale=0.3]{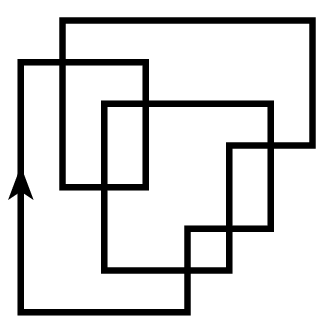}\put(-10,0){$R_1$}&&\includegraphics[scale=0.3]{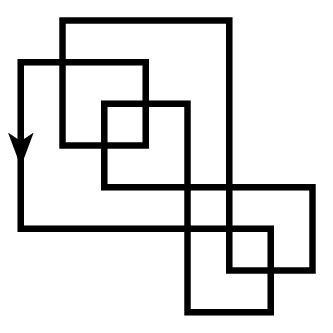}&
\raisebox{22pt}{$\stackrel{\mathrm{II}}\longleftrightarrow$}&
\includegraphics[scale=0.3]{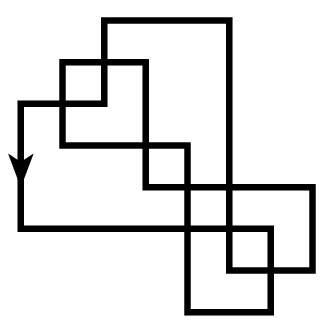}&
\raisebox{22pt}{$\stackrel{\mathrm{II}}\longleftrightarrow$}&
\includegraphics[scale=0.3]{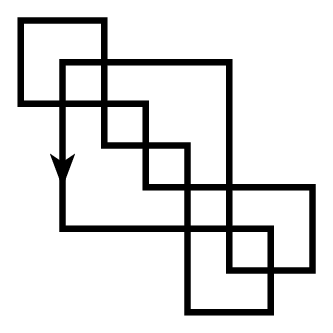}&
\raisebox{22pt}{$=$}&
\includegraphics[scale=0.3]{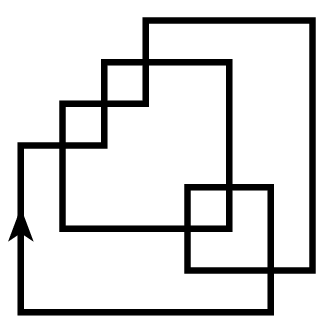}
\end{tabular}
\caption{Bi-Legendrian equivalence of~$R_1$ and~$\rho(R_1)$}\label{bi-leg-fig}
\end{figure}
are known to be bi-Legendrian equivalent~\cite{chong2013} (in the notation of~\cite{chong2013},
if~$R$ represents~$L$, then~$\rho(R)$ represents~$-\mu(L)$).
Two transitions between~$R_1$ and~$\rho(R_1)$ via elementary moves,
one without type~II (de)stabilizations and the other without type~I (de)stabilizations
are sketched in Figure~\ref{bi-leg-fig}, where each arrow marked~`I' or~`II' stands for
an operation that can be decomposed into a stabilization of
the respective type, a few exchange moves, and a destabilization of the same type as
the preceding stabilization. (These
operations are particular cases of flypes introduced in~\cite{dy03}.)
The unmarked arrows stand for (a composition of) exchange moves. The equality signs mean combinatorial equivalence.

One can deduce from the fact that~$(\delta_+,\delta_-)$ is compatible with~$R_1$ and~$(\delta_+,\sigma(\delta_-))$
is compatible with~$\rho(R_1)$ that whenever~$R$ is an oriented rectangular diagram of a knot
bi-Legendrian equivalent to~$R_1$, the following two conditions are equivalent:
\begin{enumerate}
\item
$(\delta_+,\delta_-)$ is compatible with~$R$;
\item
$(\delta_+,\sigma(\delta_-))$ is compatible with~$\rho(R)$.
\end{enumerate}

Now observe that $\delta_+$ and $\delta_-$ intersect in a `non-optimal' way: they have a bigon, that is, a disc
enclosed by two arcs one of which is a subset of $\delta_+$ and the other of $\delta_-$ (the endpoints of the arcs
are numbered~$1$ and~$2$ in the left picture in Figure~\ref{divconf}), and this bigon cannot be reduced,
since the reduction would produce a non-admissible dividing configuration. This means by Lemma~\ref{non-reducible-lem}
that the configuration $(\delta_+,\delta_-)$ is incompatible with any rigid diagram of a non-trivial knot. In particular,
it is incompatible with $R_2$ and~$\rho(R_2)$. This implies that neither $(\delta_+,\delta_-)$
nor $(\delta_+,\sigma(\delta_-))$ is compatible with $R_2$.
\end{proof}

\section{Basic moves of rectangular diagrams of surfaces}\label{basic-moves-sec}

In this section and in Appendix~B
we introduce moves of rectangular diagrams of surfaces that preserve the isotopy class of the represented surface, and allow transition
between diagrams representing isotopic surfaces.

We will refer to all the transformations of rectangular diagrams of surfaces
introduced below in this section and in Appendix~B as \emph{basic moves}. These include: (half-)wrinkle creation and reduction moves, (de)stabilization
moves, exchange moves, and flypes. Some of them are assigned \emph{a type} (I or~II),
and the others are \emph{neutral} (have no type).

\begin{theo}\label{equivalence-thm}
Let $\Pi$ and~$\Pi'$ be rectangular diagrams of a surface. The surfaces~$\widehat\Pi$ and~$\widehat\Pi'$
are isotopic if and only if~$\Pi'$ can be obtained from~$\Pi$ by a sequence
of basic moves.

The surfaces~$\widehat\Pi$ and~$\widehat\Pi'$ are equivalent as Giroux's convex surfaces with
respect to~$\xi_+$ \emph(respectively,~$\xi_-$\emph) if and only if~$\Pi'$ can be obtained from~$\Pi$ by a sequence
of type~I \emph(respectively, type~II\emph) and neutral basic moves.
\end{theo}

Although the hard part of this theorem, namely, the sufficiency of the basic moves
for transition between isotopic (convex) surfaces, sounds as the most fundamental
result among the statements formulated in this section, it is not needed to
establish our main result, which is Theorem~\ref{maintheo}. For this reason,
and in order not to overload the paper,
we omit the proof of Theorem~\ref{equivalence-thm} here.
The proof of the first part of the theorem appears in~\cite{basicmoves}.
The proof of the second part will be published elsewhere.

We also skip some details of the proof that Giroux's convexity can be maintained in the transition from~$\widehat\Pi$
to~$\widehat\Pi'$ as stated in Theorem~\ref{equivalence-thm},
since we don't use the relation to Giroux's convex surfaces in this strong form.
What we do use is the invariance of the isotopy classes of certain dividing sets, which is a consequence of
the above mentioned relation, but can be established without reference to Giroux's convexity.
We also do use in the proof of Theorem~\ref{maintheo} some results
of this section of which the most crucial ones are Propositions \ref{twist-prop} and~\ref{weak-equivalence-via-moves}.

Now we proceed with the definition of basic moves.

\subsection{Notation. Vertex types}\label{notation-subsec}

We recall from~\cite{dp17} that we use the coordinate system $\theta,\varphi,\tau$ on~$\mathbb S^3$
coming from the join presentation of $\mathbb S^3$.
For $v=(\theta_0,\varphi_0)\in\mathbb T^2$ we denote by $\widehat v$ the arc written in these coordinates
as $\{(\theta_0,\varphi_0,\tau):\tau\in[0;1]\}$. For a rectangle $r=[\theta_1;\theta_2]\times[\varphi_1;\varphi_2]$
we denote by $\widehat r$ the tile associated with $r$ (see \cite[Subsection 2.3]{dp17}).

For~$\theta,\varphi\in\mathbb S^1$ we also use the notation $m_\theta$ for
the meridian $\{\theta\}\times\mathbb S^1\subset\mathbb T^2$, and
$\ell_\varphi$ for the longitude~$\mathbb S^1\times\{\varphi\}$.
By $\widehat m_\theta$ and $\widehat\ell_\varphi$ we denote the endpoints
of the arc~$\widehat{(\theta,\varphi)}$, lying on~$\mathbb S^1_{\tau=1}$
and~$\mathbb S^1_{\tau=0}$, respectively.

Let $\Pi$ be a rectangular diagram of a surface, and let $(\delta_+,\delta_-)$ be a canonic dividing
configuration of~$\widehat\Pi$. If~$r$ is a rectangle of $\Pi$, then the tile $\widehat r$
contains a unique intersection point of $\delta_+$ and $\delta_-$, which
we denote by~$\mathring r$. If $v$ is a vertex of $\Pi$, then there is
a unique arc of $(\delta_+\cup\delta_-)\setminus(\delta_+\cap\delta_-)$ intersecting~$\widehat v$,
which we denote by $\mathring v$. Finally, the closure of every connected component of
$\widehat\Pi\setminus(\delta_+\cup\delta_-)$ is a disc containing a unique vertex of the tiling
of $\widehat\Pi$, and if this vertex is $\widehat a$, where $a$ is an occupied level of~$\Pi$,
then the corresponding disc will be denoted
by $\mathring a$.

Summarizing, if $\Pi$ is a rectangular diagram of a surface we have the following one-to-one correspondences
between objects related to the diagram $\Pi$, to the associated surface $\widehat\Pi$, and to a canonic
dividing configuration $(\delta_+,\delta_-)$ of $\widehat\Pi$ (see Figure~\ref{correspondence}):

\centerline{\begin{tabular}{p{4cm}|l|p{4cm}}
$x$&$\widehat x$&$\mathring x$\\\hline
a rectangle of $\Pi$&a tile of $\widehat\Pi$&a point of $\delta_+\cap\delta_-$\\\hline
a vertex of $\Pi$&an edge of $\widehat\Pi$&a connected component of $(\delta_+\cup\delta_-)\setminus(\delta_+\cap\delta_-)$\\\hline
an occupied level
of $\Pi$&a vertex of $\widehat\Pi$&the closure of a connected component of $\widehat\Pi\setminus
(\delta_+\cup\delta_-)$
\end{tabular}}
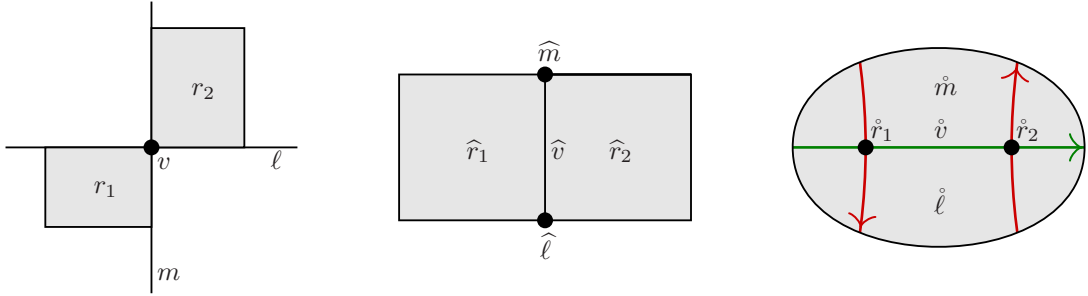
\begin{figure}[ht]
\includegraphics{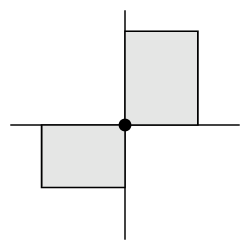}\put(-82,43){$r_1$}\put(-45,80){$r_2$}\put(-58,52){$v$}\put(-15,52){$\ell$}\put(-58,10){$m$}\hskip1cm
\includegraphics{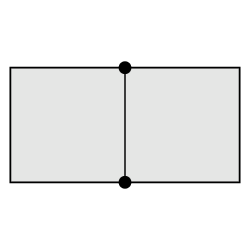}\put(-90,56){$\widehat r_1$}\put(-36,56){$\widehat r_2$}\put(-58,56){$\widehat v$}%
\put(-62,18){$\widehat\ell$}\put(-63,93){$\widehat m$}\hskip1cm
\includegraphics{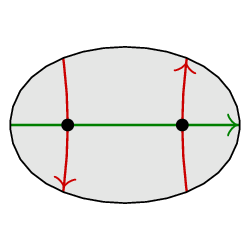}\put(-86,64){$\mathring r_1$}\put(-31,64){$\mathring r_2$}\put(-62,64){$\mathring v$}%
\put(-62,35){$\mathring\ell$}\put(-62,80){$\mathring m$}
\caption{Correspondence between objects related to~$\Pi$, to~$\widehat\Pi$, and to~$(\delta_+,\delta_-)$}\label{correspondence}
\end{figure}
\begin{defi}
If $r=[\theta_1;\theta_2]\times[\varphi_1;\varphi_2]\subset\mathbb T^2$ is a rectangle,
then the points $(\theta_1,\varphi_1)$, $(\theta_2,\varphi_2)$ will be called \emph{the $\diagdown$-vertices of $r$}
and the points $(\theta_1,\varphi_2)$, $(\theta_2,\varphi_1)$ \emph{the $\diagup$-vertices of $r$}.
A vertex $v$ is called \emph{a $\diagdown$-vertex} (respectively, \emph{a $\diagup$-vertex}) \emph{of a rectangular diagram $\Pi$}
if it is a $\diagdown$-vertex (respectively, a $\diagup$-vertex) of some rectangle $r\in\Pi$. It follows from the definition
of a rectangular diagram of a surface that no point can be simultaneously a $\diagdown$-vertex and a $\diagup$-vertex
of $\Pi$.

Equivalently, $v$ is a $\diagdown$-vertex (respectively, $\diagup$-vertex) of $\Pi$ if $\mathring v\subset\delta_+$
(respectively, $\mathring v\subset\delta_-$), where $(\delta_+,\delta_-)$ is a canonic dividing configuration of $\widehat\Pi$.
\end{defi}

\subsection{Wrinkle moves and stabilization moves}
The moves defined in this subsection are similar to each other in nature. However, wrinkle and half-wrinkle moves
are neutral, and (de)stabilization moves have type~I or type~II.

\begin{defi}\label{wrinkledef}
Let $\Pi$ be a rectangular diagram of a surface, and let $v_1=(\theta_0,\varphi_1)$ and $v_2=(\theta_0,\varphi_2)$
be a $\diagdown$-vertex and a $\diagup$-vertex of $\Pi$, respectively, lying on the same meridian $m_{\theta_0}$.
Choose an $\varepsilon>0$ so that no meridian in $[\theta_0-2\varepsilon;\theta_0+2\varepsilon]\times
\mathbb S^1\subset\mathbb T^2$  other than $m_{\theta_0}$ is an occupied level of $\Pi$.
Also choose an orientation-preserving self-homeomorphism~$\psi$ of the interval~$[\theta_0-2\varepsilon;\theta_0+2\varepsilon]$.

Let $\Pi'$ be the rectangular diagram of a surface obtained from $\Pi$ by making the following modifications:
\begin{enumerate}
\item
every rectangle of the form $[\theta_0;\theta_1]\times[\varphi';\varphi'']$ (respectively, $[\theta_1;\theta_0]\times[\varphi';\varphi'']$)
with $[\varphi';\varphi'']\subset[\varphi_1;\varphi_2]$ is replaced by
$[\psi(\theta_0+\varepsilon);\theta_1]\times[\varphi';\varphi'']$ (respectively, by $[\theta_1;\psi(\theta_0+\varepsilon)]\times[\varphi';\varphi'']$);
\item
every rectangle of the form $[\theta_0;\theta_1]\times[\varphi';\varphi'']$ (respectively, $[\theta_1;\theta_0]\times[\varphi';\varphi'']$)
with $[\varphi';\varphi'']\subset[\varphi_2;\varphi_1]$ is replaced by
$[\psi(\theta_0-\varepsilon);\theta_1]\times[\varphi';\varphi'']$ (respectively, by $[\theta_1;\psi(\theta_0-\varepsilon)]\times[\varphi';\varphi'']$);
\item
two new rectangles are added, $[\psi(\theta_0-\varepsilon);\psi(\theta_0)]\times[\varphi_1;\varphi_2]$ and
$[\psi(\theta_0);\psi(\theta_0+\varepsilon)]\times[\varphi_2;\varphi_1]$.
\end{enumerate}
Then we say that the passage from $\Pi$ to $\Pi'$ is \emph{a vertical wrinkle creation move}.
The inverse operation is referred to as \emph{a vertical wrinkle reduction move}.

\emph{Horizontal wrinkle creation} and \emph{reduction moves} are defined similarly with the roles of $\theta$ and $\varphi$
exchanged.
\end{defi}

A vertical wrinkle move is illustrated in Figure~\ref{wrinklemovefig}. The left pair of pictures shows how the rectangular
diagram changes,
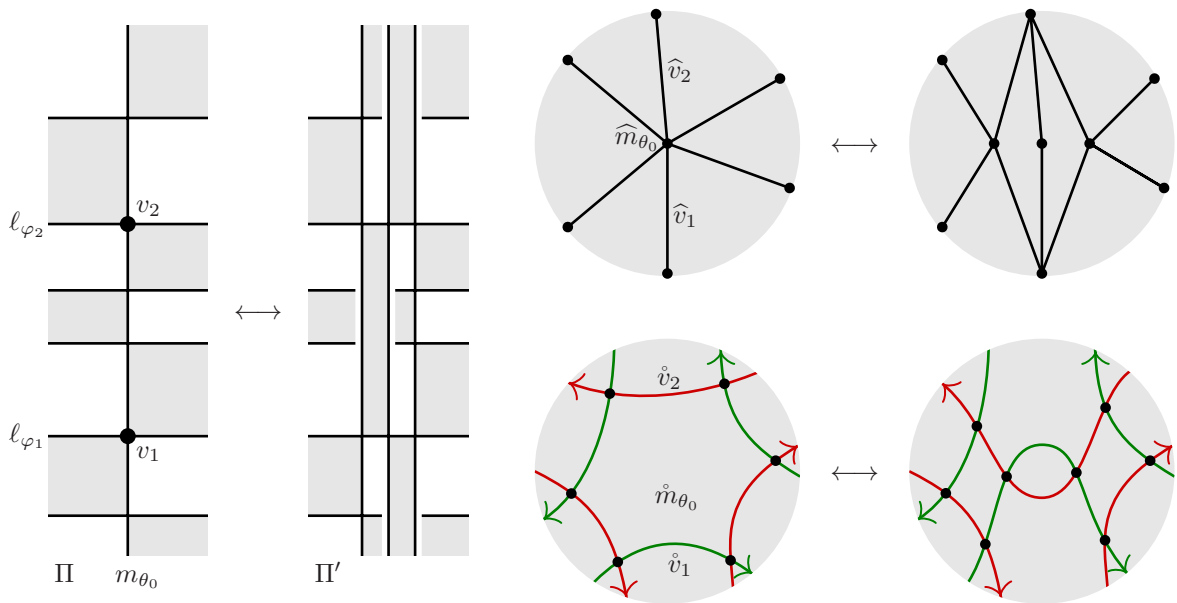
\begin{figure}[ht]
\centerline{\includegraphics{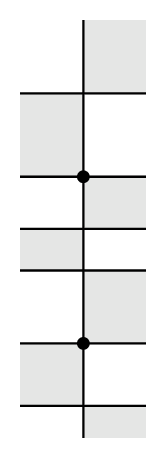}\put(-68,0){$\Pi$}
\put(-37,47){$v_1$}\put(-37,140){$v_2$}%
\put(-45,0){$m_{\theta_0}$}\put(-85,53){$\ell_{\varphi_1}$}\put(-85,133){$\ell_{\varphi_2}$}%
\raisebox{3.5cm}{$\longleftrightarrow$}\includegraphics{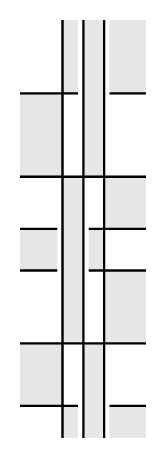}\put(-68,0){$\Pi'$}
\raisebox{3.5cm}{\begin{tabular}{ccc}
\includegraphics{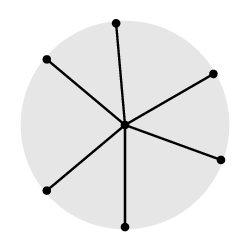}\put(-58,30){$\widehat v_1$}\put(-60,85){$\widehat v_2$}\put(-80,60){$\widehat m_{\theta_0}$}
&\hskip-3mm\raisebox{2cm}{$\longleftrightarrow$}\kern-3mm&\includegraphics{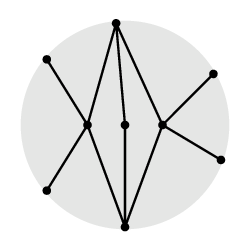}\\
\includegraphics{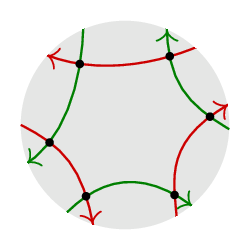}\put(-60,22){$\mathring v_1$}\put(-64,93){$\mathring v_2$}\put(-65,48){$\mathring m_{\theta_0}$}
&\hskip-3mm\raisebox{2cm}{$\longleftrightarrow$}\kern-3mm&\includegraphics{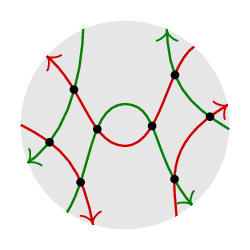}
\end{tabular}}}
\caption{A vertical wrinkle move}\label{wrinklemovefig}
\end{figure}
the top right pair of pictures shows how the corresponding tiling of $\widehat\Pi$ changes,
and the bottom right pair of pictures demonstrates the change in the canonic dividing
configuration of $\widehat\Pi$.

The combinatorial type of the obtained diagram~$\Pi'$ in Definition~\ref{wrinkledef}
does not depend on the homeomorphism~$\psi$, and, in many situations, the reader
may safely forget about~$\psi$ by assuming~$\psi=\mathrm{id}$. However,
the flexibility arising from the arbitrariness in the choice of~$\psi$ will sometimes be useful.
Namely, by choosing~$\psi$ so that~$\psi(\theta_0-\varepsilon)=\theta_0$
or~$\psi(\theta_0+\varepsilon)=\theta_0$ (we keep using the notation from Definition~\ref{wrinkledef})
we will reduce the number of rectangles of the diagram~$\Pi$
that are modified by the move. In particular, if~$m_{\theta_0}$ contains vertices of~$\partial\Pi$,
then an appropriate choice of~$\psi$ will allow us to keep the boundary of the diagram fixed,
that is, to have $\partial\Pi'=\partial\Pi$.

Similar homeomorphisms, also denoted by~$\psi$, appear in Definitions~\ref{half-wrinkle-def}, \ref{stab-move-def}, and~\ref{exchange-def}
for a like reason.

\begin{lemm}\label{wrinkle-creation-isotopy}
Let $\Pi\mapsto\Pi'$ be a wrinkle creation move, and let~$D$, $D'$ be canonic dividing
configurations of~$\widehat\Pi$ and~$\widehat\Pi'$, respectively.
Then there exists an isotopy bringing~$(\widehat\Pi,D)$ to~$(\widehat\Pi',D'')$, where~$D''$
is a dividing configuration on~$\widehat\Pi'$ weakly equivalent to~$D'$.
If, additionally, we have~$\partial\Pi=\partial\Pi'$, then the isotopy
can be chosen to be fixed on~$\partial\widehat\Pi$.

Moreover, $\widehat\Pi$ and~$\widehat\Pi'$ are isotopic in the class of
Giroux's convex surfaces with respect to either of the contact structures~$\xi_+$ and~$\xi_-$,
and if~$\partial\Pi=\partial\Pi'$, then the isotopy can be chosen to be fixed on~$\partial\widehat\Pi$.
\end{lemm}

\begin{proof}
We use the notation from Definition~\ref{wrinkledef}. Let~$U$ be a small open neighborhood of
the domain~$[\theta_0-\varepsilon;\theta_0+\varepsilon]*\mathbb S^1\subset\mathbb S^3$. If~$U$ is chosen small enough,
then it is homeomorphic to a~three-ball and intersects each of the surfaces~$\widehat\Pi$ and~$\widehat\Pi'$
in a~two-disc. Indeed, the intersection~$\widehat\Pi\cap U$ is a small neighborhood
of the star graph in~$\widehat\Pi$ formed by the edges of the tiling that emanate from~$\widehat m_{\theta_0}$.
The intersection~$\widehat\Pi'\cap U$ is a small neighborhood, in~$\widehat\Pi'$,
of the union of the edges of the new tiling that emanate from~$\widehat m_{\theta_0+\varepsilon}$
and~$\widehat m_{\theta_0-\varepsilon}$ with the two new tiles.

One can also see that the surfaces~$\widehat\Pi$ and~$\widehat\Pi'$ are close to one another outside~$U$. Moreover, we
have~$\partial\widehat\Pi\setminus U=\partial\widehat\Pi'\setminus U$, and both intersections
$\partial\widehat\Pi\cap U$, $\partial\widehat\Pi'\cap U$ are either empty or consist of a single arc.
This implies the existence of the required isotopy from~$\widehat\Pi$ to $\widehat\Pi'$.

In order to see that an isotopy from~$\widehat\Pi$ to~$\widehat\Pi'$ can be performed in the class of Giroux's convex surfaces we
note that the only `large' modification of the surface~$\widehat\Pi$ occurs near the disc~$m_{\theta_0}*\mathbb S^1$. The modification consists in cutting~$\widehat\Pi$
along the arc $\widehat v_1\cup\widehat v_2$, then shifting the banks
of the cut off one another by a $C^1$-small deformation, and finally gluing a disc consisting of two new tiles
in the obtained hole.

Due to the symmetry between~$\xi_+$ and~$\xi_-$ it suffices to discuss the convexity issue with respect to~$\xi_+$.
We observe that the only portion of the dividing set~$\delta_+$ that is involved in
the `large' modification is a subarc of the arc~$\mathring v_1$, and we have just seen
that the portion of the surface affected by this modification is a disc. The technique of~\cite{gi1}
(see also~\cite{mas1}) allows to construct an isotopy for this modification that avoids non-convex surfaces.
\end{proof}

Now we direct our attention to the modification of a canonic dividing configuration that occurs as a
result of a wrinkle creation move. 

\begin{defi}\label{bigon-def}
Let~$D=(\delta_+,\delta_-)$ and~$D'=(\delta'_+,\delta'_-)$ be two dividing configurations on
a surface~$F$, such that they are isotopic outside of an open two-disc~$d\subset F$ and
satisfy the following conditions:
\begin{enumerate}
\item
each intersection $\delta_+\cap d$, $\delta_-\cap d$, $\delta'_+\cap d$, $\delta'_-\cap d$
is an arc;
\item
the arcs $\delta_+\cap d$ and $\delta_-\cap d$ are disjoint;
\item
the arcs $\delta'_+\cap d$ and $\delta'_-\cap d$ intersect transversely
in two points.
\end{enumerate}
Then we say that~$D'$ is obtained from~$D$ by \emph{a bigon creation},
and~$D$ from~$D'$ by \emph{a bigon reduction}. The disc in~$d$ enclosed by two
arcs one of which is contained in~$\delta'_+$ and the other in~$\delta'_-$
is called \emph{a bigon of~$(\delta_+',\delta_-')$}.
\end{defi}

One can see from Figure~\ref{wrinklemovefig} that the passage from~$D''$ to~$D'$ in
Lemma~\ref{wrinkle-creation-isotopy} is a composition of a bigon creation and an isotopy in~$\widehat\Pi'$.
The following Lemma shows that \emph{any} bigon creation can be `realized' by a wrinkle creation move.

\begin{lemm}\label{bigon-creation-lemm}
Let~$\Pi$ be a rectangular diagram of  a surface, and let $(\delta_+,\delta_-)$ be a canonic dividing configuration
of~$\widehat\Pi$. Let also~$(\delta'_+,\delta'_-)$ be a dividing configuration on~$\widehat\Pi$ obtained from~$(\delta_+,\delta_-)$
by a bigon creation. Then there exists a proper realization~$(\Pi',\phi)$ of~$(\delta'_+,\delta'_-)$
such that~$\Pi\mapsto\Pi'$ is a wrinkle creation move.
\end{lemm}

\begin{proof}
Let~$d$ be as in Definition~\ref{bigon-def}. Then the arcs~$\delta_+\cap d$ and~$\delta_-\cap d$
are subarcs of~$\mathring v_1$ and~$\mathring v_2$, respectively,
for some $\diagdown$-vertex~$v_1$ and $\diagup$-vertex~$v_2$ of~$\Pi$.
These two vertices lie on the same occupied level of~$\Pi$, since
there is clearly a connected component~$\Omega$ of~$\widehat\Pi\setminus(\delta_+\cup\delta_-)$
such that both~$\mathring v_1$ and~$\mathring v_2$
contribute to the boundary of~$\Omega$. Without loss of generality we
may assume that they lie on the same meridian~$m_{\theta_0}$. Then we can apply a wrinkle creation
move exactly as described in Definition~\ref{wrinkledef}. Let~$\Pi'$ be the obtained diagram,
and let~$\phi$ be a homeomorphism from~$\widehat\Pi$ to~$\widehat\Pi'$
that arises from the isotopy discussed in Lemma~\ref{wrinkle-creation-isotopy}.
Then~$\phi$ takes~$(\delta_+',\delta_-')$ to a dividing configuration on~$\widehat\Pi'$ equivalent
to a canonic one. The claim follows.
\end{proof}

\begin{defi}\label{half-wrinkle-def}
Let $\Pi$, $v_1$, $v_2$, and~$\psi$ be as in Definition~\ref{wrinkledef} and suppose
additionally that we have $v_1,v_2\in\partial\Pi$.
Let $\Pi'$ be obtained from $\Pi$ as described in Definition~\ref{wrinkledef} with the following one distinction:
\begin{itemize}
\item
if $\Pi$ has no rectangle of the form $[\theta_0;\theta_1]\times[\varphi';\varphi'']$ or
$[\theta_1;\theta_0]\times[\varphi';\varphi'']$ with $[\varphi';\varphi'']\subset[\varphi_1;\varphi_2]$,
we do not add the rectangle $[\psi(\theta_0);\psi(\theta_0+\varepsilon)]\times[\varphi_2;\varphi_1]$;
\item
if $\Pi$ has no rectangle of the form $[\theta_0;\theta_1]\times[\varphi';\varphi'']$ or
$[\theta_1;\theta_0]\times[\varphi';\varphi'']$ with $[\varphi';\varphi'']\subset[\varphi_2;\varphi_1]$,
we do not add the rectangle~$[\psi(\theta_0-\varepsilon);\psi(\theta_0)]\times[\varphi_1;\varphi_2]$.
\end{itemize}
One of these two cases must occur.

Then we say that the passage from $\Pi$ to $\Pi'$ is \emph{a vertical half-wrinkle creation move}, and the inverse operation is \emph{a vertical half-wrinkle reduction move}.

\emph{Horizontal half-wrinkle moves} are defined similarly with the roles of $\theta$ and $\varphi$
exchanged.
\end{defi}

A half-wrinkle move is illustrated in Figure~\ref{halfwrinklemovefig}.
\begin{figure}[ht]
\centerline{\includegraphics{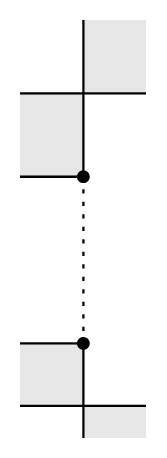}\put(-68,0){$\Pi$}\put(-45,0){$m_{\theta_0}$}\put(-37,60){$v_1$}\put(-37,127){$v_2$}%
\raisebox{3.5cm}{$\longleftrightarrow$}\includegraphics{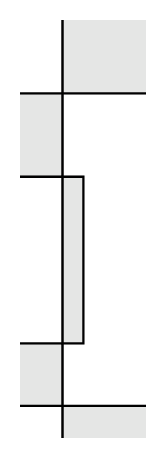}\put(-68,0){$\Pi'$}
\raisebox{3.5cm}{\begin{tabular}{ccc}
\includegraphics{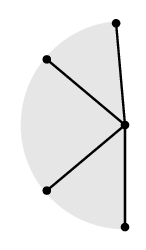}\put(-8,30){$\widehat v_1$}\put(-10,85){$\widehat v_2$}\put(-7,52){$\widehat m_{\theta_0}$}
&\raisebox{2cm}{$\longleftrightarrow$}&\includegraphics{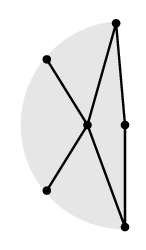}\\
\includegraphics{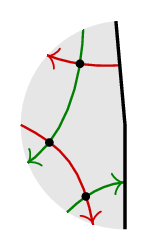}\put(-22,21){$\mathring v_1$}\put(-24,80){$\mathring v_2$}\put(-30,50){$\mathring m_{\theta_0}$}%
\put(-8,70){$\partial\widehat\Pi$}
&\raisebox{2cm}{$\longleftrightarrow$}&\includegraphics{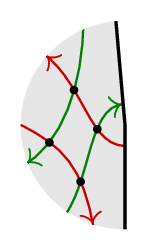}
\end{tabular}}}
\caption{A vertical half-wrinkle move}\label{halfwrinklemovefig}
\end{figure}
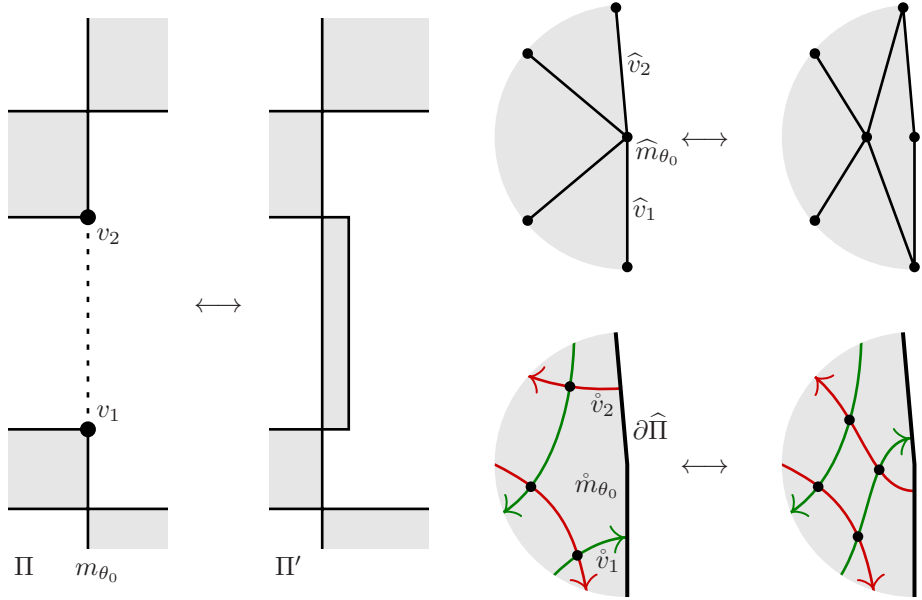
One can see that in the case $v_1,v_2\in\partial\Pi$ the respective wrinkle creation move can be decomposed into two half-wrinkle creation moves,
hence the name.

There is a full analogue of Lemma~\ref{wrinkle-creation-isotopy} for half-wrinkle moves.

\begin{lemm}\label{half-wrinkle-creation-isotopy}
Let $\Pi\mapsto\Pi'$ be a half-wrinkle creation move, and let~$D$, $D'$ be canonic dividing
configurations of~$\widehat\Pi$ and~$\widehat\Pi'$, respectively.
Then there exists a  $C^0$-isotopy that brings~$(\widehat\Pi,D)$ to~$(\widehat\Pi',D'')$, where~$D''$
is a dividing configuration on~$\widehat\Pi'$ weakly equivalent to~$D'$.
If, additionally, we have~$\partial\Pi=\partial\Pi'$, then the isotopy
can be chosen to be fixed on~$\partial\widehat\Pi$.

Moreover, $\widehat\Pi$ and~$\widehat\Pi'$ are isotopic in the class of
Giroux's convex surfaces with respect to either of the contact structures~$\xi_+$ and~$\xi_-$,
and if~$\partial\Pi=\partial\Pi'$, then the isotopy can be chosen to keep the boundary of the
surface arbitrarily $C^0$-close to~$\partial\widehat\Pi$.
\end{lemm}

We omit the proof, which is similar to that of Lemma~\ref{wrinkle-creation-isotopy}.

\begin{rema}
A small complication of the formulation of Lemma~\ref{half-wrinkle-creation-isotopy},
if compared to Lemma~\ref{wrinkle-creation-isotopy}, is due to the fact that the boundaries~$\partial\widehat\Pi$
and~$\partial\widehat\Pi'$ have singularities, some of which may need to be smoothed in order
to perform an isotopy from~$\widehat\Pi$ to~$\widehat\Pi'$ through Giroux's convex surfaces. This is because
the boundary framings induced by~$\Pi$ and~$\Pi'$ are now different, which was not the case in Lemma~\ref{wrinkle-creation-isotopy}.
\end{rema}

\begin{defi}
Let $(\delta_+,\delta_-)$ and~$(\delta'_+,\delta'_-)$ be two dividing configurations on a surface~$F$.
Assume that there is a closed disc~$b\subset F$ whose boundary consists of three arcs $\alpha,\beta,\gamma$,
such that~$\alpha\subset\delta'_+$, $\beta\subset\delta'_-$, and~$\gamma\subset\partial F$.
Such a disc will be called \emph{a half-bigon of~$(\delta'_+,\delta'_-)$}.

Assume also that there is an open neighborhood~$d$ of~$b$ in~$F$ homeomorphic to a half-disc such that the following holds:
\begin{enumerate}
\item
each intersection $\delta_+\cap d$, $\delta_-\cap d$, $\delta'_+\cap d$, $\delta'_-\cap d$
is a half-closed arc;
\item
the arcs $\delta_+\cap d$ and $\delta_-\cap d$ are disjoint;
\item
the arcs $\delta'_+\cap d$ and $\delta'_-\cap d$ intersect transversely
in a single point.\end{enumerate}
Then we say that the transition~$(\delta_+,\delta_-)\mapsto(\delta'_+,\delta'_-)$ is
\emph{a half-bigon creation} and the inverse one is \emph{a half-bigon reduction}.
\end{defi}

\begin{lemm}\label{half-bigon-creation-lemm}
Let~$\Pi$ be a rectangular diagram of  a surface, and let $(\delta_+,\delta_-)$ be a canonic dividing configuration
of~$\widehat\Pi$. Let also~$(\delta'_+,\delta'_-)$ be a dividing configuration obtained from~$(\delta_+,\delta_-)$
by a half-bigon creation. Then there exists a proper realization~$(\Pi',\phi)$ of~$(\delta'_+,\delta'_-)$
such that~$\Pi\mapsto\Pi'$ is a half-wrinkle creation move.
\end{lemm}

The proof, which is similar to that of Lemma~\ref{bigon-creation-lemm}, is omitted.

\begin{defi}\label{stab-move-def}
Let $\Pi$, $v_1$, $v_2$, and~$\psi$ be as in Definition~\ref{wrinkledef} except that $v_2\in m_{\theta_0}$ is not a vertex of $\Pi$ and,
moreover, $v_2$ does not belong to the boundary of any rectangle in~$\Pi$. Let $\Pi'$ be obtained from $\Pi$
by exactly the same modification as the one described in Definition~\ref{wrinkledef}. Then the passage
from $\Pi$ to $\Pi'$ is called \emph{a type I stabilization move} and the inverse one
\emph{a type I destabilization move}.

Similarly, if $v_2$ is as in Definition~\ref{wrinkledef}, and $v_1\in m_{\theta_0}$ does not belong to the boundary
of any rectangle, then the passage from $\Pi$ to $\Pi'$ is called \emph{a type II stabilization move} and
the inverse one \emph{a type II destabilization move}.

If the roles of $\theta$ and $\varphi$ coordinates are exchanged in this definition the obtained
moves are still called type~I or type~II (de)stabilization moves, respectively.
\end{defi}

An example of a (de)stabilization move is shown in Figure~\ref{boundarystabmovefig}.
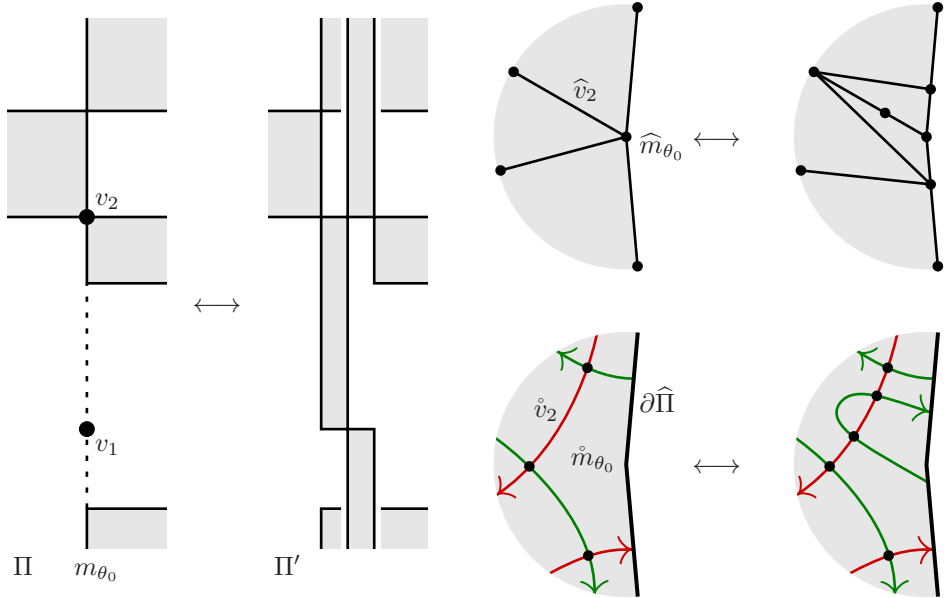
\begin{figure}[ht]
\centerline{\includegraphics{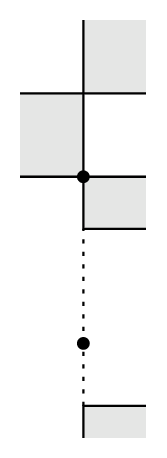}\put(-68,0){$\Pi$}\put(-45,0){$m_{\theta_0}$}\put(-37,47){$v_1$}\put(-37,140){$v_2$}%
\raisebox{3.5cm}{$\longleftrightarrow$}\includegraphics{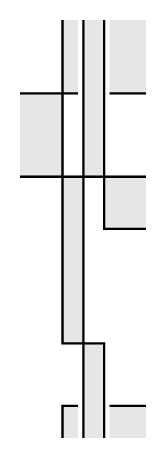}\put(-68,0){$\Pi'$}
\raisebox{3.5cm}{\begin{tabular}{ccc}
\includegraphics{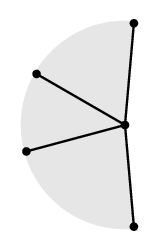}\put(-35,75){$\widehat v_2$}\put(-10,54){$\widehat m_{\theta_0}$}
&\raisebox{2cm}{$\longleftrightarrow$}&\includegraphics{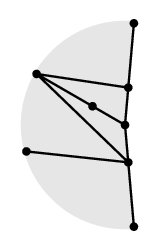}\\
\includegraphics{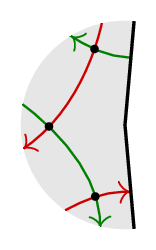}\put(-50,78){$\mathring v_2$}\put(-36,60){$\mathring m_{\theta_0}$}%
\put(-10,80){$\partial\widehat\Pi$}
&\raisebox{2cm}{$\longleftrightarrow$}&\includegraphics{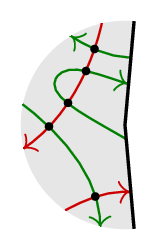}
\end{tabular}}}
\caption{A type II stabilization/destabilization moves}\label{boundarystabmovefig}
\end{figure}

\begin{lemm}\label{stab-lem}
Let $\Pi\mapsto\Pi'$ be a stabilization move, and let~$(\delta_+,\delta_-)$,
$(\delta_+',\delta_-')$ be canonic dividing configurations of~$\widehat\Pi$ and~$\widehat\Pi'$,
respectively. Then there is an isotopy that brings~$(\widehat\Pi,\delta_+)$
to~$(\widehat\Pi',\delta_+')$ if the stabilization is of type~I,
and~$(\widehat\Pi,\delta_-)$ to~$(\widehat\Pi',\delta_-')$ if the stabilization
is of type~II. Moreover, the isotopy can be chosen to keep the surface in the class of
Giroux's convex surfaces with respect to~$\xi_+$ if the stabilization is of type~I, and with
respect to~$\xi_-$ if the stabilization is of type~II.
\end{lemm}

The proof is again similar to that of Lemma~\ref{wrinkle-creation-isotopy} and is omitted. Note
that now the boundary of the surface is necessarily modified. This modification preserves
the equivalence class of~$\widehat{\partial\Pi}$ as a Legendrian
link with respect to~$\xi_+$ if the stabilization is of type~I, and with respect to~$\xi_-$ if the stabilization is of type~II.

Note also that if~$\Pi\mapsto\Pi'$ is a type~I (respectively, type~II)
stabilization of a rectangular diagram of a surface,
then~$\partial\Pi\mapsto\partial\Pi'$ is a type~I (respectively, type~II) stabilization of a rectangular diagram of a link
(in the generalized sense of~\cite{dyn06}).
Therefore, a~stabilization does always change the equivalence class of~$\widehat\Pi$ as a Giroux's convex surface
with respect to one of the contact structures~$\xi_-$ or~$\xi_+$, depending on whether it is of
type~I or type~II. This can be seen from the fact that the boundary link~$\partial\widehat\Pi$
undergoes a Legendrian stabilization with respect to one of the contact structures,
as well as from the fact that a new arc is added to the respective dividing set.

\begin{rema}
There are two more, naturally defined,
moves of rectangular diagrams of surfaces that induce a stabilization of the boundary.
They are shown in Figure~\ref{two-more-moves-fig}.
\begin{figure}[ht]
\includegraphics{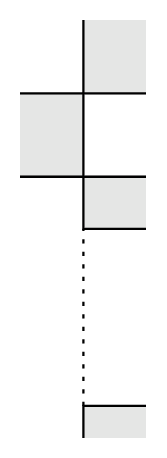}\raisebox{3.5cm}{$\longleftrightarrow$}\includegraphics{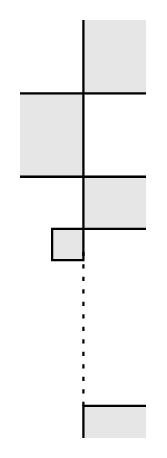}
\hskip2cm
\includegraphics{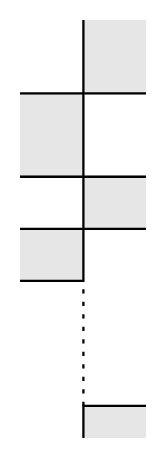}\raisebox{3.5cm}{$\longleftrightarrow$}\includegraphics{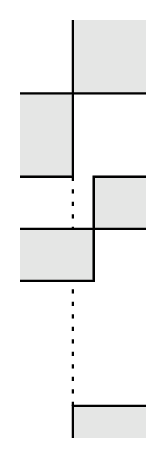}
\caption{Two more ways to stabilize the boundary}\label{two-more-moves-fig}
\end{figure}
The first one is an addition of a rectangle sharing exactly one vertex with the diagram.
The second move consists in a `splitting' of an occupied level containing an edge of
the boundary into two occupied levels close to the original one.
Though these moves look pretty simple we don't include them into the list of basic moves.
We leave it as an exercise to the reader to show explicitly that these moves can be decomposed
into basic moves.
\end{rema}

\subsection{Twisting a `rectangular' surface around the boundary}\label{twisting-subsec}
In \cite[Subsection~2.5]{dp17} we defined a framing of a cusp-free
link~$L$ in~$\mathbb S^3$ as a way, viewed up to smooth isotopy,
to attach a union of annuli (which should be
a surface with corners) to~$L$. Recall that a framing in our sense contains more information
than just the self-linking numbers of the components.
We also introduced in~\cite{dp17} the concept
of a framing of a rectangular diagram of a link as an ordering of each pair of
vertices that forms an edge of the diagram. When~$f$ is a framing we use
the signs~$<_f$ and~$>_f$ to denote the corresponding relation
(note that this is not an ordering, sometimes not even a partial ordering,
on the set all vertices of the diagram).

Let~$(R,f)$ be a framed rectangular diagram of a link, and let~$v$ be a vertex of~$R$.
Let also~$v_1$ and~$v_2$ be vertices of~$R$ such that~$\{v,v_1\}$ is a vertical edge of~$R$,
and~$\{v,v_2\}$ is a horizontal edge.
We call~$v$ \emph{a $\diagup$-vertex} (respectively, a $\diagdown$-vertex) of~$(R,f)$ if
$v_1<_fv<_fv_2$  or $v_1>_fv>_fv_2$ (respectively, $v_1<_fv>_fv_2$ or $v_1>_fv<_fv_2$).

If~$R$ is connected, then for any framing~$f$ of~$R$ there is a unique framing~$f'\ne f$ such
that~$(R,f)$ and~$(R,f')$ have the same set of~$\diagup$-vertices. This framing is defined
by~$u<_{f'}v\Leftrightarrow u>_fv$. We say that such~$f'$ is \emph{opposite} to~$f$.

Every rectangular diagram of a surface~$\Pi$ defines a framing on
the rectangular diagram of a link~$\partial\Pi$ through the rule given
in~\cite[Proposition~3]{dp17}. We call this framing \emph{the boundary
framing of~$\partial\Pi$ induced by~$\Pi$} and denote by~$f^\Pi$.
One can see that~$v$ is a $\diagup$-vertex (respectively, a $\diagdown$-vertex) of~$(\partial\Pi,f^\Pi)$
if and only if~$v$ is a $\diagup$-vertex (respectively, a $\diagdown$-vertex) of~$\Pi$.

When~$\Pi$ is a rectangular diagram of a surface and~$R$ is a rectangular diagram of a link such that~$R\subset\partial\Pi$,
we denote by~$\tb_+(R;\Pi)$ (respectively, by~$\tb_-(R;\Pi)$) the Thurston--Bennequin number~$\tb_+\bigl(\widehat R;\widehat\Pi\bigr)$
(respectively, $\tb_-\bigl(\widehat R;\widehat\Pi\bigr)$ (see \cite[Definitions~16 and~17]{dp17}). These numbers have a very simple combinatorial meaning: $-\tb_+(R;\Pi)$ (respectively,
$-\tb_-(R;\Pi)$) is one half of the number of~$\diagdown$-vertices (respectively, $\diagup$-vertices) of~$\Pi$ in~$R$.

\begin{prop}\label{twist-prop}
Let~$\Pi$ be a rectangular diagram of a surface, and let~$Q$ be a connected
component of the rectangular diagram of a link~$R=\partial\Pi$.
Let also~$F{}\subset\mathbb S^3$ be a compact surface with the following properties:
\begin{enumerate}
\item
$\partial F=\partial\widehat\Pi$;
\item
there is a tubular open neighborhood~$U$ of~$\widehat Q$ such that~$\overline U$ intersects
each of~$F$ and~$\widehat\Pi$ in an annulus, and we have
$F\setminus U=\widehat\Pi\setminus U$;
\item
at every point of~$\widehat Q$ the surface~$F$ is tangent either to~$\xi_+$ or~$\xi_-$.
\end{enumerate}

Then there exists a rectangular diagram of a surface~$\Pi'$ with the following properties:
\begin{enumerate}
\item
$\partial\Pi'=\partial\Pi$;
\item
there is a $C^0$-isotopy~$\phi$ from~$\widehat\Pi$ to~$\widehat\Pi'$
fixed on~$\partial\widehat\Pi$ and composed of two
isotopies: the first, $\phi'$, brings~$\widehat\Pi$ to~$F$, and the second, $\phi''$,
brings~$F$ to~$\widehat\Pi'$, such that the following holds:
\begin{enumerate}
\item
$\phi'$ is fixed outside~$U$;
\item
$\phi''$ is~$C^1$-smooth, it preserves the tangent plane to the surface at every point of~$\partial F$,
and for all~$t\in[0;1]$, $\phi''_t$ is~$C^1$-close to the identity outside~$U$;
\item
if~$\tb_+(Q;\Pi)<0$ (respectively, $\tb_-(Q;\Pi)<0$)
then~$\phi$ brings~$\delta_+$ (respectively, $\delta_-$) to an abstract
dividing set on~$\widehat\Pi'$ equivalent to~$\delta_+'$ (respectively, to~$\delta_-'$),
where~$(\delta_+,\delta_-)$ and~$(\delta_+',\delta_-')$ are canonic dividing configurations
of~$\widehat\Pi$ and~$\widehat\Pi'$, respectively.
\end{enumerate}
\end{enumerate}
\end{prop}

\begin{proof}
It follows from the assumptions of the proposition that there exists a $C^0$-isotopy~$\psi:[0;1]\times\widehat\Pi\rightarrow\mathbb
S^3$ from~$\widehat\Pi$ to~$F$ such that~$\psi_t|_{\widehat\Pi\setminus U}=\mathrm{id}$ for all~$t\in[0;1]$.
Clearly, such an isotopy can be chosen to be $C^1$-smooth outside~$[0;1]\times\bigl(\widehat Q\cap(\mathbb S^1_{\tau=0}\cup
\mathbb S^1_{\tau=1})\bigr)$.

Let~$v_1,v_2,\ldots,v_n$ be all the vertices of~$Q$ numbered in the order they follow on~$Q$ (choose any
orientation of~$Q$ if no one is given), starting from an arbitrarily chosen vertex.
Their indices are regarded modulo~$n$, that is,~$v_{i+n}=v_i$.

For every vertex~$v$ of~$Q$ and each point~$p\in\interior(\widehat v)$, let~$\lambda(p)$
be the signed angle by which the tangent plane~$T_p\psi_t(\widehat\Pi)$ rotates around the arc~$\widehat v$
when~$t$ runs from~$0$ to~$1$.
The function~$\lambda$ can be extended continuously to the whole curve~$\widehat Q$.
Since the surfaces~$\widehat\Pi$ and~$F$ have the same tangent planes at any point~$p\in\widehat Q\cap\bigl(\mathbb S^1_{\tau=0}\cup
\mathbb S^1_{\tau=1}\bigr)$, the value of~$\lambda(p)$ at any such point is an integer multiple of~$\pi$.
Denote by~$k_i$ the integer such that~$\lambda(\text{the end point of }\widehat v_i)=k_i\pi$.

Denote by~$f$ the framing of~$Q$ corresponding to the admissible framing of~$\widehat Q$ induced by~$F$.
For this framing, a vertex~$v\in Q$ has type~`$\diagup$'
if~$F$ is tangent to~$\xi_+$ along~$\widehat v$, and type~`$\diagdown$' if~$F$ is tangent to~$\xi_-$ along~$\widehat v$.

The relative twist of~$\xi_-$ with respect to~$\xi_+$ along~$\widehat v$, $v\in\mathbb T^2$,
is equal to~$\pi$. Therefore,
$$k_i-k_{i-1}=\left\{\begin{aligned}
1&\text{ if }v_i\text{ is a $\diagup$-vertex of }(Q,f^\Pi)\text{ and a $\diagdown$-vertex of }(Q,f),\\
-1&\text{ if }v_i\text{ is a $\diagup$-vertex of }(Q,f)\text{ and a $\diagdown$-vertex of }(Q,f^\Pi),\\
0&\text{ if the types of $v$ as a vertex of }(Q,f)\text{ and }(Q,f^\Pi)\text{ coincide}.
\end{aligned}\right.$$

Now we proceed by induction in the sum~$\sum_{i=1}^n|k_i|$, which will be referred to as
\emph{the framing distance between~$\widehat\Pi$ and~$F$}. The induction base is trivial. Indeed,
if the framing distance between~$\widehat\Pi$ and~$F$ is zero,
then the isotopy~$\psi$ above can be chosen so that the tangent plane to the surface
be fixed at any point of~$\widehat Q$ during the isotopy. We simply take~$\Pi'=\Pi$.

Suppose that~$\sum_{i=1}^n|k_i|>0$. To make the induction step, we show how to find a rectangular
diagram of a surface~$\Pi_1$ such that the surface~$\widehat\Pi_1$ is $C^1$-close to~$\widehat\Pi$
outside~$U$, intersects~$\overline U$ in an annulus,
and has a smaller framing distance from~$F$ than~$\widehat\Pi$ has.
There are several cases to consider.

\medskip\noindent\emph{Case 1}:
there exist~$l\leqslant m$ such that
$k_l=k_{l+1}=\ldots=k_m>0$ and~$k_{l-1}=k_{m+1}=k_l-1$.

By construction, $v_l$ is a $\diagup$-vertex,
and~$v_{m+1}$ is a $\diagdown$-vertex of~$\Pi$. Therefore, there exists~$j\in[l;m]$
such that~$v_j$ is a $\diagup$-vertex, and~$v_{j+1}$ is a $\diagdown$-vertex of~$\Pi$.
By construction, the pair~$(v_j,v_{j+1})$ is an edge of~$Q$.

We can, therefore, apply a half-wrinkle creation move~$\Pi\mapsto\Pi_1$ so that the diagram is modified
in a small neighborhood of the occupied level of~$\Pi$ containing~$v_j$ and~$v_{j+1}$. Moreover,
we can ensure that the boundary of the diagram is not modified, $\partial\Pi_1=\partial\Pi$.
Then~$v_j$ will become a $\diagdown$-vertex, and~$v_{j+1}$ will become a $\diagup$-vertex of~$\Pi_1$.
One can see from this that the passage from~$\Pi$ to~$\Pi_1$ reduces the framing distance to~$F$ by~$1$.

The rectangle added to the diagram~$\Pi$ to obtain~$\Pi_1$
can be chosen
arbitrarily narrow, stretched in the direction of the respective edge of~$Q$. This
ensures that the corresponding tile is contained in~$U$, and that~$\widehat\Pi$ and~$\widehat\Pi_1$
are $C^1$-close to each other outside~$U$.
A $C^0$-isotopy from~$\widehat\Pi$ to~$\widehat\Pi_1$ arising in this way,
which perturbs the surfaces only slightly (in the $C^1$ sense) outside~$U$,
brings a canonic dividing configuration of~$\widehat\Pi$
to a dividing configuration weakly equivalent to a canonic dividing configuration of~$\widehat\Pi_1$
(see Lemma~\ref{half-wrinkle-creation-isotopy}).

\medskip\noindent\emph{Case 2}:
there exist~$l\leqslant m$ such that
$k_l=k_{l+1}=\ldots=k_m<0$ and~$k_{l-1}=k_{m+1}=k_l+1$.

This case is symmetric to the previous one and is left to the reader.

\medskip\noindent\emph{Case 3}:
all~$k_i$ are equal, and we have~$\tb_+(Q;\Pi),\tb_-(Q;\Pi)<0$.

If all~$k_i$ are positive, we find~$j$ such that~$v_j$ is a $\diagup$-vertex, and~$v_{j+1}$ is a $\diagdown$-vertex of~$\Pi$,
and then proceed as in Case~1. If all~$k_i$'s are negative, we
find~$j$ such that~$v_j$ is a $\diagdown$-vertex, and~$v_{j+1}$ is a $\diagup$-vertex of~$\Pi$,
and then proceed symmetrically to Case~1.

\medskip\noindent\emph{Case 4}:
all~$k_i$ are equal, and we have~$\tb_+(Q;\Pi)=0$.

In this case, $Q$ consists only of $\diagup$-vertices of~$\Pi$.

For an~$\varepsilon>0$ such that~$\varepsilon$ is smaller than the distance between any two parallel occupied levels of~$\Pi$,
we define rectangular diagrams of a surface~$\Pi_\varepsilon$ and~~$\Pi_{-\varepsilon}$ as follows.

To obtain~$\Pi_\varepsilon$, first, shift all occupied levels, vertical and horizontal, containing vertices of~$Q$ by~$-\varepsilon$.
Then, for any~$u=(\theta_1,\varphi_1)$ and~$v=(\theta_2,\varphi_2)$ such that~$\{u,v\}$ is an edge of~$Q$
and~$u<_{f^\Pi}v$, add the rectangle~$[\theta_1-\varepsilon,\theta_1]\times[\varphi_2,\varphi_1-\varepsilon]$
to the diagram
if~$\{u,v\}$ is a vertical edge of~$Q$, and the rectangle~$[\theta_2,\theta_1-\varepsilon]\times[\varphi_1-\varepsilon,\varphi_1]$
if~$\{u,v\}$ is a horizontal edge of~$Q$.
\begin{figure}[ht]
\includegraphics[scale=0.7]{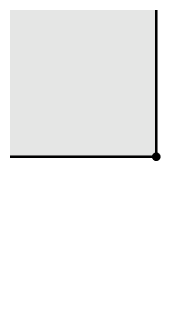}\put(-14,44){$v$}
\raisebox{48pt}{$\longrightarrow$}
\includegraphics[scale=0.7]{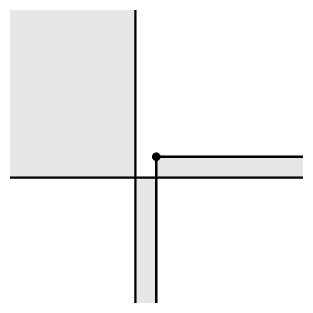}\put(-54,57){$v$}
\hskip1cm
\includegraphics[scale=0.7]{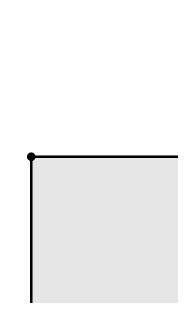}\put(-55,57){$v$}
\raisebox{48pt}{$\longrightarrow$}
\includegraphics[scale=0.7]{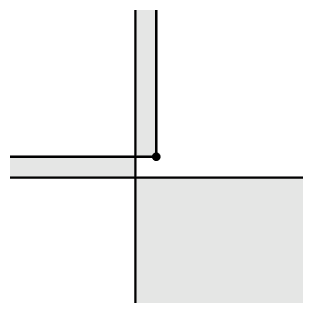}\put(-50,50){$v$}
\caption{Twisting~$\widehat\Pi$ around~$\widehat Q\subset\partial\widehat\Pi$ in the case~$\tb_+(Q;\Pi)=0$}\label{twist-tb=0}
\end{figure}
Figure~\ref{twist-tb=0} shows how the diagram is changed near a vertex~$v$ of~$Q$.

Similarly, to obtain~$\Pi_{-\varepsilon}$, first, shift all occupied levels, vertical and horizontal, containing vertices of~$Q$ by~$\varepsilon$.
Then for any~$u=(\theta_1,\varphi_1)$ and~$v=(\theta_2,\varphi_2)$ such that~$\{u,v\}$ is an edge of~$Q$
and~$u<_{f^\Pi}v$, add the rectangle~$[\theta_1,\theta_1+\varepsilon]\times[\varphi_2+\varepsilon,\varphi_1]$
to the diagram
if~$\{u,v\}$ is a vertical edge of~$Q$, and the rectangle~$[\theta_2+\varepsilon,\theta_1]\times[\varphi_1,\varphi_1+\varepsilon]$
if~$\{u,v\}$ is a horizontal edge of~$Q$.

One can see that if~$\varepsilon$ is small, then each of the surfaces~$\widehat\Pi_\varepsilon$
and~$\widehat\Pi_{-\varepsilon}$ is $C^1$-close to a surface
obtained from~$\widehat\Pi$ by attaching a narrow collar along~$\widehat Q$, and this collar has the form
of a union of tiles which correspond to the narrow rectangles added to the diagram.
For one of the passages~$\Pi\mapsto\Pi_\varepsilon$ and~$\Pi\mapsto\Pi_{-\varepsilon}$
all~$k_i$'s are incremented by~$1$, and for the other, all~$k_i$'s are dropped by~$1$.
So, the framing distances from the corresponding surfaces to~$F$ are changed by~$n$ and~$-n$, respectively.
We choose for~$\Pi_1$ the one of the two diagrams~$\Pi_\varepsilon$ and~$\Pi_{-\varepsilon}$ from which the framing distance to~$F$ is smaller. This choice depends on the orientation of~$Q$
and the framing~$Q^\Pi$.

Denote by~$(\delta^0_+,\delta^0_-)$ a canonic dividing configuration of the collar
added to~$\widehat\Pi$ to obtain~$\widehat\Pi_1$.
One can see that~$\delta^0_+$ is a separating simple closed curve, whereas~$\delta^0_-$
consists of non-separating arcs that prolong some connected components of~$\delta_-$.
Therefore, if~$(\delta_{1+},\delta_{1-})$ is a canonic dividing configuration of~$\widehat\Pi_1$,
then a $C^0$-small deformation that takes~$\widehat\Pi$ to~$\widehat\Pi_1$
also takes~$\delta_-$ to an abstract dividing set equivalent to~$\delta_{1-}$.

\medskip\noindent\emph{Case 5}:
all~$k_i$ are equal, and we have~$\tb_-(Q;\Pi)=0$.

This case is symmetric to Case~4.

The induction step follows.\end{proof}

\begin{rema}
In addition to the assertion of Proposition~\ref{twist-prop}, one can show
by means of the general theory of Giroux's convex surfaces that there
exists an isotopy from~$\widehat\Pi$ to~$\widehat\Pi'$ in the class of Giroux's
convex surfaces with respect to~$\xi_+$ if~$\tb_+(Q;\Pi)<0$, and with respect to~$\xi_-$ if~$\tb_-(Q;\Pi)<0$.
If both conditions hold, this fact also follows from Lemmas~\ref{wrinkle-creation-isotopy} and~\ref{half-wrinkle-creation-isotopy}
and the reasoning in the proof of Proposition~\ref{twist-prop}. If only one of the conditions holds,
one can construct a decomposition of the collar addition illustrated in Figure~\ref{twist-tb=0} into
basic moves excluding type~I (respectively, type~II) moves if~$\tb_+(Q;\Pi)=0$
(respectively, if~$\tb_-(Q;\Pi)=0$), and then apply the respective lemmas from this section.
\end{rema}

\subsection{Exchange moves}

Lemmas~\ref{bigon-creation-lemm} and~\ref{half-bigon-creation-lemm} say, vaguely speaking,
that any (half-)bigon creation in a canonic dividing configuration can be realized by applying a (half-)wrinkle move
to the corresponding diagram. Here we consider the inverse operations, (half-)bigon reductions, which require
more care.

If a canonic dividing configuration has a bigon or a half-bigon, this does not necessarily imply that
the (half-)bigon can be immediately reduced by a (half-)wrinkle move. One of the obstructions
comes from the fact that the two new rectangles created as a result of a wrinkle creation
(or one rectangle in the case of a half-wrinkle creation) are thin in the sense of the following definition.

\begin{defi}
A rectangle $r=[\theta_1;\theta_2]\times[\varphi_1;\varphi_2]$ of a rectangular diagram of a surface~$\Pi$
is called \emph{vertically thin} (respectively, \emph{horizontally thin})
if there are no vertical occupied levels of~$\Pi$ in the domain~$\theta\in(\theta_1;\theta_2)$
(respectively, in the domain~$\varphi\in(\varphi_1;\varphi_2)$).
\end{defi}

If a canonic dividing configuration of~$\widehat\Pi$ has a bigon, but the rectangles
corresponding to the bigon's corners are not thin, a wrinkle reduction reducing this bigon is not applicable to~$\Pi$.
Whether or not a rectangle is thin cannot, in general, be read off the combinatorial structure
of the canonic dividing configuration. However, some rectangles can be \emph{made}
thin by exchange moves, which are defined below, without modification of
that structure. Exchange moves are, by definition, neutral.

\begin{defi}\label{exchange-def}
Let $\Pi$ be a rectangular diagram of a surface, and let $\theta_1,\theta_2,\theta_3,\varphi_1,\varphi_2\in\mathbb S^1$ be such
that:
\begin{enumerate}
\item
we have $\theta_2\in(\theta_1;\theta_3)$;
\item
the rectangles $r_1=[\theta_1;\theta_2]\times[\varphi_1;\varphi_2]$ and $r_2=[\theta_2;\theta_3]\times[\varphi_2;\varphi_1]$
contain no vertices of $\Pi$;
\item
the vertices of $r_1$ and $r_2$ are disjoint from the rectangles of~$\Pi$.
\end{enumerate}
Let $f:\mathbb S^1\rightarrow\mathbb S^1$ be a map that is identical on $[\theta_3;\theta_1]$, and exchanges
the intervals $(\theta_1;\theta_2)$ and $(\theta_2;\theta_3)$:
$$f(\theta)=\left\{\begin{aligned}\theta-\theta_2+\theta_3,&\text{ if }\theta\in(\theta_1,\theta_2],\\
\theta-\theta_2+\theta_1,&\text{ if }\theta\in(\theta_2,\theta_3).
\end{aligned}\right.$$
Choose a self-homeomorphism~$\psi$ of~$\mathbb S^1$ identical on~$[\theta_3;\theta_1]$, and let
$$\Pi'=\bigl\{[\\\psi(f(\theta'));\psi(f(\theta''))]\times[\varphi';\varphi'']:[\theta';\theta'']\times[\varphi';\varphi'']\in\Pi\bigr\}.$$
Then we say that the passage from $\Pi$ to $\Pi'$, or the other way, is \emph{a vertical exchange move}.
An example is shown in Figure~\ref{exchangemovefig}.

\emph{A horizontal exchange move} is defined similarly with the roles of $\theta$ and $\varphi$ exchanged.
\end{defi}
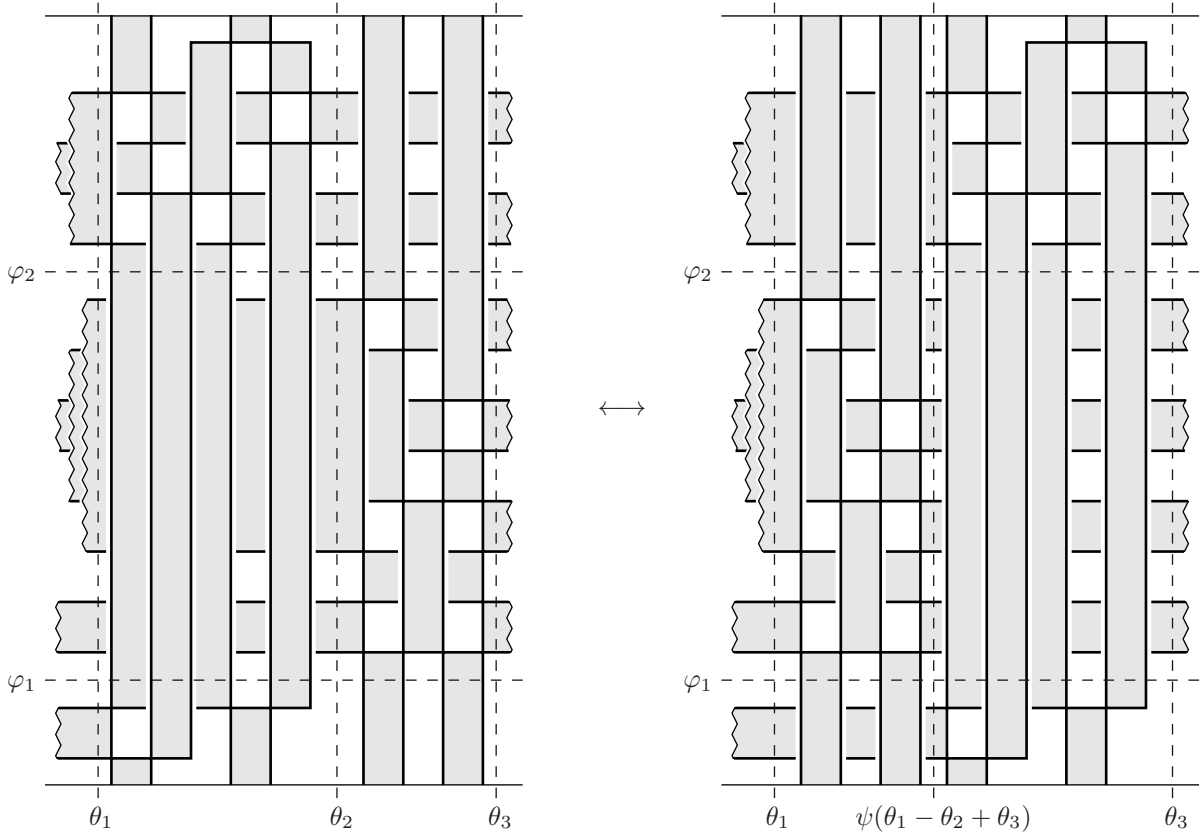
\begin{figure}[ht]
\includegraphics{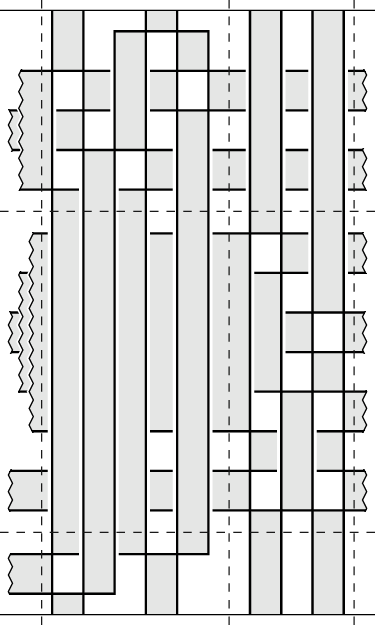}\put(-164,-10){$\theta_1$}\put(-73,-10){$\theta_2$}\put(-13,-10){$\theta_3$}%
\put(-195,42){$\varphi_1$}\put(-195,197){$\varphi_2$}
\hskip1cm\raisebox{145pt}{$\longleftrightarrow$}\hskip1cm
\includegraphics{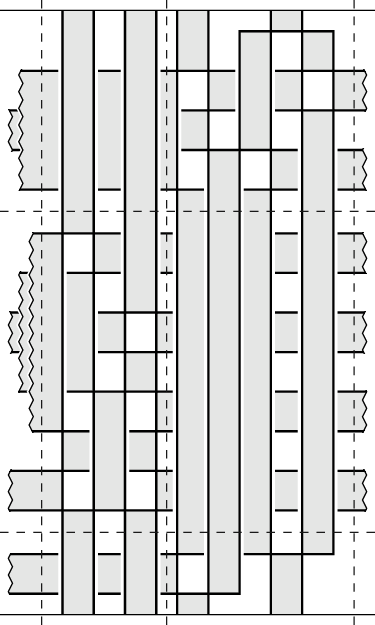}\put(-164,-10){$\theta_1$}\put(-130,-10){$\psi(\theta_1-\theta_2+\theta_3)$}\put(-13,-10){$\theta_3$}%
\put(-195,42){$\varphi_1$}\put(-195,197){$\varphi_2$}
\caption{An exchange move}\label{exchangemovefig}
\end{figure}

The following lemma is obvious.

\begin{lemm}\label{exchange-means-exchange-lem}
If~$\Pi\mapsto\Pi'$ is an exchange move of rectangular diagrams of a surface,
then the rectangular diagrams of a link~$\partial\Pi$ and~$\partial\Pi'$ are connected
by a finite sequence of exchange moves.
\end{lemm}

\begin{lemm}\label{exchange-move-equivalence-lem}
Let $\Pi\mapsto\Pi'$ be an exchange move, and let~$D,D'$ be canonic dividing
configurations of~$\widehat\Pi$ and~$\widehat\Pi'$, respectively.
Then there exists an isotopy bringing~$(\widehat\Pi,D)$ to~$(\widehat\Pi',D')$.
Moreover, $\widehat\Pi$ and~$\widehat\Pi'$ are isotopic in the class of
Giroux's convex surfaces with respect to either of the contact structures~$\xi_+$ and~$\xi_-$.
\end{lemm}

\begin{proof}
Due to symmetry it suffices to show the existence of an isotopy from~$(\widehat\Pi,D)$
to~$(\widehat\Pi',D')$ in the class of Giroux's convex surfaces with respect to~$\xi_+$.
We may assume without loss of generality that~$0\notin[\theta_1;\theta_3]$,
and treat the $\theta$-coordinate of any point from~$[\theta_1;\theta_3]*\mathbb S^1$
as a real number in the interval~$(0,2\pi)$.

We use the notation from Definition~\ref{exchange-def}.
Denote also by~$\widetilde{\mathscr G}$ the $1$-skeleton of the tiling of~$\widehat\Pi$,
that is,
$$\widetilde{\mathscr G}=\bigcup_{v\text{ is a vertex of }\Pi}\widehat v.$$
Similarly, let~$\widetilde{\mathscr G}'$ be the $1$-skeleton of the tiling of~$\widehat\Pi'$.

Let~$m_{\theta^1},\ldots,m_{\theta^k}$ be the occupied meridians of~$\Pi$ in the interval~$\theta\in[\theta_1;\theta_3]$,
listed in the increasing order, $\theta_1<\theta^1<\theta^2<\ldots<\theta^k<\theta_3$.
Denote~$\psi(f(\theta^i))$ by~${\theta'}^i$.
The points~$p_i=\widehat m_{\theta^i}$ and~$p_i'=\widehat m_{{\theta'}^i}$ are vertices of
$\widetilde{\mathscr G}$ and~$\widetilde{\mathscr G}'$, respectively, $i=1,\ldots,k$.
All the other vertices of~$\widetilde{\mathscr G}$
are the same as those of~$\widetilde{\mathscr G}'$.

Let~$\phi:[0;1]\times\widetilde{\mathscr G}\rightarrow\mathbb S^3$ be a homotopy such that:
\begin{enumerate}
\item
for all~$t\in[0,1]$ the image~$\phi_t(\widetilde{\mathscr G})$ has the form~$\widehat X$, $X\subset\mathbb T^2$;
\item
$\phi_0=\mathrm{id}|_{\widetilde{\mathscr G}}$\,;
\item
$\phi_t(p_i)=\widehat m_{(1-t)\theta^i+t{\theta'}^i}$, $i=1,\ldots,k$, $t\in[0;1]$;
\item
$\phi_t(p)=p$, $t\in[0;1]$, if $p\in\widetilde{\mathscr G}\cap(\mathbb S^1_{\tau=0}\cup\mathbb S^1_{\tau=1})$
is not in~$\{p_1,\ldots,p_k\}$.
\end{enumerate}
Such a homotopy is clearly unique.

The homotopy~$\phi$ is not an isotopy if~$\theta^1\in(\theta_1;\theta_2)$ and~$\theta^k\in(\theta_2;\theta_3)$,
but it is not very far from being an isotopy. For any~$t\in[0;1]$ the map~$\phi_t$ is an immersion and there are
only finitely many moments~$t$ when it is not an embedding. Namely, this occurs when~$\phi_t(p_i)=\phi_t(p_j)$
for some~$i$, $j$, $1\leqslant i<j\leqslant k$. In this case, we necessarily have~$\theta^i\in(\theta_1;\theta_2)$
and~$\theta^j\in(\theta_2;\theta_3)$.

For all~$t\in[0;1]$ the images of all edges of~$\widetilde{\mathscr G}$ incident to the vertex~$p_i$ under the map~$\phi_t$
are contained in the domain~$\mathbb S^1_{\tau=1}*(\varphi_2;\varphi_1)$ if~$\theta^i\in[\theta_1;\theta_2]$, and in the domain~$\mathbb S^1_{\tau=1}*(\varphi_1;\varphi_2)$ if~$\theta^i\in[\theta_2;\theta_3]$.
We can disturb~$\phi_t$ slightly so that for all~$t\in(0;1)$ the image~$\phi_t(e)$ of any edge~$e$ of~$\widetilde{\mathscr G}$
remains a Legendrian arc (with respect to~$\xi_+$), and the point~$\phi_t(p_i)$ is contained in the domain~$\varphi\in(\varphi_2;\varphi_1)$, $\tau<1$
if~$\theta^i\in(\theta_1;\theta_2)$, and in the domain~$\varphi\in(\varphi_1;\varphi_2)$, $\tau<1$ if~$\theta^i\in(\theta_2;\theta_3)$.

Then for each~$t\in[0;1]$ the map~$\phi_t$ becomes an embedding, and~$\phi$ becomes an isotopy from~$\widehat{\mathscr G}$
to~$\widehat{\mathscr G}'$ through Legendrian graphs.

The isotopy~$\phi$ can be extended to the hole surface~$\widehat\Pi$ so that~$\phi_1(\widehat\Pi)=\widehat\Pi'$.
Moreover, for any~$t_0\in[0;1]$ the extension can be chosen so
that the surface~$\phi_t(\widehat\Pi)$ is convex at the moment~$t=t_0$. This follows from the fact that convexity is a generic
property. It also follows that the surface~$\phi_t(\widehat\Pi)$ will remain convex for~$t$ close enough to~$t_0$.

Since the interval~$[0;1]$ is compact we can find finitely many points~$0=t_0<t_1<\ldots<t_n=1$ and
extensions~$\phi^1,\ldots,\phi^n$ of the isotopy~$\phi$ to~$\widehat\Pi$ such that the surface~$\phi^i_t(\widehat\Pi)$
is convex for~$t\in[t_{i-1};t_i]$. Denote by~$F_{i-1}$ the surface~$\phi^i_{t_{i-1}}(\widehat\Pi)$, and by~$F_i'$ the surface~$\phi^i_{t_i}(\widehat\Pi)$. So far we have isotopies from~$\widehat\Pi=F_0$ to~$F_1'$,
from~$F_1$ to~$F_2'$, \ldots, from~$F_{n-1}$ to~$F_n'=\widehat\Pi'$ through convex surfaces.
We also know that, for each~$i=1,\ldots,n-1$, the convex surfaces~$F_i'$ and~$F_i$ are isotopic
relative to~$\widetilde{\mathscr G}_i=\phi_{t_i}(\widetilde{\mathscr G})$.
By construction, the connected components of~$F_i'\setminus\widetilde{\mathscr G}_i$ and~$F_i\setminus\widetilde{\mathscr G}_i$
are open discs intersecting any dividing set of the respective surface in a single arc. By using~\cite[Lemma~6]{dp17}
one shows that there is an isotopy from~$F_i'$ to~$F_i$ through convex surfaces.

The constructed isotopy brings every tile of~$\widehat\Pi$ to a tile of~$\widehat\Pi'$, hence it brings~$D$
to a dividing configuration on~$\widehat\Pi'$ equivalent to~$D'$. The claim follows.
\end{proof}

\begin{lemm}\label{bigon-reduction-lemm}
Let~$\Pi$ be a rectangular diagram of  a surface, and let~$(\delta_+,\delta_-)$ be a canonic dividing configuration
of~$\widehat\Pi$. Let also~$(\delta'_+,\delta'_-)$ be an
admissible dividing configuration obtained from~$(\delta_+,\delta_-)$
by a (half-)bigon reduction. Then there exists a proper realization~$(\Pi',\phi)$ of~$(\delta'_+,\delta'_-)$
such that~$\Pi'$ can be obtained from~$\Pi$ by a sequence consisting of two or less exchange moves and
one (half-)wrinkle reduction.
\end{lemm}

\begin{proof}
Consider the case of a bigon reduction. The bigon of~$(\delta_+,\delta_-)$ that
is being reduced must be bounded
by two arcs of the form~$\mathring v_1$ and~$\mathring v_2$, where~$v_1$ is a $\diagup$-vertex
and~$v_2$ is a $\diagdown$-vertex of~$\Pi$. The common endpoints of~$\mathring v_1$
and~$\mathring v_2$ must have the form~$\mathring r_1$, $\mathring r_2$, where~$r_1$ and~$r_2$
are two rectangles of~$\Pi$ that share the vertices~$v_1$, $v_2$.

Without loss of generality we may assume that~$v_1$ and~$v_2$ lie at the same vertical occupied level of~$\Pi$,
which is~$\theta=\widetilde\theta_2$, and the rectangles~$r_1$ and~$r_2$
have form~$[\widetilde\theta_1;\widetilde\theta_2]\times[\widetilde\varphi_1;\widetilde\varphi_2]$
and~$[\widetilde\theta_2;\widetilde\theta_3]\times[\widetilde\varphi_2;\widetilde\varphi_1]$, respectively.
(Note that the situation~$\widetilde\theta_1=\widetilde\theta_3$ is impossible as otherwise~$\widehat r_1\cup\widehat r_2$
would be a sphere containing exactly two intersection points of~$\delta_+$ and~$\delta_-$, in which case
the reduction of the bigon would produce an inadmissible dividing configuration.)

Pick an~$\varepsilon>0$ smaller than one half of the distance between any two parallel occupied levels of~$\Pi$. Then~$\Pi$ satisfies
the conditions of Definition~\ref{exchange-def} after putting~$\theta_1=\widetilde\theta_1+\varepsilon$,
$\theta_2=\widetilde\theta_2-\varepsilon$, $\theta_3=\widetilde\theta_3+\varepsilon$, $\varphi_1=\widetilde\varphi_1-\varepsilon$,
$\varphi_2=\widetilde\varphi_2+\varepsilon$ (consult Figure~\ref{exchange-to-remove-wrinkle}), and
we can apply the respective exchange move (we put $\psi=\mathrm{id}$).
\begin{figure}[ht]
\
\includegraphics[scale=0.7]{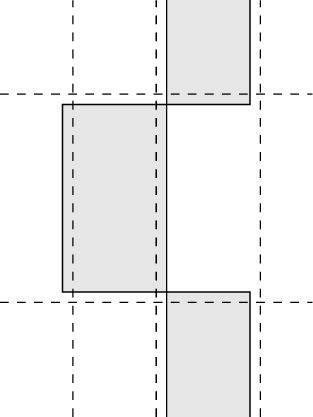}%
\put(-126,37){$\scriptstyle\widetilde\varphi_1-\varepsilon$}%
\put(-126,107){$\scriptstyle\widetilde\varphi_2+\varepsilon$}%
\put(-92,-10){$\scriptstyle\widetilde\theta_1+\varepsilon$}%
\put(-64,-10){$\scriptstyle\widetilde\theta_2-\varepsilon$}%
\put(-28,-10){$\scriptstyle\widetilde\theta_3+\varepsilon$}%
\put(-70,70){$r_1$}\put(-38,15){$r_2$}\put(-70,125){$X$}\put(-70,15){$X$}\put(-38,70){$Y$}
\hskip0.3cm\raisebox{68pt}{$\longrightarrow$}\hskip0.7cm
\includegraphics[scale=0.7]{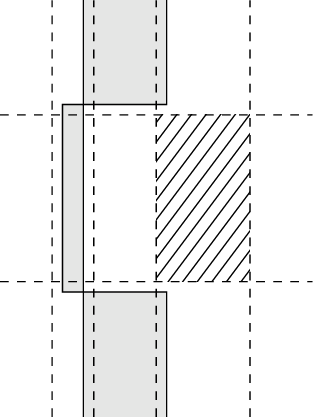}
\put(-126,44){$\scriptstyle\widetilde\varphi_1+\varepsilon$}%
\put(-126,100){$\scriptstyle\widetilde\varphi_2-\varepsilon$}%
\put(-99,-10){$\scriptstyle\widetilde\theta_1-\varepsilon$}%
\put(-64,-10){$\scriptstyle\widetilde\theta_4+\varepsilon$}%
\put(-25,-10){$\scriptstyle\widetilde\theta_3$}%
\put(-40,125){$X$}\put(-40,15){$X$}\put(-65,70){$Y$}
\hskip0.3cm\raisebox{68pt}{$\longrightarrow$}\hskip0.7cm
\includegraphics[scale=0.7]{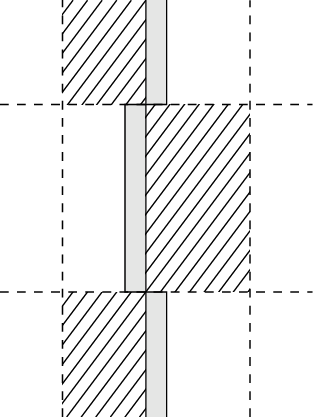}%
\put(-116,40.5){$\scriptstyle\widetilde\varphi_1$}%
\put(-116,103.5){$\scriptstyle\widetilde\varphi_2$}%
\put(-88,-10){$\scriptstyle\widetilde\theta_1$}%
\put(-60,-10){$\scriptstyle\widetilde\theta_4$}%
\put(-25,-10){$\scriptstyle\widetilde\theta_3$}%
\put(-40,125){$X$}\put(-40,15){$X$}\put(-77,70){$Y$}
\caption{Exchange moves in a bigon reduction procedure. The blocks of vertices $X$ and $Y$
move horizontally. The hatched areas are free of vertices of the diagram}\label{exchange-to-remove-wrinkle}
\end{figure}
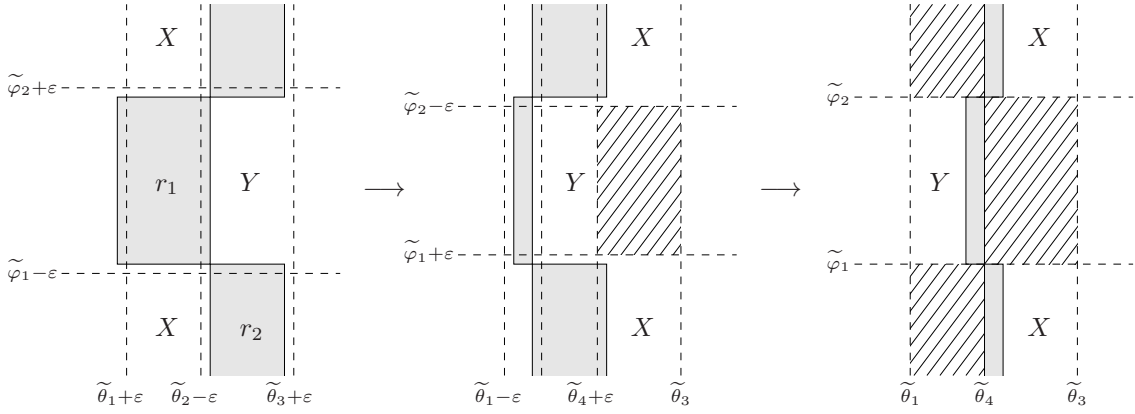

As a result of the move the rectangles~$r_1$ and~$r_2$ are transformed into $[\widetilde\theta_1;\widetilde\theta_1+2\varepsilon]\times
[\varphi_1;\varphi_2]$ and~$[\widetilde\theta_1+2\varepsilon;\widetilde\theta_4+2\varepsilon]\times[\varphi_2;\varphi_1]$,
respectively, where~$\widetilde\theta_4=\widetilde\theta_1-\widetilde\theta_2+\widetilde\theta_3$.

The obtained diagram satisfies the conditions of Definition~\ref{exchange-def} for~$\theta_1=\widetilde\theta_1-\varepsilon$,
$\theta_2=\widetilde\theta_1+3\varepsilon$, $\theta_3=\widetilde\theta_4+\varepsilon$,
$\varphi_1=\widetilde\varphi_1+\varepsilon$, $\varphi_2=\widetilde\varphi_2-\varepsilon$.
After applying the respective exchange move the two rectangles that~$r_1$ and~$r_2$
are transformed to are thin.

As a result of the two exchange moves the combinatorial structure of the canonic configuration of~$\Pi$ is unaltered. Now the
bigon can be reduced by a wrinkle reduction move, provided that the move is applicable. This is the case if and only if
at most one of the meridians~$m_{\widetilde\theta_1}$ and~$m_{\widetilde\theta_3}$ contains vertices of~$\partial\Pi$, which is equivalent to saying that the bigon reduction yields an admissible dividing configuration.
Indeed, the admissibility of the configuration obtained by the bigon reduction can be violated
only at the region that is obtained by merging the discs~$\mathring m_{\widetilde\theta_1}$
and~$\mathring m_{\widetilde\theta_3}$, and this happens if and only if both
discs have a non-empty intersection with~$\partial\widehat\Pi$.

In the case of a half-bigon reduction the proof is similar and left to the reader.
\end{proof}

As one can see from the above proof, the hypothesis that a bigon reduction yields an admissible dividing configuration
is essential in Lemma~\ref{bigon-reduction-lemm}. The following example shows that this hypothesis does
not always hold.

\begin{exam}\label{bigon-exam}
The dividing configuration in the left picture of Figure~\ref{divconf} contains a bigon (with corners marked~$1$ and~$2$) that cannot be reduced. The respective rectangles (see Figure~\ref{seifert}) are horizontally thin but
a wrinkle reduction is not applicable since both rectangles contain boundary vertices of
the diagram.
\end{exam}

\begin{lemm}\label{non-reducible-lem}
Let~$D=(\delta_+,\delta_-)$ be  an admissible dividing configuration on a surface~$F$, and let~$b$
be a bigon of~$\delta_+$ and~$\delta_-$ such that the reduction of~$b$ produces a non-admissible
dividing configuration. Let also~$F_0$ be the connected component of~$F$ containing~$b$. Then at least one of the following holds:
\begin{enumerate}
\item
for any realization~$(\Pi,\phi)$ of~$D$, the rectangular diagram of a link~$\partial\Pi$ admits an exchange move;
\item
for any realization~$(\Pi,\phi)$ of~$D$, the boundary~$\partial\Pi$
represents an unknot;
\item
$F_0$ contains exactly two intersection points of~$\delta_+$
and~$\delta_-$, and~$F_0$ is homeomorphic to the sphere~$\mathbb S^2$.
\end{enumerate}
\end{lemm}

\begin{proof}
Let~$(\Pi,\phi)$ be a realization of~$D$, and let~$r_1$, $r_2$ be the rectangles corresponding
to the points in~$\delta_+\cap\delta_-\cap\partial b$.
Suppose that~$\widehat r_1\cup\widehat r_2$ is not homeomorphic to~$\mathbb S^2$ (otherwise Case~(3) occurs).
Then~$r_1$ and~$r_2$ share exactly two vertices. Let~$x$ be the occupied level containing these vertices. The rectangles~$r_1$ and~$r_2$
can be made thin by exchange moves on~$\Pi$. This may already
result in an exchange move on~$\partial\Pi$, in which case we are done. Suppose otherwise.

Since the reduction of~$b$ produces a non-admissible dividing configuration,
both occupied levels parallel to~$x$ that contain unshared vertices of~$r_1$ and~$r_2$ contain edges of~$\partial\Pi$.
Let~$x_1$, $x_2$ be these occupied levels, and let~$\{v_1,v_2\}=\partial\Pi\cap x_1$,
$\{v_3,v_4\}=\partial\Pi\cap  x_2$.
Denote by $y_1,y_2,y_3,y_4$ the occupied levels of~$\Pi$ perpendicular to~$x$ and passing through~$v_1,v_2,v_3,v_4$, respectively.

If~$y_1$, $y_2$, $y_3$, $y_4$ are pairwise distinct,
then the edges~$\{v_1,v_2\}$, $\{v_3,v_4\}$ can be exchanged.
Suppose that some of~$y_1$, $y_2$, $y_3$, $y_4$ coincide.

An edge~$\{v',v''\}$ of a rectangular diagram of a link~$R$ will be said to be~\emph{short}
if there is an annulus~$A\subset\mathbb T^2$ of the form~$[\theta_1;\theta_2]\times\mathbb S^1$
or~$\mathbb S^1\times[\varphi_1;\varphi_2]$ such that~$v'$ and~$v''$ lie on different components of~$\partial A$,
and there are no vertices of~$R$ in~$\interior(A)$. Such an edge can
always be exchanged with one of the neighboring parallel edges unless
the other two edges contained in~$\partial A$ are also short.

Without loss of generality we may assume~$y_1=y_3$.
Then~$\{v_1,v_3\}$ is a short edge of~$\partial\Pi$, which means that
either~$\partial\Pi$ admits an exchange move
or all edges of the connected component of~$\partial\Pi$ containing~$v_1$ are short. In the latter case one can see that~$\partial\Pi$ represents an unknot.
\end{proof}

\subsection{Weak equivalence of dividing configurations via neutral basic moves}

\begin{prop}\label{weak-equivalence-via-moves}
Let $F\subset\mathbb S^3$ be a compact surface, and let~$D=(\delta_+,\delta_-)$ be a dividing configuration
on~$F$. Let also~$(\Pi,\phi)$ be a proper realization of~$D$.
Then for any admissible dividing configuration~$D'=(\delta_+',\delta_-')$ on~$F$ weakly equivalent to~$D$, there exists a proper realization $(\Pi',\phi')$ of~$D'$ such that $\Pi'$ can be obtained
from~$\Pi$ by finitely many wrinkle, half-wrinkle, and exchange moves.
\end{prop}

\begin{proof}
Due to Lemmas~\ref{bigon-creation-lemm}, \ref{half-bigon-creation-lemm}, and~\ref{bigon-reduction-lemm} it suffices
to show that~$D'$ can be obtained from~$D$ by a sequence of isotopies and
(half-)bigon reductions and creations so
that every step of the process results in an admissible dividing configuration.

We may assume without loss of generality that the surface~$F$ is connected. Otherwise
the procedure below should be repeated for each connected component of~$F$.

Suppose that~$F$ is not homeomorphic to the sphere~$\mathbb S^2$.

The admissibility of a dividing configuration is preserved under isotopies and
(half-)bigon creations, so our strategy is to modify~$D$ and~$D'$ by isotopies
and (half-)bigon creations so that they eventually coincide.

We start by applying an isotopy to either~$D$ or~$D'$ so as to have~$\delta_+=\delta_+'$.
We assume this equality to hold from now on. By a small perturbation of~$\delta_+$
we can make it transverse to~$\delta_-$ and~$\delta_-'$.

As follows from the results of~\cite{giroux01}
if~$\delta$ is a realizable abstract dividing set on~$F\not\cong\mathbb S^2$, then any closed connected
component of~$\delta$
is homotopically non-trivial in~$F$. However, a connected component of~$\delta$
can be an arc that cuts off a half-disc from~$F$. We call such an arc \emph{trivial}.
We will now reduce the general case to the case when~$\delta_-$ has no trivial arcs.

To this end, for each arc of~$\delta_-'$, we pull its endpoints along~$\partial F$ toward the respective points
of~$\partial\delta_-$ so as to have~$\partial\delta_-=\partial\delta_-'$ and
to make~$\delta_-$ and~$\delta_-'$ isotopic relative to~$\partial\delta_-$.
The idea is illustrated in Figure~\ref{pulling-endpoints}.
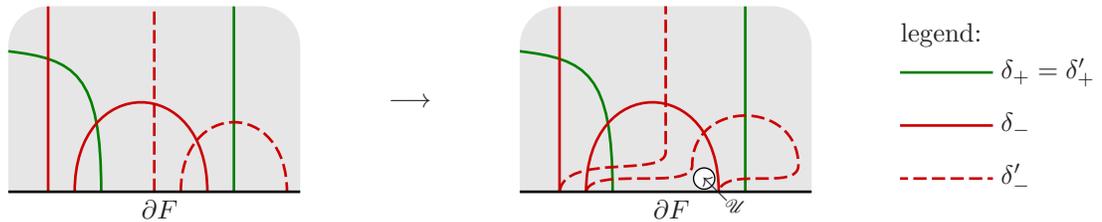
\begin{figure}[ht]
\includegraphics{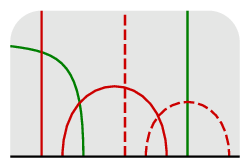}\put(-65,-5){$\partial F$}
\hskip1cm\raisebox{37pt}{$\longrightarrow$}\hskip1cm
\includegraphics{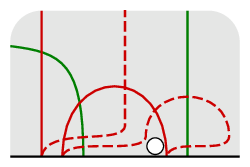}\put(-65,-5){$\partial F$}\put(-38,-3){$\scriptstyle\mathscr U$}\put(-46.5,3.5){$\nwarrow$}
\hskip1cm
\includegraphics{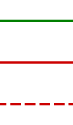}\put(-35,62){legend:}\put(3,8){$\delta_-'$}\put(3,28){$\delta_-$}\put(3,48){$\delta_+=\delta_+'$}
\caption{Pulling the endpoints of~$\delta_-'$ toward the respective points of~$\partial\delta_-$}\label{pulling-endpoints}
\end{figure}
This can be done so that~$D'$ undergoes only isotopies and half-bigon creations.

Now there is a union~$\mathscr U\subset F$ of finitely many open discs such that
\begin{enumerate}
\item
the closure~$\overline{\mathscr U}$ of~$\mathscr U$ is
disjoint from~$\delta_+$, $\delta_-$, and~$\delta_-'$;
\item
$D$ and~$D'$ are still weakly equivalent
on~$F\setminus\mathscr U$;
\item
no arc in~$\delta_-$ is trivial in~$F\setminus\mathscr U$.
\end{enumerate}
The discs can be chosen in a small vicinity of the common endpoints of~$\delta_-$ and~$\delta_-'$.

In what follows we will apply isotopies to~$\delta_\pm$ and~$\delta_\pm'$ so that during
the isotopies the curves will remain disjoint from~$\overline{\mathscr U}$.

By a small perturbation of~$\delta_-'$ we make it transverse to~$\delta_-$, and
make~$\partial\delta_-'$ disjoint from~$\partial\delta_-$.
Let~$b$ be a (half-)bigon of~$(\delta_-,\delta_-')$ disjoint from~$\mathscr U$.
By applying, if necessary, an isotopy to~$\delta_+=\delta_+'$ that results
in creation of (half-)bigons of~$D$ and/or~$D'$ we
can achieve that the intersection~$\delta_+\cap b$ consists
of arcs having one endpoint at~$\delta_-$ and the other at~$\delta_-'$;
see Figure~\ref{bigon-creation-isotopy}.
\begin{figure}[ht]
\includegraphics{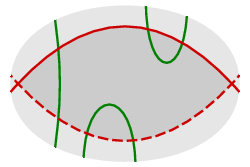}\put(-62,37){$b$}
\hskip0.2cm\raisebox{37pt}{$\longrightarrow$}\hskip0.2cm
\includegraphics{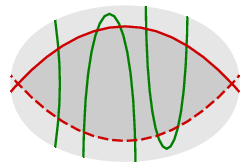}
\hskip0.2cm\raisebox{37pt}{$\longrightarrow$}\hskip0.2cm
\includegraphics{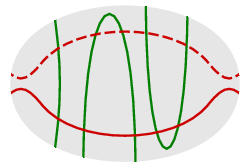}\\[5mm]
\includegraphics{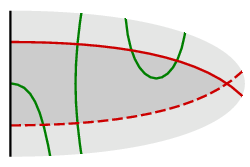}\put(-70,37){$b$}
\hskip0.2cm\raisebox{37pt}{$\longrightarrow$}\hskip0.2cm
\includegraphics{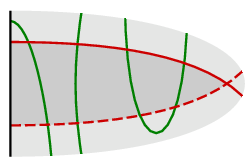}
\hskip0.2cm\raisebox{37pt}{$\longrightarrow$}\hskip0.2cm
\includegraphics{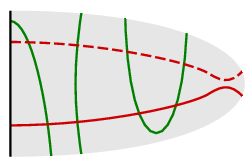}
\caption{Reducing (half-)bigons of~$(\delta_-,\delta_-')$. See Figure~\ref{pulling-endpoints} for the legend}\label{bigon-creation-isotopy}
\end{figure}
Then we can reduce the (half-)bigon~$b$ so that the equivalence classes of~$D$ and~$D'$ remain
unchanged.

Proceeding as above we reduce all bigons and half-bigons of~$(\delta_-,\delta_-')$ contained in~$F\setminus\mathscr U$.
Since there is an isotopy from~$\delta_-$ to $\delta_-'$  in~$F\setminus\mathscr U$ that
keeps~$\partial\delta_-$ in~$\partial F$, after completing the (half-)bigon reductions the abstract dividing sets~$\delta_-$
and~$\delta_-'$ are almost disjoint in the following sense.

{Two abstract dividing sets~$\delta_1,\delta_2\subset F$ are called \emph{almost disjoint}
if, for any connected components~$\gamma_1$, $\gamma_2$ of
$\delta_1$, $\delta_2$, respectively, one of the following holds:
\begin{enumerate}
\item
$\gamma_1$ and~$\gamma_2$ are disjoint;
\item
$\gamma_1$ and~$\gamma_2$ are one-sided closed curves isotopic to one another
and having a single intersection point.
\end{enumerate}}

Let~$\gamma$ be a two-sided closed connected component of~$\delta_-$. Since~$\delta_-$
and~$\delta_-'$ are almost disjoint, and~$\gamma$ is homotopically non-trivial in~$F\setminus\mathscr U$,
there exists an annulus~$A\subset F\setminus\overline{\mathscr U}$ such that its intersection with~$\delta_-\cup\delta_-'$
consists of all the connected components of~$\delta_-$ and~$\delta_-'$
that are isotopic either to~$\gamma$ or~$-\gamma$ (where~$-\gamma$ stands for~$\gamma$ with
reversed orientation) and all these components are homotopically non-trivial in~$A$.

By applying, if necessary, an isotopy to~$\delta_+$ that results
in creation of bigons of~$D$ and/or~$D'$ we can make the
intersection~$\delta_+\cap A$ consist of non-separating
arcs in~$A$ each of which intersects each of the connected components of~$\delta_-$ and~$\delta_-'$
exactly once; see Figure~\ref{annulus-reduction}.
\begin{figure}[ht]
\includegraphics{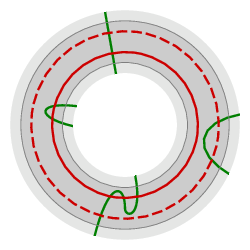}\put(-88,90){$A$}
\hskip0.2cm\raisebox{57pt}{$\longrightarrow$}\hskip0.2cm
\includegraphics{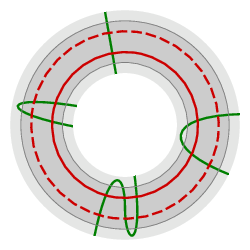}
\hskip0.2cm\raisebox{57pt}{$\longrightarrow$}\hskip0.2cm
\includegraphics{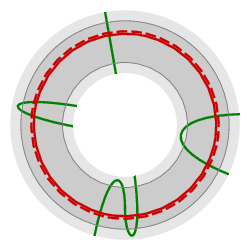}
\caption{Finally making $\delta_-$ and~$\delta_-'$ coinciding. See Figure~\ref{pulling-endpoints} for the legend}\label{annulus-reduction}
\end{figure}
Then~we apply an isotopy to~$\delta_-\cap A$ that will bring it to $\delta_-'\cap A$
keeping inside~$A$, without altering the equivalence class of~$D$. After that
we will have~$\delta_-\cap A=\delta_-'\cap A$.

A similar procedure works for connected components of~$\delta_-$ that are arcs.
Instead of an annulus we consider a `half-annulus', and also allow half-bigon
creations for~$D$ and~$D'$.

Finally, let~$\gamma$ be a one-sided closed component of~$\delta_-$. There is
a connected component~$\gamma'$ of~$\delta_-'$ isotopic to~$\gamma$, and
we have~$|\gamma\cap\gamma'|=1$. These two curves enclose a disc
with two boundary points identified, and a small neighborhood of this disc
is homeomorphic to a M\"obius strip; see Figure~\ref{mobius}.
\begin{figure}[ht]
\includegraphics{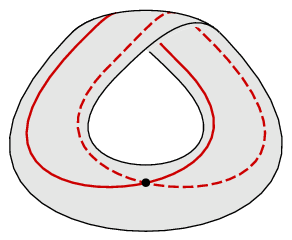}
\caption{An unavoidable intersection of two isotopic one-sided simple closed curves}\label{mobius}
\end{figure}
The intersection of~$\delta_+$ with this disc is treated in the same way as in the
case of an ordinary bigon (see Figure~\ref{bigon-creation-isotopy}), and after
that~$\gamma$ can be isotoped to~$\gamma'$ without distorting the equivalence
class of~$D$.

We repeat this procedure for any connected component of~$\delta_-$ that does not yet
coincide with the respective component of~$\delta_-'$, and keep doing so until we get~$\delta_-=\delta_-'$.
The case~$F\not\cong\mathbb S^2$ is done.

Now suppose that~$F$ is homeomorphic to~$\mathbb S^2$. Since~$D$ is a realizable
configuration, it follows from the results of~\cite{giroux01} that each of~$\delta_+$ and~$\delta_-$
is a single circle. A dividing configuration~$(\delta_1,\delta_2)$ on~$\mathbb S^2$ with~$\delta_1\cong\delta_2\cong\mathbb S^1$
is admissible if and only if~$\delta_1$ and~$\delta_2$ intersect transversely and~$\delta_1\cap\delta_2\ne\varnothing$.

So, for~$F\cong\mathbb S^2$ we go another way: we keep reducing bigons of~$D$ until
we have~$|\delta_+\cap\delta_-|=2$. This produces only admissible configurations during the process.
Do the same with~$D'$. The configurations~$D$ and~$D'$ become equivalent.
\end{proof}

\section{Spatial ribbon graphs}\label{spatial-ribbon-graph-sec}
The aim of this section is mainly to fix the terminology that we operate in the sequel.
In simple words, \emph{a spatial ribbon graph} is the $1$-skeleton of a cell decomposition of a
compact surface embedded in~$\mathbb S^3$, endowed with an additional structure that allows
to recover the surface up to certain equivalence (which amounts to an isotopy if the surface is connected).
\emph{The boundary circuits} of a spatial ribbon graph encode the boundary components
of the represented surface and the attachment maps of the two-cells.

The role of
\emph{an enhancement} of a spatial ribbon graph is to distinguish
between these two types of boundary circuits. However, for all important applications
known by far, it suffices to work with the `default' choice of enhancement,
by which all boundary circuits that \emph{can} be attachment maps of a two-cell (such boundary
circuits are called \emph{patchable}) are declared as such. So, at the expense of losing
some generality, the reader may safely ignore any mentioning of enhancements.

\emph{Morphisms} of spatial ribbon graphs are designed to represent homeomorphisms
of the surfaces represented by the graphs, viewed up to isotopy.

\subsection{The basic concept}
\begin{defi}
By \emph{a spatial graph} we mean a finite one-dimensional (or zero-dimensional) CW-complex
$\Gamma$ embedded in $\mathbb R^3$ so that the restriction of the embedding to each
$1$-cell is piecewise smooth.

We put no restrictions on the presence of multiple edges or loops in spatial graphs and
don't demand spatial graphs to be connected.
We do not distinguish between two spatial graphs if they coincide as subsets of $\mathbb R^3$,
which means that any edge of a spatial graph can be subdivided by new vertices, and this is not regarded
as a change of the graph.

Points of~$\Gamma$ that have no neighborhood homeomorphic to an open interval
are called \emph{true vertices} of~$\Gamma$.
\end{defi}

\begin{defi}
By \emph{a simple circuit} in a spatial graph~$\Gamma$ we mean any simple closed curve in~$\Gamma$.

By \emph{an} \emph{almost simple circuit} in~$\Gamma$ we mean a closed curve in $\Gamma$ that
does not pass more than twice through any point of $\Gamma$ which is not a true vertex.
More precisely, an almost simple circuit is an equivalence class of continuous maps $f:\mathbb S^1\rightarrow\Gamma$
such that $|f^{-1}(p)|\leqslant2$ whenever~$p\in\Gamma$ is not a true vertex of~$\Gamma$, where two maps $f$ and $g$
are regarded equivalent if $f\circ h_1=g\circ h_2$
for some degree-$\pm1$ continuous maps~$h_1,h_2:\mathbb S^1\rightarrow\mathbb S^1$.
\end{defi}

Note that an almost simple circuit can be presented by a constant map, that is, a map
whose image is a single point (which should be a true vertex).

\begin{defi}
Let $\Gamma$ be a spatial graph, and let~$F$ be a compact surface embedded in $\mathbb S^3$ such
that~$\Gamma\subset F$. Let also~$\gamma$ be an almost simple circuit in $\Gamma$, and $\widetilde\gamma$
a connected component of $\partial F$. We say that $\gamma$ is \emph{a projection of $\widetilde\gamma$ to $\Gamma$ in $F$}
if there is a map $f:\mathbb S^1\times[0;1]\rightarrow\mathbb S^3$ such that:
\begin{enumerate}
\item
$F\cup f(\mathbb S^1\times[0;1])$ is a compact surface;
\item
the restriction $f|_{\mathbb S^1\times[0;1)}$ is an embedding;
\item
$f|_{\mathbb S^1\times\{1\}}$ represents $\gamma$;
\item
there exists a continuous map $g:\mathbb S^1\rightarrow[0;1]$ such that the map
$t\mapsto f(t,g(t))$, $t\in\mathbb S^1$, is a parametrization of $\widetilde\gamma$.
\end{enumerate}
\end{defi}

\begin{defi}\label{ribbon-graph-def}
\emph{A spatial ribbon graph} is a $4$-tuple $\rho=(\Gamma_\rho,\partial\rho,V_\rho,S_\rho)$ in which $\Gamma_\rho$ is a spatial graph,
$\partial\rho$ is a finite family of (not necessarily distinct) almost simple circuits in~$\Gamma_\rho$,
$V_\rho$ is a finite subset of~$\Gamma_\rho$ containing all true vertices of~$\Gamma_\rho$
and intersecting any simple closed curve in~$\Gamma_\rho$ at least twice,
and $S_\rho$ is a non-empty class of compact surfaces each
endowed with a bijection between~$\partial\rho$ and the set of boundary components of $\partial F$,
such that:
\begin{enumerate}
\item
for any surface $F\in S_\rho$ the following holds:
\begin{enumerate}
\item
$F$ contains $\Gamma_\rho$;
\item
the inclusion $\Gamma_\rho\hookrightarrow F$ is a homotopy equivalence;
\item
for any connected component~$\gamma$ of $\partial F$ the corresponding circuit
in~$\partial\rho$ is a projection of~$\gamma$ to~$\Gamma_\rho$ in~$F$;
\end{enumerate}
\item
for any two surfaces~$F,F'\in S_\rho$ the following holds:
\begin{enumerate}
\item
if~$p\in V_\rho$, then
the tangent planes~$T_pF$, $T_pF'$ to~$F$ and~$F'$ at~$p$ coincide;
\item
there is an isotopy from~$F$ to~$F'$ within the class~$S_\rho$;
\end{enumerate}
\item
$S_\rho$ is maximal, that is, for any surface~$F'$ endowed with
a bijection between connected components of~$\partial F'$ and the elements of~$\partial\rho$
such that~$F'$ can be obtained from some~$F\in S_\rho$
by an isotopy that keeps the tangent plane to the surface at any point $p\in V_\rho$
fixed, we have~$F'\in S_\rho$.
\end{enumerate}

The elements of $\partial\rho$ are called \emph{boundary circuits} of~$\rho$,
and the set~$\partial\rho$ \emph{the boundary of~$\rho$}. The elements of~$V_\rho$
are called \emph{vertices of~$\rho$}, and the closure of any connected
component of~$\Gamma_\rho\setminus V_\rho$ is called \emph{an edge of~$\rho$}.
\end{defi}

An explanation is in order.

If the graph~$\Gamma_\rho$ in this definition
has a connected component~$\gamma$ homeomorphic to a circle, then~$\partial\rho$
may contain two copies of~$\gamma$, which are regarded as different elements of~$\partial\rho$.
Unless we need to deal with such graphs a simpler definition can be used instead:
a spatial ribbon graph is a triple $\rho=(\Gamma_\rho,V_\rho,S_\rho)$
in which $\Gamma_\rho$ is a spatial graph, $V_\rho$ is a finite subset of~$\Gamma_\rho$
containing all true vertices of~$\Gamma_\rho$ and intersecting any simple closed
curve in~$\Gamma_\rho$ at least twice, and $S_\rho$ is
a non-empty isotopy class of compact surfaces in~$\mathbb S^3$ having~$\Gamma_\rho$
as a deformation retract and fixed tangent planes at all points in~$V_\rho$.
Boundary circuits of~$\rho$ are then defined
as the projections of boundary components of any surface~$F\in S_\rho$.

Indeed, let $\Gamma$ be a spatial graph, and let $F$ be a compact surface containing $\Gamma$
and retracting to~$\Gamma$.
Let also~$V$ be a finite subset of~$\Gamma$ containing all true vertices and intersecting
each simple closed curve in~$\Gamma$ at least twice. One can see that every connected component of $\partial F$
has a unique projection to~$\Gamma$ in~$F$, so, by demanding $S_\rho\ni F$, a spatial ribbon graph~$\rho$
(in the sense of Definition~\ref{ribbon-graph-def})
with~$\Gamma_\rho=\Gamma$, $V_\rho=V$, and~$\partial\rho$ being the family of projections
of the connected components of~$\partial F$ to~$\Gamma$ in~$F$ is defined uniquely.
If~$\Gamma$ has no circular connected components, then
no two boundary components of~$\partial F$ have the same projection, so,
there is a unique bijection between~$\partial\rho$ and the set of connected components
of~$\partial F$ satisfying condition~(1c) from Definition~\ref{ribbon-graph-def}.

However, the simpler definition may cause ambiguity in certain context if a connected
component of~$\Gamma$ is homeomorphic to~$\mathbb S^1$ and
the corresponding connected component of~$F$ is an annulus. In this case, two
boundary components of~$\partial F$ have the same projections to~$\Gamma$,
and thus
there is no one-to-one correspondence between boundary components of~$F$ and
boundary circuits of the spatial ribbon graph defined by~$\Gamma$, $F$, and~$V$.

\begin{defi}
Let~$\gamma$ be a simple
boundary circuit of a spatial ribbon graph~$\rho$.
Then, clearly, there is a surface with corners~$F\in S_\rho$ such that the
connected component of~$\partial F$ corresponding to~$\gamma$
coincides with~$\gamma$.

In this situation, the framing of~$\gamma$ obtained by restricting the boundary framing induced by~$F$ to~$\gamma$ (see~\cite[Definitions~12 and~13]{dp17}) is said to be \emph{induced by~$\rho$}. Clearly,
it does not depend on a particular choice of~$F$.

The induced framing is similarly defined for any link whose components are pairwise disjoint simple boundary
circuits of~$\rho$.
\end{defi}

\subsection{Enhanced and patched spatial ribbon graphs}

\begin{defi}
Let~$F$ be a compact surface in~$\mathbb S^3$. A boundary component~$\gamma\subset\partial F$ is called
\emph{patchable} if there is an embedded closed disc~$d\subset\mathbb S^3$ such that~$\partial d\cap F=\gamma$.

A boundary circuit of a spatial ribbon graph~$\rho$ is called~\emph{patchable} if
the corresponding boundary component of some (and then any) surface~$F\in S_\rho$
is patchable.

Let~$\gamma_1,\ldots,\gamma_k$ be some patchable connected components of~$\partial F$.
We say that a compact embedded surface~$F'\subset\mathbb S^3$ is obtained from~$F$ by
\emph{patching the holes~$\gamma_1,\ldots,\gamma_k$}
if~$F'$ contains~$F$ and~$F'\setminus F$ is a union of open discs~$d_1,\ldots,d_k$ such
that~$\partial\overline d_i=\gamma_i$,
$i=1,\ldots,k$. The boundary components~$\gamma_i$ will then be said to be~\emph{patched}
in~$F'$.
\end{defi}

\begin{defi}
Let~$\gamma$ be a patchable boundary circuit of a spatial ribbon graph~$\rho$, and let~$p:\mathbb D^2\rightarrow\mathbb S^3$
be a continuous map from a closed two-disc~$\mathbb D^2$ to the three-sphere such that:
\begin{enumerate}
\item
the restriction of~$p$ to~$\interior(\mathbb D^2)$ is an embedding;
\item
the image of~$\interior(\mathbb D^2)$ under~$p$ is disjoint from~$\Gamma_\rho$;
\item
the restriction of~$p$ to~$\partial\mathbb D^2$ represents~$\gamma$;
\item
there is a surface~$F\in S_\rho$ such that~$F\cup p(\mathbb D^2)$ is a surface
obtained from~$F$ by patching the hole corresponding to~$\gamma$.
\end{enumerate}
Then we say that the pair~$(\mathbb D^2,p)$ is \emph{a patching disc for~$\gamma$ in~$\rho$}.
Two patching discs~$(\mathbb D^2_1,p_1)$ and~$(\mathbb D^2_2,p_2)$ are \emph{equivalent}
if
$p_1(\mathbb D^2_1)=p_2(\mathbb D^2_2)$.
\end{defi}

Since the equivalence class of a patching disc~$(\mathbb D^2,p)$ is defined uniquely by the
image~$p(\mathbb D^2)$, we will sometimes refer to the latter as a patching disc. Note, however, that~$p(\mathbb D^2)$
is generally not a disc and even not necessarily a surface (for instance, it can be a disc with two boundary points
identified).

\begin{defi}
By \emph{an enhanced spatial ribbon graph} we call a $5$-tuple~$\rho=(\Gamma_\rho,\partial\rho,V_\rho,S_\rho,H_\rho)$
in which~$\rho_0=(\Gamma_\rho,\partial\rho,V_\rho,S_\rho)$ is a spatial ribbon graph and~$H_\rho$ is
a subset of the set of patchable boundary circuits of~$\rho_0$.

A boundary circuit~$\gamma\in\partial\rho$ is said to be \emph{an inessential boundary circuit} (or \emph{a hole}) of~$\rho$
if~$\gamma\in H_\rho$, and otherwise
\emph{an essential boundary circuit of~$\rho$}. For any~$F\in S_\rho$, connected components of~$\partial F$ corresponding to essential (respectively, inessential)
boundary circuits of~$\rho$ also are called essential (respectively, inessential).
Inessential boundary components of~$\partial F$ are also called~\emph{holes}.

By \emph{a patching} of an enhanced spatial ribbon graph~$\rho$ we mean
an isotopy class of pairs of compact surfaces~$(F_0,F)$ such that~$F_0\in S_\rho$
and~$F$ is obtained by patching all the holes of~$F_0$.
An enhanced spatial ribbon graph endowed with a patching is
called \emph{a patched spatial ribbon graph}.
\end{defi}

In other words, an enhanced spatial ribbon graph is a spatial ribbon graph with a declaration
which of the patchable boundary circuits are subject to patching, but without
specifying a concrete way of patching.
If the underlying graph~$\Gamma_\rho$ is connected then the
complement to a regular neighborhood of~$\Gamma_\rho$
is an irreducible manifold, which implies that up to isotopy there is only one way to patch each patchable boundary circuit.
So, in this case, enhancement is just equivalent to patching.

\begin{defi}
The set of all essential boundary circuits of an enhanced ribbon graph~$\rho$
is called \emph{the essential boundary of~$\rho$} and denoted by~$\partial_{\mathrm e}\rho$.
For any~$F\in S_\rho$ the union of essential connected components of~$\partial F$
is denoted by~$\partial_{\mathrm e}F$.
\end{defi}

\begin{defi}\label{properly-carried-by-graph-def}
Let~$(\rho,P)$ be a patched spatial ribbon graph.
A surface~$F$ is said to be \emph{carried by~$(\rho,P)$} if~$(F_0,F)\in P$ for some~$F_0\in S_\rho$.
If, additionally, $\partial F=\cup_{K\in\partial_{\mathrm e}\rho}K$ holds we
say that~$F$ is \emph{properly carried by~$(\rho,P)$}.

A surface is said to be \emph{carried} (respectively, \emph{properly carried}) by an enhanced spatial ribbon graph~$\rho$
if it is carried (respectively, properly carried) by~$(\rho,P)$ for some patching~$P$ of~$\rho$.
\end{defi}

One can see that a surface properly carried by~$(\rho,P)$ exists if and only if~$\partial_{\mathrm e}\rho$
consists of pairwise disjoint simple circuits.

\begin{defi}
Let~$\rho$ and~$\rho'$ be two enhanced spatial ribbon graphs,
and let~$L$ be a framed link whose components are common essential boundary circuits of~$\rho$ and~$\rho'$
endowed with the induced framing, which is the same for~$\rho$ and~$\rho'$.

The enhanced spatial ribbon graphs $\rho$ and $\rho'$ are said to be \emph{stably equivalent
relative to~$L$} if
$V_\rho\cap L=V_{\rho'}\cap L$ and
one can choose their patchings~$P$ and~$P'$, respectively, so that for some (and then any) pairs of surfaces~$(F_0,F)\in P$ and~$(F_0',F')\in P'$ the following holds:
\begin{enumerate}
\item
$L\subset\partial F$, $L\subset\partial F'$ (as a framed link);
\item
$F$ can be brought to~$F'$ by a isotopy fixed on~$L$ that
preserves the tangent plane to~$F$ at any point~$p\in V_\rho\cap L$
and brings essential (respectively, inessential) boundary components of~$F$ to essential
(respectively, inessential) boundary components of~$F'$.
\end{enumerate}

\end{defi}

Note that, if enhanced spatial ribbon graphs~$\rho$ and~$\rho'$ are stably equivalent relative to~$L$,
then for \emph{any} patching~$P$ of~$\rho$
one can find a patching~$P'$ of~$\rho'$ such that the patched spatial ribbon graphs~$(\rho,P)$ and~$(\rho',P')$
are stably equivalent. This implies, in particular, that stable equivalence is transitive, which justifies the term `equivalence'.

\subsection{Morphisms of spatial ribbon graphs}\label{morphisms-subsec}
For an enhanced spatial ribbon graph~$\rho$ we denote by~$\widetilde S_\rho$
the class of all surfaces~$F$ that can be obtained from a surface~$F_0\in S_\rho$
by \emph{a partial patching}, that is, by patching any family of holes (including the extreme cases of no holes and all holes).
Each surface~$F\in\widetilde S_\rho$ is assumed to be
endowed with a bijection inherited from some~$F_0\subset F$, $F_0\in S_\rho$, between connected components of~$\partial F$
and the boundary circuits of~$\rho$ that are not patched in~$F$.

In particular,~$\widetilde S_\rho$ contains all surfaces from~$S_\rho$ and all surfaces carried by~$\rho$.

\begin{defi}
Let~$\rho_1$ and~$\rho_2$ be enhanced spatial ribbon graphs. By \emph{a morphism} from~$\rho_1$ to~$\rho_2$
we call a maximal non-empty class~$\eta$ of triples~$(F_1,F_2,h)$ in which~$F_1\in\widetilde S_{\rho_1}$, $F_2\in\widetilde S_{\rho_2}$,
and~$h:F_1\rightarrow F_2$ is an embedding, such that
\begin{enumerate}
\item
for any~$(F_1,F_2,h)\in\eta$ there exist surfaces~$F_1'$ and~$F_2'$ carried by~$\rho_1$ and~$\rho_2$, respectively,
and a homeomorphism~$h':F_1'\rightarrow F_2'$ such that~$F_1'\supset F_1$, $F_2'\supset F_2$,
$h'|_{F_1}=h$, and~$(F_1',F_2',h')\in\eta$;
\item
for any two triples~$(F_1,F_2,h),(F_1,F_2,h')\in\eta$ such that~$F_i$ is carried by~$\rho_i$, $i=1,2$, the embeddings~$h$ and~$h'$
are isotopic;
\item
if~$(F_1,F_2,h)\in\eta$ and there are isotopies from~$F_1$ and~$F_2$ to surfaces~$F_1'$ and~$F_2'$ within the classes~$\widetilde S_{\rho_1}$
and~$\widetilde S_{\rho_2}$, respectively, inducing homeomorphisms~$h_1:F_1\rightarrow F_1'$ and~$h_2:F_2\rightarrow F_2'$,
then~$(F_1',F_2',h_2\circ h\circ h_1^{-1})\in\eta$.
\end{enumerate}
\end{defi}

Informally speaking, a morphism from~$\rho$ to~$\rho'$ is a way, viewed up to isotopy, to identify surfaces
carried by~$\rho$ with those carried by~$\rho'$. In order to fix such an identification it suffices to choose
a single homeomorphism from a surface~$F$ carried by~$\rho$ to a surface~$F'$ carried by~$\rho'$. Since we
view it up to isotopy it suffices to know the restriction of this homeomorphism to any subsurface~$F_0\subset F$, $F_0\in\widetilde S_\rho$.
So the idea is to include into~$\eta$ all maps that give rise, after patching the holes, to the same, up to isotopy,
identification between surfaces carried by~$\rho$ with surfaces carried by~$\rho'$.

\emph{The composition}~$\eta_2\circ\eta_1$
of two composable morphisms~$\eta_1:\rho_1\rightarrow\rho_2$ and~$\eta_2:\rho_2\rightarrow\rho_3$ is defined in the obvious way:
if~$(F_1,F_2,h_1)\in\eta_1$ and~$(F_2,F_3,h_2)\in\eta_2$, then~$(F_1,F_3,h_2\circ h_1)\in\eta_2\circ\eta_1$.
Clearly, there is an identity morphism for every enhanced spatial ribbon graph, and every morphism has an inverse.

\subsection{Handle moves of spatial ribbon graphs}
\begin{defi}\label{handle-addition-def}
Let~$\rho$ and~$\rho'$ be enhanced spatial ribbon graphs, and let $d$ be a patching disc for~$\rho'$
such that the following holds:
\begin{enumerate}
\item
$\Gamma_{\rho'}\setminus\Gamma_\rho$ is an open simple arc;
\item
for some~$F\in S_\rho$ and~$F'\in S_{\rho'}$ we have
$F'\cup d=F$ and
$\partial_{\mathrm e}F=\partial_{\mathrm e}F'$;
\item
$V_{\rho'}\cap\Gamma_\rho=V_\rho$.
\end{enumerate}
Then we say that the transition from~$\rho$ to~$\rho'$ is \emph{a handle addition}, and
the inverse operation is \emph{a handle removal}.

The patching disc $d$ will be said to be \emph{associated} with the handle addition~$\rho\mapsto\rho'$.
We also associate with this move a morphism~$\eta$ from~$\rho$ to~$\rho'$ by requesting that~$(F,F,\mathrm{id}|_F)\in\eta$.
To state that~$\eta$ is the associated morphism we write~$\rho\xmapsto\eta\rho'$.
\end{defi}

This definition means, in particular, that for every surface~$F'\in S_{\rho'}$ there is a surface~$F\in S_\rho$ (which is not
the same as in the definition above) such
that~$F'$ is obtained from~$F$ by
attaching a $1$-handle to~$\partial F$; see Figure~\ref{attaching-handle}.
\begin{figure}[ht]
\includegraphics[scale=1.2]{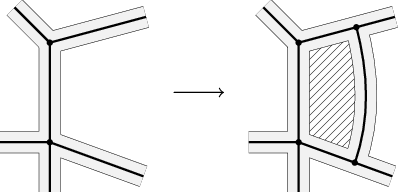}\put(-150,23){$F$}\put(-8,23){$F'$}
\caption{Handle addition}\label{attaching-handle}
\end{figure}
This is done so that all but exactly one of the boundary components of~$F$
are preserved, and the exceptional one $\gamma$, say, is replaced by two components~$\gamma'_1$, $\gamma_2'$, say,
of~$\partial F'$. If~$\gamma$ is an essential component of~$\partial F$, then exactly one of~$\gamma_1'$ and~$\gamma_2'$
is an essential component of~$\partial F'$. Otherwise both of them are inessential. Any connected component of~$\partial F$
that is also contained in~$\partial F'$ is essential for~$F'$ if and only if it is essential for~$F$.

\begin{lemm}\label{add-points-to-V-lem}
Let~$\rho$ and~$\rho'$ be enhanced spatial ribbon graphs such that~$\Gamma_{\rho'}=\Gamma_\rho$, $\partial\rho'=\partial\rho$,
$H_{\rho'}=H_\rho$, $V_{\rho'}\supset V_\rho$, $F_{\rho'}\subset F_\rho$.
Then~$\rho$ and~$\rho'$ are related by a finite sequence of handle additions and removals
preserving any boundary circuit~$c\in\partial\rho$ that satisfies~$c\cap V_\rho=c\cap V_{\rho'}$.
\end{lemm}

\begin{proof}
Clearly it suffices to prove the statement in the case when~$V_{\rho'}\setminus V_\rho$ is a single point~$p$.
This is assumed below.

Let~$\alpha$ be the edge of~$\rho$ that contains~$p$.
Pick a surface~$F\in S_{\rho'}$.
There is another simple arc $\beta$, say, in~$F$ such that~$\beta\cap\Gamma_\rho=\partial\beta=\partial\alpha$
and the arcs~$\alpha$, $\beta$ co-bound a disc~$d\subset F$.

Let~$\rho\mapsto\rho_1\mapsto\rho_2\mapsto\rho_3$ be a sequence of two handle additions and one handle removal, following in the alternating order,
associated with~$d$: we add~$\beta$, then remove~$\alpha$, then put~$\alpha$ back.
The move~$\rho_2\mapsto\rho_3$, which puts~$\alpha$ back, can be chosen so that~$\interior(\alpha)\cap V_{\rho_3}=\{p\}$.
Then~$\rho_3\mapsto\rho'$ will be a handle removal, which completes the job.
\end{proof}

\begin{prop}\label{stable-equivalence-of-ribbon-graphs-prop}
Let~$\rho$ and~$\rho'$ be two enhanced spatial ribbon graphs,
and let~$L$ be a framed link whose components are common essential boundary circuits of~$\rho$ and~$\rho'$
endowed with the induced framing, which is the same for~$\rho$ and~$\rho'$.
Let also~$\eta$ be a morphism from~$\rho$ to~$\rho'$.
Then the following two conditions are equivalent:
\begin{enumerate}
\item
$\rho$ and~$\rho'$ are stably equivalent relative to~$L$, and some isotopy realizing this stable equivalence
induces the morphism~$\eta$;
\item
there exists a sequence of handle additions and removals
$$\rho=\rho_0\xmapsto{\eta_1}\rho_1\ldots\xmapsto{\eta_N}\rho_N=\rho'$$
each of which preserves the framed boundary circuits contained in~$L$, and we have
$\eta=\eta_N\circ\ldots\circ\eta_1$.
\end{enumerate}
\end{prop}

\begin{proof}
The implication~$(2)\Rightarrow(1)$ is easy and is left to the reader. We will prove the converse one.

Assume that Condition~(1) holds, and suppose, additionally, that there is a surface~$F$ carried by both~$\rho$ and~$\rho'$ such
that~$\partial F\supset L$ and~$(F,F,\mathrm{id}|_F)\in\eta$. Suppose
also that the intersection $\Gamma_\rho\cap\Gamma_{\rho'}\setminus L$ is finite.

By Lemma~\ref{add-points-to-V-lem} we may arbitrarily subdivide edges of a spatial ribbon graph
by means of handle additions and removals, without touching the boundary circuits
disjoint from the interior of these edges. Therefore, we may assume
without loss of generality that~$\Gamma_\rho\cap\Gamma_{\rho'}\setminus L$
is a subset of~$V_\rho\cap V_{\rho'}$.

There exists a graph~$\Gamma\subset F$ containing~$\Gamma_\rho\cup\Gamma_{\rho'}$
and having no one-valent vertices and no true vertices outside~$V_\rho\cup V_{\rho'}$. It can be turned into an enhanced spatial ribbon graph~$\rho''$ with~$\Gamma_{\rho''}=\Gamma$,
$V_{\rho''}=V_\rho\cup V_{\rho'}$,
so that~$F$ be carried by~$\rho''$. One can see that~$\rho''$ can then be obtained from each of~$\rho'$ and~$\rho''$
by a sequence of handle additions, and all the associated morphisms of spatial ribbon graphs will
contain the triple~$(F,F,\mathrm{id}|_F)$, hence Condition~(2) holds in this case.

Now we drop the assumption on the finiteness of~$\Gamma_\rho\cap\Gamma_{\rho'}\setminus L$.
One can find an enhanced spatial ribbon graph~$\rho''$ carrying~$F$ such that
$V_{\rho''}\cap L=V_\rho\cap L$, and
both~$\Gamma_\rho\cap\Gamma_{\rho''}\setminus L$
and~$\Gamma_{\rho'}\cap\Gamma_{\rho''}\setminus L$ are finite. By the above argument
there are sequences of handle additions and removals producing~$\rho''$ from~$\rho$, and~$\rho'$ from~$\rho''$,
and all the associated morphisms contain the triple~$(F,F,\mathrm{id}|_F)$.
Hence Condition~(2) holds in this case, too.

Now let~$\rho$ and~$\rho'$ be arbitrary enhanced spatial ribbon graphs and let~$\eta$ be a morphism
from~$\rho$ to~$\rho'$. Assume that Condition~(1) holds, that is, there are
surfaces~$F$ and~$F'$ carried by~$\rho$ and~$\rho'$, respectively, and an isotopy from~$F$ to~$F'$ fixed on~$L$
and inducing the morphism~$\eta$.

One can find a sequence of surfaces~$F=F_1,F_2,\ldots,F_n=F'$ such that, for every~$i=1,\ldots,n-1$,
there is an isotopy from~$F_i$ to $F_{i+1}$ 
fixed outside of an open ball that intersects~$F_i$ in
an open disc (or a half-disc)~$d_i$ satisfying~$d_i\cap L=\varnothing$ and~$\overline{F_i\setminus\overline{d_i}}\supset V_\rho\cap L$.
Such surfaces and isotopies can be chosen so that the isotopy from~$F$ to~$F'$ obtained by concatenating
the latter is arbitrarily close to the original one, and hence also induces~$\eta$.
For every~$i=1,\ldots,n-1$, choose an enhanced spatial ribbon graph~$\rho_i$ with~$V_{\rho_i}\cap L=V_\rho\cap L$ so that~$\Gamma_{\rho_i}$
is contained in~$F_i\setminus d_i$, and~$F_i$ is carried by~$\rho_i$. Set also~$\rho_0=\rho$ and~$\rho_n=\rho'$.

By construction, each surface~$F_i$, $i=1,\ldots,n$, contains~$\Gamma_{\rho_i}$
and~$\Gamma_{\rho_{i-1}}$.
Therefore, for each~$i=1,\ldots,n$, there exists a sequence of handle additions and removals
that produces~$\rho_i$ from~$\rho_{i-1}$.
Choose morphisms~$\eta_i$ from~$\rho_{i-1}$ to~$\rho_i$ so that~$(F_i,F_i,\mathrm{id}|_{F_i})\in\eta_i$,
$i=1,\ldots,n$.
Clearly we will have~$\eta=\eta_n\circ\ldots\circ\eta_1$.
This implies Condition~(2) in the general case.\end{proof}

\section{Mirror diagrams}\label{mir-diagr-sec}

\subsection{Representing ribbon graphs by mirror diagrams}\label{mirror-diagram-definition-subsec}
Here comes our main instrument.
\begin{defi}
By \emph{a mirror diagram} we call a $4$-tuple $M=(\Theta_M,\Phi_M,E_M,T_M)$ in which
$(\Theta_M,\Phi_M,E_M)$ is a rectangular diagram of a graph (see~\cite[Definition~22]{dp17}) and $T_M$ is a map
from $E_M$ to the set of two symbols, `$\diagdown$' and `$\diagup$'.
An element $v\in E_M$ such that $T_M(v)=`\diagdown\text'$ (respectively, $T_M(v)=`\diagup\text'$) is referred
to as \emph{a $\diagdown$-mirror} (respectively, \emph{a $\diagup$-mirror}).
We also call~$T_M(v)$ \emph{the type} of the mirror~$v$.
The meridians~$m_\theta$, $\theta\in\Theta_M$, and the longitudes $\ell_\varphi$, $\varphi\in\Phi_M$, are called \emph{vertical and horizontal occupied levels of $M$}, respectively.
The set $\{m_\theta:\theta\in\Theta_M\}\cup\{\ell_\varphi:\varphi\in\Phi_M\}$
of all occupied levels of~$M$ is denoted by~$L_M$.
\end{defi}

To represent a mirror diagram $M$ graphically we draw all occupied levels of $M$
in a square, which is thought of as a cut torus, and mark each mirror by a dash
slanted at $\pm\pi/4$ according to the type of the mirror.
When there is no mirror at the intersection of two occupied levels we
draw the vertical occupied level passing over the horizontal one; see Figure~\ref{mirrordiagram}.
\begin{figure}[ht]
\includegraphics{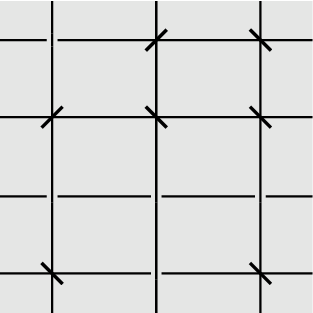}
\caption{A mirror diagram}\label{mirrordiagram}
\end{figure}

\begin{defi}\label{mirror-circuit-def}
\emph{A boundary circuit} of a mirror diagram $M$ is a closed path of a ray of light
traveling along the occupied levels of $M$ and getting reflected off the mirrors of $M$;
see Figure~\ref{circuit}. In particular, if some occupied level of~$M$ contains no mirrors, then this
level alone forms a boundary circuit.
\begin{figure}[ht]
\includegraphics{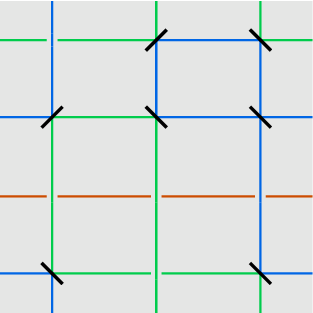}
\caption{Boundary circuits of a mirror diagram. Each circuit is shown in different color}\label{circuit}
\end{figure}
The set of all boundary circuits of~$M$ will be denoted by~$\partial M$ and called \emph{the boundary of~$M$}.
\end{defi}

Mirror diagrams represent spatial ribbon graphs in~$\mathbb S^3$ as we now describe.

Let $M$ be a mirror diagram. Denote by $G_M$ the rectangular diagram of a graph $(\Theta_M,\Phi_M,E_M)$.
All edges of the graph~$\widehat G_M$ are tangent to both contact structures $\xi_+$ and $\xi_-$ (see \cite[Section~3]{dp17}),
hence there exists a compact surface~$F$ such that (recall that we don't demand the surfaces to be $C^2$-smooth):
\begin{enumerate}
\item
$\widehat G_M$ is contained in the interior of~$F$;
\item
$F$ is tangent to $\xi_+$ along $\widehat\mu$ if $\mu$ is a $\diagup$-mirror of~$M$;
\item
$F$ is tangent to $\xi_-$ along $\widehat\mu$ if $\mu$ is a $\diagdown$-mirror of~$M$.
\end{enumerate}

Pick such a surface $F$ and fix it from now on. For a subset $X\subset\mathbb S^3$
and a real $\varepsilon>0$ we denote by~$U_\varepsilon(X)$ the open $\varepsilon$-neighborhood
of~$X$, and by $\overline X$ the closure of~$X$.

For a small enough $\varepsilon>0$ the closure of the intersection $F\cap U_\varepsilon\bigl(\widehat G_M\cap(\mathbb S^1_{\tau=0}\cup
\mathbb S^1_{\tau=1})\bigr)$ is a union of pairwise disjoint discs,
each containing a single point in $\widehat G_M\cap(\mathbb S^1_{\tau=0}\cup
\mathbb S^1_{\tau=1})$, and being such that the torus projection of its boundary is close to the corresponding occupied level of~$M$.
We fix such an~$\varepsilon$ and, for each~$x\in L_M$, denote by $\wideparen x$ the disc
corresponding to~$x$.

Now take $\delta>0$ much smaller than~$\varepsilon$ so that
the closure of $F\cap U_\delta(\widehat G_M)\setminus U_\varepsilon(\mathbb S^1_{\tau=0}\cup
\mathbb S^1_{\tau=1})$ is a union of pairwise disjoint discs, each intersecting a single arc
of the form~$\widehat\mu$, $\mu\in E_M$. We denote the disc
intersecting $\widehat\mu$, where~$\mu\in E_M$, by $\wideparen\mu$.
It has the form of a narrow strip joining the discs
$\wideparen m$ and $\wideparen\ell$, where~$m$ and~$\ell$
are the meridian and the longitude through~$\mu$,
and twisted around~$\widehat\mu$ by~$\pi/2$ in
the direction that depends on the type of the mirror~$\mu$ in~$M$;
see Figure~\ref{surface-by-mirror-diagram}.
\begin{figure}[ht]
\includegraphics{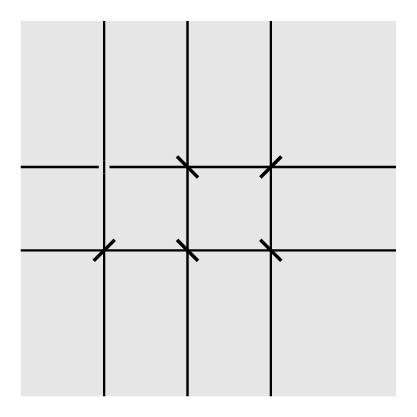}\put(-188,14){$M$}\put(-148,14){$m_1$}\put(-108,14){$m_2$}\put(-68,14){$m_3$}%
\put(-20,70){$\ell_1$}\put(-20,110){$\ell_2$}\put(-148,70){$\mu_1$}\put(-108,70){$\mu_2$}\put(-68,70){$\mu_3$}%
\put(-108,110){$\mu_4$}\put(-68,110){$\mu_5$}\hskip1cm
\includegraphics[width=200pt]{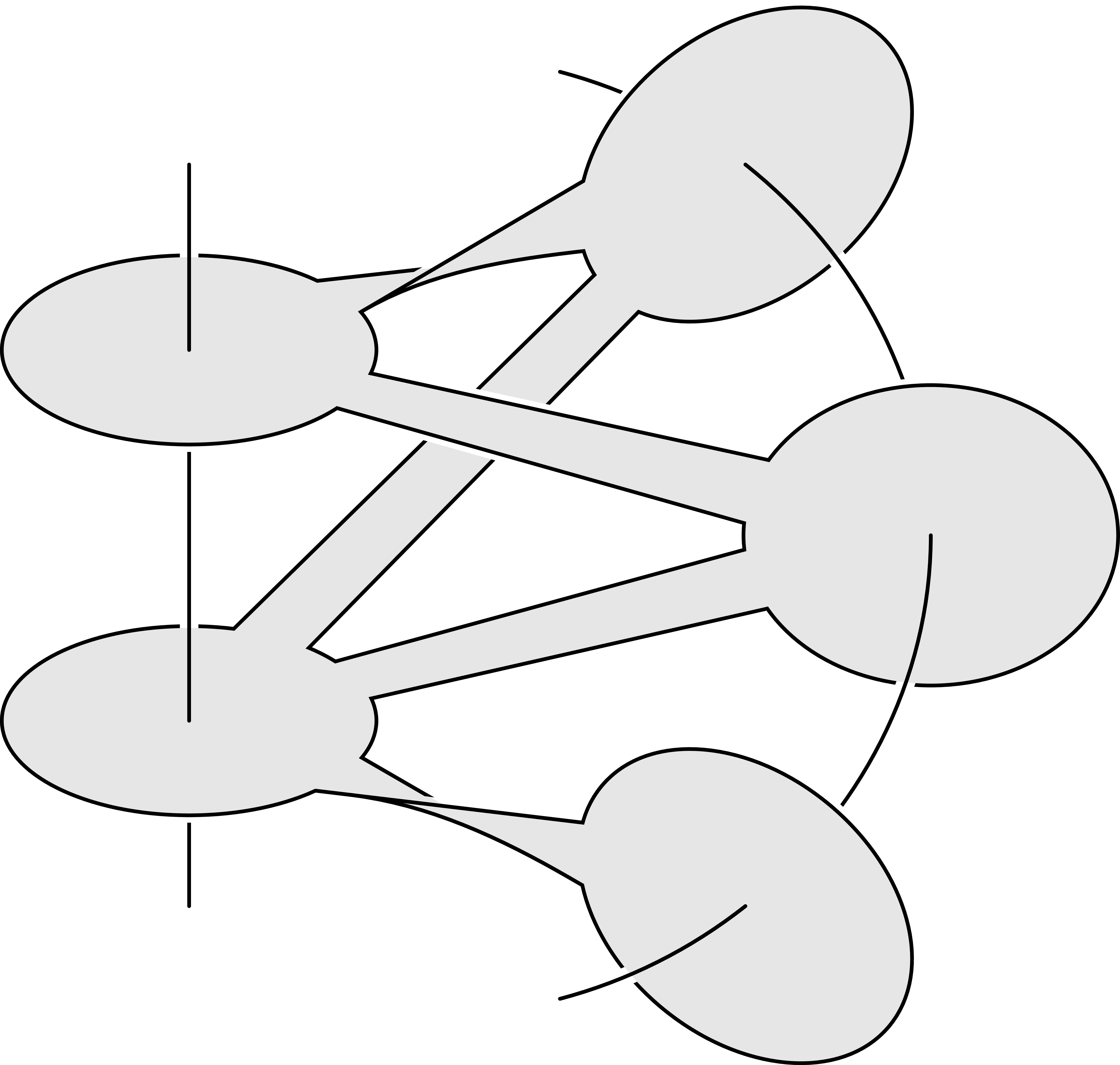}\put(-200,10){$\wideparen M$}%
\put(-62,15){$\wideparen m_1$}\put(-25,90){$\wideparen m_2$}\put(-58,167){$\wideparen m_3$}%
\put(-185,60){$\wideparen\ell_1$}\put(-185,125){$\wideparen\ell_2$}\put(-120,173){$\mathbb S^1_{\tau=1}$}%
\put(-170,165){$\mathbb S^1_{\tau=0}$}\put(-120,33){$\wideparen\mu_1$}\put(-95,65){$\wideparen\mu_2$}%
\put(-122,90){$\wideparen\mu_3$}\put(-100,95){$\wideparen\mu_4$}\put(-124,150){$\wideparen\mu_5$}
\caption{A mirror diagram~$M$ and the corresponding surface $\wideparen M$}\label{surface-by-mirror-diagram}
\end{figure}

The union $\bigl(\cup_{x\in L_M}\wideparen x\bigr)\cup\bigl(\cup_{\mu\in E_M}\wideparen\mu\bigr)$
is a surface with corners, which we denote by~$\wideparen M$. It is easy to see that~$\wideparen M$
retracts to~$\widehat G_M$, and the torus projection of each connected component of~$\partial\wideparen M$ approximates
a
unique boundary circuit of~$M$ (the smaller~$\varepsilon$ and~$\delta$, the better is the accuracy).
This establishes a natural one-to-one correspondence
between the boundary components of~$\wideparen M$ and the boundary circuits of~$M$.

More formally, this correspondence is defined as follows. Let~$x$ be an occupied level of~$M$,
and let~$\mu_1,\ldots,\mu_k$ be all mirrors of~$M$ lying on~$x$ numbered according
to their cyclic order in~$x$. There is a natural one-to-one
correspondence between connected components of~$x\setminus\{\mu_1,\ldots,\mu_k\}$ and
the connected components of~$\wideparen x\cap\partial\wideparen M$.
Namely, if~$k\geqslant2$ and~$\alpha=(\mu_i,\mu_{i+1(\mathrm{mod}\,k)})$, then the component
of~$\wideparen x\cap\partial\wideparen M$ corresponding to~$\alpha$, which we denote by~$\wideparen\alpha$,
is defined by requesting that the torus projection of~$\partial\wideparen\alpha$
projects into a subinterval of~$(\mu_i,\mu_{i+1(\mathrm{mod}\,k)})$ along the $\varphi$-direction
if~$x$ is a longitude, and into a subinterval of~$(\mu_i,\mu_{i+1(\mathrm{mod}\,k)})$ along
the $\theta$-direction
if~$x$ is a meridian.
If~$k=0$ or~$k=1$, then both sets~$x\setminus\{\mu_1,\ldots,\mu_k\}$
and~$\wideparen x\cap\partial\wideparen M$ are connected, so the bijection in question is the only possible one.
If~$c$ is the boundary circuit of~$M$ that contains~$\alpha$, then the respective
connected component of~$\partial\wideparen M$, which we denote by~$\wideparen c$, is
the one that contains~$\wideparen\alpha$. One can see that this rule being
applied to all occupied levels~$x$ establishes a one-to-one correspondence between~$\partial M$
and the set of all connected components of~$\partial\wideparen M$.

\begin{defi}\label{ass-spatial-rib-graph-def}
\emph{The spatial ribbon graph $\widehat M$ associated with a mirror diagram~$M$}
is defined by putting
$$\Gamma_{\widehat M}=\widehat G_M,\qquad V_{\widehat M}=\widehat
G_M\cap(\mathbb S^1_{\tau=0}\cup\mathbb S^1_{\tau=1}),$$
and letting $S_{\widehat M}$ be such that~$\wideparen M\in S_{\widehat M}$.
This means, in particular, that~$\partial\widehat M$ consists of the projections of
the connected components of~$\partial\wideparen M$ to~$\widehat G_M$ in~$\wideparen M$.
According to the discussion above there is a natural bijection between~$\partial\widehat M$
and~$\partial M$.

A mirror diagram~$M$ is called \emph{connected} if so is the graph~$\widehat G_M$.
A mirror diagram~$M_0$ is referred to as \emph{a connected component} of
a mirror diagram~$M$ if~$\widehat M_0$ is a connected component of~$\widehat M$.
\end{defi}

Note that a spatial ribbon graph associated with a mirror diagram~$M$ is a well defined
object whereas the definition of the surface~$\wideparen M$ contains a large arbitrariness.
However, we will often refer to surfaces of the form~$\wideparen M$ in our
explanations, silently assuming
that some reasonable choice has been made for all such surfaces that are being currently considered,
and this choice is natural in the present context.

Note also that if~$M$ is a mirror diagram, then any surface~$F\in S_{\widehat M}$
is orthogonal to~$\mathbb S^1_{\tau=0}\cup\mathbb S^1_{\tau=1}$ at
any point in~$\Gamma_{\widehat M}\cap\bigl(\mathbb S^1_{\tau=0}\cup\mathbb S^1_{\tau=1}\bigr)$.

\begin{defi}
Let~$\rho$ be a(n enhanced) spatial ribbon graph, and let~$F_0$, $F_1$ be two
surfaces.
We say that~$F_0$ and~$F_1$ are \emph{isotopic relative to~$\rho$}
if there is an isotopy from~$F_0$ to~$F_1$ fixed on~$\Gamma_\rho$
through surfaces each of which contains a subsurface from~$S_\rho$.
In particular, $F_0$ and~$F_1$ are meant to contain subsurfaces~$F_0',F_1'\in S_\rho$, respectively.
\end{defi}

Practically, this definition means that~$F_0$ (equivalently, $F_1$) contains a subsurface from~$S_\rho$,
and there is an isotopy from~$F_0$ to~$F_1$
that is fixed on~$\Gamma_\rho$ and preserves the tangent plane to the surface at any 
vertex of~$\rho$.

\begin{lemm}\label{rectangular-representative-lem}
Let~$M$ be a mirror diagram, and let~$F$ be a compact surface
containing a subsurface~$F'\in S_{\widehat M}$.
Suppose that each connected component of~$\partial F$ is either contained in~$\Gamma_{\widehat M}$
or disjoint from~$\Gamma_{\widehat M}$. Then there exists a rectangular
diagram of a surface~$\Pi$ such that~$\widehat\Pi$ is isotopic to~$F$
relative to~$\widehat M$.
\end{lemm}

\begin{proof}
Since~$F\supset F'\in S_{\widehat M}$, the surface~$F$ is isotopic relative to~$\widehat M$ to a surface
which is tangent to~$\xi_+$
along~$\widehat\mu$ if~$\mu$ is a $\diagup$-mirror of~$M$, and to~$\xi_-$
if~$\mu$ is a $\diagdown$-mirror. So, in what follows we assume that this tangency condition holds for~$F$ from the beginning,
and use only isotopies
preserving it.

Denote by~$L$ the link~$\partial F\setminus\Gamma_{\widehat M}$.
By analogy with~\cite[Lemma~2]{dp17} one can show that,
for any~$k\in\mathbb Z$, there exists a rectangular diagram of a link~$R$ such
that~$\widehat R$ is isotopic to~$L$ relative to~$\Gamma_{\widehat M}$, and each connected component of~$R$ has Thurston--Bennequin numbers~$\tb_\pm$
smaller than~$k$. The formulation of~\cite[Lemma~2]{dp17} is different
from this claim only in involving a rectangular diagram of a link in place
of a more general rectangular diagram of a graph, which is~$G_M$,
but this does not affect the proof.

By choosing~$k$ small enough and taking~$R$ as above one can find a surface~$F'$
such that
\begin{enumerate}
\item
$\partial F'\setminus\Gamma_{\widehat M}=\widehat R$;
\item
for any connected component~$K$ of~$\partial F'$ we have~$\tb_+(K;F'),\tb_-(K;F')\leqslant0$;
\item
there is an isotopy from~$F$ to~$F'$
fixed on~$\Gamma_{\widehat M}$ and preserving the tangency with~$\xi_+$
or~$\xi_-$ along~$\Gamma_{\widehat M}$.
\end{enumerate}

By~\cite[Proposition~5]{dp17} there exists an isotopy fixed on~$\Gamma_{\widehat M}\cup\widehat R$
and preserving the tangency with~$\xi_+$ or~$\xi_-$ along~$\Gamma_{\widehat M}$,
from~$F'$ to a surface of the form~$\widehat\Pi$, where~$\Pi$ is a rectangular diagram of a surface.
\end{proof}

\begin{defi}
A boundary circuit of a mirror diagram~$M$ is called \emph{patchable} if
the corresponding boundary circuit of~$\widehat M$ is patchable.
\end{defi}

\begin{defi}
\emph{An enhanced mirror diagram} is a $5$-tuple~$M=(\Theta_M,\Phi_M,E_M,T_M,H_M)$ in which
$M_0=(\Theta_M,\Phi_M,E_M,T_M)$ is a mirror diagram, and~$H_M$ is a subset
of the set of patchable boundary circuits of~$M_0$. The non-patchable boundary circuits of~$M_0$
as well as the patchable ones that are not elements of~$H_M$ are called \emph{essential boundary circuits of~$M$}.
The elements of~$H_M$ are called
\emph{inessential boundary circuits of~$M$} or \emph{holes}.
The set~$(\partial M)\setminus H_M$ is called \emph{the essential boundary of~$M$} and denoted by~$\partial_{\mathrm e}M$.

If~$M$ is an enhanced mirror diagram we denote by~$\widehat M$
the enhanced spatial ribbon graph having~$\widehat M_0$ as the underlying spatial ribbon graph,
such that the inessential boundary circuits of~$\widehat M$ are exactly those
that correspond to the elements of~$H_M$.

Every mirror diagram~$M$ will be also regarded as an enhanced one with~$H_M$
being the set of all patchable boundary circuits of~$M$.
\end{defi}

\subsection{Simple boundary circuits and framed rectangular diagrams of links}\label{simple-circuits-subsec}

\begin{defi}
A boundary circuit of a mirror diagram is said to be \emph{simple} if it hits:
\begin{enumerate}
\item
at least one mirror,
\item
each mirror at most once, and
\item
at most two mirrors at each occupied level.
\end{enumerate}
A collection of simple boundary circuits whose union satisfies~(2) and~(3) above
is also termed \emph{simple}.
\end{defi}

Recall that a framing~$f$ of a rectangular diagram of a link is an ordering~$>_f$ of each
pair of vertices forming an edge~\cite[Definition~14]{dp17} (see also Subsection~\ref{twisting-subsec} above).

A simple collection of boundary circuits of a mirror diagram is an object
of exactly the same kind as the one used to represent a framed rectangular
diagram of a link (see Definition 14 in~\cite{dp17} and the discussion afterwards).
Thus, with every simple collection~$C=\{c_1,\ldots,c_k\}$ of boundary circuits of a mirror diagram
we associate a framed rectangular diagram of a link~$(R(C),f(C))$
by letting~$R(C)$ be the set of all mirrors hit by the circuits~$c_i$, $i=1,\ldots,k$, (with the types of the mirrors forgotten),
and letting~$f(C)$ be a framing of~$R(C)$ such that for any edge~$\{v_1,v_2\}$ of~$R(C)$
we have~$v_2>_{f(C)}v_1$ if and only if~$[v_1;v_2]\subset\bigcup\limits_{i=1}^kc_i$.

For a framing~$f$ of a rectangular diagram of a link~$R$, we denote by~$\widehat f$ the respective
admissible framing of the link~$\widehat R$.

For a boundary circuit~$c$ of a mirror diagram~$M$ we denote by~$\widehat c$
the respective boundary circuit of the spatial ribbon graph~$\widehat M$. The following
statement, proof of which is left to the reader, is an easy consequence of the definitions.

\begin{prop}
A boundary circuit~$c$ of a mirror diagram~$M$ is simple if and only if so is the
boundary circuit~$\widehat c$ of~$\widehat M$.

A collection~$C=\{c_1,\ldots,c_k\}$ of boundary circuits of a mirror diagram~$M$
is simple if and only if the boundary circuits~$\widehat c_1,\ldots,\widehat c_k$
of~$\widehat M$ are all simple and pairwise disjoint. In this
case we have~$\widehat{R(C)}=\widehat c_1\cup\ldots\cup\widehat c_k$.
Moreover, the framing induced by~$\widehat M$ on~$\widehat{R(C)}$ is
opposite to~$\widehat{f(C)}$.
\end{prop}

Recall that links presented by rectangular diagrams are Legendrian with respect to both contact
structures~$\xi_+$ and~$\xi_-$. The following statement is an easy observation, proof of which is also omitted.

\begin{prop}
Let~$c$ be a simple boundary circuit of a mirror diagram~$M$, and let~$F\in S_{\widehat M}$ be a surface
such that~$\widehat c\subset\partial F$. Let also~$k$ and~$l$ be the numbers of times
the circuit~$c$ hits $\diagup$-mirrors and $\diagdown$-mirrors, respectively.
Then the Thurston--Bennequin numbers of~$\widehat c$ relative to~$F$ are as follows:
$$\tb_+(\widehat c;F)=-l/2,\quad\tb_-(\widehat c;F)=-k/2.$$
\end{prop}

This motivates the following definition.

\begin{defi}
Let~$c$ be a boundary circuit of a mirror diagram~$M$ such that~$c$ hits $\diagup$-mirrors and $\diagdown$-mirrors
$k$ times and~$l$ times, respectively. By \emph{the Thurston--Bennequin numbers of~$c$}, denoted~$\tb_+(c)$ and~$\tb_-(c)$,
we mean the numbers~$-l/2$ and~$-k/2$, respectively.
\end{defi}

It is elementary to see that~$\tb_+(c)$ and~$\tb_-(c)$ are always integers even if~$c$ is not a simple boundary circuit.

\subsection{Rectangular diagrams of a surface as mirror diagrams}\label{rectangular=mirror-subsec}

\begin{defi}\label{ass-mir-diagr-def}
For every rectangular diagram of a surface~$\Pi$, \emph{the associated enhanced mirror diagram~$M(\Pi)$}
is defined as follows:
\begin{itemize}
\item
$L_{M(\Pi)}$ is the set of all occupied levels of~$\Pi$;
\item
$E_{M(\Pi)}$ is the set of vertices of~$\Pi$, each of which has the same type (`$\diagup$' or `$\diagdown$') in~$M(\Pi)$ as
in~$\Pi$;
\item
$H_{M(\Pi)}=\{\partial r:r\in\Pi\}$.
\end{itemize}
\end{defi}

\begin{defi}\label{properly-carried-by-mir-diag-def}
Let~$M$ be an enhanced mirror diagram.
We say that a rectangular diagram of a surface~$\Pi$ is \emph{carried} (respectively,
\emph{properly carried}) \emph{by}~$M$ if
the surface~$\widehat\Pi$ is carried (respectively, properly carried) by~$\widehat M$.
\end{defi}

Propositions~\ref{obvious-prop-1} and~\ref{obvious-prop-2} below follow immediately from these definitions.

\begin{prop}\label{obvious-prop-1}
Any rectangular diagram of a surface~$\Pi$ is properly carried by~$M(\Pi)$.

The map~$\Pi\mapsto M(\Pi)$ is a one-to-one correspondence between rectangular diagrams of a surface
and enhanced mirror diagrams~$M$ satisfying the following two conditions:
\begin{enumerate}
\item
each inessential boundary circuit of~$M$ has the form of the boundary
of a rectangle~$r\subset\mathbb T^2$ such that the interior of~$r$ is disjoint from~$E_M$;
\item
the essential boundary~$\partial_{\mathrm e}M$ of~$M$ is simple.
\end{enumerate}
\end{prop}

\begin{prop}\label{obvious-prop-2}
Let~$M$ be a mirror diagram, and let~$\Pi$ be a rectangular diagram of a surface properly carried
by~$M$. Then the rectangular diagram of the link~$R(\partial_{\mathrm e}M)$ coincides with~$\partial\Pi$,
whereas the framing~$f(\partial_{\mathrm e}M)$ is opposite to the boundary framing of~$\partial\Pi$
induced by~$\Pi$.
\end{prop}

\begin{prop}
Let~$M$ be an enhanced mirror diagram. A rectangular diagram of a surface properly carried by~$M$ exists
if and only if the essential boundary of~$M$ is simple.
\end{prop}

\begin{proof}
The `only if' direction is obvious.

To prove the `if' direction note that the simplicity of~$\partial_{\mathrm e}M$ implies that there exists
a surface~$F$ properly carried by~$\widehat M$. Proposition~5 of~\cite{dp17} (in which we put~$X_1=E_M$, $X_2=\varnothing$)
implies that there exists a surface of the form~$\widehat\Pi$ isotopic to~$F$ relative to~$\Gamma_{\widehat M}$,
where~$\Pi$ is a rectangular diagram of a surface. Such a diagram~$\Pi$ is properly carried by~$M$.
\end{proof}

\subsection{Canonic dividing configurations}

\begin{defi}\label{canonic-dividing-for-mirror-diagrams-def}
Let~$M$ be a mirror diagram. A dividing configuration~$(\delta_+,\delta_-)$ on
a surface~$F\in S_{\widehat M}$ is said to be \emph{a canonic dividing configuration of~$F$}
if the following holds:
\begin{enumerate}
\item
$\delta_+\cup\delta_-$ is a union of pairwise disjoint arcs each intersecting~$\Gamma_{\widehat M}$ exactly once;
\item
$\delta_+$ (respectively,~$\delta_-$) intersects each arc of the form~$\widehat\mu$, where~$\mu$ is a $\diagdown$-mirror
(respectively, a $\diagup$-mirror) of~$M$,
exactly once, at the midpoint ($\tau=1/2$), and the intersection is transverse;
\item
the coorientation of each arc~$\widehat\mu\subset\Gamma_{\widehat M}$ defined by~$\delta_+\cup\delta_-$
at the midpoint agrees with
the one induced on~$\widehat\mu$ by the gradient of~$\varphi|_F$ near the circle~$\mathbb S^1_{\tau=1}$.
\end{enumerate}
\end{defi}

One can see that a canonic dividing configuration on any surface~$F\in S_\rho$ is defined uniquely up to
equivalence. On can also see that it is never admissible.

\begin{exam}
A canonic dividing configuration on the surface~$\wideparen M$ can be obtained as follows.
Take for~$\delta_+$ (respectively, for~$\delta_-$) the intersection of all strips of the form~$\wideparen\mu$,
where~$\mu$ is a $\diagdown$-mirror (respectively, a $\diagup$-mirror), with the torus~$\mathbb T^2_{\tau=1/2}$, and orient it so that the coordinate~$\varphi$ increases on each connected component of~$\delta_+\cup\delta_-$
in the positive direction.
\end{exam}

We conclude the section by the following statement, which is not formally needed
to establish our main results, but reveals the differential-geometric meaning of canonic dividing configurations.

\begin{prop}
Let~$M$ be an enhanced mirror diagram with simple essential boundary,
and let~$F$ be a surface properly carried by~$\widehat M$
such that~$F$ is convex in Giroux's sense with respect to both contact structures~$\xi_+$ and~$\xi_-$
(for instance, a surface of the form~$\widehat\Pi$ with~$\Pi$ a rectangular
diagram of a surface properly carried by~$M$). Let~$\delta_+$ and~$\delta_-$
be dividing sets of~$F$ corresponding to~$\xi_+$ and~$\xi_-$, respectively.

Then for any sufficiently small closed
neighborhood~$F_0$ of~$\Gamma_{\widehat M}$ in~$F$
the pair~$(\delta_+\cap F_0,\delta_-\cap F_0)$ is a dividing configuration in~$F_0$
equivalent to a canonic one.
\end{prop}

\begin{proof}
Consider the contact structure~$\xi_+$. The dividing set~$\delta_+$ can be
represented as the set of zeros of the restriction of the two-valued function~$\alpha(l)$ to the surface~$F$, where~$l$ is a contact
line element field on~$\mathbb S^3$ transverse to~$F$, and~$\alpha$ is a standard contact form (that is, a $1$-form
such that~$\xi_+=\ker\alpha$).

Let~$\beta$ be a regular fiber of the characteristic
foliation~$\mathscr F$ of the surface~$F$ approaching singularities at the endpoints. In a small neighborhood
of~$\beta$ the field~$l$ has the form~$\{v,-v\}$, where~$v$ is a contact vector field.
Fix such~$v$ from now on and orient~$\beta$ so that~$\alpha(w)>0$ whenever~$u$ is a positive tangent vector to~$\beta$,
and $w$ is a tangent vector to~$F$ such that~$(u,w,v)$ is a positively oriented basis of the three-space.

One can show that if we traverse~$\beta$ in the positive direction, then the sign of~$\alpha(v)$
can change only from~`$+$' to~`$-$'
but not the other way. So, it may change at most once.
If~$\beta$ is a $0$-arc (see~\cite[Subsection~4.3]{dp17} for the definition), then the signs
of~$\alpha(v)$ at the endpoints of~$\beta$ coincide, and if~$\beta$ is a $-1$-arc,
then the signs are opposite. Therefore, any~$0$-arc
is disjoint from~$\delta_+$, whereas any $-1$-arc
intersects~$\delta_+$ exactly once.

Note also that
$\delta_+$ must be disjoint from the singularities of~$\mathscr F$,
and no regular fiber of~$\mathscr F$ can be a $1$-arc as
the surface~$F$ is assumed to be convex.

If~$\mu$ is a $\diagup$-mirror of~$M$, then~$\widehat\mu$ is an arc with singularities of~$\mathscr F$
at the endpoints. The total rotation of the contact plane~$\xi_+(p)$ with respect to the tangent
plane~$T_pF$ when~$p$ traverses~$\widehat\mu$ is zero. Therefore, $\widehat\mu$
consists of singularities of~$\mathscr F$ and $0$-arcs, and thus is disjoint from~$\delta_+$.

If~$\mu$ is a $\diagdown$-mirror, then the total rotation of the contact plane~$\xi_+(p)$ with respect to the tangent
plane~$T_pF$ when~$p$ traverses~$\widehat\mu$ is equal to~$-\pi$.
Therefore, $\widehat\mu$
consists of singularities of~$\mathscr F$, some number (may be zero) of $0$-arcs,
and a single $-1$-arc, and hence intersects~$\delta_+$ exactly once.

By symmetry an edge~$\widehat\mu$ intersects~$\delta_-$ once if~$\mu$ is a $\diagup$-mirror
and do not intersect it if~$\mu$ is a $\diagdown$-mirror. The claim follows.
\end{proof}

\section{Elementary moves of mirror diagrams}\label{elementary-moves-sec}
\subsection{Conventions}\label{conventions-subsec}
Denote by $r_-$, $r_|$, $r_\diagup$, $r_\diagdown$ the reflections of the torus $\mathbb T^2$
about the lines $\varphi=0$, $\theta=0$, $\varphi=\theta$, and $\varphi=-\theta$, respectively:
$$r_-(\theta,\varphi)=(\theta,-\varphi),\quad
r_|(\theta,\varphi)=(-\theta,\varphi),\quad
r_\diagup(\theta,\varphi)=(\varphi,\theta),\quad
r_\diagdown(\theta,\varphi)=(-\varphi,-\theta).$$
For a mirror diagram $M$ we define
$$\begin{aligned}r_-(M)&=(\Theta_M,-\Phi_M,r_-(E_M),\nu\circ T_M\circ r_-),&
r_|(M)&=(-\Theta_M,\Phi_M,r_|(E_M),\nu\circ T_M\circ r_|),\\
r_\diagup(M)&=(\Phi_M,\Theta_M,r_\diagup(E_M),T_M\circ r_\diagup),&
r_\diagdown(M)&=(-\Phi_M,-\Theta_M,r_\diagdown(E_M),T_M\circ r_\diagdown),
\end{aligned}$$
where $-X$ with $X\subset\mathbb S^1$, stands for $\{-x:x\in X\}$, and $\nu$
is the map exchanging the symbols `$\diagdown$' and `$\diagup$'.

\begin{conv}\label{symmetries}
Whenever we give a name to a transformation
$M\mapsto M'$ of mirror diagrams, the same name
is meant to be given to the transformations $r_-(M)\mapsto r_-(M')$,
$r_|(M)\mapsto r_|(M')$, $r_\diagdown(M)\mapsto r_\diagdown(M')$,
and $r_\diagup(M)\mapsto r_\diagup(M')$.

Whenever a transformation $M\mapsto M'$ is said to be of type~I
it is meant that the transformations
$r_\diagdown(M)\mapsto r_\diagdown(M')$, $r_\diagup(M)\mapsto r_\diagup(M')$
are of type~I, and that the transformations $r_-(M)\mapsto r_-(M')$,
$r_|(M)\mapsto r_|(M')$ are of type~II.

Due to this symmetry, any statement about the moves has an equivalent `dual'
one obtained by exchanging types~I and~II of moves and
types `$\diagdown$' and `$\diagup$' of mirrors. We may, without a special notice,
refer to a previously proven statement when the dual one should be used instead.
\end{conv}

\begin{conv}
To represent graphically a transformation of mirror diagrams we often draw a portion
of each diagram where the change occurs. Indicated occupied levels may have arbitrarily
many mirrors outside the shown region unless their endpoints
are marked by perpendicular dashes: \includegraphics{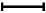}. In the latter case,
the occupied level is supposed to contain mirrors only in the shown portion of the diagram.

A dotted line \includegraphics{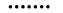} in the pictures of transformations indicate
an arbitrary family of parallel occupied levels, which is supposed to be unchanged by the transformation.
A gray box marked with a letter stands for an arbitrary family of mirrors at the intersection of
the indicated occupied levels unless otherwise specified. The family of mirrors inside such boxes may
be shifted as a whole by the considered moves.
\end{conv}

\begin{conv}
Sometimes, in order to illustrate a transformation~$M\mapsto M'$ of mirror diagrams
we draw the parts of the surfaces~$\wideparen M$ and~$\wideparen M'$
in which the surfaces differ, without pretending to show realistically their positions
in the three-space. However, if two discs of the form~$\wideparen x$, $\wideparen x'$,
where~$x$ and~$x'$ are parallel occupied levels, are close to one another according
to the context, we may draw them overlapping, one over the other:

\centerline{\includegraphics{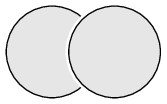}\put(-58,21){$\wideparen x$}\put(-28,21){$\wideparen x'$}}

Every strip of the form~$\wideparen\mu$, where~$\mu$ is a mirror of a mirror diagram~$M$, is marked in our
pictures by a small arc with endpoints at~$\partial\wideparen M$, intersecting~$\widehat\mu$ exactly once.
The arc is made red or green depending on the type of the mirror. If~$\mu$ is of type~`$\diagdown$' this arc is green,
and otherwise red:

\centerline{\includegraphics{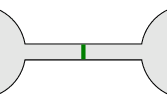}\put(-55,10){type `$\diagdown$'}
\hskip2cm
\includegraphics{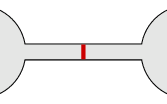}\put(-55,10){type `$\diagup$'}
\hskip.3cm\raisebox{20pt}{.}}
\end{conv}

These arcs represent a canonic dividing configuration~$(\delta_+,\delta_-)$ of~$\wideparen M$ with~$\delta_+$
shown in green and~$\delta_-$ in red, which accords with Convention~\ref{color-conv}.
We do not show the orientation of these arcs in the pictures.

\begin{conv}
Quite commonly, objects like rectangular diagrams are viewed up to
\emph{combinatorial equivalence}, without distinction between
two diagrams that can be obtained from one another by a self-homeomorphism of $\mathbb T^2$
of the form $(\theta,\varphi)\mapsto\bigl(f(\theta),g(\varphi)\bigr)$, where $f$ and $g$
are degree $1$ self-homeomorphisms of $\mathbb S^1$. This is \emph{not} the case
here. When we deal with mirror diagrams the exact position of some occupied levels does matter.
More precisely, when a diagram transforms, it is important to observe which mirrors
stay fixed and which are altered. For instance, from combinatorial point of view
a transformation may look as an exchange of two occupied levels. But geometrically
one of the levels may stay fixed, and it will matter which one of the two does so and
which one moves.
\end{conv}

\subsection{Elementary moves of mirror diagrams (without enhancement)}
The transformations of mirror diagrams introduced in this section, namely, extension/elimination, elementary bypass addition/removal,
and slide moves, are called \emph{elementary moves}. First we introduce them for ordinary mirror diagrams
and then extend the definitions to enhanced ones.

\begin{defi}\label{ext-move-def}
Let $M$ and $M'$ be mirror diagrams such that $M'$ is obtained from~$M$ by adding
a new occupied level and a $\diagup$-mirror at the intersection of the new occupied level with
an existing occupied level of $M$. Then we say that the passage $M\mapsto M'$ is
\emph{a type~I extension move}, and the inverse operation \emph{a type~I elimination move};
see Figure~\ref{mir-elem-moves}~(a).
\end{defi}

\begin{defi}\label{mirr-bypass-def}
Let $M$ and $M'$ be mirror diagrams such that for some $\theta_1\ne\theta_2$,
$\varphi_1\ne\varphi_2$ the following holds:
\begin{enumerate}
\item
$(\theta_1,\varphi_2)$ and $(\theta_2,\varphi_1)$ are $\diagup$-mirrors of both $M$ and $M'$;
\item
$(\theta_1,\varphi_1)$ is a $\diagdown$-mirror of both $M$ and $M'$;
\item
$(\theta_2,\varphi_2)$ is a $\diagdown$-mirror of $M$ but not of $M'$;
\item
there are no more mirrors of $M$ or~$M'$ in $r=[\theta_1;\theta_2]\times[\varphi_1;\varphi_2]$;
\item
$M$ and $M'$ have the same set of mirrors outside $r$ and the same set of occupied levels.
\end{enumerate}
Then we say that~$M'$ is obtained from $M$ by \emph{a type~I elementary bypass removal},
and $M$ is obtained from~$M'$ by \emph{a type~I elementary bypass addition}; see Figure~\ref{mir-elem-moves}~(b).
\end{defi}

\begin{defi}\label{mirr-slide-def}
Let $M$ and $M'$ be mirror diagrams such that for some $\theta_1\ne\theta_2$,
$\varphi_1\ne\varphi_2$ the following holds:
\begin{enumerate}
\item
$(\theta_1,\varphi_1)$, $(\theta_1,\varphi_2)$ and $(\theta_2,\varphi_1)$ are $\diagup$-mirrors of $M$;
\item
$(\theta_1,\varphi_2)$, $(\theta_2,\varphi_1)$ and $(\theta_2,\varphi_2)$ are $\diagup$-mirrors of $M'$;
\item
there are no more mirrors of $M$ or~$M'$ in $r=[\theta_1;\theta_2]\times[\varphi_1;\varphi_2]$;
\item
$M$ and $M'$ have the same set of mirrors outside $r$ and the same sets of occupied levels.
\end{enumerate}
Then we say that $M'$ is obtained from $M$ by \emph{a type~I slide move}; see Figure~\ref{mir-elem-moves}~(c).
\begin{figure}[ht]
\raisebox{50pt}{(a)}\includegraphics{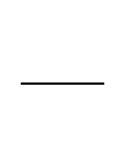}\raisebox{38pt}{$\longleftrightarrow$}\includegraphics{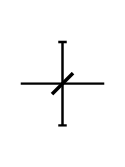}\hskip1cm
\includegraphics{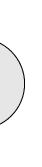}\raisebox{38pt}{$\longleftrightarrow$}\hskip2.2mm \includegraphics{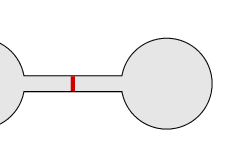}
\\
\raisebox{60pt}{(b)}\hskip.5cm\includegraphics{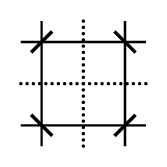}\put(-66,3){$m_{\theta_1}$}\put(-26,3){$m_{\theta_2}$}%
\put(-85,18){$\ell_{\varphi_1}$}\put(-85,58){$\ell_{\varphi_2}$}%
\raisebox{38pt}{$\longleftrightarrow$}\includegraphics{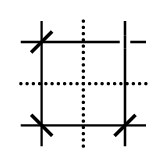}\hskip1cm
\includegraphics[height=80pt]{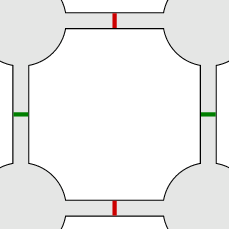}\ \raisebox{38pt}{$\longleftrightarrow$}\ \includegraphics[height=80pt]{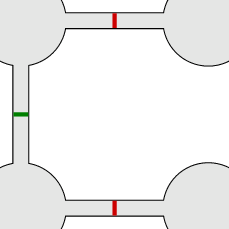}
\\[5mm]
\raisebox{60pt}{(c)}\hskip.5cm\includegraphics{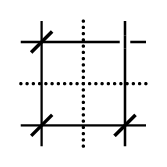}\put(-66,3){$m_{\theta_1}$}\put(-26,3){$m_{\theta_2}$}%
\put(-85,18){$\ell_{\varphi_1}$}\put(-85,58){$\ell_{\varphi_2}$}%
\raisebox{38pt}{$\longleftrightarrow$}\includegraphics{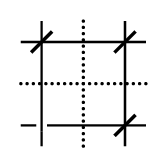}\hskip1cm
\includegraphics[height=80pt]{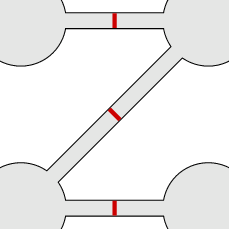}\ \raisebox{38pt}{$\longleftrightarrow$}\ \includegraphics[height=80pt]{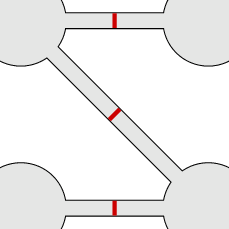}
\caption{Type~I elementary moves of mirror diagrams: (a) extension/elimination; (b) elementary bypass removal/addition; (c) slide}\label{mir-elem-moves}
\end{figure}
\end{defi}

Observe that, due to Convention~\ref{symmetries}, the inverse operation to a slide move is
also a slide move, and of the same type.

\begin{rema}As one can learn from Figure~\ref{slide} any slide move can be decomposed
into a sequence of moves including extension, elimination, elementary bypass addition, and elementary bypass removal
moves. So, the set of transformations that are decomposable into elementary moves
does not depend on whether or not the latter include slide moves.
\begin{figure}[ht]
\begin{tabular}{ccccc}
\includegraphics{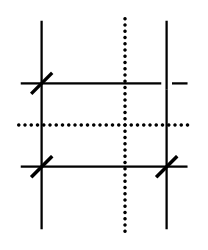}&\raisebox{58pt}{$\longrightarrow$}&
\includegraphics{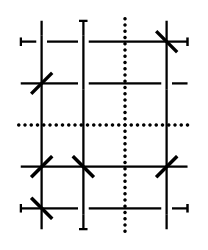}&\raisebox{58pt}{$\longrightarrow$}&
\includegraphics{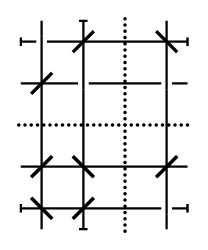}\\
&&&&$\downarrow$\\
\includegraphics{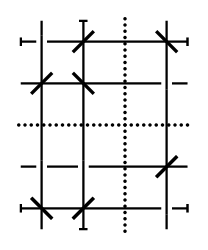}&\raisebox{58pt}{$\longleftarrow$}&
\includegraphics{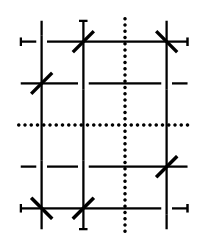}&\raisebox{58pt}{$\longleftarrow$}&
\includegraphics{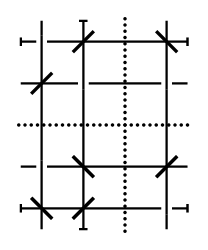}\\
$\downarrow$\\
\includegraphics{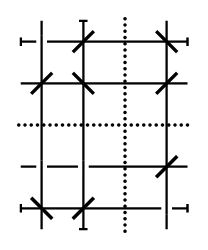}&\raisebox{58pt}{$\longrightarrow$}&
\includegraphics{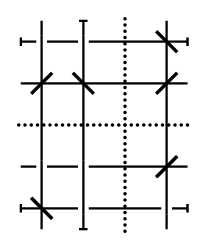}&\raisebox{58pt}{$\longrightarrow$}&
\includegraphics{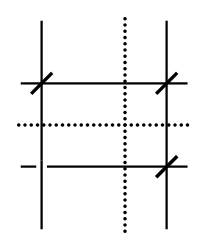}\\
\end{tabular}
\caption{Decomposition of a slide move into elementary moves of other kinds}\label{slide}
\end{figure}
However, we will need to work with transformations that are decomposable
into elementary moves of one type, either I or II, and then
the inclusion of slides into elementary moves will matter. One can see that
the decomposition in Figure~\ref{slide} involves moves of both types. This
cannot be avoided in general.
\end{rema}

\subsection{Elementary moves of enhanced mirror diagrams}
The occupied levels at which mirrors are explicitly
mentioned in Definitions~\ref{ext-move-def}, \ref{mirr-bypass-def}, and~\ref{mirr-slide-def}
and shown in Figure~\ref{mir-elem-moves}
are referred to as \emph{involved} in the respective move.

\begin{defi}\label{elem-move-enhanced-def}
Let~$M$ and~$M'$ be enhanced mirror diagrams, and let~$M_0$, $M_0'$ be the underlying
mirror diagrams (with enhancement forgotten). Assume that~$M_0\mapsto M_0'$ is an elementary move.
Then the move~$M\mapsto M'$ bears the same name (extension, elimination, elementary bypass addition, elementary bypass removal,
or slide) and the same type (I or II) as~$M_0\mapsto M_0'$ does, provided that:
\begin{enumerate}
\item
for any boundary circuit~$c$ of~$M$ that shares a non-trivial
subinterval of an occupied level not involved in the move with a boundary circuit~$c'$ of~$M'$,
either both~$c$ and~$c'$ are essential or both are inessential (for~$M$ and~$M'$, respectively);
\item
the number of essential boundary circuits of~$M$ is equal to that of~$M'$.
\end{enumerate}
\end{defi}

We now explain in more detail what this definition means for each kind of elementary moves. First, note
that for any elementary move~$M\mapsto M'$ any common boundary
circuit of~$M$ and~$M'$ is either essential for both of them or inessential for both of them.

Let~$M\mapsto M'$ be an extension or elimination move. Then there are unique~$c\in\partial M$ and~$c'\in\partial M'$
such that~$c\notin\partial M'$ and~$c'\notin\partial M$. The essentialness or inessentialness of any boundary circuit
in~$\partial M\setminus\{c\}{}=\partial M'\setminus\{c'\}$
is preserved. The boundary circuit~$c$ is essential for~$M$ if and only if so is~$c'$ for~$M'$.

Now let~$M\mapsto M'$ be an elementary bypass removal. We use the notation from Definition~\ref{mirr-bypass-def}.
$M$ has two boundary circuits that are not in~$\partial M'$, one of which is~$c_1=\partial r$. Denote
the other by~$c_2$. They are replaced in~$M'$ by a single boundary circuit $c'$
that can be viewed as the connected sum of~$c_1$ and~$c_2$.
The boundary circuit~$c_1$ of~$M$ must be inessential. The boundary circuit~$c_2$
is essential if and only if so is~$c'$.
The essentialness of any boundary circuit in~$\partial M\setminus\{c_1,c_2\}=
\partial M'\setminus\{c'\}$ is preserved by the move.

Finally let~$M\mapsto M'$ be a slide move. We use the notation from Definition~\ref{mirr-slide-def}.
There are two, not necessarily distinct, boundary circuits~$c_1$, $c_2$
of~$M$ that are replaced by another two, $c_1'$, $c_2'$, say,
so that~$c_i'$ coincides with~$c_i$ outside~$r$; see Figure~\ref{c1c2}.
\begin{figure}[ht]
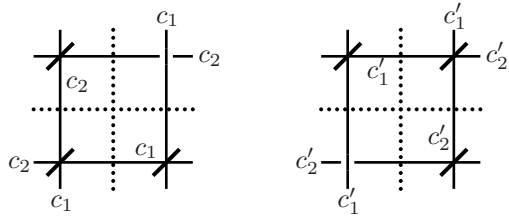

\includegraphics{mir-move5.eps}\put(-64,3){$c_1$}\put(-32,23){$c_1$}\put(-24,73){$c_1$}%
\put(-80,18){$c_2$}\put(-58,48){$c_2$}\put(-8,58){$c_2$}
\hskip1cm\includegraphics{mir-move6.eps}\put(-64,3){$c_1'$}\put(-30,27){$c_2'$}\put(-24,73){$c_1'$}%
\put(-80,18){$c_2'$}\put(-53,52){$c_1'$}\put(-8,58){$c_2'$}
\caption{The correspondence between modified boundary circuits before and after a slide move}\label{c1c2}
\end{figure}
We have~$c_1=c_2$
if and only if~$c_1'=c_2'$. The circuit~$c_i$ is essential for~$M$
if and only if so is~$c_i'$ for $M'$, $i=1,2$.
The essentialness of any boundary circuit from~$\partial M\setminus\{c_1,c_2\}=
\partial M'\setminus\{c_1',c_2'\}$ is preserved by the move.

\begin{rema}
Let~$M_0$ be the underlying mirror diagram of an enhanced mirror diagram~$M$,
and let~$M_0\mapsto M_0'$ be an elementary move. If this move is an extension, elimination,
or a slide move, then there exists a unique enhanced mirror diagram~$M'$ with
the underlying mirror diagram~$M_0'$ such that~$M\mapsto M'$ is also
an extension, elimination, or slide move, respectively.

If~$M_0\mapsto M_0'$ is an elementary bypass removal, then the respective~$M'$
may not exist. Indeed, a boundary circuit of~$M$ having form of the boundary of a rectangle,
which disappears in~$M'$, must be inessential for~$M'$ to exist. If it is essential, the move does not fit into
the definition of the enhanced version of the move. If~$M'$ exists, then it is unique
in this case.

If~$M_0\mapsto M_0'$ is an elementary bypass addition,
then the respective~$M'$ always exists but it may not be unique. This arises from the fact that
$M_0$ and~$M_0'$ may fit into the settings of Definition~\ref{mirr-bypass-def}, with Convention~\ref{symmetries} in force,
in two different ways.

To see this let~$r_1,r_2\subset\mathbb T^2$
two rectangles sharing a vertex and being disjoint otherwise. Suppose that both~$\partial r_1$ and~$\partial r_2$
are boundary circuits of~$M_0'$ and there are no mirrors of~$M_0'$ inside the rectangles~$r_1$, $r_2$.
Suppose also that the mirror added by the move~$M_0\mapsto M_0'$ is the common vertex of~$r_1$ and~$r_2$.
Denote by~$c$ the boundary circuit of~$M$ that is replaced by~$\partial r_1$, $\partial r_2$
as a result of the move. If~$c$ is essential for~$M$, then exactly one of~$\partial r_1$, $\partial r_2$
must be essential for~$M'$, but it may be either of them.

A similar situation occurs when two rectangles whose boundaries are boundary circuits of~$M_0'$
share two or four vertices.

It can also happen that a transformation~$M\mapsto M'$ fits into the definition
of a slide move in two different ways, but in this case only one boundary circuit
is modified, so the enhancement of~$M$ prescribes that of~$M'$.
\end{rema}

\begin{rema}
Conditions~(1) and~(2) of Definition~\ref{elem-move-enhanced-def} are always
satisfied when all patchable boundary circuits of~$M$ and~$M'$ are inessential, which
is the `default' option.
\end{rema}

\begin{defi}
For every elementary move~$M\mapsto M'$ of enhanced mirror diagrams we define
\emph{the associated morphism} of the respective enhanced spatial ribbon graphs as follows.

If~$M\mapsto M'$ is an extension move, the associated morphism~$\eta$
is defined by~$\bigl(\wideparen M',\wideparen M',\mathrm{id}|_{\wideparen M'}\bigr)\in\eta$.

If~$M\mapsto M'$ is an elementary bypass addition, then any surface carried by~$\widehat M'$ is also
carried by~$\widehat M$. Take any such surface~$F$. The associated morphism~$\eta$ is defined
by the condition~$(F,F,\mathrm{id}|_F)\in\eta$.

For elimination moves and elementary bypass removals the associated
morphisms are defined as the inverses to the morphisms associated
with the inverses of the respective moves.

Now let~$M\mapsto M'$ be a slide move. We use the notation from Definition~\ref{mirr-slide-def}.
Denote also the mirrors at~$(\theta_1,\varphi_1)$, $(\theta_2,\varphi_1)$, $(\theta_1,\varphi_2)$,
$(\theta_2,\varphi_2)$ by~$\mu_1$, $\mu_2$, $\mu_3$, $\mu_4$, respectively,
and the discs
$$\wideparen m_{\theta_1}\cup\wideparen m_{\theta_2}\cup\wideparen\ell_{\varphi_1}\cup
\wideparen\ell_{\varphi_2}\cup\wideparen\mu_1\cup\wideparen\mu_2\cup\wideparen\mu_3\subset\wideparen M
\quad\text{ and }\quad
\wideparen m_{\theta_1}\cup\wideparen m_{\theta_2}\cup\wideparen\ell_{\varphi_1}\cup
\wideparen\ell_{\varphi_2}\cup\wideparen\mu_2\cup\wideparen\mu_3\cup\wideparen\mu_4\subset\wideparen M'$$
by~$d$ and~$d'$, respectively.
Let~$h$ be a homeomorphism from~$\wideparen M$ to~$\wideparen M'$ such that
$h|_{M\setminus d}=\mathrm{id}$. If~$d$ is not a connected component of~$\wideparen M$,
then this determines~$h$ up to isotopy. If~$d$ is a connected component of~$\wideparen M$
we demand additionally that~$h|_d:d\rightarrow d'$ is orientation-preserving, where
the orientations of~$d$ and~$d'$ are chosen so that they agree on~$d\cap d'$.

The associated morphism~$\eta$ is now defined by the condition~$(\wideparen M,\wideparen M',h)\in\eta$.

To state that~$\eta$ is the morphism associated with the move~$M\mapsto M'$ we write~$M\xmapsto\eta M'$.
\end{defi}

\subsection{Stable equivalence via elementary moves}
We now express relative stable equivalence of spatial ribbon graphs presented
by mirror diagrams, in combinatorial terms.

\begin{theo}\label{relative-stable-equivalence-th}
Let~$M$ and~$M'$ be enhanced mirror diagrams, and let $C=\{c_1,\ldots,c_k\}$ be
a simple family of their common essential boundary circuits. Let also~$\eta$ be a morphism from~$\widehat M$
to~$\widehat M'$.
The following two conditions are equivalent.
\begin{enumerate}
\item
The enhanced ribbon
graphs~$\widehat M$ and~$\widehat M'$ are stably equivalent relative to~$L=\bigcup\limits_{i=1}^k\widehat c_i$, and some isotopy realizing this stable equivalence induces~$\eta$.
\item
There exists a sequence of elementary moves
$$M=M_0\xmapsto{\eta_1}M_1\xmapsto{\eta_2}\ldots\xmapsto{\eta_n}M_n=M'$$
preserving all the boundary circuits in~$C$, such that~$\eta=\eta_n\circ\ldots\circ\eta_1$.
\end{enumerate}
\end{theo}

\begin{proof}
The part~(2)$\Rightarrow$(1) is quite obvious and left to the reader. The inverse implication is
a consequence of Lemmas~\ref{handle-reduction-seq-lem} and~\ref{handle-reduction-decomp-lem} below.
\end{proof}

\begin{lemm}\label{handle-reduction-seq-lem}
Let~$M$ and~$M'$ be enhanced mirror diagrams, and let~$C=\{c_1,\ldots,c_k\}$ be
a simple family of their common essential boundary circuits.
If the enhanced ribbon
graphs~$\widehat M$ and~$\widehat M'$ are stably equivalent relative to
the framed link~$L=\bigcup\limits_{i=1}^k\widehat c_i$,
and~$\eta$ is a morphism from~$\widehat M$ to~$\widehat M'$ induced by an isotopy
realizing the stable equivalence,
then there exists a sequence of enhanced mirror diagrams
$$M=M_0,M_1,\ldots,M_n=M'$$
such that each transition~$\widehat M_{i-1}\xmapsto{\eta_i}\widehat M_i$, $i=1,\ldots,n$,
is a handle addition or a handle removal preserving all the boundary circuits in~$L$,
and we have~$\eta=\eta_n\circ\ldots\circ\eta_1$.
\end{lemm}

\begin{proof}
We construct the sought-for sequence of handle additions and removals in four
steps.

\smallskip\noindent\emph{Step 1.}
Reduce to the case when~$E_M\cap E_{M'}$ consists only of the mirrors that are hit by
the boundary circuits~$c_i$, $i=1,\ldots,k$:

Suppose that both~$M$ and~$M'$ have a mirror at~$(\theta_0,\varphi_0)$
(not necessarily of the same type) that is not hit by any~$c_i$. Suppose also
that it is a~$\diagup$-mirror of~$M$ (the case of a $\diagdown$-mirror is similar). Choose~$\varepsilon>0$
so small that~$\varnothing=(\theta_0;\theta_0+\varepsilon]\cap\Theta_M=(\theta_0;\theta_0+\varepsilon]\cap\Theta_{M'}=[\varphi_0-\varepsilon;\varphi_0)\cap\Phi_M=[\varphi_0-\varepsilon;\varphi_0)\cap\Phi_{M'}$.

Let~$M_1$ be a mirror diagram obtained from~$M$ by the following sequence of elementary moves (see Figure~\ref{stabilization-decomp-fig}):
\begin{itemize}
\item
two extension moves that add new occupied levels,~$m_{\theta_0+\varepsilon}$ and~$\ell_{\varphi_0-\varepsilon}$,
and two $\diagdown$-mirrors at~$(\theta_0+\varepsilon,\varphi_0)$
and~$(\theta_0,\varphi_0-\varepsilon)$;
\item
an elementary bypass addition that inserts a $\diagup$-mirror at~$(\theta_0+\varepsilon,\varphi_0-\varepsilon)$.
\end{itemize}
Denote by~$\eta_1$ the morphism from~$\widehat M$ to~$\widehat M_1$ that is the composition of the morphisms
associated with these moves.

Let~$M_2$ be the diagram obtained from~$M_1$ by an elementary bypass removal that deletes the mirror at~$(\theta_0,\varphi_0)$,
and let~$\eta_2$ be the associated morphism from~$\widehat M_1$ to~$\widehat M_2$.
\begin{figure}[ht]
\includegraphics{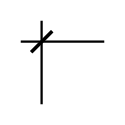}\put(-45,53){$m_{\theta_0}$}\put(-64,38){$\ell_{\varphi_0}$}\put(-35,0){$M$}\raisebox{28pt}{$\longrightarrow$}%
\includegraphics{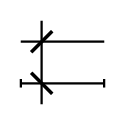}\raisebox{28pt}{$\longrightarrow$}%
\includegraphics{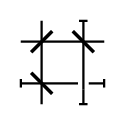}\raisebox{28pt}{$\longrightarrow$}%
\includegraphics{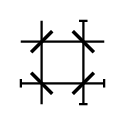}\put(-35,0){$M_1$}\raisebox{28pt}{$\longrightarrow$}%
\includegraphics{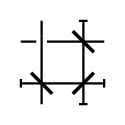}\put(-35,0){$M_2$}
\caption{Removing an unwanted mirror by elementary moves}\label{stabilization-decomp-fig}
\end{figure}
One can see that~$\widehat M\mapsto\widehat M_1$ is a handle addition, and~$\widehat M_1\mapsto\widehat M_2$ is
a handle removal.

The enhanced spatial ribbon graphs~$\widehat M_2$ and~$\widehat
M'$ are still stably equivalent relative to~$L$, and there is an isotopy realizing this stable
equivalence that induces~$\eta\circ\eta_1^{-1}\circ\eta_2^{-1}$. The set~$E_{M_2}$ shares fewer elements with~$E_{M'}$
than~$E_M$ does. By repeating this step finitely many times we get an enhanced mirror diagram~$M''$
that is obtained from~$M$ by a sequence of elementary moves preserving~$L$, and is such that all mirrors in~$E_{M''}\cap E_{M'}$
are hit by some~$c_i$, $i\in\{1,\ldots,k\}$. We redenote~$M''$ by~$M$ and proceed to the next step.

\smallskip\noindent\emph{Step 2.}
Find a transition from~$\widehat M$ to~$\widehat M'$ via spatial ribbon graphs, and put them into a single surface~$F$:

Denote by~$M\cup M'$ the mirror diagram defined by the following conditions:
$L_{M\cup M'}=L_M\cup L_{M'}$, $E_{M\cup M'}=E_M\cup E_{M'}$,
$T_{M\cup M'}|_{E_M}=T_M$, $T_{M\cup M'}|_{E_{M'}}=T_{M'}$. It exists since~$T_M$ and~$T_{M'}$ agree on~$E_M\cap E_{M'}$.

By Proposition~\ref{stable-equivalence-of-ribbon-graphs-prop} there exists a sequence of handle additions and removals
$$\widehat M=\rho_0\xmapsto{\eta_1}\rho_1\xmapsto{\eta_2}\ldots\xmapsto{\eta_N}\rho_N=\widehat M'$$
preserving the boundary circuits~$\widehat c_i$, $i=1,\ldots,k$, and such that~$\eta=\eta_N\circ\ldots\circ\eta_2\circ\eta_1$. Such a sequence
will be called \emph{a guiding sequence} for the transition~$M\mapsto M'$.

A generic choice of a guiding sequence ensures that the edges of the graphs~$\Gamma_{\rho_i}$
that are added or removed in the course of transforming~$\Gamma_{\rho_0}$ into~$\Gamma_{\rho_N}$
do not have unnecessary intersections with each other, and
all the graphs~$\Gamma_{\rho_i}$, $i=0,\ldots,N$,
are contained in a single compact surface~$F$ such that, for each~$i$, there is a surface~$F_i\in S_{\rho_i}$
contained in~$F$ and containing an open neighborhood of~$\Gamma_{\rho_i}$ in~$F$.
Let~$\Gamma$ be the union~$\bigcup\limits_{i=1}^k\Gamma_{\rho_k}$.
We may also assume that~$F$ is tangent to~$\xi_+$ along~$\widehat\mu$ if~$\mu$ is a~$\diagup$-mirror
of~$M\cup M'$ and to~$\xi_-$ if~$\mu$ is a $\diagdown$-mirror of~$M\cup M'$.

\smallskip\noindent\emph{Step 3.}
Make the surface~$F$ `rectangular':

By Lemma~\ref{rectangular-representative-lem} there is an isotopy relative to~$\widehat{M\cup M'}$
from~$F$ to a surface of the form~$\widehat\Pi$, where~$\Pi$ is a rectangular diagram of a surface.
This isotopy takes~$\rho_i$, $i=0,\ldots,N$, to spatial ribbon graphs that still form a guiding
sequence for the transition~$M\mapsto M'$. So we may assume from the beginning that~$F=\widehat\Pi$.

\smallskip\noindent\emph{Step 4.}
Make~$\Gamma$ a subcomplex of the tiling of~$\Pi$:

Fix an open neighborhood~$U\subset\widehat\Pi$ of~$\Gamma_{\widehat{M\cup M'}}$ such that~$U$ retracts to~$\Gamma_{\widehat{M\cup M'}}$,
and~$U\cap\Gamma_{\rho_i}$ retracts to~$\Gamma_{\widehat{M\cup M'}}$ for any~$i=1,\ldots,N-1$. By using 
wrinkle creation moves
we can find a rectangular diagram of a surface~$\Pi_1$ and
a homeomorphism~$h:\widehat\Pi\rightarrow\widehat\Pi_1$ such that
\begin{enumerate}
\item
$\Pi_1$ is obtained from~$\Pi$ by finitely many wrinkle creation moves;
\item
there is an isotopy relative to~$\widehat{M\cup M'}$ from~$\mathrm{id}|_{\widehat\Pi}$ to~$h$;
\item
the preimage of the tiling of~$\widehat\Pi_1$ under~$h$ is arbitrarily fine outside~$U$, i.e. for any rectangle~$r\in\Pi_1$ the diameter of~$h^{-1}(\widehat r)\setminus U$ is smaller than any~$\delta>0$ chosen in advance.
\end{enumerate}
This means that such~$\Pi_1$ and~$h$ can be chosen so that~$h(\Gamma)$
is contained in the $1$-skeleton of the tiling of~$\widehat\Pi_1$.

We replace~$\Pi$ by~$\Pi_1$, and $\rho_i$ by~$h(\rho_i)$, $i=0,1,\ldots,N$.
The new~$\rho_i$'s still form a guiding sequence for the transition~$M\mapsto M'$ and
each~$\rho_i$ now has the form~$\widehat M_i$, where~$M_i$ is an enhanced mirror diagram.
This completes the proof.
\end{proof}

\begin{lemm}\label{handle-reduction-decomp-lem}
Let~$M$ and~$M'$ be enhanced mirror diagrams such that~$\widehat M\xmapsto\eta\widehat M'$ is a handle
addition. Then there is a sequence
$$M=M_0\xmapsto{\eta_1}M_1\xmapsto{\eta_2}\ldots\xmapsto{\eta_n}M_n=M'$$
of elementary moves preserving all common boundary circuits of~$M$ and~$M'$
such that~$\eta=\eta_n\circ\ldots\circ\eta_2\circ\eta_1$.
\end{lemm}

\begin{proof}
By Lemma~\ref{rectangular-representative-lem} a surface~$F\in\widetilde S_{\widehat M}$
containing a patching disc associated with the move~$M\mapsto M'$ can
be found in the form~$\widehat \Pi$, where~$\Pi$ is
a rectangular diagram of a surface. We fix the surface~$F$ from now on.
In the construction below, for all~$i$ we will have~$F\in\widetilde S_{\widehat M_i}$
and~$(F,F,\mathrm{id}|_F)\in\eta_i$, so the equality~$\eta=\eta_n\circ\ldots\circ\eta_2\circ\eta_1$
will hold trivially.

The number of tiles of~$F=\widehat\Pi$
that form a patching disc associated with the handle addition~$\widehat M\mapsto\widehat M'$
will be called \emph{the complexity} of this handle addition.
The proof of the Lemma is by induction in this complexity.

Suppose that the move~$M\mapsto M'$ has complexity~$1$. This means
that there is a patching disc of the form~$\widehat r$, $r\in\Pi$. Clearly, $M'$ can be obtained from~$M$
by zero, one, two, or~three extension moves followed
by an elementary bypass addition. This gives the induction base.

Suppose that the move~$M\mapsto M'$ has complexity~$k>1$. Let~$d\subset F$ be a patching disc
associated with the move~$M\mapsto M'$ consisting of~$k$ tiles of~$F$. The disc~$d$ can be cut
into two nontrivial parts by a simple arc~$\alpha\subset\bigcup\limits_{i=1}^k\partial\widehat r_i$
such that at least one of the endpoints of~$\alpha$
lies in~$\Gamma_{\widehat M}$. Denote by~$M_1$ and~$M_1'$ the mirror diagrams defined by the following conditions:
\begin{enumerate}
\item
$\Gamma_{\widehat M_1}=\Gamma_{\widehat M}\cup\alpha$, $\Gamma_{\widehat M_1'}=\Gamma_{\widehat M'}\cup\alpha$;
\item
some surfaces~$F_1\in S_{\widehat M_1}$ and~$F_1'\in S_{\widehat M_1'}$ are tangent to~$d$
along~$d\cap\Gamma_{\widehat M_1}$ and~$d\cap\Gamma_{\widehat M_1'}$,
respectively.
\end{enumerate}
There are the following two cases to consider.

\smallskip\emph{Case 1.}
Both endpoints of~$\alpha$ lie in~$\Gamma_{\widehat M}$. Then for appropriate enhancements of~$M_1$ and~$M_1'$
both transitions~$M\mapsto M_1$ and~$M_1\mapsto M_1'$
are handle additions of complexity smaller than~$k$, and the transition~$M_1'\mapsto M'$ is a handle removal of complexity smaller than~$k$;
see Figure~\ref{handle-induction-case1-fig}.
\begin{figure}[ht]
\includegraphics{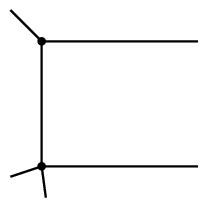}\put(-55,5){$\Gamma_{\widehat M}$}\raisebox{48pt}{$\longrightarrow$}%
\includegraphics{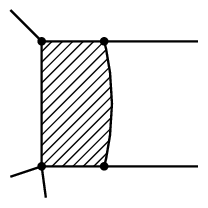}\put(-55,5){$\Gamma_{\widehat M_1}$}\put(-44,48){$\alpha$}\raisebox{48pt}{$\longrightarrow$}%
\includegraphics{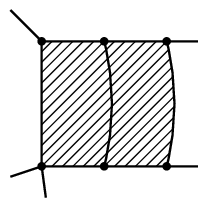}\put(-55,5){$\Gamma_{\widehat M_1'}$}\raisebox{48pt}{$\longrightarrow$}%
\includegraphics{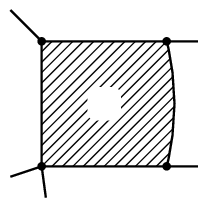}\put(-53,47){$d$}\put(-55,5){$\Gamma_{\widehat M'}$}
\caption{Induction step in Lemma~\ref{handle-reduction-decomp-lem}, Case~$1$}\label{handle-induction-case1-fig}
\end{figure}

\smallskip\emph{Case 2.}
Only one endpoint of~$\alpha$ lies in~$\Gamma_{\widehat M}$. Then for appropriate enhancements of~$M_1$ and~$M_1'$ the transition~$M\mapsto M_1$ decomposes into a few
extension moves. The transition~$M_1\mapsto M_1'$ decomposes into two handle additions of complexity smaller than~$k$.
Finally, the transition~$M_1'\mapsto M'$ is again a handle removal of complexity smaller than~$k$;
see Figure~\ref{handle-induction-case2-fig}.
\begin{figure}[ht]
\includegraphics{handle-induction1.eps}\put(-55,5){$\Gamma_{\widehat M}$}\raisebox{48pt}{$\longrightarrow$}%
\includegraphics{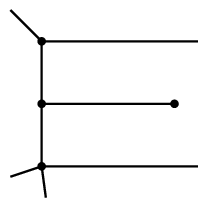}\put(-55,5){$\Gamma_{\widehat M_1}$}\put(-50,53){$\alpha$}\raisebox{48pt}{$\longrightarrow$}%
\includegraphics{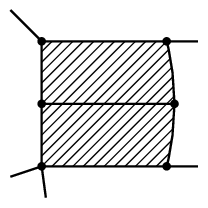}\put(-55,5){$\Gamma_{\widehat M_1'}$}\raisebox{48pt}{$\longrightarrow$}%
\includegraphics{handle-induction4.eps}\put(-53,47){$d$}\put(-55,5){$\Gamma_{\widehat M'}$}
\caption{Induction step in Lemma~\ref{handle-reduction-decomp-lem}, Case~$2$}\label{handle-induction-case2-fig}
\end{figure}

Both cases give the induction step.
\end{proof}

\section{More transformations. Neat decompositions}\label{neat-decomp-sec}
\subsection{Preliminary remarks and conventions}
This is a purely technical section in which we introduce further families of moves of mirror diagrams that are
needed to establish our main results,
and show how to decompose the new moves into elementary ones.

For brevity we give definitions in terms of ordinary, that is, not enhanced, mirror diagrams,
and extend them to enhanced mirror diagrams by using the following convention.

\begin{conv}\label{moves-extension-conv}
Definition~\ref{elem-move-enhanced-def} extends to all the moves defined in this section,
with involved occupied levels understood in a broader sense. Namely,
on occupied level is involved if it contains the mirrors explicitly mentioned
in the definition of the move and/or mirrors whose position is changed
by the move.

Additionally, if the definition of a move~$M\mapsto M'$ prescribes~$M$ or~$M'$ to have a circuit~$c$
that hits exactly four mirrors, each on one side
(this occurs in the cases of bridge, wrinkle, and double split moves),
then~$c$ is demanded to be inessential in the extension of the definition to enhanced mirror diagrams.
\end{conv}

\begin{defi}
Let~$M\mapsto M'$ be an elementary move or
one of the moves defined in this section, and let~$c$ and~$c'$ be essential boundary circuits of~$M$ and~$M'$,
respectively. We say that the move~$M\mapsto M'$ \emph{transforms~$c$ to~$c'$} if~$c$ and~$c'$ either share
a nondegenerate interval of a common occupied level of~$M$ and~$M'$ which is not involved in the move,
or both~$c$ and~$c'$ are contained entirely in the occupied levels involved in the move.

We also use this terminology for subsets of the essential boundaries of~$M$ and~$M'$: $C\subset\partial_{\mathrm e}M$
is transformed to~$C'\subset\partial_{\mathrm e}M$ if each~$c\in C$ is transformed to some~$c'\in C'$,
and for each~$c'\in C'$ there is a~$c\in C$ transformed to~$c'$.
\end{defi}

For each kind of moves~$M\mapsto M'$ that we define in this section, it is elementary to show that the relation `$c$ is transformed to~$c'$'
is a one-to-one correspondence between essential boundaries~$\partial_{\mathrm e}M$ and~$\partial_{\mathrm e}M'$.

Every move~$M\mapsto M'$ that we define below is accompanied by an associated morphism of the
respective enhanced spatial ribbon graphs~$\widehat M\rightarrow\widehat M'$. By
writing~$M\xmapsto\eta M'$ we always mean that~$\eta$ is the associated morphism.

\begin{defi}\label{neat-def}
Let a transformation $M\xmapsto\eta M'$ of enhanced mirror diagrams be decomposed into
a sequence of other transformations:
\begin{equation}\label{neat-eq}
M=M_0\xmapsto{\eta_1}M_1\xmapsto{\eta_2}\ldots\xmapsto{\eta_k}M_k=M',
\end{equation}
each endowed with a morphism of the respective enhanced spatial ribbon graphs,
and let~$C$ be a collection of boundary circuits of~$M$ preserved by the transformation~$M\xmapsto\eta M'$.
We say that decomposition~\eqref{neat-eq} is \emph{$C$-neat} if the following holds:
\begin{enumerate}
\item
each boundary circuit in~$C$ is
preserved by all the transformations $M_{i-1}\mapsto M_i$,
$i=1,\ldots,k$;
\item
for any~$0\leqslant i_1<i_2\leqslant k$,
whenever an essential boundary circuit~$c$ of~$M_{i_1}$ and the respective
boundary circuit~$c'$ of~$M_{i_2}$ have negative Thurston--Bennequin number~$\tb_+$
(respectively,~$\tb_-$), that is, hit at least one $\diagdown$-mirror (respectively, $\diagup$-mirror),
so have the corresponding boundary circuits of all~$M_j$'s with~$j\in[i_1;i_2]$;
\item
$\eta=\eta_k\circ\ldots\circ\eta_2\circ\eta_1$.
\end{enumerate}
If~$C$ includes all boundary circuits of~$M$ preserved by the transformation~$M\xmapsto\eta M'$, a $C$-neat decomposition~\eqref{neat-eq} is called~\emph{neat} for short.
\end{defi}

We will often use the following obvious properties of $C$-neat decompositions.

\begin{lemm}\label{neat-properties-lem}
Let~$M_0\xmapsto{\eta_1}M_1\xmapsto{\eta_2}\ldots\xmapsto{\eta_k}M_k$ be a decomposition
of a transformation~$M_0\xmapsto\eta M_k$ of enhanced mirror diagrams preserving
a collection~$C$ of boundary circuits,
and let~$i,j$ be such that~$0\leqslant i<j\leqslant k$.

The following
statements are equivalent:
\begin{enumerate}
\item
the decomposition~$M_0\xmapsto{\eta_1}M_1\xmapsto{\eta_2}\ldots\xmapsto{\eta_k}M_k$ of~$M_0\xmapsto\eta M_k$ 
is $C$-neat;
\item
the decomposition~$M_k\xmapsto{\eta_k^{-1}}\ldots\xmapsto{\eta_2^{-1}}M_1\xmapsto{\eta_1^{-1}}M_0$
of~$M_k\xmapsto{\eta^{-1}}M_0$ is $C$-neat;
\item
the decompositions~$M_0\xmapsto{\eta_1}M_1\xmapsto{\eta_2}\ldots\xmapsto{\eta_i}M_i
\xmapsto{\eta_j\circ\eta_{j-1}\circ\ldots\circ\eta_{i+1}}M_j\xmapsto{\eta_{j+1}}\ldots\xmapsto{\eta_k}M_k$
and~$M_i\xmapsto{\eta_{i+1}}M_{i+1}\xmapsto{\eta_{i+2}}\ldots\xmapsto{\eta_j}M_j$
of~$M_0\xmapsto\eta M_k$ and~$M_i\xmapsto{\eta_j\circ\eta_{j-1}\circ\ldots\circ\eta_{i+1}}M_j$,
respectively,
are $C$-neat.
\end{enumerate}
\end{lemm}

\begin{defi}\label{safe-def}
A transformation~$M\xmapsto\eta M'$ of enhanced mirror diagrams is said to be
\emph{safe-to-bring-forward}
if, for any essential boundary circuit~$c\in\partial M$ and the corresponding
boundary circuit~$c'\in\partial M'$, we have~$\tb_+(c)<0\Rightarrow\tb_+(c')<0$
and~$\tb_-(c)<0\Rightarrow\tb_-(c')<0$.

If these conditions hold only for~$\tb_+$ (respectively, for~$\tb_-$),
then the transformation~$M\xmapsto\eta M'$ is said to be
\emph{$+$-safe-to-bring-forward} (respectively, \emph{$-$-safe-to-bring-forward}).

The inverse of a safe-to-bring-forward (respectively, $+$-safe-to-bring-forward or
$-$-safe-to-bring-forward) transformation is said to be
\emph{safe-to-postpone} (respectively, $+$-safe-to-postpone
or $-$-safe-to-postpone).

If a transformation is safe-to-bring-forward and safe-to-postpone, it is
said to be \emph{safe}. Similarly, for $\pm$-versions.
\end{defi}

The reason for these names is the following statement, which is also obvious.

\begin{lemm}\label{safe-lem}
Let~$M_0\xmapsto{\eta_1}M_1\xmapsto{\eta_2}M_2$ be
a decomposition of a transformation~$M_0\xmapsto\eta M_2$,
$\eta=\eta_2\circ\eta_1$,
and let~$C\subset\partial M_0$,
be a collection of boundary circuits preserved by
all three transformations~$M_0\xmapsto\eta M_2$, $M_0\xmapsto{\eta_1}M_1$,
and~$M_1\xmapsto{\eta_2}M_2$. Then the decomposition~$M_0\xmapsto{\eta_1}M_1\xmapsto{\eta_2}M_2$ is $C$-neat
once at least one of the following holds:
\begin{enumerate}
\item
$M_0\xmapsto{\eta_1}M_1$ is safe-to-bring-forward;
\item
$M_1\xmapsto{\eta_2}M_2$ is safe-to-postpone;
\item
$M_0\xmapsto{\eta_1}M_1$ is $+$-safe-to-bring-forward and $M_1\xmapsto{\eta_2}M_2$
is $-$-safe-to-postpone;
\item
$M_0\xmapsto{\eta_1}M_1$ is $-$-safe-to-bring-forward and $M_1\xmapsto{\eta_2}M_2$
is $+$-safe-to-postpone.
\end{enumerate}
\end{lemm}

In what follows, all decompositions that are claimed to be neat can be
obtained from the trivial one (just a single transformation)
by recursively applying one of the following two operations or their inverses:
\begin{enumerate}
\item
replacing some move in the decomposition by a sequence of two moves for which
one of the cases~(1)--(4) of Lemma~\ref{safe-lem} occurs;
\item
append a safe-to-postpone move at the end of the sequence or a safe-to-bring-forward at the beginning.
\end{enumerate}
Due to Lemmas~\ref{neat-properties-lem} and~\ref{safe-lem}
this guarantees that all the obtained decompositions obey Condition~(2) of Definition~\ref{neat-def}.

Observe that all type~I elementary moves
are $+$-safe, and all type~II elementary moves are $-$-safe. Slide moves of both types
are just safe. Type~I extension moves and elementary bypass removals are $-$-safe-to-bring-forward,
whereas type~I elimination moves and elementary bypass additions are $-$-safe-to-postpone.
Symmetrically for type~II moves.

Condition~(3) in Definition~\ref{neat-def} is very important, but in the situations we need
to consider, either it holds trivially or follows from an easy check. For instance, if the modified
part of the surface is a disc some part of whose boundary is fixed, then we need not worry about Condition~(3),
since it holds true automatically.

For the reasons mentioned above, when we prove below that some decomposition is neat, we concentrate
on showing that the result of the obtained transformation is what we claim it to be, and don't always
provide separate comments why Conditions~(2) and~(3) of Definition~\ref{neat-def} are satisfied.

\subsection{Neutral moves}
It is important for us to distinguish between type~I and type~II elementary moves. However, there
are certain transformations of mirror diagrams that can be decomposed into
a few elementary moves of either type (and thus are safe). We call such transformations \emph{neutral moves}.
They include additions and removals of a bridge, jump moves, wrinkle creations and reductions, and double split and merge moves,
which are defined below.

\begin{defi}
Let $M$ and $M'$ be mirror diagrams such that for some $\theta_1\ne\theta_2$,
$\varphi_1\ne\varphi_2$ the following holds:
\begin{enumerate}
\item
$\mu_1=(\theta_1,\varphi_1)$ is a $\diagdown$-mirror of both $M$ and $M'$;
\item
$\mu_2=(\theta_2,\varphi_1)$ is a $\diagup$-mirror of both $M$ and $M'$;
\item
$\ell_{\varphi_2}$ is not an occupied level of $M$;
\item
$\mu_3=(\theta_1,\varphi_2)$ is a $\diagup$-mirror of $M'$;
\item
$\mu_4=(\theta_2,\varphi_2)$ is a $\diagdown$-mirror of $M'$;
\item
there are no other mirrors of $M$ or~$M'$ in $r=[\theta_1;\theta_2]\times[\varphi_1;\varphi_2]$;
\item
$M$ and $M'$ have the same set of mirrors outside $r$, and we have $\Theta_M=\Theta_{M'}$,
$\Phi_M=\Phi_{M'}\setminus\{\varphi_2\}$.
\end{enumerate}
Then we say that $M'$ is obtained from $M$ by \emph{adding a bridge},
and $M$ is obtained from $M'$ by \emph{removing a bridge}; see Figure~\ref{bridge-twist}~(a).

The morphism~$\eta$ associated with the bridge addition~$M\mapsto M'$ is defined
by~$\bigl(\wideparen M,\wideparen M,\mathrm{id}|_{\wideparen M}\bigr)\in\eta$.
\end{defi}

\begin{defi}\label{twist-def}
Let $M$ and $M'$ be mirror diagrams such that for some $\theta_1\ne\theta_2$ and
$\varphi_0$ the following holds:
\begin{enumerate}
\item
$\mu_1=(\theta_1,\varphi_0)$ is a $\diagdown$-mirror of $M$ and a $\diagup$-mirror of $M'$;
\item
$\mu_2=(\theta_2,\varphi_0)$ is a $\diagdown$-mirror of $M'$ and a $\diagup$-mirror of $M$;
\item
there are no more mirrors of $M$ or~$M'$ at $\ell_{\varphi_0}$;
\item
$M$ and $M'$ have the same set of mirrors outside $\ell_{\varphi_0}$, and the same sets of occupied levels.
\end{enumerate}
Then we say that $M$ and $M'$ are obtained from one another by \emph{a twist move}; see Figure~\ref{bridge-twist}~(b).

The morphism~$\eta$ associated with the twist move~$M\mapsto M'$ is defined by the condition~$(\wideparen M,\wideparen M',h)\in\eta$,
where~$h$ is a homeomorphism~$\wideparen M\rightarrow\wideparen M'$ identical outside of~$\wideparen\mu_1\cup
\wideparen\mu_2\cup\wideparen\ell_{\varphi_0}$.
\end{defi}
\begin{figure}[ht]
\raisebox{60pt}{(a)}\includegraphics{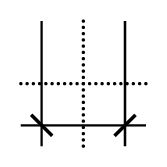}\put(-66,3){$m_{\theta_1}$}\put(-26,3){$m_{\theta_2}$}%
\put(-85,18){$\ell_{\varphi_1}$}\put(-74,28){$\mu_1$}\put(-14,28){$\mu_2$}
\raisebox{38pt}{$\longleftrightarrow$}\includegraphics{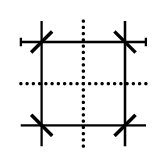}\put(-85,58){$\ell_{\varphi_2}$}%
\put(-74,28){$\mu_1$}\put(-14,28){$\mu_2$}\put(-74,50){$\mu_3$}\put(-14,50){$\mu_4$}
\hskip.7cm
\raisebox{15pt}{\includegraphics[scale=0.65]{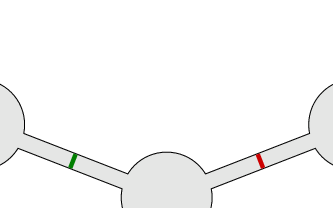}\put(-57,3){$\wideparen\ell_{\varphi_1}$}%
\put(-90,4){$\wideparen\mu_1$}\put(-22,4){$\wideparen\mu_2$}}%
\raisebox{38pt}{\ $\longleftrightarrow$\ }\raisebox{15pt}{\includegraphics[scale=0.65]{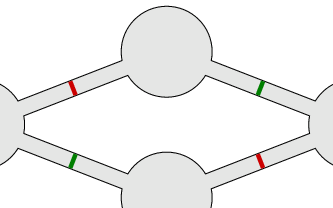}\put(-57,3){$\wideparen\ell_{\varphi_1}$}%
\put(-90,4){$\wideparen\mu_1$}\put(-22,4){$\wideparen\mu_2$}\put(-90,45){$\wideparen\mu_3$}\put(-22,45){$\wideparen\mu_4$}%
\put(-57,44){$\wideparen\ell_{\varphi_2}$}}
\\
\raisebox{45pt}{(b)}\includegraphics{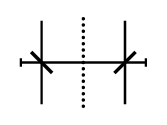}\put(-66,3){$m_{\theta_1}$}\put(-26,3){$m_{\theta_2}$}%
\put(-85,28){$\ell_{\varphi_0}$}\put(-74,38){$\mu_1$}\put(-14,38){$\mu_2$}%
\raisebox{28pt}{$\longleftrightarrow$}\includegraphics{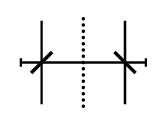}%
\hskip.7cm
\raisebox{14pt}{\includegraphics[scale=0.65]{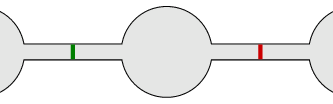}\put(-57,13){$\wideparen\ell_{\varphi_0}$}%
\put(-90,4){$\wideparen\mu_1$}\put(-22,4){$\wideparen\mu_2$}}%
\raisebox{28pt}{\ $\longleftrightarrow$\ }\raisebox{14pt}{\includegraphics[scale=0.65]{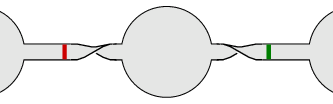}}
\caption{(a) Bridge moves and (b) twist moves}\label{bridge-twist}
\end{figure}
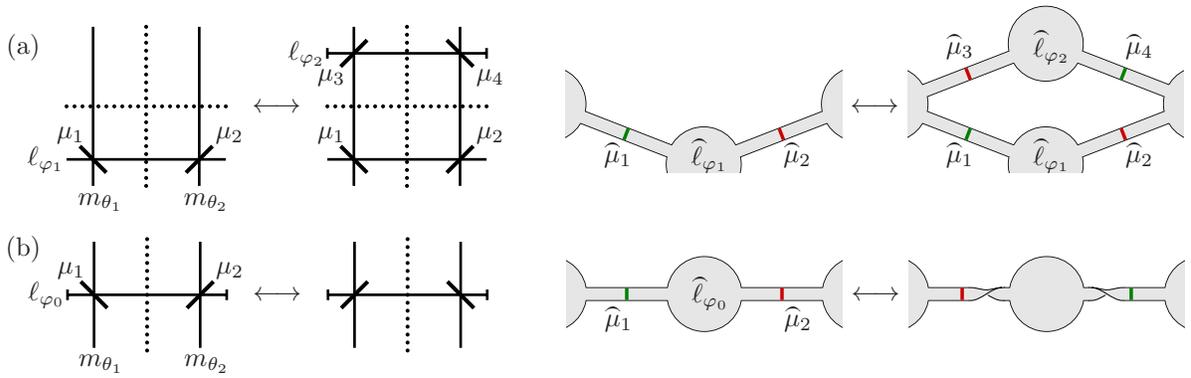

\begin{defi}\label{mir-jump-def}
Let $M$ be a mirror diagrams such that $\ell_{\varphi_1}$ is an occupied level of $M$ and $\ell_{\varphi_2}$
is not an occupied level of $M$. Assume that for some $\theta_1\ne\theta_2$ all mirrors at $\ell_{\varphi_1}$
occur in the interval $(\theta_1;\theta_2)\times\{\varphi_1\}$, and there are no mirrors of $M$ in
$(\theta_1;\theta_2)\times(\varphi_1;\varphi_2)$. Let $M'$ be obtained from $M$ by replacing
the occupied level $\ell_{\varphi_1}$ with $\ell_{\varphi_2}$, and
each mirror of the form $(\theta,\varphi_1)$ with a mirror at~$(\theta,\varphi_2)$ of the same type.
Then we say that $M\mapsto M'$ is \emph{a jump move}; see Figure~\ref{jump-fig}.

The morphism~$\eta$ associated with the jump move~$M\mapsto M'$ is defined by the condition~$(\wideparen M,\wideparen M',h)\in\eta$,
where~$h$ is a homeomorphism~$\wideparen M\rightarrow\wideparen M'$ such that~$h|_{\wideparen y}$
is an orientation-preserving self-homeomorphism of~$\wideparen y$
whenever~$y$ is an occupied level of~$M$ distinct from~$\ell_{\varphi_1}$, or a mirror located outside~$\ell_{\varphi_1}$,
and~$h$ takes~$\wideparen\ell_{\varphi_1}$ to~$\wideparen\ell_{\varphi_2}$ preserving the coorientation defined
by~$d\varphi$.
\end{defi}

\begin{figure}[ht]
\includegraphics{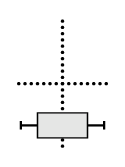}\put(-35,16.5){$X$}\put(-65,18){$\ell_{\varphi_1}$}%
\ \raisebox{38pt}{$\longleftrightarrow$}\ %
\includegraphics{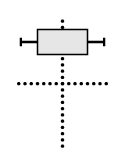}\put(-35,56.5){$X$}\put(-8,58){$\ell_{\varphi_2}$}%
\caption{A jump move}\label{jump-fig}
\end{figure}
\begin{figure}[ht]
\includegraphics{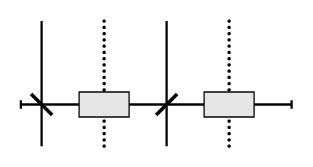}\put(-105,26.5){$X$}\put(-43,26.5){$Y$}%
\put(-136,3){$m_{\theta_1}$}\put(-76,3){$m_{\theta_2}$}%
\put(-155,28){$\ell_{\varphi_0}$}\put(-143,40){$\mu_1$}\put(-67,40){$\mu_2$}%
\raisebox{38pt}{$\longleftrightarrow$}%
\includegraphics{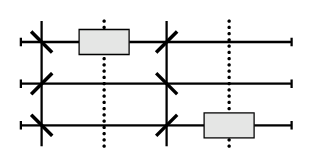}\put(-105,56.5){$X$}\put(-43,16.5){$Y$}%
\put(-136,3){$m_{\theta_1}$}\put(-76,3){$m_{\theta_2}$}%
\put(-8,18){$\ell_{\varphi_1}$}\put(-8,38){$\ell_{\varphi_2}$}\put(-8,58){$\ell_{\varphi_3}$}\put(-124,11){$\mu_1'$}%
\put(-128,30){$\mu_1''$}\put(-144,70){$\mu_1'''$}\put(-87,11){$\mu_2'$}\put(-83,30){$\mu_2''$}\put(-67,70){$\mu_2'''$}
\\[5mm]\includegraphics{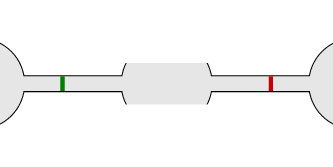}\put(-83,36){$\wideparen\ell_{\varphi_0}$}\put(-130,26){$\wideparen\mu_1$}%
\put(-38,26){$\wideparen\mu_2$}\
\ \raisebox{36pt}{$\longleftrightarrow$}\ \ \includegraphics{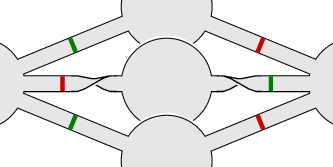}\put(-83,5){$\wideparen\ell_{\varphi_1}$}%
\put(-83,36){$\wideparen\ell_{\varphi_2}$}\put(-83,67){$\wideparen\ell_{\varphi_3}$}%
\put(-130,6){$\wideparen\mu_1'$}\put(-120,28){$\wideparen\mu_1''$}\put(-130,68){$\wideparen\mu_1'''$}
\put(-38,6){$\wideparen\mu_2'$}\put(-50,28){$\wideparen\mu_2''$}\put(-38,68){$\wideparen\mu_2'''$}
\caption{Wrinkle creation and reduction moves}\label{mir-wrinkle}
\end{figure}

\begin{defi}\label{wrinkle-def}
Let a mirror diagram $M$ have a $\diagdown$-mirror~$\mu_1$ at $(\theta_1,\varphi_0)$ and a $\diagup$-mirror~$\mu_2$ at
$(\theta_2,\varphi_0)$. Let~$\varphi_0,\varphi_1,\varphi_2,\varphi_3{}\in\mathbb S^1$
be such that~$\varphi_0\in[\varphi_1;\varphi_3]$, $\varphi_2\in(\varphi_1;\varphi_3)$, and there are no horizontal occupied levels $\ell_\varphi$ of $M$
with $\varphi\in[\varphi_1;\varphi_3]$, $\varphi\ne\varphi_0$. Let
a mirror diagram $M'$ be obtained from $M$
by:
\begin{enumerate}
\item
replacing the occupied level~$\ell_{\varphi_0}$ by three occupied levels~$\ell_{\varphi_i}$, $i=1,2,3$;
\item
replacing each mirror at $(\theta,\varphi_0)$,
with $\theta\in(\theta_1;\theta_2)$, by a mirror at $(\theta,\varphi_3)$ of the same type;
\item
replacing each mirror at $(\theta,\varphi_0)$,
with $\theta\in(\theta_2;\theta_1)$, by a mirror at $(\theta,\varphi_1)$ of the same type;
\item
replacing the mirrors $\mu_1$ and~$\mu_2$
by~$\diagdown$-mirrors~$\mu_1'=(\theta_1,\varphi_1)$, $\mu_2''=(\theta_2,\varphi_2)$, $\mu_1'''=(\theta_1,\varphi_3)$
and~$\diagup$-mirrors~$\mu_2'=(\theta_2,\varphi_1)$, $\mu_1''=(\theta_1,\varphi_2)$, $\mu_2'''=(\theta_2,\varphi_3)$.
\end{enumerate}
Then we say that $M\mapsto M'$ is \emph{a wrinkle creation move}, and the inverse passage
is \emph{a wrinkle reduction move}; see Figure~\ref{mir-wrinkle}.
The mirrors~$\mu_1$ and~$\mu_2$ of~$M$ are referred
to as \emph{the ramification mirrors} of the move~$M\mapsto M'$.

The morphism~$\eta$ associated with the wrinkle move~$M\mapsto M'$ is defined as follows.
Let~$F$ be a surface obtained from~$\wideparen M'$ by patching the two new holes,
that is, the holes corresponding to the two boundary circuits of~$M'$ that hit the mirrors
at~$\ell_{\varphi_2}$.

We demand~$(\wideparen M,F,h)\in\eta$,
where~$h$ is a homeomorphism~$\wideparen M\rightarrow F$ such that~$h|_{\wideparen y}$
is an orientation-preserving self-homeomorphism of~$\wideparen y$
whenever~$y$ is an occupied level of~$M$ different from~$\ell_{\varphi_0}$, or a mirror located outside~$\ell_{\varphi_0}$.
\end{defi}

\begin{defi}\label{double-split-def}
Let~$M$ and~$M'$ be as in Definition~\ref{wrinkle-def}, and let~$M''$ be obtained from~$M'$
by removing the occupied level~$\ell_{\varphi_2}$ together with the mirrors~$\mu_1''$, $\mu_2''$.
Then the transition~$M\mapsto M''$ is called \emph{a double split move}, and
the inverse one \emph{a double merge move}; see Figure~\ref{double-split-fig}. One can see that~$M'\mapsto M''$
is a bridge removal.
The mirrors~$\mu_1$ and~$\mu_2$ of~$M$
are called the \emph{splitting mirrors} of the move~$M\mapsto M''$,
and the intervals~$\{\theta_1\}\times(\varphi_1;\varphi_3)$, $\{\theta_2\}\times(\varphi_1;\varphi_3)$
\emph{the $\diagdown$-splitting gap} and \emph{the $\diagup$-splitting gap}, respectively.

\begin{figure}[ht]
\includegraphics{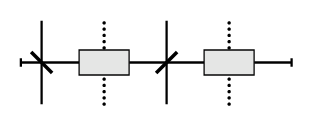}\put(-105,26.5){$X$}\put(-43,26.5){$Y$}%
\put(-136,3){$m_{\theta_1}$}\put(-76,3){$m_{\theta_2}$}%
\put(-155,28){$\ell_{\varphi_0}$}\put(-143,40){$\mu_1$}\put(-67,40){$\mu_2$}%
\raisebox{28pt}{$\longleftrightarrow$}%
\includegraphics{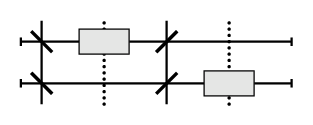}\put(-105,36.5){$X$}\put(-43,16.5){$Y$}%
\put(-136,3){$m_{\theta_1}$}\put(-76,3){$m_{\theta_2}$}%
\put(-8,18){$\ell_{\varphi_1}$}\put(-8,38){$\ell_{\varphi_3}$}
\\[5mm]\includegraphics{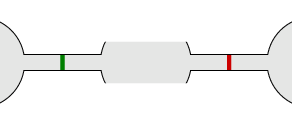}\put(-75,26){$\wideparen\ell_{\varphi_0}$}\put(-122,16){$\wideparen\mu_1$}%
\put(-30,16){$\wideparen\mu_2$}\ \ \raisebox{27pt}{$\longleftrightarrow$}\ \ \includegraphics{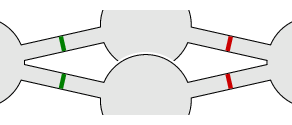}%
\put(-75,12){$\wideparen\ell_{\varphi_1}$}\put(-75,42){$\wideparen\ell_{\varphi_3}$}
\caption{Double split and double merge moves}\label{double-split-fig}
\end{figure}
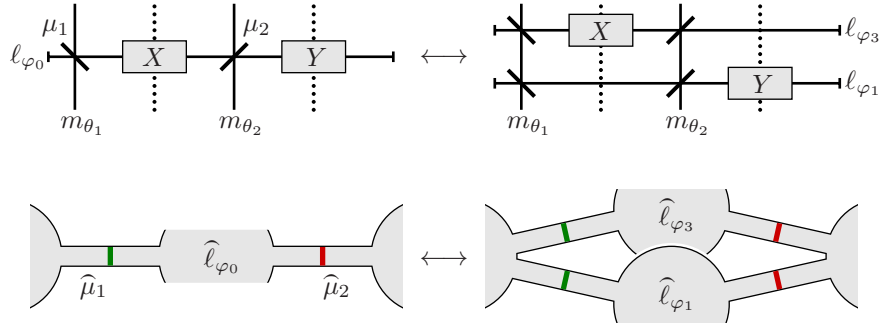
The morphism associated with the move~$M\mapsto M''$
is defined as the composition of the morphisms associated with the wrinkle creation move~$M\mapsto M'$
and the bridge removal~$M'\mapsto M''$.
\end{defi}

Recall that the symmetries~$r_\diagdown$, $r_\diagup$, $r_-$, and~$r_|$ should be
applied to all the definitions of our moves, see Convention~\ref{symmetries}.
Recall also that these definitions are supposed to be extended
to the case of enhanced mirror diagrams, see Convention~\ref{moves-extension-conv}.

\begin{lemm}\label{neutral-move-decomposition-lem}
Every bridge, twist, jump, wrinkle creation, wrinkle reduction, double split, or double merge move of enhanced mirror diagrams admits
a neat decomposition into type $T$ elementary moves with
$T$ being any of~I and~II.
\end{lemm}

\begin{proof}
Due to the symmetry of the definitions of the moves it suffices to produce a neat decomposition into type~I moves.
It is elementary to see in each case below that the decomposition is neat (see the discussion after Definition~\ref{neat-def}).

\medskip\noindent\emph{Bridge moves}.
Bridge addition consists in adding a new occupied level and two mirrors at it. This can be done in two steps:
(1) add a new occupied level and the one of the two mirrors that has type $\diagup$; (2) add the second mirror, which
is of type~$\diagdown$.
The first operation is a type~I extension move, and the second a type~I elementary bypass addition.

\medskip\noindent\emph{Twist moves}. A twist move can be neatly decomposed into several bridge moves.
Indeed, let $\ell_\varphi$ be the occupied level at which one can perform a twist move. Let $\varepsilon>0$
be so small that there are no other occupied levels $\ell_{\varphi'}$ of the diagram with $\varphi'\in[\varphi-\varepsilon;\varphi+\varepsilon]$.
Add a bridge at $\ell_{\varphi+\varepsilon}$ and remove the bridge at~$\ell_\varphi$. The result
is almost what we want but the new pair of mirrors occurs at a level different from~$\ell_\varphi$.
To fix this add a bridge at $\ell_{\varphi-\varepsilon}$, remove the one at $\ell_{\varphi+\varepsilon}$,
add one at $\ell_\varphi$, and finally, remove the one at~$\ell_{\varphi-\varepsilon}$.

Since the statement has already been settled for bridge moves it also holds for twist moves.

\medskip\noindent\emph{Jump moves}.
We use the notation from Definition~\ref{mir-jump-def}, and
proceed by induction in $n$,
where $n$ is the number of mirrors at the level $\ell_{\varphi_1}$ excluding the `rightmost' one
if it is a $\diagup$-mirror.

If $n=0$, then either there are no mirrors at $\ell_{\varphi_1}$,
or there is a single $\diagup$-mirror.
In the latter case, the jump move is a composition of a type~I extension move (which adds the new
mirror and the occupied level~$\ell_{\varphi_2}$ to the diagram) and an elimination move (which removes
the old mirror and the occupied level~$\ell_{\varphi_1}$).

If there are no mirrors at~$\ell_{\varphi_1}$ choose a free meridian~$m_{\theta_0}$,
add~$m_{\theta_0}$ and a $\diagup$-mirror at~$\ell_{\varphi_1}\cap m_{\theta_0}$
by an extension move, then proceed as before to replace~$\ell_{\varphi_1}$ by~$\ell_{\varphi_2}$,
and finally, eliminate the mirror at~$\ell_{\varphi_2}\cap m_{\theta_0}$ and the occupied level~$m_{\theta_0}$.

For the induction step there are the following three cases to consider.

\medskip\noindent\emph{Case 1}:
the two `rightmost' mirrors at $\ell_{\varphi_1}$ are of the $\diagup$-type. We proceed as shown in Figure~\ref{jump-case-1}.
\begin{figure}[ht]
\begin{tabular}{ccccc}
\includegraphics{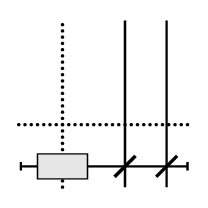}\put(-75,17){$X$}&\raisebox{47pt}{$\longrightarrow$}&
\includegraphics{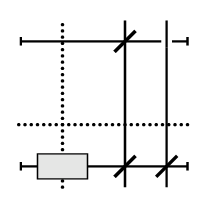}\put(-75,17){$X$}&\raisebox{47pt}{$\longrightarrow$}&
\includegraphics{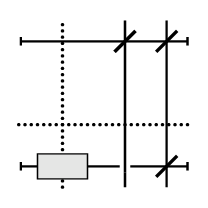}\put(-75,17){$X$}\\&&&&$\big\downarrow$\\
\includegraphics{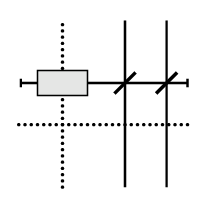}\put(-75,57){$X$}&\raisebox{47pt}{$\longleftarrow$}&
\includegraphics{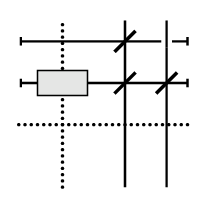}\put(-75,57){$X$}&\raisebox{47pt}{$\longleftarrow$}&
\includegraphics{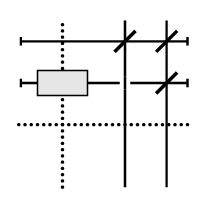}\put(-75,57){$X$}
\end{tabular}
\caption{Induction step for the decomposition of a jump move, Case~1}\label{jump-case-1}
\end{figure}

\medskip\noindent\emph{Case 2}:
the `rightmost' mirror at $\ell_{\varphi_1}$ is of the $\diagup$-type, and the preceding one of the $\diagdown$-type.
We proceed as shown in Figure~\ref{jump-case-2}.
\begin{figure}[ht]
\begin{tabular}{ccccc}
\includegraphics{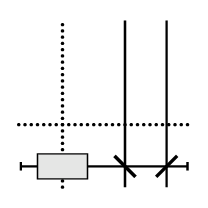}\put(-75,17){$X$}&\raisebox{47pt}{$\longrightarrow$}&
\includegraphics{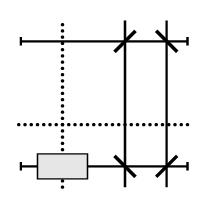}\put(-75,17){$X$}&\raisebox{47pt}{$\longrightarrow$}&
\includegraphics{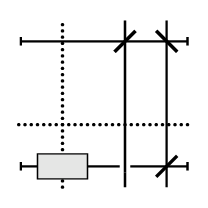}\put(-75,17){$X$}\\&&&&$\big\downarrow$\\
\includegraphics{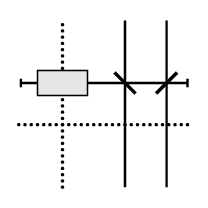}\put(-75,57){$X$}&\raisebox{47pt}{$\longleftarrow$}&
\includegraphics{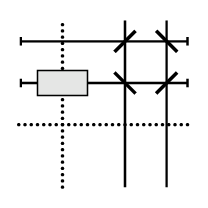}\put(-75,57){$X$}&\raisebox{47pt}{$\longleftarrow$}&
\includegraphics{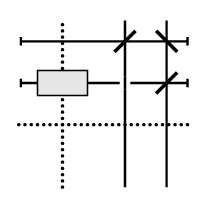}\put(-75,57){$X$}
\end{tabular}
\caption{Induction step for the decomposition of a jump move, Case~2}\label{jump-case-2}
\end{figure}

\medskip\noindent\emph{Case 3}:
the `rightmost' mirror at $\ell_{\varphi_1}$ is of the $\diagdown$-type. Then we
add, by a type~I extension move, a $\diagup$-mirror on the right of it at~$\ell_{\varphi_1}$ together with a new vertical occupied
level, then proceed as in Case~2,
and finally apply an elimination move to erase the `auxiliary' occupied level and the mirror on it.

All this is done by elementary type~I moves and bridge moves, which have already been shown to be neatly decomposable
into type~I elementary moves.

\medskip\noindent\emph{Wrinkle moves}.
We use the notation from Definition~\ref{wrinkle-def}. We assume that the mirrors in the
gray box~$Y$ stay fixed, that is,~$\varphi_0=\varphi_1$. The case when the mirrors in~$X$ stay fixed is
symmetric to this one. If none of the cases occur we first apply a jump move (which
is neatly decomposed into type~I elementary moves as shown above)
to take the block~$Y$ to the desired position and then proceed as described below.

The proof is by induction in the number $n$ of mirrors
of $M$ located in the block~$X$, that is, in $(\theta_1;\theta_2)\times\{\varphi_1\}$. If $n=0$ the wrinkle creation move
decomposes neatly into two bridge moves, which have already been shown to be neatly decomposable
into type~I elementary moves.

The induction step is similar to that in the case of a jump move. It is sketched in Figure~\ref{mir-wrinkle-induction},
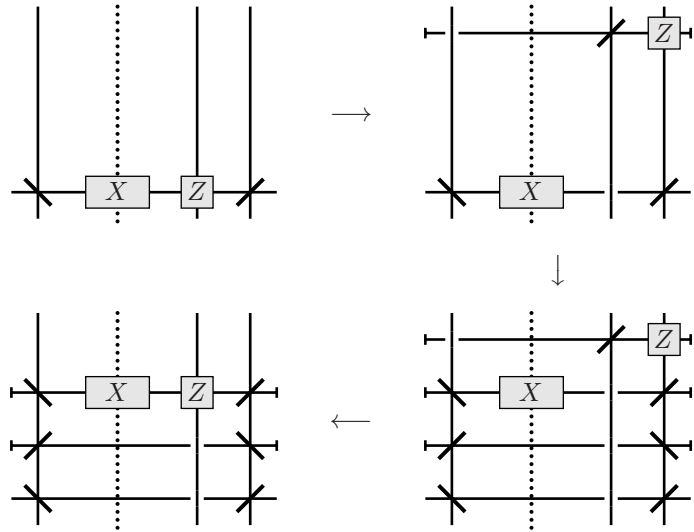
\begin{figure}[ht]
\begin{tabular}{ccc}
\includegraphics{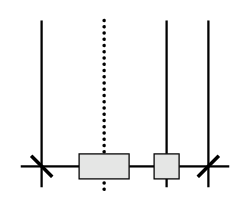}\put(-75,17){$X$}\put(-44,17){$Z$}
&\raisebox{47pt}{$\longrightarrow$}&
\includegraphics{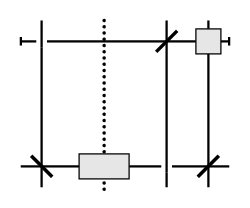}\put(-75,17){$X$}\put(-24,77){$Z$}\\
&&$\big\downarrow$\\
\includegraphics{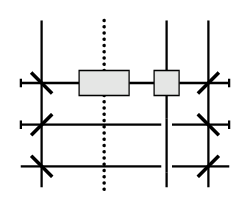}\put(-75,57){$X$}\put(-44,57){$Z$}
&\raisebox{47pt}{$\longleftarrow$}&
\includegraphics{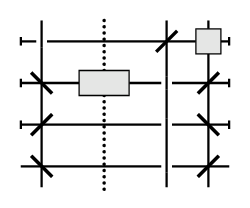}\put(-75,57){$X$}\put(-24,77){$Z$}
\end{tabular}
\caption{Induction step for the decomposition of a wrinkle move}\label{mir-wrinkle-induction}
\end{figure}
where~$Z$ stands for a mirror of either type.

\medskip\noindent\emph{Double split and double merge moves}. These are, by definition, neatly decomposed into a wrinkle move and a bridge
move, which have been discussed above.
\end{proof}

\subsection{Split moves and merge moves}
The moves defined below are similar in nature to double split moves, but unlike the latter they are not neutral.

\begin{defi}\label{split-def}
Let $\theta_1,\theta_2,\varphi_0,\varphi_1,\varphi_2{}\in\mathbb S^1$ be such that
$\theta_1\ne\theta_2$, $\varphi_1\ne\varphi_2$, $\varphi_0\in[\varphi_1;\varphi_2]$,
and let $M$ be a mirror diagram with a $\diagup$-mirror~$\mu$ at $(\theta_2,\varphi_0)$, such
that $M$ has no occupied level $\ell_\varphi$ with $\varphi\in[\varphi_1;\varphi_2]$, $\varphi\ne\varphi_0$,
and no mirror at~$(\theta_1,\varphi_0)$.
Let $M'$ be a mirror diagram obtained from $M$ by the following alterations:
\begin{enumerate}
\item
replacing the occupied level~$\ell_{\varphi_0}$ by two occupied levels~$\ell_{\varphi_1}$ and~$\ell_{\varphi_2}$;
\item
replacing~$\mu$ by two~$\diagup$-mirrors~$\mu'$ and $\mu''$
at~$(\theta_2,\varphi_1)$ and~$(\theta_2,\varphi_2)$, respectively;
\item
replacing each mirror at $(\theta,\varphi_0)$ with $\theta\in(\theta_1;\theta_2)$ by a mirror
at $(\theta,\varphi_2)$ of the same type;
\item
replacing each mirror at $(\theta,\varphi_0)$ with $\theta\in(\theta_2;\theta_1)$ by a mirror
at $(\theta,\varphi_1)$ of the same type.
\end{enumerate}
Then we say that the passage $M\mapsto M'$ is \emph{a type~I split move}, and the inverse one is \emph{a type~I merge move}; see Figure~\ref{split-merge-type-I}.
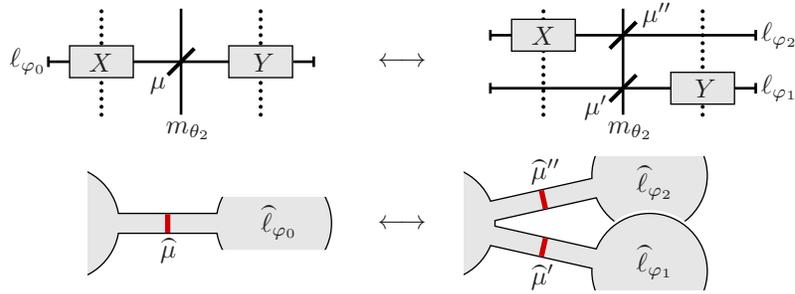
\begin{figure}[ht]
\includegraphics{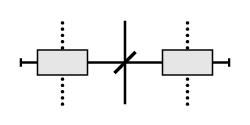}\put(-95,26.5){$X$}\put(-33,26.5){$Y$}\put(-125,28){$\ell_{\varphi_0}$}\put(-66,3){$m_{\theta_2}$}%
\put(-72,20){$\mu$}
\hskip.5cm\raisebox{28pt}{$\longleftrightarrow$}\hskip.5cm%
\includegraphics{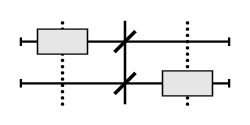}\put(-95,36.5){$X$}\put(-33,16.5){$Y$}\put(-75,10){$\mu'$}\put(-53,45){$\mu''$}%
\put(-8,18){$\ell_{\varphi_1}$}\put(-8,38){$\ell_{\varphi_2}$}\put(-66,3){$m_{\theta_2}$}\\
\includegraphics{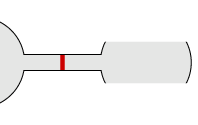}\put(-30,27){$\wideparen\ell_{\varphi_0}$}\put(-68,17){$\wideparen\mu$}
\hskip.5cm\raisebox{28pt}{$\longleftrightarrow$}\hskip.5cm%
\includegraphics{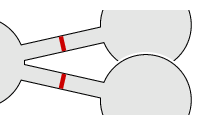}\put(-30,12){$\wideparen\ell_{\varphi_1}$}\put(-30,42){$\wideparen\ell_{\varphi_2}$}%
\put(-70,7){$\wideparen\mu'$}\put(-70,47){$\wideparen\mu''$}
\caption{Type I split/merge moves}\label{split-merge-type-I}
\end{figure}
The level $\ell_{\varphi_0}$ is said to be \emph{split} in the transition from $M$ to $M'$,
and the levels $\ell_{\varphi_1}$ and~$\ell_{\varphi_2}$ are said to be \emph{merged} in the inverse
transition.

The mirror~$\mu$ is referred to as \emph{the splitting mirror},
the point~$(\theta_1,\varphi_0)$ as \emph{the snip point},
and the interval~$\{\theta_2\}\times(\varphi_1;\varphi_2)$ as \emph{the splitting gap} of the split move~$M\mapsto M'$.

The morphism~$\eta$ associated with the merge move~$M'\mapsto M$ is defined by~$\bigl(\wideparen M',\wideparen M,h\bigr)\in\eta$,
where~$h$ is an embedding~$\wideparen M'\rightarrow\wideparen M$ such that~$h|_{\wideparen y}$ is an orientation-preserving
self-homeomorphism of~$\wideparen y$ whenever~$y$ is an occupied level of~$M'$ distinct from~$\ell_{\varphi_1}$, $\ell_{\varphi_2}$,
or a mirror located outside of these two levels, and~$h(\wideparen\ell_{\varphi_1}\cup\wideparen\ell_{\varphi_2})\subset\wideparen\ell_{\varphi_0}$.
\end{defi}

\begin{lemm}\label{split-move-decomposition-lem}
A type~I (respectively, type~II) split move of mirror diagrams admits a neat decomposition
into type~I (respectively, type~II) elementary moves.
\end{lemm}

\begin{proof}
The proof is by induction, which is similar to that in the cases of a jump move and of a wrinkle move
in Lemma~\ref{neutral-move-decomposition-lem}. We leave it to the reader.
\end{proof}

\begin{lemm}\label{slide-bypass-neat-decomp-lem}
Any type~II slide move and type II elementary bypass addition or removal admits a neat decomposition into
a sequence of moves from the following list:
\begin{itemize}
\item
type~II split moves,
\item
type~II merge moves,
\item
neutral moves.
\end{itemize}
\end{lemm}
\begin{proof}
A decomposition of a slide move is shown in Figure~\ref{slide-decomposition-fig}. We need one split move, then a jump move, and finally a merge move.
\begin{figure}[ht]
\includegraphics{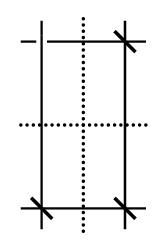}\raisebox{58pt}{$\longrightarrow$}
\includegraphics{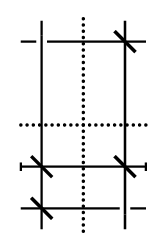}\raisebox{58pt}{$\longrightarrow$}
\includegraphics{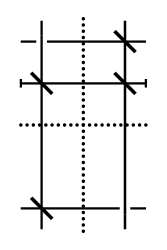}\raisebox{58pt}{$\longrightarrow$}
\includegraphics{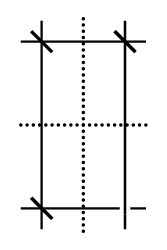}
\caption{A neat decomposition of a type~II slide move into type~II split/merge moves and a jump move}\label{slide-decomposition-fig}
\end{figure}

A decomposition of a type~II elementary bypass addition is shown in Figure~\ref{bypass-decomposition-fig}.
We need a bridge addition
and then a type~II merge move. For a type~II elementary bypass removal the sequence is inverse, and we have a type~II split move
followed by a bridge removal.
\begin{figure}[ht]
\includegraphics{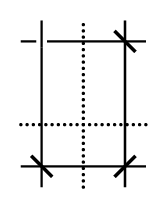}\raisebox{48pt}{$\longrightarrow$}
\includegraphics{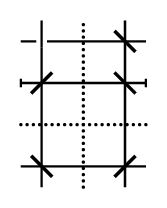}\raisebox{48pt}{$\longrightarrow$}
\includegraphics{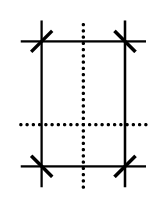}
\caption{A neat decomposition of a type~II elementary bypass addition into
a bridge move and a type~II merge move}\label{bypass-decomposition-fig}
\end{figure}
\end{proof}

\begin{lemm}\label{extension-neat-decomp-lem}
Let~$M\mapsto M'$ be a type~II extension move of enhanced mirror diagrams, and let~$c$
be the boundary circuit of~$M$ modified by this move.
Assume that~$\tb_+(c)<0$. Then the move~$M\mapsto M'$ admits
a neat decomposition into type~II split moves and neutral moves.
\end{lemm}

\begin{proof}
The proof is by induction
in the smallest number of hops it takes to go along~$c$ from the position
where a new $\diagdown$-mirror, which we denote by~$\mu$, is inserted to an existing~$\diagdown$-mirror,
which we denote by~$\mu_0$.

If~$\mu_0$ is `immediately visible' from~$\mu$, that is, if~$\mu$ is inserted in a straight line segment contained in~$c$
one of whose endpoints is~$\mu_0$. The extension move can be decomposed neatly into a type~II split move and
a jump move; see Figure~\ref{extension-induction-base}.
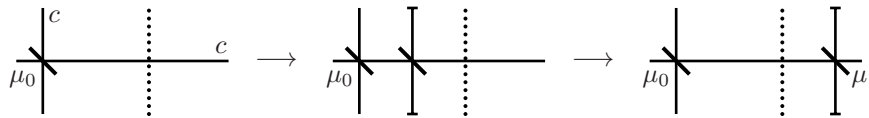
\begin{figure}[ht]
\includegraphics{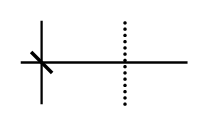}\put(-78,45){$c$}\put(-15,33){$c$}\put(-93,22){$\mu_0$}
\raisebox{28pt}{$\longrightarrow$}
\includegraphics{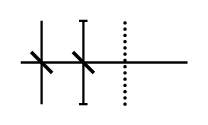}\put(-93,22){$\mu_0$}\raisebox{28pt}{$\longrightarrow$}
\includegraphics{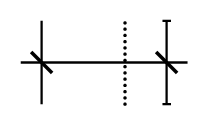}\put(-93,22){$\mu_0$}\put(-14,22){$\mu$}
\caption{Decomposition of an extension move into a split move and a jump move}\label{extension-induction-base}
\end{figure}
This gives the induction base.

If~$\mu$ is in~$k$ hops from~$\mu_0$, $k>1$, insert a~$\diagdown$-mirror~$\mu'$ in a straight-line portion of~$c$ 
adjacent to the one containing~$\mu$, so that the new mirror is
in~$k-1$ hops from~$\mu_0$, and then apply two bridge moves as shown in Figure~\ref{extension-induction-step}
to `transfer' it to the desired position.
\begin{figure}[ht]
\includegraphics{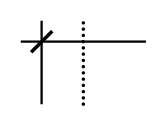}\put(-68,11){$c$}\put(-15,43){$c$}
\raisebox{28pt}{$\longrightarrow$}
\includegraphics{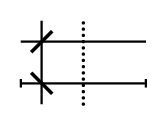}\put(-70,10){$\mu'$}\raisebox{28pt}{$\longrightarrow$}
\includegraphics{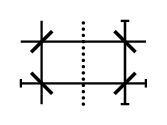}\put(-14,32){$\mu$}\put(-70,10){$\mu'$}\raisebox{28pt}{$\longrightarrow$}
\includegraphics{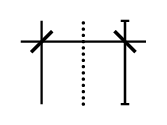}\put(-14,32){$\mu$}
\caption{Adding a $\diagdown$-mirror and `moving' it along the boundary circuit~$c$ by bridge moves}\label{extension-induction-step}
\end{figure}
This gives the induction step.
\end{proof}

\begin{lemm}\label{split-into-bypas-decomp-lem}
Any type~II split move of enhanced mirror diagrams admits a neat decomposition
into a sequence of moves one of which is a type~II elementary bypass removal,
and the others are type~I elementary moves.
\end{lemm}

\begin{proof}
 Figure~\ref{split-decomp}
demonstrates a decomposition into a sequence of the following moves:
\begin{figure}[ht]
\begin{tabular}{ccccc}
\includegraphics{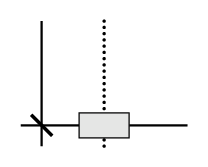}\put(-55,16.5){$X$}&\raisebox{38pt}{$\longrightarrow$}&
\includegraphics{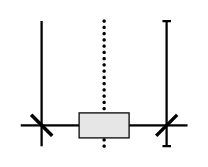}\put(-55,16.5){$X$}&\raisebox{38pt}{$\longrightarrow$}&
\includegraphics{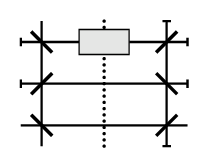}\put(-55,56.5){$X$}\\
&&&&$\big\downarrow$\\
\includegraphics{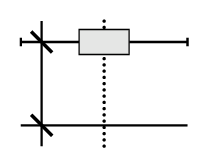}\put(-55,56.5){$X$}&\raisebox{38pt}{$\longleftarrow$}&
\includegraphics{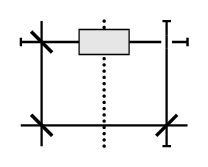}\put(-55,56.5){$X$}&\raisebox{38pt}{$\longleftarrow$}&
\includegraphics{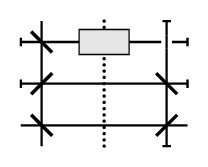}\put(-55,56.5){$X$}
\end{tabular}
\caption{A neat decomposition of a type~II split move into a type~II elementary bypass removal, type~I moves, and neutral moves}\label{split-decomp}
\end{figure}
a type~I extension move, a wrinkle creation, a type~II bypass removal,
a bridge removal, a type~I elimination move. The wrinkle and bridge
moves should be then decomposed neatly into type~I moves, which is possible by
Lemma~\ref{neutral-move-decomposition-lem}.
\end{proof}

\subsection{Handle decomposition of~$\wideparen M$ and partial homeomorphisms~$h_M^{M'}$}\label{hMM-subsec}
Surfaces of the form~$\wideparen M$, where~$M$ is a mirror diagram,
naturally come with a handle decomposition. The $0$-handles of the decomposition
are discs~$\wideparen x$, $x\in L_M$, and the $1$-handles
are strips~$\wideparen\mu$, $\mu\in E_M$. Additionally, every $1$-handle
bears a \emph{type}, which is the type of the respective mirror of~$M$ and is either `$\diagdown$' or `$\diagup$'.

For certain previously defined moves~$M\mapsto M'$ of mirror diagrams, we introduce here
a partial homeomorphism~$h_M^{M'}$ from~$\wideparen M$ to~$\wideparen M'$ preserving the handle decomposition.
By this we mean a homeomorphism from a compact subsurface~$F\subset\wideparen M$
to a compact subsurface~$F'\subset\wideparen M'$
such that~$F$ and~$F'$ inherit
a handle decomposition structure from~$M$ and~$M'$, respectively,
and~$h_M^{M'}$ takes $0$-handles of~$F$ to $0$-handles of~$F'$ and $1$-handles of~$F$
to $1$-handles of~$F'$. The partial homeomorphism~$h_M^{M'}$ will also respect the type of
each $1$-handle and the coorientations of $0$-handles defined by the normal
vector fields to~$\xi_+$ and~$\xi_-$.

The partial homeomorphism~$h_M^{M'}$ will be also extendable to the whole of~$\mathbb S^3$
and, after extension, take some surface carried by~$\widehat M$ to a surface carried by~$\widehat M'$
so as to induce the morphism~$\widehat M\rightarrow\widehat M'$ associated with the move~$M\mapsto M'$.
In each case, the surfaces~$F$ and~$F'$ are obtained from~$\wideparen M$ and~$\wideparen M'$ by `as little modifications as possible' for
allowing~$h_M^{M'}$ to exist. The modifications will include the trivial one (no modification at all), removing
some handles, and cutting along a normal arc (see below).

If~$h_M^{M'}$ is already defined we set~$h_{M'}^M=\bigl(h_M^{M'}\bigr)^{-1}$ for the inverse move.

The moves for which we introduce the partial homeomorphism~$h_M^{M'}$ include the following ones:
extension/elimination, elementary bypass addition/removal, jump, wrinkle creation/reduction, and split/merge moves.

Note that~$h_M^{M'}$ will depend not only on the diagrams~$M$ and~$M'$ but also
on how the move~$M\mapsto M'$ is interpreted. For instance, sometimes a transformation
of mirror diagrams can be viewed as an extension move and as a split move,
and the respective partial homeomorphisms~$h_M^{M'}$ will be quite different.

Let~$M$ be a(n enhanced) mirror diagram, and let~$\mu$ be a mirror of~$M$.
By \emph{the lateral boundary} of the strip~$\wideparen\mu$ we call the
closure of~$\partial\wideparen\mu\setminus\bigcup_{x\in L_M}\wideparen x$.

\begin{defi}
Let~$M$ be a(n enhanced) mirror diagram. A simple arc~$\beta$ in the surface~$\wideparen M$ is \emph{normal}
if the following holds:
\begin{enumerate}
\item
$\beta$ is transverse to~$\partial\wideparen x$ for every occupied level~$x$ of~$M$;
\item
the intersection of~$\beta$ with every strip~$\wideparen\mu$, $\mu\in E_M$,
consists of arcs each of which has endpoints at the boundaries of two different discs of the form~$\wideparen x$, $x\in L_M$,
and is disjoint from the lateral boundary of~$\wideparen\mu$;
\item
$\partial\beta\subset\bigcup_{x\in L_M}\partial\wideparen x$.
\end{enumerate}
If~$\beta\subset\wideparen M$ is a normal arc, then by \emph{cutting~$\wideparen M$ along~$\beta$}
we mean removing from~$\wideparen M$ a small open neighborhood~$U$ of~$\beta$ such
that for any~$y\in L_M$ or~$y\in E_M$ the subset~$\wideparen y\setminus U$ is a deformation retract of~$\wideparen y\setminus\beta$.
\end{defi}

It follows from this definition that the surface obtained by cutting~$\wideparen M$ along a normal arc
inherits the handle decomposition structure.

\begin{defi}\label{successor-def}
Let~$M$ and~$M'$ be enhanced mirror diagrams, and let~$M\mapsto M'$ be one of the moves for which
we define below the partial homeomorphism~$h_M^{M'}$. Let also~$y$ be an occupied level
or a mirror of~$M$, and~$y'$ an occupied level or a mirror, respectively, of~$M'$.
We say that~$y$ is \emph{a predecessor} of~$y'$, and~$y'$ is \emph{a successor} of~$y$
for the move~$M\mapsto M'$ if~$h_M^{M'}$ takes some portion of~$\wideparen y$
to some portion of~$\wideparen y'$. We use the notation~$y\mapsto y'$ to express this fact.
\end{defi}

Now we consider the moves~$M\mapsto M'$ for which we define the partial homeomorphism~$h_M^{M'}$, case by case.

\smallskip\noindent\emph{Case~1}: $M\mapsto M'$ is an extension move.\\
The domain~$F$ and the image~$F'$ of~$h_M^{M'}$ is the whole of~$\wideparen M$, and~$h_M^{M'}$ is just the identity map.
The new occupied level and the new mirror of~$M'$ have no predecessors.
For the inverse move these occupied level and mirror of~$M'$ have no successors.

All the other occupied levels and mirrors
of~$M$ and~$M'$ are the same, so each of them is a predecessor and a successor of itself for
both moves~$M\mapsto M'$ and~$M'\mapsto M$.

\smallskip\noindent\emph{Case 2}: $M\mapsto M'$ is an elementary bypass addition.\\
Similarly to the previous case we have~$F=F'=\wideparen M$, $h_M^{M'}=\mathrm{id}|_F$.
The mirror of~$M'$ that is not present in~$M$ has no predecessor for the move~$M\mapsto M'$, and no
successor for the inverse move. All the mirrors and occupied levels of~$M$ are successors and predecessors of themselves
for both moves.

\smallskip\noindent\emph{Case 3}: $M\mapsto M'$ is a jump move.\\
The partial homeomorphism~$h_M^{M'}$ is a homeomorphism~$\wideparen M\rightarrow\wideparen M'$
that takes each handle~$\wideparen y$, where~$y\in E_M$
or~$y\in L_M$, to itself if~$y$ is preserved by the move, and to~$\wideparen y'$ if Definition~\ref{mir-jump-def}
says that~$y$ is replaced by~$y'$ in~$M'$. In the latter case, if~$y$ and~$y'$ are occupied
levels, the homeomorphism~$h_M^{M'}$ restricted to~$\wideparen y$ is demanded to preserve the orientation
that comes from the coorientation of~$\mathbb S^1_{\tau=0}\cup\mathbb S^1_{\tau=1}$.

\smallskip\noindent\emph{Case 4}: $M\mapsto M'$ is a wrinkle creation.\\
We use the notation from Definition~\ref{wrinkle-def}.
From the combinatorial point of view, we have the following predecessor/successor pairs:
\begin{itemize}
\item
$\mu_1\mapsto\mu_1'$, $\mu_1\mapsto\mu_1'''$,
$\mu_2\mapsto\mu_2'$, $\mu_2\mapsto\mu_2'''$,
\item
$\ell_{\varphi_0}\mapsto\ell_{\varphi_1}$, $\ell_{\varphi_0}\mapsto\ell_{\varphi_3}$,
\item
$(\theta,\varphi_0)\mapsto(\theta,\varphi_3)$ if~$\theta\in(\theta_1;\theta_2)$ and~$(\theta,\varphi_0)\in E_M$,
\item
$(\theta,\varphi_0)\mapsto(\theta,\varphi_1)$ if~$\theta\in(\theta_2;\theta_1)$ and~$(\theta,\varphi_0)\in E_M$,
\item
$y\mapsto y$ if~$y\in E_M\setminus\ell_{\varphi_0}$ or~$y\in L_M\setminus\{\ell_{\varphi_0}\}$.
\end{itemize}

The mirrors~$\mu_1''$, $\mu_2''$ and the occupied
level~$\ell_{\varphi_2}$ of~$M'$ have no predecessors. Let~$M''$ be the mirror diagram obtained by removing
these three elements from~$M'$ (this is a bridge removal). We put~$F'=\wideparen M''$, and this is the image
of~$h_M^{M'}$; see Figure~\ref{del-bridge-wrinkle-move-fig}.
\begin{figure}[ht]
\includegraphics{surf-wrinkle2.eps}\put(-87,-15){$\wideparen M'$}\put(-83,5){$\wideparen\ell_{\varphi_1}$}%
\put(-83,36){$\wideparen\ell_{\varphi_2}$}\put(-83,67){$\wideparen\ell_{\varphi_3}$}%
\put(-130,6){$\wideparen\mu_1'$}\put(-120,28){$\wideparen\mu_1''$}\put(-130,68){$\wideparen\mu_1'''$}
\put(-38,6){$\wideparen\mu_2'$}\put(-50,28){$\wideparen\mu_2''$}\put(-38,68){$\wideparen\mu_2'''$}
\hskip1cm\raisebox{37pt}{$\longrightarrow$}\hskip1cm%
\includegraphics{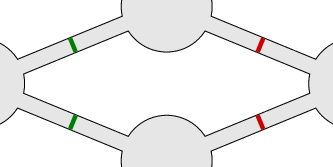}\put(-100,-15){$F'=\wideparen M''$}
\caption{The image of~$h_M^{M'}$ in the case
of a wrinkle creation}\label{del-bridge-wrinkle-move-fig}
\end{figure}

Let~$\beta$ be a normal arc in~$\wideparen M$ that cuts the strips~$\wideparen\mu_1$, $\wideparen\mu_2$
and the disc~$\wideparen\ell_{\varphi_0}$ into halves, and is disjoint from the interiors of all other handles in~$\wideparen M$.
We let the domain of~$h_M^{M'}$ be a surface~$F$ obtained from~$\wideparen M$ by cutting it along~$\beta$; see
Figure~\ref{cut-for-wrinkle-move-fig}.
\begin{figure}[ht]
\includegraphics[scale=0.8]{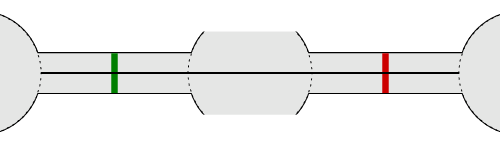}%
\put(-98,19){$\beta$}\put(-102,0){$\wideparen M$}
\hskip.5cm\raisebox{25pt}{$\longrightarrow$}\hskip.5cm%
\includegraphics[scale=0.8]{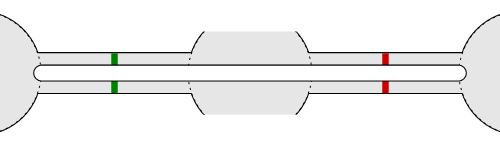}\put(-102,0){$F$}%
\caption{The domain of~$h_M^{M'}$ in the case
of a wrinkle creation}\label{cut-for-wrinkle-move-fig}
\end{figure}

The surfaces~$F$ and~$F'$ have the same handle decomposition structure, and the homeomorphism~$h_M^{M'}:F\rightarrow F'$
is defined naturally.

\smallskip\noindent\emph{Case 5}: $M\mapsto M'$ is a double split move.\\
$F$, $F'$ and~$h_M^{M'}$ are exactly as in the previous case. The difference
is that now~$h_M^{M'}$ is an onto map. The splitting mirrors and the occupied level of~$M$
containing them each have two successors in~$M'$,
all the other elements of~$M$ have a unique successor. All mirrors and occupied levels of~$M'$
have a unique predecessor in~$M$.

\smallskip\noindent\emph{Case 6}: $M\mapsto M'$ is a split move.\\
We use the notation from Definition~\ref{split-def}.
We have the following predecessor/successor pairs for the split move
\begin{itemize}
\item
$M\mapsto M'$: $\mu\mapsto\mu'$, $\mu\mapsto\mu''$,
\item
$\ell_{\varphi_0}\mapsto\ell_{\varphi_1}$, $\ell_{\varphi_0}\mapsto\ell_{\varphi_2}$,
\item
$(\theta,\varphi_0)\mapsto(\theta,\varphi_2)$ if~$\theta\in(\theta_1;\theta_2)$ and~$(\theta,\varphi_0)\in E_M$,
\item
$(\theta,\varphi_0)\mapsto(\theta,\varphi_1)$ if~$\theta\in(\theta_2;\theta_1)$ and~$(\theta,\varphi_0)\in E_M$,
\item
$y\mapsto y$ if~$y\in E_M\setminus\ell_{\varphi_0}$ or~$y\in L_M\setminus\{\ell_{\varphi_0}\}$.
\end{itemize}

The image~$F'$ of~$h_M^{M'}$ is the whole of~$\wideparen M'$. The domain of~$h_M^{M'}$ is
obtained from~$\wideparen M$ by cutting along an arc~$\beta$ that cuts the strip~$\wideparen\mu$
and the disc~$\wideparen\ell_{\varphi_0}$ into halves; see Figure~\ref{cut-for-split-fig}.

\begin{figure}[ht]
\includegraphics{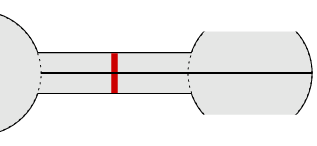}\put(-40,26){$\beta$}\put(-90,0){$\wideparen M$}%
\hskip.5cm\raisebox{32pt}{$\longrightarrow$}\hskip.5cm%
\includegraphics{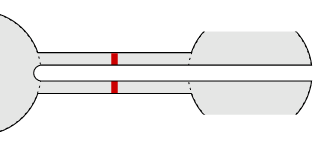}\put(-90,0){$F$}%
\caption{The domain of~$h_M^{M'}$ in the case
of a split move}\label{cut-for-split-fig}
\end{figure}

\section{Mirror diagrams and the standard contact structure}\label{invariance-sec}

In contrast to the previous section, this one is mostly ideological. We start by
discussing the connection between the formalism of mirror diagrams, Legendrian graphs,
and Giroux's convex surfaces.
When the nature of this connection is understood the results of this paper
including the commutation theorems in Section~\ref{commutation-sec}
should appear less surprising. However, we present their proofs
mostly in combinatorial terms, keeping the reference to contact topology, which
is the actual source of the ideas, to a minimum.
In particular, among numerous definitions and statements given in this section,
we \emph{formally} use in the sequel only Definitions~\ref{neglibigle-circuit-def} and~\ref{coherent-mirror-def},
and Lemmas~\ref{remove-coherent-mirror-lem} and~\ref{rem-obst-lem}.
We also omit the `enhanced' versions, which
can be easily guessed, of Definitions~\ref{div-spa-rib-gra-def} and~\ref{negl-hole-def}--\ref{propag-shri-def},
as well as of Proposition~\ref{equivalence-of-divided-graphs-prop},
Theorem~\ref{convex-surf-through-div-leg-graph-th}, and Theorem~\ref{type-i-moves-meaning-th}. Without enhancements, `essential' and `inessential' referring to a boundary circuit or a boundary component mean `non-patchable'
and `patchable', respectively.

We conclude this section by proving Theorem~\ref{invariant-thm} in a combinatorial manner.

\subsection{Legendrian graphs}
Recall from~\cite[Section~3]{dp17} that \emph{a Legendrian graph} is a spatial graph~$\Gamma$ whose
edges are everywhere tangent to the standard contact structure~$\xi_+$
and such that any simple arc in~$\Gamma$ is cusp-free.

\begin{defi}
A spatial ribbon graph~$\rho$ is said to be~\emph{a Legendrian ribbon graph} (with respect to
the standard contact structure~$\xi_+$) if~$\Gamma_\rho$ is a Legendrian
graph, all edges of~$\rho$ are smooth arcs, and there is a surface~$F\in S_\rho$ tangent to~$\xi_+$ at all points of~$\Gamma_\rho$.
\end{defi}

\begin{rema}
The surface~$F$ in this definition need not, and typically cannot, be $C^2$-smooth.
\end{rema}

It is easy to see that a Legendrian ribbon graph~$\rho$ is uniquely determined by the spatial graph~$\Gamma_\rho$
and the set of vertices~$V_\rho$,
and, moreover, $\Gamma_\rho$ may be an arbitrary Legendrian graph.
So, there is no essential difference between the concept of a Legendrian graph and that of a Legendrian
ribbon graph. However, the context of spatial ribbon graphs suggests to introduce a weaker equivalence relation
than the one that looks more natural in the context of Legendrian graphs (see~\cite[Definition~21]{dp17}).

\begin{defi}\label{leg-rib-graph-equiv-def}
Two Legendrian ribbon graphs~$\rho_0$ and~$\rho_1$ are said to be \emph{Legendrian equivalent}
(with respect to~$\xi_+$) if there is an isotopy~$F_t$, $t\in[0,1]$, from a surface~$F_0\in S_{\rho_0}$ to a surface~$F_1\in S_{\rho_1}$
such that, for any~$t\in[0,1]$, the surface~$F_t$ is tangent to~$\xi_+$ along a Legendrian graph~$\Gamma_t$
which is a deformation retract of~$F_t$.
\end{defi}

In terms of the respective Legendrian graphs this equivalence relation coincides with the one introduced in~\cite{BaIs09}
and studied also in~\cite{p14}, where
Legendrian graphs are considered modulo isotopy in the class of Legendrian graphs and edge contraction/blow-up operations.

Generalized rectangular diagrams of~\cite{p14} are, in the present terms, essentially the same thing as rectangular diagrams of a graph.
(In contrast to~\cite{p14}, we allow Legendrian graphs to have isolated vertices. This difference plays no important role.)
If~$G$ is such a diagram then~$\widehat G$ is a Legendrian graph. The associated Legendrian ribbon graph
is then~$\widehat M$, where~$M$ is the mirror diagram obtained from~$G$ by assigning the~`$\diagup$' type to every vertex of~$G$.

With this correspondence, the vertex addition moves in~\cite{p14} translate exactly into type~I extension moves
of mirror diagrams,
the end shifts of type L in~\cite{p14}
translate into moves that are decomposed easily into type~I extension and slide moves,
and the commutations in~\cite{p14} become a particular case of jump moves if the mirror
diagrams are viewed up to combinatorial equivalence.

On the other hand, the type~I slide moves translate back to transformations that are easily decomposed into
elementary moves of type~L from~\cite{p14}. Thus, from Theorems 3.2 and 2.2 of~\cite{p14} and Lemma~\ref{neutral-move-decomposition-lem}
above we have the following.

\begin{theo}\label{equivalence-of-legendrian-graphs-thm}
Any Legendrian ribbon graph is Legendrian equivalent to a Legendrian ribbon graph of the form~$\widehat M$,
where~$M$ is a mirror diagram having only~$\diagup$-mirrors.

Let~$M_1$ and~$M_2$ be mirror diagrams having only~$\diagup$-mirrors. The Legendrian ribbon graphs~$\widehat M_1$
and~$\widehat M_2$ are Legendrian equivalent if and only if~$M_1$ and~$M_2$ can be obtained from each
other by type~I elementary moves (which can only be extension, elimination, and slide moves).
\end{theo}

Note that no analogue of stable equivalence makes sense for Legendrian ribbon graphs. This is because
there cannot exist any `non-trivial' patchable boundary circuits of such a graph. More precisely, a boundary circuit
of a Legendrian ribbon graph~$\rho$ is patchable if and only if it belongs to a contractible
connected component of~$\Gamma_\rho$ (that is, a tree). This can be deduced from the well known result of Bennequin~\cite{ben}
on the non-existence of overtwisted discs (i.e.\ discs tangent to~$\xi_+$ along the whole boundary), and
this implies, in particular, that handle addition is impossible if we restrict ourselves to the class of Legendrian
ribbon graphs.

Note also that if a mirror diagram has only $\diagup$-mirrors, then the only elementary moves applicable to it
are extensions, eliminations, and slides (which keep the diagram free of $\diagdown$-mirrors).

As the following statement shows, the equivalence relation for Legendrian ribbon graphs introduced in Definition~\ref{leg-rib-graph-equiv-def}
is quite natural in the context of Giroux's convex surfaces.
By~$\mathscr G(F)$ we denote the Giroux graph of~$F$ (see~\cite[Definition~27]{dp17}).

\begin{theo}
\emph{(i)} Let~$F$ and~$F'$ be convex surfaces with Legendrian boundaries and very nice characteristic foliations.
Assume that they are isotopic in the class of convex surfaces. Then
the Legendrian ribbon graphs associated with Legendrian graphs~$\mathscr G(F)$ and~$\mathscr G(F')$
are Legendrian equivalent.

\emph{(ii)} Let~$F$ be a convex surface with Legendrian boundary and very nice characteristic foliation,
and let~$\rho$ be a Legendrian ribbon graph
equivalent to the one associated with~$\mathscr G(F)$. Then there exists
a convex surface~$F'$ isotopic to~$F$ through convex surfaces such that
\begin{enumerate}
\item
$\partial F'$ is Legendrian;
\item
the characteristic foliation of~$F'$
is very nice;
\item
$\mathscr G(F')$ coincides with~$\Gamma_\rho$.
\end{enumerate}
\end{theo}

The proof will be published elsewhere. We don't use this result in the sequel.

\subsection{Divided Legendrian graphs}\label{divided-legendrian-subsec}
Here we discuss what general mirror diagrams are if viewed up to type~I elementary moves.
To start, let us forget for a while about contact structures.

\begin{defi}\label{div-spa-rib-gra-def}
By \emph{a divided spatial ribbon graph} we mean a spatial ribbon graph~$\rho$ in which
each surface~$F\in S_\rho$ is endowed with an isotopy class~$\Delta(F)$ of abstract dividing sets
so that the following holds:
\begin{enumerate}
\item
for any~$F\in S_\rho$ there exists a representative~$\delta\in\Delta(F)$ consisting
of arcs each of which intersects~$\Gamma_\rho$ exactly once and transversely (in the topological sense), at a point which is not a true vertex of~$\Gamma_\rho$;
\item
for any~$F,F'\in S_\rho$ any isotopy in the class~$S_\rho$ that brings~$F$ to~$F'$ (respecting the correspondence
between connected components of~$\partial F$ and~$\partial F'$ with the elements of~$\partial\rho$)
also brings~$\Delta(F)$ to~$\Delta(F')$.
\end{enumerate}

Two divided spatial ribbon graphs~$\rho$ and~$\rho'$ are said to be \emph{equivalent}
if a surface~$F\in S_\rho$ can be brought to a surface~$F'\in S_{\rho'}$ by an isotopy
that also brings~$\Delta(F)$ to~$\Delta(F')$.
\end{defi}

\begin{defi}
Let~$F$ be a surface contained in another surface~$F'$, and let~$\delta$ and~$\delta'$ be
abstract dividing sets on~$F$ and~$F'$, respectively. We say that~$\delta'$
\emph{extends}~$\delta$ if~$\delta=\delta'\cap F$ and any connected component of~$\delta'$ has
a non-empty intersection with~$F$.

If surfaces~$F$ and~$F'$ are endowed with isotopy classes
of abstract dividing sets~$\Delta(F)$ and~$\Delta(F')$, respectively,
we write~$F\subset_{\mathrm d}F'$ to state that~$F\subset F'$
and any~$\delta\in\Delta(F)$ can be extended to some~$\delta'\in\Delta(F')$.
\end{defi}

\begin{defi}\label{negl-hole-def}
Let~$\rho$ be a divided ribbon graph.
An inessential boundary circuit~$\gamma$ of~$\rho$ is called
\emph{negligible} if, for some (and then any) surface~$F\in S_\rho$
and some (and then any) representative~$\delta\in\Delta(F)$, the connected component~$\widetilde\gamma$ of~$\partial F$ corresponding to~$\gamma$
meets~$\delta$ exactly twice. In this case, the component~$\widetilde\gamma$ of~$\partial F$
is also called negligible.
\end{defi}

The reason why we single out this type of hole is the following one. If a surface~$F'$ is
obtained from a surface~$F\in S_\rho$ by patching a negligible hole, where~$\rho$ is a divided ribbon graph,
then there is a unique way to endow~$F'$ with an isotopy class of
abstract dividing sets so as to have~$F\subset_{\mathrm d}F'$.

For a divided spatial ribbon graph~$\rho$, we denote by~$\overline S_\rho$
the class of all surfaces~$F$, each endowed with an isotopy class~$\Delta(F)$ of abstract dividing sets,
such that
$F$ is obtained from some~$F_0\in S_\rho$ by patching all negligible holes, and
we have~$F_0\subset_{\mathrm d}F$.

\begin{defi}
Two divided spatial ribbon graphs~$\rho$ and~$\rho'$ are said to be \emph{stably equivalent}
if a surface~$F\in\overline S_\rho$ can be brought to~$F\in\overline S_{\rho'}$ by an isotopy
that also takes~$\Delta(F)$ to~$\Delta(F')$.
\end{defi}

\begin{defi}
Let~$\rho$ be a divided ribbon graph,
and let~$F_0\in S_\rho$, $\delta_0\in\Delta(F)$.
Two connected
components~$\beta_1$ and~$\beta_2$, say, of~$\delta_0$ are said to be
\emph{coherent} if there is a surface~$F\in\overline S_\rho$
and a representative~$\delta\in\Delta(F)$ such that~$\delta_0=\delta\cap F_0$
and~$\beta_1$, $\beta_2$ are contained in the same connected component of~$\delta$.
If no two connected components of~$\delta_0$ are coherent, we say that~$\rho$ is \emph{reduced}.
\end{defi}

\begin{defi}\label{propagation-def}
Let~$\rho$ and~$\rho'$ be two divided ribbon graphs
such that~$\rho'$ viewed as an ordinary spatial ribbon
graph is
obtained from~$\rho$ by a handle addition. Let also~$F$, $F'$, and~$d$ be as in Definition~\ref{handle-addition-def}.
We say that the transition~$\rho\mapsto\rho'$ is \emph{a propagation},
and the inverse transition is \emph{a shrinking} if
the hole patched by~$d$ is negligible, and we have~$F'\subset_{\mathrm d}F$.
\end{defi}

Let us explain this operation in different terms.  One can see that if~$F$ is a compact surface,
and~$\delta$ is an abstract dividing set on~$F$ consisting of pairwise disjoint arcs, then there exists a unique equivalence
class of divided spatial ribbon graphs~$\rho$ such that~$F\in S_\rho$ and~$\delta\in\Delta(F)$.
Thus, isotopy classes of such pairs~$(F,\delta)$, which we call \emph{divided surfaces},
are in a natural one-to-one correspondence with
equivalence classes of divided spatial ribbon graphs.

For~$(F,\delta)$ a divided surface, a connected component~$\gamma$ of~$\partial F$
is \emph{negligible} if~$\gamma$ is inessential and there are
exactly two endpoints of~$\delta$ on~$\gamma$.

If a connected component~$\beta$ of~$\delta$ has one endpoint on a negligible
component~$\gamma$ of~$\partial F$, and the other endpoint on a component~$\gamma'\ne\gamma$,
the removal of a tubular neighborhood~$U$ of~$\beta$ such that~$U\cap\delta=\beta$ from~$F$ represents a shrinking.
This is illustrated in Figure~\ref{propagation-fig}, where the dividing sets of the divided surfaces are shown in green.
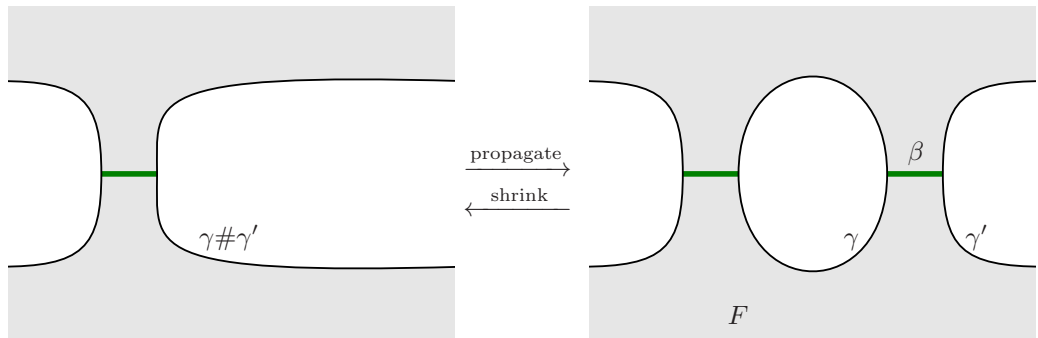
\begin{figure}[ht]
\newlength\tmp
\settowidth\tmp{$\scriptstyle\text{propagate}$}
\includegraphics[scale=0.7]{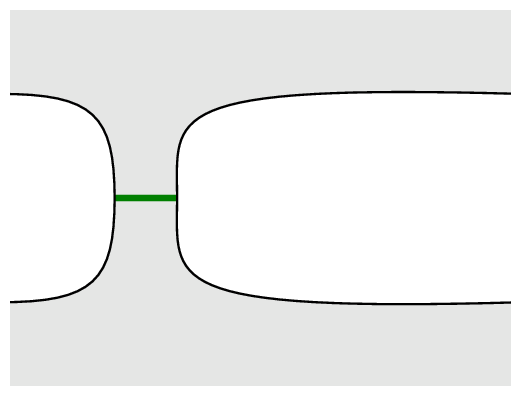}\put(-100,40){$\gamma\#\gamma'$}
\raisebox{60pt}{$\begin{matrix}\xrightarrow{\text{propagate}}\\
\xleftarrow{\hbox to\tmp{\hss$\scriptstyle\text{shrink}$\hss}}\end{matrix}$}
\includegraphics[scale=0.7]{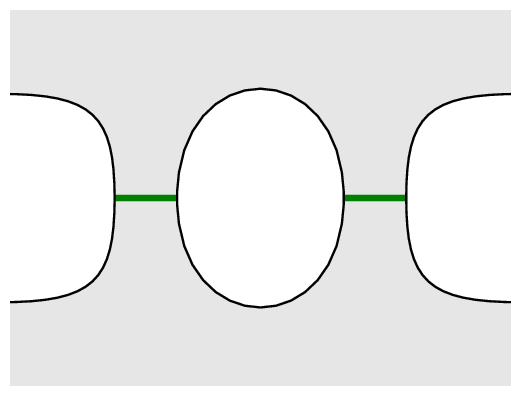}\put(-120,10){$F$}\put(-52,72){$\beta$}%
\put(-76,40){$\gamma$}\put(-30,40){$\gamma'$}
\caption{Propagation and shrinking in terms of divided surfaces}\label{propagation-fig}
\end{figure}
It is understood that the new boundary component, which is the `connected sum'~$\gamma\#\gamma'$
of~$\gamma$ and~$\gamma'$
is essential if and only if so is~$\gamma'$. This implies, in particular, that~$\gamma\#\gamma'$
is negligible if and only if so is~$\gamma'$.

It is easy to see that if~$\beta_1$, $\beta_2$ are two connected components of~$\delta$ distinct from~$\beta$,
then they are coherent if and only if they remain coherent after the shrinking.
So, connected components of~$\delta$ fall into several coherence classes,
and a propagation adds an element to one of these classes, whereas a shrinking removes
an element from a class. The number of coherence classes remains fixed.

With this explanation at hand the following statement is obvious.

\begin{prop}\label{equivalence-of-divided-graphs-prop}
\emph{(i)} Two divided ribbon graphs~$\rho$ and~$\rho'$
are stably equivalent if and only if there is an ribbon graph~$\rho''$
equivalent to~$\rho'$ that is obtained from~$\rho$ by a finite sequence of propagations and shrinkings.

\emph{(ii)} For any divided spatial ribbon graphs~$\rho$, there is a sequence of shrinkings that produces
a reduced divided ribbon graph from~$\rho$.

\emph{(iii)} Two reduced divided spatial ribbon graphs are equivalent if and only if they
are stably equivalent.
\end{prop}

Now we plug back in the contact structure~$\xi_+$.

\begin{defi}\label{divided-leg-rib-graph-def}
\emph{A divided Legendrian ribbon graph}
is a divided spatial ribbon graph~$\rho$ such that~$\Gamma_\rho$
is a Legendrian graph,
and there are representatives~$F\in S_\rho$ and~$\delta\in\Delta(F)$ with the following properties:
\begin{enumerate}
\item
$\delta$ is disjoint from~$V_\rho$;
\item
each connected component of~$\delta$ is an arc intersecting~$\Gamma_\rho$ exactly once;
\item
each edge of~$\rho$  has at most one intersection with~$\delta$;
\item
$F$ is tangent to~$\xi_+$ along all edges of~$\rho$ that are disjoint from~$\delta$;
\item
if an edge~$e$ of~$\rho$ intersects~$\delta$ (such edges are called~\emph{divided}), then~$e$ is a $-1$-arc of the characteristic foliation of~$F$,
i.e.\ a Legendrian arc such that the contact plane~$\xi_+(p)$ makes a negative half-turn relative to the tangent plane~$T_pF$
when~$p$ traverses the edge~$e$, and~$\xi_+(p)\ne T_pF$ for~$p\in\interior(e)$;
\item
the orientation of each connected component~$\beta\subset\delta$  agrees with the coorientation of~$\xi_+$
at the point~$\beta\cap\Gamma_\rho$.
\end{enumerate}
\end{defi}

It is not hard to see that, if~$\Gamma$ is a Legendrian graph with the set of vertices~$V$ such that
any edge is a smooth simple arc, and~$X$ is a subset of edges of~$\Gamma$,
then there is a unique divided Legendrian ribbon graph~$\rho$ with~$\Gamma_\rho=\Gamma$,
$V_\rho=V$ such that the set of divided edges of~$\Gamma$ is~$X$.

The main source of divided Legendrian ribbon graphs is the class of extended Giroux graphs of convex surfaces, see~\cite[Definition~27]{dp17}.
Let~$F$ be a Giroux's convex surface with Legendrian boundary and very nice characteristic foliation, and let~$\delta$
be a dividing set of~$F$. Let also~$\mathscr G$ be the Giroux graph, and~$\widetilde{\mathscr G}$ an extended Giroux graph of~$F$.
By declaring the $-1$-arcs in~$\widetilde{\mathscr G}$, which are not contained in~$\mathscr G$, to be divided edges
we get a divided Legendrian ribbon graph, which we denote by~$\rho(F,\widetilde{\mathscr G})$.
To get representatives~$F'\in S_{\rho(F,\widetilde{\mathscr G})}$ and~$\delta'\in\Delta(F')$
it suffices to take a small open neighborhood~$U$ of~$\widetilde{\mathscr G}$ in~$F$ and put~$F'=\overline U$, $\delta'=\delta\cap F'$.

The point is that the divided Legendrian ribbon graph~$\rho(F,\widetilde{\mathscr G})$
constructed in this way carries essentially all interesting information
about~$F$ if the latter is viewed up to isotopy in the class of Giroux's convex
surfaces. This means that we can study convex surfaces by operating essentially one-dimensional
objects, namely, Legendrian graphs with some additional structure.

The link to extended Giroux graphs also motivates introducing the following
equivalence relation, up to which it is natural to consider divided Legendrian ribbon
graphs.

\begin{defi}
Two divided Legendrian ribbon graphs~$\rho_0$ and~$\rho_1$ are said to be
\emph{Legendrian equivalent} if there is a $1$-parametric family of divided Legendrian ribbon graphs~$\rho_t$, $t\in[0,1]$,
connecting~$\rho_0$ and~$\rho_1$, and an isotopy~$F_t$, $t\in[0,1]$, from a surface~$F_0\in S_{\rho_0}$
to a surface~$F_1\in S_{\rho_1}$ such that~$F_t\in S_{\rho_t}$
and this isotopy brings~$\Delta(F_0)$ to~$\Delta(F_t)$ for all~$t\in[0,1]$.
\end{defi}

In terms of Legendrian graphs this definition means the following. We consider Legendrian
graphs some edges of which are declared divided. Two such Legendrian graphs are regarded equivalent
if one can be obtained from the other by an isotopy in the class of Legendrian graphs and operations
of edge contraction/blow-up which are allowed only on non-divided edges.
The divided edges are forbidden to collapse and to be born as a result of an edge blow up.

There is an arbitrariness in the definition of an extended Giroux graph~$\widetilde{\mathscr G}$ of a convex surface~$F$
that allows to add more divided edges to it at our will. Such an addition will change the Legendrian
equivalence class of~$\rho(F,\widetilde{\mathscr G})$. Propagations and shrinkings take care of this.

\begin{defi}\label{propag-shri-def}
In the case of divided Legendrian ribbon graphs, \emph{a propagation} and \emph{a shrinking}
have the same meaning as in the case of general divided spatial ribbon graphs (see Definition~\ref{propagation-def}).
Two divided ribbon graphs~$\rho$ and~$\rho'$ are said to be \emph{stably Legendrian equivalent}
if there is a divided Legendrian ribbon graph~$\rho''$ such that~$\rho''$
is Legendrian equivalent to~$\rho'$ and is obtained from~$\rho$ by finitely many propagations
and shrinkings.
\end{defi}

\begin{theo}\label{convex-surf-through-div-leg-graph-th}
Let~$F$ and~$F'$ be two connected Giroux's convex surfaces with Legendrian boundary and very nice
characteristic foliation, and let~$\mathscr G$, $\mathscr G'$ be their respective extended Giroux graphs.
Suppose that all connected components of~$\partial F$ and~$\partial F'$ are essential. Then divided Legendrian
ribbon graphs~$\rho(F,\mathscr G)$ and~$\rho(F',\mathscr G')$ are stably
Legendrian equivalent if and only if the surfaces~$F$ and~$F'$ are isotopic through the class
of Giroux's convex surfaces with Legendrian boundary.
\end{theo}

We omit the proof, which will be published elsewhere.

With every mirror diagram~$M$ we associate
a divided Legendrian ribbon graph, which we denote by~$\rho^+(M)$,
by letting the underlying spatial ribbon graph be~$\widehat M$,
and letting~$\delta_+$ be a representative of~$\Delta(\wideparen M)$, where~$\delta_+$
is the first entry of a canonic dividing configuration of~$\wideparen M$ (see Definition~\ref{canonic-dividing-for-mirror-diagrams-def}).

One can, of course, define~$\rho^-(M)$ similarly using the symmetry between~$\xi_+$ and~$\xi_-$.

\begin{theo}\label{type-i-moves-meaning-th}
\emph{(i)}
Any divided Legendrian ribbon graph is stably Legendrian equivalent to~$\rho^+(M)$
for some mirror diagram~$M$.

\emph{(ii)}
Let~$M$ and~$M'$ be mirror diagrams. The divided Legendrian ribbon graphs~$\rho^+(M)$
and~$\rho^+(M')$ are stably Legendrian equivalent if and only if~$M$ and~$M'$ are related by
a finite sequence of type~I elementary moves.
\end{theo}

We need a little preparation before the proof.

\begin{defi}\label{neglibigle-circuit-def}
A boundary circuit~$c$ of a(n enhanced) mirror diagram~$M$ is called $+$-negligible (respectively, $-$-negligible)
if~$c$ is inessential and~$\tb_+(c)=-1$ (respectively, $\tb_-(c)=-1$).
\end{defi}

Recall that the equality~$\tb_+(c)=-1$ is equivalent to saying that~$c$ hits $\diagdown$-mirrors of~$M$ exactly twice.
Thus~$c$ is $+$-negligible if and only if the respective boundary circuit~$\widehat c$ of~$\rho^+(M)$ is negligible.

\begin{defi}\label{coherent-mirror-def}
Two distinct $\diagdown$-mirrors~$\mu$ and~$\mu'$ of a(n enhanced) mirror
diagram~$M$ are said to be \emph{immediately coherent}
if there is a $+$-negligible boundary circuit that hits both~$\mu$ and~$\mu'$.

Two mirrors~$\mu$ and~$\mu'$ are said to be \emph{coherent} if there are mirrors~$\mu_0=\mu,\mu_1,\mu_2,\ldots,\mu_k=\mu'$
such that~$\mu_{i-1}$ and~$\mu_i$ are immediately coherent for all~$i=1,\ldots,k$.

Immediate coherence and coherence of~$\diagup$-mirrors is defined similarly, using $-$-negligible boundary circuits.
\end{defi}

Let~$M$ be a mirror diagram, and let~$(\delta_+,\delta_-)$ be a canonic dividing configuration of~$\wideparen M$.
It is easy to see that two $\diagdown$-mirrors (respectively, $\diagup$-mirrors) of~$M$ are coherent if and only if
so are the respective components of~$\delta_+$ (respectively, $\delta_-$).

Recall that Definition~\ref{ass-mir-diagr-def} introduces the mirror diagram~$M(\Pi)$ associated with a rectangular
diagram~$\Pi$. This definition makes perfect sense even if~$\Pi$ is any finite collection of pairwise compatible rectangles,
that is, a collection of rectangles that forms a subset of a rectangular diagram of a surface.
Recall also that, in the proof of Lemma~\ref{handle-reduction-seq-lem},
we introduced the union of two mirror diagrams if they agree on the common mirrors.

\begin{defi}\label{patching-mirror-diagr-def}
Let~$M$ be an enhanced mirror diagram, and let~$c$ be an inessential circuit of~$M$.
Let also~$\Pi$ be a collection of rectangles in~$\mathbb T^2$ 
such that~$\widehat\Pi=\cup_{r\in\Pi}\widehat r$
is a patching disc for~$\widehat c$ in~$\widehat M$. Denote by~$M'$ the enhanced mirror diagram
whose underlying mirror diagram is~$M\cup M(\Pi)$, and the set of essential
boundary circuits is the same as that of~$M$.

Then the transition from~$M$ to~$M'$ is called \emph{a patching of~$c$}. With such an operation we
associate a morphism~$\eta$ defined by~$(\wideparen M,\wideparen M',\mathrm{id}|_{\wideparen M})\in\eta$.
\end{defi}

The mirrors explicitly mentioned in Definitions~\ref{ext-move-def}, \ref{mirr-bypass-def}, and~\ref{mirr-slide-def}
are said to \emph{participate} in the respective moves.

\begin{lemm}\label{patching-lem}
Let~$M$ be an enhanced mirror diagram, and let~$c$ be a $+$-negligible boundary circuit of~$M$.
Then any patching of~$c$ admits a neat decomposition
into type~I extensions and elementary bypass additions. Moreover,
if~$c$ hits two distinct $\diagdown$-mirrors of~$M$, then
the decomposition of any patching of~$c$ into type~I elementary moves
can be chosen so that one of the $\diagdown$-mirrors, at our choice,
does not participate in the moves.
\end{lemm}

\begin{proof}
Let~$\Pi$ be
a collection of rectangles in~$\mathbb T^2$ 
such that~$\widehat\Pi=\cup_{r\in\Pi}\widehat r$ is a patching disc for~$\widehat c$ in~$\widehat M$.
We proceed by induction in the number~$k$ of rectangles in~$\Pi$. If~$k=1$ there is nothing to prove, since~$M=M\cup M(\Pi)$.

Suppose that~$k>1$.
Let~$r$ be a rectangle in~$\Pi$ such that~$\partial r$ hits a~$\diagdown$-mirror of~$M$. Denote the $\diagdown$-vertices of~$r$ by~$\mu_1$ and~$\mu_2$,
and the $\diagup$-vertices of~$r$ by~$\mu_3$ and~$\mu_4$.
Denote also by~$M'$ the diagram~$M\cup M(\{r\})$ having the same essential boundary circuits as~$M$ has.

Only one of~$\mu_1$ and~$\mu_2$ can be a mirror of~$M$.
Indeed, otherwise the diagram~$M'$ would have an inessential boundary circuit hitting only $\diagup$-mirrors,
that is, with zero Thurston--Bennequin number~$\tb_+$, which would mean the existence of an overtwisted disc.
For the same reason the intersection of~$\partial\widehat r$ with~$\Gamma_{\widehat M}$ must be connected.

If~$c$ hits two distinct $\diagdown$-mirrors and we want one of them~$\mu_0$, say,
not to be involved in the moves from the sought-for decomposition, the rectangle~$r$
should be chosen disjoint from~$\mu_0$. This is possible as no rectangle in~$\Pi$
hits both $\diagdown$-mirrors on~$c$.

We may assume without loss of generality that~$\mu_1\in E_M$ and~$\mu_2\notin E_M$.
If~$\mu_3,\mu_4\in E_M$, then~$M\mapsto M'$ is a type~I elementary bypass addition.
Suppose that~$\mu_3\notin E_M$. Let~$x$ be the meridian or longitude passing through~$\mu_2$ and~$\mu_3$.
We must have~$x\notin L_M$, since otherwise the intersection~$\partial\widehat r\cap\Gamma_{\widehat M}$
would not be connected. The addition of~$\mu_3$ and~$x$ to~$M$ is, therefore, a type~I extension move.
Similarly, if~$\mu_4\notin E_M$, then the addition of~$\mu_4$ and the meridian or longitude passing through~$\mu_2$ and~$\mu_4$
is a type~I extension move.

Thus, we can proceed from~$M$ to~$M'$ by zero, one, or two type~I extension moves followed by a type~I
elementary bypass addition.

Now let~$\Pi'=\Pi\setminus\{r\}$. One can see that~$\widehat\Pi'$ is a patching disc for~$\widehat M'$,
and~$M'\cup M(\Pi')=M\cup M(\Pi)$. The induction step follows.

All the morphisms associated with the discussed moves have a representative identical on~$\wideparen M$,
which implies that the obtained decomposition is neat.
\end{proof}

\begin{lemm}\label{remove-coherent-mirror-lem}
Let~$M$ be an enhanced mirror diagram, and let~$c$ be a $+$-negligible boundary circuit that
hits two distinct $\diagdown$-mirrors~$\mu_1$ and~$\mu_2$.
Let also~$c'$ be the boundary circuit that hits the other side of~$\mu_2$.

Denote by~$M'$ the enhanced mirror diagram obtained from~$M$ by removing
the mirror~$\mu_2$ and declaring the new boundary circuit, which is naturally denoted~$c\# c'$,
essential or inessential depending on whether~$c'$ is essential or not. Denote also
by~$\eta$ the morphism from~$M$ to~$M'$ defined by~$(F,F,\mathrm{id}|_F)\in\eta$, where~$F$ is an arbitrary surface carried by~$\widehat M$.
Then the transformation~$M\xmapsto\eta M'$ admits a neat decomposition into type~I elementary moves.
\end{lemm}

\begin{proof}
Let~$M''$ be an enhanced mirror diagram obtained from~$M$
by a patching of~$c$. By Lemma~$\ref{patching-lem}$ there exists a sequence of type~I extensions and elementary
bypass additions
$$M=M_0\mapsto M_1\mapsto\ldots\mapsto M_k=M''$$
not involving the mirror~$\mu_2$.

In each diagram~$M_i$, $i=0,\ldots,k$, the
mirror~$\mu_2$ is hit by~$c'$ on one side and by an inessential boundary circuit~$c_i$ on the
other side. Let~$M_i'$ be obtained from~$M_i$ by removing~$\mu_2$ and declaring the new boundary circuit~$c_i\#c'$
essential or inessential depending on whether~$c'$ is essential or inessential. Then~$M_k\mapsto M_k'$
is a type~I elementary bypass removal, and~$M_i'\mapsto M_{i-1}'$ is either a type~I elimination move
or a type~I elementary bypass removal depending on whether~$M_{i-1}\mapsto M_i$ is an extension move or an
elementary bypass addition.

Thus, the sought-for decomposition is
$$M=M_0\mapsto M_1\mapsto\ldots\mapsto M_k\mapsto M_k'\mapsto M_{k-1}'\mapsto\ldots\mapsto M_1'\mapsto M_0'=M'.$$
The details are left to the reader.
\end{proof}

\begin{proof}[Proof of Theorem~\ref{type-i-moves-meaning-th}]
The part~(i) follows from the Approximation Principle of~\cite{dp17},
and holds without `stably'. Namely, for a given divided Legendrian ribbon graph~$\rho$ we can obtain
a mirror diagram~$M$ with~$\rho^+(M)$ equivalent to~$\rho$, as follows.
Make a generic Legendrian perturbation of~$\Gamma_\rho$ and take the torus projection of the obtained graph.
Then approximate this torus projection by a rectangular diagram of a graph~$G$ as described
in the proof of~\cite[Proposition~8]{dp17}.

We may assume that the torus projection of each edge~$e$ of~$\Gamma_\rho$ (after a generic perturbation)
is approximated in~$G$ by a staircase arc consisting of more than one edge.
If~$e$ is a divided edge, choose a vertex of~$G$ in the respective staircase arc
distinct from the endpoints (the latter may be shared with other staircase arcs) and assign
type~`$\diagdown$' to it. Do it for all divided edges. All the other vertices of~$G$ are assigned
type~`$\diagup$'. We get a mirror diagram~$M$ that represents a divided Legendrian ribbon graph
equivalent to~$\rho$.

Now proceed with proving part~(ii) of the theorem. The fact that type~I elementary moves
of mirror diagrams preserve the stable Legendrian equivalence class
of~$\rho^+(M)$ follows from a direct check: type~I extensions and slides
preserve the equivalence class, whereas a type~I elementary bypass addition translates
into a propagation.

The fact that these moves suffice to
transform~$M$ to~$M'$, provided that~$\rho^+(M)$ and~$\rho^+(M')$ are stably Legendrian
equivalent, is established similarly to Theorem~\ref{equivalence-of-legendrian-graphs-thm}.
Lemma~\ref{remove-coherent-mirror-lem} implies that propagations and shrinkings
of divided Legendrian ribbon graphs of the form~$\rho^+(M)$ also can be realized as compositions type~I
moves.\end{proof}

\subsection{Type~II moves from the point of view of~$\xi_+$-divided Legendrian ribbon graphs}\label{type-ii-moves-for-divided-graph-subsec}
Here we consider what happens to the divided Legendrian ribbon graph~$\rho^+(M)$
when a type~II move is applied to~$M$. More honestly, we look at what happens to
the stable equivalence class of the underlying divided ribbon graph.
The type~II moves that we discuss include extension, elimination, elementary bypass addition,
elementary bypass removal, split, and merge moves.

Let~$M\mapsto M'$ be one of the listed moves, and let~$F\in\overline S_{\rho^+(M)}$ and $F'\in\overline S_{\rho^+(M')}$
be obtained from~$\wideparen M$ and~$\wideparen M'$,
respectively, by patching all negligible holes.
Choose some representatives~$\delta\in\Delta(F)$ and~$\delta'\in\Delta(F')$. We are going
to discuss how the divided surfaces~$(F,\delta)$ and~$(F',\delta')$ are related in each case
depending on the kind of the move~$M\mapsto M'$, which we consider now one by one.
The moves come in three pairs of mutually inverses, so we need to consider
only one move in each pair.

We assume that the surfaces~$F$ and~$F'$ are chosen
to have as much in common as possible. In particular, the common negligible boundary circuits of~$\rho^+(M)$ and~$\rho^+(M')$
are patched in ~$F$ and~$F'$ in the same way.

\smallskip\noindent\emph{Case 1}:~$M\mapsto M'$ is a type~II extension move.
Denote by~$c$ the boundary circuit of~$M$ on which a new mirror is added, and by~$c'$ the
respective boundary circuit of~$M'$.

The hole~$\wideparen c'\subset\partial\wideparen M'$ is never negligible (since either~$c$ and~$c'$ are essential, or~$\tb_+(c)<0$
and hence~$\tb_+(c')<-1$),
but it may happen that~$\wideparen c$ is negligible. In this case,
$\wideparen c$ is patched in~$F$, but~$\wideparen c'$ is not patched in~$F'$. One can see that~$(F',\delta')$
is obtained from~$(F,\delta)$ by removing an open disc that intersects~$\delta$ in an open arc
and then attaching to~$\wideparen c$ a half-disc with a new component of~$\delta'$ in it whose endpoints
lie on~$\wideparen c'$; see Figure~\ref{extension-for-divided-surface}~(a).
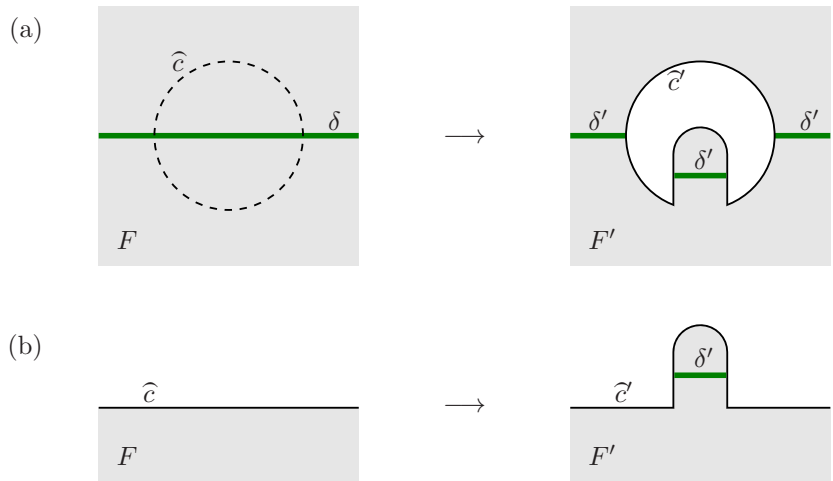
\begin{figure}[ht]
\raisebox{90pt}{\hbox to 0pt{\hss(a)\hskip.5cm}}
\includegraphics[scale=0.7]{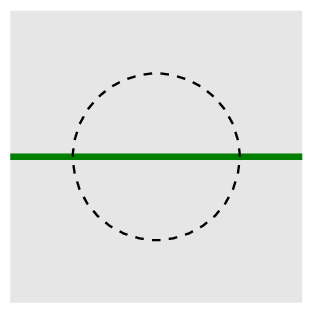}\put(-95,10){$F$}\put(-74,77){$\wideparen c$}\put(-15,55){$\delta$}
\hskip1cm\raisebox{50pt}{$\longrightarrow$}\hskip1cm
\includegraphics[scale=0.7]{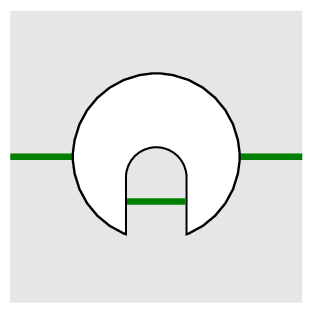}\put(-95,10){$F'$}\put(-65,70){$\wideparen c'$}\put(-15,55){$\delta'$}\put(-55,40){$\delta'$}%
\put(-95,55){$\delta'$}

\vskip.5cm
\raisebox{52pt}{\hbox to 0pt{\hss(b)\hskip.5cm}}
\includegraphics[scale=0.7]{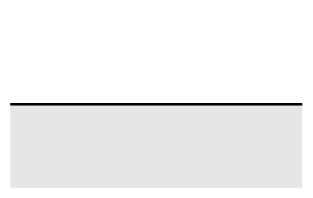}\put(-95,10){$F$}\put(-85,34){$\wideparen c$}
\hskip1cm\raisebox{30pt}{$\longrightarrow$}\hskip1cm
\includegraphics[scale=0.7]{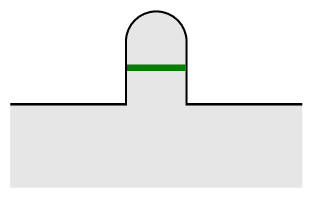}\put(-95,10){$F'$}\put(-85,34){$\wideparen c'$}\put(-55,47){$\delta'$}
\caption{A type~II extension move from the divided surface associated with~$\rho^+(M)$ point of view}\label{extension-for-divided-surface}
\end{figure}
The overall effect, viewed up to isotopy,
of these two operations can also be viewed as a removal of an open disc intersecting~$\delta$ in two open arcs.

If~$\wideparen c$ is not negligible, then~$F'$ is obtained from~$F$ just by attaching a half-disc with a new component
of~$\delta'$ inside; see Figure~\ref{extension-for-divided-surface}~(b).

\smallskip\noindent\emph{Case 2}:~$M\mapsto M'$ is a type~II elementary bypass removal.

By the definition of the move, there are two boundary circuits~$c_1$, $c_2$ of~$M$
that are replaced by a single one in~$M'$,
which we denote by~$c_1\#c_2$.
One of~$\wideparen c_1$ and~$\wideparen c_2$ must be negligible, and we assume that so is~$\wideparen c_1$.
If~$\wideparen c_2$ is negligible, too, then~$(F',\delta')$ is obtained from~$(F,\delta)$ by removing
two open discs, each intersecting~$\delta$ in an open arc, and then removing a $1$-handle `in between'.
The overall effect can be viewed as a removal of
an open disc intersecting~$\delta$ in two open arcs; see Figure~\ref{bypass-removal-for-divided-surface}~(a).
\begin{figure}[ht]
\raisebox{60pt}{\hbox to 0pt{\hss(a)\hskip.5cm}}
\includegraphics[scale=0.7]{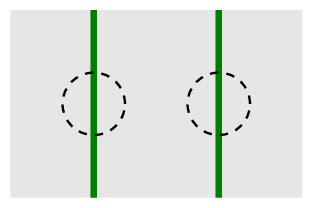}\put(-95,10){$F$}\put(-90,45){$\wideparen c_1$}\put(-48,45){$\wideparen c_2$}%
\put(-72,13){$\delta$}\put(-30,13){$\delta$}
\hskip1cm\raisebox{30pt}{$\longrightarrow$}\hskip1cm
\includegraphics[scale=0.7]{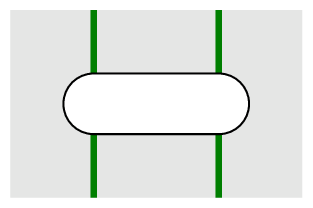}\put(-95,10){$F'$}\put(-65,29){$\wideparen{c_1\#c_2}$}%
\put(-72,13){$\delta'$}\put(-30,13){$\delta'$}\put(-72,50){$\delta'$}\put(-30,50){$\delta'$}

\vskip.8cm
\raisebox{60pt}{\hbox to 0pt{\hss(b)\hskip.5cm}}
\includegraphics[scale=0.7]{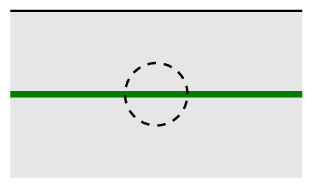}\put(-95,10){$F$}\put(-85,64){$\wideparen c_2$}\put(-15,34){$\delta$}%
\put(-69,42){$\wideparen c_1$}
\hskip1cm\raisebox{30pt}{$\longrightarrow$}\hskip1cm
\includegraphics[scale=0.7]{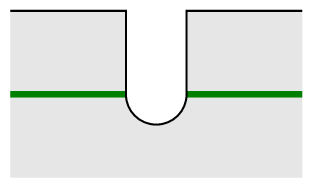}\put(-95,10){$F'$}\put(-95,64){$\wideparen{c_1\#c_2}$}\put(-15,34){$\delta'$}%
\put(-90,34){$\delta'$}
\caption{A type~II elementary bypass removal
from the divided surface associated with~$\rho^+(M)$ point of view}\label{bypass-removal-for-divided-surface}
\end{figure}
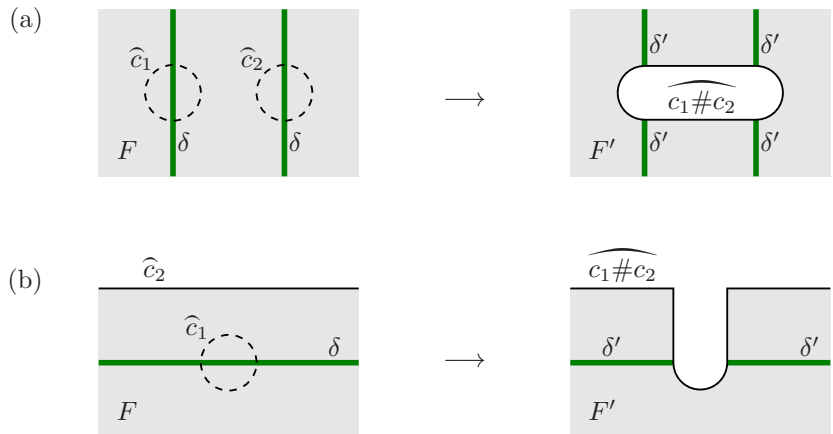

Otherwise, the effect of the move~$M\mapsto M'$ on~$(F,\delta)$ is a removal of a half-disc intersecting~$\delta$ in
a single open arc; see Figure~\ref{bypass-removal-for-divided-surface}~(b).

\smallskip\noindent\emph{Case 3}:~$M\mapsto M'$ is a type~II split move.

Let~$c$ be the boundary circuit of~$M$ containing the snip point of the move~$M\mapsto M'$.
As we saw in Subsection~\ref{hMM-subsec} the surface~$\wideparen M'$
is obtained from~$\wideparen M$ by cutting along an arc~$\beta$ that has exactly one endpoint on~$\partial\wideparen M$, and intersects~$\delta$ transversely in a single point, followed by a small deformation.
If~$c$ is not negligible, then the corresponding component of~$\partial\wideparen M$
is not patched in~$F$, and the effect of the move~$M\mapsto M'$ on the surface~$F$
consists in a removal of a half-disc intersecting~$\delta$ in a single open arc.

If~$c$ is negligible, then the corresponding connected component of~$\partial\wideparen M$
is patched, and the transition from~$F$ to~$F'$ consists in removing the respective
patching disc and then removing a half-disc as in the previous case.
The overall effect is a removal of an open disc intersecting~$\delta$ in two open arcs.

We see that the effect of a type~II split move on the divided surface~$(F,\delta)$ is analogous
to that of a type~II elementary bypass removal. In the case of a type~II extension
move the effect is also a particular case of this, provided that the modified boundary circuit has negative
Thurston--Bennequin number~$\tb_+$ before the move.

Indeed, in the case of a negligible boundary circuit the negativity of~$\tb_+$ is automatic, and
the topological effect of an extension is a removal of an open disc intersecting~$\delta$
in two open arcs as noticed above.
If a type~II extension move modifies a non-negligible boundary circuit with~$\tb_+<0$,
then the addition of a half-disc containing a new connected component of~$\delta'$ can
be viewed, up to isotopy, as a removal of a half-disc intersecting~$\delta$ in an open arc.
This is illustrated in Figure~\ref{tb<0extension-fig}.
\begin{figure}[ht]
\includegraphics[scale=0.7]{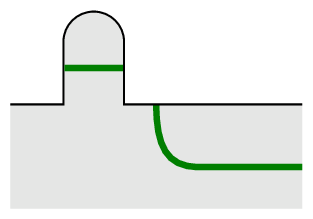}
\hskip1cm\raisebox{30pt}{$\sim$}\hskip1cm
\includegraphics[scale=0.7]{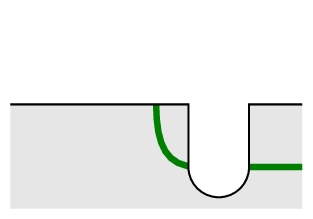}
\caption{Attaching a half-disc with a new component of~$\delta'$
is equivalent to removing a closed half-disc intersecting~$\delta$ in an open arc}\label{tb<0extension-fig}
\end{figure}

These observations give a hint for the idea behind generalized type~II moves introduced in Section~\ref{commutation-sec}.
The generalized moves are designed to have the same effect on the isotopy
class of the respective divided surfaces as split/merge and elementary bypass removal/addition
moves have: either removing/gluing up a disc intersecting the dividing set in two open arcs, or removing/gluing up a half-disc intersecting
the dividing set in a single open arc.

We also see that type~II extension moves modifying a boundary circuit with~$\tb_+=0$
fall out of this scheme. A natural generalization of such extension moves is not `good enough' like other generalized moves,
which is, roughly, the reason for including the flexibility assumptions in Theorems~\ref{warm-up-thm} and~\ref{commutation-2-thm}.

\subsection{Subdiagrams of mirror diagrams}

\begin{defi}
An enhanced mirror diagram~$M_1$ is said to be \emph{a subdiagram} of an enhanced mirror diagram~$M_2$
if the following holds:
$$L_{M_1}\subset L_{M_2},\quad
E_{M_1}\subset E_{M_2},\quad
T_{M_2}\bigr|_{E_{M_1}}=T_{M_1},\quad
\partial M_1\cap\partial M_2\cap H_{M_1}= \partial M_1\cap\partial M_2\cap H_{M_2}.$$
In this case we write~$M_1\subset M_2$.
\end{defi}

In other words, this definition means that~$M_2$ is obtained from~$M_1$ by adding
some (may be none) new occupied levels and mirrors without modifying the enhancement on the
preserved boundary circuits.

\begin{prop}\label{subdiagram-move-prop}
Let~$M_1$ be a subdiagram of an enhanced mirror diagram~$M_2$, and let $M_1\mapsto M_1'$ be
a move from the following list:
\begin{itemize}
\item
a type~I extension, elimination, split, or merge move;
\item
a jump or twist move.
\end{itemize}

Then there exist a mirror diagram~$M_2'$ and
a sequence~$s$ of type~I elementary moves moves producing~$M_2'$ from~$M_2$ such that
$M_1'$ is a subdiagram of~$M_2'$, and the composition of the moves in~$s$ transforms
any essential boundary circuit~$c\in\partial M_1\cap\partial M_2$ in exactly the same way
as the move~$M_1\mapsto M_1'$ does. This means, in particular, that
the boundary circuit of~$M_2'$ corresponding to such~$c$ belongs to~$\partial M_1'$.
\end{prop}

We need the following preparatory lemma.

\begin{lemm}\label{rem-obst-lem}
Let~$M$ be an enhanced mirror diagram, and let~$\theta_1,\theta_2,\varphi_1,\varphi_2\in\mathbb S^1$
be such that~$\theta_1\ne\theta_2$, $\varphi_1\ne\varphi_2$. Let also~$\Omega\subset\mathbb T^2$
be either~$(\theta_1;\theta_2]\times(\varphi_1;\varphi_2]$
or~$[\theta_1;\theta_2)\times[\varphi_1;\varphi_2)$. Then there exists a sequence
of type~I elementary moves transforming~$M$ to another enhanced mirror diagram~$M'$
so that~$M'$ does not have any mirrors in~$\Omega$, and all boundary circuits of~$M$ that don't hit any mirror in~$\Omega$
are preserved by all these moves. Moreover, for any~$\varepsilon>0$ chosen in advance such that~$\theta_2+\varepsilon\in(\theta_2;\theta_1)$
and~$\varphi_2+\varepsilon\in(\varphi_2;\varphi_1)$, we can ensure that
\begin{enumerate}
\item
all mirrors in~$E_{M'}\setminus E_M$ are contained in~$(\theta_1;\theta_2+\varepsilon)\times(\varphi_1;\varphi_2+\varepsilon)\setminus\Omega$
if~$\Omega=(\theta_1;\theta_2]\times(\varphi_1;\varphi_2]$, and in~$(\theta_1-\varepsilon;\theta_2)\times(\varphi_1-\varepsilon;\varphi_2)\setminus\Omega$ if~$\Omega=[\theta_1;\theta_2)\times[\varphi_1;\varphi_2)$;
\item
if~$\varphi\in(\varphi_1;\varphi_2)$ and~$M$ has no mirrors on~$\ell_\varphi\setminus\Omega$, then~$M'$ has no mirrors on~$\ell_\varphi$
\end{enumerate}
\end{lemm}

\begin{proof}
Due to symmetry it suffices to consider the case~$\Omega=(\theta_1;\theta_2]\times(\varphi_1;\varphi_2]$.
The proof is by induction in the number of mirrors of~$M$ contained in~$\Omega$. If there are no such
mirrors then there is nothing to prove. Suppose that there are some.

Then there exists a longitude~$\ell_{\varphi_3}$ such that the following holds:
\begin{enumerate}
\item
$\varphi_3\in(\varphi_1;\varphi_2]$;
\item
there are some mirrors of~$M$ in~$(\theta_1;\theta_2]\times\{\varphi_3\}$;
\item
either~$\varphi_3=\varphi_2$ or there is no mirror of~$M$ in~$(\theta_1;\theta_2]\times(\varphi_3;\varphi_2]$.
\end{enumerate}
Pick a~$\varphi_4\in(\varphi_2;\varphi_2+\varepsilon)$ such that~$M$ has no mirrors in~$\mathbb S^1\times(\varphi_2;\varphi_4]$.
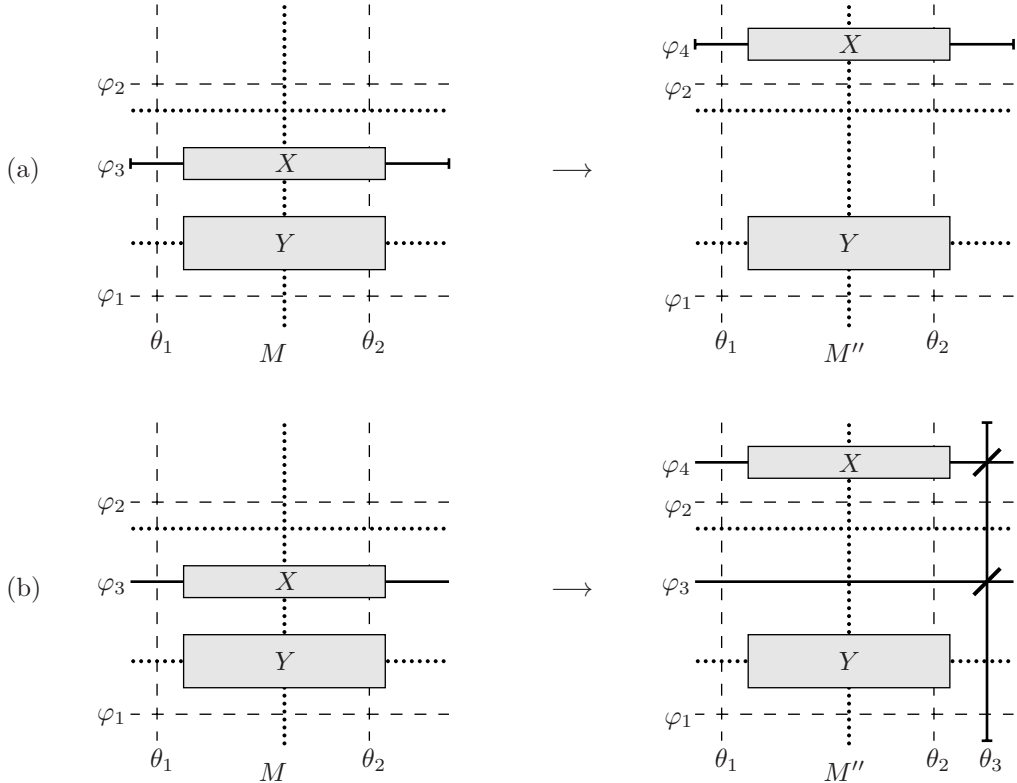
\begin{figure}[ht]\begin{tabular}{lc}
\raisebox{65pt}{(a)\hskip.5cm}&
\includegraphics{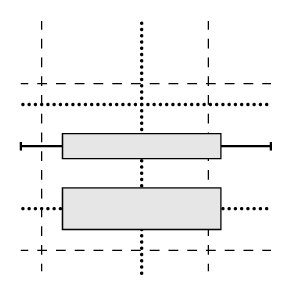}\put(-123,0){$\theta_1$}\put(-43,0){$\theta_2$}%
\put(-143,18){$\varphi_1$}\put(-143,98){$\varphi_2$}%
\put(-76,37){$Y$}\put(-76,67){$X$}\put(-143,67){$\varphi_3$}\put(-82,-5){$M$}
\hskip1cm\raisebox{65pt}{$\longrightarrow$}\hskip1cm
\includegraphics{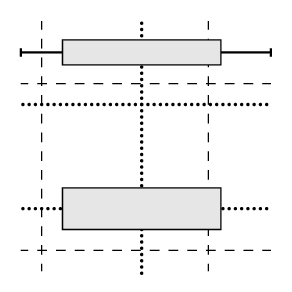}\put(-123,0){$\theta_1$}\put(-43,0){$\theta_2$}%
\put(-142,16.5){$\varphi_1$}\put(-142,96.5){$\varphi_2$}%
\put(-76,37){$Y$}\put(-76,112){$X$}\put(-143,112){$\varphi_4$}\put(-82,-5){$M''$}%
\\[.5cm]
\raisebox{65pt}{(b)}&
\includegraphics{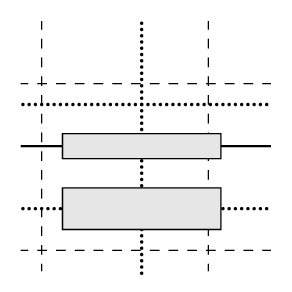}\put(-123,0){$\theta_1$}\put(-43,0){$\theta_2$}%
\put(-143,18){$\varphi_1$}\put(-143,98){$\varphi_2$}%
\put(-76,37){$Y$}\put(-76,67){$X$}\put(-143,67){$\varphi_3$}\put(-82,-5){$M$}
\hskip1cm\raisebox{65pt}{$\longrightarrow$}\hskip1cm
\includegraphics{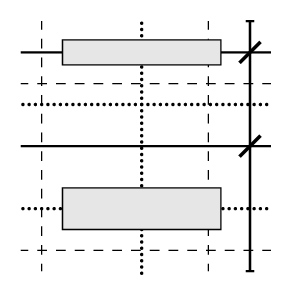}\put(-123,0){$\theta_1$}\put(-43,0){$\theta_2$}%
\put(-142,16.5){$\varphi_1$}\put(-142,96.5){$\varphi_2$}\put(-143,67){$\varphi_3$}%
\put(-76,37){$Y$}\put(-76,112){$X$}\put(-143,112){$\varphi_4$}%
\put(-23,0){$\theta_3$}\put(-82,-5){$M''$}
\end{tabular}
\caption{Removing mirrors from~$\Omega=(\theta_1;\theta_2]\times(\varphi_1;\varphi_2]$. The meridians~$m_{\theta_1}$,
$m_{\theta_2}$ and the longitudes~$\ell_{\varphi_1}$, $\ell_{\varphi_2}$ may or may not be
occupied levels of both diagrams}\label{rem-obst-fig}
\end{figure}

If all mirrors of~$M$ lying on~$\ell_{\varphi_3}$ are in~$\Omega$, we apply a jump move to~$M$
that shifts all these mirrors to~$\ell_{\varphi_4}$; see Figure~\ref{rem-obst-fig}~(a).  Denote by~$M''$ the obtained
enhanced mirror diagram.

If not all mirrors of~$M$ lying on~$\ell_{\varphi_3}$ are in~$\Omega$, we pick a~$\theta_3\in(\theta_2;\theta_2+\varepsilon)$
so that the domain~$(\theta_2;\theta_3]\times\mathbb S^1$ contains no mirrors of~$M$, and then apply to~$M$
a sequence of the following three moves:
\begin{enumerate}
\item
a type~I extension move that adds a new mirror~$\mu$ at~$(\theta_3,\varphi_3)$;
\item
a type~I split move that splits~$\ell_{\varphi_3}$ at~$\mu$ so that all mirrors in~$(\theta_3;\theta_1)\times\{\varphi_3\}$
stay fixed, and mirrors in~$(\theta_1;\theta_2)\times\{\varphi_3\}$ shift upwards slightly;
\item
a jump move that shifts the successors of the mirrors in~$(\theta_1;\theta_2)\times\{\varphi_3\}$ to~$\ell_{\varphi_4}$.
\end{enumerate}
The result is again denoted by~$M''$; see Figure~\ref{rem-obst-fig}~(b).

In both cases, the operations used to produce~$M''$ from~$M$ do not
modify any boundary circuit that is disjoint from~$\Omega$.
The transformation~$M\mapsto M''$ admits a neat decomposition
into type~I elementary moves, due to Lemmas~\ref{neutral-move-decomposition-lem} and~\ref{split-move-decomposition-lem},
and~$M''$ has fewer mirrors in~$\Omega$ than~$M$ has. The induction step follows.
\end{proof}

\begin{proof}[Proof of Proposition~\ref{subdiagram-move-prop}]
We consider different kinds of moves~$M_1\mapsto M_1'$ case by case. In each case, except
the one of elimination move, the sequence~$s$ consists
of two parts. The first part, which we refer to as \emph{the removing of obstacles},
includes only moves that do not alter any boundary circuit in~$\partial M_1\cap
\partial M_2$, and produces a diagram for which the preconditions of the move~$M_1\mapsto M_1'$ hold.
This part is trivial if the preconditions hold already for~$M_2$.

The removing of obstacles part is described in terms of moves that are not elementary, but admit a neat
decomposition into type~I elementary moves by Lemmas~\ref{neutral-move-decomposition-lem} and~\ref{split-move-decomposition-lem}.

The second part of the sequence~$s$
is a neat decomposition into type~I elementary moves of a move that `does the same thing
as~$M_1\mapsto M_1'$ does'.

\medskip\noindent\emph{Case 1}: $M_1\mapsto M_1'$ is a type~I extension move.\\
Denote by~$\mu$ the mirror added
by this move and by~$x$ and~$y$ the two levels passing through~$\mu$ so that
$x$ is an occupied level of both~$M_1$ and~$M_1'$, and $y$ is the occupied level of~$M_1'$ added by the move.

If~$y$ is an occupied level of~$M_2$
we start~$s$ by a jump move
that replaces~$y$ by a parallel occupied level~$y'$ close to~$y$ and shifts the mirrors on~$y$ accordingly, decomposed neatly into a sequence of type~I elementary moves.
Such a decomposition exists by Lemma~\ref{neutral-move-decomposition-lem}.
Since~$y$ is not an occupied level of~$M_1$ no boundary circuit in~$\partial M_1\cap\partial M_2$ can have an edge at~$y$, hence
all the boundary circuits in~$\partial M_1\cap\partial M_2$ remain untouched. This is the removing of obstacles part.

If~$y$ is not an occupied level of~$M_2$ there is no obstacle to remove.

Now append~$s$ by an extension move that adds the occupied level~$y$ and the mirror~$\mu$ to the diagram.
The obtained diagram is taken for~$M_2'$.

\medskip\noindent\emph{Case 2}: $M_1\mapsto M_1'$ is a type~I elimination move.\\
Denote by~$\mu$ and~$y$ the eliminated mirror and occupied level, respectively,
and by~$c$ the boundary circuit of~$M_1$ modified by the move.

If there are some other mirrors of~$M_2$ on~$y$, this means that~$c\notin\partial M_2$.
In this case, we take~$\varnothing$ for~$s$, and~$M_2$ for~$M_2'$.
If~$\mu$ is the only mirror of~$M_2$ at the level~$y$,
then we take for~$s$ a sequence consisting of a single
elimination move~$M_2\mapsto M_2'$ that removes~$\mu$ and~$y$ from
the diagram.

\medskip\noindent\emph{Case 3}: $M_1\mapsto M_1'$ is a type~I split move.\\
We use the notation from Definition~\ref{split-def} and substitute~$M_1$ and~$M_1'$ for~$M$ and~$M'$,
respectively. There are two preconditions that may not hold for~$M_2$:
\begin{enumerate}
\item
the diagram~$M_2$
may have a mirror at the snip point, which is~$(\theta_1,\varphi_0)$,
\item
the diagram~$M_2$ may have occupied levels in~$\mathbb S^1\times[\varphi_1;\varphi_2]$
other than~$\ell_{\varphi_0}$.
\end{enumerate}
An obstacle of the first kind is removed by a small alteration of the snip point. An obstacle of the second kind
is removed by applying a few jump moves that shift all occupied levels contained in~$\mathbb S^1\times[\varphi_1;\varphi_2]$
upward or downward. The positions of occupied levels of~$M_1$ need not be altered for that.

When obstacles are removed we can apply a split move~$M_2\mapsto M_2'$ so that all mirrors of
the subdiagram~$M_1$ have the same successors in~$M_2'$ as in~$M_1'$.

\medskip\noindent\emph{Case 4}: $M_1\mapsto M_1'$ is a type~I merge move.\\
We again use the notation from Definition~\ref{split-def}, but now substitute~$M_1'$ and~$M_1$ for~$M$ and~$M'$,
respectively.
If there are some mirrors of~$M_2$ in the splitting gap of the move~$M_1'\mapsto M_1$,
we apply a split move that splits the occupied level~$m_{\theta_2}$ at~$\mu'$
and has the snip point at~$(\theta_2,\varphi_3)$, where~$\varphi_3\in(\varphi_1;\varphi_2)$
is chosen so that~$M_2$ has no mirrors in~$\{\theta_2\}\times[\varphi_3;\varphi_2)$. This
split move is chosen to keep the mirrors of~$M_1$ in~$\{\theta_2\}\times[\varphi_2;\varphi_1)$ fixed.

Then we apply Lemma~\ref{rem-obst-lem} twice to remove all mirrors of~$M_2$, as well as their successors, from~$[\theta_1;\theta_2)\times[\varphi_1;\varphi_2)$, and then from~$(\theta_2;\theta_1]\times(\varphi_1;\varphi_2]$. The second application does not introduce
any new mirrors in~$[\theta_1;\theta_2)\times[\varphi_1;\varphi_2)$, but some empty horizontal occupied levels
of the obtained diagram may remain in~$\mathbb S^1\times(\varphi_1;\varphi_2)$.
These can be shifted outside this domain by jump moves. This creates
preconditions for a merge move as described in Definition~\ref{split-def}, where we now
substitute the currently obtained diagram for~$M'$ and the sought-for diagram~$M_2'$ for~$M$.

\medskip\noindent\emph{Case 5}: $M_1\mapsto M_1'$ is a jump move.\\
Clearly, the obstacles, if any, can be removed by jump moves that alter the positions of
occupied level of~$M_2$ that do not belong to~$L_{M_1}$.

\medskip\noindent\emph{Case 6}: $M_1\mapsto M_1'$ is a twist move.\\
We use the notation from Definition~\ref{twist-def}. The obstacles, if any,
can be removed by one or two double split moves with the splitting mirrors~$\mu_1$
and~$\mu_2$ (see Definition~\ref{double-split-def}).
These moves should be chosen to keep the occupied level~$\ell_{\varphi_0}$ in the diagram, and to clear the intervals~$(\theta_1;\theta_2)\times\{\varphi_0\}$ and~$(\theta_2;\theta_1)\times\{\varphi_0\}$
from the mirrors of~$M_2$.
\end{proof}

\subsection{Invariance of~$\mathscr I$}

\begin{lemm}\label{realization-invariance-lemm}
Let~$M\xmapsto\eta M'$ be a type~I elementary move of enhanced mirror diagrams,
and let~$\Pi$, $\Pi'$ be rectangular diagrams of a surface properly carried by~$M$ and~$M'$, respectively (see
Definitions~\ref{properly-carried-by-graph-def} and~\ref{properly-carried-by-mir-diag-def}).
Let also~$(\delta_+,\delta_-)$ and~$(\delta_+',\delta_-')$ be canonic dividing configurations of~$\widehat\Pi$
and~$\widehat\Pi'$, respectively, and~$\phi$ be a homeomorphism~$\widehat\Pi\rightarrow\widehat\Pi'$
such that~$(\widehat\Pi,\widehat\Pi',\phi)\in\eta$.

Suppose
that all inessential boundary circuits of~$M$ are $+$-negligible.
Then all inessential boundary circuits of~$M'$ are also $+$-negligible, and the abstract dividing sets~$\phi(\delta_+)$
and~$\delta_+'$ are equivalent.
\end{lemm}

\begin{proof}
This statement follows easily from the discussion in Subsection~\ref{divided-legendrian-subsec}.
\end{proof}

\begin{proof}[Proof of Theorem~\ref{invariant-thm}]
Since~$\widehat R$ and~$\widehat R'$ are Legendrian equivalent,
there exists a sequence
of exchange moves, type~I stabilizations, and type~I destabilizations producing~$R'$ from~$R$
(see \cite[Proposition~4.4]{ost2008}):
$$R=R_0\mapsto R_1\mapsto\ldots\mapsto R_k=R'.$$
Such a sequence can be chosen so that all stabilizations in it occur before all destabilizations. This ensures
that all the diagrams~$R_i$, $i=0,\ldots,k$, are $-$-compatible with~$F$.
Indeed, each~$R_i$ is obtained by exchange moves and type~I stabilizations either
from~$R$ or from~$R'$. Exchange moves preserve Thurston--Bennequin numbers of the
connected components of the link, whereas stabilizations may only decrease them.

Thus, to prove the theorem it suffices to assume that~$R\mapsto R'$ is either a single exchange
move, or a single type~I stabilization, or a single type~I
destabilization, and to show, in all three cases, that~$\mathscr I_{F,L,+}(R)\subset\mathscr I_{F,L,+}(R')$.

Let~$\delta\in\mathscr I_{F,L+}(R)$, and let~$(\Pi,\phi)$ be a proper $+$-realization of~$\delta$.
Let also~$M=M(\Pi)$ be the enhanced mirror diagram associated with~$\Pi$ (see Definition~\ref{ass-mir-diagr-def}).
Denote by~$C$ the simple collection of essential boundary circuits of~$M$ representing~$R$,
that is, such that~$R(C)=R$ (see Subsection~\ref{simple-circuits-subsec}).
Denote by~$M_0$ the subdiagram of~$M$ defined by the conditions: $E_{M_0}=R$,
$L_{M_0}=\{x\in L_M:x\text{ contains an edge of }R\}$.
One can see that~$C$ is a simple collection of essential boundary circuits of~$M_0$, too.

We claim that there is a sequence of moves
\begin{equation}\label{annuli-transform-seq-eq}
M_0\mapsto M_1\mapsto\ldots\mapsto M_k
\end{equation}
including only the moves listed in the assumption of Proposition~\ref{subdiagram-move-prop},
such that, for the induced transformations~$C_0=C\mapsto C_1\mapsto\ldots\mapsto C_k$,
we have~$R(C_k)=R'$.

We take this claim for granted for the moment and deduce the assertion~$\delta\in\mathscr I_{F,L,+}(R')$.

Let~$C'=C_k$.
Since~$M_0$ is a subdiagram of~$M$, and~$C\subset\partial M\cap\partial M_0$,
by Proposition~\ref{subdiagram-move-prop} there exists a sequence
$M\xmapsto{\eta_1}M_1'\xmapsto{\eta_2}M_2'\xmapsto{\eta_3}\ldots\xmapsto{\eta_{k'}}M_{k'}'$
of type~I elementary moves that induces a sequence of transformations
$C\mapsto C_1'\mapsto C_2'\mapsto\ldots\mapsto C_{k'}'$ such that~$C_{k'}'=C'$.
Denote~$M_{k'}'$ by~$M'$.

By Theorem~\ref{relative-stable-equivalence-th}
the enhanced spatial ribbon graphs~$\widehat M$ and~$\widehat M'$
are stably equivalent, and, moreover, this equivalence can be realized
by an isotopy inducing the morphism~$\eta=\eta_{k'}\circ\ldots\circ\eta_2\circ\eta_1$
and bringing~$\widehat\Pi$ to a surface~$F'$ carried by~$\widehat M'$.
Since this isotopy induces the morphism~$\eta$,
it also brings~$\widehat R$ to~$\widehat R'$.

Thus, we have a homeomorphism~$\psi:\mathbb S^3\rightarrow\mathbb S^3$
such that~$\psi(\widehat R)=\widehat R'$ and~$\psi(\widehat\Pi)=F'$.
By Lemma~\ref{rectangular-representative-lem} the surface~$F'$ can be chosen
in the form~$\widehat{\Pi'}$, where~$\Pi'$ is a rectangular diagram of a surface
properly carried by~$M'$. It now follows from Lemma~\ref{realization-invariance-lemm}
that~$(\Pi',\psi\circ\phi)$ is a proper $+$-realization of~$\delta$ at~$R'$.

It remains to prove the claim about  the existence of the
sequence~\eqref{annuli-transform-seq-eq}. There are several cases and subcases to consider.

\medskip\noindent\emph{Case 1:} $R\mapsto R'$ is an exchange move. The sought-for sequence
can obviously be made of jump moves (at most three are needed).

\medskip\noindent\emph{Case 2:}  $R\mapsto R'$ is a type~I stabilization move. Due to symmetry, we may assume that a vertex~$(\theta_0,\varphi_0)$ is replaced by the vertices~$(\theta_0+\varepsilon,\varphi_0)$,
$(\theta_0,\varphi_0+\varepsilon)$, and~$(\theta_0+\varepsilon,\varphi_0+\varepsilon)$, with a small~$\varepsilon$.
First, we perform an extension move that creates a mirror at~$(\theta_0+\varepsilon,\varphi_0)$,
and then a split move that splits the occupied level~$\ell_{\varphi_0}$ at~$(\theta_0+\varepsilon,\varphi_0)$
and moves the mirror at~$(\theta_0,\varphi_0)$ to~$(\theta_0,\varphi_0+\varepsilon)$; see
Figure~\ref{stabilization-via-mirdiagr-fig}.
\begin{figure}[ht]
\includegraphics{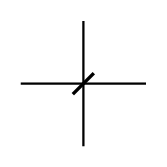}
\quad\raisebox{38pt}{$\longrightarrow$}\quad
\includegraphics{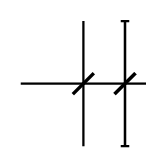}
\quad\raisebox{38pt}{$\longrightarrow$}\quad
\includegraphics{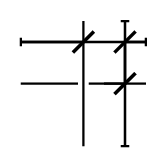}

\includegraphics{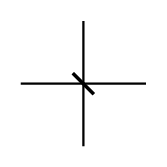}
\quad\raisebox{38pt}{$\longrightarrow$}\quad
\includegraphics{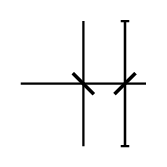}
\quad\raisebox{38pt}{$\longrightarrow$}\quad
\includegraphics{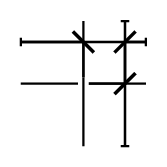}

\caption{Doing a type~I stabilization of a boundary circuit by means of
type~I moves of a mirror diagram}\label{stabilization-via-mirdiagr-fig}
\end{figure}
There are two boundary circuits modified by these moves, and both undergo a stabilization as the result
of these operations.

\medskip\noindent\emph{Case 3:}  $R\mapsto R'$ is a type~I destabilization move. Due to symmetry, we may assume that, for some~$\theta_0,\varphi_0$, and a small~$\varepsilon$,
the vertices~$(\theta_0+\varepsilon,\varphi_0)$,
$(\theta_0,\varphi_0+\varepsilon)$, and~$(\theta_0+\varepsilon,\varphi_0+\varepsilon)$
are replaced by a vertex~$(\theta_0,\varphi_0)$.
Let~$c\in C$ be the boundary circuit of~$M$ on which the destabilization occurs.
Since the destabilization makes the Thurston--Bennequin number of the knot~$\widehat c$
larger, and after the destabilization the diagram~$R'$ is still $-$-compatible with~$F$,
the relative Thurston--Bennequin number~$\tb_-(\widehat c,\widehat\Pi)$,
which is the same thing as~$\tb_-(c)$, is strictly negative. Therefore~$c$ hits $\diagup$-mirrors at least twice.

By applying twist moves to~$M_0$ we can permute $\diagup$- and $\diagdown$-mirrors arbitrarily
along a connected component of~$R$. In particular, we can achieve that the mirrors at~$(\theta_0+\varepsilon,\varphi_0)$
and~$(\theta_0+\varepsilon,\varphi_0+\varepsilon)$ become $\diagup$-mirrors. After that we can
apply the moves from the decompositions in Case~2, in the reverse order.
\end{proof}

\section{Commutation theorems}\label{commutation-sec}
\subsection{Formulation and the strategy of the proof}\label{strategy-subsec}
In this section we formulate and prove the key technical result of this paper.
We start from a warm-up version of it. Recall (see Definition~\ref{ass-spatial-rib-graph-def}) that the connectedness
of a mirror diagram is defined in terms of the corresponding spatial ribbon graph.

\begin{theo}\label{warm-up-thm}
Let~$M$ and~$M'$ be mirror diagrams
such that each connected component of~$M$ contains at least one $\diagdown$-mirror,
and each connected component of~$M'$ contains at least one $\diagup$-mirror.
Assume that~$M$ and~$M'$ are related by
a sequence of elementary moves
\begin{equation}\label{both-types-move-seq-eq}
M=M_0\xmapsto{\eta_1}M_1\xmapsto{\eta_2}\ldots
\xmapsto{\eta_n}M_n=M'.
\end{equation}
Then there exists an enhanced mirror diagram~$M''$ related to~$M$ by a sequence
\begin{equation}\label{type-i-move-seq-eq}
M=M_0'\xmapsto{\eta_1'}M_1'\xmapsto{\eta_2'}\ldots
\xmapsto{\eta_k'}M_k'=M''
\end{equation}
of type~I elementary moves,
and to~$M'$ by a sequence
\begin{equation}\label{type-ii-move-seq-eq}
M''=M_0''\xmapsto{\eta_1''} M_1''\xmapsto{\eta_2''}\ldots
\xmapsto{\eta_l''} M_l''=M'
\end{equation}
of type~II elementary moves such that
$$\eta_n\circ\ldots\circ\eta_2\circ\eta_1=\eta_l''\circ\ldots\circ\eta_2''\circ\eta_1''\circ
\eta_k'\circ\ldots\circ\eta_2'\circ\eta_1'.$$
\end{theo}

The whole of this section is devoted to the proof of this theorem and Theorem~\ref{commutation-2-thm}, which
is more elaborate. Before formulating the latter and proceeding with the actual proof we
make a few remarks about the ideas, and provide the vague outline.

If the sequence~\eqref{both-types-move-seq-eq} consists of just two moves the first of which is of type~II and the
second one is of type~I, then the respective sequences~\eqref{type-i-move-seq-eq} and~\eqref{type-ii-move-seq-eq}
are found fairly easily even without additional assumptions on the presence of $\diagdown$-mirrors in~$M$
and $\diagup$-mirrors in~$M'$.
It is then tempting to try the following strategy in the general case: whenever
a type~II move in the sequence~\eqref{both-types-move-seq-eq} is followed by a type~I move
replace the subsequence of these two moves by a sequence in which all type~I moves occur before all type~II
moves, and keep repeating this procedure until all type~I moves occur before all type~II moves.

The problem with this approach is to show that the process ever stops, as the number of moves of both
types in the sequence may grow. To overcome this difficulty we introduce a larger family of moves,
which we call generalized type~II moves, and prove two \emph{commutation properties} for them.
\emph{The first commutation property} is a statement similar to Theorem~\ref{warm-up-thm}
in which type~II elementary moves are replaced by generalized type~II moves, and the assertion claims additionally
that the number of generalized type~II moves in the sequence~\eqref{type-ii-move-seq-eq} is
the same as that in the sequence~\eqref{both-types-move-seq-eq}. To establish the first commutation
property in the general case it suffices to do so when a single generalized type~II move is
followed by a single elementary type~I move.

\emph{The second commutation property} of generalized type~II moves is a statement that these moves admit
a decomposition into a sequence of elementary moves such that all type~I moves in it occur before all type~II moves.

So, the strategy to prove Theorem~\ref{warm-up-thm} is as follows. First, we modify the initial sequence of moves
so that it consists of elementary type~I moves and generalized type~II moves for which we establish the
two commutation properties mentioned above. Then we proceed by induction in the number of
generalized type~II moves preceding some elementary type~I moves in the sequence.
To make the induction step we take the last generalized type~II move that is immediately followed by a type~I
elementary move, use the first commutation property
to `shift' it to the end of the sequence, and then decompose into elementary moves so that all type~I moves
occur before all type~II moves in this decomposition.

This method works perfectly for Theorem~\ref{warm-up-thm}, but in order to have applications to Legendrian
knots we need a relative version of Theorem~\ref{warm-up-thm}
in which a collection~$C$ of essential boundary circuits is taken special care. It would be
ideal for the relative version of Theorem~\ref{warm-up-thm}
to assume that~$C$ is preserved by all the moves in~\eqref{both-types-move-seq-eq}
and then to assert the same about the sequences~\eqref{type-i-move-seq-eq} and~\eqref{type-ii-move-seq-eq}.
However, the required relative versions of the commutation properties of generalized type~II moves
do not always hold, so, in order to make the general scheme outlined above work, we weaken
the relative formulation by letting the sequences of moves modify~$C$ in a `delicate' way
(namely, so that~$C$ undergoes only exchange moves).
To overcome the difficulties we are also forced to include jump moves into the list of moves occurring in the discussed sequences. According to
Lemma~\ref{neutral-move-decomposition-lem} jump moves can be decomposed into elementary moves of either type, but
such a decomposition may no longer be delicate with~$C$, so, in the relative case, some jump moves have to be left as is,
not decomposed into elementary moves.

In order to formulate the relative version of the commutation theorem we need a couple more definitions.

\begin{defi}\label{c-delicate-def}
Let~$M$ and~$M'$ be enhanced mirror diagrams, and let $C\subset\partial_{\mathrm e}M$ be a
collection of essential
boundary circuits of~$M$. Let also
\begin{equation}\label{elem-move-seq-eq}
M=M_0\xmapsto{\eta_1}M_1\xmapsto{\eta_2}\ldots
\xmapsto{\eta_n}M_n=M'
\end{equation}
be a sequence of moves including elementary moves and jump moves,
and
$$C=C_0\mapsto C_1\mapsto\ldots\mapsto C_n$$
be the induced by~\eqref{elem-move-seq-eq} sequence of transformations of the boundary circuits that starts from~$C$.
The sequence of moves~\eqref{elem-move-seq-eq} is called \emph{$C$-delicate} if~$C_{i-1}=C_i$
whenever~$M_{i-1}\mapsto M_i$ is an elementary move.
A $C$-delicate sequence of moves is said to be \emph{of type~I} (respectively, \emph{of type~II})
if all the elementary moves in it are of type~I (respectively, of type~II).
\end{defi}

In other words, a $C$-delicate sequence consists of elementary moves preserving the chosen
family of essential boundary
circuits, and jump moves that are allowed to modify this family.

\begin{defi}\label{flexibility-def}
Let~$M$ be an enhanced mirror diagram, and let~$C$ be a collection of essential boundary circuits of~$M$.
The diagram~$M$ is said to be \emph{$+$-flexible relative to~$C$} (respectively, \emph{$-$-flexible relative to~$C$})
if, for any boundary circuit~$c\in\partial M$ that does not belong to~$C$, there exists a sequence
$$c_1,c_2,\ldots,c_m=c$$
of boundary circuits of~$M$ not belonging to~$C$
such that~$\tb_+(c_1)<0$ (respectively, $\tb_-(c_1)<0$), and for all~$i=1,\ldots,m$,
the boundary circuits~$c_{i-1}$ and~$c_i$ are adjacent to one another in the
sense that some mirror of~$M$ is hit by both of them.

$M$ is said to be \emph{flexible relative to~$C$} if it is both $+$-flexible and $-$-flexible relative to~$C$.
$M$ is called \emph{flexible} (respectively, \emph{$+$-flexible} or \emph{$-$-flexible}) if it is flexible relative to~$\varnothing$
(respectively, $+$-flexible or $-$-flexible relative to~$\varnothing$).
\end{defi}

The reader may be puzzled for the moment why `flexibility' refers to this strange property. The name will
be justified in Subsection~\ref{flexibility-subsec}.

The following statement is trivial.

\begin{prop}
An enhanced mirror diagram~$M$ is $+$-flexible (respectively, $-$-flexible) if and only if
every connected component of~$M$ contains a $\diagdown$-mirror (respectively, a $\diagup$-mirror).
\end{prop}

Thus, the assumption of Theorem~\ref{warm-up-thm} on the presence of $\diagdown$- and $\diagup$-mirrors
can be rephrased by requesting that~$M$ and~$M'$ be $+$-flexible and $-$-flexible,
respectively.

The following statement demonstrates that flexibility is not
a very restrictive condition, and it holds in the most important case,
which motivated all this machinery.

\begin{prop}\label{simple-essential-boundary-implies-flexibility-prop}
If the essential boundary of an enhanced mirror diagram~$M$
is simple, and no connected component of~$\Gamma_{\widehat M}$ is a tree,
then~$M$ is flexible relative to~$C$ for any~$C\subset\partial_{\mathrm e}M$.
\end{prop}

\begin{proof}
Let~$c$ be  an inessential boundary circuit of~$M$, and let~$(\mathbb D^2,\phi)$
be a patching disc for~$\widehat c$. Since the connected component of~$\Gamma_{\widehat M}$
containing~$\widehat c$ is not a tree, the image~$\phi(\mathbb D^2)$ is not a sphere.
Therefore, $\phi$ can be disturbed slightly to become and embedding so
that the image~$\phi(\mathbb D^2)$ becomes a disc with Legendrian boundary having
relative Thurston--Bennequin number equal to~$\tb_+(c)$.

It follows from Bennequin's theorem on the non-existence of overtwisted discs
that~$\tb_+(c)<0$. 
Similarly, we have~$\tb_-(c)<0$ for any inessential boundary circuit of~$M$.

If the essential boundary is simple,
then any boundary circuit is either inessential or adjacent to an inessential one. The claim follows.
\end{proof}

Here is the relative version of the commutation theorem.

\begin{theo}\label{commutation-2-thm}
Let~$M$ and~$M'$ be enhanced mirror diagrams related by
a $C$-delicate sequence of moves
$$M=M_0\xmapsto{\eta_1}M_1\xmapsto{\eta_2}\ldots
\xmapsto{\eta_n}M_n=M',$$
where~$C$ is a collection of essential boundary circuits of~$M$.
Let~$C'$ be the respective collection of essential boundary circuits of~$M'$.
Assume that~$M$ is $+$-flexible relative to~$C$,
and~$M'$ is $-$-flexible relative to~$C'$.

Then there exists an enhanced mirror diagram~$M''$ related to~$M$ by a type~I $C$-delicate sequence of moves
\begin{equation}\label{del-type-i-move-seq-eq}
M=M_0'\xmapsto{\eta_1'}M_1'\xmapsto{\eta_2'}\ldots
\xmapsto{\eta_k'}M_k'=M''
\end{equation}
and to~$M'$ by a type~II $C''$-delicate sequence of moves
$$M''=M_0''\xmapsto{\eta_1''} M_1''\xmapsto{\eta_2''}\ldots
\xmapsto{\eta_l''} M_l''=M',$$
where~$C''$ is the collection of essential boundary circuits of~$M''$
which~$C$ is transformed to by~\eqref{del-type-i-move-seq-eq},
such that
$$\eta_n\circ\ldots\circ\eta_2\circ\eta_1=\eta_l''\circ\ldots\circ\eta_2''\circ\eta_1''\circ
\eta_k'\circ\ldots\circ\eta_2'\circ\eta_1'.$$
\end{theo}

Before proceeding with the proof of Theorems~\ref{warm-up-thm} and~\ref{commutation-2-thm}
we demonstrate that the flexibility assumption is essential in each of them.

\begin{exam}\label{counter-example1-ex}
Let~$M$ be an enhanced mirror diagram consisting of a single occupied level without any mirrors, and let~$M'$ be
the mirror diagram shown in Figure~\ref{counter-example1-fig}.
\begin{figure}[ht]
\includegraphics{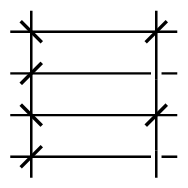}
\caption{The diagram~$M'$ in Example~\ref{counter-example1-ex}}\label{counter-example1-fig}
\end{figure}
All boundary circuits of both diagrams are inessential.
These two diagrams are related by elementary moves. Indeed, there are two coherent $\diagdown$-mirrors in~$M'$,
one of which can be eliminated by means of a sequence of
three type~I elementary moves
(see Lemma~\ref{remove-coherent-mirror-lem}). The respective spatial graph will become a tree. So,
all the mirrors now can be reduced by elimination moves.

We claim that it is impossible to get~$M'$ from~$M$ by a sequence of elementary moves
such that all type~I moves in it precede all type~II moves. This can be seen as follows.
Let $M''$ be obtained from~$M$ by type~I elementary moves.
Since $\Gamma_{\widehat M}$ is a single point, the spatial graph~$\Gamma_{\widehat M''}$ is a tree.
The diagram~$M''$ cannot have any $\diagdown$-mirrors, hence no two $\diagup$-mirrors
of~$M''$ are coherent.

The graph~$\Gamma_{\widehat M'}$ is not simply connected, and no two $\diagup$-mirrors
of~$M'$ are coherent. It follows from Proposition~\ref{equivalence-of-divided-graphs-prop}
that~$\rho^-(M')$ and~$\rho^-(M'')$ viewed as divided spatial ribbon graphs
are not stably equivalent. By Theorem~\ref{type-i-moves-meaning-th}~(ii)
the diagrams~$M''$ and~$M'$ are not related by type~II elementary moves.
\end{exam}

\begin{exam}\label{counter-example2-ex}
This example is similar in nature to the previous one, but slightly more complicated.
Consider the mirror diagrams~$M$ and~$M'$ shown in~Figure~\ref{counter-example2-fig}.
Again, $M$ can be transformed to~$M'$ by means of
a sequence of elementary moves, but any such sequence have a type~II move preceding a type~I move.
\begin{figure}[ht]
\includegraphics{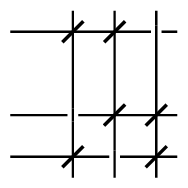}\put(-50,-10){$M$}
\hskip2cm
\includegraphics{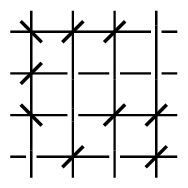}\put(-50,-10){$M'$}
\caption{The diagrams in Example~\ref{counter-example2-ex}}\label{counter-example2-fig}
\end{figure}
The details are left to the reader.
\end{exam}

We now proceed with the constructions needed to prove Theorems~\ref{warm-up-thm} and~\ref{commutation-2-thm}.

\subsection{Splitting routes and splitting paths}

For~$x$ an occupied level of a mirror diagram~$M$ and~$p$ a point at~$x$
distinct from all mirrors of~$M$, we denote by~$\wideparen p$ the point at~$\partial\wideparen x$
defined by~$\varphi(\wideparen p)=\varphi(p)$ if~$x$ is a meridian and~$\theta(\wideparen p)=\theta(p)$
if~$x$ is a longitude of~$\mathbb T^2$. By using this notation we assume that the surface~$\wideparen M$
is chosen slim
enough, so that~$\wideparen p$ is outside of all the strips~$\wideparen\mu$, $\mu\in E_M$.
This means, in particular, that~$\wideparen p$ lies on~$\partial\wideparen M$.

A sequence~$(\mu_1,\ldots,\mu_m)$ of mirrors of a mirror diagram will be called \emph{cancellable}
if there is a balanced parenthesis sequence of length~$m$ such that whenever the $i$th and the~$j$th
parentheses in it are matched we have~$\mu_i=\mu_j$. (In particular, this means that~$m$ is even.)
For instance, $(\mu,\mu',\mu',\mu'',\mu''',\mu''',\mu'',\mu)$
is cancellable for any mirrors~$\mu,\mu',\mu'',\mu'''$.

\begin{defi}\label{simple-route-def}
By \emph{a single-headed type~II splitting path} in the surface~$\wideparen M$ we mean
an oriented simple normal arc~$\wideparen\omega\subset\wideparen M$ such that the following holds:
\begin{enumerate}
\item
$\wideparen\omega$ starts at an interior point of~$\wideparen M$
and arrives at a point on~$\partial\wideparen M$, that is, a point of the form~$\wideparen p$, where~$p\in y\in L_M$, $p\notin E_M$;
\item
$\wideparen\omega$ can be cut into subarcs~$\wideparen\omega^1,\ldots,\wideparen\omega^k$
(following in~$\wideparen\omega$ in this order)
so that
\begin{enumerate}
\item
for all~$2\leqslant i\leqslant k-1$
each endpoint of~$\wideparen\omega^i$ lies in the interior of a disc of the form~$\wideparen x$, $x\in L_M$;
\item
each~$\wideparen\omega^i$, $i=1,\ldots,k$, intersects a single strip~$\wideparen\mu_i$, $\mu_i\in E_M$;
\item
$\mu_1$ is a $\diagdown$-mirror of~$M$, all the other~$\mu_i$'s are $\diagup$-mirrors;
\item
whenever~$i,i',j,j'\in\{2,\ldots,k\}$ are such that $i'>i$, $j'>j$, $\mu_i=\mu_{i'}=\mu_j=\mu_{j'}$
and the sequences~$(\mu_i,\mu_{i+1},\ldots,\mu_{i'})$ and~$(\mu_j,\mu_{j+1},\ldots,\mu_{j'})$
are cancellable, we have~$i\equiv j\pmod2$.
\end{enumerate}
\end{enumerate}
The sequence~$\omega=(\mu_1,\ldots,\mu_k,p)$ arising in this way is referred to as \emph{a single-headed type~II splitting route} in~$M$,
and the path~$\wideparen\omega$ is said to be \emph{associated with~$\omega$}.

If~$k=1$, then~$\omega$ is called \emph{an elementary type~II splitting route}, and~$\wideparen\omega$
\emph{an elementary type~II splitting path}.

\emph{A single-headed (or elementary) type~I splitting path} and \emph{a single-headed (or elementary) type~I splitting route} are defined
similarly, with the roles of $\diagup$-mirrors and $\diagdown$-mirrors exchanged.
\end{defi}
\begin{figure}[ht]
\includegraphics{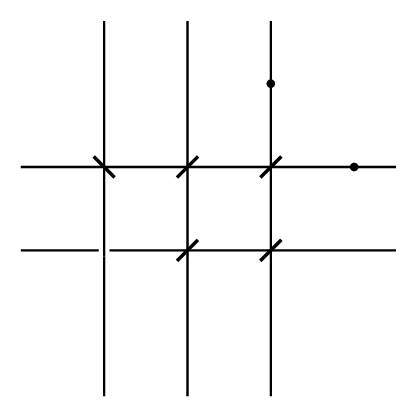}\put(-148,15){$m_1$}\put(-108,15){$m_2$}\put(-68,15){$m_3$}%
\put(-20,85){$\ell_1$}\put(-20,125){$\ell_2$}\put(-148,125){$\mu_1$}\put(-89,125){$\mu_{2,6}$}%
\put(-83,85){$\mu_3$}\put(-123,85){$\mu_4$}\put(-123,125){$\mu_5$}\put(-66,160){$p$}\put(-32,112){$q$}
\hskip.5cm
\raisebox{20pt}{\includegraphics{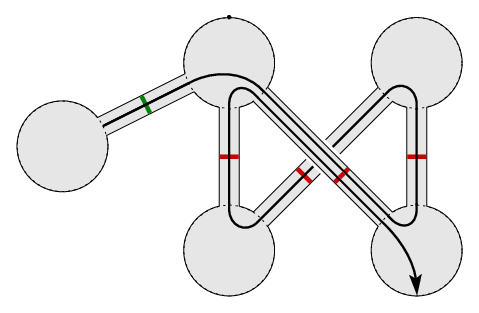}\put(-205,76){$\wideparen m_1$}%
\put(-123,122){$\wideparen\ell_2$}\put(-32,122){$\wideparen\ell_1$}\put(-126,24){$\wideparen m_2$}%
\put(-25,30){$\wideparen m_3$}\put(-33,0){$\wideparen p$}\put(-170,110){$\wideparen\mu_1$}%
\put(-138,72){$\wideparen\mu_5$}\put(-90,100){$\wideparen\mu_{2,6}$}\put(-90,45){$\wideparen\mu_4$}%
\put(-23,72){$\wideparen\mu_3$}\put(-123,147){$\wideparen q$}}
\caption{A single-headed splitting path associated with the splitting route~$(\mu_1,\mu_2,\mu_3,\mu_4,\mu_5,$ $\mu_6=\mu_2,p)$}\label{splitting-route-exam-fig}
\end{figure}
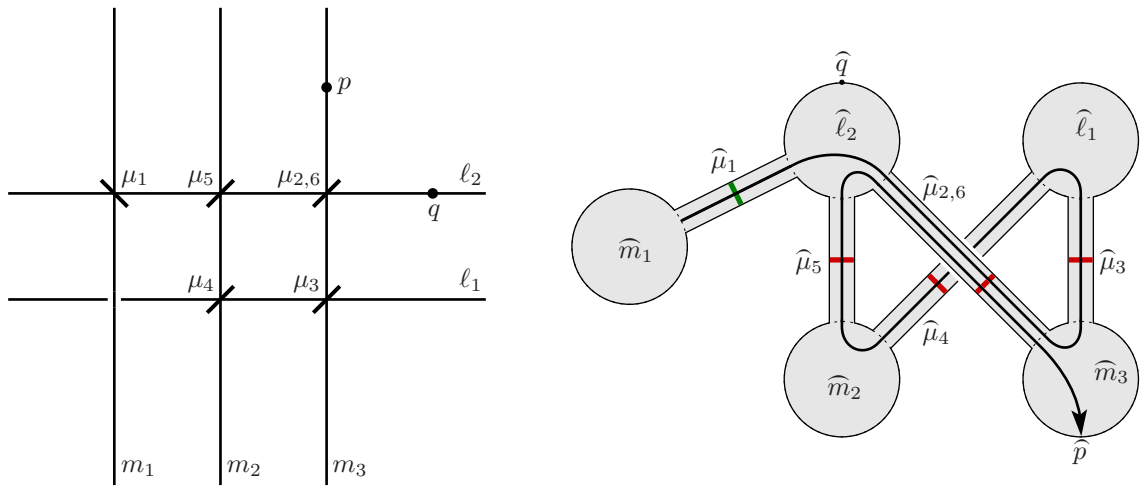

\begin{exam}\label{splitting-route-exam}
Shown in Figure~\ref{splitting-route-exam-fig} is a mirror diagram~$M$ and a single-headed splitting path in~$\wideparen M$
associated with the single-headed type~II splitting route~$(\mu_1,\mu_2,\mu_3,\mu_4,\mu_5,\mu_6=\mu_2,p)$.
Observe that the sequence~$(\mu_1,\mu_2,\mu_3,\mu_4,\mu_5,q)$ is not a single-headed type~II splitting route as an arc crossing
the~$1$-handles in the indicated order and ending at~$\wideparen q$ must have a self-intersection in~$\wideparen\ell_2$.
\end{exam}

Note that a~$\diagup$-mirror may appear in a single-headed type~II splitting route arbitrarily many times but at most two times in a row.
This constraint is a part of Condition~(2d) in Definition~\ref{simple-route-def}.

Condition~(2d), which looks somewhat artificial, is added
in order for any single-headed type~II splitting route to define the associated single-headed splitting path uniquely
up to isotopy in the class of normal arcs. If this condition is omitted, then the
path~$\wideparen\omega$ may be defined by~$\omega$ too loosely, which, in turn, would complicate
the exposition of some arguments below.

Condition~(2d) prohibits, in particular, having~$\mu_i=\mu_{i+1}=\mu_{i+2}$ or~$\mu_i=\mu_{i+1}=\mu_{i+2m-1}=\mu_{i+2m}$ with~$m>1$.
The issues that are avoided by
these prohibitions
are illustrated in Figures~\ref{no-split-fig1} and~\ref{no-split-fig2}, where
pairs of fragments of normal arcs are shown such that, for each pair, replacing one fragment of a normal arc by the other
does not change the combinatorial presentation of the path by a sequence of mirrors.

\begin{figure}[ht]
\includegraphics[scale=0.7]{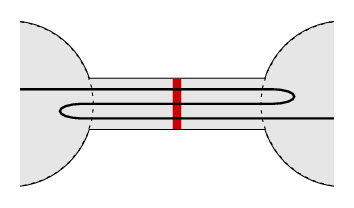}\hskip2cm
\includegraphics[scale=0.7]{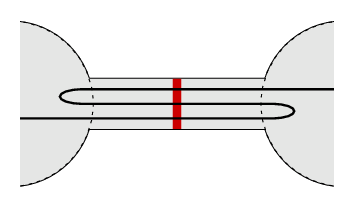}
\caption{Non-uniqueness of~$\wideparen\omega$ if we allow~$\mu_i=\mu_{i+1}=\mu_{i+2}$}\label{no-split-fig1}
\end{figure}

\begin{figure}[ht]
\includegraphics[scale=0.7]{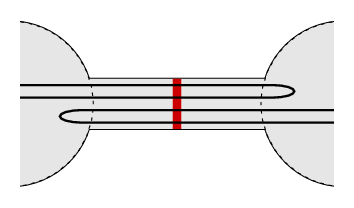}\hskip2cm
\includegraphics[scale=0.7]{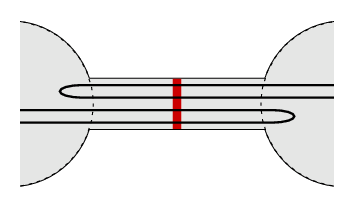}
\caption{Non-uniqueness of~$\wideparen\omega$ if we allow~$\mu_i=\mu_{i+1}=\mu_{i+2m-1}=\mu_{i+2m}$, $m>1$}\label{no-split-fig2}
\end{figure}

The general form of Condition~(2d)
allows to avoid similar issues with longer pieces of~$\wideparen\omega$, for instance,
with sequences of the form
$$(\ldots,\mu,\nu,\nu,\mu,\mu,\nu\ldots),\quad(\ldots,\mu,\nu,\nu,\mu,\ldots,\nu,\mu,\mu,\nu,\ldots),$$
and also guarantees that such pieces do not appear if a cancellable portion of a splitting route
is removed.

Note that the equalities~$\mu_i=\mu_{i+1}=\mu_j=\mu_{j+1}$ on their own cannot cause the problem
shown in Figure~\ref{no-split-fig2} if~$j-i$ is even and is larger than~$2$ in absolute value. Indeed, for $l=1,\ldots,k-1$, denote by~$x_l$ the
occupied level of~$M$ containing~$\mu_l$ and~$\mu_{l+1}$. If~$i$ and~$j$ are
of the same parity, then the occupied levels~$x_i$ and~$x_j$ are parallel. Therefore, the equality~$\mu_i=\mu_j$
implies~$x_i=x_j$, and hence,
the `tongues'~$\wideparen\omega^i\cup\wideparen\omega^{i+1}$ and~$\wideparen\omega^j\cup\wideparen\omega^{j+1}$
are directed coherently unlike the ones in Figure~\ref{no-split-fig2}.
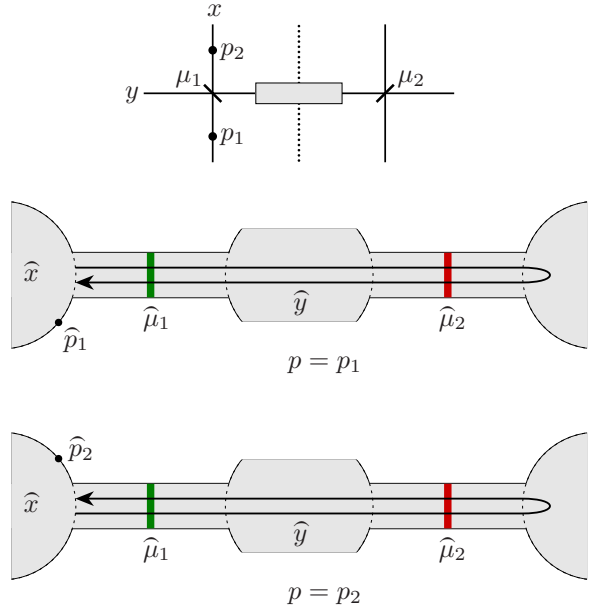
\begin{figure}[ht]
\includegraphics[scale=.65]{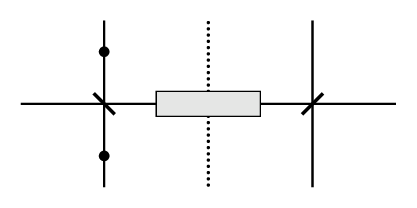}\put(-100,62){$x$}\put(-131,31){$y$}%
\put(-95,47){$p_2$}\put(-95,15){$p_1$}\put(-112,37){$\mu_1$}\put(-28,37){$\mu_2$}\\
\includegraphics[scale=.7]{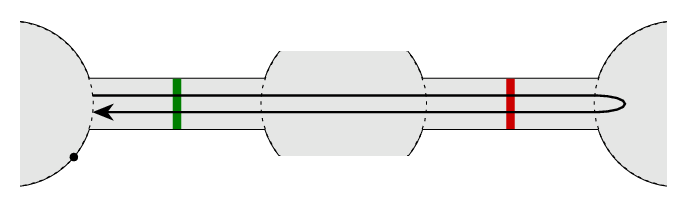}\put(-118,22){$\wideparen y$}\put(-220,33){$\wideparen x$}\put(-205,8){$\wideparen p_1$}%
\put(-175,16){$\wideparen\mu_1$}\put(-63,16){$\wideparen\mu_2$}\put(-120,0){$p=p_1$}
\\\vskip.5cm
\includegraphics[scale=.7]{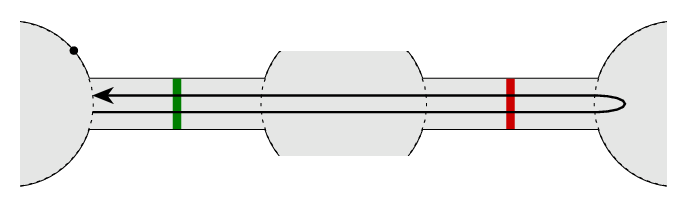}\put(-118,22){$\wideparen y$}\put(-220,33){$\wideparen x$}\put(-203,54){$\wideparen p_2$}%
\put(-175,16){$\wideparen\mu_1$}\put(-63,16){$\wideparen\mu_2$}\put(-120,0){$p=p_2$}
\caption{Dependence of~$\wideparen\omega$ on~$p$ in the case of a double-headed
splitting route~$\omega=(\mu_1,\mu_2,\mu_2,\mu_1,p)$}\label{path-depend-on-p3-fig}
\end{figure}

\begin{defi}\label{double-headed-splitting-route-def}
By \emph{a double-headed type~II splitting path} in the surface~$\wideparen M$ we mean
an oriented normal simple arc~$\wideparen\omega\subset\wideparen M$ such that the following holds:
\begin{enumerate}
\item
both endpoints of $\wideparen\omega$ are interior points of~$\wideparen M$
\item
$\wideparen\omega$ can be cut into subarcs~$\wideparen\omega^1,\ldots,\wideparen\omega^k$
(following in~$\wideparen\omega$ in this order), $k\geqslant2$,
so that
\begin{enumerate}
\item
for all~$2\leqslant i\leqslant k-1$,
each endpoint of~$\wideparen\omega^i$ lies in the interior of a disc of the form~$\wideparen x$, $x\in L_M$;
\item
each~$\wideparen\omega^i$, $i=1,\ldots,k$, intersects a single strip~$\wideparen\mu_i$, $\mu_i\in E_M$;
\item
$\mu_1$ and~$\mu_k$ are $\diagdown$-mirrors of~$M$, all the other~$\mu_i$'s are $\diagup$-mirrors;
\item
whenever~$i,i',j,j'\in\{2,\ldots,k-1\}$ are such that $i'>i$, $j'>j$, $\mu_i=\mu_{i'}=\mu_j=\mu_{j'}$
and the sequences~$(\mu_i,\mu_{i+1},\ldots,\mu_{i'})$ and~$(\mu_j,\mu_{j+1},\ldots,\mu_{j'})$
are cancellable, we have~$i\equiv j\pmod2$.
\end{enumerate}
\end{enumerate}
To encode~$\wideparen\omega$ we use a sequence~$\omega=(\mu_1,\ldots,\mu_k,p)$
in which~$p\in\mathbb T^2$ satisfies the following conditions:
\begin{enumerate}
\item[(i)]
$p$ lies on the occupied level~$x_0$ of~$M$ characterized by~$z\in\partial\wideparen x_0$,
where~$z$ is the terminal point of~$\wideparen\omega$;
\item[(ii)]
$p\notin E_M$;
\item[(iii)]
the shorter of the two open arcs into
which~$z$ and~$\wideparen p$ cut the circle~$\partial\wideparen x_0$
is disjoint from~$\wideparen\omega$ and from the strips of the form~$\wideparen\mu$, $\mu\in E_M\setminus\{\mu_k\}$.
\end{enumerate}
The sequence~$\omega$ is then referred to as \emph{a double-headed type~II splitting route in~$M$},
and the path~$\wideparen\omega$ is said to be \emph{associated with~$\omega$}.

A double-headed type~II splitting route is called \emph{special} if it has
the form~$(\mu_1,\mu_2,p)$.

By \emph{a type~II splitting route} or \emph{path} we mean either a single-headed or a double-headed type~II splitting route or path,
respectively.
\end{defi}

\begin{figure}[ht]
\includegraphics[scale=.65]{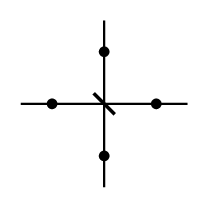}\put(-35,62){$x$}\put(-66,31){$y$}\put(-19,25){$p_1$}\put(-51,25){$p_3$}%
\put(-30,47){$p_2$}\put(-30,15){$p_4$}\put(-42,37){$\mu$}\\
\includegraphics[scale=.7]{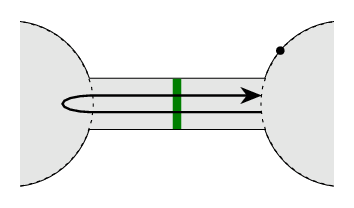}\put(-63,16){$\wideparen\mu$}\put(-17,33){$\wideparen y$}\put(-108,33){$\wideparen x$}
\put(-36,54){$\wideparen p_1$}\put(-88,-5){$\omega=(\mu,\mu,p_1)$}\hskip2cm
\includegraphics[scale=.7]{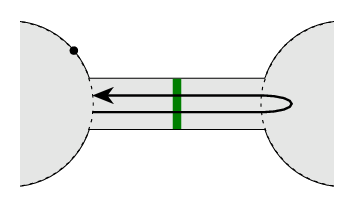}\put(-63,16){$\wideparen\mu$}\put(-17,33){$\wideparen y$}\put(-108,33){$\wideparen x$}
\put(-91,54){$\wideparen p_2$}\put(-88,-5){$\omega=(\mu,\mu,p_2)$}
\\\vskip.5cm
\includegraphics[scale=.7]{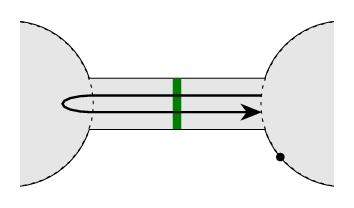}\put(-63,16){$\wideparen\mu$}\put(-17,33){$\wideparen y$}\put(-108,33){$\wideparen x$}
\put(-32,8){$\wideparen p_3$}\put(-88,-5){$\omega=(\mu,\mu,p_3)$}\hskip2cm
\includegraphics[scale=.7]{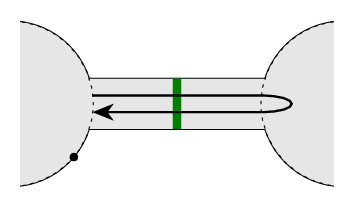}\put(-63,16){$\wideparen\mu$}\put(-17,33){$\wideparen y$}\put(-108,33){$\wideparen x$}
\put(-93,8){$\wideparen p_4$}\put(-88,-5){$\omega=(\mu,\mu,p_4)$}
\caption{Dependence of~$\wideparen\omega$ on~$p$ in the case of a special splitting route~$\omega$
with~$\mu_1=\mu_2$}\label{path-depend-on-p1-fig}
\end{figure}

\begin{figure}
\includegraphics[scale=.65]{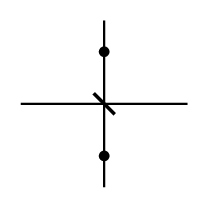}\put(-35,62){$x$}\put(-66,31){$y$}%
\put(-30,47){$p_2$}\put(-30,15){$p_1$}\put(-42,37){$\mu$}\\
\includegraphics[scale=.7]{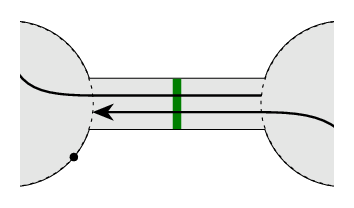}\put(-93,8){$\wideparen p_1$}%
\put(-17,38){$\wideparen y$}\put(-108,28){$\wideparen x$}\put(-63,16){$\wideparen\mu$}\put(-72,-5){$p=p_1$}
\hskip2cm
\includegraphics[scale=.7]{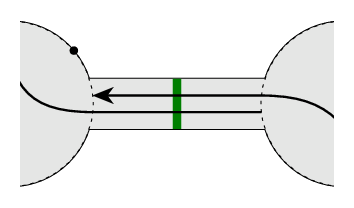}\put(-91,54){$\wideparen p_2$}%
\put(-17,45){$\wideparen y$}\put(-108,21){$\wideparen x$}\put(-63,16){$\wideparen\mu$}\put(-72,-5){$p=p_2$}
\caption{Dependence of~$\wideparen\omega$ on~$p$ in the case of a double-headed splitting route~$\omega=(\mu,\mu_2,\mu_3,
\ldots,\mu_{k-1},\mu,p)$}\label{path-depend-on-p2-fig}
\end{figure}
Let~$\omega=(\mu_1,\ldots,\mu_k,p)$ be a double-headed type~II splitting route. If~$\mu_1\ne\mu_k$,
then the last entry in it, the point~$p$, carries no additional information. Neither does it if~$k\equiv0\,(\mathrm{mod}\,2)$
and we have~$\mu_i\ne\mu_{k+1-i}$ for some~$i\in\{2,\ldots,k/2\}$.
If $k\equiv0\,(\mathrm{mod}\,2)$, $k>2$, and~$\mu_i=\mu_{k+1-i}$ for all~$i=1,\ldots,k/2$,
then~$p$ is important only for specifying the orientation of~$\wideparen\omega$;
see Figure~\ref{path-depend-on-p3-fig}.
If~$\mu_1=\mu_k$ and~$k$ is either odd (it should then be greater than four) or equal to two, then even the isotopy
class of the unoriented arc underlying~$\wideparen\omega$
depends on~$p$; see Figures~\ref{path-depend-on-p1-fig} and~\ref{path-depend-on-p2-fig}.

Another reason to keep~$p$ in the notation for a double-headed splitting routes
is the similarity with the notation for single-headed splitting routes, which sometimes
allows for shorter formulations.

\begin{prop}
If~$\omega$ is a type~II splitting route in a mirror diagram~$M$, then the associated type~II splitting
path~$\wideparen\omega$ in~$\wideparen M$ is defined uniquely up to isotopy in the class of normal
arcs.
\end{prop}

We skip the easy proof.

\begin{defi}
Two type~II (or type~I) splitting routes~$\omega$ and~$\omega'$ are said to be \emph{equivalent} if the
respective splitting paths~$\wideparen\omega$ and~$\wideparen\omega'$ are isotopic
in the class of normal arcs.
\end{defi}

Clearly, it two splitting routes~$\omega=(\mu_1,\ldots,\mu_k,p)$ and~$\omega'=(\mu_1',\ldots,\mu_{k'}',p')$
are equivalent, then~$k'=k$ and $\mu_i=\mu_i'$ for~$1\leqslant i\leqslant k$, so the only difference
that can occur in equivalent splitting routes is~$p\ne p'$.

\subsection{Generalized split and generalized merge moves}
For proving Theorems~\ref{warm-up-thm} and~\ref{commutation-2-thm} we need
to generalize merge moves. However, it is more intuitive to speak
in terms of the inverse operations, which are generalized split moves and which we now define.

\begin{defi}\label{gen-split-def}
Let~$M$ be an enhanced mirror diagram, and let~$\omega=(\mu_1,\mu_2,\ldots,\mu_k,p)$ be a type~II splitting route in~$M$.
\emph{A generalized type~II split move~$M\xmapsto\eta M'$
associated with~$\omega$} is defined inductively as follows.

\smallskip\noindent\emph{Case 1}: $k=1$.\\$M\xmapsto\eta M'$ is a type~II split move for which~$\mu_1$
is the splitting mirror and~$p$ is the snip point.

\smallskip\noindent\emph{Case 2}: $k>1$, $\omega$ is a single-headed splitting route.\\
Let~$M\xmapsto{\eta_1}M_1$ be a type~I split move for which~$\mu_k$
is the splitting mirror and~$p$ the snip point, and
let~$M_1\xmapsto{\eta_2}M'$ be a generalized type~II split move
associated with a type~II splitting route~$\omega'=(\mu_1',\mu_2',\ldots,\mu_{k-1}',p')$
in which~$\mu_i'$ is a successor of~$\mu_i$, $i=1,\ldots,k-1$, and~$p'$
lies in the splitting gap of
the move~$M\xmapsto{\eta_1}M_1$.
Then~$M\xmapsto\eta M'$ with~$\eta=\eta_2\circ\eta_1$
is a generalized type~II split move
associated with~$\omega$.

\smallskip\noindent\emph{Case 3}: $k>2$, $\omega$ is a double-headed splitting route.\\
Let~$M\xmapsto{\eta_1}M_1$ be a double split move for which~$\mu_k$ and~$\mu_{k-1}$
are the splitting mirrors (see Definition~\ref{double-split-def}).
Denote by~$\mu_k'$ and~$\mu_k''$ the successors of~$\mu_k$
ordered so that the $\diagdown$-splitting gap of the move is~$(\mu_k';\mu_k'')$
if $(\mu_k;p)$ is shorter than~$(p;\mu_k)$, and~$(\mu_k'';\mu_k')$ otherwise.

Let~$M_1\xmapsto{\eta_2}M'$ be a generalized type~II split move
associated with a single-headed type~II splitting route~$\omega'=(\mu_1',\mu_2',\ldots,\mu_{k-2}',p')$
in which~$\mu_i'$ is a successor of~$\mu_i$, $i=1,\ldots,k-2$, and~$p'$
lies in the $\diagup$-splitting gap of the move~$M\xmapsto{\eta_1}M_1$.
If~$\mu_1=\mu_k$ we demand additionally that~$\mu_1'=\mu_k'$.

Then~$M\xmapsto\eta M'$ with~$\eta=\eta_2\circ\eta_1$
is a generalized type~II split move associated with~$\omega$.

\smallskip\noindent\emph{Case 4}: $\omega=(\mu_1,\mu_2,p)$ is a special double-headed splitting route.
In this case, the generalized type~II split move~$M\xmapsto\eta M'$ defined below
is also called \emph{special}.\\
Denote by~$y$ the occupied level of~$M$ containing~$\mu_1$ and~$\mu_2$, and not containing~$p$.
Let~$M\xmapsto{\eta_1}M_1$ be a type~I extension move that adds a new mirror~$\mu_0$
on~$y$, and let~$M_1\xmapsto{\eta_2}M'$ be a generalized type~II
split move associated with the splitting route~$(\mu_1,\mu_0,\mu_0,\mu_2,p)$.
Then~$M\xmapsto\eta M'$ with~$\eta=\eta_2\circ\eta_1$
is a generalized type~II split move associated with~$\omega$.
The mirror~$\mu_0$ is referred to as \emph{the auxiliary mirror} of the move~$M\xmapsto\eta M'$.
The new occupied level added by the extension move~$M\xmapsto{\eta_1}M_1$
is called \emph{the auxiliary level} of the move~$M\xmapsto\eta M'$.

\smallskip
The inverse of a generalized type~II split move
is called \emph{a generalized type~II merge move}.
\end{defi}

One can see that generalized split moves are safe-to-bring-forward, whereas generalized merge moves are
safe-to-postpone (see Definition~\ref{safe-def}).

\begin{prop}
For any type~II splitting route~$\omega$ in an enhanced mirror diagram~$M$, there exists an associated generalized type~II
split move~$M\mapsto M'$. Unless~$\omega$ is special,
the combinatorial type of the resulting diagram~$M'$ is prescribed
by the combinatorial class~$M$ and the equivalence class of~$\omega$.
If~$\omega$ is special, then the combinatorial type of the resulting diagram~$M'$
is prescribed by~$M$, $\omega$, and the position of the auxiliary mirror.
\end{prop}

\begin{proof}
We use the notation from Definition~\ref{gen-split-def}.
We only need to prove the existence of the type~II splitting route~$\omega'$
used for the inductive step in the Cases~2 and~3 of Definition~\ref{gen-split-def}, and show that it is unique up to replacing~$p'$ by another point in the ($\diagup$-)splitting gap of the move~$M\xmapsto{\eta_1}M_1$.

One can see that the requirements on~$\omega'$ can be reformulated in topological terms
by means of the partial homeomorphisms~$h_M^{M'}$ introduced in Subsection~\ref{hMM-subsec} as follows:
the image under~$h^M_{M_1}$ of a type~II splitting path~$\wideparen\omega'$ associated with~$\omega'$
is a subarc of a type~II splitting path~$\wideparen\omega$ associated with~$\omega$.

\begin{figure}[ht]
\begin{tabular}{ccc}
\includegraphics[scale=0.65]{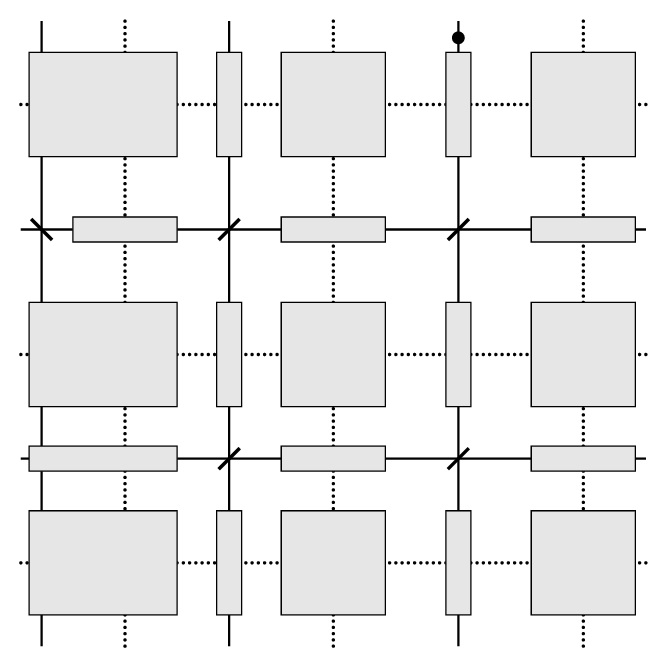}\put(-194,139){$\scriptstyle1$}\put(-63,131){$\scriptstyle2,6$}%
\put(-135,131){$\scriptstyle5$}\put(-63,59){$\scriptstyle3$}\put(-135,59){$\scriptstyle4$}
&
\raisebox{100pt}{$\longrightarrow$}
&
\includegraphics[scale=0.65]{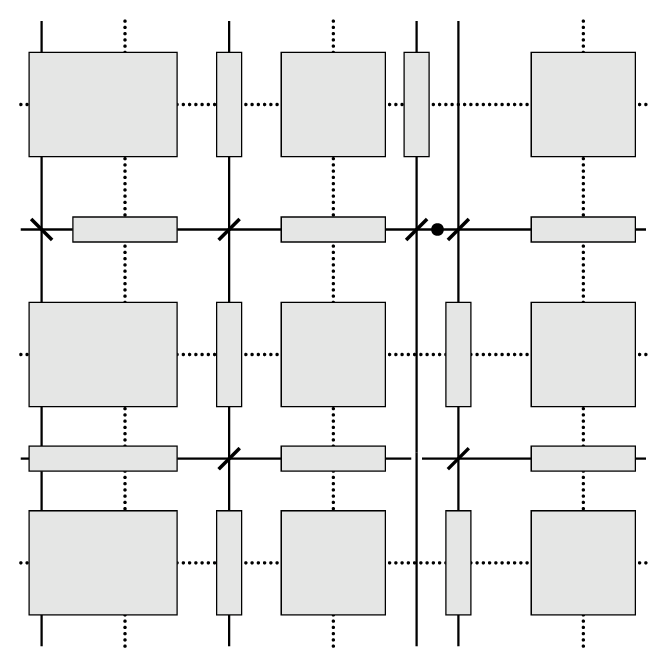}\put(-194,139){$\scriptstyle1$}\put(-63,131){$\scriptstyle2$}%
\put(-135,131){$\scriptstyle5$}\put(-63,59){$\scriptstyle3$}\put(-135,59){$\scriptstyle4$}\\
&&$\big\downarrow$\\
\includegraphics[scale=0.65]{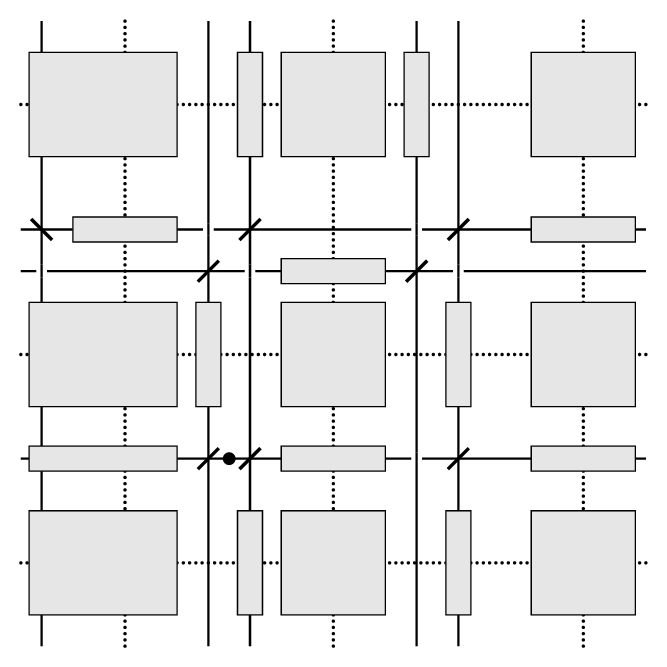}\put(-194,139){$\scriptstyle1$}\put(-63,131){$\scriptstyle2$}%
\put(-63,59){$\scriptstyle3$}
&
\raisebox{100pt}{$\longleftarrow$}
&
\includegraphics[scale=0.65]{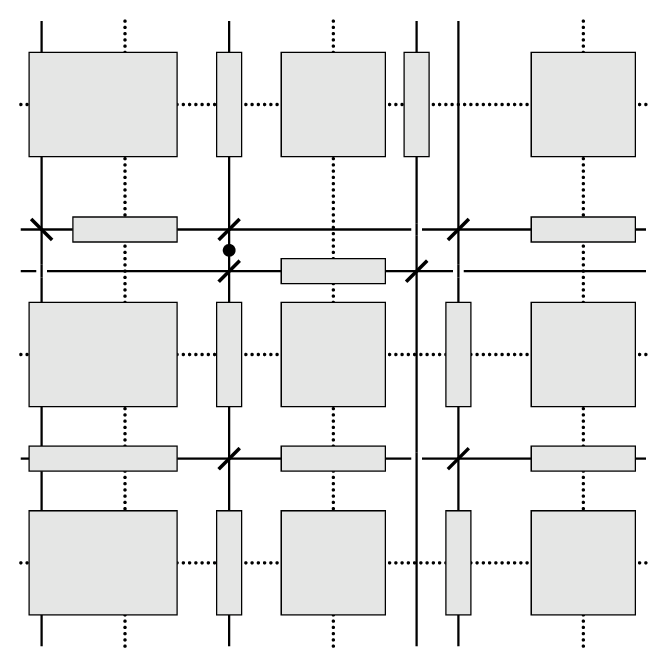}\put(-194,139){$\scriptstyle1$}\put(-63,131){$\scriptstyle2$}%
\put(-63,59){$\scriptstyle3$}\put(-135,59){$\scriptstyle4$}\\
\end{tabular}
\caption{An example of a generalized type~II splitting, part~1}\label{gen-split-pic-1}
\end{figure}
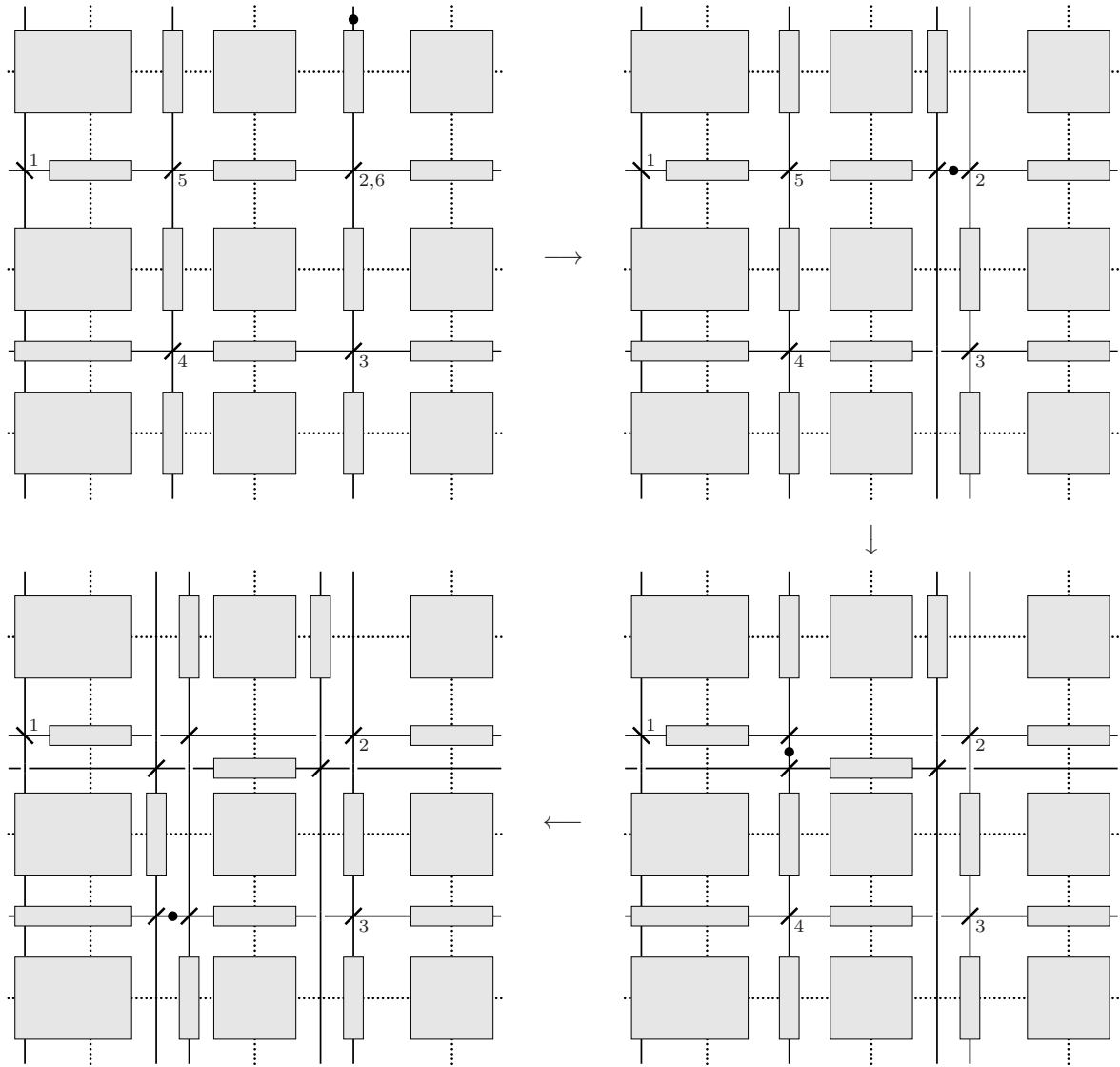
\begin{figure}[ht]
\begin{tabular}{ccc}
\includegraphics[scale=0.65]{gen-split-ex4.eps}\put(-194,139){$\scriptstyle1$}\put(-63,131){$\scriptstyle2$}%
\put(-63,59){$\scriptstyle3$}
&
\raisebox{100pt}{$\longrightarrow$}
&
\includegraphics[scale=0.65]{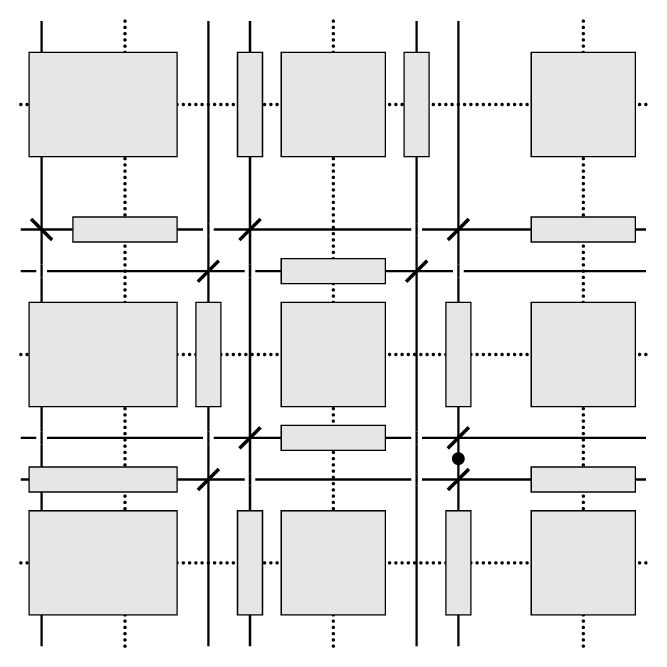}\put(-194,139){$\scriptstyle1$}\put(-63,131){$\scriptstyle2$}\\
&&$\big\downarrow$\\
\includegraphics[scale=0.65]{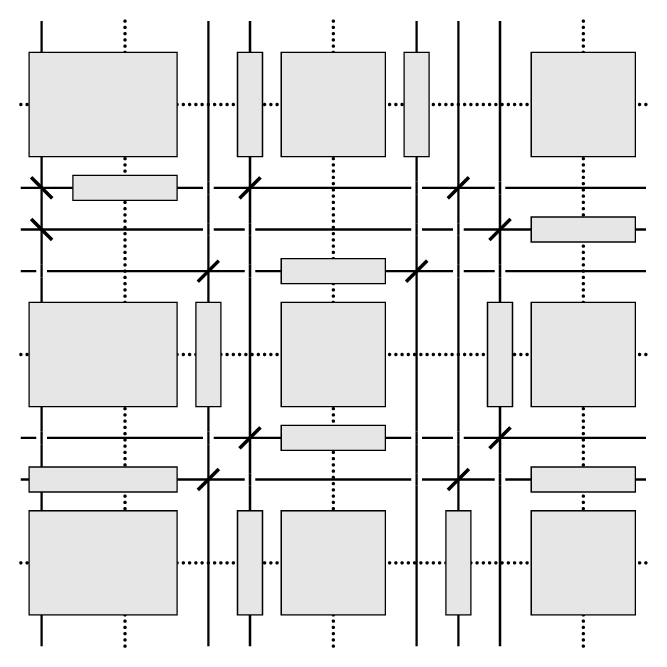}
&
\raisebox{100pt}{$\longleftarrow$}
&
\includegraphics[scale=0.65]{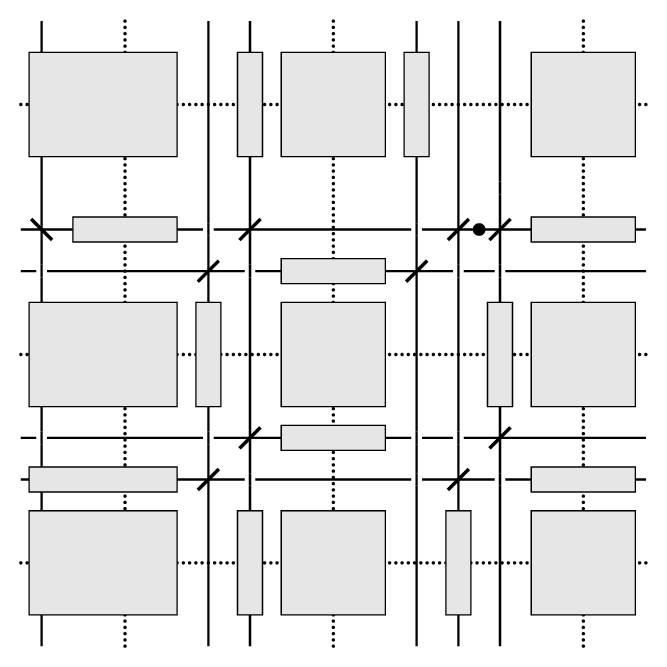}\put(-194,139){$\scriptstyle1$}\\
\end{tabular}
\caption{An example of a generalized type~II splitting, part~2}\label{gen-split-pic-2}
\end{figure}

This means that the sought-for~$\omega'$ can be obtained as follows. Take any type~II
splitting path~$\wideparen\omega$ associated with~$\omega$ and decompose it
into subarcs~$\wideparen\omega^i$, $i=1,\ldots,k$, as in Definition~\ref{simple-route-def}.
The partial homeomorphism~$h_M^{M_1}$ can be chosen so that the intersection of the domain of~$h_M^{M_1}$
with~$\wideparen\omega$ is precisely~$\wideparen\omega^1\cup\ldots\wideparen\omega^{k-1}$
if~$\omega$ is a single-headed splitting route and~$\wideparen\omega^1\cup\ldots\wideparen\omega^{k-2}$
if it is a double-headed one.
We can take for~$\wideparen\omega'$ the image~$h_M^{M_1}(\wideparen\omega)$.
One can see that any eligible choice of~$\wideparen\omega'$ fits into this construction.
This means that the choice of~$\mu_i'$, $i=1,\ldots,k-1$, in Definition~\ref{gen-split-def}
is unique.
\end{proof}

\begin{exam}\label{gen-split-move-exam}
Consider
a general mirror diagram obtained from that in Example~\ref{splitting-route-exam}
by adding more occupied levels and mirrors, and take the same
splitting route~$\omega$ as in Example~\ref{splitting-route-exam}.

Figures~\ref{gen-split-pic-1} and~\ref{gen-split-pic-2} show a sequence of ordinary splittings
arising from the definition of a generalized type~II split move associated with~$\omega$.
In each picture we mark the snip point of the forthcoming split move, and number
the mirrors of the splitting route with which the remaining part of the process (which
is by definition also a generalized type~II split move) is associated. Each mirror~$\mu_i$
and its successors in the respective splitting routes are marked by the respective number~$i$.
\end{exam}

\begin{conv}
Though a generalized split move associated with a splitting route~$\omega$ depends only on the equivalence class of~$\omega$,
by saying that~$M\mapsto M'$ is a (generalized) split move we assume that a concrete splitting route for this move has been
chosen and fixed.
\end{conv}

Recall that equivalent splitting routes may be different only in the last entry, which is the snip point in the case of single-headed splitting routes.

For every non-special generalized type~II split move~$M\mapsto M'$ and its inverse, which is the generalized type~II
merge move~$M'\mapsto M$,
we define the partial homeomorphisms~$h_M^{M'}$, $h_{M'}^M$ as follows.
Let~$M=M_0\xmapsto{\eta_1}M_1\xmapsto{\eta_2}\ldots\xmapsto{\eta_k}M_k=M'$
be the sequence of split moves that arises from the inductive definition of
a non-special generalized type~II split move. We take for~$h_{M'}^M$ the composition~$h_{M_1}^{M_0}\circ h_{M_2}^{M_1}\circ\ldots\circ
h_{M_k}^{M_{k-1}}$. Its domain is the whole of the surface~$\wideparen M'$, and the image is obtained from~$\wideparen M$
by cutting along an arc of the form~$\wideparen\omega$, where~$\omega$ is the type~II splitting route
with which the split move~$M\mapsto M'$ is associated.
Clearly the triple~$(\wideparen M',\wideparen M,h_{M'}^M)$ represents the morphism from~$\widehat M'$
to~$\widehat M$ associated with the move~$M'\xmapsto{\eta^{-1}} M$,
where~$\eta=\eta_k\circ\ldots\circ\eta_2\circ\eta_1$.

Accordingly, $h_M^{M'}$ is defined as the inverse of~$h_{M'}^M$. Its domain is obtained by cutting~$\wideparen M$
along~$\wideparen\omega$, and the image is the whole of~$\wideparen M'$. Clearly, there is a surface~$F\in\widetilde S_{\widehat M'}$
and an extension~$\widetilde h$ of the partial homeomorphism~$h_M^{M'}$ to a homeomorphism~$\wideparen M\rightarrow F$
such that the triple~$(\wideparen M,F,\widetilde h)$ represents the morphism~$\eta$, which is associated with the move~$M\mapsto M'$.

Successors and predecessors of mirrors and occupied levels are defined for the moves~$M\mapsto M'$ and~$M'\mapsto M$
by following the general rule (see Definition~\ref{successor-def}). One can see that every mirror and every occupied
level of~$M'$ has a predecessor in~$M$ for the move~$M\mapsto M'$. The number of successors of a mirror~$\mu$
in~$M$ is one greater than the number of entries in~$\omega$ equal to~$\mu$.

\begin{defi}
Let~$C$ be a collection of boundary circuits of a mirror diagram~$M$, and let~$\omega$
be a splitting route in~$M$. Denote by~$\wideparen C$ the union of the corresponding
boundary components of~$\wideparen M$.
We say that~$\omega$ \emph{separates}~$C$
if there exists an occupied level~$x$ of~$M$ such that at least two
connected components of~$\wideparen x\setminus\wideparen\omega$ have
a non-empty intersection with~$\wideparen C$.
\end{defi}

\begin{prop}\label{non-separating-splitting-prop}
Let~$C$ be a collection of boundary circuits of a mirror diagram~$M$, and let~$\omega$
be a non-special splitting route in~$M$. The following four conditions are equivalent:
\begin{enumerate}
\item
$\omega$ does not separate~$C$;
\item
for any two successive mirrors~$\mu$ and~$\mu'$ in~$\omega$, and also, if~$\omega$ is single-headed,
for~$\mu$ the last mirror and~$\mu'$ the snip point  of~$\omega$, either~$(\mu,\mu')$ or~$(\mu',\mu)$
is disjoint from~$\bigcup_{c\in C}c$;
\item
every split move in the decomposition of~$M\mapsto M'$ arising from Definition~\ref{gen-split-def}
can be chosen so as to preserve all the boundary circuits in~$C$;
\item
a generalized split move~$M\mapsto M'$ associated
with~$\omega$ can be chosen so as to preserve all the boundary circuits in~$C$.
\end{enumerate}
\end{prop}

This proposition is an easy consequence of the correspondence between boundary circuits of~$M$
and boundary components of~$\wideparen M$ described in Subsection~\ref{mirror-diagram-definition-subsec},
and the definition of a split move. We omit the proof.

\begin{defi}
Let~$M\mapsto M'$ be a generalized type~II split move, and let
\begin{equation}\label{canonical-decomposition-eq}
M=M_0\mapsto M_1\mapsto\ldots\mapsto M_k=M'
\end{equation}
be a sequence of moves including a type~II split move, some number of type~I split moves, and possibly
a double split move and a type~I extension move, that arises from the inductive definition of
a generalized type~II split move. We say that the decomposition~\eqref{canonical-decomposition-eq}
of the move~$M\mapsto M'$ is \emph{canonical} if it is neat.
\end{defi}

It follows from Proposition~\ref{non-separating-splitting-prop} that any generalized type~II split move always has a canonical decomposition.

\subsection{Commutation of moves}

Formally speaking, two moves of mirror diagrams~$M_1\mapsto M_2$ and~$M_3\mapsto M_4$
are composable only if~$M_2=M_3$. However, in certain situation, a natural meaning can be given
to a composition of non-trivial moves~$M\mapsto M_1$ and~$M\mapsto M_2$ starting from the same diagram.
If they can be `composed' in either way it makes sense to ask whether the results of the two compositions are (combinatorially)
the same. If they are, then it is natural to say that the moves \emph{commute}. In the following definition we list
specific cases in which we use such terminology.
All mirror diagrams in it are assumed to be enhanced.

\begin{defi}\label{friendly-resemble-def}
Let~$M\xmapsto\eta M_1$ be a move of one of the following kinds:
generalized split, double split, wrinkle creation, jump, elementary bypass removal, or elimination move,
and let~$M\xmapsto\zeta M_2$ be one of the moves for which we have defined the partial homeomorphism~$h_M^{M_2}$.
Let also~$C$ be the set of all boundary circuits of~$M$ preserved by both moves.

We say that the move~$M\xmapsto\zeta M_2$ is \emph{friendly} to the move~$M\xmapsto\eta M_1$ if
one of the cases~1--7 listed below occurs (one extra case will be introduced
in Definition~\ref{friendly-to-gen-wrinkle-def}). In each case we define (non-uniquely) a move~$M_2\xmapsto{\eta^\zeta}M_{21}$
\emph{resembling} the move~$M\xmapsto\eta M_1$. To be called such
it must preserve all boundary circuits in~$C$.
If~$M\xmapsto\zeta M_2$ is a jump move, then, additionally,~$M_2\xmapsto{\eta^\zeta}M_{21}$
is demanded to preserve any boundary circuit of~$M_2$ corresponding
to a boundary circuit of~$M$ preserved by the move~$M\xmapsto\eta M_1$.

\smallskip\noindent\emph{Case~1}:
$M\xmapsto\eta M_1$ is a generalized type~II (respectively, type~I) split move associated with
a type~II (respectively, type~I) splitting route~$\omega$, and there is a unique, up to equivalence,
type~II (respectively, type~I) splitting route~$\omega'$
in~$M_2$ such that the partial homeomorphism~$h_M^{M_2}$
can be chosen to take~$\wideparen\omega$ to~$\wideparen\omega'$.
\\
The move~$M_2\xmapsto{\eta^\zeta}M_{21}$ is defined as a generalized type~II (respectively, type~I)
split move associated with~$\omega'$.

\smallskip\noindent\emph{Case~2}:
$M\xmapsto\eta M_1$ is a jump move,
$M\xmapsto\zeta M_2$ is a move different from a jump move such
that every occupied level and mirror of~$M_2$ has exactly one predecessor in~$M$.
\\
Suppose that the jump move~$M\xmapsto\eta M_1$ shifts a meridian~$m_{\theta_1}\in L_M$.
Let~$m_{\theta_1},m_{\theta_2},\ldots,m_{\theta_n}$ be all occupied meridians of~$M$ listed so that
$\theta_1<\theta_2<\ldots<\theta_n<\theta_1+2\pi$, and let~$m_{\theta_{11}},\ldots,m_{\theta_{1l}}$
be all the successors of~$m_{\theta_1}$ in~$M_2$. By construction of our moves of mirror diagrams, there are $2n$ real numbers
$\theta_1'<\theta_1''<\theta_2'<\theta_2''<\ldots<\theta_n'<\theta_n''<\theta_1'+2\pi$ such that
any successor of the meridian~$m_{\theta_i}$ is contained in~$(\theta_i';\theta_i'')\times\mathbb S^1$.
We may also assume that~$\theta_1'<\theta_{11}<\theta_{12}<\ldots<\theta_{1l}<\theta_1''$.

If the move~$M\xmapsto\eta M_1$ shifts $m_{\theta_1}$ to a position between~$m_{\theta_j}$ and~$m_{\theta_{j+1}}$,
then the move~$M_2\xmapsto{\eta^\zeta}M_{21}$ is a composition of jump moves that replaces
the meridians~$m_{\theta_{11}},\ldots,m_{\theta_{1l}}$ by~$m_{\theta_{11}'},\ldots,m_{\theta_{1l}'}$
such that~$\theta_j''<\theta_{11}'<\theta_{12}'<\ldots<\theta_{1l}'<\theta_{j+1}'$.

If~$M\xmapsto\eta M_1$ shifts not a meridian but a longitude, the move~$M_2\xmapsto{\eta^\zeta}M_{21}$ is defined
similarly with~$\varphi$'s instead of~$\theta$'s.

\smallskip\noindent\emph{Case~3}:
$M\xmapsto\eta M_1$ is a jump move that replaces an occupied level~$x$
by another occupied level~$y$, and~$M\xmapsto\zeta M_2$ is an extension move
such that replacing of~$x$ by~$y$ in~$M_2$ can still be realized by a jump move.
\\
Then this jump move is the one that is taken for~$M_2\xmapsto{\eta^\zeta} M_{21}$.

\smallskip\noindent\emph{Case~4}:
$M\xmapsto\eta M_1$ is an elimination move removing a mirror~$\mu$ and an occupied level~$x$
such that~$\mu$ and~$x$ have unique successors~$\mu'$ and~$x'$, respectively, in~$M_2$, and~$x'$
contains no mirrors of~$M_2$ other than~$\mu'$. The move~$M_2\xmapsto{\eta^\zeta}M_{21}$
is then the elimination of~$\mu'$ and~$x'$.

\smallskip\noindent\emph{Case~5}:
$M\xmapsto\eta M_1$ is an elementary bypass removal, $c$ is the corresponding
inessential boundary circuit having the form of the boundary of a rectangle,
and the move~$M\xmapsto\zeta M_2$ transforms~$c$ to a boundary circuit~$c'$ that
also has the form of the boundary of a rectangle.
\\
Let~$\mu$ be the mirror removed by~$M\xmapsto\eta M_1$. The move~$M_2\xmapsto{\eta^\zeta}M_{21}$
removes a successor of~$\mu$ lying on~$c'$, which is unique.

\smallskip\noindent\emph{Case~6}:
$M\xmapsto\eta M_1$ is a double split move with splitting mirrors~$\mu_1$, $\mu_2$ lying
on an occupied level~$x$,~$M\xmapsto\zeta M_2$ is any move for which the partial
homeomorphism~$h_M^{M_2}$ is defined, and there is a unique triple~$(\mu_1',\mu_2',x')\in E_{M_2}\times E_{M_2}\times L_{M_2}$
such that~$\mu_1'$, $\mu_2'$, and~$x'$ are successors of~$\mu_1$, $\mu_2$ and~$x$, respectively,
and~$\mu_1',\mu_2'\in x'$.
\\
The move~$M_2\xmapsto{\eta^\zeta}M_{21}$ is then defined as a double split move with
splitting mirrors~$\mu_1'$ and~$\mu_2'$.

\smallskip\noindent\emph{Case~7}:
$M\xmapsto\eta M_1$ is a wrinkle creation move with ramification mirrors~$\mu_1$, $\mu_2$,
and the conditions from the previous case hold.
\\
The move~$M_2\xmapsto{\eta^\zeta}M_{21}$ is a wrinkle creation move with
ramification mirrors~$\mu_1'$ and~$\mu_2'$.
\end{defi}

One can see that in each case listed in Definition~\ref{friendly-resemble-def} the move~$M_2\xmapsto{\eta^\zeta}M_{21}$ does exist.

\begin{defi}\label{moves-commute-def}
Let~$M\xmapsto\eta M_1$ and~$M\xmapsto\zeta M_2$ be two moves that are friendly to each other and
neither of which is a special generalized split move, and let~$C$ be the set of all boundary circuits of~$M$ preserved by both moves.
Let also~$M_1\xmapsto{\zeta^\eta}M_{12}$ and~$M_2\xmapsto{\eta^\zeta}M_{21}$
be moves that resemble~$M\xmapsto\zeta M_2$ and~$M\xmapsto\eta M_1$, respectively.
We say that the moves~$M\xmapsto\eta M_1$ and~$M\xmapsto\zeta M_2$ \emph{commute} with
each other if the move
\begin{equation}\label{4composition-eq}
M_{21}\xmapsto{\zeta^\eta\circ\eta\circ\zeta^{-1}\circ(\eta^\zeta)^{-1}}M_{12}
\end{equation}
admits a $C$-neat decomposition into jump moves that preserve the combinatorial type of the diagram.

If the transformation~\eqref{4composition-eq} admits a $C$-neat decomposition into jump moves
each of which either preserves the combinatorial type of the diagram or
modifies it only by exchanging occupied levels that have a common predecessor in~$M$, then we say that
the moves~$M\xmapsto\eta M_1$ and~$M\xmapsto\zeta M_2$ \emph{almost commute}.
\end{defi}

Below we prove a number of statements about almost commutation. In the proofs, we
don't make a special check that the involved jump moves exchange only the successors
of the same occupied level of the original diagram, because this is kind of automatic.
The long sequences of moves that we consider are designed so that the families of successors
of different parallel occupied occupied levels never mix with each other,
and the non-trivial part of establishing the almost commutation of the discussed
moves is to show that the order of successors in each family can be changed
by jump moves in the expected way.

\begin{lemm}\label{splits-commute-with-each-other-lem}
Let~$M\xmapsto\eta M_1$ and~$M\xmapsto\zeta M_2$ be two split moves (not generalized, but of any type)
associated with splitting routes~$\omega_1=(\mu,p)$ and~$\omega_2=(\nu,q)$, respectively,
and let~$x$ and~$y$ be the respective occupied levels of~$M$ that are split by these moves.
Assume that~$p\ne q$.

Then the moves~$M\xmapsto\eta M_1$ and~$M\xmapsto\zeta M_2$ are friendly to each other unless
one of the following situations occurs:
\begin{enumerate}
\item
we have~$\mu=\nu$ and~$x\ne y$;
\item
we have~$x=y$, $\mu\ne\nu$, and the pairs~$\{\mu,p\}$ and~$\{\nu,q\}$ \emph{interleave},
which means that the interval~$(\mu;p)$ contains exactly one of the points~$\nu$ and~$q$.
This is equivalent to saying that~$\wideparen\omega_1$ and~$\wideparen\omega_2$
have an unavoidable intersection.
\end{enumerate}
If the moves~$M\xmapsto\eta M_1$ and~$M\xmapsto\zeta M_2$ are friendly to each other,
then they almost commute.
\end{lemm}

\begin{proof}
The claim of the lemma is quite obvious if~$\mu\ne\nu$ and~$x\ne y$. In this case, the
moves~$M\xmapsto\eta M_1$ and~$M\xmapsto\zeta M_2$, when viewed combinatorially,
`do not interfere' which each other and clearly commute.

The moves~$M\xmapsto\eta M_1$ and~$M\xmapsto\zeta M_2$ are friendly to each other
if and only if the following two conditions hold:
\begin{enumerate}
\item
the splitting paths~$\wideparen\omega_1$, $\wideparen\omega_2$ associated with~$\omega_1$
and~$\omega_2$ can be chosen to be disjoint;
\item
if~$\wideparen\omega_1'$ and~$\wideparen\omega_2'$ are another disjoint splitting paths
associated with~$\omega_1$ and~$\omega_2$, respectively, then there is a self-homeomorphism
of~$\wideparen M$ isotopic to the identity, and preserving the handle decomposition
structure, that takes the pair~$(\wideparen\omega_1,\wideparen\omega_2)$ to~$(\wideparen\omega_1',\wideparen\omega_2')$
\end{enumerate}
Indeed, the first condition means that we can define the image of~$\omega_1$ in~$M_2$
and the image of~$\omega_2$ in~$M_1$ in a consistent way, and the second condition means
that the choice of these images is essentially unique.

The only possible reason for~$\wideparen\omega_1$ and~$\wideparen\omega_2$ to have an
unavoidable intersection is the coincidence of~$x$ and~$y$ and
a cyclic order of~$\mu,p,\nu,q$ such that the pairs~$\{\mu,p\}$ and~$\{\nu,q\}$ interleave;
see Figure~\ref{unfriendly-fig}~(a).
\begin{figure}[ht]
\begin{tabular}{ccc}
\includegraphics[scale=0.65]{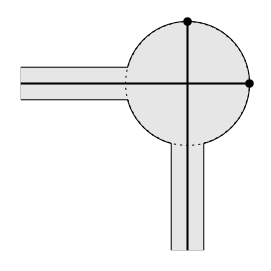}\put(-34,49){$\wideparen x=\wideparen y$}\put(-68,42){$\wideparen\mu$}%
\put(-18,19){$\wideparen\nu$}\put(-3,55){$\wideparen p$}\put(-29,83){$\wideparen q$}&\hbox to 1cm{\hss}&
\raisebox{20pt}{\includegraphics[scale=0.65]{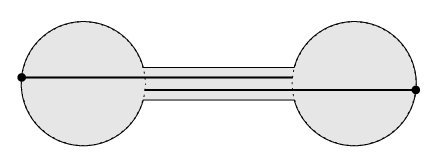}\put(-113,13){$\wideparen x$}\put(-28,13){$\wideparen y$}%
\put(-84,12){$\wideparen\mu=\wideparen\nu$}\put(-140,27){$\wideparen p$}\put(-3,23){$\wideparen q$}}\\
(a)&&(b)
\end{tabular}
\caption{Mutually unfriendly splitting paths~$\wideparen{(\mu,p)}$ and~$\wideparen{(\nu,q)}$}\label{unfriendly-fig}
\end{figure}

Assume that~$\wideparen\omega_1$ and~$\wideparen\omega_2$ can be chosen to be disjoint.
Their mutual position is ambiguous only
if~$\mu$ has two successors for the move~$M\xmapsto\zeta M_2$, or~$\nu$ has two successors for the move~$M\xmapsto\eta M_1$.
Both conditions are equivalent to~$\mu=\nu$. However, if~$x=y$, then the ambiguity does not occur
since the mutual position of~$\wideparen p$ and~$\wideparen q$ on~$\partial\wideparen x$ prescribes
the mutual position of~$\wideparen\omega_1\cap\wideparen\mu$
and~$\wideparen\omega_2\cap\wideparen\mu$
in~$\wideparen\mu$. The ambiguity does occur if~$\mu=\nu$ and~$x\ne y$; see Figure~\ref{unfriendly-fig}~(b).

It remains to show that the moves~$M\xmapsto\eta M_1$ and~$M\xmapsto\zeta M_2$ almost commute with one another
in the case when they are mutually friendly, and we have~$x=y$. Up to various
symmetries there are the following five cases to consider. In each case, we denote by~$\nu_1$ the splitting mirror
of the move~$M_1\xmapsto{\zeta^\eta}M_{12}$,
by~$x_1$ the successor of~$x$ in~$M_1$ containing~$\nu_1$,
by~$\mu_2$ the splitting mirror
of the move~$M_2\xmapsto{\eta^\zeta}M_{21}$, by~$x_2$ the successor of~$x$
in~$M_2$ containing~$\mu_2$, and by~$q_1\in x_1$, $p_2\in x_2$ points such that~$h_M^{M_1}(\wideparen q)=
\wideparen q_1$, $h_M^{M_2}(\wideparen p)=\wideparen p_2$.

\medskip
\noindent\emph{Case 1}:
the moves~$M\xmapsto\eta M_1$ and~$M\xmapsto\zeta M_2$ are of type~II and type~I, respectively,
and we have~$q\in(\mu;\nu)$ and~$p\in(\nu;\mu)$.

One can see from Figure~\ref{split-comm-lem-case-x=y-1}
that the results of the transformations~$M\xmapsto{\zeta^\eta\circ\eta}M_{12}$
and~$M\xmapsto{\eta^\zeta\circ\zeta}M_{21}$ are combinatorially the same.
It follows easily that the moves commute.
\begin{figure}[ht]
\includegraphics[scale=0.65]{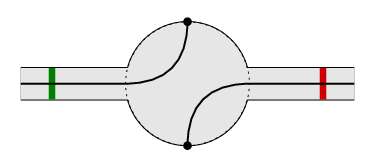}\put(-61,23){$\wideparen x$}\put(-100,10){$\wideparen\mu$}%
\put(-24,10){$\wideparen\nu$}\put(-61,-3){$\wideparen q$}\put(-61,50){$\wideparen p$}

\begin{tabular}{ccc}
\includegraphics[scale=0.65]{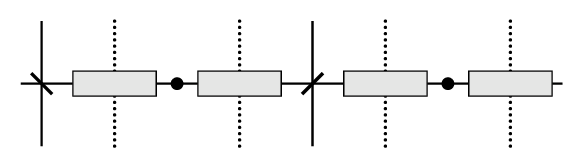}\put(-166,17){$\mu$}\put(-129,17){$q$}\put(-95,17){$\nu$}%
\put(-45,17){$p$}\put(-160,-5){$M$}\put(-184,24){$x$}
&\raisebox{23pt}{$\stackrel\eta\longrightarrow$}&
\includegraphics[scale=0.65]{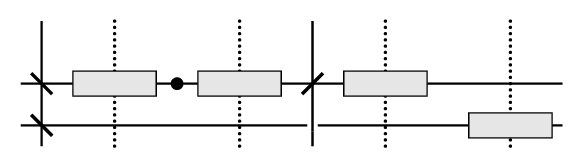}\put(-129,17){$q_1$}\put(-95,17){$\nu_1$}\put(-160,-5){$M_1$}\put(-187,24){$x_1$}\\
$\big\downarrow$\hbox to 0pt{$\scriptstyle\zeta$\hss}&&$\big\downarrow$\hbox to 0pt{$\scriptstyle\zeta^\eta$\hss}\\
\includegraphics[scale=0.65]{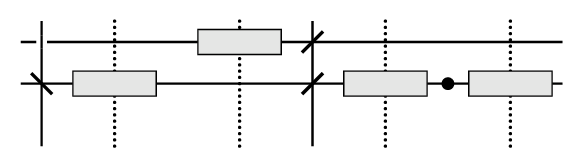}\put(-169,17){$\mu_2$}\put(-45,17){$p_2$}\put(-160,-5){$M_2$}\put(-187,24){$x_2$}
&\raisebox{23pt}{$\stackrel{\eta^\zeta}\longrightarrow$}&
\includegraphics[scale=0.65]{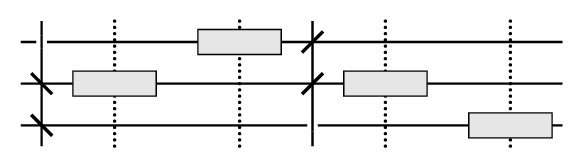}\put(-160,-5){$M_{12}=M_{21}$}
\end{tabular}
\caption{Proof of Lemma~\ref{splits-commute-with-each-other-lem} in the case $x=y$, $\mu$ is of type~`$\diagdown$',
$\nu$ is of type~`$\diagup$',
$q\in(\mu;\nu)$, $p\in(\nu;\mu)$}\label{split-comm-lem-case-x=y-1}
\end{figure}

\medskip
\noindent\emph{Case 2}:
The moves~$M\xmapsto\eta M_1$ and~$M\xmapsto\zeta M_2$ are of type~II and type~I, respectively,
and we have~$q\in(\nu;\mu)$ and~$p\in(q;\mu)$.

One can see from Figure~\ref{split-comm-lem-case-x=y-2}
that the combinatorial types of the diagrams~$M_{12}$ and~$M_{21}$ are obtained from each
other by exchanging two successors of~$x$, and this exchange can
be realized by jump moves. This implies that
the moves almost commute.
\begin{figure}[ht]
\includegraphics[scale=0.65]{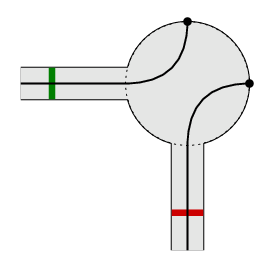}\put(-29,55){$\wideparen x$}\put(-68,42){$\wideparen\mu$}%
\put(-18,19){$\wideparen\nu$}\put(-3,55){$\wideparen q$}\put(-29,83){$\wideparen p$}

\begin{tabular}{ccc}
\includegraphics[scale=0.65]{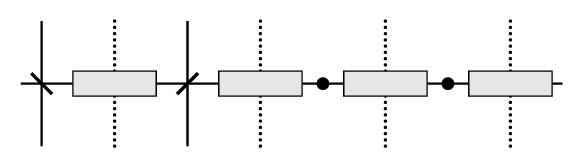}\put(-166,17){$\mu$}\put(-85,17){$q$}\put(-133,17){$\nu$}%
\put(-45,17){$p$}\put(-160,-5){$M$}\put(-184,24){$x$}
&\raisebox{23pt}{$\stackrel\eta\longrightarrow$}&
\includegraphics[scale=0.65]{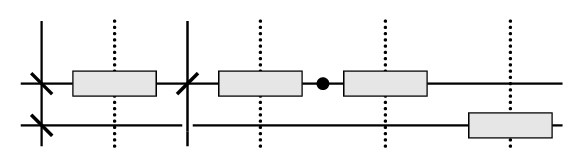}\put(-85,17){$q_1$}\put(-134,17){$\nu_1$}\put(-160,-5){$M_1$}\put(-187,24){$x_1$}\\
$\big\downarrow$\hbox to 0pt{$\scriptstyle\zeta$\hss}\\
\includegraphics[scale=0.65]{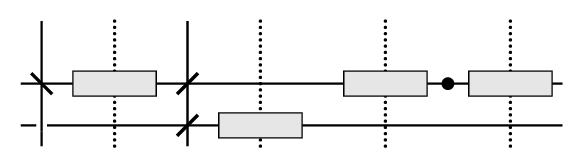}\put(-169,17){$\mu_2$}\put(-45,17){$p_2$}\put(-160,-5){$M_2$}\put(-187,24){$x_2$}
&&\raisebox{20pt}{$\big\downarrow$\hbox to 0pt{$\scriptstyle\zeta^\eta$\hss}}\\
$\big\downarrow$\hbox to 0pt{$\scriptstyle\eta^\zeta$\hss}\\
\includegraphics[scale=0.65]{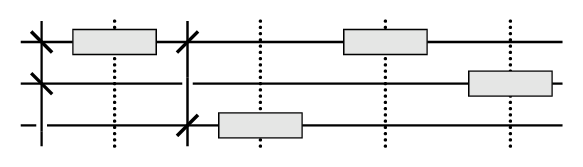}\put(-160,-5){$M_{21}$}
&&
\includegraphics[scale=0.65]{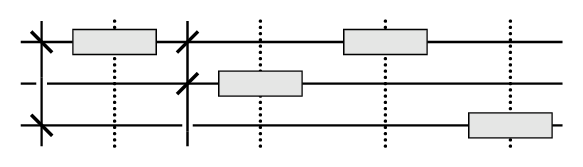}\put(-160,-5){$M_{12}$}\\
\end{tabular}
\caption{Proof of Lemma~\ref{splits-commute-with-each-other-lem} in the case $x=y$, $\mu$ is of type~`$\diagdown$',
$\nu$ is of type~`$\diagup$',
$q\in(\nu;\mu)$, $p\in(q;\mu)$}\label{split-comm-lem-case-x=y-2}
\end{figure}

\medskip
\noindent\emph{Case 3}:
The moves~$M\xmapsto\eta M_1$ and~$M\xmapsto\zeta M_2$ are both of type~I,
and we have~$\mu=\nu$. As one can see from Figure~\ref{split-comm-lem-case-mu=nu-x=y-fig}
the diagrams~$M_{12}$ and~$M_{21}$ are combinatorially equivalent, which
means that the moves commute.
\begin{figure}[ht]
\includegraphics[scale=0.65]{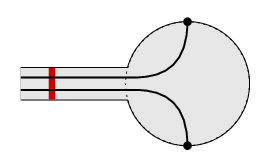}\put(-26,23){$\wideparen x$}\put(-75,10){$\wideparen\mu=\wideparen\nu$}%
\put(-29,-3){$\wideparen p$}\put(-29,50){$\wideparen q$}

\begin{tabular}{ccc}
\includegraphics[scale=0.65]{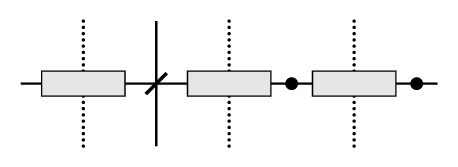}\put(-112,17){$\mu=\nu$}\put(-16,17){$q$}%
\put(-55,17){$p$}\put(-130,-5){$M$}\put(-145,24){$x$}
&\raisebox{23pt}{$\stackrel\eta\longrightarrow$}&
\includegraphics[scale=0.65]{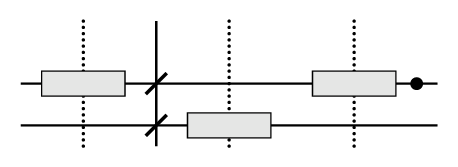}\put(-105,17){}\put(-16,17){$q_1$}\put(-90,30){$\nu_1$}%
\put(-130,-5){$M_1$}\put(-148,24){$x_1$}\\
$\big\downarrow$\hbox to 0pt{$\scriptstyle\zeta$\hss}&&$\big\downarrow$\hbox to 0pt{$\scriptstyle\zeta^\eta$\hss}\\
\includegraphics[scale=0.65]{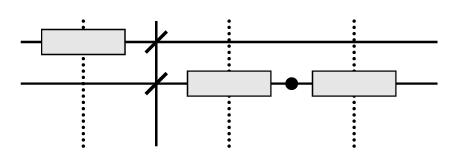}\put(-107,17){$\mu_2$}%
\put(-55,17){$p_2$}\put(-130,-5){$M_2$}\put(-148,24){$x_2$}
&\raisebox{23pt}{$\stackrel{\eta^\zeta}\longrightarrow$}&
\includegraphics[scale=0.65]{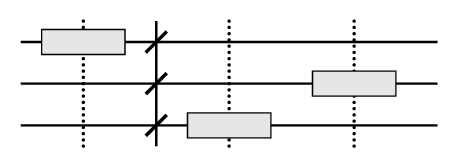}\put(-130,-5){$M_{12}=M_{21}$}
\end{tabular}
\caption{Proof of Lemma~\ref{splits-commute-with-each-other-lem} in the case $x=y$, $\mu=\nu$ are of type~`$\diagup$'}\label{split-comm-lem-case-mu=nu-x=y-fig}
\end{figure}

\medskip
\noindent\emph{Case 4}:
The moves~$M\xmapsto\eta M_1$ and~$M\xmapsto\zeta M_2$ are both of type~I,
and we have~$q\in(\mu;\nu)$, $p\in(\nu;\mu)$. This case is similar to Case~2.
The almost commutation of the moves is illustrated in Figure~\ref{split-comm-lem-case-type-I-x=y-1-fig}.
\begin{figure}[ht]
\raisebox{20pt}{\includegraphics[scale=0.65]{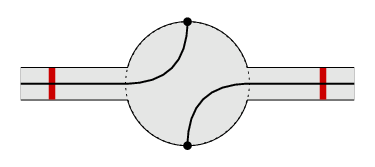}\put(-61,23){$\wideparen x$}\put(-100,10){$\wideparen\mu$}%
\put(-24,10){$\wideparen\nu$}\put(-61,-3){$\wideparen q$}\put(-61,50){$\wideparen p$}}

\begin{tabular}{ccc}
\includegraphics[scale=0.65]{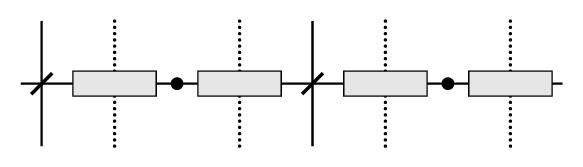}\put(-178,17){$\mu$}\put(-129,17){$q$}\put(-93,17){$\nu$}%
\put(-45,17){$p$}\put(-160,-5){$M$}\put(-184,24){$x$}
&\raisebox{23pt}{$\stackrel\eta\longrightarrow$}&
\includegraphics[scale=0.65]{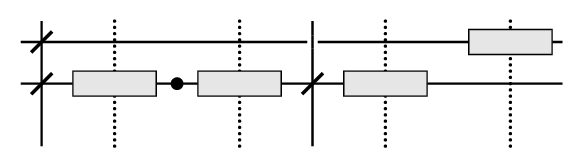}\put(-129,17){$q_1$}\put(-95,17){$\nu_1$}\put(-160,-5){$M_1$}\put(-187,24){$x_1$}\\
$\big\downarrow$\hbox to 0pt{$\scriptstyle\zeta$\hss}\\
\includegraphics[scale=0.65]{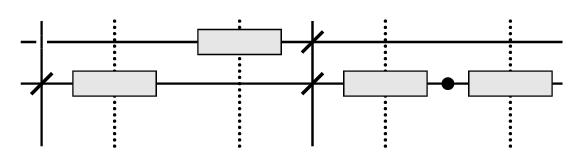}\put(-180,17){$\mu_2$}\put(-45,17){$p_2$}\put(-160,-5){$M_2$}\put(-187,24){$x_2$}
&&\raisebox{20pt}{$\big\downarrow$\hbox to 0pt{$\scriptstyle\zeta^\eta$\hss}}\\
$\big\downarrow$\hbox to 0pt{$\scriptstyle\eta^\zeta$\hss}\\
\includegraphics[scale=0.65]{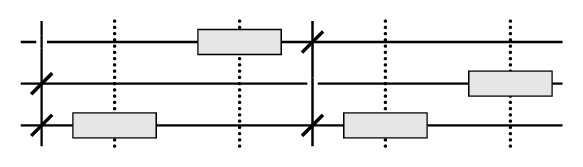}\put(-160,-5){$M_{21}$}
&&
\includegraphics[scale=0.65]{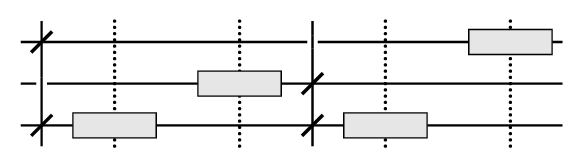}\put(-160,-5){$M_{12}$}\\
\end{tabular}
\caption{Proof of Lemma~\ref{splits-commute-with-each-other-lem} in the case $x=y$, $\mu\ne\nu$ are of type~`$\diagup$', $q\in(\mu;\nu)$,
$p\in(\nu;\mu)$}\label{split-comm-lem-case-type-I-x=y-1-fig}
\end{figure}

\medskip
\noindent\emph{Case 5}:
The moves~$M\xmapsto\eta M_1$ and~$M\xmapsto\zeta M_2$ are both of type~I,
and we have~$q\in(\nu;\mu)$, $p\in(q;\mu)$. This case is similar to Case~1.
The commutation of the moves is illustrated in Figure~\ref{split-comm-lem-case-type-I-x=y-2-fig}.
\begin{figure}[ht]
\includegraphics[scale=0.65]{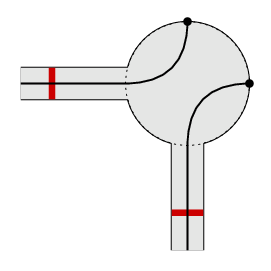}\put(-29,55){$\wideparen x$}\put(-68,42){$\wideparen\mu$}%
\put(-18,19){$\wideparen\nu$}\put(-3,55){$\wideparen q$}\put(-29,83){$\wideparen p$}

\begin{tabular}{ccc}
\includegraphics[scale=0.65]{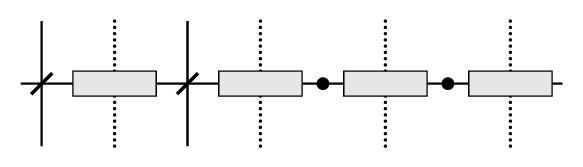}\put(-178,17){$\mu$}\put(-85,17){$q$}\put(-133,17){$\nu$}%
\put(-45,17){$p$}\put(-160,-5){$M$}\put(-184,24){$x$}
&\raisebox{23pt}{$\stackrel\eta\longrightarrow$}&
\includegraphics[scale=0.65]{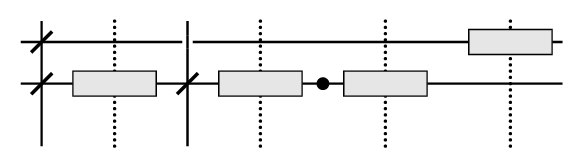}\put(-85,17){$q_1$}\put(-134,17){$\nu_1$}\put(-160,-5){$M_1$}\put(-187,24){$x_1$}\\
$\big\downarrow$\hbox to 0pt{$\scriptstyle\zeta$\hss}&&$\big\downarrow$\hbox to 0pt{$\scriptstyle\zeta^\eta$\hss}\\
\includegraphics[scale=0.65]{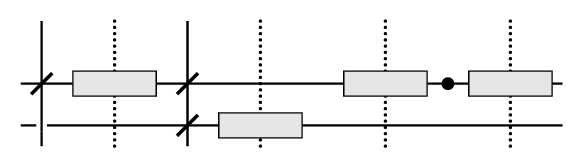}\put(-180,17){$\mu_2$}\put(-45,17){$p_2$}\put(-160,-5){$M_2$}\put(-187,24){$x_2$}
&\raisebox{23pt}{$\stackrel{\eta^\zeta}\longrightarrow$}&
\includegraphics[scale=0.65]{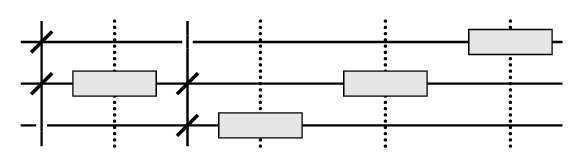}\put(-160,-5){$M_{12}=M_{21}$}
\end{tabular}
\caption{Proof of Lemma~\ref{splits-commute-with-each-other-lem} in the case $x=y$, $\mu\ne\nu$ are of type~`$\diagup$', $q\in(\nu;\mu)$,
$p\in(q;\mu)$}\label{split-comm-lem-case-type-I-x=y-2-fig}
\end{figure}

\medskip
Thus, in all five cases the moves~$M\xmapsto\eta M_1$ and~$M\xmapsto\zeta M_2$ almost commute.
\end{proof}

\begin{lemm}\label{splits-commute-with-each-other-lem-2}
Let~$M\xmapsto\eta M_1$ be a double split move with splitting
$\diagdown$-mirror~$\mu_1$ and splitting $\diagup$-mirror~$\mu_2$, and let~$M\xmapsto\zeta M_2$ be a type~I split move
associated with a splitting route~$\omega_2=(\nu,q)$.
Let also~$x$ and~$y$ be the respective occupied levels of~$M$ that are split by these moves.

Then the moves~$M\xmapsto\eta M_1$ and~$M\xmapsto\zeta M_2$ are friendly to each other unless
one of the following situations occurs:
\begin{enumerate}
\item
we have~$\mu_2=\nu$ and~$x\ne y$;
\item
we have~$x=y$, $\mu_2\ne\nu$, and the pairs~$\{\mu_1,\mu_2\}$ and~$\{\nu,q\}$ interleave.
This is equivalent to saying that~$\wideparen\omega_2$
has an unavoidable intersection with a splitting path associated
with some (and then any) type~II splitting route starting from~$(\mu_1,\mu_2)$.
\end{enumerate}
If the moves~$M\xmapsto\eta M_1$ and~$M\xmapsto\zeta M_2$ are friendly to each other,
then they almost commute.
\end{lemm}

We omit the proof, which is completely analogous to the proof of Lemma~\ref{splits-commute-with-each-other-lem}.

\begin{lemm}\label{split-commute-with-jump-lem}
Any non-special generalized type~II split move commutes with any jump move.
\end{lemm}

We omit the easy proof.

The following is an analogue of Lemma~\ref{split-commute-with-jump-lem} for the case of a special generalized split move.

\begin{lemm}\label{special-split-commute-with-jumpe-lem}
Let~$M\xmapsto\eta M_1$ be a special generalized type~II split move, and let~$M\xmapsto\zeta M_2$ be a jump move.
Let also~$M\xmapsto{\eta_1}M_1'$ be the first move of a canonical decomposition of~$M\xmapsto\eta M_1$, which is a type~I extension move, and let~$M_1'\xmapsto{\eta_2}M_1$ be the `remaining part' of the decomposition,
which is a generalized type~II split move such that~$\eta=\eta_2\circ\eta_1$. Denote by~$C$ the set of
all boundary circuits preserved by both moves~$M\xmapsto\eta M_1$ and~$M\xmapsto\zeta M_2$.

Suppose that the move~$M\xmapsto{\eta_1}M_1'$ is friendly to~$M\xmapsto\zeta M_2$. Then the move~$M_2\xmapsto{\eta^\zeta}M_{21}$
can be chosen so that it preserves all boundary circuits in~$C$,
and the transformation~~$M_1\xmapsto{\eta^\zeta\circ\zeta\circ\eta^{-1}}M_{21}$
admits a $C$-neat decomposition into jump moves.
\end{lemm}

\begin{proof}
\begin{figure}[ht]
\includegraphics{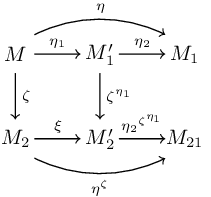}
\caption{Scheme of the moves in the proof of Lemma~\ref{special-split-commute-with-jumpe-lem}}\label{special-split-commute-fig}
\end{figure}
By the assumption of the lemma there is a well defined jump move~$M_1'\xmapsto{\zeta^{\eta_1}}M_2'$. One can see
that~$M_2\xmapsto\xi M_2'$, where~$\xi=\zeta^{\eta_1}\circ\eta_1\circ\zeta^{-1}$, is a type~I extension move
(consult Figure~\ref{special-split-commute-fig} for the scheme).
The moves~$M_1'\xmapsto{\eta_2}M_1$ and~$M_1'\xmapsto{\zeta^{\eta_1}}M_2'$ commute by Lemma~\ref{split-commute-with-jump-lem},
so we can define the non-special generalized type~II split move~$M_2'\xmapsto{\eta_2^{\zeta^{\eta_1}}}M_{21}$ so
that the transformation~$M_1\xmapsto{\eta_2^{\zeta^{\eta_1}}\circ\zeta^{\eta_1}\circ{\eta_2^{-1}}}M_{21}$ admits a $C$-neat
decomposition into jump moves. It remains to note that~$M_2\xmapsto{\eta_2^{\zeta^{\eta_1}}\circ\xi}M_{21}$ is
a special generalized type~II split move that can be taken for~$M_2\xmapsto{\eta^\zeta}M_{21}$.
\end{proof}

\begin{lemm}\label{elimination-commutes-with-split-lem}
Let~$M\xmapsto\eta M_1$ be a generalized type~II split move associated with a
non-special type~II splitting route~$\omega$
in~$M$, and let~$M\xmapsto\zeta M_2$ be a type~I elimination move
such that the mirror being removed does not appear in~$\omega$.
Then these moves are friendly to one another and, moreover, they commute.
\end{lemm}

We omit the easy proof.

\begin{lemm}\label{type-i-split-commutes-with-generalized-type-ii-split-lem}
Let~$M\xmapsto\eta M_1$ be a generalized type~II split move associated with a
non-special type~II splitting route~$\omega$
in~$M$, and let~$M\xmapsto\zeta M_2$ be a type~I split move
that is friendly to~$M\xmapsto\eta M_1$. Then the move~$M\xmapsto\eta M_1$
is also friendly to~$M\xmapsto\zeta M_2$, and the moves almost commute with one another.
\end{lemm}

\begin{proof}
Let~$C$ be the set of all boundary circuits of~$M$ preserved by both moves~$M\xmapsto\eta M_1$ and~$M\xmapsto\zeta M_2$,
and let~$\omega_1=(\nu,q)$ be the type~I splitting route with which the move~$M\xmapsto\zeta M_2$ is associated.
Since this move is friendly to the move~$M\xmapsto\eta M_1$ the
path~$\wideparen\omega_1$ has no unavoidable intersections with~$\wideparen\omega$, and,
moreover, the mutual position of~$\wideparen\omega$ and~$\wideparen\omega_1$ is
essentially unique. This implies that the move~$M\xmapsto\eta M_1$ is also friendly to~$M\xmapsto\zeta M_2$.
We will now show, by induction in the length of~$\omega$, that the moves~$M\xmapsto\eta M_1$ and~$M\xmapsto\zeta M_2$ almost commute.

Let~$\omega=(\mu_1,\mu_2,\ldots,\mu_k,p)$. If~$k=1$, then~$M\xmapsto\eta M_1$ is an ordinary
type~II split move, whose almost commutation with~$M\xmapsto\zeta M_2$ follows from
Lemma~\ref{splits-commute-with-each-other-lem}.

Now assume that~$k>1$ and the lemma is proved for generalized type~II split moves
associated with shorter non-special splitting routes than~$\omega$. Let~$M\xmapsto{\eta_1}M'$ be the first move of
a canonical decomposition of the move~$M\xmapsto\eta M_1$, and let~$M'\xmapsto{\eta_2}M_1$
be the `remaining part' of this decomposition, that is, a generalized type~II split move
such that~$M\xmapsto\eta M_1$ is the composition of~$M\xmapsto{\eta_1}M'$ and~$M'\xmapsto{\eta_2}M_1$.
Let also~$M_2\xmapsto\xi M_2'$
be the first move of a canonical decomposition of the move~$M_2\xmapsto{\eta^\zeta}M_{21}$,
and~$M'\xmapsto\chi M_2'$ be the composition of the moves~$M'\xmapsto{\eta_1^{-1}}M$,
$M\xmapsto\zeta M_2$, and~$M_2\xmapsto\xi M_2'$.

By construction, the transformation~$M'\xmapsto\chi M_2'$ is friendly to the generalized type~II split move~$M'\xmapsto{\eta_2}M_1$,
and the move~$M_2'\xmapsto{{\eta_2}^\chi}M_{21}$ is the `remaining part' of a canonical decomposition of the
move~$M_2\xmapsto{\eta^\zeta}M_{21}$.
Thus, we have the commutative diagram of moves shown in Figure~\ref{comm-diagr1-dig} on the left, in which all the transformations preserve
the boundary circuits in~$C$.
\begin{figure}[ht]
\centerline{\includegraphics{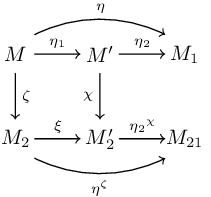}\hskip1cm\includegraphics{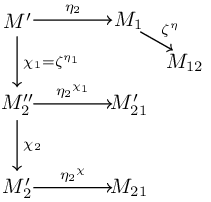}}
\caption{The scheme of the induction step in the proof of Lemma~\ref{type-i-split-commutes-with-generalized-type-ii-split-lem}}\label{comm-diagr1-dig}
\end{figure}

Suppose that the moves~$M\xmapsto{\eta_1}M'$ and~$M\xmapsto\zeta M_2$ almost commute. Then~$\xi={\eta_1}^\zeta$,
and the move~$M'\xmapsto\chi M_2'$ admits a $C$-neat decomposition into a type~I split move of the form~$M'\xmapsto{\chi_1}M_2''$
with~$\chi_1=\zeta^{\eta_1}$, followed by a transformation~$M_2''\xmapsto{\chi_2}M_2'$ that can be further $C$-neatly decomposed into jump moves.
The move~$M'\xmapsto{\zeta^{\eta_1}}M_2''$ is friendly to~$M'\xmapsto{\eta_2}M_1$, and the latter
is a generalized type~II split move associated with a non-special type~II splitting route which is shorter than~$\omega$.

By the induction hypothesis there is a generalized type~II split move~$M_2''\xmapsto{{\eta_2}^{\chi_1}}M_{21}'$
and a type~I split move~$M_1\xmapsto{\chi_1^{\eta_2}=\zeta^\eta}M_{12}$, both preserving the boundary circuits in~$C$,
such that~$M_{21}'$ and~$M_{12}$ are related by a sequence of jump moves also preserving the boundary circuits in~$C$
and realizing the morphism~$\zeta^\eta\circ\eta_2\circ\chi_1^{-1}\circ({\eta_2}^{\chi_1})^{-1}$.

It follows from Lemma~\ref{split-commute-with-jump-lem} that the moves~$M_2''\xmapsto{{\eta_2}^{\chi_1}}M_{21}'$
and~$M_2''\xmapsto{\chi_2}M_2'$ commute, and by construction we have~${\eta_2}^\chi=({\eta_2}^{\chi_1})^{\chi_2}$.
Therefore, the transformation~$M_{21}'\xmapsto{{\eta_2}^\chi\circ\chi_2\circ({\eta_2}^{\chi_1})^{-1}}M_{21}$
also admits a $C$-neat decomposition into jump moves. We see from here that the moves~$M\xmapsto\eta M_1$
and~$M\xmapsto\zeta M_2$ almost commute.

Now we turn to the case when the moves~$M\xmapsto{\eta_1}M'$ and~$M\xmapsto\zeta M_2$ do not almost commute.
Suppose that~$\omega$ is a single-headed splitting route, so~$M\xmapsto{\eta_1}M'$ is a type~I split move.

The splitting paths~$\wideparen{(\mu_k,p)}$ and~$\wideparen\omega_1$ don't have an unavoidable intersection,
since neither do~$\wideparen\omega$ and~$\wideparen\omega_1$. By Lemma~\ref{splits-commute-with-each-other-lem}
the only reason for~$M\xmapsto{\eta_1}M'$ and~$M\xmapsto\zeta M_2$ to be unfriendly to one another is
an ambiguity in the mutual position of~$\wideparen{(\mu_k,p)}$ and~$\wideparen\omega_1$.
This is the issue illustrated in Figure~\ref{unfriendly-fig}~(b).

However, by the assumption of the lemma the mutual position of~$\wideparen\omega$ and~$\wideparen\omega_1$
is not ambiguous if they are chosen disjoint. Let us fix such a choice. Then there is a unique
splitting route~$\omega_1'=(\nu',q')$ in~$M'$ such that~$h_M^{M'}(\wideparen\omega_1)=\wideparen\omega_1'$.
We define the move~$M'\xmapsto{\chi_1} M_2''$ as a type~I split move associated with~$\omega_1'$
and preserving all boundary circuits in~$C$.

As one can learn from Figure~\ref{special-case-gen-split-commute-with-split-fig}
the diagrams~$M_2''$ and~$M_2'$ are combinatorially equivalent, so one of them can be obtained from the other
by jump moves adjusting the positions of the occupied levels without changing the combinatorial type.
\begin{figure}[ht]
\includegraphics[scale=0.65]{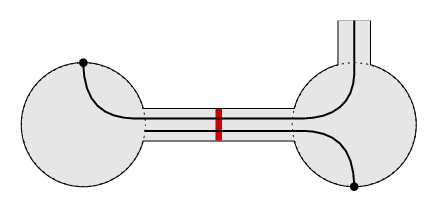}\put(-115,23){$\wideparen x$}\put(-24,23){$\wideparen y$}%
\put(-84,12){$\wideparen\mu_k=\wideparen\nu$}\put(-113,50){$\wideparen p$}\put(-29,-4){$\wideparen q$}%
\put(-18,52){$\wideparen\mu_{k-1}$}

\begin{tabular}{ccc}
\includegraphics[scale=0.65]{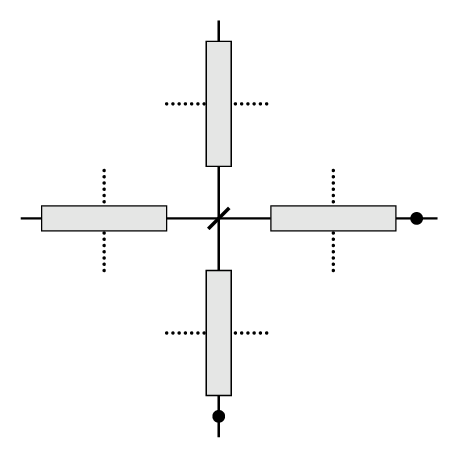}\put(-70,10){$p$}\put(-15,66){$q$}\put(-72,80){$\nu$}%
\put(-125,88){\begin{rotate}{-90}{$\in$}\end{rotate}}\put(-134,92){$\mu_{k-1}$}\put(-145,73){$y$}\put(-77,140){$x$}\put(-130,20){$M$}
&\raisebox{72pt}{$\stackrel{\eta_1}\longrightarrow$}&
\includegraphics[scale=0.65]{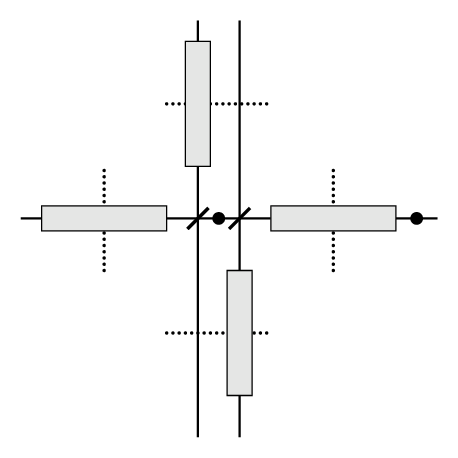}\put(-15,66){$q'$}\put(-65,80){$\nu'$}\put(-80,66){$p'$}\put(-130,20){$M'$}\\
$\big\downarrow$\hbox to 0pt{$\scriptstyle\zeta$\hss}&&$\big\downarrow$\hbox to 0pt{$\scriptstyle\chi_1$}\\
\includegraphics[scale=0.65]{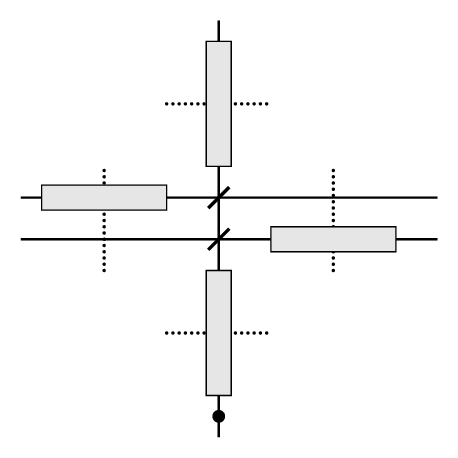}\put(-130,20){$M_2$}\put(-70,10){$p_2$}
&\raisebox{72pt}{$\stackrel\xi\longrightarrow$}&
\includegraphics[scale=0.65]{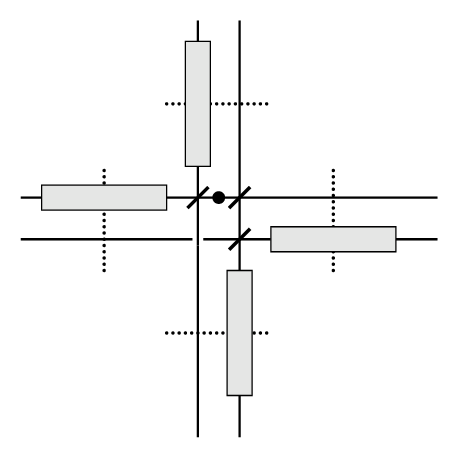}\put(-130,20){$M_2'=M_2''$}\put(-80,72){$p_2'$}
\end{tabular}
\caption{Induction step in the proof of Lemma~\ref{type-i-split-commutes-with-generalized-type-ii-split-lem} in the case when
the moves~$M\xmapsto{\eta_1}M'$ and~$M\xmapsto\zeta M_2$ do not almost commute}\label{special-case-gen-split-commute-with-split-fig}
\end{figure}
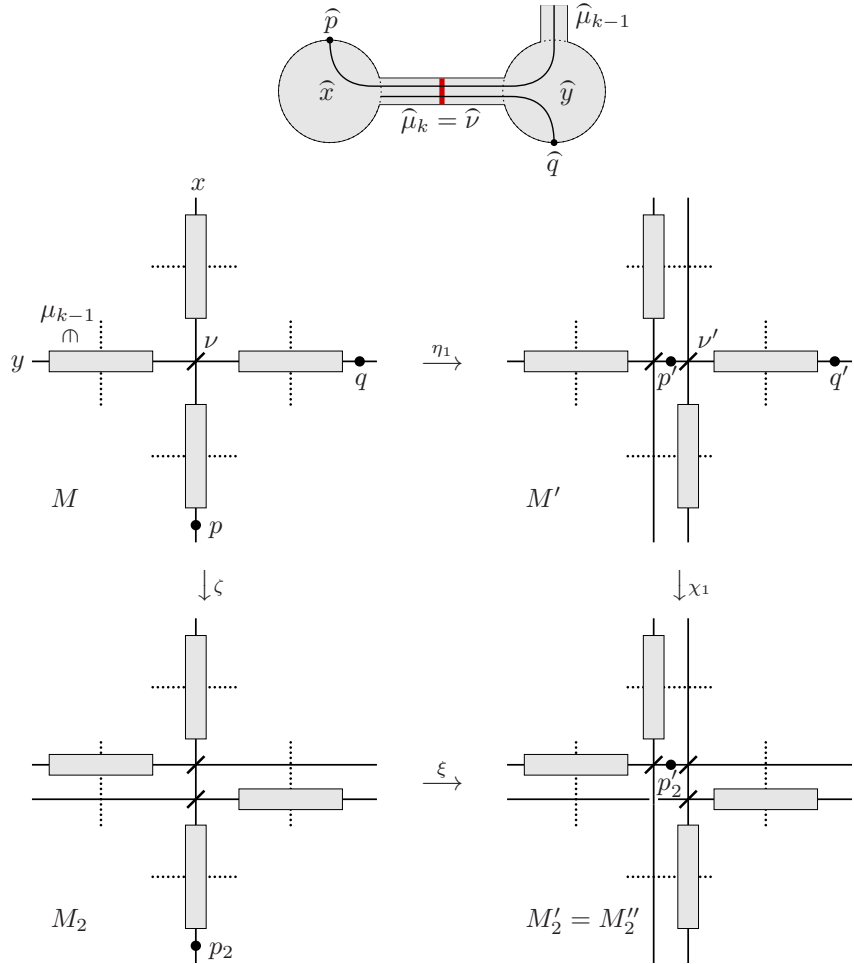
In this figure we denote by~$p'$, $p_2$, and~$p_2'$ the snip points of the generalized type~II split moves~$M'\xmapsto{\eta_2}M_1$,
$M_2\xmapsto{\eta^\zeta}M_{21}$, and~$M_2'\xmapsto{{\eta_2}^\chi}M_{21}$, respectively. We also denote
by~$x$ and~$y$ the occupied levels of~$M$ that are split by the moves~$M\xmapsto{\eta_1}M'$ and~$M\xmapsto\zeta M_2$, respectively.

Figure~\ref{special-case-gen-split-commute-with-split-fig} shows, up to symmetry, what happens in the particular case when~$\mu_{k-1}\ne\mu_k$.
It is, however, possible that~$\mu_{k-1}=\mu_k$. Since there is no ambiguity in the mutual position of~$\wideparen\omega$
and~$\wideparen\omega_1$, there must be another arc in~$\wideparen y\cap\wideparen\omega$
separating~$\wideparen q$ from the `tail' of~$\wideparen\omega$ in~$\wideparen y$.
This means that, for some~$j\notin\{k,k-1\}$, we
have~$\mu_j=\mu_k$ and, for some~$\epsilon\in\{-1,1\}$, we have~$\mu_{j+\epsilon}\in y$,
$\mu_{j+\epsilon}\ne\mu_k$. One should then substitute~$\mu_{j+\epsilon}$ for~$\mu_{k-1}$ in Figure~\ref{special-case-gen-split-commute-with-split-fig}.

If~$\omega$ is a double-headed splitting route, then the reasoning is almost the same with the roles of~$\mu_k$ and~$\mu_{k-1}$ played
by~$\mu_{k-1}$ and~$\mu_{k-2}$, respectively. In place of~$p$ we have the mirror~$\mu_k$, which has
two successors in each of~$M'$, $M_2'$, and~$M_2''$,
and Lemma~\ref{splits-commute-with-each-other-lem-2} should be used instead of Lemma~\ref{splits-commute-with-each-other-lem}.

Thus, the scheme described above and sketched in Figure~\ref{comm-diagr1-dig} works perfectly in the case when
the moves~$M\xmapsto{\eta_1}M'$ and~$M\xmapsto\zeta M_2$ do not almost commute, too,
though~$\chi$ and~$\xi$ are not interpreted as~$\zeta^{\eta_1}$ and~${\eta_1}^\zeta$, respectively.
So, the almost commutation of the moves~$M\xmapsto\eta M_1$ and~$M\xmapsto\zeta M_2$ also holds in this case.
\end{proof}

\begin{lemm}\label{bypass-commutes-with-split-lem}
Let~$M\xmapsto\eta M_1$ be a generalized type~II split move associated with a non-special
type~II splitting route~$\omega$ in~$M$, and let~$M\xmapsto\zeta M_2$ be a type~I elementary bypass removal.
Denote by~$\mu_0$ the $\diagdown$-mirror removed by~$M\xmapsto\zeta M_2$
and by~$c$ the inessential boundary circuit of~$M$ hitting~$\mu_0$ and having the form~$\partial r$,
where~$r$ is a rectangle with interior disjoint from~$E_M$.

Suppose that the following conditions hold:
\begin{enumerate}
\item
the mirror~$\mu_0$ does not appear in~$\omega$;
\item
either~$\omega$ is double-headed, or~$\omega$ is single-headed and the snip point of the move~$M\xmapsto\eta M_1$
does not lie on~$c$.
\end{enumerate}
Then the moves~$M\xmapsto\eta M_1$ and~$M\xmapsto\zeta M_2$ are friendly to one another and,
moreover, they commute.
\end{lemm}

\begin{proof}
The move~$M\xmapsto\zeta M_2$ is clearly friendly to~$M\xmapsto\eta M_1$, since it only removes a single mirror that
does not appear in~$\omega$.

If~$r$ is a rectangle such that~$\partial r\in\partial M$ and~$\interior(r)\cap E_M=\varnothing$, then any double split
move~$M\mapsto M'$ takes the boundary circuit~$\partial r$ to a boundary circuit of the same
kind, that is~$\partial r'\in\partial M'$, $\interior(r')\cap E_{M'}=\varnothing$, where~$r'$ is a rectangle.
This follows easily from the definition of a double split move.

An ordinary split move can `destroy'~$\partial r$ only if the snip point of the move lies on~$\partial r$. Thus,
the first move in a canonical decomposition does not destroy~$c$. The subsequent moves
don't do so either because they are split moves whose snip point lies in the splitting gap of the previous move,
and this splitting gap cannot be a part of~$c$ or of the boundary circuit into which~$c$ is transformed.
One can see that the removal of~$\mu_0$ or its successor
is a type~I elementary bypass removal for all diagrams arising in the canonical decomposition
of the move~$M\xmapsto\eta M_1$, and this bypass removal commutes with the
ordinary split moves from the canonic decomposition. The claim follows.
\end{proof}

\subsection{Similarity of splitting routes}

\begin{defi}\label{pull-tight-loosen-def}
Let~$\omega=(\mu_1,\mu_2,\ldots,\mu_k,p)$ and~$\omega'$ be type~II splitting routes in a mirror diagram~$M$. We say that~$\omega'$
is obtained from~$\omega$ by \emph{a pulling tight operation},
and that~$\omega$ is obtained from~$\omega'$ by \emph{a loosening operation},
if, for some~$i\in\{2,\ldots,k-1\}$, we have~$\mu_i=\mu_{i+1}$, and~$\omega'$ is the result of removing the~$i$th and
the~$(i+1)$st entries in~$\omega$.
\end{defi}

\begin{defi}\label{shrink-stretch-def}
Let~$\omega=(\mu_1,\mu_2,\ldots,\mu_k,p)$ and~$\omega'$ be single-headed
type~II splitting routes in a mirror diagram~$M$. We say that~$\omega'$
is obtained from~$\omega$ by \emph{a tail shrinking operation},
and that~$\omega$ is obtained from~$\omega'$ by \emph{a tail stretching operation},
if  $\omega'$ has the form~$(\mu_1,\ldots,\mu_{k-1},q)$, and one of the unions of intervals~$(\mu_k;p)\cup(q;\mu_k)$
and~$(p;\mu_k)\cup(\mu_k;q)$ does not contain mirrors of~$M$.
\end{defi}

Observe that each union of intervals mentioned in Definition~\ref{shrink-stretch-def} is composed
of two mutually orthogonal straight line
portions of the same boundary circuit of~$M$.
\begin{figure}[ht]
\includegraphics[scale=.7]{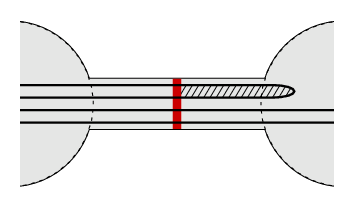}\put(-80,15){$\wideparen\mu_i=\wideparen\mu_{i+1}$}
\hskip1cm\raisebox{32pt}{$\longrightarrow$}\hskip1cm
\includegraphics[scale=.7]{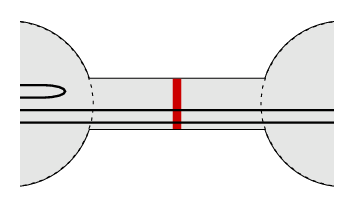}

\vskip.5cm

\includegraphics[scale=.7]{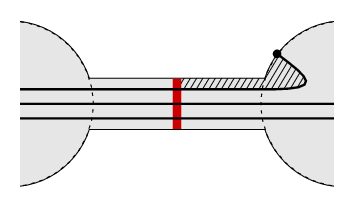}\put(-65,15){$\wideparen\mu_k$}\put(-33,53){$\wideparen p$}
\hskip1cm\raisebox{32pt}{$\longrightarrow$}\hskip1cm
\includegraphics[scale=.7]{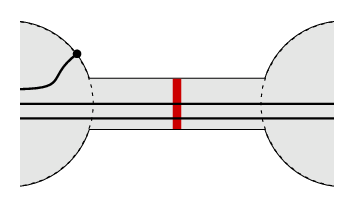}\put(-92,55){$\wideparen q$}
\caption{Pulling tight and tail shrinking operations as (half-)bigon reductions}\label{pull-tight-fig}
\end{figure}

Definitions~\ref{pull-tight-loosen-def} and~\ref{shrink-stretch-def} have a simple topological meaning, which we now explain.

Let~$M$ be a mirror diagram, and let~$(\delta_+,\delta_-)$ be a canonic dividing configuration of~$\wideparen M$
(see Definition~\ref{canonic-dividing-for-mirror-diagrams-def}) such
that~$\delta_+\cup\delta_-$ is contained in the union of $1$-handles of~$\wideparen M$.
If~$\omega=(\mu_1,\ldots,\mu_k,p)$ is a type~II splitting route, and the associated splitting path~$\wideparen\omega$
is chosen to have minimal possible number of intersections with~$\delta_+\cup\delta_-$,
then each arc~$\wideparen\omega^i$, $i=1,\ldots,k$, from Definitions~\ref{simple-route-def}
and~\ref{double-headed-splitting-route-def} intersects~$\delta_+\cup\delta_-$ exactly once.
We say in this case
that~$\wideparen\omega$ is chosen \emph{in the optimal way}.

Let~$\omega\mapsto\omega'$ be a pulling tight operation, and let~$\wideparen\omega$ and~$\wideparen\omega'$
be chosen in the optimal way. One can see that, topologically, the transition from~$\wideparen\omega$
to~$\wideparen\omega'$ is a reduction of a bigon of~$\wideparen\omega$ and~$\delta_-$.
Similarly, if~$\omega\mapsto\omega'$ is a tail shrinking operation then the
transition from~$\wideparen\omega$
to~$\wideparen\omega'$ is a reduction of a half-bigon of~$\wideparen\omega$ and~$\delta_-$.
This is illustrated in Figure~\ref{pull-tight-fig}.

\begin{defi}
A type~II splitting route~$\omega=(\mu_1,\ldots,\mu_k,p)$ and the associated
type~II splitting path are said to be \emph{reduced} if~$\mu_i\ne\mu_{i+1}$ for
all~$i=2,\ldots,k-1$, and one of the following two conditions holds:
\begin{enumerate}
\item
$\omega$ is single-headed and both intervals~$(\mu_k;p)$ and~$(p;\mu_k)$ contain at least one mirror of~$M$;
\item
$\omega$ is double-headed.
\end{enumerate}
\end{defi}

\begin{prop}\label{reducing-splitting-route-prop}
For any type~II splitting route~$\omega$ in a mirror diagram~$M$, there exists
a reduced type~II splitting route in~$M$ obtained from~$\omega$
by a finite sequence of pulling tight and tail shrinking operations.
\end{prop}

\begin{proof}
Pulling tight and tail shrinking operations shorten the corresponding splitting route,
hence it suffices to show that, whenever a type~II splitting route is not reduced, one of these operations can be applied.
The latter would be obvious without Condition~(2d) in Definitions~\ref{simple-route-def}
and~\ref{double-headed-splitting-route-def}.
This condition is designed to be robust under pulling tight and shrinking
operations, so it does not prohibit any of them. The claim follows.
\end{proof}

\begin{lemm}\label{relaxation-lem}
Let~$M\mapsto M'$ be a generalized type~II split move associated
with a type~II splitting route~$\omega$, and let~$\omega'$ be a type~II splitting route obtained
from~$\omega$ either by a pulling tight or tail shrinking
operation. Let also~$C$ be the set of boundary circuits of~$M$ preserved by the move~$M\mapsto M'$.
Then there is a choice of the generalized type~II split move~$M\mapsto M''$
associated with~$\omega'$ such that all the boundary circuits in~$C$ are also preserved by
this move.
\end{lemm}

\begin{proof}
If~$\omega'$ is not special, the statement follows from Proposition~\ref{non-separating-splitting-prop} and
the obvious fact that pulling tight and tail shrinking operations cannot take a splitting route that does not separate~$C$, to a splitting route that does.

Now suppose that~$\omega'$ is special. This means that~$\omega$ and~$\omega'$ have the form~$(\mu_1,\mu_2,\mu_2,\mu_3,p)$
and~$(\mu_1,\mu_3,p)$, respectively. There is a boundary circuit~$c$ that hits~$\mu_2$ and does not belong to~$C$,
since at least one of the boundary circuits hitting~$\mu_2$ is modified by the move~$M\mapsto M'$.
Choose the auxiliary mirror~$\mu_0$ for~$M\mapsto M''$ on~$c$ in a vicinity of~$\mu_2$. Then the extension that introduces~$\mu_0$ will not modify any boundary circuit in~$C$,
and the splitting path~$(\mu_1,\mu_0,\mu_0,\mu_3,p)$ will not separate~$C$. The claim follows.
\end{proof}

\begin{defi}\label{wandering-def}
Let~$M$ be a mirror diagram, and let~$\omega$ and~$\omega'$ be type~II splitting routes in~$M$
having the form~$\omega=(\mu_1,\mu_2,\ldots,\mu_k,p)$ and
$\omega'=(\mu_1',\mu_1'',\mu_2,\ldots,\mu_k,p)$. Assume that~$M$ has an inessential boundary circuit~$c$
such that the following conditions hold:
\begin{enumerate}
\item
$c$ has the form of the boundary of a rectangle~$r$, and~$\interior(r)\cap E_M=\varnothing$;
\item
$\mu_1,\mu_1',\mu_1''\in c$;
\item
$p\notin c$.
\end{enumerate}
Then we say that~$\omega$ and~$\omega'$ are obtained from one another by \emph{a head wandering operation}.
\end{defi}
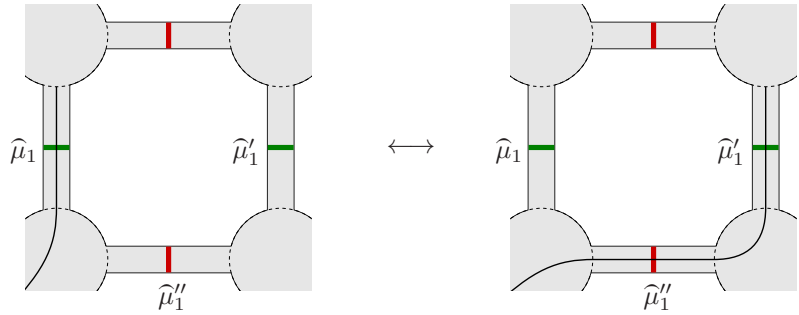
\begin{figure}[ht]
\includegraphics[scale=0.65]{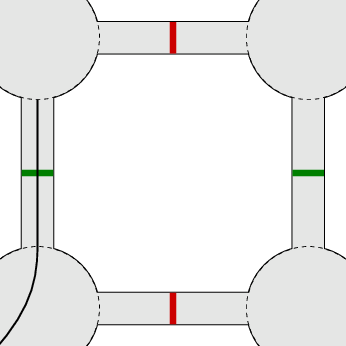}\put(-114,50){$\wideparen\mu_1$}\put(-30,50){$\wideparen\mu_1'$}%
\put(-58,-5){$\wideparen\mu_1''$}
\hskip1cm\raisebox{52pt}{$\longleftrightarrow$}\hskip1cm
\includegraphics[scale=0.65]{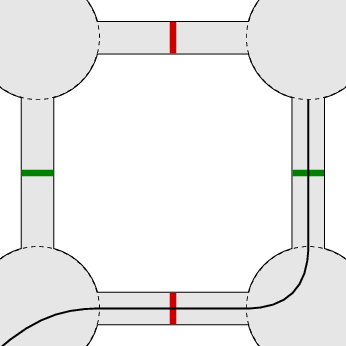}\put(-114,50){$\wideparen\mu_1$}\put(-30,50){$\wideparen\mu_1'$}%
\put(-58,-5){$\wideparen\mu_1''$}
\caption{Head wandering}\label{head-wandering-fig}
\end{figure}
The idea of this move is illustrated in Figure~\ref{head-wandering-fig}.

\begin{defi}\label{conversion-def}
Let~$M$ be a mirror diagram, and let~$\omega=(\mu_1,\mu_2,\ldots,\mu_k,p)$
be a single-headed type~II splitting route in~$M$. Let also~$\omega'$ be a double-headed type~II splitting route in~$M$
of the form~$\omega'=(\mu_1,\mu_2,\ldots,\mu_k,\mu_{k+1},p')$. Assume that~$M$ has an inessential boundary circuit~$c$
such that the following conditions hold:
\begin{enumerate}
\item
$c$ has the form of the boundary of a rectangle~$r$, and~$\interior(r)\cap E_M=\varnothing$;
\item
$\mu_{k+1},p,p'\in c$.
\end{enumerate}
Then we say~$\omega'$ is obtained from~$\omega$ by \emph{a tail-to-head conversion},
and~$\omega$ is obtained from~$\omega'$ by \emph{a head-to-tail conversion}.
\end{defi}
\begin{figure}[ht]
\includegraphics[scale=0.65]{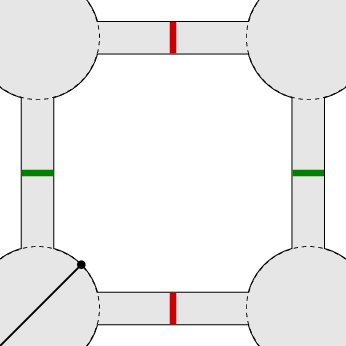}\put(-79,28){$\wideparen p$}
\hskip1cm\raisebox{52pt}{$\longleftrightarrow$}\hskip1cm
\includegraphics[scale=0.65]{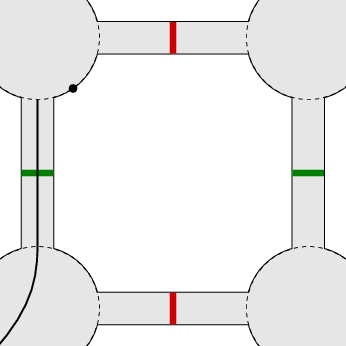}\put(-124,50){$\wideparen\mu_{k+1}$}\put(-85,71){$\wideparen p'$}
\caption{Tail-to-head and head-to-tail conversions}\label{conversion-fig}
\end{figure}
The tail-to-head and head-to-tail conversion moves are illustrated in Figure~\ref{conversion-fig}.

\begin{defi}\label{similarity-def}
Two splitting routes~$\omega$ and~$\omega'$ are said to be \emph{similar} if
there is a sequence of splitting routes
$$\omega=\omega_0,\omega_1,\ldots,\omega_N=\omega'$$
in which~$\omega_i$ and~$\omega_{i+1}$ are either equivalent
or obtained from each other by one of the transformations introduced in
Definitions~\ref{pull-tight-loosen-def}, \ref{shrink-stretch-def}, \ref{wandering-def},
and \ref{conversion-def}, for all~$i=0,\ldots,N-1$.
\end{defi}

\begin{prop}\label{similarity-prop}
Let~$\omega$ and~$\omega'$ be similar type~II splitting routes in a mirror diagram~$M$,
and let~$M\xmapsto\eta M_1$ and~$M\xmapsto\zeta M_2$
be generalized type~II split moves associated with~$\omega$ and~$\omega'$, respectively.
Let also~$C$ be the set of essential boundary circuits of~$M$
preserved by both moves.

Then the transformation~$M_1\xmapsto{\zeta\circ\eta^{-1}}M_2$ admits a $C$-neat decomposition into
type~I elementary moves.
\end{prop}

\begin{proof}
\begin{figure}[ht]
\includegraphics[scale=0.65]{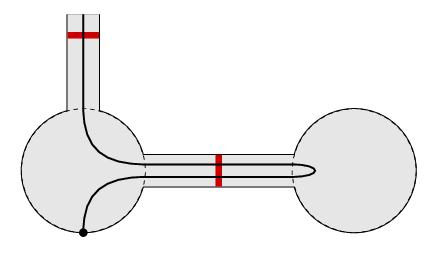}\put(-115,22){$\wideparen x$}\put(-28,22){$\wideparen y$}%
\put(-90,10){$\wideparen\mu_k=\wideparen\mu_{k-1}$}\put(-103,60){$\wideparen\mu_{k-2}$}\put(-114,-4){$\wideparen p$}

\vskip.5cm

\begin{tabular}{ccc}
\includegraphics[scale=0.65]{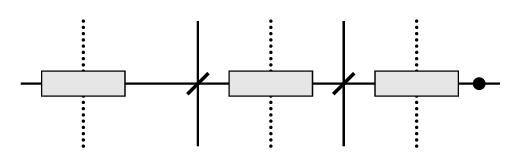}\put(-155,0){$M$}\put(-115,16){$\mu_k$}\put(-77,16){$\mu_{k-2}$}%
\put(-15,18){$p$}\put(-164,24){$x$}\put(-103,50){$y$}
&\raisebox{24pt}{$\stackrel{\eta_1}\longrightarrow$}&
\includegraphics[scale=0.65]{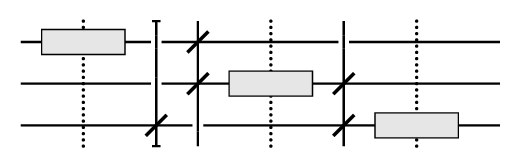}\put(-155,0){$M_1'$}\put(-116,50){$y'$}
\\
$\big\downarrow$\hbox to 0pt{$\scriptstyle\zeta_1$\hss}&&
$\big\downarrow$\hbox to 0pt{$\scriptstyle\chi$\hss}
\\
\includegraphics[scale=0.65]{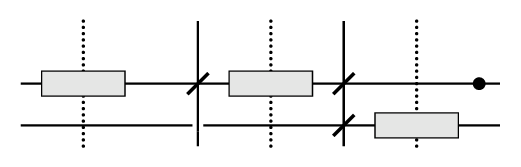}\put(-155,0){$M_2'$}\put(-115,17){$\mu_k'$}\put(-166,24){$x'$}\put(-15,17){$p'$}
&\raisebox{24pt}{$\stackrel\xi\longrightarrow$}&
\includegraphics[scale=0.65]{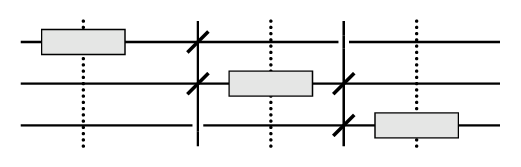}\put(-155,0){$M_3'=M_4'$}
\end{tabular}
\caption{Proof of Proposition~\ref{similarity-prop} in the pulling tight case, $\omega$ and~$\omega'$ are single-headed,
$k>3$}\label{pull-split-fig1}
\end{figure}
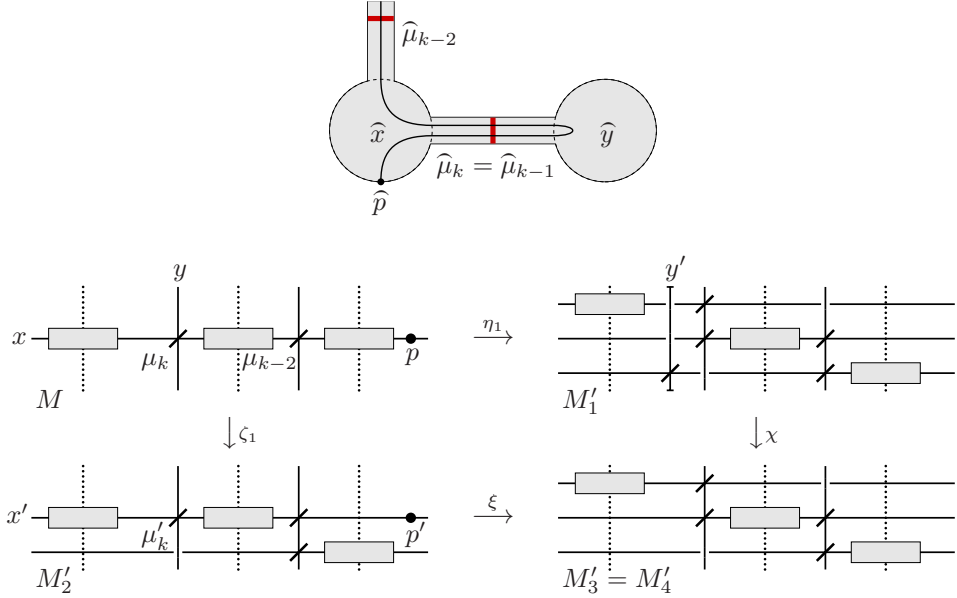
It suffices to prove the proposition in the case when~$\omega\mapsto\omega'$
is one of the transformations introduced in Definitions~\ref{pull-tight-loosen-def}, \ref{shrink-stretch-def}, \ref{wandering-def},
and \ref{conversion-def}, and in the case when~$\omega=\omega'$ is a special route.
Indeed, let~$\omega_0,\ldots,\omega_N$ be as in Definition~\ref{similarity-def}, and let~$M\xmapsto{\eta_i}M_i'$
be a generalized type~II split move associated with~$\omega_i$, $i=0,1,\ldots,N$, and preserving all the boundary circuits in~$C$.

Generalized type~II merge moves are safe-to-postpone, which implies, by induction in~$N$,
that
$$M_0'\xmapsto{\eta_1\circ\eta_0^{-1}}M_1'\xmapsto{\eta_2\circ\eta_1^{-1}}\ldots
\xmapsto{\eta_N\circ\eta_{N-1}^{-1}}M_N'\xmapsto{\eta_N^{-1}}M$$%
is a $C$-neat decomposition of~$M_0'\xmapsto{\eta_0^{-1}}M$. Therefore, the first~$N$ moves
in this sequence form a $C$-neat decomposition of the move~$M_1=M_0'\xmapsto{\zeta\circ\eta^{-1}}M_N'=M_2$.

Thus, in what follows we assume~$N=1$.

\smallskip\noindent\emph{Case 1}: $\omega\mapsto\omega'$ is a pulling tight operation.\\
We use the notation from Definition~\ref{pull-tight-loosen-def}, that is, we assume that~$\omega'$ is obtained from~$\omega$
by removing the~$i$th and the~$(i+1)$st entries, which are equal. Suppose that~$\omega$ and~$\omega'$
are single-headed.

The first~$k-1-i$ steps of a canonical decomposition of the move~$M\xmapsto\eta M_1$ coincide---combinatorially---with
those of a canonical decomposition of the move~$M\xmapsto\zeta M_2$. This reduces
the general case to the case when~$k=i+1$,
that is, when the two entries in~$\omega$ being cancelled are the last two mirrors.

Let~$M\xmapsto{\eta_1}M_1'$ be the composition of the first three split moves
in a canonical decomposition of the generalized type~II split move~$M\xmapsto\eta M_1$,
and let~$M_1'\xmapsto{\eta_2}M_1$ be the remaining part of the decomposition.
Let also~$M\xmapsto{\zeta_1}M_2'$ be the first split move in a canonical
decomposition of the move~$M\xmapsto\zeta M_2$, and~$M_2'\xmapsto{\zeta_2}M_2$
the remaining part.
\begin{figure}[ht]
\includegraphics[scale=0.65]{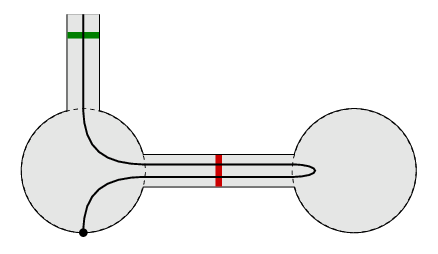}\put(-115,22){$\wideparen x$}\put(-28,22){$\wideparen y$}%
\put(-84,10){$\wideparen\mu_3=\wideparen\mu_2$}\put(-103,60){$\wideparen\mu_1$}

\vskip.5cm

\begin{tabular}{ccc}
\includegraphics[scale=0.65]{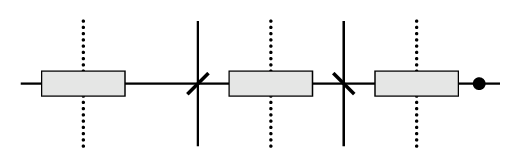}\put(-155,0){$M$}\put(-113,16){$\mu_3$}\put(-53,16){$\mu_1$}%
\put(-15,18){$p$}\put(-164,24){$x$}\put(-103,50){$y$}
&\raisebox{24pt}{$\stackrel{\eta_1}\longrightarrow$}&
\includegraphics[scale=0.65]{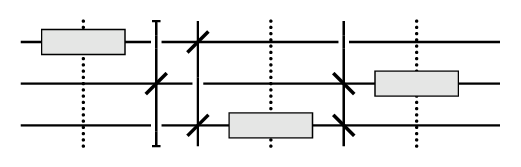}\put(-155,0){$M_1'$}\put(-116,50){$y'$}
\\
$\big\downarrow$\hbox to 0pt{$\scriptstyle\zeta_1$\hss}
\\
\includegraphics[scale=0.65]{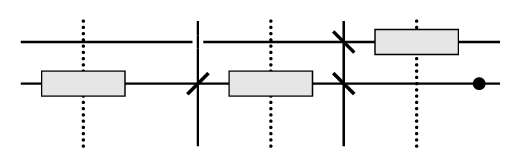}\put(-155,0){$M_2'$}\put(-115,17){$\mu_3'$}\put(-166,24){$x'$}\put(-15,17){$p'$}
&&
\raisebox{24pt}{$\big\downarrow$\hbox to 0pt{$\scriptstyle\chi$\hss}}
\\
$\big\downarrow$\hbox to 0pt{$\scriptstyle\xi$\hss}
\\
\includegraphics[scale=0.65]{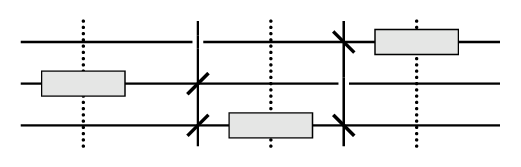}\put(-155,0){$M_4'$}
&&
\includegraphics[scale=0.65]{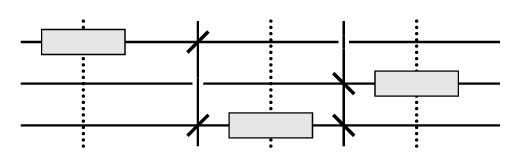}\put(-155,0){$M_3'$}
\end{tabular}
\caption{Proof of Proposition~\ref{similarity-prop} in the pulling tight case, $\omega$ and~$\omega'$
are single-headed, $k=3$}\label{pull-split-fig2}
\end{figure}
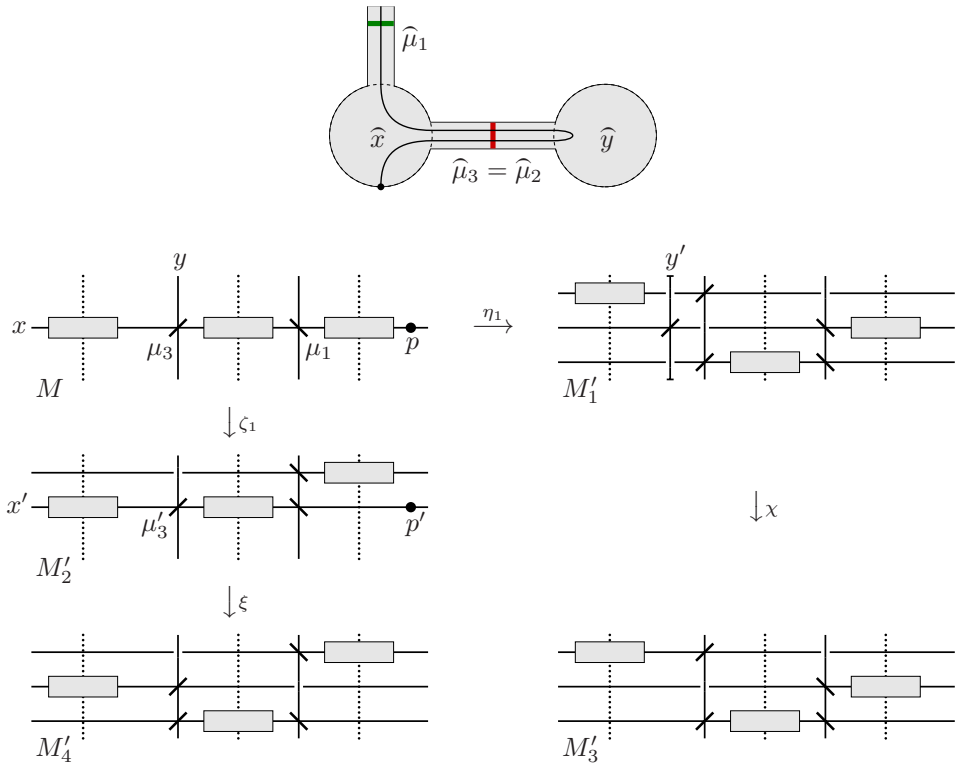

Denote by~$x$ the occupied level of~$M$ containing~$\mu_k=\mu_{k-1}$ and~$\mu_{k-2}$,
and by~$y$ the other occupied level passing through~$\mu_k$.
The occupied level~$y$ has two successors in~$M_1'$ one of which $y'$, say, carries
only one mirror of~$M_1'$. Let~$M_1'\xmapsto\chi M_3'$ be the elimination move
that removes this mirror and the level~$y'$.

The occupied level~$x$ has two successors in~$M_2'$ only one of which contains
a successor of~$\mu_k$. Denote these successors of~$x$ and~$\mu_k$ by~$x'$ and~$\mu_k'$, respectively.
Denote also by~$p'$ a point on~$x'$ such that~$\wideparen p'$
is the image of an interior point of~$\wideparen M$ under~$h_M^{M_2'}$.
Let~$M_2'\xmapsto\xi M_4'$ be a type~I split move associated with the splitting
route~$(\mu_k',p')$. We can choose it to preserve all boundary circuits in~$C$.

We claim that the transformation~$M_3'\xmapsto\psi M_4'$, where~$\psi=\xi\circ\zeta_1\circ\eta_1^{-1}\circ\chi^{-1}$,
admits a $C$-neat decomposition into jump moves. Indeed, if~$k>3$, then~$\mu_{k-2}$ is of type~`$\diagup$',
and one can directly check that the combinatorial types of~$M_3'$ and~$M_4'$ just coincide.
This is illustrated in Figure~\ref{pull-split-fig1}. If~$k=3$, then~$\mu_{k-2}=\mu_1$ is of type~`$\diagdown$',
and the combinatorial type of~$M_4'$ differs from that of~$M_3'$ only in the mutual position
of two successors of~$x$; see Figure~\ref{pull-split-fig2}.

The moves~$M_1'\xmapsto\chi M_3'$ and~$M_2'\xmapsto\xi M_4'$ are friendly to the moves~$M_1'\xmapsto{\eta_2}M_1$
and~$M_2'\xmapsto{\zeta_2}M_2$, respectively, so we can define
the moves~$M_3'\xmapsto{{\eta_2}^\chi}M_3$ and~$M_4'\xmapsto{{\zeta_2}^\xi}M_4$,
again preserving all the boundary circuits in~$C$. The scheme of all the involved moves is shown in Figure~\ref{comm-diagr-pull-split-fig}.
\begin{figure}[ht]
\includegraphics{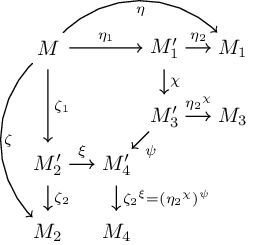}
\caption{The scheme of the moves in the proof of Proposition~\ref{similarity-prop}, the pulling tight case, single-headed
splitting routes}\label{comm-diagr-pull-split-fig}
\end{figure}

Now note that the splitting routes of the moves~$M_3'\xmapsto{{\eta_2}^\chi}M_3$
and~$M_4'\xmapsto{{\zeta_2}^\xi}M_4$ originate from the coincident starting portions of the splitting routes~$\omega$ and~$\omega'$.
One can see from this that~$M_4'\xmapsto{{\zeta_2}^\xi}M_4$ is actually the same thing as~$M_4'\xmapsto{({\eta_2}^\chi)^\psi}M_4$.
It follows from Lemmas~\ref{elimination-commutes-with-split-lem}, \ref{split-commute-with-jump-lem},
\ref{type-i-split-commutes-with-generalized-type-ii-split-lem}, and~\ref{neutral-move-decomposition-lem}
that all three transformations
$$M_1\xmapsto{{\eta_2}^\chi\circ\chi\circ\eta_2^{-1}}M_3,\quad
M_3\xmapsto{{\zeta_2}^\xi\circ\psi\circ({\eta_2}^\chi)^{-1}}M_4,\quad
M_4\xmapsto{\zeta_2\circ\xi^{-1}\circ({\zeta_2}^\xi)^{-1}}M_2$$
admit $C$-neat decompositions into type~I elementary moves.

Now suppose that~$\omega$ and~$\omega'$ are double-headed. Then the first~$k-2-i$ steps
of canonical decompositions of the moves~$M\xmapsto\eta M_1$ and~$M\xmapsto\zeta M_2$
are the same from the combinatorial point of view, and if~$k>i+2$, then the rest of the proof is
literally the same as above, since, starting from the second step of a canonical decomposition, we deal with single-headed splitting routes.

Suppose~$k=i+2>4$. The above scheme still works with the only change occurring in the numeration of
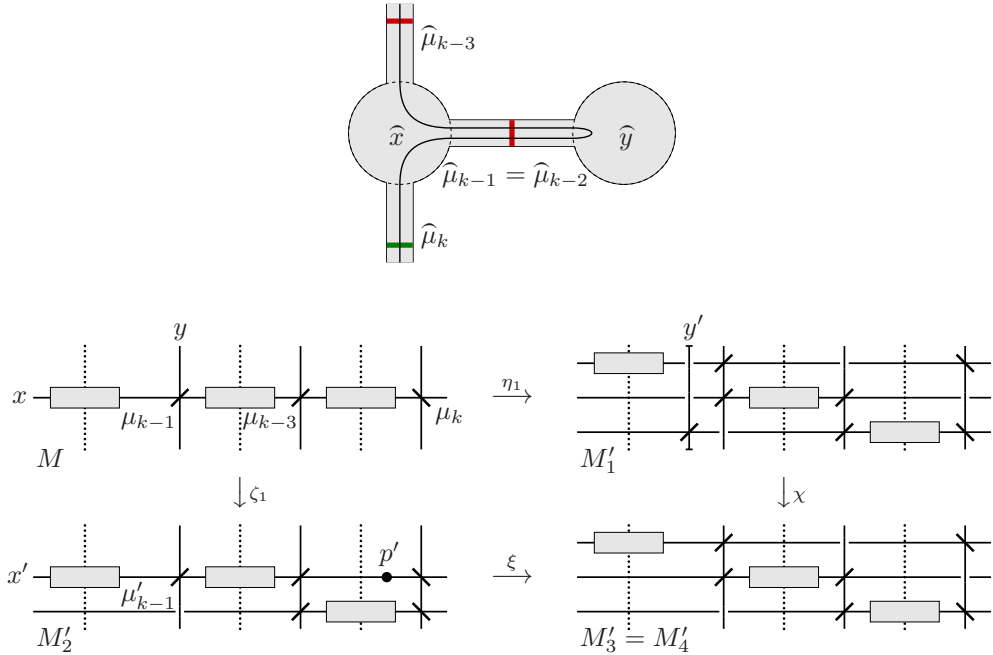
\begin{figure}[ht]
\includegraphics[scale=0.65]{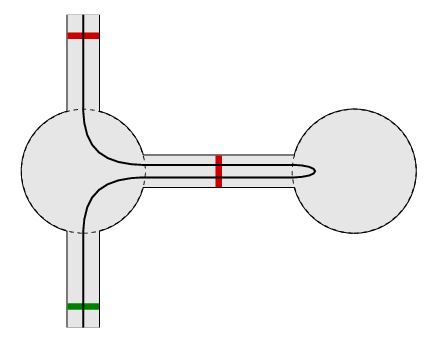}\put(-115,49){$\wideparen x$}\put(-28,49){$\wideparen y$}%
\put(-95,35){$\wideparen\mu_{k-1}=\wideparen\mu_{k-2}$}\put(-103,87){$\wideparen\mu_{k-3}$}\put(-103,12){$\wideparen\mu_k$}

\vskip.5cm

\begin{tabular}{ccc}
\includegraphics[scale=0.65]{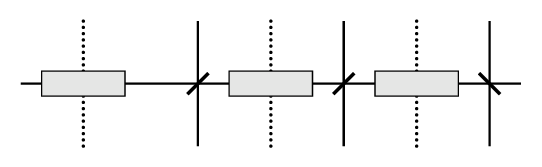}\put(-162,0){$M$}\put(-130,16){$\mu_{k-1}$}\put(-84,16){$\mu_{k-3}$}%
\put(-11,18){$\mu_k$}\put(-171,24){$x$}\put(-110,50){$y$}
&\raisebox{24pt}{$\stackrel{\eta_1}\longrightarrow$}&
\includegraphics[scale=0.65]{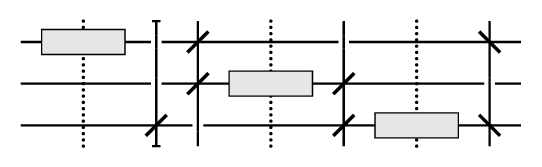}\put(-162,0){$M_1'$}\put(-123,50){$y'$}
\\
$\big\downarrow$\hbox to 0pt{$\scriptstyle\zeta_1$\hss}&&
$\big\downarrow$\hbox to 0pt{$\scriptstyle\chi$\hss}
\\
\includegraphics[scale=0.65]{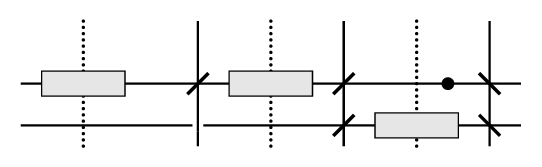}\put(-162,0){$M_2'$}\put(-130,17){$\mu_{k-1}'$}\put(-173,24){$x'$}\put(-32,32){$p'$}
&\raisebox{24pt}{$\stackrel\xi\longrightarrow$}&
\includegraphics[scale=0.65]{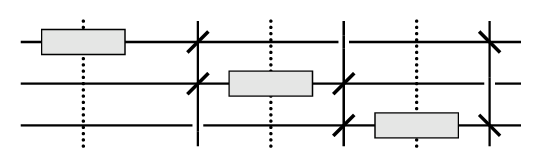}\put(-162,0){$M_3'=M_4'$}
\end{tabular}
\caption{Proof of Proposition~\ref{similarity-prop} in the pulling tight case, $\omega$ and~$\omega'$ are double-headed,
$k>4$}\label{pull-split-fig3}
\end{figure}
the mirrors, which should be shifted by~$1$. The discussed parts of the diagrams are shown in Figure~\ref{pull-split-fig3}.

Finally, suppose that~$k=4$, $\omega=(\mu_1,\mu_2,\mu_2,\mu_3,p)$, and $\omega'=(\mu_1,\mu_3,p)$.
Let~$M\xmapsto{\zeta_1}M_2'$ be the first move in a canonical decomposition of~$M\xmapsto\zeta M_2$,
and let~$M_2'\xmapsto{\zeta_2}M_2$ be the remaining part of the decomposition.
By definition, $M\xmapsto{\zeta_1}M_2'$ is an extension move, and~$M_2'\xmapsto{\zeta_2}M_2$
is a generalized type~II split move associated with the double-headed type~II splitting path~$\omega''=(\mu_1,\mu_0,\mu_0,\mu_3,p)$,
where~$\mu_0$ is the auxiliary mirror of the move~$M\xmapsto\zeta M_2$.

At least one of the sequences~$(\mu_1,\mu_0,\mu_0,\mu_2,\mu_2,\mu_3,p)$ and~$(\mu_1,\mu_2,\mu_2,\mu_0,\mu_0,\mu_3,p)$
is a double-headed type~II splitting route in~$M_2'$.
We denote it by~$\omega'''$. Which one of the two can be taken for~$\omega'''$ depends on the cyclic
order of~$\mu_0,\mu_1,\mu_2,\mu_3$ on the occupied level of~$M_2'$ that contains all of them,
and, in the case~$\mu_1=\mu_3$, also on the position of~$p$.
Both splitting routes~$\omega$ and~$\omega''$ are obtained from~$\omega'''$ by pulling tight operations.

Let~$M_2'\xmapsto\xi M_3$ be a generalized type~II split move associated with~$\omega'''$. Choose it
to preserve all the boundary circuits in~$C$.

The extension move~$M\xmapsto{\zeta_1}M_2'$ is friendly to~$M\xmapsto\eta M_1$, and the move~$M_2'\xmapsto{\eta^{\zeta_1}}M_1'$
resembling~$M\xmapsto\eta M_1$ (chosen to preserve the boundary circuits in~$C$)
is a generalized type~II split move associated with~$\omega$, which is also a type~II splitting route in~$M_2'$.
Consult Figure~\ref{comm-diagr-pull-split-fig2} for the general scheme.
\begin{figure}[ht]
\includegraphics{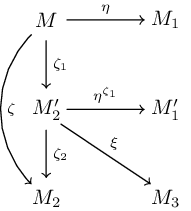}
\caption{The scheme of the moves in the proof of Proposition~\ref{similarity-prop}, the pulling tight case,
double-headed splitting routes, $k=4$}\label{comm-diagr-pull-split-fig2}
\end{figure}

The transformation~$M_1\xmapsto{\eta^{\zeta_1}\circ\zeta_1\circ\eta^{-1}}M_1'$ 
is a type~I extension move preserving the boundary circuits in~$C$
by Lemma~\ref{elimination-commutes-with-split-lem} (applied to the moves~$M_2'\xmapsto{\eta^{\zeta_1}}M_1'$
and~$M_2'\xmapsto{\zeta_1^{-1}}M$), and the transformations~$M_1'\xmapsto{\xi\circ(\eta^{\zeta_1})^{-1}}M_3$,
$M_3\xmapsto{\zeta_2\circ\xi^{-1}}M_2$ admit a $C$-neat decompositions
into type~I elementary moves by the established above cases of Proposition~\ref{similarity-prop}.
This completes the proof in the pulling tight case.

\smallskip\noindent\emph{Case 2}: $\omega\mapsto\omega'$ is a tail shrinking operation.\\
We use the notation from Definition~\ref{shrink-stretch-def}.
Let~$M\xmapsto{\eta_1}M_1'$ be the first split move of a canonical decomposition of~$M\xmapsto\eta M_1$,
and let~$M_1'\xmapsto{\eta_2}M_1$ be the remaining part. From the combinatorial point of view
the move~$M\xmapsto{\eta_1}M_1'$ is a type~I extension move,
whose inverse is friendly to~$M_1'\xmapsto{\eta_2}M_1$, so~$M\xmapsto\zeta M_2$
can be interpreted as~$M\xmapsto{{\eta_2}^{\eta_1^{-1}}}M_2$. We leave the details, which are
easy in this case, to the reader.

\smallskip\noindent\emph{Case 3}: $\omega=\omega'$ is a special type~II splitting route,
$\omega=(\mu_1,\mu_2,p)$.\\
Let~$\nu_1$, $x_1$ be the auxiliary mirror and the auxiliary level, respectively,
of the move~$M\xmapsto\eta M_1$,
and let~$\nu_2$, $x_2$ be the auxiliary mirror  and the auxiliary level, respectively, of the move~$M\xmapsto\zeta M_2$.
Let also~$M\xmapsto{\eta_1}M_1'$ and~$M\xmapsto{\zeta_1}M_2'$ be the extension moves
from which the canonic decompositions of the moves~$M\xmapsto\eta M_1$ and~$M\xmapsto\zeta M_2$, respectively,
start, and let~$M_1'\xmapsto{\eta_2}M_1$ and~$M_2'\xmapsto{\zeta_2}M_2$ be the remaining
parts of the canonical decompositions. If~$\nu_1=\nu_2$, there is nothing to prove, so we assume~$\nu_1\ne\nu_2$,
which implies~$x_1\ne x_2$.
\begin{figure}[ht]
\includegraphics{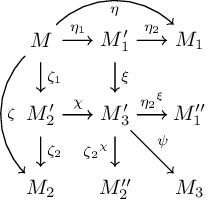}
\caption{The scheme of the moves in the proof of Proposition~\ref{similarity-prop} in the
case~$\omega=\omega'=(\mu_1,\mu_2,p)$}\label{coincident-special-similarity-fig}
\end{figure}

Denote by~$M_3'$ the mirror diagram obtained from~$M$ by adding both occupied levels~$x_1$ and~$x_2$
and both mirrors~$\nu_1$, $\nu_2$. We have two extension moves~$M_1'\xmapsto\xi M_3'$ and~$M_2'\xmapsto\chi M_3'$
such that~$\xi\circ\eta_1=\chi\circ\zeta_1$, and these moves are friendly to~$M_1'\xmapsto{\eta_2}M_1$
and~$M_2'\xmapsto{\zeta_2}M_2$, respectively; see Figure~\ref{coincident-special-similarity-fig}.
Thus, we have generalized type~II split moves~$M_3'\xmapsto{{\eta_2}^\xi}M_1''$ and~$M_3'\xmapsto{{\zeta_2}^\chi}M_2''$
resembling~$M_1'\xmapsto{\eta_2}M_1$ and~$M_2'\xmapsto{\zeta_2}M_2$, respectively.

At least one of the sequences~$(\mu_1,\nu_1,\nu_1,\nu_2,\nu_2,\mu_2,p)$ and~$(\mu_1,\nu_2,\nu_2,\nu_1,\nu_1,\mu_2,p)$ is a double-headed
type~II splitting route in~$M_3'$, which we denote by~$\omega''$. There exists a generalized type~II split move~$M_3'\xmapsto\psi M_3$ associated with~$\omega''$ and
preserving the boundary circuits in~$C$.

The splitting routes that the moves~$M_3'\xmapsto{{\eta_2}^\xi}M_1''$ and~$M_3'\xmapsto{{\zeta_2}^\chi}M_2''$
are associated with, which are~$(\mu_1,\nu_1,\nu_1,\mu_2,p)$ and~$(\mu_1,\nu_2,\nu_2,\mu_2,p)$, respectively,
are obtained from~$\omega''$ by pulling tight operations and are not special. Thus,
the existence of a $C$-neat decomposition of the transformations~$M_1''\xmapsto{\psi\circ({\eta_2}^\xi)^{-1}}M_3$
and~$M_3\xmapsto{{\zeta_2}^\chi\circ\psi^{-1}}M_2''$ into type~I elementary moves
follows from the previously considered cases of Proposition~\ref{similarity-prop}.
The existence of such a decomposition for the transformations~$M_1\xmapsto{{\eta_2}^\xi\circ\xi\circ\eta_2^{-1}}M_1''$
and~$M_2''\xmapsto{\zeta_2\circ\chi^{-1}\circ({\zeta_2}^\chi)^{-1}}M_2$ follows from Lemma~\ref{elimination-commutes-with-split-lem}.

\smallskip\noindent\emph{Case 4}: $\omega\mapsto\omega'$ is a head wandering operation.\\
We use the notation from Definition~\ref{wandering-def}. Suppose that~$\omega$ and~$\omega'$
are single-headed. The first~$k-1$ steps in canonical decompositions
of the moves~$M\xmapsto\eta M_1$ and~$M\xmapsto\zeta M_2$ are combinatorially the same, and the type~II splitting routes that
the remaining parts of the decompositions are associated with are still related by a head wandering operation.
So, it suffices to consider the case~$k=1$, $\omega=(\mu,p)$, $\omega'=(\mu',\mu'',p)$.

Denote by~$x$ the occupied level of~$M$ containing~$p$ and~$\mu$, and by~$y$
the occupied level of~$M$ containing~$\mu'$ and~$\mu''$. Denote also by~$c_1$
and~$c_2$ the inessential boundary circuits corresponding to~$c$ in~$M_1$ and~$M_2$, respectively.
Up to various symmetries the portion of~$M$ in question looks as shown at the top of Figure~\ref{head-wand-split-fig1}.
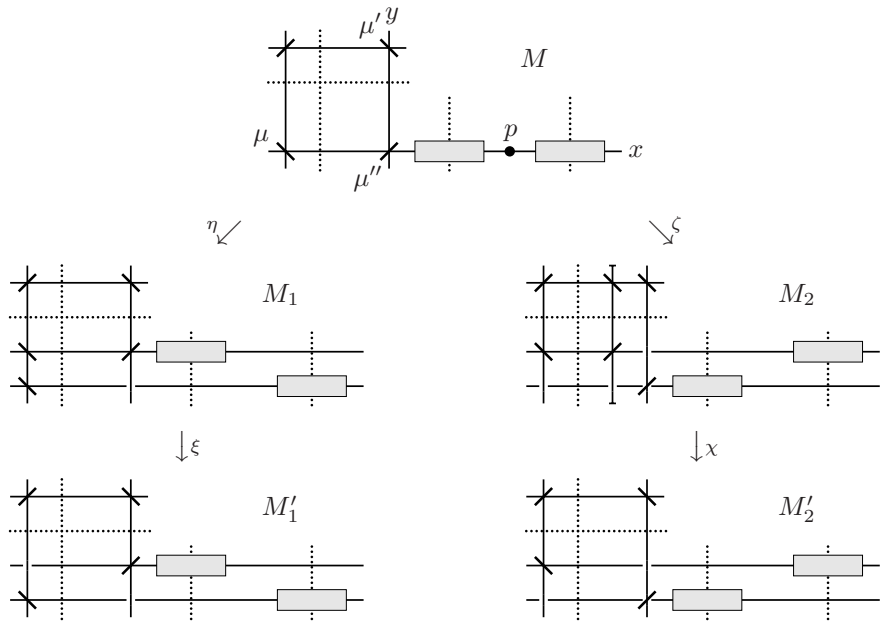
\begin{figure}[ht]
\includegraphics[scale=0.65]{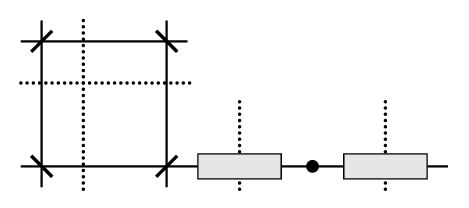}\put(-146,17){$\mu$}\put(-108,0){$\mu''$}\put(-106,58){$\mu'$}%
\put(-51,19){$p$}\put(-45,45){$M$}\put(-4,11){$x$}\put(-96,63){$y$}
\vskip.3cm
\begin{tabular}{ccc}
\hbox to 1cm{\hss}${}^\eta\kern-.4em\swarrow$
&\hbox to 1cm{\hss}&
$\searrow\kern-.4em{}^\zeta$\hbox to 1cm
\\
\includegraphics[scale=0.65]{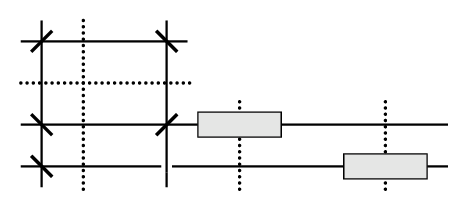}\put(-45,45){$M_1$}
&&
\includegraphics[scale=0.65]{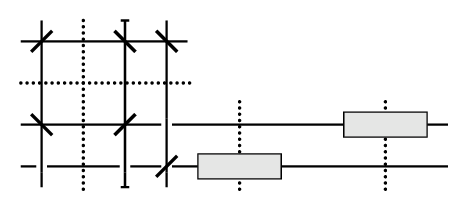}\put(-45,45){$M_2$}
\\$\big\downarrow\scriptstyle\xi$&&$\big\downarrow\scriptstyle\chi$\\
\includegraphics[scale=0.65]{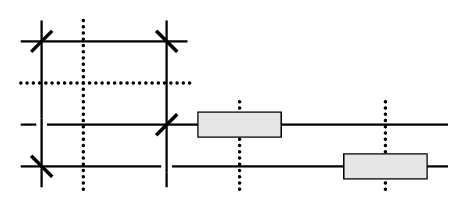}\put(-45,45){$M_1'$}
&&
\includegraphics[scale=0.65]{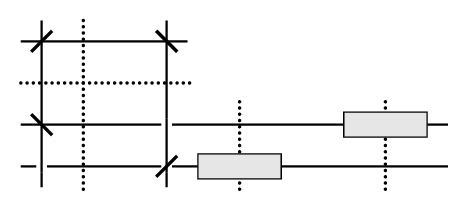}\put(-45,45){$M_2'$}
\end{tabular}
\caption{Proof of Proposition~\ref{similarity-prop} in the head wandering case, $\omega$ and~$\omega'$ are single-headed}\label{head-wand-split-fig1}
\end{figure}

The mirror~$\mu$ has a unique successor in~$M_1$ that is hit by~$c_1$.
Let~$M_1\xmapsto\xi M_1'$ be the elementary bypass removal that deletes this successor.
This move will preserve all boundary circuits except~$c_1$, which is inessential,
and another one that is already modified by~$M\xmapsto\eta M_1$. Thus
all the boundary circuits in~$C$ are preserved in~$M_1'$.

The occupied level~$y$ has a unique successor in~$M_2$ containing a portion of~$c_2$.
It carries exactly two mirrors of~$M_2$, which are successors of~$\mu'$ and~$\mu''$.
Let~$M_2\xmapsto\chi M_2'$ be the bridge removal that deletes~$y$ together with
these two mirrors. This move preserves all boundary circuits except~$c_2$
and another one that is already modified by~$M\xmapsto\eta M_1$. Thus
all the boundary circuits in~$C$ are also preserved in~$M_2'$.

One can see that~$M_1'$ and~$M_2'$ are obtained from one another by exchanging the two
occupied levels that are successors of~$x$. One can also see that
there is no obstruction to exchange them by means of jump moves preserving all the boundary
circuits that have not yet been modified by other moves discussed above.
To complete the proof in this case it remains to decompose neatly the jump moves
as well as the bridge addition~$M_2'\xmapsto{\chi^{-1}}M_2$ into
elementary type~I moves, which is possible by Lemma~\ref{neutral-move-decomposition-lem}.

Now suppose that~$\omega$ and~$\omega'$ are double-headed. If~$k>2$, then the first
moves in canonical decompositions of~$M\xmapsto\eta M_1$ and~$M\xmapsto\zeta M_2$
are combinatorially the same, and the type~II splitting routes responsible for the
remaining parts of the canonical decompositions are single-headed, so we can proceed as above.

Suppose that~$k=2$ and~$\mu_1\ne\mu_2$. Denote by~$x$ the occupied level of~$M$
passing through~$\mu_1$ and~$\mu_2$. Since~$\mu_1\ne\mu_2$, the mirror~$\mu_2$ lies outside~$c$,
which means that the union of boundary circuits in~$C$ does not cover~$x\setminus c$,
and hence there exists a generalized type~II split move~$M\xmapsto\xi M_3$
associated with~$\omega$ and preserving the boundary circuits in~$C$, such that the auxiliary mirror of this move is outside~$c$.
We have seen in Case~3 above that~$M_1\xmapsto{\xi\circ\eta^{-1}}M_3$ admits a $C$-neat
decomposition into type~I elementary moves. This reduces the currently
considered case to the subcase when the auxiliary mirror of
the move~$M\xmapsto\eta M_1$ is outside~$c$. We assume this in the sequel.

Denote by~$\mu_0$ the auxiliary mirror. Let~$M\xmapsto{\eta_1}M_1'$
be the first move of a canonical decomposition of~$M\xmapsto\eta M_1$,
and let~$M_1'\xmapsto{\eta_2}M_1$ be the remaining part.
The type~I extension move~$M\xmapsto{\eta_1}M_1'$ is friendly to~$M\xmapsto\zeta M_2$,
so we can define~$M_1'\xmapsto{\zeta^{\eta_1}}M_2'$ in such a way that the transformation~$M_2'\xmapsto{\zeta\circ\eta_1^{-1}\circ(\zeta^{\eta_1})^{-1}}M_2$
is a type~I elimination move (see Lemma~\ref{elimination-commutes-with-split-lem}). This move is associated with
the type~II splitting route~$\omega'=(\mu_1',\mu_1'',\mu_2,p)$ in~$M_1'$.

Let~$\omega''$ be the type~II splitting route~$(\mu_1',\mu_1'',\mu_0,\mu_0,\mu_2,p)$ in~$M_1'$,
and let~$M_1'\xmapsto\chi M_3$ be the respective generalized type~II split
move; see the scheme in Figure~\ref{head-wandering-special-fig}.
\begin{figure}[ht]
\includegraphics{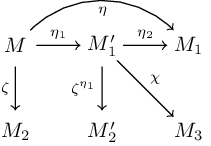}
\caption{The scheme of the moves in the proof of Proposition~\ref{similarity-prop} in the
head wandering case with special~$\omega$}\label{head-wandering-special-fig}
\end{figure}
By construction, the generalized type~II split move~$M_1'\xmapsto{\eta_2}M_1$
is associated with the splitting route~$\omega_1=(\mu_1,\mu_0,\mu_0,\mu_2,p)$ in~$M_1'$.

Thus, the splitting routes $\omega_1$ and~$\omega''$ are related by a head wandering operation,
and neither of them is special. As has been shown above, in this case,
the transformation~$M_1\xmapsto{\chi\circ{\eta_2^{-1}}}M_3$ admits a $C$-neat decomposition into
type~I elementary moves.

The splitting route~$\omega'$ in~$M_1'$ is obtained from~$\omega''$ by a pulling tight operation.
So, the required decomposition for~$M_3\xmapsto{\zeta^{\eta_1}\circ\chi^{-1}}M_2'$ exists by Case~1 considered above.

It remains to consider the subcase when~$\omega$ and~$\omega'$ have the form~$\omega=(\mu,\mu,p)$,
$\omega'=(\mu',\mu'',\mu,p)$. Denote by~$\mu_0$ and~$y$ the auxiliary mirror and the auxiliary occupied level
of the move~$M\xmapsto\eta M_1$,
by~$x$ the occupied level of~$M$ containing~$\mu$ and~$\mu_0$ (and~$\mu''$),
and by~$z$ the occupied level of~$M$ containing~$\mu'$ and~$\mu''$.
\begin{figure}[ht]
\begin{tabular}{ccccccc}
\includegraphics[scale=0.65]{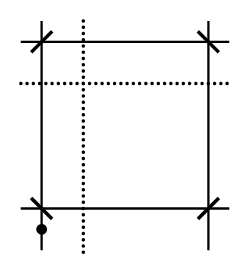}\put(-76,25){$\mu$}\put(-73,12){$p$}\put(-28,10){$\mu''$}%
\put(-5,17){$x$}\put(-9,64){$\mu'$}\put(-60,-5){$M$}\put(-15,81){$z$}
&\raisebox{40pt}{$\stackrel{\eta_1}\longrightarrow$}&
\includegraphics[scale=0.65]{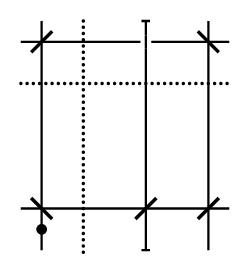}\put(-60,-5){$M_1'$}\put(-45,10){$\mu_0$}\put(-35,83){$y$}
&\raisebox{40pt}{$\stackrel{\eta_2}\longrightarrow$}&
\includegraphics[scale=0.65]{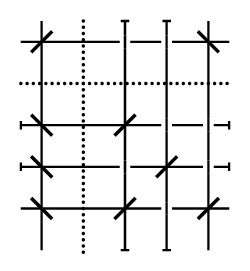}\put(-60,-5){$M_1$}
&\raisebox{40pt}{$\stackrel{\xi_1}\longrightarrow$}&
\includegraphics[scale=0.65]{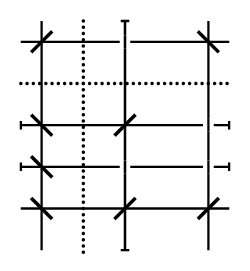}\put(-60,-5){$M_3$}
\\
$\big\downarrow$\hbox to 0pt{$\scriptstyle\zeta$\hss}
&&&&&&
$\big\downarrow$\hbox to 0pt{$\scriptstyle\xi_2$\hss}
\\
\includegraphics[scale=0.65]{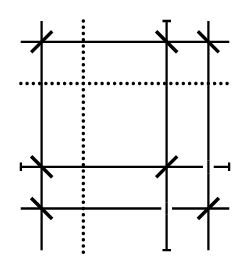}\put(-60,-5){$M_2$}
&\raisebox{40pt}{$\stackrel{\xi_5}\longleftarrow$}&
\includegraphics[scale=0.65]{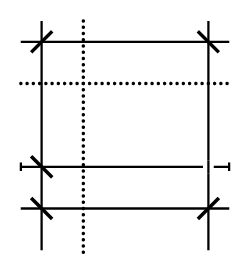}\put(-60,-5){$M_6$}
&\raisebox{40pt}{$\stackrel{\xi_4}\longleftarrow$}&
\includegraphics[scale=0.65]{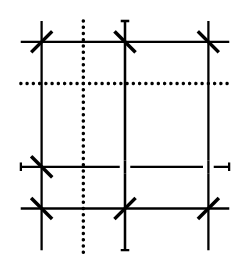}\put(-60,-5){$M_5$}
&\raisebox{40pt}{$\stackrel{\xi_3}\longleftarrow$}&
\includegraphics[scale=0.65]{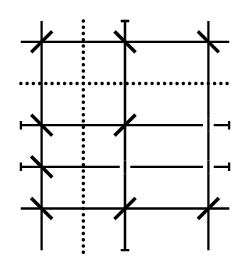}\put(-60,-5){$M_4$}
\end{tabular}
\caption{Proof of Proposition~\ref{similarity-prop} in the head
wandering case, $\omega=(\mu,\mu,p)$, $\omega'=(\mu',\mu'',\mu,p)$}\label{wandering-very-special-fig}
\end{figure}
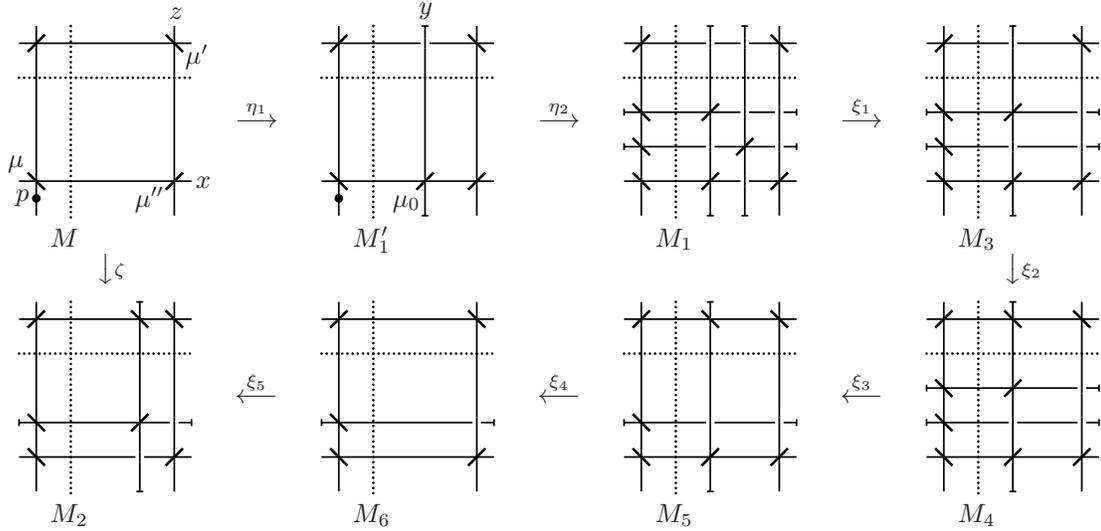

If~$\mu_0$ is outside~$c$ we proceed exactly as before, when
we had~$\mu_1\ne\mu_2$. Suppose that~$\mu_0\in c$.
The sought-for decomposition is sketched in Figure~\ref{wandering-very-special-fig}
(the other cases are obtained by symmetries), in which:
\begin{itemize}
\item
the move~$M\xmapsto{\eta_1}M_1'$
is the first move of a canonical decomposition of~$M\xmapsto\eta M_1$,
and~$M_1'\xmapsto{\eta_2}M_1$ is the remaining part,
\item
$M_1\xmapsto{\xi_1}M_3$ is a type~I elimination move that deletes a successor of~$y$
and a successor of~$\mu_0$,
\item
$M_3\xmapsto{\xi_2} M_4$ is a type~I elementary bypass addition,
\item
$M_4\xmapsto{\xi_3} M_5$ is a bridge removal, which deletes a successor of~$x$
with two mirror on it,
\item
$M_5\xmapsto{\xi_4} M_6$ is a composition
of a jump move that shifts the remaining successor of~$y$ toward~$z$
and a double merge move that merges this successor with~$z$,
\item
and finally, $M_6\xmapsto{\xi_5} M_2$ is a bridge addition.
\end{itemize}

If the moves~$M\xmapsto\eta M_1$ and~$M\xmapsto\zeta M_2$ preserve
all boundary circuits except~$c$, then so do all the moves in Figure~\ref{wandering-very-special-fig}.
Otherwise, some adjustment may be needed by means of jump moves.
All jump moves and bridge moves should be decomposed neatly into type~I moves to
get the sought-for decomposition of~$M_1\xmapsto{\zeta\circ\eta^{-1}}M_2$.

\smallskip\noindent\emph{Case 5}: $\omega\mapsto\omega'$ is a tail-to-head conversion.\\
We use the notation from Definition~\ref{conversion-def}. Suppose that~$k>1$, that is,~$\omega'$
is not special. Let~$M\xmapsto{\eta_1}M_1'$ and~$M\xmapsto{\zeta_1}M_2'$ be the first
moves of canonical decompositions of~$M\xmapsto\eta M_1$ and~$M\xmapsto\zeta M_2$,
respectively, and let~$M_1'\xmapsto{\eta_2}M_1$ and~$M_2'\xmapsto{\zeta_2}M_2$
be the remaining parts of the decompositions.

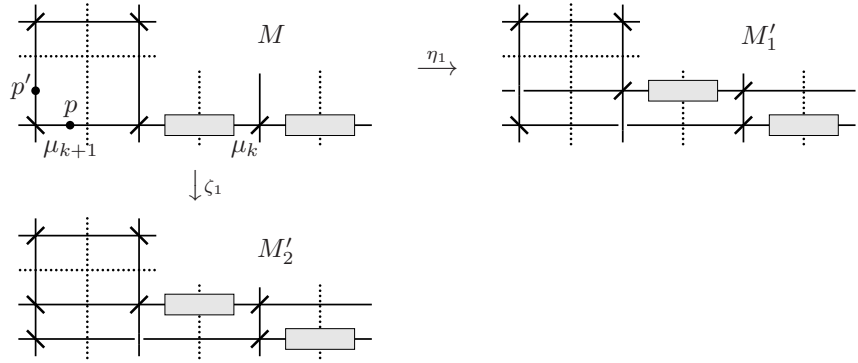
\begin{figure}[ht]
\begin{tabular}{ccc}
\includegraphics[scale=0.65]{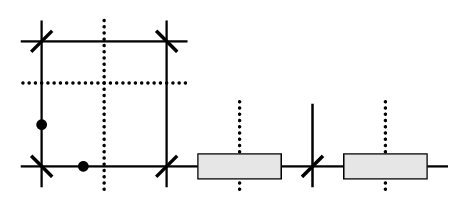}\put(-60,3){$\mu_k$}\put(-131,3){$\mu_{k+1}$}%
\put(-123,18){$p$}\put(-143,25){$p'$}\put(-50,44){$M$}
&\raisebox{32pt}{$\stackrel{\eta_1}\longrightarrow$}&
\includegraphics[scale=0.65]{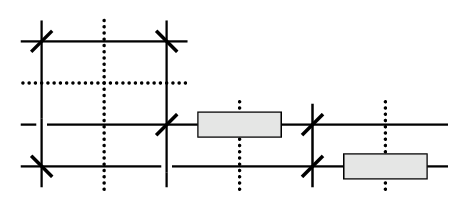}\put(-50,44){$M_1'$}\\
$\big\downarrow$\hbox to 0pt{$\scriptstyle\zeta_1$\hss}\\
\includegraphics[scale=0.65]{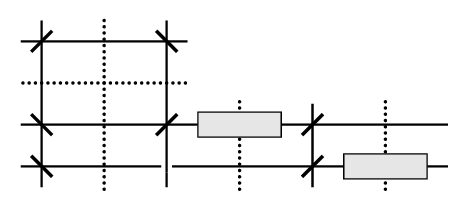}\put(-50,44){$M_2'$}
\end{tabular}
\caption{Proof of Proposition~\ref{similarity-prop} in the tail-to-head conversion case}\label{tail-to-head-split-fig}
\end{figure}
One can see from Figure~\ref{tail-to-head-split-fig} that $M_1'\xmapsto\xi M_2'$ with~$\xi=\zeta_1\circ\eta_1^{-1}$
is a type~I elementary bypass addition possibly composed with jump moves that do not
change the combinatorial type of the diagram. It is also clear that~$M_2'\xmapsto{\zeta_2}M_2$
can be interpreted as~$M_2'\xmapsto{{\eta_2}^\xi}M_2$; refer to Figure~\ref{tail-to-head-similarity-fig} for the general scheme.
\begin{figure}[ht]
\includegraphics{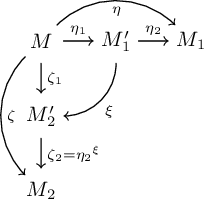}
\caption{The scheme of the moves in the proof of Proposition~\ref{similarity-prop},
the tail-to-head conversion case with non-special~$\omega'$}\label{tail-to-head-similarity-fig}
\end{figure}
We skip the details, which are similar to many previously considered cases.

Finally, suppose that~$k=1$. Let~$x$ be the occupied level of~$M$ containing~$\mu_1$, $\mu_2$, and~$p$.
There is a unique $\diagup$-mirror in~$c\cap x$, which we denote by~$\mu_0$. The sequences~$\omega''=(\mu_1,\mu_0,\mu_0,p)$
and~$\omega'''=(\mu_1,\mu_0,\mu_0,\mu_2,p')$ are non-special
type~II splitting routes related by a tail-to-head conversion operation,
and in each of the pairs $\omega'',\omega$ and~$\omega',\omega'''$ the splitting routes are related by a pulling tight
operation. The assertion of the proposition follows from the previously considered cases.
\end{proof}

\subsection{First commutation property of generalized type~II merge moves}

\begin{lemm}\label{pre-1st-commutation-lem}
Let~$M\xmapsto\eta M_1$ be a generalized type~II split move associated
with a splitting route~$\omega$, and let~$M\xmapsto\zeta M_2$
be a type~I move of one of the following kinds: elimination, extension, split, merge,
elementary bypass addition, or elementary bypass removal.
Denote by~$C$ the set of all boundary circuits of~$M$ preserved by both moves.

Then there exists a generalized type~II split move~$M\xmapsto{\eta'}M_1'$ associated with
a type~II splitting route~$\omega'$ such that
the following holds:
\begin{enumerate}
\item
the splitting route $\omega'$ is similar to~$\omega$;
\item
the move~$M\xmapsto{\eta'}M_1'$ preserves all boundary circuits in~$C$.
\item
the move~$M\xmapsto\zeta M_2$ is friendly to~$M\xmapsto{\eta'}M_1'$;
\item
if~$M_2\xmapsto{{\eta'}^\zeta}M_{21}$ is a move resembling~$M\xmapsto{\eta'}M_1'$,
then the transformation~$M_{21}\xmapsto{\eta'\circ\zeta^{-1}\circ({\eta'}^\zeta)^{-1}}M_1'$
admits a $C$-neat decomposition into type~I elementary moves.
\end{enumerate}
\end{lemm}

\begin{proof}
Let~$\omega''$
be a reduced type~II splitting route obtained from~$\omega$ by pulling tight and tail shrinking
operations, which exists by Proposition~\ref{reducing-splitting-route-prop}, and
let~$M\xmapsto{\eta''}M_1''$ be a generalized type~II split move associated with~$\omega''$.
By Lemma~\ref{relaxation-lem}
the move~$M\xmapsto{\eta''}M_1''$ can be chosen
to preserve all the boundary circuits in~$C$.
We assume that such a choice has been made.

If the move~$M\xmapsto\zeta M_2$ is friendly to~$M\xmapsto{\eta''}M_1''$ and the moves
almost commute with one another, then due to Lemma~\ref{neutral-move-decomposition-lem}
Condition~(4) above is satisfied for~$\omega'=\omega''$, and we are done.
Similarly, we are done if the move~$M_2\xmapsto{{\eta''}^{\zeta}}M_{21}$ almost commutes with~$M_2\xmapsto{\zeta^{-1}}M$.
If neither of these occurs, additional arguments are required and/or~$\omega''$ has to be modified.
There are again a number of cases to consider.

\medskip\noindent\emph{Case~1}:~$\omega''$ is not special.

\medskip\noindent\emph{Case~1a}:~$M\xmapsto\zeta M_2$ is a type~I elimination move.
\\
If~$x$ is an occupied level of~$M$ containing a single mirror~$\mu$, and this mirror is of type~`$\diagup$', then no reduced
type~II splitting route visits~$\mu$. Hence, such a $\mu$ cannot appear in~$\omega''$, and
any type~I elimination move is friendly to~$M\xmapsto{\eta''}M_1''$. The moves~$M\xmapsto{\eta''}M_1''$
and~$M\xmapsto\zeta M_2$ commute with one another by Lemma~\ref{elimination-commutes-with-split-lem},
so we can take~$\omega''$ for~$\omega'$.

\medskip\noindent\emph{Case~1b}:~$M\xmapsto\zeta M_2$ is a type~I extension move.
\\
An extension move can be unfriendly to the non-special generalized
split move~$M\xmapsto{\eta''}M_1''$ only if~$\omega''$ is single-headed,
and the place where a new mirror is
inserted coincides with the snip point of the move~$M\xmapsto{\eta''}M_1''$. However, a small shift
of the snip point has no combinatorial effect on the respective generalized type~II split move, hence
we may consider all extension moves being friendly to~$M\xmapsto{\eta''}M_1''$.
Then the move~$M_2\xmapsto{{\eta''}^{\zeta}}M_{21}$ is friendly
to the elimination move~$M_2\xmapsto{\zeta^{-1}}M$, and
these moves commute with one another, again, by Lemma~\ref{elimination-commutes-with-split-lem}.

\medskip\noindent\emph{Case~1c}:~$M\xmapsto\zeta M_2$ is a type~I merge move.
\\
Denote by~$\nu_1$, $\nu_2$ the two mirrors of~$M$ that have a common successor in~$M_2$,
which we denote by~$\nu$. We number them so that~$(\nu_1;\nu_2)$ is the splitting gap of the inverse
move~$M_2\xmapsto{\zeta^{-1}}M$. Denote also by~$x$ the occupied level of~$M$
passing through both~$\nu_1$ and~$\nu_2$, and by~$y_1$ and~$y_2$ the
occupied levels of~$M$ perpendicular to~$x$ and passing through~$\nu_1$ and~$\nu_2$, respectively.
Denote the common successor of~$y_1$ and~$y_2$ in~$M_2$ by~$y$.

Under an appropriate choice of~$h_M^{M_2}$
the image of~$\wideparen\omega''$ in~$\wideparen M_2$ is a normal arc
unless the splitting route~$\omega''$ is single-headed and its snip point lies in~$(\nu_1;\nu_2)$.
If the latter is the case, apply a tail stretching operation to~$\omega''$ to obtain~$\omega'$.
Otherwise, put~$\omega'=\omega''$.

Even if we have to apply a tail stretching to~$\omega''$, the obtained splitting route~$\omega'$
still does not separate~$C$, since~$\wideparen\omega'$ is obtained from~$\wideparen\omega''$
by pulling the snip point along 
a boundary circuit of~$M$ which is modified by
the move~$M\xmapsto\eta M_1$, and hence is not in~$C$. So,
we can always choose the move~$M\xmapsto{\eta'}M_1'$ associated with~$\omega'$
so that it preserves all boundary circuits in~$C$.

Let~$(\mu_1,\ldots,\mu_k,p)=\omega'$. In the case when~$\omega'$ is double-headed
we assume that~$p$ is chosen so close to~$\mu_k$ that
no occupied level passes between~$\mu_k$ and~$p$.
For~$i=1,\ldots,k$, denote by~$\mu_i'$ be the successor of~$\mu_i$
in~$M_2$. Denote also by~$p'$ the point in~$\mathbb T^2$ such that~$h_M^{M_2}(\wideparen p)=\wideparen p'$.
The sequence~$\omega'''=(\mu_1',\mu_2',\ldots,\mu_k',p')$ is a type~II splitting route in~$M_2$, provided
that Condition~(2d) of the corresponding
Definition~\ref{simple-route-def} or~\ref{double-headed-splitting-route-def} holds.

We claim that this is the case due to the fact that~$\omega'$ is either reduced or is
obtained from the reduced splitting route~$\omega''$ by a single tail stretching operation. Indeed,
we can have~$\mu_i'=\mu_{i+1}'$ only if~$\{\mu_i,\mu_{i+1}\}=\{\nu_1,\nu_2\}$.
In this case, if~$3\leqslant i\leqslant k-2$, then one of the mirrors~$\mu_{i-1}$ and~$\mu_{i+2}$ lies on~$y_1$ and the other on~$y_2$,
and we have~$\mu_{i-1},\mu_{i+2}\notin\{\nu_1,\nu_2\}$, $\mu_{i-1}\ne\mu_{i+2}$. Therefore, we cannot have~$\mu_{i-1}'=\mu_{i+2}'$.

One can see from this that the only cancellable portions of~$\omega'''$ have the form~$(\mu_i',\mu_{i+1}')=(\nu,\nu)$.
One can also see that if~$\{\mu_i,\mu_{i+1}\}=\{\mu_j,\mu_{j+1}\}=\{\nu_1,\nu_2\}$, then~$i-j=0\pmod2$,
so Condition~(2d) holds.

Thus, the image of~$\omega'$ in~$M_2$ is a type~II splitting route and we can define the
associated generalized type~II split move~$M_2\xmapsto{{\eta'}^\zeta}M_{21}$.
We now show that the move~$M_2\xmapsto{\zeta^{-1}}M$ is friendly to~$M_2\xmapsto{{\eta'}^\zeta}M_{21}$.

The move~$M_2\xmapsto{\zeta^{-1}}M$ is a type~I split move associated with a
splitting route of the form~$(\nu,q)$, where~$q\in y$. By construction, the associated
splitting path~$\wideparen{(\nu,q)}$ is disjoint from~$\wideparen\omega'''$.
We claim that their mutual position is unambiguous.

Indeed, suppose otherwise. Then there is a type~II splitting route in~$M$
not equivalent to~$\omega'$ whose image in~$M_2$ coincides with~$\omega'''$.
Such a route should be obtained from~$\omega'$ by replacing some entries equal to~$\nu_1$ or~$\nu_2$
with~$\nu_2$ or~$\nu_1$, respectively. By construction we have~$\mu_i\ne\mu_{i+1}$
for all~$i=1,\ldots,k-1$. Therefore, if~$\mu_i=\nu_1$, $2\leqslant i\leqslant k$,
then at least one of~$\mu_{i-1}$ and~$\mu_{i+1}$
is not in~$\{\nu_1,\nu_2\}$ and lies on~$y_1$ (if~$\omega'$ is single-headed we assume that~$\mu_{k+1}$
refers to~$p$). This means that replacing~$\mu_i$ by~$\nu_2$
will produce a sequence that is not a type~II splitting route even if some other~$\nu_1$'s are
replaced by~$\nu_2$'s or vice versa. Similarly, no $\nu_2$-entry in~$\omega'$ can be replaced by~$\nu_1$.
A contradiction.

Thus the move~$M_2\xmapsto{\zeta^{-1}}M$ is friendly to~$M_2\xmapsto{{\eta'}^\zeta}M_{21}$,
and by Lemma~\ref{type-i-split-commutes-with-generalized-type-ii-split-lem} these moves almost commute.

\medskip\noindent\emph{Case~1d}:~$M\xmapsto\zeta M_2$ is a type~I split move.
\\
Let~$(\nu,q)$ be the respective
type~I splitting route, and let~$(\mu_1,\mu_2,\ldots,\mu_k,p)=\omega''$. Denote by~$x$ the occupied
level of~$M$ containing~$\nu$ and~$q$.

The only possible reason for~$M\xmapsto\zeta M_2$ to be unfriendly
to~$M\xmapsto{\eta''}M_1''$ in this case
is an unavoidable intersection of~$\wideparen{(\nu,q)}$ and~$\wideparen\omega''$,
since the other possible reason, an ambiguity in the relative position of~$\wideparen{(\nu,q)}$ and~$\wideparen\omega''$
is ruled out by the fact that~$\omega''$ is reduced.
Clearly, all intersections of~$\wideparen{(\nu,q)}$ and~$\wideparen\omega''$
occur in~$\wideparen x$ and each of them is due to the fact that, for some~$i\in\{2,3,\ldots,k\}$,
the following condition holds:
\begin{equation}\label{where-to-insert-a-loop-eq}
\text{the mirrors~$\mu_i$, $\mu_{i+1}$ lie on~$x$, and the pairs~$\{\mu_i,\mu_{i+1}\}$ and~$\{\nu,q\}$ interleave,}
\end{equation}
where we put~$\mu_{k+1}=p$ if~$\omega'$ is single-headed; see the left picture in Figure~\ref{loosening-fig}.

\begin{figure}[ht]
\includegraphics[scale=0.7]{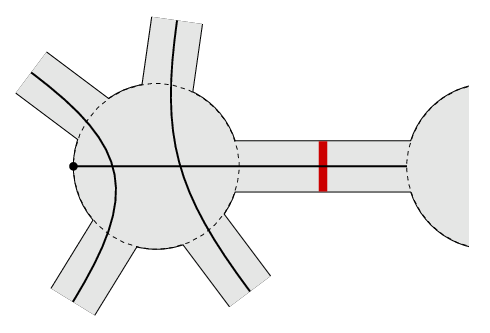}\put(-150,53){$\wideparen q$}\put(-60,36){$\wideparen\nu$}%
\put(-125,40){$\wideparen\omega''$}\put(-120,67){$\wideparen\omega''$}
\hskip2cm
\includegraphics[scale=0.7]{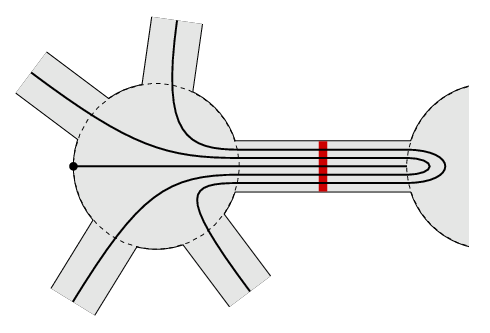}
\caption{Loosening~$\omega''\mapsto\omega'$ that resolves intersections with~$\wideparen{(\nu,q)}$}\label{loosening-fig}
\end{figure}

Let~$\omega'$ be obtained from~$\omega''$ by inserting the subsequence~$(\nu,\nu)$ between~$\mu_i$ and~$\mu_{i+1}$
whenever condition~\eqref{where-to-insert-a-loop-eq} holds. Topologically this means that we replace~$\wideparen\omega''$
by an isotopic path~$\wideparen\omega'$ avoiding intersections with~$\wideparen{(\nu,q)}$.
Clearly, the new path can be chosen to be a normal arc.
Since~$\omega''$ is reduced, one can see that~$\omega'$ satisfies Conditions~(2d)
of the corresponding
Definition~\ref{simple-route-def} or~\ref{double-headed-splitting-route-def},
so~$\omega'$ is a type~II splitting route.

By construction, the move~$M\xmapsto\zeta M_2$ is now friendly to~$M\xmapsto{\eta'}M_1'$ if the latter
is a generalized type~II split move
associated with~$\omega'$. It remains to show that~$\omega'$ does not separate~$C$,
and hence the move~$M\xmapsto{\eta'}M_1'$ can be chosen
to preserve all the boundary circuits in~$C$.

Let~$d_1,\ldots,d_l$ be the connected components of~$\wideparen x\setminus\bigl(\wideparen{(\nu,q)}\cup\wideparen\omega''\bigr)$.
For any~$i,j\in\{1,\ldots,l\}$, $i\ne j$, the components~$d_i$ and~$d_j$ are separated by at least one of~$\wideparen{(\nu,q)}$
and~$\wideparen\omega''$. This means that the respective portions of~$x$ appear at different occupied levels
either in~$M_1$, or in~$M_2$ (or in both). Therefore, for at most one~$i$ the component~$d_i$ may have
a non-empty intersection with the boundary components corresponding to boundary circuits from~$C$. This means
that at most one connected component of~$x\setminus(\{\nu,q,\}\cup\omega'')=x\setminus(\{\nu,q,\}\cup\omega')$
has a non-empty intersection with~$\bigcup_{c\in C}c$. Therefore,
$\omega'$ does not separate~$C$.

\medskip\noindent\emph{Case~1e}:~$M\xmapsto\zeta M_2$ is a type~I elementary bypass removal.
\\
Let~$c$ be the inessential boundary circuit of~$M$ whose presence allows this move.
This means that~$c$ has  the form of the boundary of a rectangle~$r$
such that~$\interior(r)\cap E_M=\varnothing$. Let~$\nu_1$, $\nu_2$, $\nu_3$, $\nu_4$ be the mirrors
at the corners of~$r$ numbered so that~$c=[\nu_1;\nu_2]\cup[\nu_1;\nu_3]\cup[\nu_2;\nu_4]\cup[\nu_3;\nu_4]$.
This means, in particular, that~$\nu_1$ and~$\nu_4$ are $\diagdown$-mirrors, and~$\nu_2$, $\nu_3$ are
$\diagup$-mirrors.

We may assume without loss of generality that the move~$M\xmapsto\zeta M_2$ removes~$\nu_1$, since
the case of removing~$\nu_4$ is symmetric to this one.
By Lemma~\ref{bypass-commutes-with-split-lem} the moves~$M\xmapsto\zeta M_2$ and~$M\xmapsto{\eta''}M_1''$
commute with one another unless one of the following holds:
\begin{itemize}
\item
$\nu_1$ is an entry of~$\omega''$;
\item
$\omega''$ is single-headed and the snip point of the move~$M\xmapsto{\eta''}M_1''$ lies on~$c$.
\end{itemize}
So, it suffices to find a non-special type~II splitting route~$\omega'$ similar to~$\omega''$ such that~$\nu_1$
is not an entry of~$\omega'$ and either~$\omega'$ is double-headed or~$\omega'$ is a single-headed
splitting route ending outside~$c$.

There are a number of cases
in each of which we either take~$\omega''$ for~$\omega'$ or apply a few operations transforming~$\omega''$ to~$\omega'$,
which include tail stretching, tail-to-head and head-to-tail conversions, and head wandering.
All the operations modify the corresponding splitting path near the boundary component~$\wideparen c\subset\partial\wideparen M$.
One can see from this and from the fact that~$c\notin C$ that, in each case, the obtained
type~II splitting route~$\omega'$ still does not separate~$C$.

Let~$(\mu_1,\mu_2,\ldots,\mu_k,p)=\omega''$.
Denote by~$x$ the occupied level of~$M$ passing through~$\nu_1$ and~$\nu_2$.
The roles of~$\nu_2$ and~$\nu_3$ are symmetric to each other in the present context, so we omit the cases obtained from
the considered ones by the exchange~$\nu_2\leftrightarrow\nu_3$.

Suppose first that~$\omega''$ is single-headed.

If $\mu_1\ne\nu_1$, $p\notin c$ take~$\omega''$ for~$\omega'$.

If $\mu_1\ne\nu_1$ and $p\in(\nu_1;\nu_2)$ take~$(\mu_1,\ldots,\mu_k,\nu_2,\nu_4,q)$ for~$\omega'$, where
$q\in(\nu_3;\nu_4)$ is close to~$\nu_4$. This splitting route is obtained from~$\omega''$ by a tail stretching
followed by a tail-to-head conversion:
$$(\mu_1,\mu_2,\ldots,\mu_k,p)\mapsto(\mu_1,\mu_2,\ldots,\mu_k,\nu_2,p')\mapsto(\mu_1,\ldots,\mu_k,\nu_2,\nu_4,q),$$
where~$p'\in(\nu_2;\nu_4)$.

If $\mu_1\ne\nu_1$ and $p\in(\nu_2;\nu_4)$ take~$(\mu_1,\ldots,\mu_k,\nu_4,q)$ for~$\omega'$, where
$q\in(\nu_3;\nu_4)$ is close to~$\nu_4$. This splitting route is obtained from~$\omega''$ by a tail-to-head conversion.

If~$\mu_1=\nu_1$, $p\notin c$, and~$\mu_2\in x$ (here and in the next case,
if~$k=1$, then~$\mu_2$ refers to~$p$), take~$(\nu_4,\nu_2,\mu_2,\ldots,\mu_k,p)$ for~$\omega'$. This
splitting route is obtained from~$\omega''$ by a head wandering operation.

If~$\mu_1=\nu_1$, $p\in(\nu_1;\nu_2)$, and~$\mu_2\in x$, take~$(\nu_4,\nu_2,\mu_2,\ldots,\mu_k,\nu_2,\nu_4,q)$ for~$\omega'$,
where~$q\in(\nu_4;\nu_3)$ is close to~$\nu_4$. This splitting route is obtained from~$\omega''$ by
a sequence of four operations including a tail stretching, a tail-to-head conversion, replacing
the splitting route by an equivalent one, and then a head wandering operation:
$$\omega''\mapsto(\mu_1,\ldots,\mu_k,\nu_2,p')\mapsto(\mu_1,\ldots,\mu_k,\nu_2,\nu_4,q')\mapsto(\mu_1,\ldots,\mu_k,\nu_2,\nu_4,q)\mapsto
(\nu_4,\nu_2,\mu_2,\ldots,\mu_k,\nu_2,\nu_4,q),$$
where~$q'\in(\nu_3;\nu_4)$ is close to~$\nu_4$, and $p'\in(\nu_2;\nu_4)$.

If~$\mu_1=\nu_1$, $p\in(\nu_1;\nu_3)$, and~$\mu_2\in x$, take~$(\nu_4,\nu_2,\mu_2,\ldots,\mu_k,\nu_3,\nu_4,q)$ for~$\omega'$,
where~$q\in(\nu_4;\nu_2)$ is close to~$\nu_4$. The sequence of transformations producing~$\omega'$ from~$\omega''$
is similar to the previous case:
$$\omega''\mapsto(\mu_1,\ldots,\mu_k,\nu_3,p')\mapsto(\mu_1,\ldots,\mu_k,\nu_3,\nu_4,q')\mapsto(\mu_1,\ldots,\mu_k,\nu_3,\nu_4,q)\mapsto
(\nu_4,\nu_2,\mu_2,\ldots,\mu_k,\nu_3,\nu_4,q),$$
where~$q'\in(\nu_2;\nu_4)$ is close to~$\nu_4$, and $p'\in(\nu_3;\nu_4)$.

If~$\mu_1=\nu_1$, $p\in(\nu_2;\nu_4)$, and~$\mu_2\in x$, take~$(\nu_4,\nu_2,\mu_2,\ldots,\mu_k,\nu_4,q)$ for~$\omega'$,
where~$q\in(\nu_4;\nu_3)$ is close to~$\nu_4$. This splitting route is obtained from~$\omega''$ by
a sequence of three operations including a tail-to-head conversion, replacing
the splitting route by an equivalent one, and then a head wandering operation:
$$\omega''\mapsto(\mu_1,\ldots,\mu_k,\nu_4,q')\mapsto(\mu_1,\ldots,\mu_k,\nu_4,q)\mapsto
(\nu_4,\nu_2,\mu_2,\ldots,\mu_k,\nu_4,q),$$
where~$q'\in(\nu_3;\nu_4)$ is close to~$\nu_4$.

If~$\mu_1=\nu_1$, $p\in(\nu_3;\nu_4)$, and~$\mu_2\in x$, take~$(\nu_4,\nu_2,\mu_2,\ldots,\mu_k,\nu_4,q)$ for~$\omega'$,
where~$q\in(\nu_4;\nu_2)$ is close to~$\nu_4$. The splitting route~$\omega'$ is produced from~$\omega''$
similarly to the previous case:
$$\omega''\mapsto(\mu_1,\ldots,\mu_k,\nu_4,q')\mapsto(\mu_1,\ldots,\mu_k,\nu_4,q)\mapsto
(\nu_4,\nu_2,\mu_2,\ldots,\mu_k,\nu_4,q),$$
where~$q'\in(\nu_2;\nu_4)$ is close to~$\nu_4$.

Now suppose that the splitting route~$\omega''$ is double-headed.

If $\mu_1\ne\nu_1$ and $\mu_k\ne\nu_1$ take~$\omega''$ for~$\omega'$.

If~$\mu_k=\nu_1$ and $p\in(\nu_1;\nu_3)$,
apply a head-to-tail conversion to~$\omega''$:
$$(\mu_1,\ldots,\mu_{k-1},\nu_1,p)\mapsto(\mu_1,\ldots,\mu_{k-1},q),$$
where~$q\in(\nu_1;\nu_2)$. We obtain a single-headed type~II splitting route similar to~$\omega''$
and proceed as in one of the above listed cases.

If~$\mu_k=\nu_1$, $\mu_1\ne\nu_1$, $p\in(\nu_3;\nu_1)$, pick a point~$p'\in(\nu_1;\nu_3)$ close to~$\nu_1$.
The type~II splitting route~$(\mu_1,\ldots,\mu_k,p')$ is then equivalent to~$\omega''$
and we can proceed as in the previous case.

If~$\mu_1=\nu_1$, $p\notin c$, and~$\mu_2\in x$, apply a head wandering
$$(\nu_1,\mu_2,\ldots,\mu_k,p)\mapsto(\nu_4,\nu_2,\mu_2,\ldots,\mu_k,p)$$
and then proceed as in one of the above listed cases (or symmetric to them).

If~$\mu_1=\nu_1$, $\mu_k\ne\nu_1$, $p\in c$, replace~$\omega''$ with an equivalent splitting route with~$p\notin c$, and
then proceed as in the previous case.

\medskip\noindent\emph{Case~1f}:~$M\xmapsto\zeta M_2$ is a type~I elementary bypass addition.
\\
If~$\omega''$ is double-headed, then the move~$M_2\xmapsto{{\eta''}^\zeta}M_{21}$ is defined and
commutes with~$M_2\xmapsto{\zeta^{-1}}M$ by Lemma~\ref{bypass-commutes-with-split-lem}, so we can take~$\omega''$ for~$\omega'$.

Suppose that~$\omega''=(\mu_1,\ldots,\mu_k,p)$ is single-headed.
Let~$c=[\nu_1;\nu_2]\cup[\nu_1;\nu_3]\cup[\nu_2;\nu_4]\cup[\nu_3;\nu_4]=\partial r$
be the inessential boundary circuit of~$M_2$ created by this move.
We may assume without loss of generality that the mirror added by the move is~$\nu_1$
(the case when the added mirror is~$\nu_4$ is symmetric to this one).
So,~$\nu_1$ is not present in~$M$.

If~$p\in[\nu_1;\nu_2)\cup[\nu_1;\nu_3)$ there is an equivalent type~II splitting route~$(\mu_1,\ldots,\mu_k,p')$ with~$p'\notin c$.
Take such a route for~$\omega'$.

If~$p\in(\nu_2;\nu_4)\cup(\nu_3;\nu_4)$, then a tail stretching operation on~$\omega''$ yields
a type~II splitting route of the form~$(\mu_1,\ldots,\mu_k,\nu_j,p')$
with $j\in\{2,3\}$, $p'\in[\nu_1;\nu_2)\cup[\nu_1;\nu_3)$, which can then be changed to an equivalent one~$(\mu_1,\ldots,\mu_k,\nu_j,p'')$
with~$p''\notin c$. We take the obtained splitting route for~$\omega'$.

The move~$M\xmapsto\zeta M_2$ is friendly to the type~II split move~$M\xmapsto{\eta'}M_1'$ associated with~$\omega'$,
and, moreover, it follows from Lemma~\ref{bypass-commutes-with-split-lem} that the moves~$M_2\xmapsto{{\eta'}^\zeta}M_{21}$,
$M_2\xmapsto{\zeta^{-1}}M$ commute.

\medskip\noindent\emph{Case~2}: $\omega''$ is special, $\omega''=(\mu_1,\mu_2,p)$.
\\
We may assume without loss of generality that the interval~$(\mu_2;p)$ is shorter than~$(p;\mu_2)$,
since the other case is symmetric to this one.

The special case requires special care, but
in the present context, it is actually easier than the non-special one. There are several subcases to consider
in each of which the proof is by a straightforward check.
We skip some boring details and provide only the most essential ones.

By Proposition~\ref{similarity-prop} it suffices to prove the assertion of the lemma
for a concrete eligible choice of the auxiliary mirror of the move~$M_2\xmapsto{{\eta'}^\zeta}M_{21}$. So,
in each case, we specify this choice.

\medskip\noindent\emph{Case~2a}:~$M\xmapsto\zeta M_2$ is a type~I elimination move.
We take~$\omega''$ for~$\omega'$ and take the auxiliary mirror of the move~$M\xmapsto\eta M_1$
for the auxiliary mirror of the move~$M_2\xmapsto{\eta'^\zeta}M_{21}$.

\medskip\noindent\emph{Case~2b}:~$M\xmapsto\zeta M_2$ is a type~I extension move.
\\
We take~$\omega''$ for~$\omega'$ unless the mirror~$\mu_*$ added
by the move~$M\xmapsto\zeta M_2$ appears in~$(\mu_2;p)$,
in which case we put~$\omega'=(\mu_1,\mu_2,p')$, where~$p'\in(\mu_2;\mu_*)$
is any point not lying at the intersection of two occupied levels of~$M$.
Then the move~$M_2\xmapsto{\eta'^\zeta}M_{21}$ is also associated with~$\omega'$.

We choose the same auxiliary level for
the move~$M_2\xmapsto{\eta'^\zeta}M_{21}$ as the one for~$M\xmapsto\eta M_1$, provided that it does not coincide with the occupied level created
by the move~$M\xmapsto\zeta M_2$. In the latter case, the position of
the auxiliary level should be disturbed slightly.

\medskip\noindent\emph{Case~2c}:~$M\xmapsto\zeta M_2$ is a type~I merge move.
\\
Let~$(\nu_1;\nu_2)$ be the splitting gap of the inverse move~$M_2\xmapsto{\zeta^{-1}}M$.
We take~$\omega''$ for~$\omega'$. Since~$\mu_1$, $\mu_2$
are $\diagdown$-mirrors, and~$\nu_1$, $\nu_2$ are $\diagup$-mirrors, the image~$(\mu_1',\mu_2',p')$ of~$\omega'$
in~$M_2$ is a well defined (up to equivalence) special type~II splitting route.

At least one of the two (not necessarily distinct) boundary circuits of~$M$ that hit~$\mu_1$ is modified by the move~$M\xmapsto\eta M_1$,
and hence, does not belong to~$C$. So, we can choose the auxiliary mirror~$\mu_0$ of the move~$M\xmapsto{\eta'}M_1'$
on such a circuit in a small vicinity of~$\mu_1$. Similarly, we choose the auxiliary mirror of the move~$M_2\xmapsto{{\eta'}^\zeta}M_{21}$
in a small vicinity of~$\mu_1'$, on the same side of~$\mu_1'$ on which~$\mu_0$ is located relative to~$\mu_1$.

\medskip\noindent\emph{Case~2d}:~$M\xmapsto\zeta M_2$ is a type~I split move.
\\
Let~$(\nu,q)$ be the splitting route with which the move~$M\xmapsto\zeta M_2$ is associated.
If the splitting paths~$\wideparen\omega''$ and~$\wideparen{(\nu,q)}$ have no unavoidable intersection,
we take~$\omega''$ for~$\omega'$. 
The positions of the auxiliary mirrors for the moves~$M\xmapsto{\eta'}M_1'$ and~$M_2\xmapsto{{\eta'}^\zeta}M_{21}$
should be chosen in a small vicinity of the mirror~$\mu_1$ and its successor in~$M_2$, respectively.

Otherwise we put~$\omega'=(\mu_1,\nu,\nu,\mu_2,p)$, similarly to
Case~1d.

\medskip\noindent\emph{Case~2e}:~$M\xmapsto\zeta M_2$ is a type~I elementary bypass removal.
\\
We use the same notation as in Case~1e, as well as a similar strategy.
Unless~$\nu_1$ or~$\nu_4$ appear in~$\omega''$ we take~$\omega''$ for~$\omega'$
and use the same positions of the auxiliary mirrors for the moves~$M\xmapsto{\eta'}M_1'$ and~$M_2\xmapsto{{\eta'}^\zeta}M_{21}$.
These positions should be chosen outside~$r$.

If~$\nu_1$ does not appear in~$\omega''$, but~$\nu_4$ does, we take for~$\omega'$ a splitting
route of the form~$(\mu_1,\nu_2,\nu_2,\mu_2,p')$ or~$(\mu_1,\nu_3,\nu_3,\mu_2,p')$,
whichever fits the definition. One of these must do so. The splitting route~$\omega'$
is now non-special, and the associated move~$M\xmapsto{\eta'}M_1'$ and the move~$M\xmapsto\zeta M_2$ are mutually friendly.

If~$\nu_1$ appears in~$\omega''$ we resolve this by means of head wandering operations, tail-to-head conversions,
and head-to-tail conversions, similarly to the procedure in Case~1e. This produces a double-headed splitting route
with no entry equal~$\nu_1$ and we proceed as before.

\medskip\noindent\emph{Case~2f}:~$M\xmapsto\zeta M_2$ is a type~I elementary bypass addition.
\\
We use the same notation as in Case~1f above.

If~$\nu_4$ does not appear in~$\omega''$ we take~$\omega''$ for~$\omega'$
and use the same position of the auxiliary mirrors for the moves~$M\xmapsto{\eta'}M_1'$ and~$M_2\xmapsto{{\eta'}^\zeta}M_{21}$
chosen outside~$r$.

If~$\nu_4=\mu_1$ or~$\mu_2$, we take for~$\omega'$ a splitting route of the form~$(\mu_1,\nu_2,\nu_2,\mu_2,p)$
or~$(\mu_1,\nu_3,\nu_3,\mu_2,p)$, whichever fits the definition. One of these must do so. Such splitting route~$\omega'$
is now non-special, and the associated move~$M\xmapsto{\eta'}M_1'$ and the move~$M\xmapsto\zeta M_2$ are mutually friendly.

\medskip
Thus, in each case, we have a recipe how to produce~$\omega'$, and to choose the associated generalized split move,
and it is straightforward to check that the assertion of the Lemma holds in each case.
\end{proof}

We are ready to establish the first commutation property of generalized type~II merge moves.

\begin{prop}\label{1st-comm-prop-of-merge-prop}
Let~$M_1$ be an enhanced mirror diagram, and let~$C$ be a collection of essential boundary circuits of~$M_1$.
Let also~$M_1\xmapsto\eta M_2$ be a generalized type~II merge move preserving all~$c\in C$, and~$M_2\xmapsto\zeta M_3$
a transformation that admits a $C$-neat decomposition into type~I elementary moves.

Then there exists a composition~$M_1\xmapsto{\zeta'}M_4$ of type~I elementary moves, and
a generalized type~II merge move~$M_4\xmapsto{\eta'}M_3$ 
that give rise to a $C$-neat decomposition of the transformation~$M_1\xmapsto{\zeta\circ\eta}M_3$.
\end{prop}

\begin{proof}
A type~I slide move admits a neat decomposition into a sequence of three moves:
a type~I split move, a jump move,
and a type~I merge move (refer to Figure~\ref{slide-decomposition-fig}, where
a similar decomposition for a type~II slide move is shown).
Therefore, it suffices to prove the proposition, substituting for~$M_2\xmapsto\zeta M_3$
a single move from the following list: a type~I elimination, extension, split, or merge move, an
elementary bypass addition or removal, or a jump move.

If~$M_2\xmapsto\zeta M_3$ is a jump move, then the assertion of the proposition
follows from Lemmas~\ref{split-commute-with-jump-lem} and~\ref{neutral-move-decomposition-lem}, provided that the generalized split move~$M_2\xmapsto{\eta^{-1}}M_1$ is not special.

Suppose that~$M_2\xmapsto\zeta M_3$ is a jump move and
the move~$M_2\xmapsto{\eta^{-1}}M_1$ is special. Let~$(\mu_1,\mu_2,p)$ be the respective
special splitting route. By choosing the position of the auxiliary mirror
close enough to~$\mu_1$, and outside~$\bigcup_{c\in C}c$, we can ensure that the obtained generalized type~II split move~$M_2\xmapsto\xi M_1'$,
also associated with~$(\mu_1,\mu_2,p)$, has the following properties:
\begin{enumerate}
\item
the move~$M_2\xmapsto\xi M_1'$ preserves the boundary circuits in~$C$;
\item
the first move in the canonic decomposition of~$M_2\xmapsto\xi M_1'$ is friendly to~$M_2\xmapsto\zeta M_3$.
\end{enumerate}
The assertion of the proposition now follows from Proposition~\ref{similarity-prop}
and Lemmas~\ref{special-split-commute-with-jumpe-lem} and~\ref{neutral-move-decomposition-lem}.

If~$M_2\xmapsto\zeta M_3$ is a type~I elimination, extension, split, or merge move, or an
elementary bypass addition{} or removal,
the assertion follows from Lemma~\ref{pre-1st-commutation-lem}
(where we substitute~$M_2\xmapsto{\eta^{-1}}M_1$ and~$M_2\xmapsto\zeta M_3$
for~$M\xmapsto\eta M_1$ and~$M\xmapsto\zeta M_2$, respectively),
and Proposition~\ref{similarity-prop}.
\end{proof}

\subsection{Second commutation property of generalized type~II merge moves, non-special case}
Generalized type~II split moves come with a decomposition into split moves (and a single extension move
in the special case),
which are all of type~I except the last one, which is of type~II. By Lemma~\ref{split-move-decomposition-lem}
this implies that
any generalized type~II split move admits a neat decomposition into a sequence
of elementary moves such that all type~I moves in it precede all type~II moves.
So, for generalized type~II split moves the second commutation property holds trivially, but
this is not what we need.
In order to establish the second commutation property for the inverse moves
we need a decomposition of the opposite kind, in which type~II elementary moves precede
type~I elementary moves.

\begin{prop}\label{2ndcommutation-for-merge-non-special-prop}
Any non-special generalized type~II split move~$M\xmapsto\eta M'$ admits a neat decomposition into
a sequence of elementary moves
such that all type~II moves in it occur before all type~I moves.
\end{prop}

\begin{proof}
It suffices to prove the proposition in the case when the move~$M\xmapsto\eta M'$ is associated
with a single-headed type~II splitting route. Indeed, suppose otherwise,
namely that~$M\xmapsto\eta M'$ is associated with a non-special
double-headed type~II splitting route. The canonic decomposition of the move~$M\xmapsto\eta M'$
starts with a double split move, which can be decomposed neatly into
type~II elementary moves according to Lemma~\ref{neutral-move-decomposition-lem},
and the rest of the decomposition compiles into a generalized type~II split move associated
with a single-headed splitting route. If the latter admits a desired decomposition,
then so does the move~$M\xmapsto\eta M'$.

Thus, in the sequel we assume that the move~$M\xmapsto\eta M'$
is associated with a single-headed splitting route.
Let~$\omega=(\mu_1,\mu_2,\ldots,\mu_k,p)$ be that route.
We proceed by induction in~$k$ and construct a neat decomposition of the move~$M\mapsto M'$
into a sequence of moves in which the first~$k-1$ moves are wrinkle creation
moves, the~$k$th move is a type~II split move, and
all the other moves are
type~I elementary moves (including elementary bypass removals and eliminations) and
jump moves. Such a decomposition of the move~$M\mapsto M'$
will be referred to as \emph{anticanonical}.

Each wrinkle creation move and the type~II split move in the anticanonical decomposition should
be then neatly decomposed into type~II elementary moves, and the jump moves should be neatly decomposed
into type~I elementary moves.
This is possible according to Lemmas~\ref{neutral-move-decomposition-lem} and~\ref{split-move-decomposition-lem}.

If~$k=1$ the move
$M\xmapsto\eta M'$ is an ordinary type~II move whose anticanonical
decomposition consists of the move itself.

Suppose that~$k\geqslant2$.
The first move~$M\xmapsto\xi M_1$ of the anticanonical decomposition is a wrinkle
creation move for which~$\mu_1$ and~$\mu_2$ are the ramification mirrors.
Denote by~$\mu_1'$ and~$\mu_2'$ the~$\diagup$-mirror and the~$\diagdown$-mirror in~$M_1$
that have no predecessors in~$M$ (they are located near~$\mu_1$ and~$\mu_2$,
respectively).

The proposition will follow from the Claims~1--4 proven below.

\begin{description}
\item[Claim~1] there is a unique, up to equivalence, type~II splitting route~$\omega'=(\mu_2',\mu_3',\ldots,\mu_k',p')$ in~$M_1$
such that~$h^{M_1}_M(\wideparen\omega\cap F)$ is a subarc of~$\wideparen\omega'$, where~$F$ is the domain of~$h^{M_1}_M$ (the
partial homeomorphism~$h_M^{M_1}$ is associated with the move~$M\xmapsto\xi M_1$ as explained in Subsection~\ref{hMM-subsec}).

Let~$M_1\xmapsto{\eta'} M_2$ be a generalized type~II split move associated with~$\omega'$. Let~$\mu_1''$ be the successor of~$\mu_1'$,
and~$\mu_2''$, $\mu_2'''$ be the successors of~$\mu_2'$ for this move. These three mirrors lie on the same occupied level of~$M_2$,
which contains no other mirrors. Let~$M_3$ be the diagram obtained from~$M_2$ by removing
the mirrors~$\mu_1''$, $\mu_2''$, and~$\mu_2'''$ together with the occupied level of~$M_2$ passing through them.
This eliminates the two inessential boundary circuits that have been created by the wrinkle move~$M\xmapsto\xi M_1$
and modifies one more boundary circuit, whose essentialness or inessentialness is assumed to be preserved
by the transformation~$M_2\mapsto M_3$ (consult Figures~\ref{anticanonical-fig} and~\ref{cut-from-other-end} below).
It means, in particular, that any surface carried by~$\widehat M_2$
is also carried by~$\widehat M_3$.
Define a morphism~$\chi:\widehat M_2\rightarrow\widehat M_3$
by~$\bigl(F,F,\mathrm{id}|_F\bigr)\in\chi$, where~$F$ is any such surface.
\item[Claim~2] the transformation~$M_2\xmapsto\chi M_3$ admits a decomposition into type~I elementary moves.
\item[Claim~3] the transformation~$M_3\xmapsto\zeta M'$, where~$\zeta=\eta\circ\xi^{-1}\circ{\eta'}^{-1}\circ\chi^{-1}$,
admits a decomposition into jump moves shifting only the successors
of the occupied levels of~$M$ that are being split by the move~$M\mapsto M'$.
\item[Claim~4] the moves~$M\xmapsto\xi M_1$,
$M_1\xmapsto{\eta'} M_2$ and the decompositions of~$M_2\xmapsto\chi M_3$ and~$M_3\xmapsto\zeta M'$
into type~I elementary moves and jump moves, respectively, can be chosen to give
rise to a neat decomposition of the move~$M\xmapsto\eta M'$.
\end{description}

Thus, an anticanonical decomposition of the move~$M\xmapsto\eta M'$ is defined as the concatenation of the following
four sequences of moves:
\begin{enumerate}
\item
the wrinkle creation move~$M\xmapsto\xi M_1$,
\item
an anticanonical decomposition of the generalized type~II split move~$M_1\xmapsto{\eta'} M_2$;
\item
a decomposition into type~I elementary moves of the transformation~$M_2\xmapsto\chi M_3$
(this will be referred to as \emph{a clean-up});
\item
a sequence of jump moves that compiles into the transformation~$M_3\xmapsto\zeta M'$.
\end{enumerate}

\begin{exam}\label{anticanonical-exam}
Consider again the generalized type~II split move from Example~\ref{gen-split-move-exam}.
Figure~\ref{anticanonical-fig} shows the transitions~$M\xmapsto\xi M_1$, $M_1\xmapsto{\eta'} M_2$,
and~$M_2\xmapsto\chi M_3$. By~$1,2,\ldots$ we mark the mirrors~$\mu_1,\mu_2,\ldots$,
and by~$1',2',\ldots$ the mirrors~$\mu_1',\mu_2',\ldots$, respectively.
One can see that the final diagram in Figure~\ref{gen-split-pic-2}
can be obtained from~$M_3$ by two jump moves.

\begin{figure}[ht]
\begin{tabular}{ccc}
\includegraphics[scale=0.65]{gen-split-ex1.eps}\put(-194,139){$\scriptstyle1$}\put(-63,131){$\scriptstyle2,6$}%
\put(-135,131){$\scriptstyle5$}\put(-63,59){$\scriptstyle3$}\put(-135,59){$\scriptstyle4$}%
\put(-160,-4){$M$}
&
\raisebox{100pt}{$\longrightarrow$}
&
\includegraphics[scale=0.65]{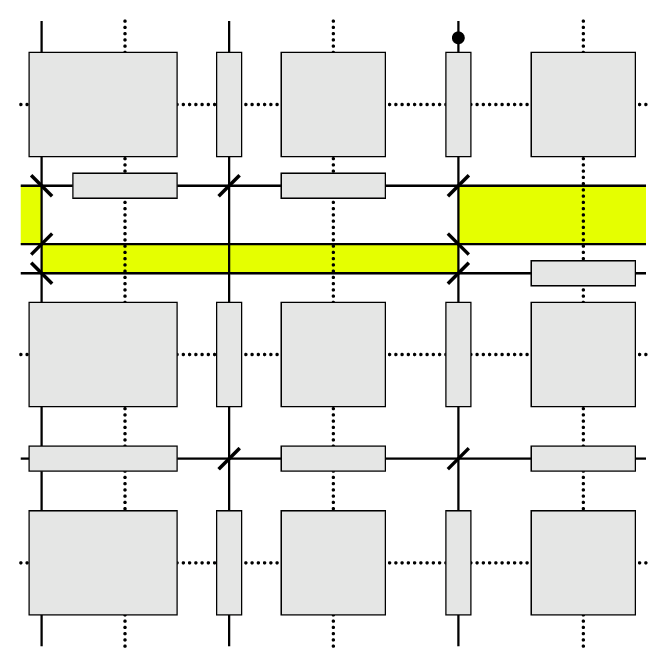}\put(-63,134.5){$\scriptstyle2'$}\put(-63,144){$\scriptstyle6'$}%
\put(-135,144){$\scriptstyle5'$}\put(-63,58){$\scriptstyle3'$}\put(-135,58){$\scriptstyle4'$}\put(-203,134.5){$\scriptstyle1'$}%
\put(-160,-4){$M_1$}
\\
&&$\big\downarrow$\\
\includegraphics[scale=0.65]{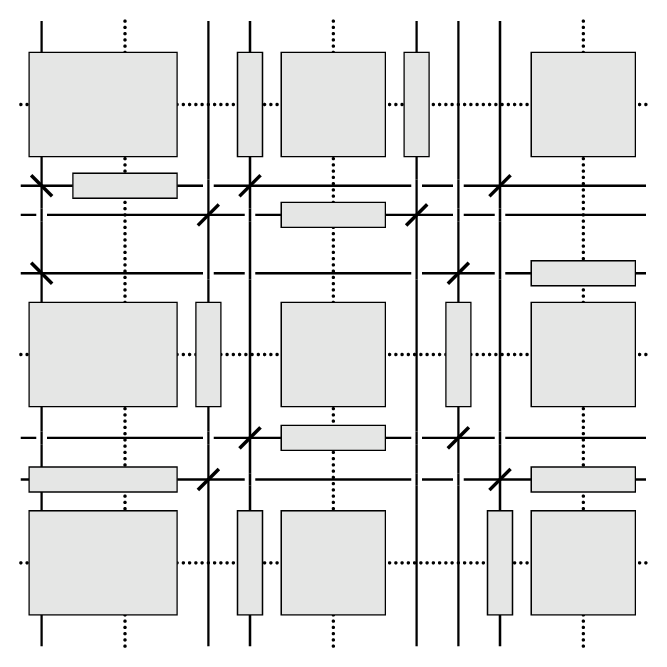}
\put(-160,-4){$M_3$}
&
\raisebox{100pt}{$\longleftarrow$}
&
\includegraphics[scale=0.65]{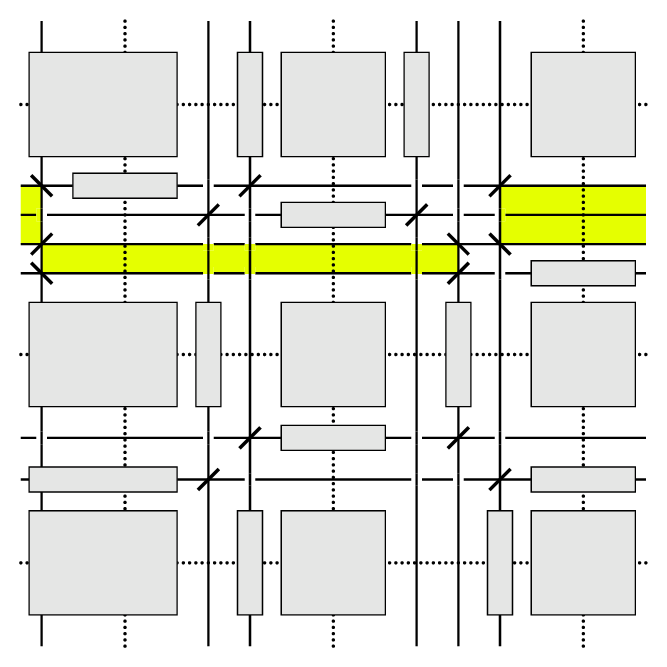}\put(-160,-4){$M_2$}\put(-192,134.5){$\scriptstyle1''$}\put(-64,134.5){$\scriptstyle2''$}\put(-51,134.5){$\scriptstyle2'''$}
\\
\end{tabular}
\caption{First three steps of the anticanonical decomposition of a generalized type~II split move
in Example~\ref{anticanonical-exam}}\label{anticanonical-fig}
\end{figure}
\end{exam}

Now we proceed with the proof of Claims~1--4. Let~$\wideparen\omega=\wideparen\omega^1\cup\ldots\cup\wideparen\omega^k$
be the decomposition of the arc~$\wideparen\omega$ into subarcs as in Definition~\ref{simple-route-def}.
This decomposition can be chosen so that the domain~$F$ of~$h_M^{M_1}$ is obtained from~$\wideparen M$ by cutting
along~$(\wideparen\mu_1\cup\wideparen x\cup\wideparen\mu_2)\cap(\wideparen\omega^1\cup\wideparen\omega^2)$,
where~$x$ is the occupied level of~$M$ containing~$\mu_1$ and~$\mu_2$,
and the intersection of~$F$ with~$\wideparen\omega$
is~$\wideparen\omega^3\cup\ldots\cup\wideparen\omega^k$.
We can also ensure that the starting point of~$h_M^{M_1}(\wideparen\omega\cap F)$
appears on the boundary of the strip~$\wideparen\mu_2'$ but not on~$\partial\wideparen M_1$.
The arc~$h_M^{M_1}(\wideparen\omega\cap F)$ is then continued uniquely (up to isotopy of~$\wideparen M_1$ preserving the handle
decomposition structure) to a type~II splitting path; see Figure~\ref{cut-from-other-end}. Claim~1 follows.
\begin{figure}[ht]
\includegraphics[scale=0.6]{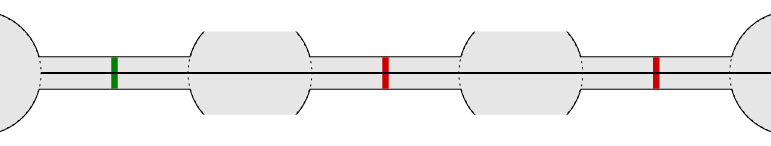}\put(-192,5){$\wideparen\mu_1$}\put(-114,5){$\wideparen\mu_2$}%
\put(-36,5){$\wideparen\mu_3$}\put(10,17){$\wideparen M$}\put(-153,10){$\wideparen x$}
$$\Big\downarrow\hbox to 0pt{wrinkle creation\hss}$$

\vskip3mm
\includegraphics[scale=0.6]{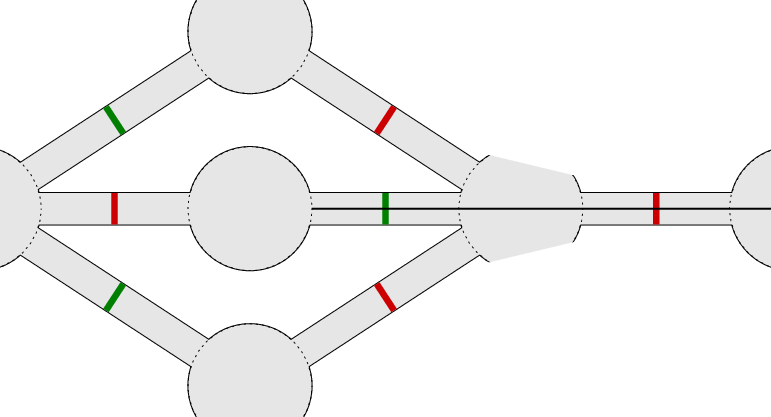}\put(-180,44){$\wideparen\mu_1'$}\put(-126,44){$\wideparen\mu_2'$}%
\put(-36,44){$\wideparen\mu_3'$}\put(-153,56){$\wideparen x'$}\put(10,56){$\wideparen M_1$}
$$\Big\downarrow\hbox to 0pt{generalized type~II splitting\hss}$$
\vskip2mm
\includegraphics[scale=0.6]{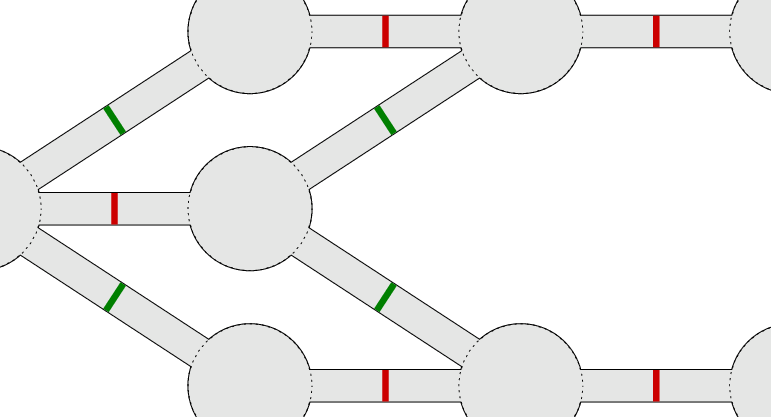}\put(-180,44){$\wideparen\mu_1''$}\put(-104,75){$\wideparen\mu_2''$}%
\put(-104,42){$\wideparen\mu_2'''$}\put(-153,56){$\wideparen x'$}\put(10,56){$\wideparen M_2$}
$$\Big\downarrow\hbox to 0pt{clean-up\hss}$$
\vskip2mm
\includegraphics[scale=0.6]{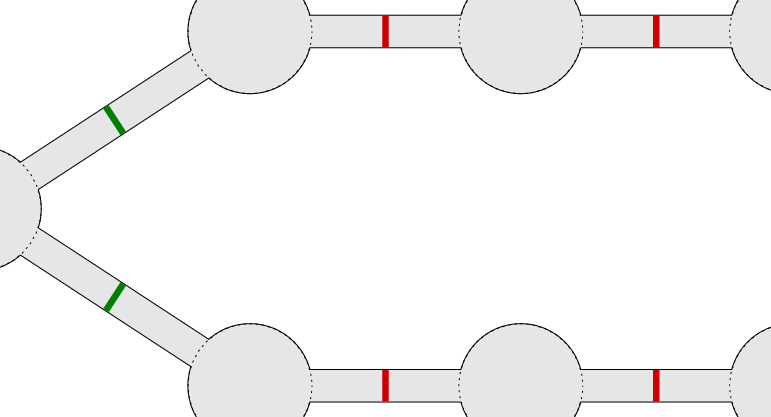}\put(10,56){$\wideparen M_3$}
\caption{Anticanonical decomposition of a generalized type~II split move from the `internal' point of view
of the surface~$\wideparen M$}\label{cut-from-other-end}
\end{figure}

The wrinkle move~$M\xmapsto\xi M_1$ creates two new boundary circuits, the ones that hit the mirror~$\mu_1'$.
These circuits have the form~$\partial r_1$, $\partial r_2$, where~$r_1$ and~$r_2$ are some rectangles
with no mirrors in the interior.
(In Figure~\ref{anticanonical-fig} these rectangles are highlighted.) One can see that the move~$M_1\mapsto M_2$
transforms them to circuits that still have the form of the boundary of a rectangle with no mirrors inside.
No splitting of the level~$x'$ containing~$\mu_1'$ and~$\mu_2'$ occurs in the canonical decomposition of~$M_1\xmapsto{\eta'} M_2$,
since the mirror~$\mu_1'$ cannot appear in~$\omega'$.
Hence $\mu_1''$, $\mu_2''$, and~$\mu_2'''$ remain on the level~$x'$.

One can now see that the deletions of~$\mu_2''$ and~$\mu_2'''$ are type~I elementary bypass removals,
and the subsequent deletion of~$\mu_1''$ together with the
occupied level~$x'$ is a type~I elimination move. Thus, Claim~2 has been settled.

We define the partial homeomorphisms~$h_{M_2}^{M_3}$ and~$h_{M_3}^{M_2}$ for the transformation~$M_2\xmapsto\chi M_3$ in the obvious way:~$h_{M_3}^{M_2}$ is the embedding~$\wideparen M_3\rightarrow\wideparen M_2$, and~$h_{M_2}^{M_3}$ is the inverse
map. Clearly, these partial homeomorphisms decompose into the partial homeomorphisms corresponding
to the elementary type~I moves involved in the decomposition of the transformation~$M_2\xmapsto\chi M_3$
discussed above.

From the `internal' point of view of the surface~$\wideparen M$ the composition of the transformations~$M\xmapsto\xi
M_1\xmapsto{\eta'} M_2\xmapsto\chi M_3$
and the generalized type~II split move~$M\xmapsto\eta M'$ do the same thing: they cut the surface~$\wideparen
M$ along the path~$\wideparen\omega$
and deform the obtained surface so that the parts of each $1$- or $0$-handle are brought to positions
near the original position of the handle.
The latter formally means that whenever~$x_1,x_2,\ldots,x_l$ are parallel occupied levels of~$M$ and~$x_1',x_2',\ldots,x_l'$ are
their respective successors for either of the transformations~$M\xmapsto\eta M'$ and~$M\xmapsto{\chi\circ\eta'\circ\xi} M_3$,
the cyclic order of~$x_1',x_2',\ldots,x_l'$ is the same as that of~$x_1,x_2,\ldots,x_l$. So, the combinatorial difference
between~$M'$ and~$M_3$ is in that the successors of each individual occupied level of~$M$ in the diagrams~$M'$ and~$M_3$
may appear in different orders, and to prove Claim~3 it is only needed to show that the required change of the order of
successors of each occupied level of~$M$ can be achieved by means of jump moves.

The proof of Claim~3 is by induction in~$k$. The induction base, $k=2$, is verified by a direct check
illustrated in Figure~\ref{claim3-induction-base-fig}, where the top row shows a canonical decomposition
of the move~$M\xmapsto\eta M'$.
\begin{figure}[ht]
\begin{tabular}{ccccc}
\includegraphics[scale=0.65]{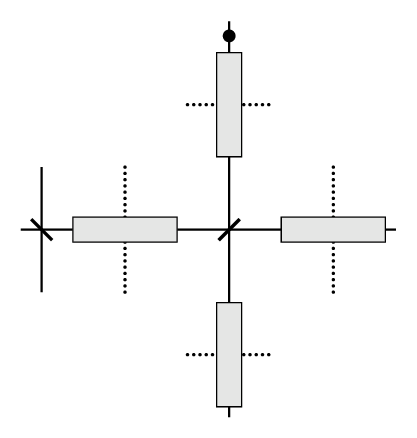}\put(-120,10){$M$}\put(-115,55){$\mu_1$}\put(-70,55){$\mu_2$}\put(-55,124){$p$}
&\raisebox{60pt}{$\longrightarrow$}&
\includegraphics[scale=0.65]{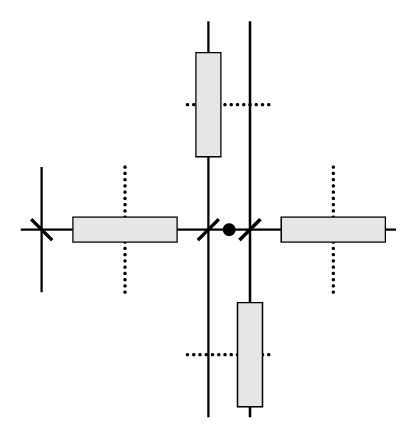}
&\raisebox{60pt}{$\longrightarrow$}&
\includegraphics[scale=0.65]{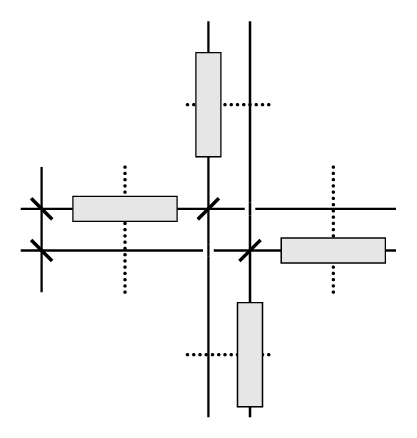}\put(-120,10){$M'$}\\
$\big\downarrow$&&&&$\big\uparrow$\\
\includegraphics[scale=0.65]{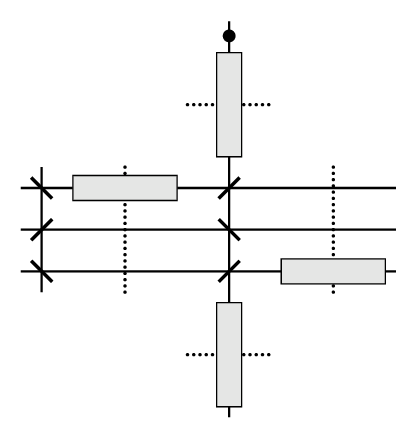}\put(-120,10){$M_1$}
&\raisebox{60pt}{$\longrightarrow$}&
\includegraphics[scale=0.65]{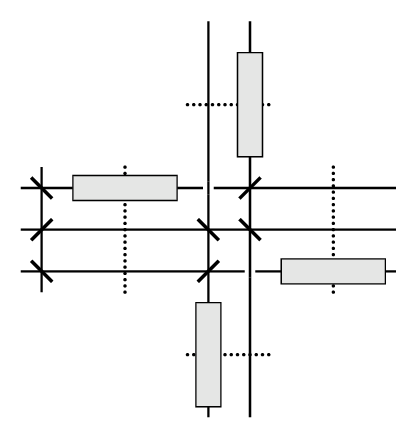}\put(-120,10){$M_2$}
&\raisebox{60pt}{$\longrightarrow$}&
\includegraphics[scale=0.65]{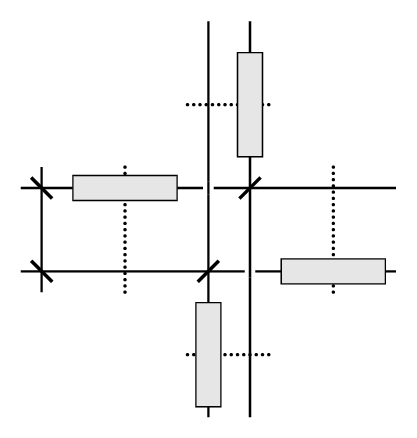}\put(-120,10){$M_3$}
\end{tabular}
\caption{Proof of Claim~3 in the case~$k=2$}\label{claim3-induction-base-fig}
\end{figure}
One can see that the combinatorial types of~$M_3$ and~$M'$ are different only in the order in which appear the two successors
of the occupied level of~$M$ containing~$\mu_2$ and~$p$, and this difference can be eliminated by a single jump move.
\begin{figure}[ht]
\centerline{\includegraphics{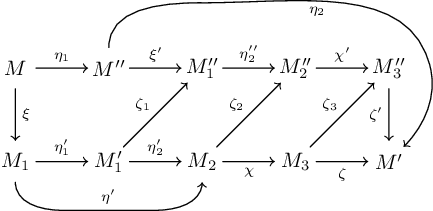}}
\caption{Proof of Claim~3, the scheme of the induction step}\label{scheme-fig}
\end{figure}

To make the induction step, let~$M\xmapsto{\eta_1}M''$ be the first type~I split move
in a canonical decomposition of the generalized type~II split move~$M\xmapsto\eta M'$,
and let~$M_1\xmapsto{\eta_1'}M_1'$ be the first type~I split move
in a canonical decomposition of the generalized type~II split move~$M_1\xmapsto{\eta'}M_2$.
Let also~$M''\xmapsto{\eta_2}M'$ and~$M_1'\xmapsto{\eta_2'}M_2$
be the remaining parts of the decompositions.

The transformation~$M''\xmapsto{\eta_2}M'$ is a generalized type~II split move
associated with a splitting route having~$k-1$ mirror entries, so we can use the induction hypothesis
and decompose it into the following transformations as described above (consult Figure~\ref{scheme-fig}):
\begin{enumerate}
\item
a wrinkle creation move~$M''\xmapsto{\xi'}M_1''$;
\item
a generalized type~II split move~$M_1''\xmapsto{\eta_2''}M_2''$ associated with a splitting route of length~$k-2$;
\item
a clean-up transformation~$M_2''\xmapsto{\chi'}M_3''$;
\item
a composition of jump moves~$M_3''\xmapsto{\zeta'}M'$.
\end{enumerate}

Let~$\omega''$ and~$\omega'''$ be the type~II splitting routes with which associated are the moves~$M_1''\xmapsto{\eta_2''}M_2''$
and~$M_1'\xmapsto{\eta_2'}M_2$, respectively. We claim that the transformation~$M_1'\xmapsto{\zeta_1}M_1''$,
where~$\zeta_1=\xi'\circ\eta_1\circ\xi^{-1}\circ{\eta_1'}^{-1}$, decomposes neatly into jump moves, and
takes~$\omega'''$ to~$\omega''$.

Again, if we look at the transformations~$M\xmapsto{\eta_1'\circ\xi}M_1'$ and~$M\xmapsto{\xi'\circ\eta_1}M_1''$ from the `internal' point of view
of the surface~$\wideparen M$, there is no distinction: both transformations result in cutting~$\wideparen M$ along the initial
(corresponding to~$\mu_1,\mu_2$)
and terminal (corresponding to~$\mu_k$) parts of~$\wideparen\omega$ and adding a `bridge' across
the new hole, which is the result of the cutting along the initial part of~$\wideparen\omega$.
This means, in particular, that there is a homeomorphism~$\psi:\wideparen M_1'\rightarrow\wideparen
M_1''$ preserving the handle decomposition structure and representing the morphism~$\zeta_1$, such that~$h^M_{M_1}\circ h_{M_1'}^{M_1}=h_{M''}^M\circ h_{M_1''}^{M''}\circ\psi$,
and for an appropriate choice of the partial homeomorphisms~$h_*^*$, we have~$\psi(\wideparen\omega''')=\wideparen\omega''$.

Let~$y_1$ and~$y_2$ be the occupied levels of~$M$ perpendicular to~$x$ and containing~$\mu_1$ and~$\mu_2$, respectively.
Without loss of generality we may assume that the occupied level~$x$ is horizontal, that is, has the form~$\ell_{\varphi_0}$
for some~$\varphi_0$. The occupied levels~$y_1$, $y_2$ have then the form~$m_{\theta_1}$ and~$m_{\theta_2}$, respectively,
for some~$\theta_1$,~$\theta_2$. To show that the transformation~$M_1'\xmapsto{\zeta_1}M_1''$ decomposes
neatly into jump moves we need to consider the following five cases.

\begin{figure}[ht]
\includegraphics[scale=0.65]{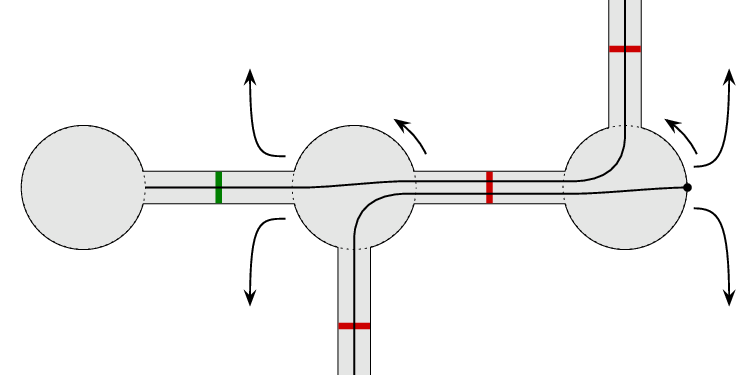}\put(-211,54){$\wideparen y_1$}\put(-126,65){$\wideparen x$}%
\put(-173,42){$\wideparen\mu_1$}\put(-96,42){$\wideparen\mu_2=\wideparen\mu_k$}%
\put(-160,12){$\wideparen x'''$}\put(-160,98){$\wideparen x''$}\put(-8,12){$\wideparen y_2'$}\put(-8,100){$\wideparen y_2''$}%
\put(-115,12){$\wideparen\mu_{k-1}$}\put(-56,100){$\wideparen\mu_3$}%
\put(-104,77){$\theta$}\put(-19,77){$\varphi$}\put(-40,45){$\wideparen y_2$}\put(-16,56){$\wideparen p$}

\begin{tabular}{ccc}
\includegraphics[scale=0.65]{split-comm01.eps}\put(-120,10){$M$}\put(-132,63){$x$}\put(-122,90){$y_1$}\put(-63,135){$y_2$}%
\put(-55,124){$p$}\put(-85,30){$\mu_3\in{}$}\put(-85,78){\begin{rotate}{-90}{$\in$}\end{rotate}}\put(-87,82){$\mu_{k-1}$}%
\put(-115,55){$\mu_1$}\put(-70,55){$\mu_2$}
&\raisebox{60pt}{$\longrightarrow$}&
\includegraphics[scale=0.65]{split-comm02.eps}\put(-120,10){$M''$}\put(-70,135){$y_2'$}\put(-56,135){$y_2''$}\\
$\big\downarrow$&&$\big\downarrow$\\
\includegraphics[scale=0.65]{split-comm04.eps}\put(-120,10){$M_1$}\put(-137,64){$x'$}\put(-137,51){$x''$}\put(-137,77){$x'''$}
&\raisebox{60pt}{$\longrightarrow$}&
\includegraphics[scale=0.65]{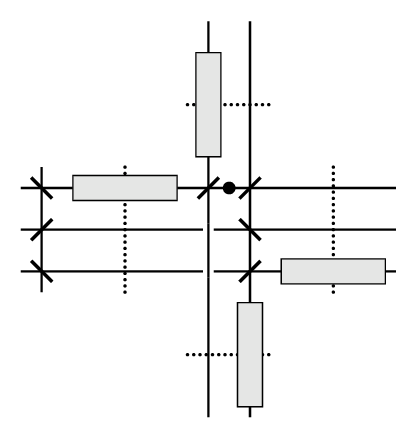}\put(-120,10){$M_1'=M_1''$}
\end{tabular}
\caption{Induction step in the proof of Claim~3, the case~$\mu_k=\mu_2$, $k\equiv0\pmod2$, $\mu_{k-1}\in(\mu_1;\mu_2)$, $\mu_3\in(p;\mu_2)$}\label{claim3case2a}
\end{figure}
\medskip\noindent\emph{Case~1}: $\mu_k$ does not lie on the occupied level~$x$.\\
In this case it is quite obvious that
the moves~$M\xmapsto{\eta_1}M''$ and~$M\xmapsto\xi M_1$
commute.

\medskip\noindent\emph{Case~2}:  $\mu_k$ coincides with~$\mu_2$, and~$k$ is even.

The occupied level being split by the move~$M\mapsto M''$ is~$y_2$.
Let~$y_2'=m_{\theta_2'}$ and $y_2''=m_{\theta_2''}$ be the two successors of~$y_2$ for
the split move~$M\mapsto M''$, and let~$x''=\ell_{\varphi_0''}$, $x'''=\ell_{\varphi_0'''}$ be the successors of~$x$ for the wrinkle
creation move~$M\mapsto M_1$. Let also~$x'$ be the occupied level of~$M_1$ with no predecessor. We order these levels
so that~$\theta_2'\in(\theta_1;\theta_2'')$ and~$\varphi_0'\in(\varphi_0'';\varphi_0''')$,
where~$x'=\ell_{\varphi_0'}$.

Due to the fact that~$\omega=(\mu_1,\mu_2,\ldots,\mu_k,p)$ is a type~II splitting route, we have either~$\mu_{k-1}\in(\mu_1;\mu_2)$
and~$\mu_3\in(p;\mu_2)$, or~$\mu_{k-1}\in(\mu_2;\mu_1)$
and~$\mu_3\in(\mu_2;p)$.
These two subcases are obtained from one another by the symmetry~$(\theta,\varphi)\mapsto(-\theta,-\varphi)$,
so it suffices to consider the first one, when~$\mu_{k-1}\in(\mu_1;\mu_2)$ and~$\mu_3\in(p;\mu_2)$.
In this case, the
$\diagup$-ramification mirror of the move~$M''\xmapsto{\xi'}M_1''$
is located at~$x\cap y_2''$, whereas the
splitting mirror of the move~$M_1\xmapsto{\eta_1'}M_1'$ is located at~$x'''\cap y_2$. 
One can see from
Figure~\ref{claim3case2a} that the diagrams~$M_1'$ and~$M_1''$ are combinatorially equivalent.

\medskip\noindent\emph{Case~3}: $\mu_k$ coincides with~$\mu_2$, and~$k$ is odd.

We may assume that~$p\in(\mu_2;\mu_1)$ as the other case is obtained by applying the transformation~$(\theta,\varphi)\mapsto(-\theta,-\varphi)$.
One can see from Figure~\ref{claim3case3} that the moves~$M\xmapsto{\eta_1}M''$ and~$M\xmapsto\xi M_1$
commute with one another.
\begin{figure}[ht]
\includegraphics[scale=0.65]{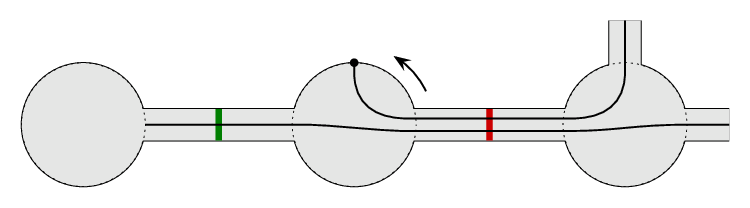}\put(-211,22){$\wideparen y_1$}\put(-126,13){$\wideparen x$}%
\put(-173,10){$\wideparen\mu_1$}\put(-96,10){$\wideparen\mu_2=\wideparen\mu_k$}%
\put(-104,45){$\theta$}\put(-42,13){$\wideparen y_2$}\put(-126,51){$\wideparen p$}

\vskip5mm
\begin{tabular}{ccc}
\includegraphics[scale=0.65]{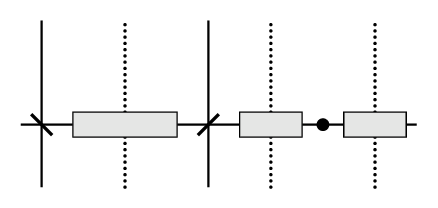}\put(-120,-5){$M$}\put(-138,25){$x$}\put(-127,63){$y_1$}\put(-75,63){$y_2$}%
\put(-120,16){$\mu_1$}\put(-83,16){$\mu_2$}\put(-38,16){$p$}
&\raisebox{30pt}{$\longrightarrow$}&
\includegraphics[scale=0.65]{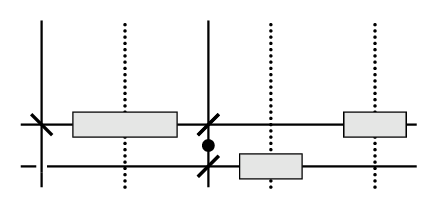}\put(-120,-5){$M''$}\\
$\big\downarrow$&&$\big\downarrow$\\
\includegraphics[scale=0.65]{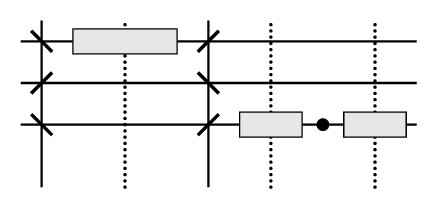}\put(-120,-5){$M_1$}
&\raisebox{30pt}{$\longrightarrow$}&
\includegraphics[scale=0.65]{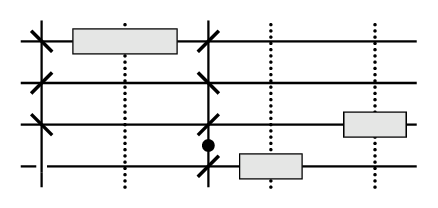}\put(-120,-5){$M_1'=M_1''$}
\end{tabular}
\caption{Induction step in the proof of Claim~3, the case~$\mu_k=\mu_2$, $k\equiv1\pmod2$}\label{claim3case3}
\end{figure}

\begin{figure}[ht]
\begin{tabular}{ccc}
\includegraphics[scale=0.65]{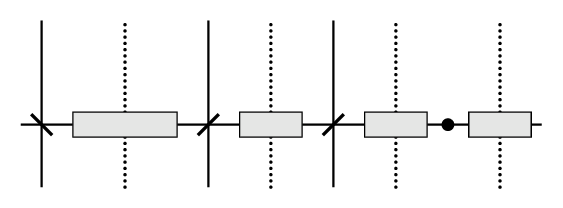}\put(-178,25){$x$}\put(-167,63){$y_1$}\put(-115,63){$y_2$}%
\put(-160,16){$\mu_1$}\put(-123,16){$\mu_2$}\put(-38,16){$p$}\put(-83,16){$\mu_k$}\put(-160,-5){$M$}
&\raisebox{30pt}{$\longrightarrow$}&
\includegraphics[scale=0.65]{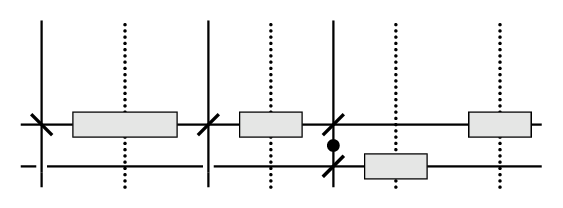}\put(-160,-5){$M''$}\\
$\big\downarrow$&&$\big\downarrow$\\
\includegraphics[scale=0.65]{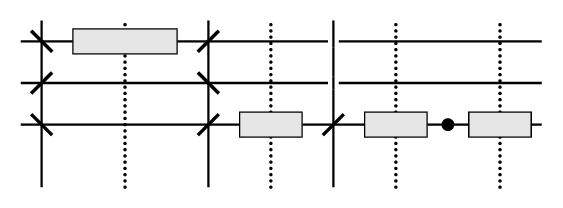}\put(-160,-5){$M_1$}
&\raisebox{30pt}{$\longrightarrow$}&
\includegraphics[scale=0.65]{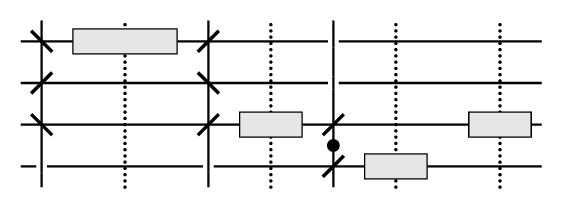}\put(-160,-5){$M_1'=M_1''$}
\end{tabular}
\caption{Induction step in the proof of Claim~3, the case~$\mu_k\in x$, $\mu_k\ne\mu_2$, $k\equiv1\pmod2$, $p\in(\mu_k;\mu_1)$}\label{claim3case5a}
\end{figure}
\medskip\noindent\emph{Case~4}: $\mu_k$ lies on~$x$, $\mu_k\ne\mu_2$, $k$ is even.

It is quite easy to see in this case that  the moves~$M\xmapsto{\eta_1}M''$ and~$M\xmapsto\xi M_1$
commute with one another. We leave this to the reader.

\medskip\noindent\emph{Case~5}: $\mu_k$ lies on~$x$, $\mu_k\ne\mu_2$, $k$ is odd.

The point~$p$ lies on the occupied level~$x$ in this case.
We may assume that~$\mu_k\in(\mu_2;\mu_1)$ as the other case is obtained by applying the transformation~$(\theta,\varphi)\mapsto(-\theta,-\varphi)$.
We then also have~$p\in(\mu_2;\mu_1)$ as otherwise~$\omega=(\mu_1,\mu_2,\ldots,\mu_k,p)$ would not be a type~II splitting route.

There are, however, two different subcases, $p\in(\mu_k;\mu_1)$ and~$p\in(\mu_2;\mu_k)$.
In the former subcase the moves producing~$M_1'$ and~$M_1''$ are shown in Figure~\ref{claim3case5a}.
One can see that the moves~$M\xmapsto{\eta_1}M''$ and~$M\xmapsto\xi M_1$ again
commute with one another.

The latter subcase is the only situation when the diagrams~$M_1'$ and~$M_1''$ are not combinatorially equivalent.
The moves in question are shown in Figure~\ref{claim3case5b}. One can see that
the moves~$M\xmapsto{\eta_1}M''$ and~$M\xmapsto\xi M_1$
almost commute with one another, so~$M_1''$, viewed up to combinatorial equivalence,
is obtained from~$M_1'$ by a single jump move.
\begin{figure}[ht]
\begin{tabular}{ccc}
\includegraphics[scale=0.65]{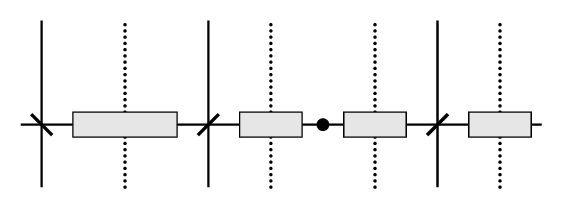}\put(-178,25){$x$}\put(-167,63){$y_1$}\put(-115,63){$y_2$}%
\put(-160,16){$\mu_1$}\put(-123,16){$\mu_2$}\put(-78,16){$p$}\put(-51,16){$\mu_k$}\put(-160,-5){$M$}
&\raisebox{30pt}{$\longrightarrow$}&
\includegraphics[scale=0.65]{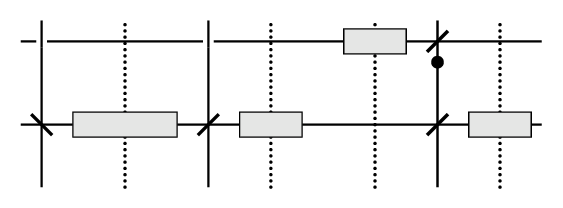}\put(-160,-5){$M''$}\\
$\big\downarrow$\\
\includegraphics[scale=0.65]{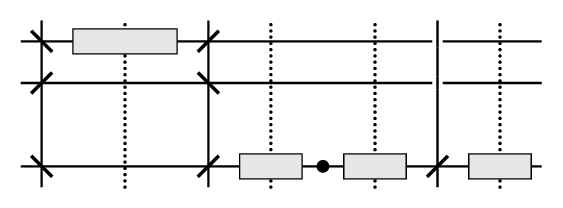}\put(-160,-5){$M_1$}
&&\raisebox{30pt}{$\big\downarrow$}\\
$\big\downarrow$\\
\includegraphics[scale=0.65]{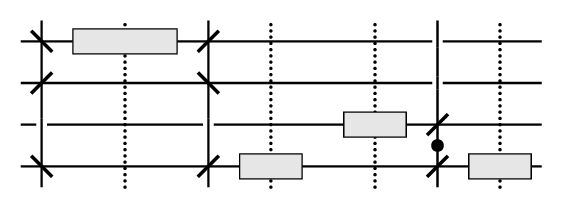}\put(-160,-5){$M_1'$}
&&
\includegraphics[scale=0.65]{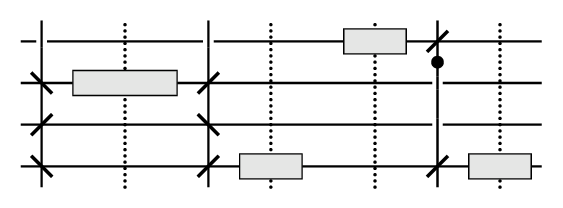}\put(-160,-5){$M_1''$}
\end{tabular}
\caption{Induction step in the proof of Claim~3, the case~$\mu_k\in x$, $\mu_k\ne\mu_2$, $k\equiv1\pmod2$, $p\in(\mu_2;\mu_k)$}\label{claim3case5b}
\end{figure}
This jump move exchanges one of the successors of the occupied level~$x$ with two
successive occupied levels of~$M_1'$, one of which is also a successor of~$x$, and the other
is~$x'$. Recall that the occupied level~$x'$ is untouched by the move~$M_1\xmapsto{\eta'} M_2$ and eliminated by
the subsequent clean-up~$M_2\xmapsto\chi M_3$.

Thus, in all cases the diagrams~$M_1'$ and~$M_1''$ are related by a sequence of jump moves that takes~$\omega'''$
to~$\omega''$ and represents the morphism~$\zeta_1=\xi'\circ\eta_1\circ\xi^{-1}\circ{\eta_1'}^{-1}$.
It follows from Lemma~\ref{split-commute-with-jump-lem} that the transformation~$M_2\xmapsto{\zeta_2}M_2''$,
where~$\zeta_2=\eta_2''\circ\zeta_1\circ{\eta_2'}^{-1}$, also admits a decomposition
into jump moves. Moreover, one can see that this transformation establishes a bijection between the
elements of~$M_2$ (occupied levels and mirrors) and those of~$M_2''$ so that an element of~$M_2$
is removed by the clean-up~$M_2\xmapsto\chi M_3$ if and only if the respective element of~$M_2''$ is removed
by the clean-up~$M_2''\xmapsto{\chi'}M_3''$. Therefore, the transformation~$M_3\xmapsto{\zeta_3}M_3''$,
where~$\zeta_3=\chi'\circ\zeta_2\circ\chi^{-1}$, also decomposes into jump moves. The induction step follows, which concludes the proof of Claim~3.

It remains to prove Claim~4, that is, to show that an anticanonical decomposition of
the generalized type~II split move~$M\mapsto M'$ can be chosen to be neat.

Let~$C$ be a collection of boundary circuits of~$M$ untouched by the move~$M\mapsto M'$,
and let~$x$ be an occupied level of~$M$.
 Among the successors of~$x$, only one can coincide with~$x$.
This means that at most one connected component of~$\wideparen x\setminus\wideparen\omega$
has a non-empty intersection with the union of the boundary components of~$\wideparen M$
corresponding to the circuits in~$C$. Let~$d$ be such a component.

As noted above, the principal difference of an anticanonical decomposition
from a canonical one, is that we start cutting the surface~$\wideparen M$ along
the path~$\wideparen\omega$ from the other end. However, the cutting path is still the same,
so~$d$ is not going to be cut as a result of a wrinkle creation move
included in an anticanonical decomposition. This implies that
we can keep the position of the $1$-handles that remain attached to~$d$ fixed
when choosing a concrete wrinkle creation move that splits the
occupied level~$x$ or one of its successors.

The clean-up operations modify only the boundary circuits that has
already been modified (or created) by the preceding moves of an anticanonical decompositions.
Finally, the jump moves involved in an anticanonical decomposition are needed
to exchange some successors of~$x$ with each other and are not needed to exchange
them with successors of another occupied level of~$\wideparen M$. So,
there is no problem to keep one selected successor of~$x$ fixed while reordering the
successors by jump moves.

Thus, we see that an anticanonical decomposition of the move~$M\mapsto M'$
can be chosen to satisfy Condition~(1) of Definition~\ref{neat-def}.
One can see that the other two conditions are satisfied automatically.

This concludes the proof of Proposition~\ref{2ndcommutation-for-merge-non-special-prop}.
\end{proof}

\subsection{Generalized wrinkle moves and flexibility}\label{flexibility-subsec}
In order to prove the second commutation property for generalized merge moves in
the special case we need some preparations.

\begin{lemm}\label{flexibility-lem}
Let~$M$ be an enhanced mirror diagram, and let~$C$ be a collection of essential
boundary circuits of~$M$. The following two conditions are equivalent.
\begin{enumerate}
\item
$M$ is $+$-flexible (respectively, $-$-flexible) relative to~$C$ (see Definition~\ref{flexibility-def}).
\item
For any $\diagup$-mirror (respectively, $\diagdown$-mirror)~$\mu$ hit by some boundary circuit~$c\notin C$,
there exists a type~II (respectively, type~I) single-headed splitting route~$\omega=(\mu_1,\ldots,\mu_k,p)$
with~$\mu_k=\mu$ such that~$\omega$ does not separate~$C$.
\end{enumerate}
\end{lemm}

\begin{proof}
(1)$\Rightarrow$(2). We consider only the case of $+$-flexibility as the other is symmetric to it.

Let~$\mu$ be a $\diagup$-mirror on~$c\in\partial M\setminus C$, and
let~$c_1,c_2,\ldots,c_m=c$ be a
sequence of boundary circuits of~$M$ satisfying the
conditions from Definition~\ref{flexibility-def}. It follows from these conditions
that there exists a sequence of mirrors
$\mu_1,\mu_2,\ldots,\mu_{i_1},\mu_{i_1+1},\ldots,\mu_{i_2},\mu_{i_2+1},\ldots,\mu_{i_m}$
of~$M$ such that:
\begin{itemize}
\item[(\romannumeral1)]
$\mu_1$ is a $\diagdown$-mirror and all the others are $\diagup$-mirrors;
\item[(\romannumeral2)]
the last mirror~$\mu_{i_m}$ is~$\mu$;
\item[(\romannumeral3)]
for all~$j=1,\ldots,m$ and $i=i_{j-1},i_{j-1}+1,\ldots,i_j-1$ the mirrors~$\mu_i$ and~$\mu_{i+1}$
lie on the same occupied level of~$M$ (we put~$i_0=1$), and either~$[\mu_i;\mu_{i+1}]$
or~$[\mu_{i+1};\mu_i]$ is a subset of~$c_j$.
\end{itemize}
The following condition on a sequence~$\mu_1,\ldots,\mu_k$ is clearly weaker than~(iii):
\begin{itemize}
\item[(\romannumeral3$'$)]
for all~$i=1,\ldots,k-1$, the mirrors~$\mu_i$ and~$\mu_{i+1}$
lie on the same occupied level of~$M$, and either~$[\mu_i;\mu_{i+1}]$ or~$[\mu_{i+1};\mu_i]$
is a subset of~$\bigcup_{c\in\partial M\setminus C}c$.
\end{itemize}
So, there exists a sequence satisfying~(i), (ii), and~(iii$'$). Take a shortest such sequence
and choose a point~$p\ne\mu$ on a boundary circuit~$c\notin C$
in a small neighborhood of~$\mu$ so that~$\mu$, $\mu_{k-1}$, and~$p$
don't lie on the same occupied level of~$M$. Then~$(\mu_1,\ldots,\mu_k,p)$
is a type~II splitting route not separating~$C$.

The implication~(2)$\Rightarrow$(1) is easy and left to the reader. We don't use it in the sequel.
\end{proof}

\begin{figure}[ht]
\includegraphics[scale=0.6]{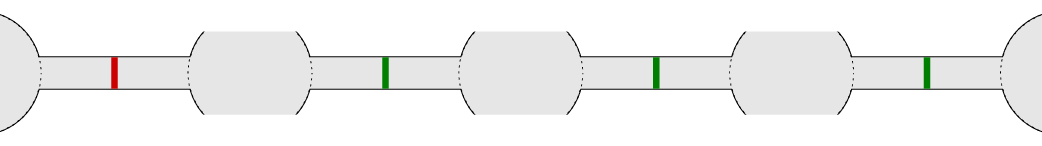}\put(-10,18){$\wideparen x_4$}%
\put(-76,18){$\wideparen x_3$}\put(-154,18){$\wideparen x_2$}\put(-232,18){$\wideparen x_1$}\put(-300,18){$\wideparen x_0$}%
\put(-37,6){$\wideparen\mu_4$}\put(-115,6){$\wideparen\mu_3$}\put(-193,6){$\wideparen\mu_2$}\put(-271,6){$\wideparen\mu_1$}

\centerline{$\big\downarrow$}

\vskip0.4cm

\includegraphics[scale=0.6]{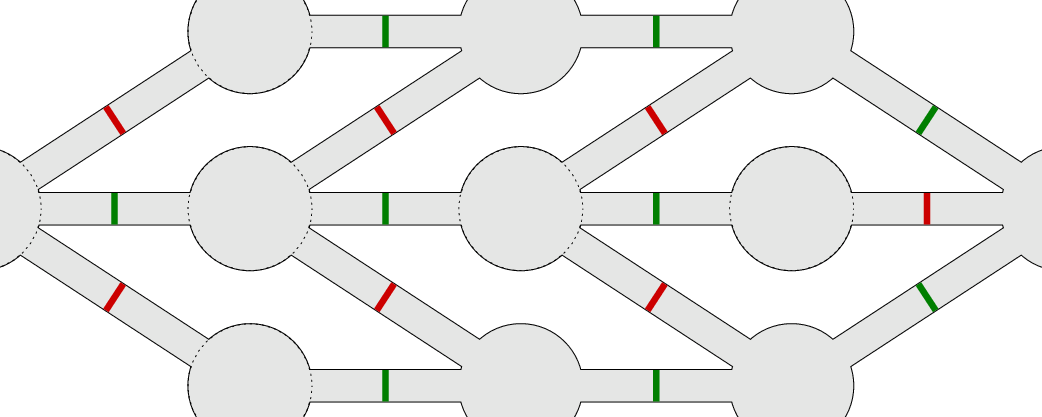}

\vskip.5cm

\includegraphics{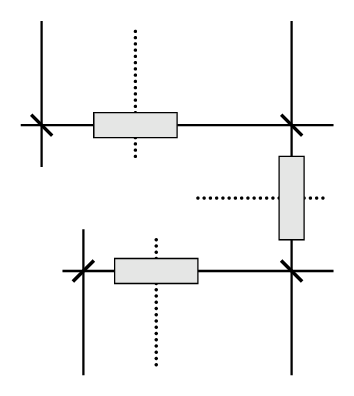}\put(-134,3){$x_0$}\put(-151,59){$x_1$}\put(-143,49){$\mu_1$}%
\put(-33,3){$x_2$}\put(-26,49){$\mu_2$}\put(-26,119){$\mu_3$}\put(-172,129){$x_3$}\put(-146,119){$\mu_4$}%
\put(-154,183){$x_4$}
\hskip.5cm\raisebox{95pt}{$\longrightarrow$}\hskip.5cm
\includegraphics{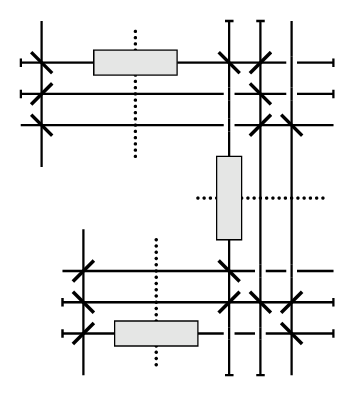}

\caption{Generalized wrinkle creation move associated with a type~I splitting route}\label{gen-wrinkle-fig}
\end{figure}
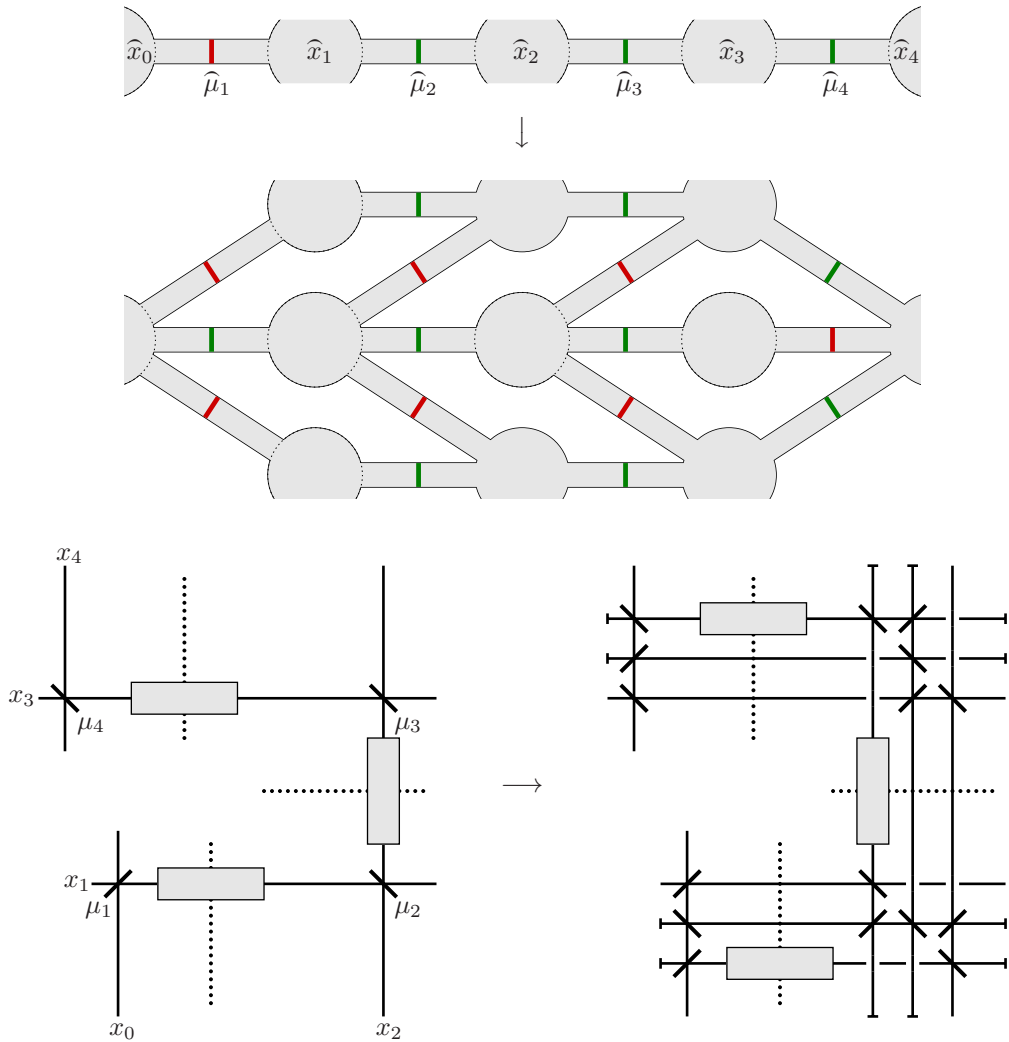

In the previous subsection we defined an anticanonical decomposition of a generalized type~II
split move. An anticanonical decomposition of a generalized type~I split move
is defined by symmetry, exchanging the roles of $\diagup$- and $\diagdown$-mirrors.

\begin{defi}
Let~$\omega=(\mu_1,\ldots,\mu_k,p)$, $k\geqslant2$, be a single-headed splitting route (of either type)
in an enhanced mirror diagram~$M$.
By \emph{a generalized wrinkle creation move associated with~$\omega$} we mean
the composition of the first~$k-1$ moves (which are wrinkle creation moves) of an anticanonical decomposition of
a generalized split move~$M\mapsto M'$ associated with~$\omega$.
Thus defined generalized wrinkle creation move will also be said to \emph{comply} with the move~$M\mapsto M'$.

The inverse operation is referred to as \emph{a generalized wrinkle reduction move}.
\end{defi}

Figure~\ref{gen-wrinkle-fig} illustrates the idea of a generalized wrinkle creation move.
Note that if~$k=2$, then a generalized wrinkle creation move
associated with~$(\mu_1,\mu_2,p)$ is an ordinary wrinkle creation move having~$\mu_1$ and~$\mu_2$
as the ramification mirrors.

The existence of an anticanonical decomposition of a non-special generalized split move implies
that, for every single-headed splitting route~$\omega=(\mu_1,\mu_2,\ldots,\mu_k,p)$ with~$k\geqslant2$
such that~$\omega$ does not separate a chosen collection~$C$ of essential boundary circuits,
there exists a generalized wrinkle creation move associated with~$\omega$ that preserves all boundary
circuits in~$C$.

So, the meaning of Lemma~\ref{flexibility-lem}
is, roughly, that we can create wrinkles anywhere in the diagram without disturbing
selected boundary circuits, provided that the diagram is
flexible relative to these boundary circuits.

Another thing the flexibility allows us to do is to avoid the use of type~II extension moves.

\begin{lemm}\label{extension-via-flexibility-lem}
Let~$M$ be an enhanced mirror diagram, and let~$C$ be a collection of essential
boundary circuits of~$M$ such that~$M$ is $+$-flexible relative to~$C$.
Then any type~II extension move that preserves the boundary circuits in~$C$
admits a $C$-neat decomposition into
type~I elementary moves combined with type~II split/merge moves.
\end{lemm}

\begin{proof}
We call a boundary circuit~$c\in\partial M\setminus C$ \emph{flexible} if
any type~II extension move~$M\mapsto M'$ that modifies~$c$ admits a $C$-neat decomposition
into type~I elementary moves and type~II split/merge moves.
It follows from Lemmas~\ref{extension-neat-decomp-lem} and~\ref{neutral-move-decomposition-lem}
that any boundary circuit~$c\in\partial M\setminus C$ with~$\tb_+(c)<0$ is flexible. It also follows that~$c$
is flexible, once \emph{some} type~II extension move~$M\mapsto M'$ that modifies~$c$ admits a $C$-neat decomposition
into type~I elementary moves and type~II split/merge moves.

It therefore suffices to show that whenever~$c,c'\in\partial M\setminus C$ are adjacent
boundary circuits one of which is flexible, the other is flexible, too.
This is demonstrated in Figure~\ref{extension-flexible-fig}, where
we assume that~$c'$ is flexible and shares a $\diagup$-mirror with~$c$.
\begin{figure}[ht]
\begin{tabular}{ccccccc}
\includegraphics[scale=.8]{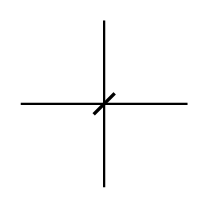}\put(-70,43){$c$}\put(-38,65){$c$}\put(-38,12){$c'$}\put(-14,43){$c'$}
&\raisebox{38pt}{$\longrightarrow$}&
\includegraphics[scale=.8]{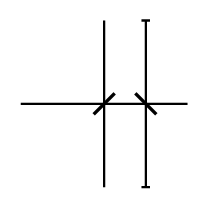}&\raisebox{38pt}{$\longrightarrow$}&
\includegraphics[scale=.8]{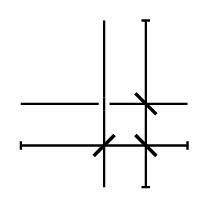}&\raisebox{38pt}{$\longrightarrow$}&
\includegraphics[scale=.8]{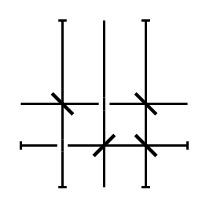}\\
&&&&&&$\big\downarrow$\\
\includegraphics[scale=.8]{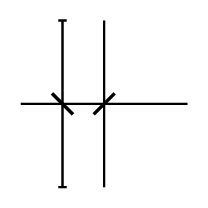}&\raisebox{38pt}{$\longleftarrow$}&
\includegraphics[scale=.8]{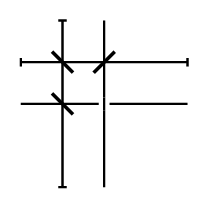}&\raisebox{38pt}{$\longleftarrow$}&
\includegraphics[scale=.8]{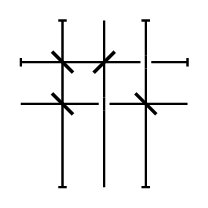}&\raisebox{38pt}{$\longleftarrow$}&
\includegraphics[scale=.8]{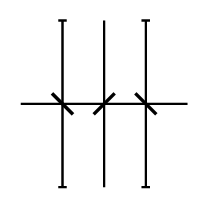}
\end{tabular}
\caption{If~$c'$ is flexible, then so is~$c$}\label{extension-flexible-fig}
\end{figure}

We apply the following moves:
\begin{enumerate}
\item
a type~II extension move that modifies~$c'$;
\item
a type~II split move;
\item
a type~II extension move that modifies the boundary circuit obtained from~$c$ by the previous
transformations, which now has~$\tb_+<0$;
\item
a type~II merge move;
\item
a type~II split move;
\item
a type~II elimination move that modifies the boundary circuit obtained from~$c'$
by the previous transformations, which now has~$\tb_+<-1$;
\item
a type~II merge move.
\end{enumerate}
None of these moves modifies any boundary circuit of the original diagram except for~$c$ and~$c'$.
\end{proof}

\subsection{Second commutation property of generalized type~II merge moves, special case}
Here we complete the proof of the second commutation property of generalized type~II merge moves.
Special generalized type~II merge moves don't have this property in general, and
a $-$-flexibility assumption on the obtained diagram is needed.

To construct the required decomposition we will need generalized wrinkle creation moves.
In order to make them friendlier to certain split moves we introduce the following

\begin{conv}
A generalized wrinkle creation move~$M\mapsto M_1$ associated with a single-headed splitting route~$\omega=(\mu_1,\ldots,\mu_k,p)$
does not depend on~$p$. However, in the context where generalized wrinkle moves are used
we will assume that a concrete choice for~$p$ has been made.

If~$M\xmapsto{\eta_1}M_1$ is a generalized wrinkle creation move that complies
with a generalized split move~$M\xmapsto\eta M'$, and~$M_1\xmapsto{\eta_2}M_2\xmapsto{\eta_3}\ldots\xmapsto{\eta_m}M_m=M'$
are the remaining moves from an anticanonical decomposition of~$M\xmapsto\eta M'$, then
we define~$h_M^{M_1}$ as $h_{M_2}^{M_1}\circ h_{M_3}^{M_2}\circ\ldots h_{M_m}^{M_{m-1}}\circ h_M^{M'}$.
(Thus, the domain of~$h_M^{M_1}$ is the same as that of~$h_M^{M'}$, and it does depend essentially on~$p$ though
the move~$M\mapsto M_1$ does not.)
\end{conv}

\begin{defi}\label{friendly-to-gen-wrinkle-def}
Let~$M\xmapsto{\eta_1}M_1$ be a generalized wrinkle creation move complying with a generalized split move~$M\xmapsto\eta M'$,
and let~$M\xmapsto\zeta M_2$ be one of the moves for which the partial homeomorphism~$h_M^{M_2}$ has
been defined.

We say that the move~$M\xmapsto\zeta M_2$ is \emph{friendly} to the move~$M\xmapsto{\eta_1}M_1$ if
it is friendly to~$M\xmapsto\eta M'$. In this case, the move~$M_2\xmapsto{{\eta_1}^\zeta}M_{21}$
\emph{resembling}~$M\xmapsto{\eta_1}M_1$ is defined as a generalized wrinkle
creation move complying with a move~$M_2\xmapsto{\eta^\zeta}M_2'$ resembling~$M\xmapsto\eta M'$.
\end{defi}

Now Definition~\ref{moves-commute-def} extends to generalized wrinkle moves accordingly.

\begin{lemm}\label{split-commutes-with-wrinkle-lem}
Let~$M$ be an enhanced mirror diagram in which a $\diagup$-mirror~$\mu_0$ can be removed
by a type~I elimination move, and let~$\omega$ be a type~II splitting route in~$M$ of the
form~$\omega=(\mu_1,\mu_0,\mu_0,\mu_2,p)$. Let also~$\sigma=(\nu_1,\nu_2,\ldots,\nu_k,q)$ be a single-headed type~I splitting route
such that, $\nu_k\in\{\mu_1,\mu_2\}$, $\nu_1,\ldots,\nu_{k-1}\notin\{\mu_0,\mu_1,\mu_2\}$,
and the associated splitting paths~$\wideparen\sigma$
and~$\wideparen\omega$ have no unavoidable intersection.
Finally, let~$M\xmapsto\eta M_1$ and~$M\xmapsto\zeta M_2$
be a generalized type~II split move associated with~$\omega$, and a generalized
wrinkle creation move associated with~$\sigma$, respectively.

Then the moves~$M\xmapsto\eta M_1$ and~$M\xmapsto\zeta M_2$
almost commute.
\end{lemm}

\begin{proof}
The proof is by induction in the length of the decomposition of~$M\xmapsto\zeta M_2$
into a sequence of ordinary wrinkle creation moves, and is similar in nature
to that of Lemma~\ref{type-i-split-commutes-with-generalized-type-ii-split-lem}.
We omit most of the detail leaving it to the reader.

\begin{figure}[ht]
\includegraphics[scale=0.65]{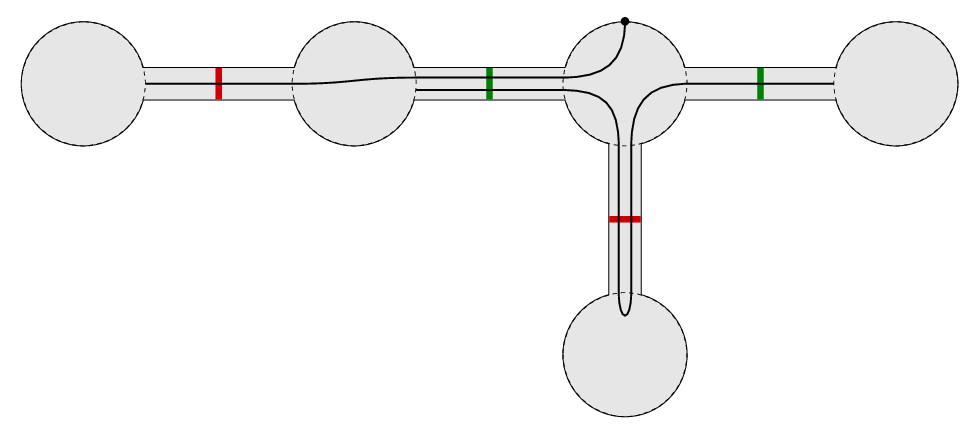}\put(-242,94.5){$\wideparen\nu_1$}%
\put(-169,94.5){$\wideparen\nu_2=\wideparen\mu_j$}\put(-73,94.5){$\wideparen\mu_{3-j}$}%
\put(-102,67){$\wideparen\mu_0$}

\vskip-2.5cm
\includegraphics[scale=0.65]{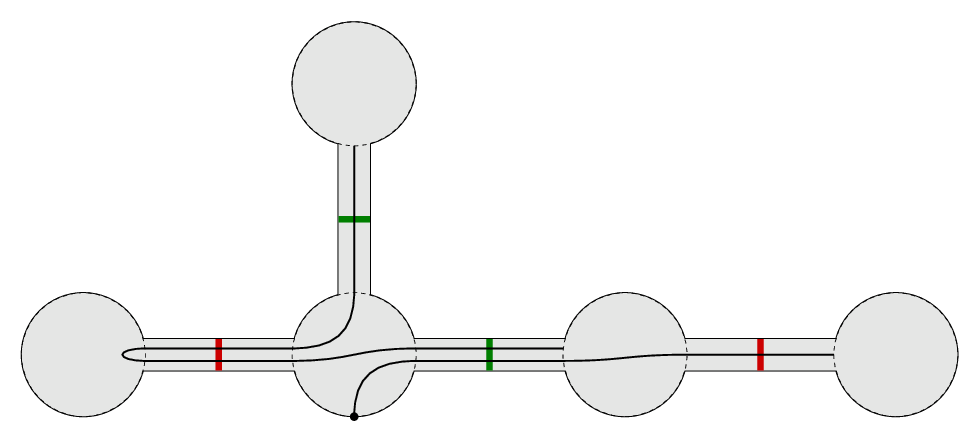}\put(-72,10){$\wideparen\nu_1$}%
\put(-169,10){$\wideparen\nu_2=\wideparen\mu_j$}\put(-242,10){$\wideparen\mu_0$}%
\put(-187,67){$\wideparen\mu_{3-j}$}

\includegraphics[scale=0.65]{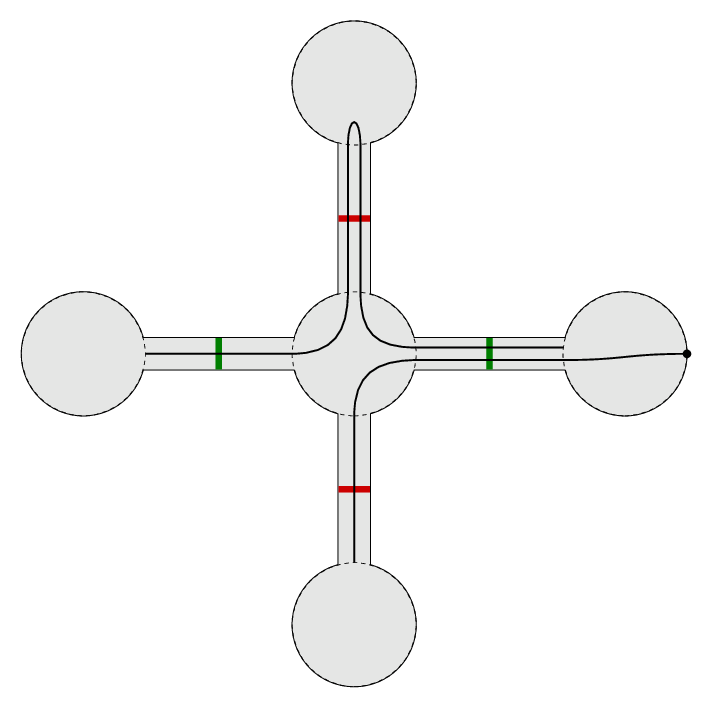}\put(-102,67){$\wideparen\nu_1$}%
\put(-84.5,94.5){$\wideparen\nu_2=\wideparen\mu_j$}\put(-102,151.5){$\wideparen\mu_0$}%
\put(-157.5,94.5){$\wideparen\mu_{3-j}$}
\includegraphics[scale=0.65]{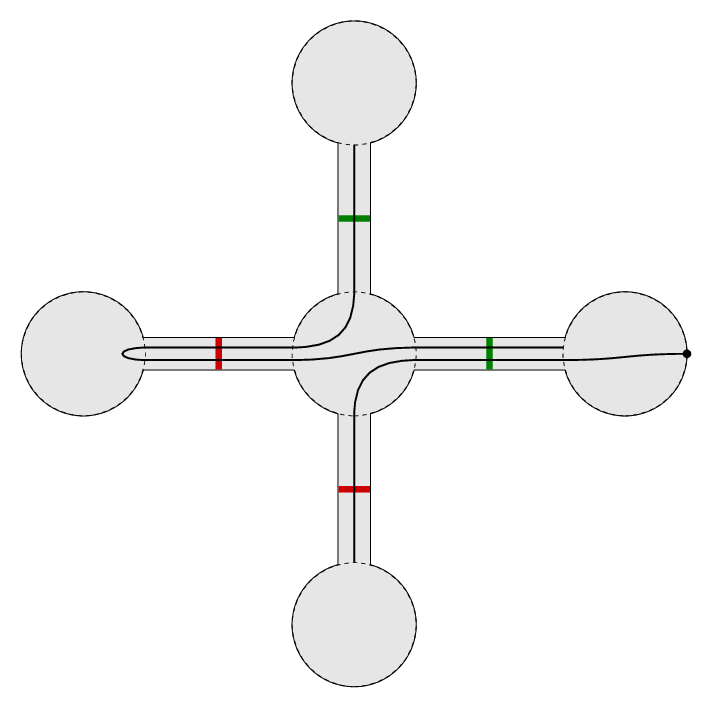}\put(-102,67){$\wideparen\nu_1$}%
\put(-84.5,94.5){$\wideparen\nu_2=\wideparen\mu_j$}\put(-102,151.5){$\wideparen\mu_{3-j}$}%
\put(-157.5,94.5){$\wideparen\mu_0$}

\vskip-2.5cm
\includegraphics[scale=0.65]{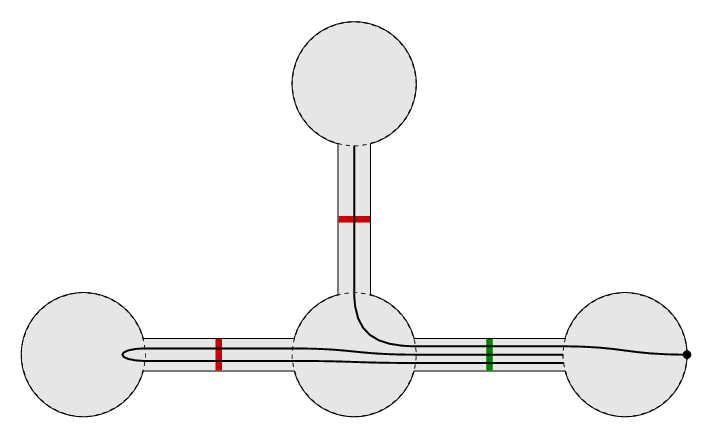}\put(-95.5,10){$\wideparen\nu_2=\wideparen\mu_1=\wideparen\mu_2$}%
\put(-169,10){$\wideparen\mu_0$}\put(-102,67){$\wideparen\nu_1$}

\vskip5mm
\includegraphics[scale=0.65]{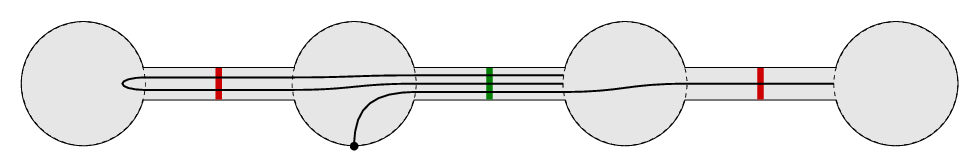}\put(-73,10){$\wideparen\nu_1$}%
\put(-180,10){$\wideparen\nu_2=\wideparen\mu_1=\wideparen\mu_2$}%
\put(-242,10){$\wideparen\mu_0$}
\caption{Possible mutual positions of~$\wideparen\omega$ and~$\wideparen\sigma$
in the case~$k=2$ in Lemma~\ref{split-commutes-with-wrinkle-lem} ($j=1,2$)}\label{wrinkle-commute-fig1}
\end{figure}

\begin{figure}[ht]
\includegraphics[scale=0.65]{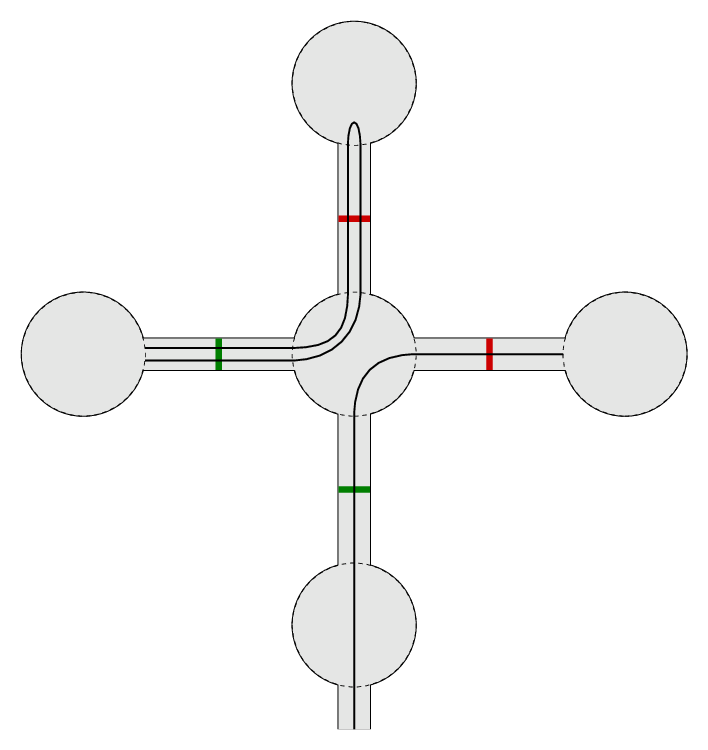}\put(-73,107.5){$\wideparen\nu_1$}%
\put(-169,107.5){$\wideparen\mu_1=\wideparen\mu_2$}\put(-102,80){$\wideparen\nu_2$}%
\put(-102,164.5){$\wideparen\mu_0$}
\includegraphics[scale=0.65]{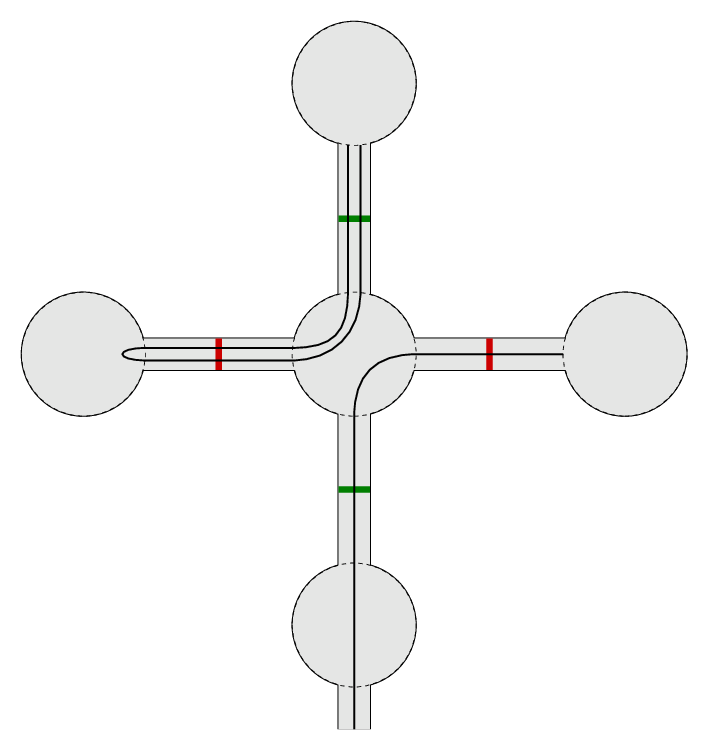}\put(-73,107.5){$\wideparen\nu_1$}%
\put(-102,164.5){$\wideparen\mu_1=\wideparen\mu_2$}\put(-102,80){$\wideparen\nu_2$}%
\put(-157.5,107.5){$\wideparen\mu_0$}

\vskip-2.5cm
\includegraphics[scale=0.65]{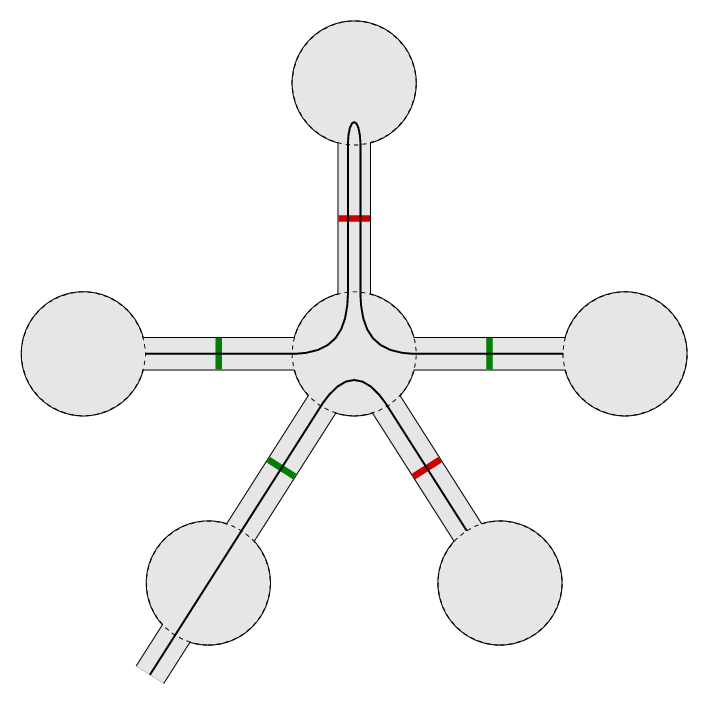}\put(-101,64){$\wideparen\nu_1$}%
\put(-127,64){$\wideparen\nu_2$}\put(-73,95){$\wideparen\mu_j$}\put(-157.5,95){$\wideparen\mu_{3-j}$}%
\put(-102,152){$\wideparen\mu_0$}

\vskip-2.5cm
\includegraphics[scale=0.65]{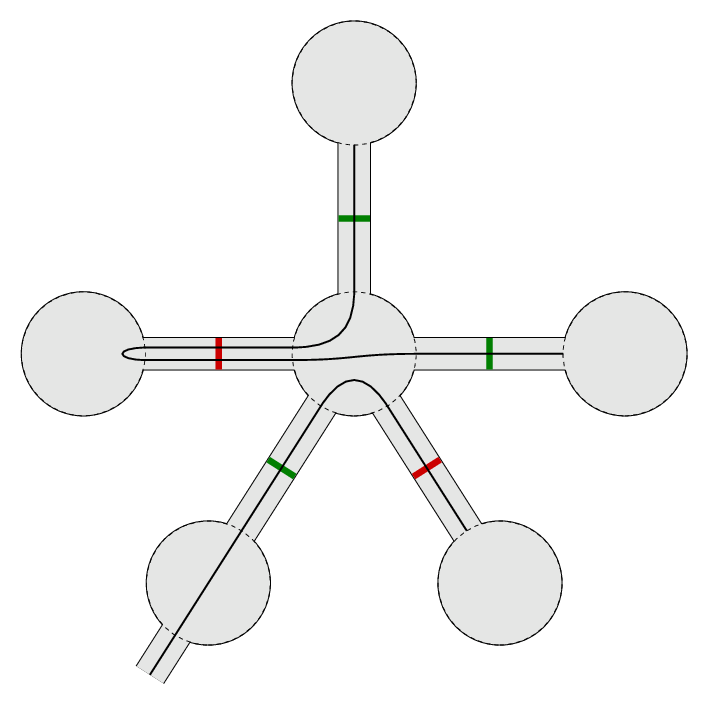}\put(-101,64){$\wideparen\nu_1$}%
\put(-127,64){$\wideparen\nu_2$}\put(-73,95){$\wideparen\mu_j$}\put(-157.5,95){$\wideparen\mu_0$}%
\put(-102,152){$\wideparen\mu_{3-j}$}
\includegraphics[scale=0.65]{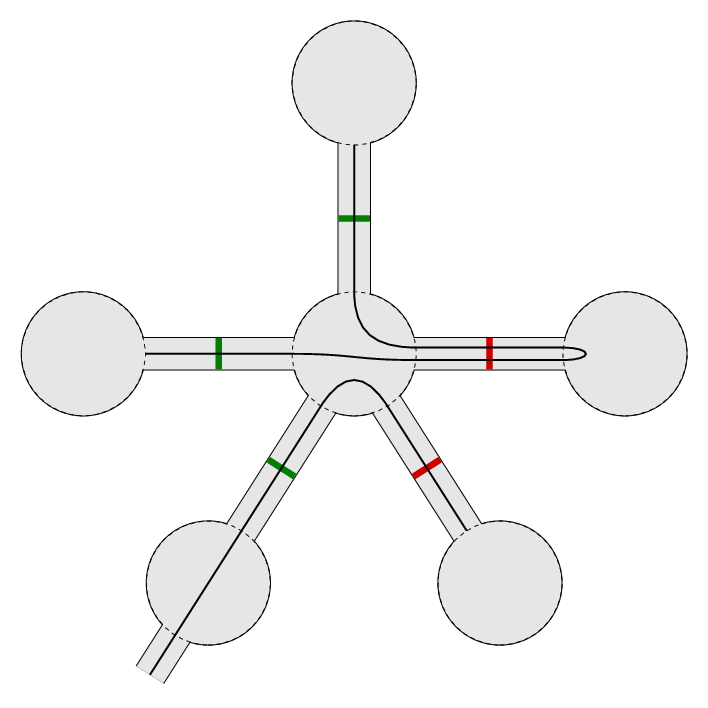}\put(-101,64){$\wideparen\nu_1$}%
\put(-127,64){$\wideparen\nu_2$}\put(-73,95){$\wideparen\mu_0$}\put(-157.5,95){$\wideparen\mu_{3-j}$}%
\put(-102,152){$\wideparen\mu_j$}
\caption{Cases to consider for the induction step in the proof of Lemma~\ref{split-commutes-with-wrinkle-lem}
($j=1,2$)}\label{wrinkle-commute-fig2}
\end{figure}

The induction base,~$k=2$, amounts to considering the six, up to symmetries, possible
mutual positions of the splitting paths~$\wideparen\omega$ and~$\wideparen\sigma$
shown in Figure~\ref{wrinkle-commute-fig1}.

For the induction step, one uses the fact that generalized wrinkle moves commute with jump moves,
which is obvious, and shows that an ordinary wrinkle creation move with ramification mirrors~$\nu_1,\nu_2$
disjoint from~$\mu_0,\mu_1,\mu_2$ almost commutes with the move~$M\xmapsto\eta M_1$.
To do so one considers six cases, in one of which the moves do not interfere, so their commutation
is obvious, and the five others are shown in Figure~\ref{wrinkle-commute-fig2}.
\end{proof}

\begin{prop}\label{2ndcommutation-for-merge-special-prop}
Let~$M$ be an enhanced mirror diagram, and let~$C$ be a collection of essential
boundary circuits of~$M$ such that~$M$ is $-$-flexible relative to~$C$.
Let also~$M\xmapsto\eta M_1$ be
a special generalized type~II split move preserving all
boundary circuits in~$C$.

Then the move~$M\xmapsto\eta M_1$ admits a $C$-neat decomposition into
a sequence of elementary moves
such that all type~II moves in it occur before all type~I moves.
\end{prop}

\begin{proof}
Let~$\omega=(\mu_1,\mu_2,p)$ be the special type~II splitting route
with which the move~$M\xmapsto\eta M_1$ is associated. Let also~$M\xmapsto{\eta'}M_1'$ be
another such move associated with~$\omega$, and preserving all~$c\in C$. It follows from Proposition~\ref{similarity-prop}
that the sought-for decomposition
of the move~$M\xmapsto\eta M_1$ exists if and only if it exists for~$M\xmapsto{\eta'}M_1'$.
This means that, without loss of generality, we can prescribe the position of the auxiliary
mirror, which we denote by~$\mu_0$, of the move~$M\xmapsto\eta M_1$ at our wish.

By Lemma~\ref{flexibility-lem} there exists a type~I single-headed splitting route~$\sigma=(\nu_1,\ldots,\nu_k,q)$
with~$\nu_k\in\{\mu_1,\mu_2\}$ such that~$\sigma$ does not separate~$C$. Take a shortest
such~$\sigma$. This will ensure that~$\nu_i\notin\{\mu_1,\mu_2\}$ for all~$i=1,\ldots,k-1$.
Moreover, the associated splitting path~$\wideparen\sigma$ will have no unavoidable intersection with~$\wideparen\omega$.

Now choose the position of~$\mu_0$ using the following rules. If~$\mu_1=\nu_k$
put~$\mu_0$ close to~$\mu_2$ and outside~$\bigcup_{c\in C}c\cup\{q\}$. If~$\mu_2=\nu_k$ put~$\mu_0$
close to~$\mu_1$ and outside~$\bigcup_{c\in C}c\cup\{q\}$. Let~$M\xmapsto{\eta_1}M'$ be the first
move in a canonical decomposition of the move~$M\xmapsto\eta M_1$, that is,
the type~I extension move that adds the mirror~$\mu_0$ to the diagram.
Then the remaining part of the canonical decomposition is a generalized
type~II split move~$M'\xmapsto{\eta_2}M_1$ associated with~$\omega'=(\mu_1,\mu_0,\mu_0,\mu_2,p)$.

Define~$M'\xmapsto\zeta M_2'$ to be a generalized wrinkle creation move associated with~$\sigma$.
We choose it to preserve all the boundary circuits in~$C$. This move clearly commutes
with the elimination move~$M'\xmapsto{\eta_1^{-1}}M$, so we have
a generalized wrinkle creation move~$M\xmapsto{\zeta^{\eta_1^{-1}}}M_2$
and an elimination move~$M_2'\xmapsto\xi M_2$, where~$\xi=(\eta_1^{-1})^\zeta$.

By construction, the splitting paths~$\wideparen\omega'$ and~$\wideparen\sigma$ in~$\wideparen M'$
have no unavoidable intersection, so we can apply Lemma~\ref{split-commutes-with-wrinkle-lem} to
conclude that the moves~$M'\xmapsto{\eta_2}M_1$ and~$M'\xmapsto\zeta M_2'$ almost commute.
Let~$M_1\xmapsto{\zeta^{\eta_2}}M_{12}$ and~$M_2'\xmapsto{\eta_2^\zeta}M_{21}$
be the moves resembling~$M'\xmapsto\zeta M_2'$ and~$M'\xmapsto{\eta_2}M_1$, respectively
(see Figure~\ref{special-split-ii-i-decomposition-fig} for the general scheme of moves).

Denote by~$\omega''=(\mu_1',\mu_0',\mu_0',\mu_1',p')$ the type~II splitting route in~$M_2'$
with which the move~$M_2'\xmapsto{\eta_2^\zeta}M_{21}$ is associated.
This route is the image of~$\omega'$ under the generalized wrinkle creation move~$M'\xmapsto\zeta M_2'$.
Denote also by~$x$ the occupied level of~$M_2'$ containing the mirrors~$\mu_0'$, $\mu_1'$, and~$\mu_2'$
(since~$\mu_0'\ne\mu_1'$ this level is unique).

The move~$M'\xmapsto\zeta M_2'$ creates a number of inessential boundary circuits
having form of the boundary of a rectangle. Since~$\nu_k\in\{\mu_1,\mu_2\}$, one of these
circuits~$c$, say, hits either~$\mu_1'$ or~$\mu_2'$. There is a unique $\diagup$-mirror~$\mu_*$ of~$M_2'$ on~$x$
hit by~$c$. One can see that~$\omega'''=(\mu_1',\mu_*,\mu_*,\mu_2',p')$ is a type~II splitting route
similar to~$\omega''$. One can also see that~$\omega'''$ does not separate~$C$.

Let~$M_2'\xmapsto\chi M_3$ be a generalized type~II splitting route associated with~$\omega'''$
and preserving all the boundary circuits in~$C$. Since~$\mu_0'$ does not appear in~$\omega'''$,
this move commutes with the elimination move~$M_2'\xmapsto\xi M_2$, so we have
a type~I elimination move~$M_3\xmapsto{\xi^\chi}M_4$ and a generalized type~II
split move~$M_2\xmapsto{\chi^\xi}M_4$ that preserve all the boundary circuits in~$C$.

We are ready to produce the sought-for decomposition of the move~$M\xmapsto\eta M_1$.
It consists of the following parts:
\begin{figure}[ht]
\includegraphics{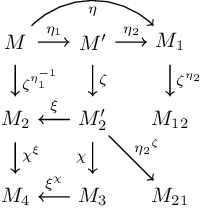}
\caption{The scheme of the moves in the proof of
Proposition~\ref{2ndcommutation-for-merge-special-prop}}\label{special-split-ii-i-decomposition-fig}
\end{figure}
\begin{enumerate}
\item
a decomposition of the generalized wrinkle creation move~$M\xmapsto{\zeta^{\eta_1^{-1}}}M_2$ into type~II elementary moves,
(the decomposition exists by Lemma~\ref{neutral-move-decomposition-lem});
\item
a decomposition of the non-special type~II split move~$M_2\xmapsto{\chi^\xi}M_4$ into a sequence of elementary
moves such that all type~II moves in it occur before type~I moves
(the decomposition exists by Proposition~\ref{2ndcommutation-for-merge-non-special-prop})
\item
type~I extension move~$M_4\xmapsto{(\xi^\chi)^{-1}}M_3$;
\item
a decomposition into type~I elementary moves of the transformation~$M_3\xmapsto{\eta_2^\zeta\circ\chi^{-1}}M_{21}$
(the decomposition exists by Proposition~\ref{similarity-prop});
\item
a decomposition into type~I elementary moves of the transformation~$M_{21}\xmapsto{\zeta^{\eta_2}\circ\eta_2\circ\zeta^{-1}
\circ(\eta_2^\zeta)^{-1}}M_{12}$ (the decomposition exists by Lemmas~\ref{split-commutes-with-wrinkle-lem}
and~\ref{neutral-move-decomposition-lem});
\item
and finally, a decomposition into type~I elementary moves of the generalized wrinkle
reduction~$M_{12}\xmapsto{(\zeta^{\eta_2})^{-1}}M_1$
(the decomposition exists by Lemma~\ref{neutral-move-decomposition-lem});
\end{enumerate}
All the listed decompositions are assumed to be $C$-neat.
\end{proof}

\subsection{Generalized type~II bypass removals. First commutation property}
Generalized type~II split moves don't have the first commutation property, and should
be replaced by something that has it.
Generalized type~II bypass removals, which are defined below, provide for a proper replacement.
If the mirror diagrams are viewed up to type~I moves, then
generalized type~II bypass removals do exactly the same job as generalized
type~II splittings do (see Subsection~\ref{type-ii-moves-for-divided-graph-subsec}), but unlike the latter they do have the first commutation property
as we will see in a moment.
Another important property of generalized type~II bypass removals is that
these moves, unlike generalized type~II split moves, are always $-$-safe-to-postpone (see Definition~\ref{safe-def}).

\begin{defi}\label{generalized-bypass-def}
Let~$M$ and~$M'$ be enhanced mirror diagram such that the following holds:
\begin{enumerate}
\item
the transition~$\widehat M\mapsto\widehat M'$
is a handle removal;
\item
all mirrors in~$E_M\setminus E_{M'}$ are $\diagup$-mirrors;
\item
there is a $+$-negligible boundary circuit~$c\in\partial M\setminus\partial M'$ (see Definition~\ref{neglibigle-circuit-def}).
\end{enumerate}
Then we say that~$M'$ is obtained from~$M$ by \emph{a generalized
type~II bypass removal}. With this operation we associate a morphism~$\eta:\widehat M\rightarrow\widehat M'$
by requesting that~$\bigl(F,F,\mathrm{id}|_F\bigr)\in\eta$, where~$F$ is a surface
obtained from~$\wideparen M$ by a patching of~$\wideparen c$.
The inverse operation~$M'\xmapsto{\eta^{-1}}M$ is called \emph{a generalized type~II bypass addition}.
\end{defi}

According to this definition, if~$M\mapsto M'$ is a generalized type~II bypass removal, then the closure of~$\widehat M\setminus\widehat M'$
is an arc  with endpoints on~$\widehat M'$. (The endpoints may coincide, thus making this arc a loop.)
This means that there is a sequence~$\beta=(y_0,\nu_1,y_1,\nu_2,\ldots,\nu_k,y_k)$
in which~$y_i\in L_M$, $\nu_i=y_{i-1}\cap y_i\in E_M\setminus E_{M'}$, $i=1,\ldots,k$, and $y_0,y_k\in L_{M'}$.
This sequence, which is defined uniquely up to reversing the order,
will be referred to as \emph{the type~II bypass removed by the move~$M\mapsto M'$}.
Note that the elements that are actually removed do not include the first and the last ones, which are
the occupied levels~$y_0$ and~$y_k$. Note also that the occupied levels~$y_i$, $i=0,1,\ldots,k$, must be pairwise distinct, with the only exception that~$y_0$ may coincide with~$y_k$.
Each level~$y_i$ with~$1\leqslant i\leqslant k-1$ must contain exactly two mirrors of~$M$,
which are~$\nu_i$ and~$\nu_{i+1}$.

Definition~\ref{generalized-bypass-def} requires the existence of a $+$-negligible boundary circuit~$c$
of~$M$ hitting each of~$\nu_i$, $i=1,\ldots,k$, exactly once. Since $c$ is inessential, there must
exist a patching disc for~$\widehat c$, and, as follows from Lemma~\ref{rectangular-representative-lem}, the patching disc can be chosen in the form~$\bigcup_{r\in\Pi}\widehat r$,
where~$\Pi$ is a collection of rectangles. This can be visualized by drawing the rectangles from~$\Pi$
together with the mirror diagram~$M$, and following the conventions made for both kinds of pictures,
presenting mirror diagrams and presenting rectangular diagrams of a surface.

There is another boundary circuit of~$M$ hitting each of~$\nu_i$, $i=1,\ldots,k$, exactly once, which we denote by~$c'$.
The boundary circuits~$c$ and~$c'$ are said to be \emph{adjacent} to the bypass~$\beta$. They are
replaced in~$M'$ by a single boundary circuit, which is referred to as~$c\#c'$.

\begin{exam}
\begin{figure}[ht]
\includegraphics[scale=0.5]{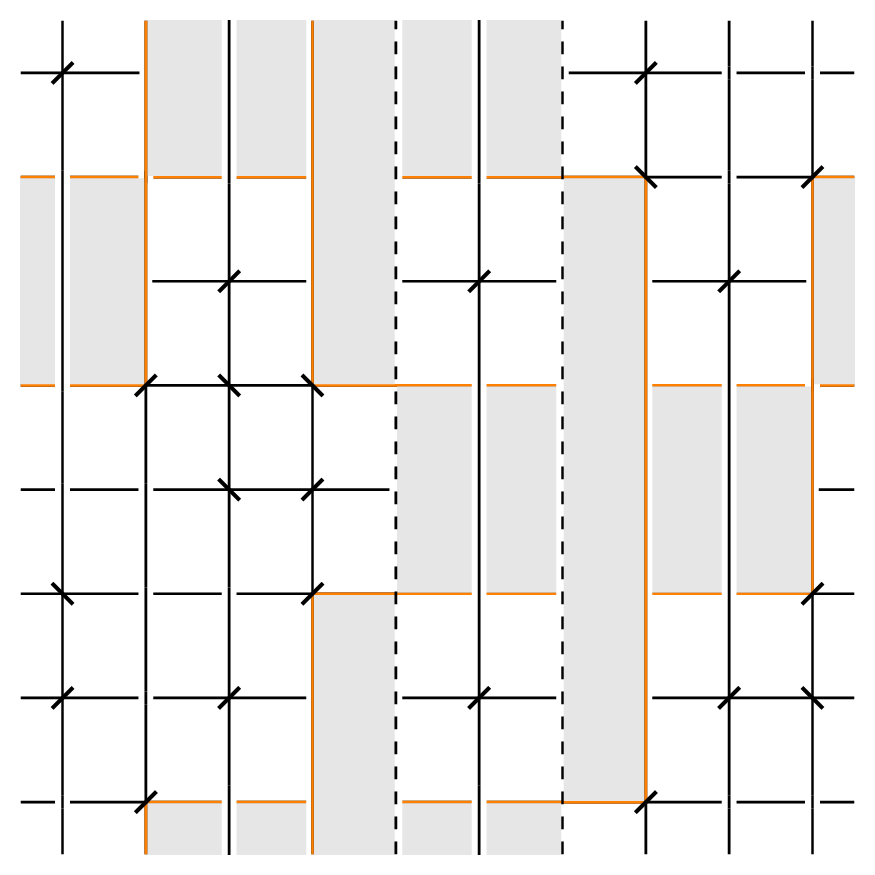}\put(-186,110){$\nu_1$}\put(-186,10){$\nu_2$}\put(-66,10){$\nu_3$}%
\put(-130,125){$r_1$}\put(-110,75){$r_2$}\put(-190,125){$r_3$}\put(-170,175){$r_4$}\put(-70,25){$r_5$}%
\put(-133,110){$\mu'$}\put(-53,160){$\mu''$}%
\put(-217,116){$y_0$}\put(-217,16){$y_2$}\put(-179,-2){$y_1$}\put(-59,-2){$y_3$}
\hskip.3cm\raisebox{100pt}{$\longrightarrow$}\hskip.3cm
\includegraphics[scale=0.5]{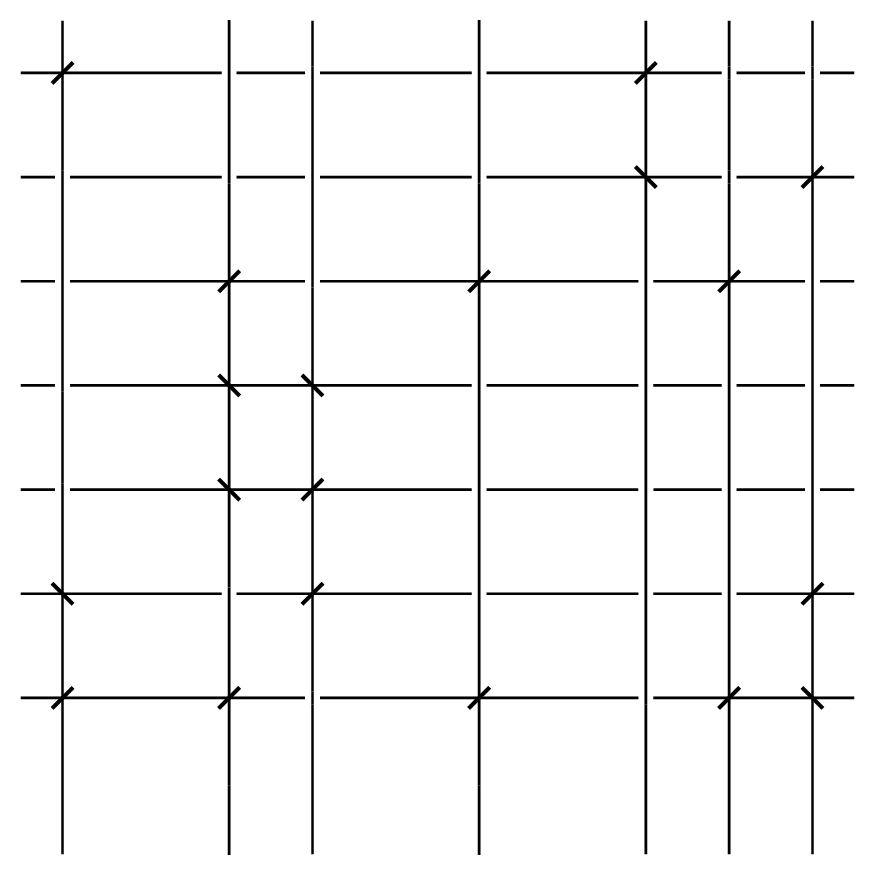}\\\vskip.5cm
\includegraphics[scale=0.5]{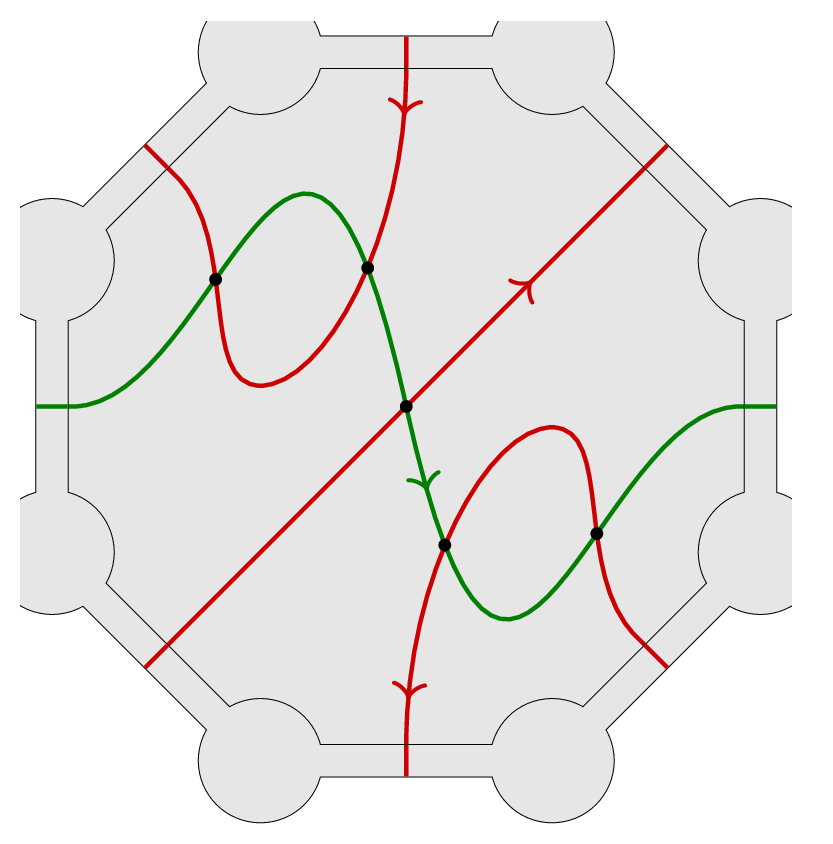}\put(-174,33){$\wideparen\nu_1$}\put(-105,5){$\wideparen\nu_2$}\put(-39,33){$\wideparen\nu_3$}%
\put(-205,103){$\wideparen\mu'$}\put(-10,103){$\wideparen\mu''$}%
\put(-160,136){$\mathring r_1$}\put(-124,137){$\mathring r_2$}\put(-115,106){$\mathring r_3$}%
\put(-105,72){$\mathring r_4$}\put(-68,76){$\mathring r_5$}\put(-50,120){$\widehat\Pi$}%
\put(-191,67){$\wideparen y_0$}\put(-141,17){$\wideparen y_1$}\put(-71,17){$\wideparen y_2$}\put(-21,67){$\wideparen y_3$}
\caption{Generalized type~II bypass removal and the structure of the corresponding
patching disc~$\widehat\Pi=\widehat r_1\cup\widehat r_2\cup\widehat r_3\cup\widehat r_4\cup\widehat r_5$}\label{gen-bypass-rem-fig}
\end{figure}
Shown in Figure~\ref{gen-bypass-rem-fig} is an example of a generalized type~II bypass removal~$M\mapsto M'$.
The top left picture shows the diagram~$M$ together with a collection of rectangles representing
a patching disc for a $+$-negligible circuit~$c$ which is destroyed by the move.
The circuit~$c$ is shown in orange. The dashed lines show two meridians that are not occupied
levels of~$M$. They correspond to two intersection points of the interior of the patching
disc with~$\mathbb S^1_{\tau=1}$. Denoted by~$\mu'$ and~$\mu''$
are the only two $\diagdown$-mirrors on~$c$ due to which~$c$ is $+$-negligible.

The bottom picture shows the structure of a canonic dividing configuration~$(\delta_+,\delta_-)$ of the patching disc.
The crucial point about this dividing configuration is that~$\delta_+$ is a single arc. This follows from
the condition~$\tb_+(c)=-1$.
\end{exam}

\begin{prop}\label{1st-comm-prop-for-bypass-1-prop}
Let~$M\xmapsto\eta M'$ be a generalized type~II bypass removal,
and let~$M'\xmapsto\zeta M''$ be a type~I elementary move. Let also~$C$
be the set of all boundary circuits of~$M$ preserved by both moves.
Then there exist transformations~$M\xmapsto{\zeta'}M'''$
and~$M'''\xmapsto{\eta'}M''$ preserving all the boundary circuits in~$C$ such that the following holds:
\begin{enumerate}
\item
we have~$\eta'\circ\zeta'=\zeta\circ\eta$;
\item
the transformation~$M'''\xmapsto{\eta'}M''$ is a generalized type~II bypass removal;
\item
the transformation~$M\xmapsto{\zeta'}M'''$ admits a neat decomposition into a sequence of type~I elementary moves
which yields a $C$-neat decomposition of~$M\xmapsto{\zeta\circ\eta}M''$
after appending the move~$M'''\xmapsto{\eta'}M''$ at the end of the sequence.
\end{enumerate}
\end{prop}

\begin{proof}
Generalized type~II bypass removals are $-$-safe-to-postpone. Therefore,
we need not worry about the $C$-neatness of the decomposition of~$M\xmapsto{\zeta\circ\eta}M''$, once the constructed decomposition of~$M\xmapsto{\zeta'}M'''$ is neat.

The statement is established similarly to~\cite[Proposition~2]{DyPr}
and to Proposition~\ref{subdiagram-move-prop} above. So, we will be a little sketchy, providing
only the details that are specific in the present case.

As usually, we have to consider different kinds of moves~$M'\xmapsto\zeta M''$ one by one.
The bypass removed by~$M\xmapsto\eta M'$ is denoted~$(y_0,\nu_1,y_1,\ldots,\nu_k,y_k)$.

\medskip\noindent\emph{Case 1}: $M'\xmapsto\zeta M''$ is a type~I extension move.\\
Let~$\mu$ and~$x$ be the mirror and the occupied level that are added by the move.
The move~$M\xmapsto{\zeta'}M'''$ should also add~$x$ and~$\mu$.
If~$x=y_i$ for some~$i\in\{1,\ldots,k-1\}$, then prior to that the diagram should be modified by
a jump move that disturbs the position of~$y_i$ slightly.

\medskip\noindent\emph{Case 2}: $M'\xmapsto\zeta M''$ is a type~I elimination move.\\
Let~$\mu$ and~$x$ be the mirror and the occupied level that are eliminated by the move.
Let also~$z$ be the other occupied level of~$M$ passing through~$\mu$.
The move~$M\xmapsto{\zeta'}M'''$ simply removes~$\mu$ and~$x$, unless~$x\in\{y_0,y_k\}$.

If~$x=y_0\ne y_k$ we put~$M'''=M$ and~$\zeta'=\mathrm{id}|_{\widehat M}$.
In this case, the transformation~$M\xmapsto{\zeta\circ\eta}M''$ is already
a generalized type~II bypass removal, and the removed bypass is~$(z,\mu,y_0,\nu_1,y_1,\ldots,\nu_k,y_k)$.

Similarly, if~$x=y_k\ne y_0$, we can take~$M$ for~$M'''$ as~$M\xmapsto{\zeta\circ\eta}M''$
removes the bypass~$(y_0,\nu_1,y_1,\ldots,\nu_k,y_k,\mu,z)$.

Suppose that~$x=y_0=y_k$. Pick a point~$p$ on~$x$ distinct from~$\mu$, $\nu_1$, and~$\nu_2$ so that the pair~$\{\mu,p\}$ interleaves
with~$\{\nu_1,\nu_k\}$, and take for~$M\xmapsto{\zeta'}M'''$ a type~I split move
with splitting mirror~$\mu$ and snip point~$p$. The occupied level~$x$ will have two successors in~$M'''$,
which we denote by~$x_1$, $x_2$, numbering so that~$x_1$ contains the successor of~$\nu_1$ (which is unique),
and~$x_2$ contains the successor of~$\nu_k$ (which is also unique).
We also denote the successors of~$\nu_1$ and~$\nu_k$ by~$\nu_1'$ and $\nu_k'$,
respectively, and the two successors of~$\mu$ contained in~$x_1$ and~$x_2$ by~$\mu_1$ and~$\mu_2$, respectively.

One can see that~$M'''\xmapsto{\eta'}M''$, where~$\eta'=\zeta\circ\eta\circ{\zeta'}^{-1}$, removes the following
bypass:
$$(z,\mu_1,x_1,\nu_1',{}y_1,\nu_2,\ldots,\nu_{k-1},y_{k-1},\nu_k',x_2,\mu_2,z).$$

\medskip\noindent\emph{Case 3}: $M'\xmapsto\zeta M''$ is a type~I elementary bypass addition.\\
We use the notation from Definition~\ref{mirr-bypass-def}, in which we substitute~$M''$
for~$M$. Denote also by~$\mu$ the mirror being added by the move,
and by~$\mu_1$, $\mu_2$, $\mu_3$ the mirrors at~$(\theta_1,\varphi_1)$, $(\theta_2,\varphi_1)$,
and~$(\theta_1,\varphi_2)$, respectively.

If no mirrors of the bypass appear in~$r=[\theta_1;\theta_2]\times[\varphi_1;\varphi_2]$,
then~$M\xmapsto{\zeta'}M'''$ simply adds a $\diagdown$-mirror at~$(\theta_2,\varphi_2)$,
and this is also a type~I elementary bypass addition.
Otherwise we have to modify the bypass in order to remove the obstacles.

There are the following two possible kinds of obstacles:
\begin{enumerate}
\item
one of, or both, of the mirrors~$\nu_1$, $\nu_k$ appear in~$[\mu_1;\mu_2]\cup[\mu_1;\mu_3]$;
\item
some of the mirrors~$\nu_i$ appear in~$\Omega=(\theta_1;\theta_2]\times(\varphi_1;\varphi_2]$.
\end{enumerate}

Obstacles of the first kind are removed by type~I split moves with splitting mirror~$\mu_2$
or~$\mu_3$. For instance, suppose that~$\nu_1\in(\mu_1;\mu_2)$
and~$\nu_k\notin(\nu_1;\mu_2)$. Let~$\theta_0=\theta(\nu_1)$.

For a small enough~$\varepsilon>0$, the longitude~$\ell_{\varphi_1+\varepsilon}$ is not an occupied level of~$M$,
and there are no mirrors of~$M$ in~$[\theta_0;\theta_2]\times(\varphi_1,\varphi_1+\varepsilon]$.
We apply a type~I split move that adds a $\diagup$-mirrors~$\nu_1'$ and~$\mu_2'$ at~$(\theta_0,\varphi_1+\varepsilon)$
and~$(\theta_2,\varphi_1+\varepsilon)$, respectively, to~$M$,
and removes~$\nu_1$.

The new bypass is
$$(m_{\theta_2},\mu_2',\ell_{\varphi_1+\varepsilon},\nu_1',y_1,\nu_2,\ldots,\nu_k,y_k).$$

Obstacles of the second kind are removed by applying Lemma~\ref{rem-obst-lem}. Note, however,
that we have
to use here not only the formulation of the Lemma, but also the proof. One can
see that the procedure
used to remove mirrors from~$\Omega$ transforms a bypass to another bypass.

The removing of obstacles procedure gives us the sought-for transformation~$M\xmapsto{\zeta'}M'''$
coming with a decomposition into type~I split moves
and jump moves, which can be decomposed neatly into type~I elementary moves.
These moves do not alter any part of the diagram~$M$ except for the bypass.

\medskip\noindent\emph{Case 4}: $M'\xmapsto\zeta M''$ is a type~I elementary bypass removal.\\
We again use the notation from Definition~\ref{mirr-bypass-def}, in which we now substitute~$M'$ and~$M''$
for~$M$ and~$M'$, respectively, and denote by~$\mu$ and~$\mu'$ the mirrors at~$(\theta_2,\varphi_2)$
and~$(\theta_1,\varphi_1)$, respectively. Denote also by~$c$ and~$c'$
the boundary circuits of~$M$ adjacent to the bypass being removed by~$M\xmapsto\eta M'$.

If the circuit~$c\#c'$ of~$M'$ is patchable, then so are both~$c$ and~$c'$, and hence~$\tb_+(c\#c')=
\tb_+(c)+\tb_+(c')\leqslant-2$. This means that we cannot have~$\partial r=c\#c'$,
since~$\tb_+(\partial r)=-1$. Therefore, $\partial r$ is a boundary circuit of~$M'$
preserved by the generalized type~II bypass addition~$M'\xmapsto{\eta^{-1}}M$.

If a patchable hole is preserved by a generalized bypass addition, this hole
remains patchable. Thus, $\partial r$ is a patchable hole of~$M$,
which means that the mirrors~$\mu$ and~$\mu'$ are coherent (see Definition~\ref{coherent-mirror-def}),
and by Lemma~\ref{remove-coherent-mirror-lem} the removal of~$\mu$
can be decomposed neatly into type~I elementary moves. So,
in this case, the transformation~$M\xmapsto{\zeta'}M'''$ is just the removal of~$\mu$
endowed with the morphism~$\zeta'$ defined by the identical map on
a surface obtained from~$\wideparen M$ by a patching of~$\wideparen c$.

\medskip\noindent\emph{Case 5}: $M'\xmapsto\zeta M''$ is a type~I slide move.\\
This case is treated similarly to Case~3. If the bypass being removed by the move~$M\xmapsto\eta M'$
creates an obstacle for doing the desired slide move on~$M$, the bypass is modified
by the same means.
\end{proof}

\begin{defi}\label{almost-neat-def}
Let~$M_0\xmapsto{\eta_1}M_1\xmapsto{\eta_2}\ldots\xmapsto{\eta_k}M_k$ be a sequence
of transformations of enhanced mirror diagrams, and let~$C=C_0$ be a family of boundary
circuits of~$M_0$. For each~$i=1,\ldots,k$, let~$C_i$ be the image of~$C_0$
in~$M_i$ under the transformation~$M_0\xmapsto{\eta_i\circ\eta_{i-1}\circ\ldots\circ\eta_1}M_i$.

The sequence~$M_0\xmapsto{\eta_1}M_1\xmapsto{\eta_2}\ldots\xmapsto{\eta_k}M_k$
is called \emph{an almost $C$-neat decomposition}
of the transformation~$M_0\xmapsto{\eta_k\circ\eta_{k-1}\circ\ldots,\eta_1}M_k$
if the following two condition hold:
\begin{enumerate}
\item
for any~$i=1,\ldots,k$, either~$M_{i-1}\xmapsto{\eta_i}M_i$ is a jump move or it preserves
all the boundary circuits in~$C_{i-1}$;
\item
for any~$0\leqslant i_1<i_2\leqslant k$,
whenever an essential boundary circuit~$c\in\partial M_{i_1}\setminus C_{i_1}$ and the respective
boundary circuit~$c'\in\partial M_{i_2}\setminus C_{i_2}$ have negative Thurston--Bennequin number~$\tb_+$
(respectively,~$\tb_-$), so also  have the corresponding boundary circuits in each $\partial M_j\setminus C_j$ with~$j\in[i_1,i_2]$.
\end{enumerate}

An almost $C$-neat decomposition is said to be of \emph{type~I} (respectively, of type~II) if
it includes only type~I (respectively, type~II) elementary moves and jump moves.
\end{defi}

Note that an almost $C$-neat decomposition is also $C$-delicate, provided
that~$C$ includes only essential boundary circuits, and that the moves include
only jump moves and elementary ones.

\begin{prop}\label{1st-comm-prop-for-bypass-2-prop}
Let~$M\xmapsto\eta M'$ be a generalized type~II bypass removal,
and let~$M'\xmapsto\zeta M''$ be a jump move. Let also~$C$
be the set of all boundary circuits of~$M$ preserved by the move~$M\xmapsto\eta M'$
(but not necessarily by~$M'\xmapsto\zeta M''$).
Then there exist transformations~$M\xmapsto{\zeta'}M'''$
and~$M'''\xmapsto{\eta'}M''$ such that the following holds:
\begin{enumerate}
\item
we have~$\eta'\circ\zeta'=\zeta\circ\eta$;
\item
the transformation~$M\xmapsto{\zeta'}M'''$ admits a type~I almost $C$-neat decomposition;
\item
the transformation~$M'''\xmapsto{\eta'}M''$ is a generalized type~II bypass removal;
\item
the transformation~$M\xmapsto{\zeta'}M'''$ takes~$C$ to a collection~$C'\subset\partial_{\mathrm e} M'''$
of boundary circuits such that the move~$M'''\xmapsto{\eta'}M''$ preserves them.
\end{enumerate}
\end{prop}

\begin{proof}
We proceed similarly to Case~3 of the proof of Proposition~\ref{1st-comm-prop-for-bypass-1-prop}.
The bypass removed by~$M\xmapsto\eta M'$ is denoted~$(y_0,\nu_1,y_1,\ldots,\nu_k,y_k)$.

We use the notation from Definition~\ref{mir-jump-def} substituting~$M'$ and~$M''$ for~$M$ and~$M'$,
respectively. So, the jump move~$M'\xmapsto\zeta M''$ takes the occupied level~$\ell_{\varphi_1}$
with all mirrors on it to~$\ell_{\varphi_2}$. The almost $C$-neat
decomposition of the move~$M\xmapsto{\zeta'}M'''$
should end up with a jump move that does the same thing, namely, shifts~$\ell_{\varphi_1}$ to~$\ell_{\varphi_2}$.
The preconditions for this move may not hold for~$M$, and the possible obstacles may be
only of the following two kinds:
\begin{enumerate}
\item
one of, or both of, the mirrors~$\nu_1$ and~$\nu_k$ appear in~$[\theta_2;\theta_1]\times\{\varphi_1\}$;
\item
some mirrors~$\nu_i$ appear in~$(\theta_1;\theta_2)\times(\varphi_1;\varphi_2]$.
\end{enumerate}

Obstacles of the first kind are removed by sliding `the attachment points' of the bypass along~$\ell_{\varphi_1}$
into the interval~$(\theta_1;\theta_2)\times\{\varphi_1\}$. For instance, suppose that~$\nu_1\in[\theta_2;\theta_1]\times\{\varphi_1\}$
and~$\nu_k\notin[\theta_2;\theta_0)\times\{\varphi_1\}$,
where~$\theta_0=\theta(\nu_1)$. Then for a small enough~$\varepsilon$ there are no mirrors of~$M$
in~$[\theta_2-\varepsilon;\theta_0)\times[\varphi_1-\varepsilon,\varphi_1)$, and~$m_{\theta_2-\varepsilon}$, $\ell_{\varphi_1-\varepsilon}$
are not occupied levels of~$M$.

We apply to~$M$ two type~I extension moves that add a $\diagup$-mirror~$\nu_1'$ at~$(\theta_2-\varepsilon,\varphi_1)$, and a $\diagup$-mirror~$\nu_1'''$ at~$(\theta_0,\varphi_1-\varepsilon)$, and then apply a type~I slide move
that replaces~$\nu_1$ by a $\diagup$-mirror~$\nu_1''$ at~$(\theta_2-\varepsilon,\varphi_1-\varepsilon)$. The new bypass is
$$(\ell_{\varphi_1}=y_0,\nu_1',m_{\theta_2-\varepsilon},\nu_1'',\ell_{\varphi_1-\varepsilon},\nu_1''',m_{\theta_0}=y_1,\nu_2,\ldots,\nu_k,y_k).$$

The obstacles of the second kind are again removed by means of Lemma~\ref{rem-obst-lem}. We leave the details to the reader.
\end{proof}

\subsection{Second commutation property of generalized type~II bypass removals}
The whole of this subsection is devoted to the proof of the following statement.

\begin{prop}\label{2ndcommutation-for-bypass-prop}
Let~$M$ be an enhanced mirror diagram, and let~$C$ be a collection of essential
boundary circuits of~$M$.
Let also~$M\xmapsto\eta M'$ be a generalized type~II bypass removal that preserves all the
boundary circuits in~$C$. Assume that~$M'$ is $-$-flexible relative to~$C$.

Then the move~$M\xmapsto\eta M'$
admits 
an almost $C$-neat decomposition into a sequence of elementary moves and jump moves
such that all type~I elementary moves in it occur before all type~II elementary moves.
\end{prop}

\begin{proof}
Let~$\beta=(y_0,\nu_1,y_1,\ldots,\nu_k,y_k)$ be the bypass removed by the move~$M\xmapsto\eta M'$.
Denote by~$c$ the $+$-negligible circuit of~$M$ adjacent to~$\beta$
(there can be two such boundary circuits, in which case take any of them),
and by~$c'$ the other boundary circuit of~$M$ adjacent to~$\beta$.

As noted above there exists a patching disc for~$\widehat c$ having
the form~$\bigcup_{r\in\Pi}\widehat r$, where~$\Pi$ is a collection of rectangles.
For brevity we will refer to~$\Pi$ as a patching disc.
Though~$\Pi$ is not necessarily a rectangular diagram of a surface (topologically~$\widehat\Pi$
can be a surface with finitely many identifications between boundary points),
the associated mirror diagram~$M(\Pi)$ (see Definition~\ref{ass-mir-diagr-def})
still has perfect sense for~$\Pi$.

\emph{The complexity} of the patching disc~$\Pi$ is defined as the sum~$n_1+n_2+n_3$, where~$n_1$
is the number of rectangles in~$\Pi$, $n_2$ is the number of mirrors in~$E_{M'}$ hit by~$c$,
and~$n_3$ is the number of occupied levels of~$M(\Pi)$ that do not belong to~$L_M$.

We proceed by induction in the complexity of the patching disc.

We say that the patching disc~$\Pi$ can be \emph{simplified} if there are transformations~$M\xmapsto\zeta M_1$
and~$M'\xmapsto{\zeta'}M_1'$ with the following properties:
\begin{enumerate}
\item
$M\xmapsto\zeta M_1$ (respectively, $M'\xmapsto{\zeta'}M_1'$)
admits
a type~I (respectively, type~II) almost $C$-neat decomposition;
\item
the transformation~$M_1\xmapsto{\eta_1}M_1'$, where~$\eta_1=\zeta'\circ\eta\circ\zeta^{-1}$
is a generalized type~II bypass removal;
\item
if~$a\in C$, and~$M\xmapsto\zeta M_1$ transforms~$a$ to~$a_1$, then
the move~$M_1\xmapsto{\eta_1}M_1'$ preserves~$a_1$;
\item
there is a patching disc~$\Pi'$ associated with~$M_1\xmapsto{\eta_1}M_1'$
that has lower complexity than~$\Pi$ has.
\end{enumerate}

Clearly, if such a simplification exists and the assertion of the proposition holds for the move~$M_1\xmapsto{\eta_1}M_1'$,
then it also holds for the original move~$M\xmapsto\eta M'$.

For~$j=1,2,\ldots,k-1$ we denote by~$I_j$ the one of the two subintervals~$(\nu_j;\nu_{j+1}),(\nu_{j+1};\nu_j)$ of~$y_j$ that is
contained in~$c$. By~$I_0$ (respectively,~$I_k$) we denote the minimal subinterval of~$y_0$
(respectively, $y_k$) contained in~$c$ that either has the form~$(\mu;\nu_1)$ or~$(\nu_1;\mu)$
(respectively, $(\mu;\nu_k)$ or~$(\nu_k;\mu)$) with~$\mu\in E_M$.
Denote also by~$\mu'$ (respectively, $\mu''$) the mirror of~$M$ such that $I_0$ is either~$(\mu';\nu_1)$
or~$(\nu_1;\mu')$ (respectively, $I_k$ is either~$(\mu'';\nu_k)$ or~$(\nu_k;\mu'')$).

\begin{lemm}
If~$\bigcup_{j=0}^kI_j\cup\{\mu',\mu''\}$ contains a $\diagup$-mirror of~$M(\Pi)$,
then~$\Pi$ can be simplified.
\end{lemm}

\begin{proof}
Suppose there is a $\diagup$-mirror $\lambda$ of~$M(\Pi)$ in~$I_j$, $j\in\{1,\ldots,k-1\}$.
Denote by~$x$ the occupied level of~$M(\Pi)$ perpendicular to~$y_j$ and
passing through~$\lambda$. We claim that~$x\notin L_M$.
Indeed, if~$x\in L_M$, then
the addition of~$\lambda$ to~$M$ would split the boundary circuit~$c$
into two patchable boundary circuits~$c_1$, $c_2$, say, such that~$\tb_+(c_1)+\tb_+(c_2)=-1$,
which is impossible.

Pick a point~$p$ in $y_j\setminus\overline I_j$. Let~$M\xmapsto\zeta M_1$
be the composition of the type~I extension move that adds the $\diagup$-mirror~$\lambda$ together
with the occupied level~$x$,
and a type~I split move associated with~$(\lambda,p)$. The mirrors~$\nu_j$, $\nu_{j+1}$
have unique successors in~$M_1$, which we denote by~$\nu_j'$ and~$\nu_{j+1}'$, respectively.
The occupied level~$y_j$ will have two successors, which we denote by~$y_j'$ and~$y_j''$
so as to have~$\nu_j'\in y_j'$, $\nu_{j+1}'\in y_j''$. The mirror~$\lambda$ will also have
two successors, $\lambda'\in y_j'$ and~$\lambda''\in y_j''$.

Now let~$\beta'$ be the sequence obtained from~$\beta$ by replacing
the subsequence~$(\nu_j,y_j,\nu_{j+1})$ with
\begin{equation}\label{bypass-replacement-eq}
(\nu_j',y_j',\lambda',x,\lambda'',y_j'',\nu_{j+1}').
\end{equation}
We put~$M_1'=M'$ and~$\zeta'=\mathrm{id}$. One can see that~$M_1\xmapsto{\eta\circ\zeta^{-1}}M_1'$
is a generalized bypass removal, and the respective bypass is~$\beta'$.

One can also see that a patching disc~$\Pi'$ associated with~$M_1\xmapsto{\eta\circ\zeta^{-1}}M_1'$
can be obtained from~$\Pi$ by a small perturbation of rectangles having one of the sides on~$y_j$
(see the second move in Figure~\ref{two-more-moves-fig} for an example of what happens with~$\Pi$).

In the passage from~$\Pi$ to~$\Pi'$, the
numbers~$n_1$ and~$n_2$ in the definition of the complexity of~$\Pi$ are unchanged, whereas~$n_3$ drops by~$1$, since~$x$ is now
an occupied level of~$M_1$. Hence, a simplification of the patching disc occurs.

In the case when~$\lambda\in I_0$ or~$\lambda\in I_k$ we proceed as above
substituting~$\mu'$ for~$\nu_0$ if~$j=0$, and~$\mu''$ for~$\nu_{k+1}$ if~$j=k$.
If~$j=0$ (respectively,~$j=k$) we also choose~$p$ close to~$\nu_1$ (respectively, to~$\nu_k$), and
define the move~$M\xmapsto\zeta M_1$ so as to have~$y_0'=y_0$ (respectively, $y_k''=y_k$). We also drop the first two (respectively, the last two) entries in~\eqref{bypass-replacement-eq}.

If~$\mu'$ is a $\diagup$-mirror, we let~$M\xmapsto\zeta M_1$ be a type~I
split move associated with the splitting route~$(\mu',p)$, where~$p\in y_0\setminus\overline I_0$
is close to~$\nu_1$, such that the successor of~$y_0$ that does not
contain the successor of~$\nu_1$ coincides with~$y_0$.

We again put~$M_1'=M'$ and~$\zeta'=\mathrm{id}$, and
have that a patching disc associated with~$M_1\xmapsto{\eta\circ\zeta^{-1}}M_1'$
is obtained from~$\Pi$ by a small perturbation of rectangles
having a side on~$y_0$.

In this case the numbers~$n_1$, $n_3$ are preserved, and~$n_2$ drops by~$1$, since~$c$ transforms
to a boundary circuit that does not hit~$\mu'$.

The case when~$\mu''$ is a $\diagup$-mirror is similar.
\end{proof}

So, we assume in the sequel that the mirrors~$\mu'$ and~$\mu''$ of~$M$ are of type~`$\diagdown$'.
Since~$\tb_+(c)=-1$, these are the only $\diagdown$-mirrors hit by~$c$.
We also assume that there are no $\diagup$-mirrors of~$M(\Pi)$ in~$\bigcup_{j=0}^kI_j$.
This implies that~$I_0$ and~$I_k$ contain no mirror of~$M(\Pi)$,
and each~$I_j$ with~$j\in\{1,\ldots,k-1\}$ contains exactly one $\diagdown$-mirror,
which we denote by~$\mu_j$. We also put~$\mu_0=\mu'$ and~$\mu_k=\mu''$.

For every~$j\in\{1,\ldots,k\}$ we have a unique rectangle in~$\Pi$ with a vertex at~$\nu_j$.
Denote this rectangle by~$r_j$. It follows from the assumptions we have just made that~$\mu_{j-1}$
and~$\mu_j$ are also vertices of~$r_j$. Denote the remaining, the fourth vertex of~$r_j$ by~$\lambda_j$.
It is opposite to~$\nu_j$ in~$r_j$, and it is a $\diagup$-mirror of~$M(\Pi)$. It may or may not be a mirror of~$M$,
and if it is, then its type in~$M$ is also~`$\diagup$'.

Let~$(\delta_+,\delta_-)$ be a canonic dividing configuration of~$\widehat\Pi$, and
let~$d=\widehat\Pi\setminus\widehat c$. By construction, $d$ is an open disc.
Since~$\tb_+(c)=-1$, the intersection~$d\cap\delta_+$ consists of a single open arc
that approaches the midpoints of~$\widehat\mu'$ and~$\widehat\mu''$ at the ends.
This implies that there are no rectangles in~$\Pi$ except~$r_1,\ldots,r_k$,
as otherwise we would have another connected component of~$\delta_+$.

\begin{lemm}\label{bypass-simplification-lem-2}
If~$k>2$ and, for some~$j\in\{1,\ldots,k-1\}$, we have~$\lambda_j=\lambda_{j+1}{}\not\in E_{M'}$, then~$\Pi$
can be simplified.
\end{lemm}

\begin{proof}
Suppose that~$\lambda_j=\lambda_{j+1}{}\not\in E_{M'}$, $1\leqslant j\leqslant k-1$, $k>2$. Without loss
of generality we may assume that the rectangles~$r_j$ and~$r_{j+1}$ have the form~$[\theta_1;\theta_2]\times[\varphi_1;\varphi_2]$
and~$[\theta_2;\theta_3]\times[\varphi_2;\varphi_1]$ for some~$\theta_1,\theta_2,\theta_3,\varphi_1,\varphi_2\in\mathbb S^1$.
Indeed, the other cases are obtained from this one by applying the symmetries~$r_\diagup$, $r_\diagdown$,
and~$r_\diagup\circ r_\diagdown$ (see Subsection~\ref{conventions-subsec} for notation).
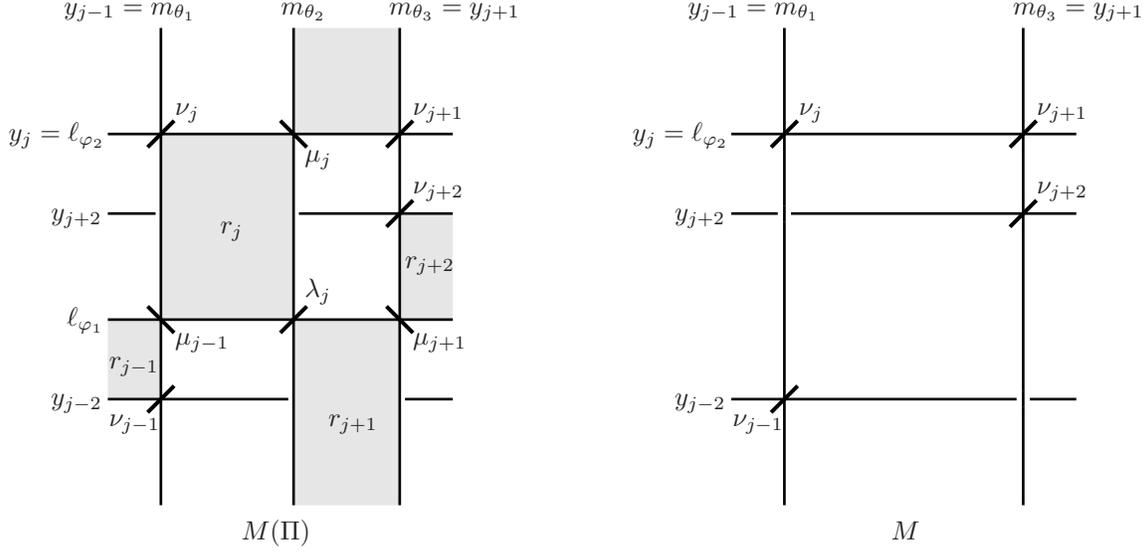
\begin{figure}[ht]
\includegraphics{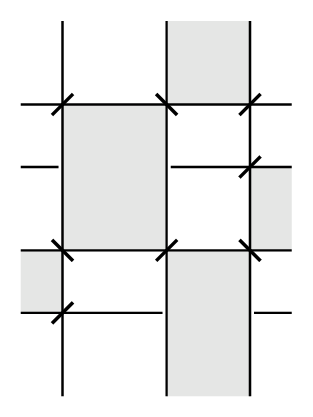}\put(-157,196){$y_{j-1}=m_{\theta_1}$}\put(-75,196){$m_{\theta_2}$}\put(-34,196){$m_{\theta_3}=y_{j+1}$}%
\put(-162,48){$y_{j-2}$}\put(-156,78){$\ell_{\varphi_1}$}\put(-178,148){$y_j=\ell_{\varphi_2}$}\put(-162,118){$y_{j+2}$}%
\put(-140,40){$\nu_{j-1}$}\put(-115,71){$\mu_{j-1}$}\put(-66,89){$\lambda_j$}\put(-66,140){$\mu_j$}\put(-115,158){$\nu_j$}%
\put(-25,158){$\nu_{j+1}$}\put(-25,128){$\nu_{j+2}$}\put(-25,71){$\mu_{j+1}$}\put(-98,113){$r_j$}\put(-57,40){$r_{j+1}$}%
\put(-140,62){$r_{j-1}$}\put(-28,100){$r_{j+2}$}
\put(-90,-3){$M(\Pi)$}
\hskip3cm
\includegraphics{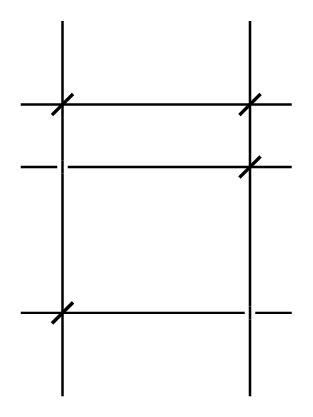}\put(-157,196){$y_{j-1}=m_{\theta_1}$}\put(-34,196){$m_{\theta_3}=y_{j+1}$}%
\put(-162,48){$y_{j-2}$}\put(-178,148){$y_j=\ell_{\varphi_2}$}\put(-162,118){$y_{j+2}$}%
\put(-140,40){$\nu_{j-1}$}\put(-115,158){$\nu_j$}%
\put(-25,158){$\nu_{j+1}$}\put(-25,128){$\nu_{j+2}$}
\put(-80,-3){$M$}
\caption{The mutual position of~$r_j$ and~$r_{j+1}$ in the case~$\lambda_j=\lambda_{j+1}{}\notin E_{M'}$}\label{bypass-wrinkle-fig}
\end{figure}

Our constructions are also invariant under the substitution~$\nu_j\mapsto\nu_{k+1-j}$, $\lambda_j\mapsto\lambda_{k+1-j}$,
$\mu_j\mapsto\mu_{k-j}$, $y_j\mapsto y_{k-j}$. So, we may also safely assume that~$j\leqslant k-2$.
Shown in Figure~\ref{bypass-wrinkle-fig} is the case when we also have~$j>1$, though we don't assume this in the sequel.
If~$j=1$, one should remove~$y_{j-2}$, $\nu_{j-1}$, and~$r_{j-1}$ from the pictures in Figure~\ref{bypass-wrinkle-fig},
and add~$\mu_0$ and~$\ell_{\varphi_1}$ to the right picture.

By construction, there are no mirrors of~$M$ inside~$r_j$ and~$r_{j+1}$ and in~$[\nu_j;\nu_{j+1}]\cup[\mu_{j-1};\mu_{j+1}]$.
Pick an~$\varepsilon>0$ such that there are no occupied levels of~$M$ in
the domains~$\Omega_1=(\theta_1-\varepsilon;\theta_1)\times\mathbb S^1$ and~$\Omega_2=(\theta_3;\theta_3+\varepsilon)\times\mathbb S^1$.
By means of jump moves applied to the diagrams~$M$ and~$M'$
we can shift all their occupied meridians intersecting~$r_j$ into~$\Omega_2$,
and all occupied meridians intersecting~$r_{j+1}$ into~$\Omega_1$. Let~$M\xmapsto{\zeta_1}M_2$
and~$M'\xmapsto{\zeta'}M_1'$ be the obtained transformations.
The transformation~$M_2\xmapsto{\zeta'\circ\eta\circ\zeta_1^{-1}}M_1'$ is still a
generalized type~II bypass removal, and there is an associated patching disc
having the same complexity as~$\Pi$ has. This patching disc is produced from~$\Pi$
by two or less exchange moves.

Now observe that~$m_{\theta_2}$ is not an occupied level
of~$M$, and hence not an occupied
level of~$M_2$. Therefore, $M_2$ has no mirrors in~$(\theta_1;\theta_3)\times\mathbb S^1$,
and there is no obstruction to merging the meridians~$m_{\theta_1}=y_{j-1}$ and~$m_{\theta_3}=y_{j+1}$,
and simultaneously applying the respective wrinkle reduction move to the patching disc. Since we assumed~$j\leqslant k-2$,
the meridian~$y_{j+1}$ does not belong to~$L_{M'}$ (which is not the case for the meridian~$y_{j-1}$ when~$j=1$).
We define~$M_2\xmapsto{\zeta_2}M_1$ to be the composition of a type~I merge move that
merges~$y_{j-1}$ and~$y_{j+1}$ so that the successor of both meridians is~$y_{j-1}$,
and a type~I elimination move that removes the common successor of~$\nu_j$ and~$\nu_{j+1}$.
In other words, the transformation~$M_2\xmapsto{\zeta_2}M_1$ replaces~$\nu_j$, $\nu_{j+1}$, and~$\nu_{j+2}$
by a single $\diagup$-mirror at~$y_{j-1}\cap y_{j+2}$, and removes~$y_j$ and~$y_{j+1}$.

Now we put~$\zeta=\zeta_2\circ\zeta_1$ and note that the transformation~$M_1\xmapsto{\eta'}M_1'$,
where~$\eta'={}\zeta'\circ\eta\circ\zeta^{-1}$, is a generalized bypass removal with the corresponding bypass being
obtained from~$\beta$ by replacing the subsequence~$(\nu_j,y_j,\nu_{j+1},y_{j+1},\nu_{j+2})$
with a single entry~$\nu_j'=y_{j-1}\cap y_{j+2}$.

The new patching disc will have two rectangles less than~$\Pi$ has, whereas~$n_2$ and~$n_3$
in the definition of the complexity are unchanged. So, a simplification of the patching disc occurs.\end{proof}

\begin{lemm}\label{simplification-lem-3}
If either~$\lambda_j\ne\lambda_{j+1}$ or~$\lambda_j=\lambda_{j+1}\in E_{M'}$ holds
for all~$j\in\{1,\ldots,k-1\}$, then~$\lambda_j\in E_{M'}$ for all~$j\in\{1,\ldots,k\}$.
\end{lemm}

\begin{proof}
If~$\lambda_j\notin E_{M'}$, then~$\lambda_j$ coincides with some~$\lambda_i$, $i\ne j$. This means that
a connected component of~$\delta_-$ is an arc~$\alpha$ connecting the midpoints of~$\widehat\nu_i$ and~$\widehat\nu_j$.
The arcs~$\alpha$ and~$\interior(\delta_+)$ have two intersection points, hence there is a subdisc~$d'\subset d$
bounded by a subarc of~$\alpha$ and a subarc of~$\delta_+$. Such a disc must contain a bigon of~$\delta_+$ and~$\delta_-$,
which implies that for some~$l\leqslant k-1$ we have~$\lambda_l=\lambda_{l+1}{}\notin E_{M'}$. The claim follows.\end{proof}

\begin{lemm}
If~$\lambda_1\in E_{M'}$ and~$k>1$, then~$\Pi$ can be simplified.
\end{lemm}

\begin{proof}
If~$\lambda_1\in E_{M'}$ and~$k>1$, the addition of~$\mu_1$ to~$M$
is a type~I elementary bypass addition. Indeed, the three vertices~$\nu_1$, $\mu_0$, $\lambda_1$
of~$r_1$ belong to~$E_M$, whereas the forth, $\mu_1$ does not. We let~$M\xmapsto\zeta M_1$
be this addition, and let~$M'\xmapsto{\zeta'}M_1'$ be the addition of a bridge that
adds~$\mu_1$, $\nu_1$, and~$y_1$ to~$M'$. It can be neatly decomposed into type~II
elementary moves by Lemma~\ref{neutral-move-decomposition-lem}.

One can see that~$M_1\xmapsto{\zeta'\circ\eta\circ\zeta^{-1}}M_1'$
is a generalized bypass removal, and the corresponding bypass is~$(y_1,\nu_2,y_2,\ldots,\nu_k,y_k)$.
The collection~$\{r_j\}_{j=2}^k$ can be taken for a patching disc, which is simpler than~$\Pi$.
\end{proof}

Thus, the only cases in which we have not established a possibility to simplify~$\Pi$
are the following two:
\begin{enumerate}
\item
$k=1$, $\lambda_1\in E_{M'}$;
\item
$k=2$, $\lambda_1=\lambda_2\notin E_{M'}$.
\end{enumerate}

In the first case, the move~$M\xmapsto\eta M'$ is a type~II elementary bypass removal, hence
the assertion of the proposition holds trivially.

To treat the second case,
we revisit the proof of Lemma~\ref{bypass-simplification-lem-2}, where we put~$j=1$.
We use the same notation and make the same assumption about the positions of~$r_1$ and~$r_2$.

The elements~$\nu_{j-1}$, $\nu_{j+2}$, $y_{j-2}$, $y_{j+2}$, $r_{j-1}$, $r_{j+2}$ now do not exist,
whereas~$\mu_0=\mu'$, $\mu_2=\mu''$ are~$\diagdown$-mirrors of~$M$ and~$M'$,
and~$\ell_{\varphi_1}$ is an occupied level of both. By using the same family of jump
moves as in the proof of Lemma~\ref{bypass-simplification-lem-2} we can reduce the general
case to the one when no meridian in~$(\theta_1;\theta_2)\times\mathbb S^1$
is an occupied level of~$M$.

Thus, it remains to consider the situation shown in Figure~\ref{special-bypass-fig} on the left
(the gray boxes again stand for families of mirrors of the diagram, not rectangles
of~$\Pi$).
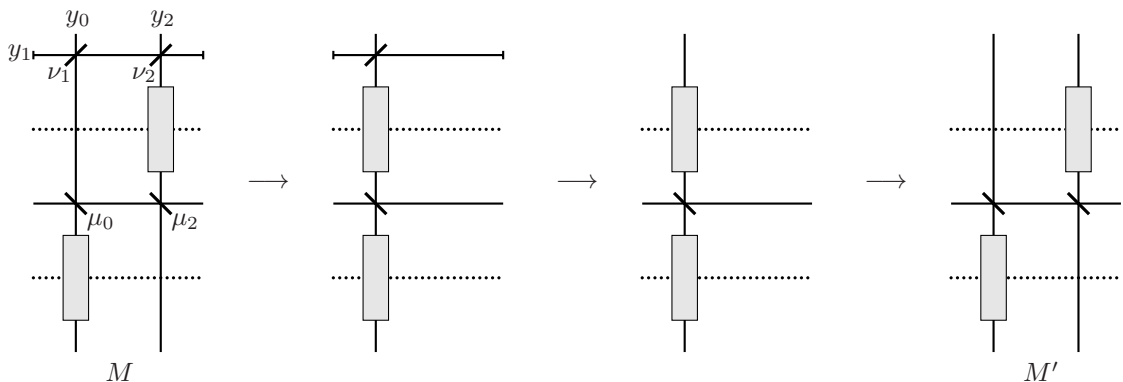
\begin{figure}[ht]
\includegraphics[scale=.8]{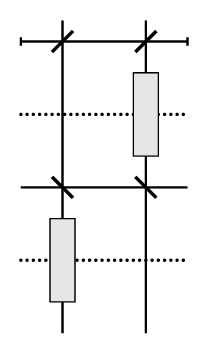}\put(-60,133){$y_0$}\put(-28,133){$y_2$}\put(-82,120){$y_1$}%
\put(-52,56){$\mu_0$}\put(-20,56){$\mu_2$}\put(-67,111.5){$\nu_1$}\put(-35,111.5){$\nu_2$}%
\put(-45,-3){$M$}
\hskip.3cm\raisebox{70pt}{$\longrightarrow$}\hskip.3cm
\includegraphics[scale=.8]{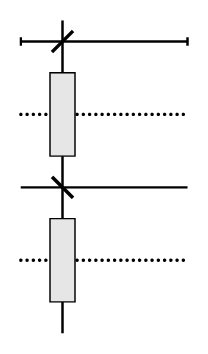}
\hskip.3cm\raisebox{70pt}{$\longrightarrow$}\hskip.3cm
\includegraphics[scale=.8]{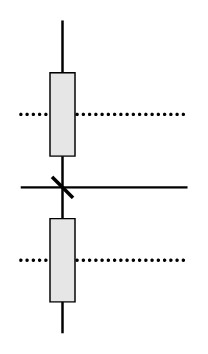}
\hskip.3cm\raisebox{70pt}{$\longrightarrow$}\hskip.3cm
\includegraphics[scale=.8]{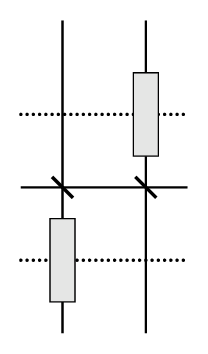}\put(-45,-3){$M'$}
\caption{Decomposition of a generalized bypass removal into type~I moves followed by type~II moves
in the case~$k=2$, $\lambda_1=\lambda_2{}\notin E_{M'}$}\label{special-bypass-fig}
\end{figure}

The diagram~$M'$ can be obtained from~$M$ in three steps illustrated in Figure~\ref{special-bypass-fig}:
\begin{enumerate}
\item
a double merge move that merges~$y_0$ with~$y_2$;
\item
a type~I elimination move that removes~$y_1$ with the remaining mirror on it;
\item
a type~II split move that recovers~$y_0$ and~$y_1$ with all mirrors on them except for~$\nu_1$ and~$\nu_2$.
\end{enumerate}
The double merge move should then be neatly decomposed into type~I elementary moves, and the type~II split move
should be neatly decomposed into type~II elementary moves, which is possible due to
Lemmas~\ref{neutral-move-decomposition-lem} and~\ref{split-move-decomposition-lem}.

These decompositions can be chosen to preserve either all boundary circuits visiting~$y_0$ and not hitting~$\nu_1$,
or all boundary circuits visiting~$y_2$ and not hitting~$\nu_2$. If no mirror on either~$y_0$ or~$y_2$
is hit by a boundary circuit in~$C$, we are done.

It is possible, however, that both~$y_0$ and~$y_2$ contain mirrors hit by boundary circuits in~$C$, in which
case we would get stuck without the $-$-flexibility of~$M'$. However, $-$-flexibility allows us to avoid the difficulty.

Suppose that there is a single-headed type~I splitting route~$\omega$ in~$M'$ such that~$\omega$ does not separate~$C$
and has the following property: no connected component of~$\wideparen y_2\setminus\wideparen\omega$
has a non-empty intersection with both~$\wideparen \nu_2$ and~$\bigcup_{a\in C}\wideparen a$.
By a small perturbation of the snip point we can ensure
that~$\omega$ is also suitable as a type~I splitting route in~$M$.

Let~$M\xmapsto\zeta M_1$ be a generalized wrinkle creation move associated with~$\omega$, and preserving all boundary circuits in~$C$. The mirrors~$\nu_1$, $\nu_2$, and the occupied level~$y_1$
have unique successors in~$M_1$, which we denote by~$\nu_1'$, $\nu_2'$, and~$y_1'$, respectively.
Let~$M_1'$ be obtained from~$M_1$ by removing these three elements (the occupied level~$y_1'$ contains
no other mirrors). One can see that~$M_1\mapsto M_1'$ is a generalized type~II bypass removal,
and~$M'\mapsto M_1'$ is a generalized wrinkle creation move associated with~$\omega$.
Moreover, if~$\eta_1:\widehat M_1\rightarrow\widehat M_1'$ and~$\zeta':\widehat M'\rightarrow\widehat M_1'$
are the associated morphisms, then~$\eta_1\circ\zeta=\zeta'\circ\eta$.

Due to the properties of~$\omega$ the successor of~$y_2$ in~$M_1$ that contains~$\nu_2'$
has an empty intersection with~$\bigcup_{a\in C}a$.
This means that the generalized bypass removal~$M_1\xmapsto{\eta_1}M_1'$
admits a decomposition illustrated in Figure~\ref{special-bypass-fig} without
disturbing any boundary circuits in~$C$. Thus, it remains to show how to find~$\omega$
with required properties.

Denote, as usually, by~$c\#c'$ the unique boundary circuit in~$\partial M'\setminus\partial M$.
If~$c\#c'$ hits a $\diagup$-mirror, the sought-for~$\omega$ can be found by starting from such a mirror and following~$c\#c'$.
More precisely, let~$\kappa_1,\kappa_2,\kappa_3,\ldots,\kappa_l$ be all mirrors hit by~$c\#c'$ numbered
in the order they follow on~$c\#c'$, so that~$\kappa_1=\mu_0$ and~$\kappa_l=\mu_2$. Let~$j$ be the maximal
index such that~$\kappa_j$ is a $\diagup$-mirror. We can take for~$\omega$ the sequence~$(\kappa_j,\kappa_{j+1},\ldots,\kappa_l,p)$,
where~$p$ is picked on the segment of~$c\#c'$ connecting~$\mu_2$ and~$\mu_0$.

If~$c\#c'$ does not hit any $\diagup$-mirror, we pick any single-headed splitting route in which the last mirror
appears on~$c\#c'$ (this exists by Lemma~\ref{flexibility-lem}). Then we prolong this splitting route along~$c\#c'$
until the respective splitting path cuts~$\wideparen y_2$ as requested.

This completes the proof of Proposition~\ref{2ndcommutation-for-bypass-prop}.\end{proof}

\subsection{Proof of the commutation theorems}
Theorem~\ref{warm-up-thm} is a particular case of Theorem~\ref{commutation-2-thm}, so it suffices to prove the latter.

A little more preparation is in order.

\begin{defi}
Let~$M$ be an enhanced mirror diagram, and let~$C$ be a collection of essential boundary
circuits of~$M$. The diagram~$M$ is said to be \emph{very $+$-flexible} (respectively,
\emph{very $-$-flexible}) \emph{relative to~$C$} if, for any~$c\in\partial M\setminus C$ we have~$\tb_+(c)<0$
(respectively, $\tb_-(c)<0$). If~$M$ is both very $+$-flexible and very $-$-flexible relative to~$C$,
we say that it is \emph{very flexible relative to~$C$}.
\end{defi}

Clearly `very ($\pm$)-flexible' implies `($\pm$)-flexible'.

We call a connected component~$M_0$ of an enhanced mirror diagram~$M$ \emph{spherical}
if the corresponding connected component of~$\wideparen M_0$ is a sphere with holes,
and all elements of~$\partial M_0$ are inessential boundary circuits of~$M$.
A connected component~$M_0$ of~$M$ is called \emph{a disc component} if
the corresponding connected component of~$\wideparen M_0$ is a sphere with holes,
and~$\partial M_0$ contains exactly one essential boundary circuit of~$M$.

\begin{lemm}\label{almost-c-neat=>flex-lem}
Let $M_0\xmapsto{\eta_1}M_1\xmapsto{\eta_2}\ldots\xmapsto{\eta_n}M_n$ be an almost $C_0$-neat
decomposition of a transformation~$M_0\xmapsto\eta M_n$ of enhanced mirror diagrams
into a sequence of moves of any kind introduced earlier in this paper,
and let~$C_0\mapsto C_1\mapsto\ldots\mapsto C_n$ be the induced
transformations of collections of boundary circuits.

Assume that~$M_0$ (equivalently: $M_n$) has no spherical components.
Assume also that~$M_0$ is $+$-flexible relative to~$C_0$, and~$M_n$ is very $+$-flexible relative to~$C_n$.
Then all diagrams~$M_1,M_2,\ldots,M_{n-1}$ are $+$-flexible relative to the respective~$C_i$'s, too.
\end{lemm}

\begin{proof}
Without loss of generality we may assume that all moves~$M_{i-1}\xmapsto{\eta_i}M_i$
are elementary or jump moves. Indeed, all the moves introduced in this paper
admit a neat decomposition into elementary ones, hence, if we replace all
moves in the original sequence, except jump moves, by their respective neat decompositions
into elementary moves, we will still have an almost $C_0$-neat decomposition of~$M_0\xmapsto\eta M_n$.

We proceed by induction in~$n$. The induction base, $n=1$, is trivial.
To make the induction step it suffices to show that~$M_1$ is $+$-flexible relative to~$C_1$.

All $+$-safe-to-bring-forward moves (see Definition~\ref{safe-def})
preserve $+$-flexibility by design. So, if~$M_0\xmapsto{\eta_1}M_1$
is a type~I elementary move, jump move, type~II extension, slide, or elementary bypass removal
move, $M_1$ is flexible and we are done.

Elementary bypass additions (of either type) are not always
$+$-safe-to-bring-forward, but they preserve $+$-flexibility
due to the fact that they replace a boundary circuit by two mutually adjacent
boundary circuits one of which is $+$-negligible. So, the only `dangerous'
moves in the present context are type~II elimination moves.

Suppose that~$M_0\xmapsto{\eta_1}M_1$ is a type~II elimination move.
Let~$c_0\in\partial M_0$ be the modified boundary circuit of~$M_0$,
and let~$c_1$ be the corresponding circuit of~$M_1$.

If~$c_0$ is inessential, then~$c_1$ is inessential, too. Since no component of~$M_1$
is spherical, this implies that~$\tb_+(c_1)<0$. Indeed, otherwise one would
be able produce an overtwisted disc from a patching disc for~$\widehat c_1$,
which is impossible according to Bennequin's result~\cite{ben}.

If~$c_0$ is essential, then the inequality~$\tb_+(c_1)<0$ follows
from the assumption that the given decomposition is almost $C_0$-neat.
Indeed, let~$c_0$ be transformed to~$c_n$ in~$\partial M_n$.
We must have~$c_0\notin C_0$, $c_n\notin C_n$ as~$c_0$
is modified by~$M_0\xmapsto{\eta_1}M_1$.

We have~$\tb_+(c_0)=\tb_+(c_1)-1<0$ since~$M_0\xmapsto{\eta_1}M_1$
is a type~II elimination move, and~$\tb_+(c_n)<0$ since~$M_n$
is assumed to be very $+$-flexible relative to~$C_n$.
By Definition~\ref{almost-neat-def} we must have~$\tb_+(c_1)<0$.

Therefore, $M_1$ is $+$-flexible relative to~$C_1$.\end{proof}

Now we proceed with the proof of Theorem~\ref{commutation-2-thm}.

Let
\begin{equation}\label{sequence-eq}
M=M_0\xmapsto{\eta_1}M_1\xmapsto{\eta_2}\ldots
\xmapsto{\eta_n}M_n=M'
\end{equation}
be a $C$-delicate sequence of elementary moves and jump moves.
Let~$C=C_0\mapsto C_1\mapsto\ldots\mapsto C_n=C'$ be the induced
sequence of transformations of the collection of selected boundary circuits.

Suppose, for the time being, that~$M_0$ has no spherical components.

For every~$j=0,\ldots,n$ we can find transformations~$M_j\xmapsto{\zeta_j'}M_j'$
and~$M_j\xmapsto{\zeta_j''}M_j''$ preserving all boundary circuits in~$C_j$ such
that the following holds (see the scheme in Figure~\ref{saw-fig}):
\begin{figure}[ht]
\includegraphics{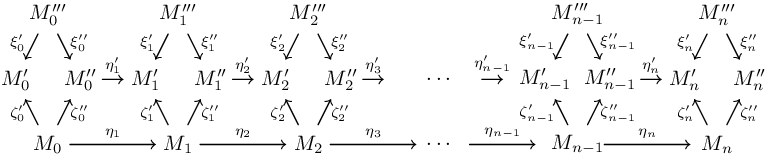}
\caption{Adding flexibility to intermediate diagrams}\label{saw-fig}
\end{figure}
\begin{enumerate}
\item
$M_j\xmapsto{\zeta_j'}M_j'$
and~$M_j\xmapsto{\zeta_j''}M_j''$ are compositions of extension moves for any~$j=0,\dots,n$;
\item
$M_0\xmapsto{\zeta_0'}M_0'$ is a composition of type~I extension moves;
\item
$M_n\xmapsto{\zeta_n''}M_n''$ is a composition of type~II extension moves;
\item
for any~$j=1,\ldots,n-1$, both diagrams~$M_j'$ and~$M_j''$ are very flexible relative to~$C_j$;
\item
the diagrams~$M_0'$ and~$M_0''$ are very $-$-flexible and very flexible, respectively, relative to~$C_0$;
\item
the diagrams~$M_n'$ and~$M_n''$ are very flexible and very $+$-flexible, respectively, relative to~$C_n$;
\item
if~$j<n$ the transformation~$M_j''\xmapsto{\eta_j'}M_{j+1}'$ with~$\eta_j'$ such that~$\eta_j'\circ\zeta_j''=\zeta_{j+1}'\circ\eta_j$
is an elementary move or a jump move;
\item
for any~$j=0,\ldots,n$ we have~$E_{M_j'}\cap E_{M_j''}=E_{M_j}$.
\end{enumerate}

For each~$j$, let~$M_j'''$ be the union~$M_j'\cup M_j''$ (a boundary circuit of~$M_j'''$
is declared essential if and only if it has a non-trivial interval in common with an essential
boundary circuit of~$M_j$). Both diagrams~$M_j'$ and~$M_j''$ can be obtained
from~$M_j'''$ by a composition of elimination moves. Let~$M_j'''\xmapsto{\xi_j'}M_j'$
and~$M_j'''\xmapsto{\xi_j''}M_j''$ be the obtained transformations.
We clearly have~$\xi_j''\circ{\xi_j'}^{-1}=\zeta_j''\circ{\zeta_j'}^{-1}$.

Thus, we have the following decomposition of the transformation~$M\xmapsto{\eta_n\circ\ldots\circ\eta_1}M'$:
\begin{equation}\label{saw-eq}
M_0\xmapsto{\zeta_0'}M_0'\xmapsto{{\xi_0'}^{-1}}M_0'''\xmapsto{\xi_0''}M_0''\xmapsto{\eta_1'}
M_1'\xmapsto{{\xi_1'}^{-1}}M_1'''\xmapsto{\xi_1''}M_1''\xmapsto{\eta_2'}\ldots\xmapsto{\eta_n'}
M_n'\xmapsto{{\xi_n'}^{-1}}M_n'''\xmapsto{\xi_n''}M_n''\xmapsto{{\zeta_n''}^{-1}}M_n,
\end{equation}
which starts from the composition~$M_0\xmapsto{\zeta_0'}M_0'$ of type~I extension moves
and ends up with the composition~$M_n''\xmapsto{{\zeta_n''}^{-1}}M_n$ of type~II elimination moves.
So, it suffices to prove the assertion of the theorem for the transition from~$M_0'$ to~$M_n''$.

All the diagrams in the truncated sequence
\begin{equation}\label{truncated-eq}
M_0'\xmapsto{{\xi_0'}^{-1}}M_0'''\xmapsto{\xi_0''}M_0''\xmapsto{\eta_1'}
M_1'\xmapsto{{\xi_1'}^{-1}}M_1'''\xmapsto{\xi_1''}M_1''\xmapsto{\eta_2'}\ldots\xmapsto{\eta_n'}
M_n'\xmapsto{{\xi_n'}^{-1}}M_n'''\xmapsto{\xi_n''}M_n''
\end{equation}
except for the first and the last ones are very flexible.
After decomposing each transformation~$M_j'\xmapsto{{{\xi_j'}^{-1}}}M_j'''$ in~\eqref{truncated-eq}
into a sequence of extension moves, and each transformation~$M_j'''\xmapsto{\xi_j''}M_j''$
into a sequence of elimination moves, we get an almost $C$-neat decomposition
of the transformation~$M_0'\xmapsto\chi M_n''$, where~$\chi=\xi_n''\circ{\xi_n'}^{-1}\circ\eta_n'\circ\ldots
\circ\eta_2'\circ\xi_1''\circ{\xi_1'}^{-1}\circ\eta_1'\circ\xi_0''\circ{\xi_0'}^{-1}$.
We also recall that the diagram~$M_0'$ is $+$-flexible and very $-$-flexible relative to~$C_0$,
whereas~$M_n''$ is $-$-flexible and very $+$-flexible relative to~$C_n$.

Thus, without loss of generality, we may assume the following from the beginning:
\begin{enumerate}
\item
the diagram~$M$ is $+$-flexible and very $-$-flexible relative to~$C_0=C$;
\item
the diagram~$M'$ is $-$-flexible and very $+$-flexible relative to~$C_n=C'$;
\item
the decomposition~\eqref{sequence-eq} is almost $C$-neat.
\end{enumerate}

In what follows we will modify this decomposition so that these conditions will always hold.
Due to Lemma~\ref{almost-c-neat=>flex-lem}, in which, clearly, the roles of~$M_0$ and~$M_n$
can be exchanged and $+$-flexibility can be replaced with $-$-flexibility,
guarantees that all diagrams which will arise in our decompositions will be flexible
relative to the corresponding collection of selected boundary circuits.

Now we use Lemmas~\ref{extension-via-flexibility-lem}, \ref{slide-bypass-neat-decomp-lem}, \ref{neutral-move-decomposition-lem},
and~\ref{split-into-bypas-decomp-lem} to modify the sequence~\eqref{sequence-eq}
so that it includes only jump moves, type~I elementary moves, type~II merge moves, and type~II elementary bypass removals.
Type~II merge moves and type~II elementary bypass removals are particular cases
of generalized type~II merge moves and generalized type~II bypass removals,
which are collectively called generalized type~II moves.

We are ready
to follow the strategy outlined in Subsection~\ref{strategy-subsec}.
Namely, we deal with sequences~$s$ of moves that have the following properties:
\begin{enumerate}
\item
$s$ transforms~$M$ to~$M'$ and induces the same morphism from~$\widehat M$ to~$\widehat M'$ as
the original sequence does;
\item
$s$ consists of two successive parts, $s_1$ and~$s_2$, such that~$s_1$ includes only jump moves,
type~I elementary moves, and generalized type~II moves, whereas~$s_2$ includes only type~II elementary moves;
\item
once all generalized type~II moves in~$s$ are neatly decomposed into elementary moves,
the obtained sequence will be $C$-delicate.
\end{enumerate}

\emph{The complexity} of such a sequence is defined as the pair~$(N_1,N_2)$ in which~$N_1$ is the number
of generalized type~II moves in~$s_1$, and~$N_2$ is the number of elementary type~I moves and jump moves in~$s_1$
that occur after all generalized type~II moves. We order such pairs lexicographically, that is,
$(N_1,N_2)<(N_1',N_2')$ if either~$N_1<N_1'$, or~$N_1=N_1'$ and~$N_2<N_2'$.

Suppose that~$N_2>0$ and~$N_1>0$. Let~$M_{j-1}\xmapsto{\eta_j}M_j$ be the last generalized type~II move
in~$s_1$. We apply:
\begin{itemize}
\item
Lemmas~\ref{split-commute-with-jump-lem} and~\ref{special-split-commute-with-jumpe-lem} and Proposition~\ref{similarity-prop}
if~$M_{j-1}\xmapsto{\eta_j}M_j$ is a generalized type~II merge move
and~$M_j\xmapsto{\eta_{j+1}}M_{j+1}$ is a jump move,
\item
Proposition~\ref{1st-comm-prop-of-merge-prop} if~$M_{j-1}\xmapsto{\eta_j}M_j$ is a generalized type~II merge move
and~$M_j\xmapsto{\eta_{j+1}}M_{j+1}$ is a type~I elementary move,
\item
Proposition~\ref{1st-comm-prop-for-bypass-2-prop} if~$M_{j-1}\xmapsto{\eta_j}M_j$ is a generalized type~II bypass removal
and~$M_j\xmapsto{\eta_{j+1}}M_{j+1}$ is a jump move,
\item
Proposition~\ref{1st-comm-prop-for-bypass-1-prop} if~$M_{j-1}\xmapsto{\eta_j}M_j$ is a generalized type~II bypass removal
and~$M_j\xmapsto{\eta_{j+1}}M_{j+1}$ is a type~I elementary move
\end{itemize}
to reduce~$N_2$ while keeping~$N_1$ fixed.

Suppose that~$N_2=0$ and~$N_1>0$. Let again~$M_{j-1}\xmapsto{\eta_j}M_j$ be the last generalized type~II move
in~$s_1$. We apply:
\begin{itemize}
\item
Proposition~\ref{2ndcommutation-for-merge-non-special-prop} if~$M_{j-1}\xmapsto{\eta_j}M_j$ is a non-special
generalized type~II merge move,
\item
Proposition~\ref{2ndcommutation-for-merge-special-prop} if~$M_{j-1}\xmapsto{\eta_j}M_j$ is a special
generalized type~II merge move,
\item
Proposition~\ref{2ndcommutation-for-bypass-prop} if~$M_{j-1}\xmapsto{\eta_j}M_j$ is a generalized type~II bypass removal
\end{itemize}
to reduce~$N_1$.

By following these rules we eventually get~$N_1=0$, in which case we are done.

We are left to treat the case when~$M$ has spherical components. This case does not seem
to be of any importance, so we will skip some details.

Let~$B_1,\ldots,B_k$ be all spherical
components of~$M$, and let~$B_1',\ldots,B_k'$ be the respective components of~$M'$.
Clearly, the morphism~$\eta=\eta_n\circ\ldots\circ\eta_1:\widehat M\rightarrow\widehat M'$
can be presented by a triple~$(F,F',h)$
such that, for any~$i=1,\ldots,k$, the connected component of~$F$ (respectively, of~$F'$)
containing~$\Gamma_{\widehat B_i}$ (respectively, $\Gamma_{\widehat B_i'}$)
is a two-disc, which we denote by~$D_i$ (respectively, by~$D_i'$).

Let~$M_*$ (respectively, $M_*'$) be the enhanced mirror diagram
obtained from~$M$ (respectively, from~$M'$) by declaring all the boundary
circuits corresponding to the boundaries of~$D_i$'s (respectively, of~$D_i'$'s)
essential, and let~$\eta_*$ be the morphism~$\widehat M_*\rightarrow\widehat M_*'$
represented by~$(F,F',h)$. The spherical connected components~$B_1,\ldots,B_k$
of~$M$ will turn into disc components of~$M_*$.

One can find a $C$-delicate decomposition of the transformation~$M_*\xmapsto{\eta_*}M_*'$
by modifying the original decomposition~\eqref{sequence-eq}.
This can be easily derived from Theorem~\ref{relative-stable-equivalence-th} and the obvious fact
that the addition of a disc component to an enhanced mirror diagram
is unique up to stable equivalence (relative to
the family of all boundary circuits not belonging to the added disc component).

The diagram~$M_*$ does not have spherical components, hence we can modify
the decomposition of~$M_*\xmapsto{\eta_*}M_*'$ so that all type~I elementary moves
occur before all type~II elementary moves. Finally, we convert the artificially created
disc components of~$M_*$ back to spherical components by declaring all their boundary circuits
inessential, and do so with all the respective components of the diagrams through
which we transform~$M_*$ into~$M_*'$, thus obtaining the sought-for decomposition
of the transformation~$M\xmapsto\eta M'$.

 This completes the proof of the commutation theorems.

\section*{Appendix A: Finding realizations}
\def\thesection{A}
\setcounter{lemm}{0}
\setcounter{equation}{0}
\setcounter{figure}{0}

Here we present a manual check that certain dividing codes have
only those realizations that are mentioned in the proof of Proposition~\ref{6-2-prop}.
In many cases of interest, such a manual check is infeasible, but an exhaustive
search of all realizations can be
implemented on a computer and takes reasonable time to run~\cite{dyn-script}.

\begin{lemm}\label{only-one-lemm}
Up to combinatorial equivalence, there exists only one realization of
an admissible dividing configuration having the dividing code
\begin{equation}\label{code6-2-1}\begin{matrix}\{(1,2,3,4),(5,6),(7,8,9,10),(11,12),(13,14),(15,16),(17,18)\},\\
\{(9,12,5,4,7,14,15,18,3,6,13,8,17,16,11,10,1,2)\}.\end{matrix}
\end{equation}
\end{lemm}

\begin{proof}
Figure~\ref{abcd1} shows again the realization that we already know from Figure~\ref{seifert}, with
vertical occupied levels labeled by uppercase letters and horizontal occupied levels by lowercase letters.
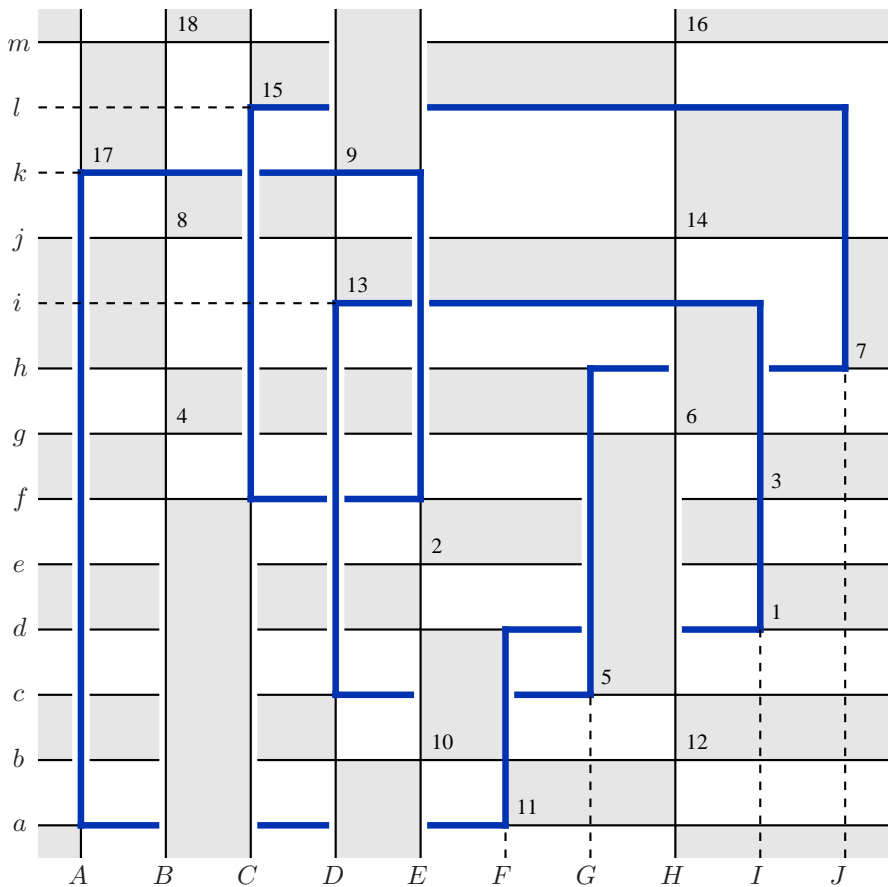
\begin{figure}[ht]
\includegraphics[scale=0.8]{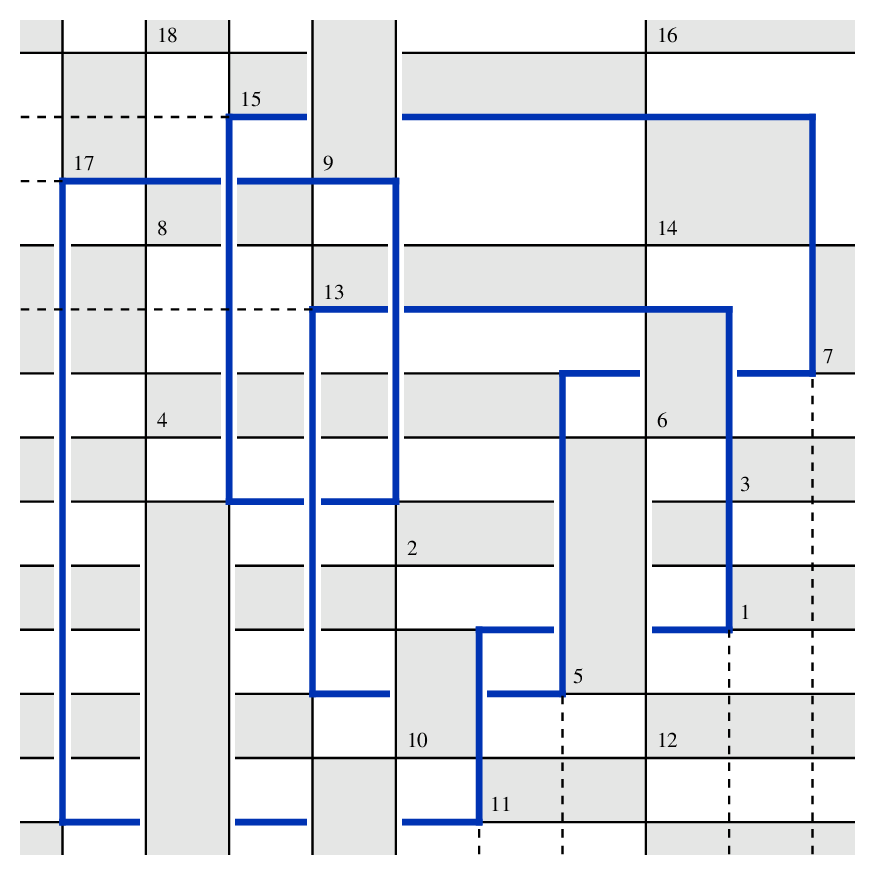}
\put(-318,-2){$A$}%
\put(-286,-2){$B$}%
\put(-254,-2){$C$}%
\put(-222,-2){$D$}%
\put(-190,-2){$E$}%
\put(-158,-2){$F$}%
\put(-126,-2){$G$}%
\put(-94,-2){$H$}%
\put(-60,-2){$I$}%
\put(-30,-2){$J$}%
\put(-339,18){$a$}%
\put(-339,42.6153846154){$b$}%
\put(-339,67.2307692308){$c$}%
\put(-339,91.8461538462){$d$}%
\put(-339,116.461538462){$e$}%
\put(-339,141.076923077){$f$}%
\put(-339,165.692307692){$g$}%
\put(-339,190.307692308){$h$}%
\put(-339,214.923076923){$i$}%
\put(-339,239.538461538){$j$}%
\put(-339,264.153846154){$k$}%
\put(-339,288.769230769){$l$}%
\put(-341,313.384615385){$m$}%
\caption{A realization of~\eqref{code6-2-1}}\label{abcd1}
\end{figure}

Any other realization of~\eqref{code6-2-1} must have the same number of vertical and horizontal
occupied levels as the one in Figure~\ref{abcd1}, and they must admit a labeling with the same
incidence between rectangles and occupied levels as in Figure~\ref{abcd1}: rectangle~$1$
must be $[I;E]\times[d;e]$, rectangle~$2$ must be $[E;I]\times[e;f]$, etc. To find all
realizations we have to search all possible ways to choose
cyclic orderings of horizontal and vertical labels so that all rectangles
are pairwise compatible.

To do so we introduce some notation. If $x,y,z,\ldots$ are labels of some occupied levels,
we denote by $\langle xyz\ldots\rangle$ the statement that these levels follow in  the
indicated cyclic order.

Some information about the cyclic ordering is known
from the beginning. For instance, the horizontal level~$a$ in Figure~\ref{abcd1}
contains three vertices of the diagram, $(A,a)$, $(F,a)$, $(H,a)$.
These three vertices must follow at the level~$a$
in this cyclic order for any realization of~$\eqref{code6-2-1}$, for otherwise the rectangles~$11$ and~$16$, which have
horizontal sides at the level~$a$, would not be compatible.
We abbreviate this implication as~$a\Rightarrow\langle AFH\rangle$. A similar
implication holds for each horizontal and vertical occupied level: the vertices at any level
must follow in the same cyclic order for all realizations.

By $\fbox{i}$, where $i$ is the number of a rectangle, we denote
the statement that there are no vertices inside this rectangle.
For instance, let this rectangle be $[X;Y]\times[x;y]$, and let there be a vertex at $(Z,z)$.
Assume that~$\langle XZY\rangle$ has already been established.
Then~$\langle xzy\rangle$ would imply that~$(Z,z)$ lies
inside the rectangle. Since this is forbidden, we conclude~$\langle xyz\rangle$.
We abbreviate this argument as $\fbox i\ \&\ \langle XZY\rangle\Rightarrow\langle xyz\rangle$.

With this notation at hand the proof is as follows:
\begin{align*}
&\text{argument}&\text{conclusion}\\
b&\Rightarrow\langle DEFH\rangle&\langle DEFH\rangle\\
i&\Rightarrow\langle DHI\rangle&\langle DEFHI\rangle\\
E&\Rightarrow\langle bdefk\rangle&\langle DEFHI\rangle&\ \&\ \langle bdefk\rangle\\
\fbox{12}\ \&\ \langle DHI\rangle&\Rightarrow\langle bcd\rangle&\langle DEFHI\rangle&\ \&\ \langle bcdefk\rangle\\
B&\Rightarrow\langle fghjkm\rangle&\langle DEFHI\rangle&\ \&\ \langle bcdefghjk\rangle\ \&\ \langle fkm\rangle\\
\fbox{5}\ \&\ \langle cdg\rangle&\Rightarrow\langle FGH\rangle&\langle DEFGHI\rangle&\ \&\ \langle bcdefghjk\rangle\ \&\ \langle fkm\rangle\\
\fbox{13}\ \&\ \langle DGH\rangle&\Rightarrow\langle hij\rangle&\langle DEFGHI\rangle&\ \&\ \langle bcdefghijk\rangle\ \&\ \langle fkm\rangle\\
H&\Rightarrow\langle abcjlm\rangle&\langle DEFGHI\rangle&\ \&\ \langle abcdefghijkm\rangle\ \&\ \langle jlm\rangle\\
\fbox{6}\ \&\ \langle ghi\rangle&\Rightarrow\langle HIJ\rangle&\langle DEFGHI\rangle\ \&\ \langle HIJ\rangle&\ \&\ \langle abcdefghijkm\rangle\ \&\ \langle jlm\rangle\\
j&\Rightarrow\langle BDHJ\rangle&\langle BDEFGHIJ\rangle&\ \&\ \langle abcdefghijkm\rangle\ \&\ \langle jlm\rangle\\
\fbox{9}\ \&\ \langle bkm\rangle&\Rightarrow\langle CDE\rangle&\langle BDEFGHIJ\rangle\ \&\ \langle CDE\rangle&\ \&\ \langle abcdefghijkm\rangle\ \&\ \langle jlm\rangle\\
f&\Rightarrow\langle BCE\rangle&\langle BCDEFGHIJ\rangle&\ \&\ \langle abcdefghijkm\rangle\ \&\ \langle jlm\rangle\\
\fbox{8}\ \&\ \langle BCD\rangle&\Rightarrow\langle jkl\rangle&\langle BCDEFGHIJ\rangle&\ \&\ \langle abcdefghijklm\rangle\\
m&\Rightarrow\langle ABH\rangle&\langle BCDEFGHIJ\rangle\ \&\ \langle ABH\rangle&\ \&\ \langle abcdefghijklm\rangle\\
\fbox{14}\ \&\ \langle jkl\rangle&\Rightarrow\langle AHJ\rangle&\langle ABCDEFGHIJ\rangle&\ \&\ \langle abcdefghijklm\rangle.
\end{align*}
We see that we are left with a single option, which is the one we started with.
\end{proof}

\begin{lemm}\label{only-two-lemm}
Up to combinatorial equivalence, there exist exactly two realizations of
an admissible dividing configuration having the dividing code
\begin{equation}\label{code6-2-2}\begin{matrix}
\{(14,9,22,1,18,5),(2,21,10,17),(16,11,20,3),(4,19),(6,13),(12,7),(8,15)\},\\
\{(1,2,3,4,5,6,7,8,9,10,11,12,13,14,15,16,17,18,19,20,21,22)\}.
\end{matrix}\end{equation}
\end{lemm}

\begin{proof}
We use the same notation system as in the proof of Lemma~\ref{only-one-lemm}. The labeling
of the occupied levels is shown in Figure~\ref{abcd2}.
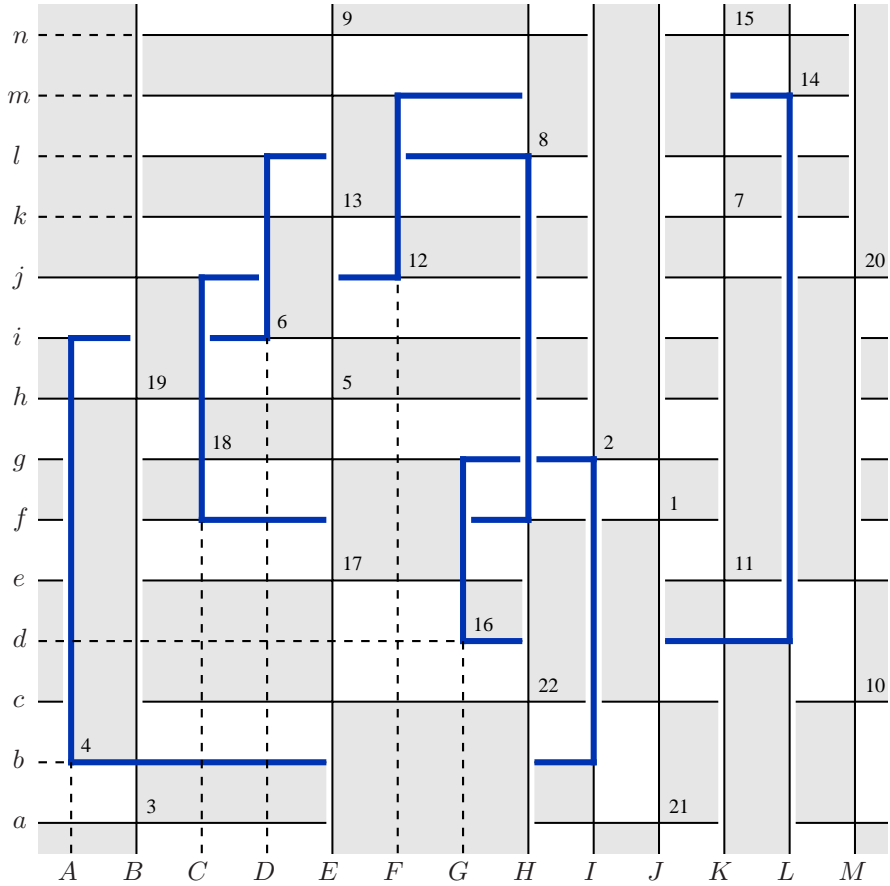
\begin{figure}[ht]
\includegraphics[scale=0.8]{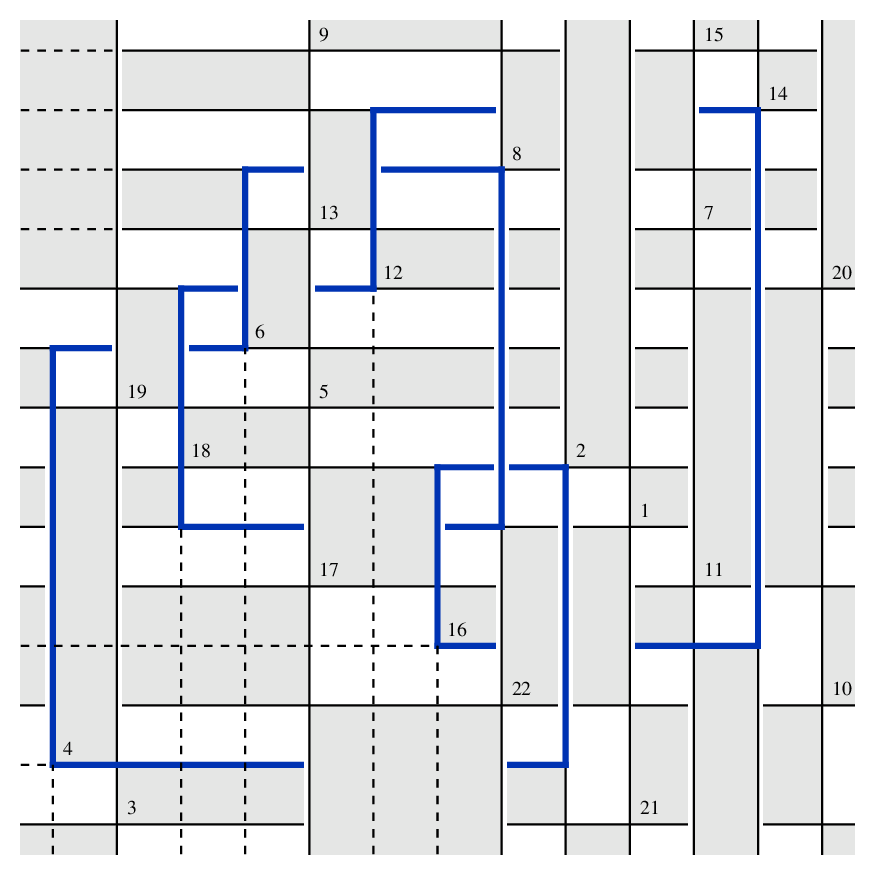}%
\put(-322,-2){$A$}%
\put(-297.384615385,-2){$B$}%
\put(-272.769230769,-2){$C$}%
\put(-248.153846154,-2){$D$}%
\put(-223.538461538,-2){$E$}%
\put(-198.923076923,-2){$F$}%
\put(-174.307692308,-2){$G$}%
\put(-149.692307692,-2){$H$}%
\put(-123,-2){$I$}%
\put(-99,-2){$J$}%
\put(-75.8461538462,-2){$K$}%
\put(-50,-2){$L$}%
\put(-26.6153846154,-2){$M$}%
\put(-339,17.5){$a$}%
\put(-339,40.3571428571){$b$}%
\put(-339,63.2142857143){$c$}%
\put(-339,86.0714285714){$d$}%
\put(-339,108.928571429){$e$}%
\put(-339,131.785714286){$f$}%
\put(-339,154.642857143){$g$}%
\put(-339,177.5){$h$}%
\put(-339,200.357142857){$i$}%
\put(-339,223.214285714){$j$}%
\put(-339,246.071428571){$k$}%
\put(-339,268.928571429){$l$}%
\put(-341,291.785714286){$m$}%
\put(-339,314.642857143){$n$}%
\caption{A realization of~\eqref{code6-2-2}}\label{abcd2}
\end{figure}

The reasoning here is slightly more complicated because at some point we need to consider two cases, $\langle CME\rangle$
and~$\langle CEM\rangle$. In the former we get a contradiction. Here is the complete proof:
\begin{align*}
&\text{argument}&\text{conclusion}\\
g&\Rightarrow\langle CEJ\rangle&\langle CEJ\rangle\\
c&\Rightarrow\langle EHJM\rangle&\langle CEHJ\rangle\ \&\ \langle EJM\rangle\\
\text{assume}&\quad\langle CME\rangle&\langle CMEHJ\rangle\\
h&\Rightarrow\langle BCE\rangle&\langle CMEHJ\rangle\ \&\ \langle BCME\rangle\\
K&\Rightarrow\langle ejk\rangle&\langle CMEHJ\rangle\ \&\ \langle BCME\rangle&\ \&\ \langle ejk\rangle\\
\fbox{20}\ \&\ \langle BME\rangle&\Rightarrow\langle akj\rangle&\langle CMEHJ\rangle\ \&\ \langle BCME\rangle&\ \&\ \langle akej\rangle\\
E&\Rightarrow\langle cek\rangle&\langle CMEHJ\rangle\ \&\ \langle BCME\rangle&\ \& \ \langle akcej\rangle\\
j&\Rightarrow\langle CFM\rangle&\langle CFMEHJ\rangle\ \&\ \langle BCME\rangle&\ \&\ \langle akcej\rangle\\
\fbox{21}\ \&\ \langle akc\rangle&\Rightarrow\langle FJM\rangle&\text{a contradiction}\\
\text{hence}&\quad\langle CEM\rangle&\langle CEHJM\rangle\\
J&\Rightarrow\langle acfg\rangle&\langle CEHJM\rangle&\ \& \ \langle acfg\rangle\\
\fbox{10}\ \&\ \langle CEM\rangle&\Rightarrow\langle cef\rangle&\langle CEHJM\rangle&\ \& \ \langle acefg\rangle\\
\fbox{22}\ \&\ \langle cef\rangle&\Rightarrow\langle GHJ\rangle&\langle CEHJM\rangle\ \&\ \langle GHJ\rangle&\ \& \ \langle acefg\rangle\\
g&\Rightarrow\langle EGIJ\rangle&\langle CEGHJM\rangle\ \&\ \langle GIJ\rangle&\ \& \ \langle acefg\rangle\\
j&\Rightarrow\langle BCM\rangle&\langle BCEGHJM\rangle\ \&\ \langle GIJ\rangle&\ \& \ \langle acefg\rangle\\
\fbox{3}\ \&\ \langle BEI\rangle&\Rightarrow\langle abc\rangle&\langle BCEGHJM\rangle\ \&\ \langle GIJ\rangle&\ \& \ \langle abcefg\rangle\\
e&\Rightarrow\langle GKM\rangle&\langle BCEGHJM\rangle\ \&\ \langle GIJ\rangle\ \&\ \langle GKM\rangle&\ \& \ \langle abcefg\rangle\\
n&\Rightarrow\langle EHK\rangle&\langle BCEGHJM\rangle\ \&\ \langle GIJ\rangle\ \&\ \langle HKM\rangle&\ \& \ \langle abcefg\rangle\\
\fbox{22}\ \&\ \langle cef\rangle&\Rightarrow\langle HJK\rangle&\langle BCEGHJKM\rangle\ \&\ \langle GIJ\rangle&\ \& \ \langle abcefg\rangle\\
E&\Rightarrow\langle cghkn\rangle&\langle BCEGHJKM\rangle\ \&\ \langle GIJ\rangle&\ \& \ \langle abcefg\rangle\ \&\ \langle cgn\rangle\\
\fbox{21}\ \&\ \langle JKM\rangle&\Rightarrow\langle acn\rangle&\langle BCEGHJKM\rangle\ \&\ \langle GIJ\rangle&\ \& \ \langle abcefghkn\rangle\\
\fbox{9}\ \&\ \langle acn\rangle&\Rightarrow\langle EHI\rangle&\langle BCEGHIJKM\rangle&\ \& \ \langle abcefghkn\rangle\\
\fbox{16}\ \&\ \langle GHK\rangle&\Rightarrow\langle cde\rangle&\langle BCEGHIJKM\rangle&\ \& \ \langle abcdefghkn\rangle\\
\fbox{15}\ \&\ \langle cdn\rangle&\Rightarrow\langle KLM\rangle&\langle BCEGHIJKLM\rangle&\ \& \ \langle abcdefghkn\rangle\\
K&\Rightarrow\langle ejkln\rangle&\langle BCEGHIJKLM\rangle&\ \& \ \langle abcdefghkln\rangle\ \&\ \langle ejk\rangle\\
\fbox{4}\ \&\ \langle bch\rangle&\Rightarrow\langle ABM\rangle&\langle ABCEGHIJKLM\rangle&\ \& \ \langle abcdefghkln\rangle\ \&\ \langle ejk\rangle\\
i&\Rightarrow\langle ADE\rangle&\langle ABCEGHIJKLM\rangle\ \&\ \langle ADE\rangle&\ \& \ \langle abcdefghkln\rangle\ \&\ \langle ejk\rangle\\
\fbox{14}\ \&\ \langle DEL\rangle&\Rightarrow\langle lmn\rangle&\langle ABCEGHIJKLM\rangle\ \&\ \langle ADE\rangle&\ \& \ \langle abcdefghklmn\rangle\ \&\ \langle ejk\rangle\\
\fbox{5}\ \&\ \langle AEH\rangle&\Rightarrow\langle hil\rangle&\langle ABCEGHIJKLM\rangle\ \&\ \langle ADE\rangle&\ \& \ \langle abcdefghklmn\rangle\ \&\ \langle hil\rangle\ \&\ \langle ejk\rangle\\
D&\Rightarrow\langle ikl\rangle&\langle ABCEGHIJKLM\rangle\ \&\ \langle ADE\rangle&\ \& \ \langle abcdefghiklmn\rangle\ \&\ \langle ejk\rangle\\
\fbox{20}\ \&\ \langle ABM\rangle&\Rightarrow\langle aij\rangle&\langle ABCEGHIJKLM\rangle\ \&\ \langle ADE\rangle&\ \& \ \langle abcdefghijklmn\rangle\\
\fbox{6}\ \&\ \langle ijk\rangle&\Rightarrow\langle CDE\rangle&\langle ABCDEGHIJKLM\rangle&\ \& \ \langle abcdefghijklmn\rangle\\
\fbox{13}\ \&\ \langle klm\rangle&\Rightarrow\langle EFH\rangle&\langle ABCDEGHIJKLM\rangle\ \&\ \langle EFH\rangle&\ \& \ \langle abcdefghijklmn\rangle.
\end{align*}
We are left with two options shown in Figure~\ref{realizations-after-flip}.
\end{proof}

\section*{Appendix B: Flypes}
\def\thesection{B}
\setcounter{defi}{0}
\setcounter{lemm}{0}
\setcounter{figure}{0}
\setcounter{prop}{0}

This section is included in the paper for completeness of the exposition.
It is moved to the very end of the paper because
flypes are not involved in establishing our main result.

\begin{defi}\label{flype-def}
Let $\Pi$ be a rectangular diagram of a surface, and let $\theta_1,\theta_2,\theta_3,\varphi_1,\varphi_2,\varphi_3\in\mathbb S^1$
be such that
\begin{enumerate}
\item
$\theta_2\in(\theta_1;\theta_3)$, $\varphi_2\in(\varphi_1;\varphi_3)$;
\item
points $v_1,v_2,v_3,v_4,v_5\in\mathbb T^2$ having coordinates
$(\theta_1,\varphi_3)$, $(\theta_2,\varphi_3)$, $(\theta_3,\varphi_3)$, $(\theta_3,\varphi_2)$,
$(\theta_3,\varphi_1)$, respectively, are vertices of $\Pi$, and, moreover, $v_1,v_3,v_5$ are $\diagup$-vertices,
and $v_2,v_4$ are $\diagdown$-vertices;
\item
$v_2$, $v_3$, and~$v_4$ do not belong to $\partial\Pi$;
\item
there are no more vertices of $\Pi$ in $[\theta_1;\theta_3]\times[\varphi_1;\varphi_3]$.
\end{enumerate}

These assumptions imply that $\Pi$ contains, among others, four rectangles of the following form:
$$r_1=[\theta_1;\theta_2]\times[\varphi';\varphi_3],\quad
r_2=[\theta_2;\theta_3]\times[\varphi_3;\varphi''],\quad
r_3=[\theta_3;\theta'']\times[\varphi_2;\varphi_3],\quad
r_4=[\theta';\theta_3]\times[\varphi_1;\varphi_2]$$
with some $\theta',\theta''\in(\theta_3;\theta_1)$, $\varphi',\varphi''\in(\varphi_3;\varphi_1)$.

Let $\Pi'$ be obtained from $\Pi$ by replacing these four rectangles with the following ones:
$$r_1'=[\theta_1;\theta_2]\times[\varphi';\varphi_1],\quad
r_2'=[\theta_2;\theta_3]\times[\varphi_1;\varphi''],\quad
r_3'=[\theta_1;\theta'']\times[\varphi_2;\varphi_3],\quad
r_4'=[\theta';\theta_1]\times[\varphi_1;\varphi_2];
$$
see Figure~\ref{flypefig}.
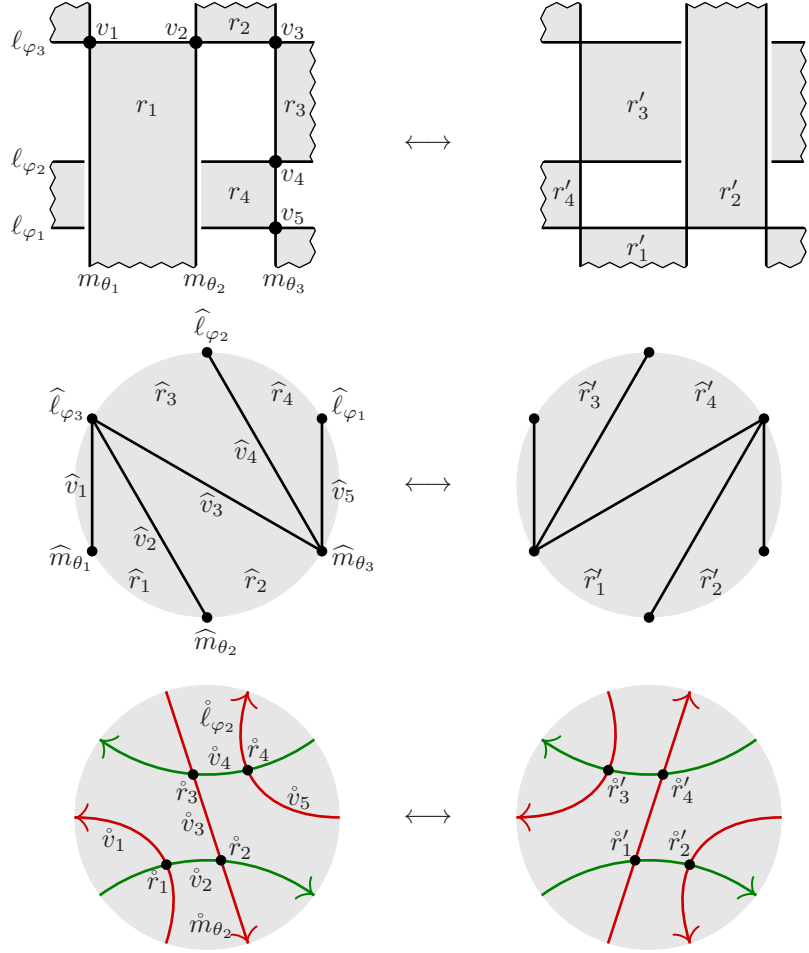
\begin{figure}[ht]
\includegraphics{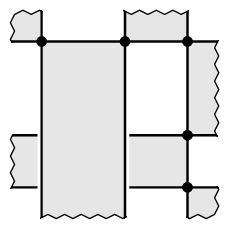}\put(-88,93){$v_1$}\put(-62,93){$v_2$}\put(-18,93){$v_3$}\put(-18,38){$v_4$}\put(-18,23){$v_5$}%
\put(-73,64){$r_1$}\put(-38,95){$r_2$}\put(-17,64){$r_3$}\put(-38,31){$r_4$}\put(-95,-2){$m_{\theta_1}$}\put(-55,-2){$m_{\theta_2}$}%
\put(-25,-2){$m_{\theta_3}$}\put(-120,18){$\ell_{\varphi_1}$}\put(-120,43){$\ell_{\varphi_2}$}\put(-120,88){$\ell_{\varphi_3}$}
\hskip1cm\raisebox{50pt}{$\longleftrightarrow$}\hskip1cm
\includegraphics{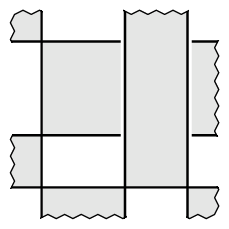}%
\put(-73,10){$r_1'$}\put(-38,31){$r_2'$}\put(-73,64){$r_3'$}\put(-101,31){$r_4'$}
\\[5pt]
\includegraphics{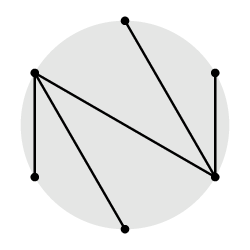}\put(-114,57){$\widehat v_1$}\put(-88,37){$\widehat v_2$}\put(-63,50){$\widehat v_3$}%
\put(-50,70){$\widehat v_4$}\put(-13,55){$\widehat v_5$}\put(-90,21){$\widehat r_1$}\put(-46,21){$\widehat r_2$}%
\put(-80,92){$\widehat r_3$}\put(-36,92){$\widehat r_4$}\put(-65,117){$\widehat\ell_{\varphi_2}$}%
\put(-13,87){$\widehat\ell_{\varphi_1}$}\put(-120,87){$\widehat\ell_{\varphi_3}$}\put(-13,30){$\widehat m_{\theta_3}$}%
\put(-65,-2){$\widehat m_{\theta_2}$}\put(-120,30){$\widehat m_{\theta_1}$}%
\hskip0.5cm\raisebox{58pt}{$\longleftrightarrow$}\hskip0.5cm\includegraphics{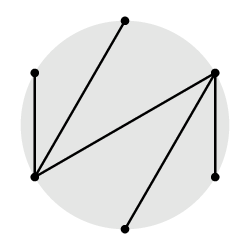}\put(-85,21){$\widehat r_1'$}\put(-41,21){$\widehat r_2'$}%
\put(-87,92){$\widehat r_3'$}\put(-43,92){$\widehat r_4'$}\\
\includegraphics{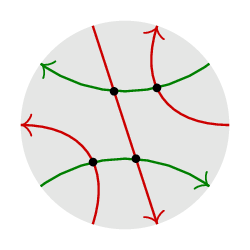}\put(-100,49){$\mathring v_1$}\put(-67,34){$\mathring v_2$}\put(-70,55){$\mathring v_3$}%
\put(-60,80){$\mathring v_4$}\put(-30,65){$\mathring v_5$}\put(-63,95){$\mathring\ell_{\varphi_2}$}\put(-67,17){$\mathring m_{\theta_2}$}%
\put(-83,33){$\mathring r_1$}\put(-52,46){$\mathring r_2$}\put(-72,67){$\mathring r_3$}\put(-45,83){$\mathring r_4$}%
\hskip0.5cm\raisebox{58pt}{$\longleftrightarrow$}\hskip0.5cm\includegraphics{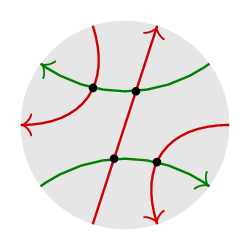}%
\put(-75,48){$\mathring r_1'$}\put(-53,48){$\mathring r_2'$}\put(-76,68){$\mathring r_3'$}\put(-52,68){$\mathring r_4'$}%
\caption{A type~I flype}\label{flypefig}
\end{figure}
\end{defi}
Then we say that the passage from $\Pi$ to $\Pi'$, or the other way, is \emph{a type~I flype}.
Note that the other rectangles of $\Pi$ and $\Pi'$, not
shown in Figure~\ref{flypefig}, are allowed to pass through
$[\theta_1;\theta_3]\times[\varphi_1;\varphi_3]$, the region where the modification occurs.

\emph{A type~II flype} is defined by reversing the $\theta$-direction in the
definition of a type~I flype and exchanging the types of vertices: $\diagdown\leftrightarrow\diagup$.

\begin{lemm}\label{flype-equivalence-lem}
Let $\Pi\mapsto\Pi'$ be a flype, and let~$(\delta_+,\delta_-)$,
$(\delta_+',\delta_-')$ be canonic dividing configurations of~$\widehat\Pi$ and~$\widehat\Pi'$,
respectively. Then there is an isotopy fixed on~$\widehat\Pi\cap\widehat\Pi'$ that brings~$(\widehat\Pi,\delta_+)$
to~$(\widehat\Pi',\delta_+')$ if the flype is of type~I,
and~$(\widehat\Pi,\delta_-)$ to~$(\widehat\Pi',\delta_-')$ if the flype
is of type~II. Moreover, the isotopy can be chosen to keep the surface in the class of
Giroux's convex surfaces with respect to~$\xi_+$ if the flype is of type~I, and with
respect to~$\xi_-$ if the flype is of type~II.
\end{lemm}

\begin{proof}
We use the notation from Definition~\ref{flype-def} and assume the flype is of type~I (the type~II case
is similar). The surface~$\widehat\Pi'$ is obtained
from~$\widehat\Pi$ by replacing the union of tiles~$d=\widehat r_1\cup\widehat r_2\cup
\widehat r_3\cup\widehat r_4$ with~$d'=\widehat r_1'\cup\widehat r_2'\cup\widehat r_3'\cup\widehat r_4'$.
One can see that both~$d$ and~$d'$ are two-discs tangent to
one another along their common boundary~$\partial d=\partial d'$.
One can also see that they enclose a three-ball~$B$ whose interior is disjoint from the
common part of the surfaces~$\widehat\Pi$ and~$\widehat\Pi'$. This implies the existence
of an isotopy from~$\widehat\Pi$ to~$\widehat\Pi'$ relative to~$\widehat\Pi\cap\widehat\Pi'$.
The fact that such an isotopy brings~$\delta_+$ to an abstract dividing set
isotopic to~$\delta_+'$ can be learned from Figure~\ref{flypefig}.

The second assertion of the lemma follows from the fact that the boundary~$\partial d$ is Legendrian,
and a(ny) dividing set of~$\widehat\Pi$
(with respect to~$\xi_+$) intersects~$d$ in two arcs. The technique of~\cite{gi1}
allows to show that an isotopy of such a disc preserving its boundary and the tangent
plane at every point of the boundary can be $C^0$-approximated by an isotopy within
the class of Giroux's convex surfaces.

For more detail see~\cite[Lemma~6]{dp17}, where a very little modification of
the proof is needed to obtain a proof of the second part of Lemma~\ref{flype-equivalence-lem}.
Namely, in the final part of the proof we conclude that
all singularities in the interior of the disc disappear because there is
just one dividing arc. This implies that $1$-arc also disappears.

In the present context we have two dividing arcs instead of just one,
which means that, after reducing all nodes in the interior of the disc, only one saddle
singularity is left, which is still not enough for the $1$-arc to survive.
So the concluding argument of the proof is still valid.
\end{proof}

Unlike (de)stabilizations, flypes of rectangular diagrams of surfaces do not
necessarily change the equivalence class of the corresponding surface
viewed as a Giroux's convex surface
with respect to one of the contact structures~$\xi_+$ or~$\xi_-$.
From the contact topology point of view flypes of rectangular diagrams of surfaces
are bypass attachments introduced by K.\,Honda in~\cite{honda00}.
The concept of an ineffective flype defined below is a combinatorial counterpart of
a trivial bypass attachment from~\cite{honda-trivial-bypass}.

\begin{defi}\label{effective-flype-def}
We use the notation from Definition~\ref{flype-def}. We also denote by~$(\delta_+,\delta_-)$ a canonic
dividing configuration of~$\widehat\Pi$.
The type~I flype~$\Pi\mapsto\Pi'$ is called \emph{ineffective} if
there is an embedded disc~$b\subset\widehat\Pi$ with boundary consisting of two arcs~$\alpha$ and~$\beta$ such that
(consult Figures~\ref{flypefig} and~\ref{ineffective-fig}):
\begin{enumerate}
\item
$\alpha\subset\delta_-\setminus(\mathring v_1\cup\mathring v_5)$;
\item
either~$\beta=\mathring v_4$ and~$b\supset\mathring\ell_{\varphi_2}$, or~$\beta=\mathring v_2$
and~$b\supset\mathring m_{\theta_2}$ (note that~$b$ can have more intersections with~$\delta_+$
than shown in Figure~\ref{ineffective-fig}, so the equality $b=\mathring\ell_{\varphi_2}$
or $b=\mathring m_{\theta_2}$ does not necessarily hold).
\end{enumerate}
\begin{figure}[ht]
\includegraphics{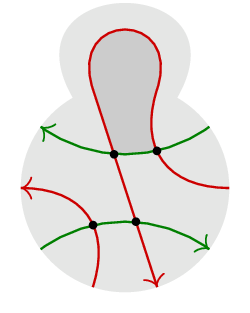}\put(-62,105){$b$}\put(-57,67){$\beta$}\put(-62,139){$\alpha$}
\hskip0.5cm\raisebox{56pt}{$\sim$}\hskip0.5cm
\includegraphics{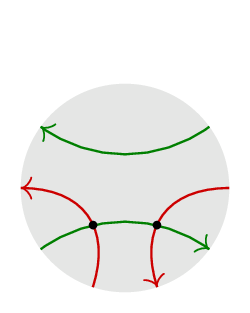}\hskip0.5cm\raisebox{56pt}{$\sim$}\hskip0.5cm
\includegraphics{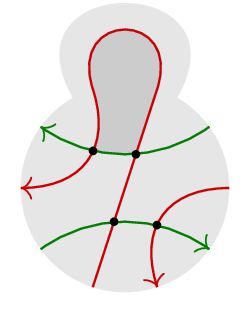}
\caption{Ineffective flype has no effect on the weak equivalence class
of a canonic dividing configuration}\label{ineffective-fig}
\end{figure}
Otherwise the flype is called \emph{effective}.

Effective and ineffective flypes of type~II are defined similarly with the roles of~$\delta_-$ and~$\delta_+$ exchanged.
\end{defi}

One can see that the inverse flype of an ineffective one is also ineffective.

\begin{prop}
If~$\Pi\mapsto\Pi'$ is an ineffective flype, then~$\Pi'$ can be obtained from~$\Pi$
be a finite sequence of exchange, wrinkle creation, and wrinkle reduction moves.

An ineffective flype has no effect on the
equivalence class of~$\widehat\Pi$ viewed as a Giroux's convex surface with respect to either of the contact
structures~$\xi_+$ and~$\xi_-$.
\end{prop}

\begin{proof}
We show how to modify~$\Pi$ and~$\Pi'$ by exchange and wrinkle reduction and creation moves
so that they eventually coincide. We use the notation from Definitions~\ref{flype-def} and~\ref{effective-flype-def}.
We may assume that $\beta=\mathring v_4$ and~$b\supset\mathring\ell_{\varphi_2}$, as
the other case is symmetric to this one.

The interior of the disc~$b$ may have nontrivial intersection with~$\delta_+$. In this case, there are bigons of~$\delta_+$
and~$\delta_-$ in~$b$ smaller than~$b$. Denote by~$D$
the dividing configuration obtained from~$(\delta_+,\delta_-)$ by reducing all bigons in~$b$
including~$b$ itself.

Suppose
for the moment that the dividing configuration~$D$ is admissible.

The proof is by induction in~$k$, where~$k$ is the number of intersections
of the interior of~$\alpha$ with~$\delta_+$. If~$k=0$, then~$b$ is a bigon
with corners at~$\mathring r_3$ and~$\mathring r_4$. The rectangles~$r_3$ and~$r_4$
share two vertices at the level~$\ell_{\varphi_2}$. The same holds for
the rectangles~$r_3'$ and~$r_4'$, and there is a respective bigon~$b'$ of a canonic
dividing configuration~$(\delta_+',\delta_-')$ of~$\widehat\Pi'$.

Reducing the bigons~$b$ and~$b'$ in~$\widehat\Pi$ and~$\widehat\Pi'$, respectively, as described
in the proof of Lemma~\ref{bigon-reduction-lemm} produces identical diagrams from~$\Pi$ and~$\Pi'$. So, the induction base~$k=0$
is settled.

If~$k>0$, then there must be a bigon~$b_0\subset b$ of~$(\delta_+,\delta_-)$
with corners~$\mathring r'$, $\mathring r''$ distinct from~$\mathring r_i$, $i=1,2,3,4$.
The rectangles~$r'$ and~$r''$ are not affected by the flype, so we have~$r',r''\in\Pi'$, and there is a respective
bigon~$b_0'$ of~$(\delta_+',\delta_-')$.

The reduction of the bigons~$b_0$ and~$b_0'$ made by following the recipe from the
proof of Lemma~\ref{bigon-reduction-lemm} produces diagrams that are still related by
a flype. Indeed, let~$S$ be the minimal strip of the form~$[\theta';\theta'']\times\mathbb S^1$ or~$\mathbb
S^1\times[\varphi';\varphi'']$ containing~$r'$ and~$r''$. If~$S$ is disjoint from the site of the flype,
which is~$R=[\theta_1;\theta_3]\times[\varphi_1;\varphi_3]$, then the bigon reductions
do not interfere with the flype.
\begin{figure}[ht]
{\includegraphics[scale=0.6]{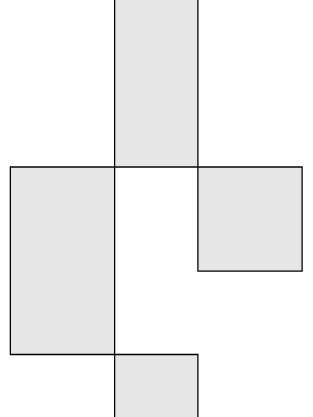}\put(-74,42){$r'$}\put(-47,95){$r''$}\put(-32,55){\parbox{1cm}{\begin{center}flype\\site\end{center}}}\hskip1cm
\includegraphics[scale=0.6]{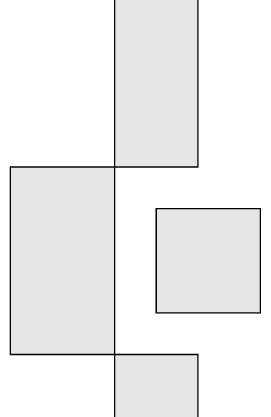}\put(-62,42){$r'$}\put(-35,95){$r''$}\put(-32,44){\parbox{1cm}{\begin{center}flype\\site\end{center}}}\hskip1cm
\includegraphics[scale=0.6]{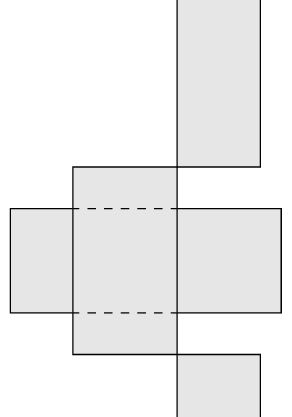}\put(-50,42){$r'$}\put(-23,95){$r''$}\put(-32,44){\parbox{1cm}{\begin{center}flype\\site\end{center}}}}
\caption{Possible mutual positions of the flype site and the ``large wrinkle''}\label{interference}
\end{figure}
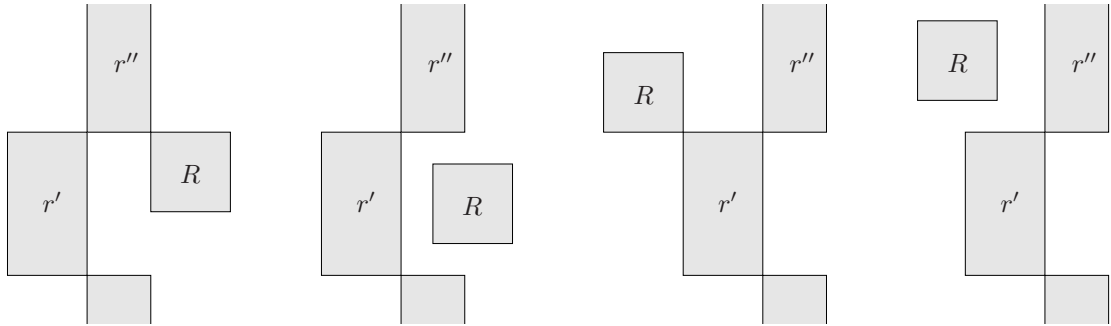

An interference occurs when either~$v_1$ or~$v_5$ is a vertex of~$r'$ or~$r''$,
or some of the vertices~$v_i$, $i=1,2,3,4,5$, are contained in one of the blocks
that are exchanged in the course of the bigon reduction (they are denoted by~$X$ and~$Y$
in Figure~\ref{exchange-to-remove-wrinkle}), see~Figure~\ref{interference}. One can see
that this interference is not essential. The moves applied to reduce the bigon,
change only the shape of the region where the flype occurs
but preserve all the flype preconditions.

The flype we arrive at after reducing the bigons~$b_0$ and~$b_0'$ is still inessential,
and~$k$ is reduced. Hence the induction step follows.

Now we reduce the general case to the one considered above, where~$D$ was admissible.
Denote by~$v_6$ and~$v_7$ the vertices opposite to~$v_4$ in the rectangles~$r_3$ and~$r_4$,
respectively. Let~$D_1$ be a dividing configuration obtained
from~$(\delta_+,\delta_-)$ by replacing a subarc of~$\mathring v_5$ by an arc
that intersects~$\mathring v_7$, then goes `parallel' to~$\alpha$ until it intersects~$\mathring v_6$,
then turns around and goes back remaining close to~$\alpha$, intersects~$\mathring v_7$
again, and finally arrives at an endpoint of~$\mathring v_5$; see Figure~\ref{detour}.
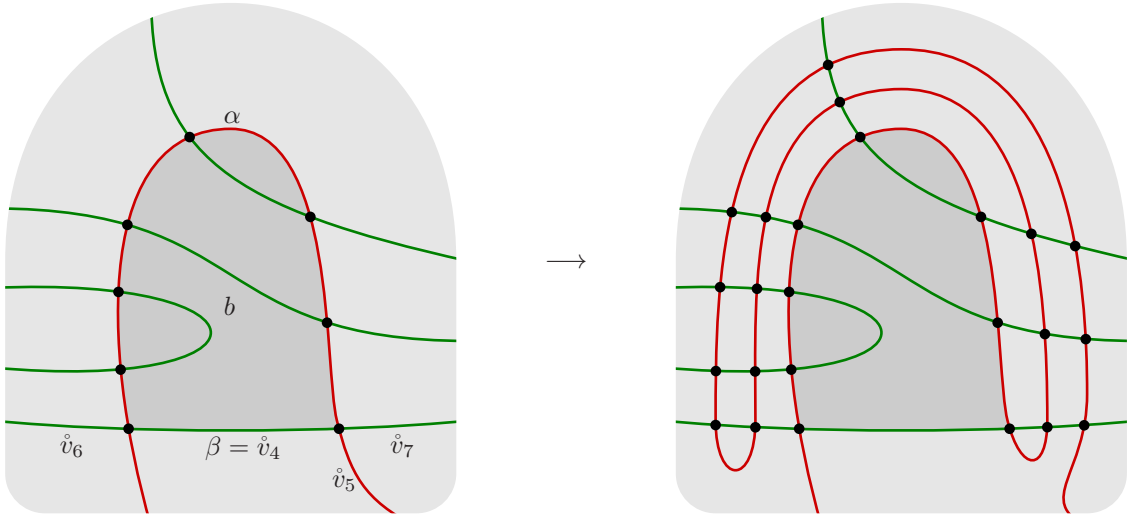
\begin{figure}[ht]
\includegraphics{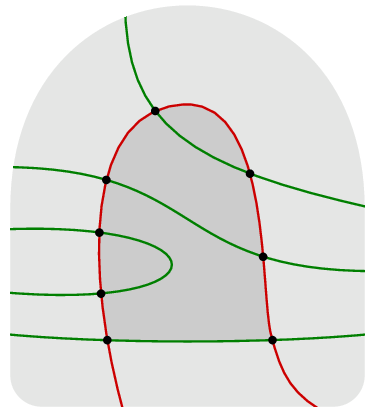}\put(-100,27){$\beta=\mathring v_4$}\put(-155,28){$\mathring v_6$}\put(-30,28){$\mathring v_7$}%
\put(-52,15){$\mathring v_5$}\put(-93,153){$\alpha$}\put(-93,80){$b$}
\hskip1cm\raisebox{98pt}{$\longrightarrow$}\hskip1cm
\includegraphics{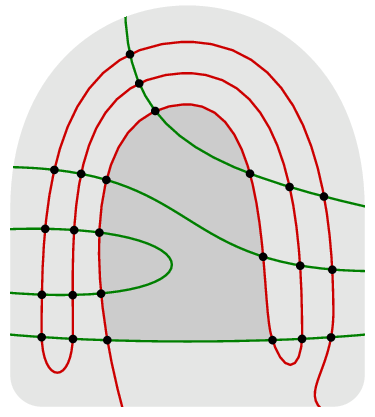}
\caption{Making the dividing configuration which is obtained by reducing all bigons in~$b$ admissible}\label{detour}
\end{figure}

The dividing configuration~$D_1$ is obtained from~$(\delta_+,\delta_-)$ by bigon creations, which,
by Lemma~\ref{bigon-creation-lemm}, can be realized by wrinkle creations in~$\Pi$. These wrinkle creations do not interfere with the flype we consider,
which means that the same wrinkle creations can be made in~$\Pi'$ so that the new~$\Pi$ and~$\Pi'$ will be still related
by an ineffective flype. Reducing all bigons of~$D_1$ contained in~$b$, including~$b$ itself, now produces an admissible
configuration, which completes the proof.
\end{proof}

\section*{Acknowledgements}
We are indebted to our anonymous referees for their careful reading of our paper and valuable suggestions for
improving the exposition.

\end{document}